\documentclass[11pt]{amsart}
\usepackage{amsmath,amsthm, amscd, amssymb, amsfonts, mathrsfs,pb-diagram,pb-xy}
\usepackage{multicol}
\usepackage[all]{xy}
\usepackage{enumitem}
\usepackage{mathrsfs}
\usepackage[dvips, dvipsnames, usenames]{color}
\usepackage{hyperref}

\allowdisplaybreaks

\makeatletter
\newcommand{\leqnomode}{\tagsleft@true}
\newcommand{\reqnomode}{\tagsleft@false}
\makeatother

\renewcommand{\_}[1]{_{\left( #1 \right)}}

\newcommand{\coef}[2]{\mathsf{c}_{#1}^{(#2)}}

\newcommand{\coeff}[2]{\mathsf{d}_{#1}^{(#2)}}

\newcommand{\Dchaintwo}[3]{\xymatrix@C-4pt{\overset{#1}{\underset{\ }{\circ}}\ar
@{-}[r]^{#2}
& \overset{#3}{\underset{\ }{\circ}}}}

\newcommand{\Dchainfive}[9]{\xymatrix@C-6pt{\overset{#1}{\underset{\ }{\circ}}\ar
@{-}[r]^{#2}  & \overset{#3}{\underset{\ }{\circ}}\ar  @{-}[r]^{#4}  &
\overset{#5}{\underset{\ }{\circ}}
\ar  @{-}[r]^{#6}  & \overset{#7}{\underset{\ }{\circ}}\ar  @{-}[r]^{#8}  &
\overset{#9}{\underset{\ }{\circ}}}}

\newcommand{\supera}[2]{{\mathbf A}(#1|#2)}
\newcommand{\superb}[2]{{\mathbf B}(#1|#2)}
\newcommand{\superd}[2]{{\mathbf D}(#1|#2)}
\newcommand{\superda}[1]{{\mathbf D}(2,1;#1)}
\newcommand{\superf}{{\mathbf F}(4)}
\newcommand{\superg}{{\mathbf G}(3)}

\newcommand{\mor}{\sim_{\text{Mor}}}

\newcommand{\Brown}{\mathtt{br}}

\newcommand{\Bgl}{\mathtt{wk}}

\newcommand{\roots}{\mathtt X}
\newcommand{\x}{\mathtt x}
\newcommand{\y}{\mathtt y}
\newcommand{\ad}{\operatorname{ad}}
\newcommand{\cop}{\operatorname{cop}}

\newcommand{\car}{\operatorname{char}}
\newcommand{\coh}{\operatorname{H}}
\newcommand{\hoch}{\operatorname{HH}}
\newcommand{\gr}{\operatorname{gr}}
\newcommand{\Hom}{\operatorname{Hom}}
\newcommand{\Ext}{\operatorname{Ext}}
\newcommand{\id}{\operatorname{id}} 
\newcommand{\im}{\operatorname{im}}
\newcommand{\lcm}{\operatorname{lcm}}
\newcommand{\ord}{\operatorname{ord}}
\newcommand{\ztu}{\overline{\zeta}}

\newcommand{\supp}{\operatorname{supp}}

\newcommand{\ku}{ \Bbbk}
\newcommand{\I}{\mathbb I}
\newcommand{\G}{\mathbb G}
\newcommand{\N}{\mathbb N}
\newcommand{\Z}{\mathbb Z}

\newcommand{\toba}{\mathscr{B}}
\newcommand{\bq}{\mathfrak{q}}

\newcommand{\Cc}{{\mathcal C}}
\newcommand{\Dc}{\mathcal D}
\newcommand{\Xc}{\mathcal X}

\newcommand{\yd}[1]{{}^{#1}_{#1}\mathcal{YD}}
\newcommand{\ydh}{{}^{H}_{H}\mathcal{YD}}

\newcommand{\ydho}{{}^{H_0}_{H_0}\mathcal{YD}}
\newcommand{\ydhoo}{{}^{H_{[0]}}_{H_{[0]}}\mathcal{YD}}
\newcommand{\ydk}{{}^{K}_{K}\mathcal{YD}}

\newcommand{\ydkd}{{}^{K^\#}_{K^\#}\mathcal{YD}}

\newcommand{\ydkG}{{}^{\ku G}_{\ku G}\mathcal{YD}}
\newcommand{\ydkg}{{}^{\ku\Gamma}_{\ku\Gamma}\mathcal{YD}}
\newcommand{\ydkgd}{{}^{\ku^\Gamma}_{\ku^\Gamma}\mathcal{YD}}

\newcommand{\Ac}{\mathcal{A}}
\newcommand{\Ic}{\mathcal{I}}
\newcommand{\Rc}{\mathcal{R}}
\newcommand{\Tc}{\mathcal{T}}

\newcommand{\Asj}{q_{\alpha \beta}}
\newcommand{\Bsj}{\mathtt{b}}
\newcommand{\Csj}{q_{\alpha \gamma}}
\newcommand{\Dsj}{q_{\gamma \beta}}

\newcommand{\basis}{\mathtt{B}}
\newcommand{\chain}{\mathtt{M}}
\newcommand{\wchain}{\widehat{\chain}}
\newcommand{\Vt}{\mathtt{V}}

\newcommand{\pf}{\begin{proof}}
\newcommand{\epf}{\end{proof}}

\numberwithin{equation}{subsection}
\newtheorem{theorem}[equation]{Theorem}
\newtheorem{assumption}[equation]{Condition}
\newtheorem{lema}[equation]{Lemma}
\newtheorem{lemma}[equation]{Lemma}
\newtheorem{lem}[equation]{Lemma}
\newtheorem{cor}[equation]{Corollary}

\newtheorem{prop}[equation]{Proposition}

\theoremstyle{definition}
\newtheorem{definition}[equation]{Definition}
\newtheorem{example}[equation]{Example}
\newtheorem{observation}[equation]{Observation}
\newtheorem*{observation*}{Observation}
\newtheorem{question}[equation]{Question}

\newtheorem*{claim*}{Claim}
\newtheorem{claimA}{Claim}

\theoremstyle{remark}
\newtheorem{remark}[equation]{Remark}

\newcommand{\ot}{\otimes}

\newcommand{\bu}{\setlength{\unitlength}{1pt}\begin{picture}(2.5,2)
                      (1,1)\put(2,3.5){\circle*{2}}\end{picture}}

\newcommand{\hopfuno}{\mathcal H}
\newcommand{\hopfdos}{\mathcal L}
\newcommand{\Wc}{\mathcal W}

\begin{document}

\noindent
\title[Cohomology rings of finite-dimensional Hopf algebras]
{Cohomology rings of finite-dimensional pointed Hopf algebras over abelian groups}
\author{N. Andruskiewitsch, I. Angiono, J. Pevtsova, S. Witherspoon} 

\address{FaMAF-CIEM (CONICET), Universidad Nacional de C\'ordoba,
Medina A\-llen\-de s/n, Ciudad Universitaria, 5000 C\' ordoba, Rep\'
ublica Argentina.} 
\email{nicolas.andruskiewitsch@unc.edu.ar}
\email{ivan.angiono@unc.edu.ar}

\address{Department of Mathematics University of Washington Seattle, WA 98195,
USA}
\email{julia@math.washington.edu}

\address{Department of Mathematics Texas A\&M University College Station, TX 77843, USA}
\email{sjw@math.tamu.edu}

\date{\today} 
\begin{abstract}  
We show that the cohomology ring of a finite-dimensional complex pointed Hopf algebra 
with an abelian group of group-like elements is finitely generated. Our strategy has three major steps. 
We first reduce the problem to the finite generation of cohomology of finite dimensional Nichols 
algebras of diagonal type. For the Nichols algebras we do a detailed analysis of cohomology via 
the Anick resolution reducing the problem further to specific combinatorial properties. Finally, 
to check these properties we turn to the classification of Nichols algebras of diagonal type due 
to Heckenberger. In this paper we complete the verification of these combinatorial properties for 
major parametric families, including Nichols algebras of Cartan and super types and develop all 
the theoretical foundations necessary for the case-by-case analysis. The remaining discrete families 
are addressed in a separate publication.  As an application of the main theorem we deduce finite generation 
of cohomology for other classes of finite-dimensional Hopf algebras, including basic Hopf algebras with 
abelian groups of characters and finite quotients of quantum groups at roots of one.
\end{abstract} 
\maketitle
 \vspace{0.1in} 

\setcounter{tocdepth}{2}
\tableofcontents

\newpage
\section{Introduction}

\subsection{Antecedents}
A fundamental result in representation theory of a
finite group scheme \cite[Theorem 1.1]{FS} is that its cohomology 
satisfies the finite generation property. Using the language of Hopf algebras 
it can be phrased as follows. Let $H$ be a 
finite-dimensional cocommutative Hopf algebra over a field $k$.  Then

\begin{enumerate}[leftmargin=*,label=\rm{(fgc-\alph*)}]
\item\label{item:fgca} The cohomology ring $\coh(H, \ku)$ is finitely generated.

\item\label{item:fgcb} For any finitely generated $H$-module $M$, $\coh(H,M)$ is a finitely generated
module over $\coh(H, \ku)$. 
\end{enumerate}

Prior to the Friedlander-Suslin theorem, the result was known for group algebras of finite groups \cite{Go,V,evens}, 
restricted enveloping algebras \cite{FPa,AJ} and finite dimensional subalgebras of the Steenrod algebra \cite{wilkerson}. 
At the end of the introduction of \cite{FS},
the authors observe that the cohomology ring of a finite-dimensional 
\emph{commutative} Hopf algebra is easily seen to be finitely generated using the structure as in \cite{waterhouse} and add: 

\begin{quote}
\emph{We do not know whether it is reasonable to expect finite generation of the
 cohomology of an arbitrary finite-dimensional Hopf algebra.}
\end{quote}

Slowly, evidence confirming that this is indeed a reasonable question has emerged. In \cite{GK}
the cohomology ring of 
Lusztig's small quantum groups $u_q({{\mathfrak g}})$ (in characteristic 0)  
under some restrictions on the parameters was identified as the 
coordinate ring of the nilpotent cone of the Lie algebra $\mathfrak g$.
The restrictions on the parameters were weakened in \cite{BNPP}.
The finite generation of 
cohomology was established  for the duals of Lusztig's small quantum groups (in characteristic 0) \cite{Gordon},
for Lusztig's small quantum groups in positive characteristic  \cite{Du},
for finite supergroup schemes \cite{Du2}, for 
finite-dimensional complex pointed Hopf algebras  whose group of grouplike elements
is abelian and has order coprime to 210 \cite{MPSW},
for some pointed Hopf algebras of dimension $p^3$ \cite{NWW,EOW} (in characteristic $p > 0$), for
the bosonizations of the Fomin-Kirillov algebra $\mathcal{FK}_3$ with the group algebra 
of $\mathbb S_3$ and its dual \cite{stefan-vay}, for Drinfeld doubles of finite group schemes \cite{FN, N}.
In all the cases above, the approach  is based to a greater or lesser extent 
on the knowledge of the structure of the Hopf algebras under consideration. 

Finite tensor categories were introduced in \cite{etingof-ostrik03}, where it was also conjectured
that finite generation holds in this more general context. 
A systematic study of this question was started in \cite{NP}.

\subsection{The main result and applications}
In the present paper we work over an algebraically closed field $\ku$ of characteristic 0.
For brevity we shall say that an augmented algebra $H$ has finitely generated cohomology (abbreviated as fgc) 
when both \ref{item:fgca} and \ref{item:fgcb} hold.
Our main result is the following:

\begin{theorem}\label{thm:fingencoh-pointed}
Let $H$ be a finite-dimensional pointed Hopf algebra whose group of group-like elements is abelian. Then $H$ has finitely generated cohomology.
\end{theorem}

The class of finite-dimensional pointed Hopf algebras is
the best understood and the subclass of those with abelian group of group-like elements is the only one whose classification is essentially complete.
Theorem \ref{thm:fingencoh-pointed} goes beyond the situation treated in \cite{MPSW} but uses the same approach to the classification of
pointed Hopf algebras proposed in \cite{ASjalg,AS-survey}. 
Let us mention the main differences between the setting of \cite{MPSW}, that invoked the classification result \cite{AS-ann}, and the present work. In the former, 
the associated braided vector space $V$ (described below) was of Cartan type and the deformations of the defining relations in the liftings took values in the group algebras. 
These restrictions, in the terminology introduced in \cite{NPe}, guaranteed that the Nichols algebra $\toba(V)$ had a smooth integration $Q \to \toba(V)$ by an algebra of finite global dimension. That property, though not stated as explicitly, was crucial for the techniques in \cite{MPSW}. 
When the restriction on the order of the group of group-like elements $G(H)$ is dropped, $V$ belongs to the list in the celebrated classification of \cite{H-classif} but is not necessarily of Cartan type. The defining relations of the Nichols algebras and their deformations are more involved, see \cite{A-jems,Ang-crelle,AAG,AG,GJ,helbig} and conceptually 
different resulting in the absence of the crucial smooth integration property. In particular, as our results demonstrate, generating classes of the cohomology ring of a general Nichols algebra of diagonal type can lie in {\it arbitrary large degrees} whereas in the context of \cite{MPSW} and whenever the algebra is smoothly integrable, generating classes lie in degree 2. We do have control over the degrees of the generators: they depend on vanishing of certain coefficients as stated in Remark~\ref{obs:main}. In the very computationally heavy Section~\ref{sec:computational-lemmas} we calculate these coefficients which turn out to be quantum integers determined by the defining parameters of the Nichols algebra. By choosing different Nichols algebras from the classification list one can obtain various high values for the minimal degrees of the generators of the cohomology ring. Hence, handling this more general case of {\it all} Nichols algebras of diagonal type calls for development of new techniques which we present in this work.  

\medbreak
We state two direct applications of Theorem \ref{thm:fingencoh-pointed} extending further the number of types of finite-dimensional Hopf algebras with finitely generated cohomology. We also observe that Theorem \ref{thm:ppalrealiz-Kss}
provides another class of Hopf algebras satisfying fgc.

\begin{theorem}\label{thm:fingencoh-basic}
Let $H$ be a finite-dimensional basic Hopf algebra whose group of characters is abelian. Then $H$ has finitely generated cohomology.
\end{theorem}

Basic Hopf algebras with abelian group of characters are just the duals of the Hopf algebras in 
Theorem \ref{thm:fingencoh-pointed}; thus  Theorem \ref{thm:fingencoh-basic}, that generalizes \cite{Gordon}, follows from 
Theorem \ref{thm:fingencoh-pointed}, Lemma \ref{lema:morita-dual}, Corollary \ref{cor:morita} and Theorem \ref{thm:fingencoh-drinfeld-double}.

\begin{theorem}\label{thm:fingencoh-extensions}
Let $H$ be a finite-dimensional Hopf algebra that fits into an extension   
$\ku \to K \rightarrow H \rightarrow L \to \ku$, where 
$K$ is semisimple and $L$ is either pointed with abelian group of group-like elements or else basic with abelian group of characters.
Then $H$ has fgc.
\end{theorem}

Theorem \ref{thm:fingencoh-extensions} follows from Lemma \ref{lema:extensions} and one of the previous two theorems.
Quotients of algebras of functions on quantum groups at roots of one (of various kinds)
were classified in \cite{AG-compo,Ga,GG}.
In particular, these results provide families of Hopf algebras $H$ 
that fit into an extension   
$\ku \to \ku^{|G|} \rightarrow H \rightarrow L \to \ku$
where $G$ is a finite group and $L$ is a finite-dimensional basic Hopf algebra with abelian group of characters; 
thus Theorem \ref{thm:fingencoh-extensions} applies to them.

\subsection{Scheme of the proof of Theorem \ref{thm:fingencoh-pointed}} 
\label{sub:scheme}
Let $H$ be a finite-dimensional pointed Hopf algebra with abelian group of group-like elements $\Gamma := G(H)$
so that the coradical of $H$ is $H_0 \simeq \ku \Gamma$.
Let $D(H)$ be the the Drinfeld double of $H$, let $\gr H$ be the graded Hopf algebra associated to the coradical filtration and let $V$ be the infinitesimal braiding of $H$, see \S \ref{subsec:na-role}.  We know that $\gr H \simeq \toba(V) \# \ku \Gamma$ \cite{Ang-crelle}.
Then the Nichols algebra $\toba(V)$ is finite-dimensional. We shall use $V^\# = \Hom_k(V,k)$ to denote the $k$-linear dual. 
The key point in the proof of Theorem \ref{thm:fingencoh-pointed} is the following.

\begin{theorem}\label{th:fingencoh-nichols-diagonal}
Let $U$ be a braided vector space of diagonal type such that the Nichols algebra $\toba(U)$ 
has finite dimension. Then $\toba(U)$ has fgc.
\end{theorem}

This Theorem being proved, the rest of the proof proceeds in the following steps: 
\begin{align*}
\xymatrix@C=60pt{
 *\txt{Theorem \ref{th:fingencoh-nichols-diagonal} \quad}\ar@{=>}[r] 
 &  *\txt{ $\toba(V)$, $\toba(V^\#)$ have fgc \quad}\ar@{=>}[r]^{\text{\quad Theorem \ref{th:RtoRsmashH} \quad}}
& *\txt{ $\gr H$, $(\gr H)^\#$ have fgc \quad} \ar@{=>}[d]_{\text{Theorem \ref{thm:fingencoh-drinfeld-double}}}
\\
*\txt{ $H$ has fgc \quad}  & *\txt{ $D(H)$ has fgc \quad} \ar@{=>}[l]_{\text{Theorem \ref{lema:frobenius-reciprocity}}}
& *\txt{ $D(\gr H)$ has fgc \quad} \ar@{<=>}[l]_{\text{Theorem \ref{th:summary} \quad}}}
\end{align*}

\medbreak

\subsection{Finite-generation of cohomology for Nichols algebras}\label{subsec:intro-nichols-fgc}
We next outline the proof of  Theorem \ref{th:fingencoh-nichols-diagonal} 
referring to \S \ref{subsec:na-diag} for unexplained terminology.

\subsubsection{Reduction to the connected case}\label{subsubsec:connected-reduction}

By Theorem \ref{lem:BO},  we conclude
that Theorem \ref{th:fingencoh-nichols-diagonal} holds  for $U$ if and only if it holds for $U_J$ for every connected component $J \in \Xc$. 

\emph{We assume for the rest of this Subsection that the Dynkin diagram of $U$ is connected.}

\subsubsection{The Anick resolution}
The Nichols algebra $\toba(U)$ has a convex PBW-basis, hence a suitable filtration. 
Its associated graded ring $\gr \toba(U)$ is a  quantum linear space. 
The cohomology ring of $\gr \toba(U)$ is well-known, but we provide a 
computation using the Anick resolution \cite{Anick} specifically in order to
relate it to permanent cycles in a suitable spectral sequence. 
See \S \ref{subsec:qls}.

Since the Anick resolution is compatible with the mentioned filtration on $\toba(U)$, 
we may use a spectral sequence argument based on Evens Lemma \ref{Lch2}
to  reduce the finite generation of $\coh(\toba(U), \ku)$ to the
verification of the following  statement.

\begin{assumption}\label{assumption:intro-combinatorial}
For every positive root $\gamma \in \varDelta_+^{U}$, there exists
$L_{\gamma} \in \N$ such that
the cochain $\left(\x_{\gamma}^{L_{\gamma}} \right)^*$ is a
cocycle, 
that is, represents an element in $\coh(\toba(U), \ku)$.
\end{assumption}
The elements $\left(\x_{\gamma}^{L_{\gamma}} \right)^*$ are cochains of the complex $\Hom(C_*(\toba(U),\ku)$ which computes $\coh(\toba(U),\ku)$ and are defined prior to Theorem~\ref{lemma:reduction-to-cocycles}. 
If Condition \ref{assumption:intro-combinatorial} holds, then 
Theorem \ref{lemma:reduction-to-cocycles} implies that $\toba(U)$ has fgc.

\subsubsection{Reduction to Weyl-equivalence}
In practice, given $U$ we shall prove that  Condition \ref{assumption:intro-combinatorial} holds for any braided vector space with the same Dynkin 
diagram as $U$, particularly for $U^\#$. 
Let $G$ be any finite abelian group such that $U$ is realized in $\ydkG$.
By Theorem \ref{th:RtoRsmashH}  and Theorem \ref{thm:fingencoh-drinfeld-double} we see that
$D(\toba(U) \# \ku G)$ has fgc.

We apply  this last claim as follows: let $U'$ be a braided vector space of diagonal type
which is \emph{Weyl-equivalent} to $U$ (see \S~\ref{sub:weyl}). This implies that 
$U'$ is realized as Yetter-Drinfeld module over $G$ and there is an algebra isomorphism
\begin{align*}
D(\toba(U) \# \ku G) \simeq D(\toba(U') \# \ku G).
\end{align*}
By Corollary  \ref{cor:frobeniusreciprocity-nichols} $\toba(U')$ has fgc. That is, we only need to deal 
with one representative of each Weyl-equivalence class which drastically reduced the amount of computations we need
 to  perform  to finish the proof.

\subsubsection{Verification of Condition \ref{assumption:intro-combinatorial}} 
We argue case-by-case using the list of \cite{H-classif}; by the preceding discussion
we just need to consider one representative in each Weyl-equivalence class---and we could choose the most convenient for our purpose.
We also argue recursively on $\dim U$. All in all, we reduce the verification to claims on 
Nichols algebras of diagonal type, see \S \ref{subsec:gral-lemmas} and we deal with them 
using  information  on the PBW-basis from \cite{AA17}.

\bigbreak

\subsection{Future directions/applications} 
We expect that our methods can be applied to prove fgc property for finite-dimensional Hopf algebras beyond the setting of Theorem~\ref{thm:fingencoh-pointed}, for instance, pointed but with non-abelian group of group-like elements $G(H)$. We outline here the steps which will need to be taken to follow our general strategy.

Let $H$ be a finite-dimensional  Hopf algebra whose coradical $H_0$ is a Hopf subalgebra. Let $V$ be the infinitesimal braiding of $H$.
Then $\gr H \simeq R \# H_0$ where $R$ is a connected graded Hopf algebra in the category of Yetter-Drinfeld modules over $H_0$, $\ydho$, and $\toba(V)$ is a graded Hopf subalgebra of $R$.  To prove that $H$ has fgc following the scheme presented in \ref{sub:scheme} one would need to address these problems:
\begin{enumerate}[leftmargin=*,label=\rm{(\roman*)}]
\item\label{it:pbm-B(V)-fgc} Prove that $\toba(V)$ and $\toba(V^\#)$ have fgc.

\medbreak
\item\label{it:question-gendegone} Is $R =\toba(V)$? (in all known examples in characteristic 0 the answer is positive).
If not, prove that $R$ and $R^\#$ have fgc.

\medbreak
\item\label{it:pbm-invariants} Prove Lemma \ref{lema:braidedcommut-Noetherian} for any semisimple Hopf algebra (and not just for group algebras and their duals).
Together with \ref{it:question-gendegone} this would give that $\gr H$, $(\gr H)^\#$ have fgc.

\medbreak
\item\label{it:pbm-fgc-drinfeld-double} Extend Theorem \ref{thm:fingencoh-drinfeld-double} to prove that $D(\gr H)$ has fgc.
Even in the pointed case, we would need Lemma \ref{lema:braidedcommut-Noetherian} for $D(\ku G(H))$
to prove this conjectural extension.

\medbreak
\item\label{it:question-morita} Is $H$ a cocycle deformation of $\gr H$ or at least Morita equivalent to $\gr H$ as in \S \ref{subsec:morita}?
(in all known examples in characteristic 0 the answer is positive).
This would imply that $D(H)$, and a fortiori  $H$, have fgc since Theorem \ref{lema:frobenius-reciprocity} holds in general.
\end{enumerate} 

We also notice that a large part of this approach could be used in positive characteristic under appropriate assumptions, e.g. the coradical $H_0$
needs to be a semisimple Hopf subalgebra.

\bigbreak

\subsection{Organization of the paper}

Part \ref{part:hopf} starts with a recollection of facts on Hopf and Nichols algebras in Section \ref{sec:hopf}. 
Section \ref{sec:preliminaries-cohomology} contains several preliminary results on cohomology including the passage from 
the cohomology of $\toba(V)$ to the cohomology of  $\toba(V)\# \ku \Gamma$ and versions of the Evens Lemma 
and the May spectral sequence crucial for our arguments. Section \ref{sec:Anick}  presents the Anick resolution
and the reduction to Condition \ref{assumption:intro-combinatorial}. In the last Section \ref{subsec:drinfeld-double}
of this Part it is shown that the Drinfeld double of $\toba(V)\# \ku \Gamma$ has fgc provided that $\toba(V)$ has
via considerations of cohomology for twisted tensor products. 

Parts \ref{part:nichols} is devoted to the proof of Condition \ref{assumption:intro-combinatorial}.
Section \ref{sec:nichols-strategy} presents the strategy of the verification with proofs of technical Lemmas postponed to
Section \ref{sec:computational-lemmas}.  
We verify Condition \ref{assumption:intro-combinatorial} for
 finite-dimensional Nichols algebras of diagonal type belonging to  families with continuous parameter. 
We proceed case by case in Sections \ref{sec:classical}, \ref{sec:exceptional} and \ref{sec:discrete}
corresponding respectively to classical (Cartan, standard and super) types, exceptional (Cartan, standard and super) types, and 
Nichols algebras with the same root systems as the modular Lie algebras $\Bgl(4)$ and $\Brown(2)$.
The remaining Nichols algebras of diagonal type are dealt with in a separate publication \cite{AAPPW} of more computational nature.


\subsection{Conventions}\label{subsection:conventions}
For $\ell < \theta \in\N_0$, we set $\I_{\ell, \theta}=\{\ell, \ell +1,\dots,\theta\}$, $\I_\theta 
= \I_{1, \theta}$. 
Let $\G_N$ be the group of roots of unity of order $N$ in $\ku$ and $\G_N'$ the subset of primitive roots of order $N$;
$\G_{\infty} = \bigcup_{N\in \N} \G_N$ and $\G'_{\infty} = \G_{\infty} - \{1\}$.
If $L \in \N$ and $q\in \ku^{\times}$, then $(L)_q := \sum_{j=0}^{L-1}q^{j}$.

All  vector spaces, algebras and tensor products  are over $\ku$. We use $V^\#$ to denote the linear dual to a vector space $V$, $V^\# = \Hom_k(V,k)$. 

By abuse of notation, $\langle a_i: i\in I\rangle$ denotes either the group, the subgroup or the vector subspace generated by all $a_i$ for $i$ in an indexing set $I$,
the meaning being clear from the context. Instead, the subalgebra generated by all $a_i$ for $i \in I$ is denoted by $\ku \langle a_i: i\in I\rangle$.

If $A$ is an associative augmented algebra and $M$ is an $A$-module, then we set 
\begin{align*}
\coh^n(A, M) &= \Ext^n_A(\ku, M),& \coh(A, M) &= \oplus_{n\in \N_0} \coh^n(A, M).
\end{align*}
Then $\coh(A, \ku) = \oplus_{n\in \N_0} \coh^n(A, \ku)$ is isomorphic to the Hochschild cohomology
$\hoch(A, \ku) = \oplus_{n\in \N_0} \Ext^n_{A \otimes A^{\operatorname{op}}} (A, \ku)$ via an equivalence of bar complexes; see for example
\cite[(2.4.1)]{MPSW}.

\medbreak
Let $P_*(A)$ be the normalized bar resolution  of $\ku$ in the category of left	 $A$-modules
and let $\Omega^*(A) = \Hom_{A} (P_*(A), \ku)$, in particular $\Omega^n(A) = \Hom_{\ku} (A^{\otimes n}_+, \ku)$, 
where $A_+$ is the augmentation ideal.

Let $\widehat{\Gamma} = \Hom_{\rm ab}(\Gamma, \ku^\times)$ be the character group of an abelian group $\Gamma$.

\subsection*{Acknowledgements} 
J.~P.~ thanks Eric Friedlander for introducing her to the problem of finite generation and for many invaluable discussions about it. 
S.~W.~ and J.~P.~ thank Jon Carlson for bestowing his computer prowess and wisdom on them in the early stages of this project.
N.~A.~thanks Cris Negron and Sonia Natale for very useful conversations.
I.~A.~ thanks Leandro Vendramin for joint work on the program to compute root systems of Nichols algebras.  We also thank the 
referee for the careful reading of our work and very valuable comments. 

The authors are extremely grateful to the hospitality and support of the American Institute for Mathematics 
(AIM) via their SQuaREs program. N.~A.~ was supported by the NSF under Grant No. DMS-1440140, while he was in residence at the Mathematical Sciences Research Institute in Berkeley, California, in the Spring semester of 2020. 
J.~P.~ was supported by the NSF grants DMS-1501146, DMS-1901854 and Brian and Tiffinie Pang faculty fellowship.
S.~ W.~ was supported by the NSF grants DMS-1401016 and DMS-1665286.
The work of N.~A.~ and I.~A.~ was partially supported by CONICET and Secyt (UNC).

\vspace{0.5in}
\renewcommand\thepart{\Roman{part}}

\part{From cohomology of Nichols algebras to cohomology of Hopf algebras }\label{part:hopf}

\section{Finite-dimensional Hopf algebras}\label{sec:hopf}

\subsection{Morita equivalence of Hopf algebras}\label{subsec:morita}
In this subsection, no restrictions on the the field $\ku$ are needed.
Let $H$ be a finite-dimensional Hopf algebra. 
We refer to \cite{R-book} for the definitions of the Drinfeld double 
$D(H)$ of $H$ and of the (braided tensor) category $\ydh$ of Yetter-Drinfeld modules over $H$.
It is well-known that $\ydh$ is the Drinfeld center of the category of $H$-modules and that it is
braided tensor equivalent to the category of $D(H)$-modules.

Let $H'$ be another finite-dimensional Hopf algebra. 
Borrowing terminology from \cite{Mu1,ENO2},
we say that  $H$
and $H'$ are \emph{Morita equivalent}, denoted $H \mor H'$, if there 
is an isomorphism of quasitriangular Hopf algebras $D(H) \simeq D(H')$. 
This is not the same as Morita equivalent as algebras!

\begin{lemma}\label{lema:morita-dual} $H$ is Morita equivalent to $H'$ in the following cases:
\begin{enumerate}[leftmargin=*,label=\rm{(\alph*)}]
\item\label{item:morita-dual} $H' \simeq H^{\#}$, the dual  Hopf algebra.

\item\label{item:morita-twist} $H' \simeq H^{F}$ is a twist of $H$ \cite{Dr2,Re}, i.e.~there exists $F\in H \otimes H$ invertible  
such that $H^{F} = H$ as algebra and has the comultiplication $\Delta^F = F \Delta F^{-1}$.

\item\label{item:morita-cocycle} $H' \simeq H_{\sigma}$ is a cocycle deformation of $H$ \cite{DT}, i.e.~there exists an invertible 2-cocycle $\sigma: H \otimes H \to \ku$ 
such that $H_{\sigma} = H$ as coalgebra and has the multiplication $x \cdot_{\sigma} y 
= \sigma(x\_1 \otimes y\_1) x\_2 y\_2 \sigma^{-1}(x\_3 \otimes y\_3)$.
\end{enumerate}
\end{lemma}

\pf \ref{item:morita-dual} This follows from \cite[Proposition 2.2.1]{AGr} because of the identification of
the category of Yetter-Drinfeld modules over H with the category of representations of D(H) as mentioned in Subsection~\ref{subsec:morita}. 
\ref{item:morita-twist} follows since the categories of $H$ and $H^F$-modules 
are tensor equivalent \cite[p. 1422]{Dr2}.  Finally \ref{item:morita-cocycle} is a consequence of the preceding, as 
$(H_{\sigma})^{\#}\simeq (H^{\#})^F$
where $F = \sigma$ in $H^{\#} \otimes H^{\#}$.
\epf

\subsection{The role of Nichols algebras}\label{subsec:na-role} Even though our primary interest is in cohomology 
of finite dimensional Hopf algebras, the Nichols algebras which originated in independent work of Nichols and Woronowicz 
show up very naturally in our study.  Here we give a very brief account of the general approach to the classification of 
finite dimensional Hopf algebras over an algebraically closed field $\ku$ of characteristic 0, highlighting the importance 
of Nichols algebras in their structure. The reader can find all the precise definitions and details
 in the surveys \cite{AS-survey,A-leyva}.

Let $H$ be a Hopf algebra, $H_{0}$ its coradical \cite{R-book}
and $H_{[0]} = \ku \langle H_0 \rangle$ its Hopf coradical, a Hopf subalgebra of $H$ \cite{AC}.
The classification of finite-dimensional Hopf algebras  can be organized 
in four classes, according to the relative behavior of $H_{0}$ and $H_{[0]}$:

\begin{multicols}{2}
\begin{enumerate}[leftmargin=*,label=\rm{(\alph*)}]
\item\label{item:intro-class-1} $H = H_{0}$, i.e.~$H$ is cosemisimple.

\medbreak
\item\label{item:intro-class-2} $H = H_{[0]} \neq H_{0}$.

\medbreak
\item\label{item:intro-class-3} $H \neq H_{[0]} = H_{0}$.

\medbreak
\item\label{item:intro-class-4} $H \neq H_{[0]}  \neq H_{0}$.
\end{enumerate}
\end{multicols}
Hopf algebras in class \ref{item:intro-class-1} are semisimple by a theorem of Larson and Radford.
Albeit families of examples and some classification results in low dimension are known, no systematic 
approach to the classification is available. Being semisimple, they are not interesting for our cohomology 
explorations since cohomology simply vanishes in positive degrees.  

Similarly to \ref{item:intro-class-1}, the class  \ref{item:intro-class-2} has no proposed method to deal with the 
classification, but contrary to \ref{item:intro-class-1}, cohomology rings are far from trivial. Even though 
there is no classification, many examples are known; some of them have the fgc property by Theorem~\ref{thm:fingencoh-basic}.

Hopf algebras in classes \ref{item:intro-class-3} and \ref{item:intro-class-4} have interesting cohomology rings and 
this is where Nichols algebras become highly relevant. The approach we develop allows us to deal with class 
\ref{item:intro-class-3}  under some additional assumptions on $H_{0}$  though we expect that it can be applied 
more generally. The class \ref{item:intro-class-4} is not as rigidly structured as the class \ref{item:intro-class-3} and 
even though the Nichols algebras still play a central role we do not know  enough about the structure to make conclusions 
about cohomology yet. We will mention a little bit more about case \ref{item:intro-class-4} towards the end of this subsection.


Now let us consider case \ref{item:intro-class-3} which is the case of interest for this paper. 
If  $H$ is  in class \ref{item:intro-class-3}  then 
$H_0$ is a proper Hopf subalgebra. Let $\gr H$ be the  graded Hopf algebra associated to the coradical filtration of $H$; then
\begin{align}\label{eq:bosonization}
\gr H \simeq R \# H_0
\end{align}
where $R = \oplus_{n\in \N_0} R^n$ is a connected
graded Hopf algebra in the braided monoidal category $\ydho$, called the \emph{diagram} of $H$.
We also say that $H$ is a lifting of $R$, or of $R \# H_0$. 
Then $R$ is coradically graded, hence its subalgebra generated by $V := R^1$ is isomorphic to the Nichols algebra $\toba(V)$; see \cite{AS-survey} for details.
The braided vector space $V$ is an important invariant of $H$ called its \emph{infinitesimal braiding}.

We make the following {\it additional assumptions} on $H$.  Assume that $H$ is pointed, that is $H_0 = k\Gamma$; also assume that $\Gamma$ is abelian. In that case the infinitesimal braiding $V$ is of {\it diagonal type} (see Subsection~\ref{subsec:na-diag}). Then the following two properties hold: 

\begin{gather}\label{eq:gen-deg1}
R \simeq \toba(V),
\\\label{eq:cocycle-def}
\exists \sigma: \gr H \otimes \gr H \to \ku \text{ such that } (\gr H)_{\sigma} \simeq H.
\end{gather}

\vspace{0.1in}
The first property  \eqref{eq:gen-deg1} holds by \cite[Theorem 2]{Ang-crelle}, see also \cite[Theorem 5.5]{AS-ann}; notice that the proof uses the
classification in \cite{H-classif} and the main result on convex orders from \cite{A-jems}. 
The second property \eqref{eq:cocycle-def} holds  by \cite[Theorem 1.1]{AG}, based on previous studies
of the lifting question and the explicit relations from \cite[Theorem 3.1]{Ang-crelle},  that again uses \cite{H-classif,A-jems}. 
Summarizing, we get the following structure theorem:

\begin{theorem}\label{th:summary}
Let $H$ be a finite-dimensional pointed Hopf algebra  such that the group of group-like elements $G(H)$ is abelian.
Then $H$ is a cocycle deformation of the bosonization of a Nichols algebra of diagonal type: 
$H \simeq \left(\toba(V) \# \ku G(H)\right)_{\sigma}$. Hence $H \mor \gr H$. \qed
\end{theorem}
The theorem implies that to verify fgc for $H$ it suffices to verify it for the Nichols algebra $\toba(V)$ and for the Drinfeld double $D(\gr H)$ which is what we do in Theorem \ref{thm:ppalrealiz-Kss} and Corollary \ref{cor:morita}.

We point out that \eqref{eq:gen-deg1} and \eqref{eq:cocycle-def} have been verified in most known 
examples in class \ref{item:intro-class-3} beyond pointed Hopf algebras with abelian group of group-like 
elements and are expected to hold generally.  Hence, generalizing Theorem \ref{thm:ppalrealiz-Kss} and 
Corollary \ref{cor:morita} to the other cases in class \ref{item:intro-class-3} one should be able to reduce the 
question of fgc for $H$ again to the same question for $\toba(V)$ and the Drinfeld double for $\gr H$.  
We do not pursue this direction in this paper but we do expect our methods to be fruitful in these cases as 
well, potentially allowing one to finish off proving fgc property for all Hopf algebras in class \ref{item:intro-class-3}.

\medbreak 
Finally, let $H$ be in class \ref{item:intro-class-4}. Then one considers
the  graded Hopf algebra $\gr H$ associated to the standard filtration of $H$ \cite{AC}; again
\begin{align}\label{eq:bosonization2}
\gr H \simeq R \# H_{[0]}
\end{align}
where $R = \oplus_{n\in \N_0} R^n$ is a connected
graded Hopf algebra in the braided monoidal category $\ydhoo$. But
it is not known whether $R$ is coradically graded, or
its subalgebra $R'$ generated by $V := R^1$ is isomorphic to the Nichols algebra $\toba(V)$. 
We do know that $\toba(V)$ is a quotient of $R'$ but the present approach does not allow to reduce
the question of finitely generated cohomology for $H$ to the analogous question 
for $\toba(V)$.  So fgc property for class \ref{item:intro-class-4} is completely open. 

\subsection{Nichols algebras of diagonal type}\label{subsec:na-diag}
Since finite-dimensional
Nichols algebras of diagonal type are central in this paper, we present here the features more
relevant for our goals
and refer to \cite{AA17} for an exposition.
The input is a  matrix of non-zero scalars $\bq = (q_{ij})_{i, j \in \I}$ where $\I = \I_{\theta}$,  $\theta \in \N$.
To this datum we attach a braided vector space \emph{of diagonal type}
$V$ with a basis $(x_{i})_{i \in \I}$ and braiding $c^{\bq} \in GL(V \otimes V)$ given
by 
\begin{align*}
c^{\bq}(x_{i} \otimes x_{j}) &= q_{ij} x_{j} \otimes x_{i},& i, j &\in \I. 
\end{align*}
The corresponding  Nichols algebra is a graded connected algebra with strong properties
denoted here mostly as $\toba_{\bq}$ instead of $\toba(V)$. For these Nichols algebras substantial information is available.

\subsubsection{Dynkin diagrams and positive roots}
We codify as usual the matrix $\bq$ 
in a (generalized) Dynkin diagram $\Dc$ with vertices numbered by $\I$ 
and labelled with $q_{ii}$,
while two different vertices $i$ and $j$ are joined by an edge only if $\widetilde{q}_{ij} := q_{ij}q_{ji} \neq 1$
in which case the edge is labelled with $\widetilde{q}_{ij}$:
\begin{align}
\xymatrix{ \overset{q_{ii} }{\underset{i}{\circ}} \ar  @{-}[r]^{\widetilde{q}_{ij}}  & \overset{q_{jj}}{\underset{j}{\circ}} }.
\end{align} 
Two different matrices with the same Dynkin diagram are called \emph{twist-equivalent} \cite{AS-survey}.

\medbreak
The Nichols algebra $\toba_{\bq}$ has a very useful $\N_0^{\I}$-grading 
given by the rule $\deg x_i = \alpha_i$, $i\in \I$,
where $(\alpha_i)_{i\in \I}$ is the canonical basis of $\Z^{\theta}$.
By \cite[Theorem 2.2]{Kh},  $\toba_{\bq}$ has  a  PBW basis
\begin{align*}
B &= \big\{s_1^{e_1}\dots s_t^{e_t}: t \in \N_0,\ s_i \in S,  s_1>\dots >s_t, 
  0<e_i<h(s_i) \big\}.
\end{align*}
where $S$ is an ordered set of $\N_0^{\I}$-homogeneous elements and  $h: S \mapsto \N \cup \{ \infty \}$ 
is a function called the \emph{height}.
The following set does not depend on the choice of $B$:
\begin{align*}
\varDelta_+^{\bq} := \{\deg s: s \in S \} \subset \N_0^{\I}.
\end{align*}
Occasionally we set $\varDelta_+^{V} = \varDelta_+^{\bq}$. 
The elements of $\varDelta_+^{\bq}$ are called the (positive) roots of $\toba_{\bq}$.
We assume from now on  that
\begin{align*}
\dim \toba_{\bq} < \infty.
\end{align*}
Then $\varDelta_+^{\bq}$ is a finite set and  the map $S \to \varDelta_+^{\bq}$, $s \mapsto \deg s$, is bijective.
Also $\varDelta_+^{\bq}$ admits a convex (total) order in the sense
\begin{align*}
\alpha, \beta, \alpha + \beta \in \varDelta_+^{\bq}, \alpha < \beta \implies \alpha < \alpha + \beta < \beta.
\end{align*}
See \cite{A-jems}. The convex order is not unique; in the case-by-case analysis below we use that of \cite{AA17} 
except when a more suitable choice is possible that we mention explicitly. 

\medbreak
A connected component of $\Dc$ is a subset $J$ of $\I$ such that the matrix $\bq_J = (q_{ij})_{i, j \in J}$
gives rise to a connected Dynkin subdiagram $\Dc_J$ of $\Dc$ and is maximal with this property.
Let $\Xc$ be the set of connected components of $\Dc$.  For $J \subset  \I$, we say that a positive root 
(an element of $\varDelta_+^{\bq})$ is {\it supported} on $J$ if it belongs to the subset $\N_0^{J} \subset \N_0^{\I}$ 
where the containment is induced by the inclusion $J \subset \I$. We say that a root has {\it full support} if it's not supported on any proper subset of $\I$. 
If $J \subset  \I$, we identify  $\varDelta_+^{\bq_J}$ with
the subset $\varDelta_+^{J}$ of $\varDelta_+^{\bq}$ of roots with support in $J$. We also denote by $V_J$ the subspace of $V$
spanned by $(x_j)_{j\in J}$.
 By a result of Gra\~na, we have
\begin{align}\label{eq:connected-components}
\toba_{\bq} &\simeq \underline{\bigotimes}_{J\in \Xc} \toba_{\bq_J}, &
\varDelta_+^{\bq} &= \coprod_{J\in \Xc} \varDelta_+^{J}.
\end{align}
Here $\underline{\otimes}$ means the braided tensor product of algebras.

\subsubsection{Weyl equivalence}\label{sub:weyl} 

Let $\mathfrak{q}$ be such that $\dim \mathscr B_{\mathfrak{q}} < \infty$ and let $G$ be a finite abelian group 
such that $(V, c^{\mathfrak q})$ is realized in $\yd{\ku G}$. 
Given $i\in \mathbb I$,
one defines a matrix $\rho_i(\mathfrak{q})$ by a precise rule \cite{H-Weyl grp} or \cite[(2.25)]{AA17}. 
Then, although the matrices $\mathfrak{q}$ and $\rho_i(\mathfrak{q})$ might be quite different, 
there is an algebra isomorphism
$T_i :D(\mathscr B_{\mathfrak{q}} \# \ku G) \to D(\mathscr B_{\rho_i(\mathfrak{q})} \# \ku G)$, see \cite{H-Weyl grp}.
The assignments $\mathfrak{q} \mapsto \rho_i(\mathfrak{q})$, for all $i\in \mathbb I$, 
generate the so-called Weyl groupoid of $\mathfrak{q}$ \cite{H-Weyl grp}.
Two matrices
$\mathfrak{q}$ and $\mathfrak{q}'$ are \emph{Weyl-equivalent} if there exists a sequence $i_1, \dots, i_N \in \mathbb I$
such that $\mathfrak{q}' = \rho_{i_1} \dots \rho_{i_s}(\mathfrak{q})$. In this case one also says that the 
braided vector spaces $(V, c^{\mathfrak q})$  and $(V', c^{{\mathfrak q}'})$ are Weyl-equivalent.
Hence there is an algebra isomorphism
$D(\mathscr B_{\mathfrak{q}} \# \ku G) \to D(\mathscr B_{\mathfrak{q}'} \# \ku G)$.
See \cite[Section 2.6]{AA17} for an exposition.

\subsubsection{Classification}
The classification of the matrices $\bq$ such that $\dim \toba_{\bq} < \infty$ was achieved in \cite{H-classif} (the result is slightly more general).
By the preceding discussion we may assume that $\Dc$ is connected.
As in \cite{AA17} we organize the classification in 4 types:
\begin{multicols}{2}
\begin{enumerate}[leftmargin=*,label=\rm{(\alph*)}]
\item\label{item:nichols-type-cartan} Cartan type.

\medbreak
\item\label{item:nichols-type-super} Super type.

\medbreak
\item\label{item:nichols-type-modular} Modular type.

\medbreak
\item\label{item:nichols-type-ufo} UFO type.
\end{enumerate}
\end{multicols}

The type refers to the connection with different parts of Lie theory, see \emph{loc.~cit.}
We shall check Condition \ref{assumption:intro-combinatorial} for each entry of the classification of \cite{H-classif}.

\subsubsection{Root vectors}
For brevity, we set 
\begin{align}\label{eq:xij}
x_{ij} &= \ad_c x_i (x_j),& i&\neq j \in \I;
\end{align}
more generally, the iterated braided commutators are
\begin{align}\label{eq:iterated}
x_{i_1i_2\cdots i_k}& :=(\ad_c x_{i_1})\cdots(\ad_c x_{i_{k-1}})\, (x_{i_k}),& &i_1, i_2, \cdots, i_k\in\I.
\end{align}
In particular, we will use repeatedly the following further abbreviation:
\begin{align}\label{eq:roots-Atheta}
x_{(k \, l)} &:= x_{k\,(k+1)\, (k+2) \dots l},& &k < l.
\end{align}

Using a fixed convex order, we define the root vector
$x_{\alpha} \in \toba_{\bq}$ for every $\alpha \in \varDelta_+^{\bq}$ 
as iterated braided commutators proceeding case-by-case, see \cite{AA17}.

\medbreak
For those $\bq$ with $\dim \toba_{\bq} < \infty$, the defining relations of $\toba_{\bq}$ were given in \cite{A-jems,Ang-crelle}, again see \cite{AA17}, but these are not needed in this paper. Instead, we use systematically that for $\alpha < \beta \in \varDelta_+^{\bq}$ and a suitable defined $q_{\alpha \beta}\in \ku^{\times}$, we have
\begin{align}\label{eq:convex-order-relations}
x_{\alpha} x_{\beta} - q_{\alpha \beta} x_{\beta} x_{\alpha}  \in \sum_{\alpha< \gamma_1 \leq \gamma_2 \leq \dots \leq \gamma_t < \beta \in \varDelta_+^{\bq}}  \ku x_{\gamma_1}x_{\gamma_2} \dots x_{\gamma_t}.
\end{align}

See \cite[Theorem 4.9]{A-jems}, inspired by the original work of \cite{DCK}.

\subsection{Realizations}\label{se:realizations}
Let $\Gamma$ be a finite abelian group. A Yetter-Drinfeld module $V$ over $\ku\Gamma$
is determined by families $(g_i)_{i \in \I_\theta}$ of elements of $\Gamma$ and characters 
$(\chi_i)_{i \in \I_\theta}$ in $\widehat{\Gamma}$. Then $V$ is a braided vector space of diagonal type with braiding matrix
$\bq = (q_{ij})_{i, j \in \I}$ with respect to  a basis $(x_{i})_{i \in \I}$,
i.e.~$c(x_{i} \otimes x_{j}) = q_{ij} x_{j} \otimes x_{i}$, $i, j \in \I$, where $q_{ij} = \chi_j(g_i)$. 
That is, the same braided vector space $V$ with braiding matrix
$\bq = (q_{ij})_{i, j \in \I}$ can be realized in many ways over many $\Gamma$. Even more, it can be realized 
over other Hopf algebras than group algebras over abelian groups. To discuss the possible realizations we need the notion of a YD-pair.

\medbreak
Let $H$ be a Hopf algebra.
A pair $(g, \chi) \in G(H) \times \Hom_{\text{alg}}(H, \ku)$ is a \emph{YD-pair} for $H$ if
\begin{align}\label{eq:yd-pair}
\chi(h)\,g  &= \chi(h_{(2)}) h_{(1)}\, g\, \mathcal S(h_{(3)}),& h&\in H.
\end{align}
When this is the case, $g\in Z(G(H))$; 
also $\ku_g^{\chi} = \ku$ with action and coaction given by $\chi $ and $g$ respectively, is in $\ydh$.
YD-pairs classify the 1-dimensional objects in $\ydh$. Note that, if $\dim H<\infty$, then $(g, \chi)$ is a YD-pair for $H$ if and only if
$(\chi, g)$ is a YD-pair for $H^\#$.

\begin{definition}
Let $V$ be a braided vector space of diagonal type with braiding matrix
$\bq = (q_{ij})_{i, j \in \I}$.
A \emph{principal realization} of $V$ over $H$ is a family 
$(g_i, \chi_i)_{i \in \I}$ of YD-pairs such that
$\chi_j(g_i)=q_{ij}$, $i, j \in \I$, so that $V\in\ydh$ 
via $\ku x_i \simeq \ku_{g_i}^{\chi_i}$, and the braiding $c$ is the categorical one from $\ydh$. 
\end{definition}

Given a principal realization of $V$ over $H$, we have
$\Gamma :=\langle g_1,\dots,g_\theta\rangle \leq Z(G(H))$;
hence we can also realize $V$ as an object in $\ydkg$.

\begin{example}\label{exa:ppal-gps}
If $\Gamma$ is a finite group, then the YD-pairs of $H = \ku \Gamma$ are of the form 
$(g, \chi) \in Z(\Gamma) \times \Hom_{\text{grp}}(\Gamma, \ku^{\times})$. For example, for $\Gamma = GL_n(\mathbb F_p)$, the YD-pairs are 
$(\operatorname{diag}(t), \varphi \det^n)$, where $t\in \mathbb F_p^{\times}$ and $\varphi \in \widehat{\mathbb F_p^{\times}}$.
\end{example}

\begin{example}\label{exa:ppal-ab-ext}
Not every realization is principal: if $g \in Z(\Gamma)$ and $\rho \in \operatorname{Irrep} \Gamma$ with $\dim \rho = d >1$ and 
$\rho (g) = \zeta \id$, then the simple Yetter-Drinfeld  module $M(g, \rho)$ \cite[Example 24]{A-leyva} is
a braided vector space of diagonal type with braiding matrix
$(q_{ij})_{i, j \in \I_d}$ where $q_{ij} = \zeta$ for all $i,j$.
Other examples arise from simple Yetter-Drinfeld  modules $M(g, \rho)$ such that the elements in the conjugacy class of $g$ commute with each other.
\end{example}

\subsection{The Drinfeld double of a bosonization}\label{se:drinfeld}

Recall that $\car \ku =0$.
Let $\hopfdos$ be a Hopf algebra whose coradical $\hopfdos_0$ is a Hopf subalgebra
and let $\hopfuno$ be another Hopf subalgebra of $\hopfdos$.
Then $\hopfuno_0 = \hopfdos_0  \cap \hopfuno$ by \cite[4.2.2]{R-book} and this  is a Hopf subalgebra of $\hopfuno$.
By \cite[4.4.11]{R-book} we have an injective map of graded Hopf algebras $\gr \hopfuno \hookrightarrow \gr \hopfdos$.
Let $R$ and $S$ be the diagrams of $\hopfuno$ and $\hopfdos$ respectively, see \eqref{eq:bosonization}. 
Hence we have an injective map of graded  \emph{braided} Hopf algebras $R \hookrightarrow S$.

\smallbreak
Let $\hopfdos$ be a finite-dimensional Hopf algebra. The Drinfeld double of $\hopfdos$, denoted by $D(\hopfdos)$, is a Hopf algebra 
whose underlying coalgebra is $ \hopfdos\ot \hopfdos^{\# \operatorname{op}}$. Let  $\big\langle - , - \big\rangle: \hopfdos\ot \hopfdos^\# \to \ku$ denote the evaluation map. 
The multiplication on $D(\hopfdos)$ is given by the following formula:
\begin{align*}
(h\bowtie f)(h'\bowtie f') &= \big\langle f_{(1)},h'_{(1)}\big\rangle \big\langle f_{(3)}, \mathcal S(h'_{(3)}) \big\rangle (hh'_{(2)}\bowtie f'f_{(2)}),&
f, f'&\in  \hopfdos^\# \quad h, h'\in \hopfdos,
\end{align*}
where 
$h\bowtie f:= h \otimes f$ in $D(\hopfdos)$ and 
$fr = m(f\ot r)$ is the multiplication in $\hopfdos^\#$ rather than in $\hopfdos^{\# \operatorname{op}}$.

\smallbreak
Let $K$ be a semisimple Hopf algebra, hence cosemisimple by the Larson-Radford theorem.
Let $V \in \ydk$ with $\dim \toba(V) < \infty$ and $H = \toba(V) \# K$. 
Then $V^\# \in \ydkd$ appropriately and 
$H^\# \simeq  \toba(V^\#) \# K^\#$.  
Since $D(H) \simeq {H^\#}^{\cop} \otimes H$ as coalgebras,
the coradical of $H$, respectively $H^\#$, can be identified with $K$, respectively $K^\#$. 
We identify $D(K)$ with a Hopf subalgebra of $D(H)$ in a natural way.

The following result generalizes, with an analogous proof, Theorem 2.5 in \cite{Beattie}.

\begin{prop}\label{prop:double-lifting}
The Drinfeld double $D(H)$ is a lifting of a Nichols algebra $\toba(W)$ where $W = V \oplus V^\#$
and $V$ braided commutes with $V^\#$.
\end{prop}

\pf First, the coradical $D(H)_0$ of $D(H)$ equals $D(K)$; this follows from \cite[4.1.8]{R-book}.
Hence the coradical filtration of $D(H)$ is a Hopf algebra filtration and $\gr D(H) \simeq R\# D(K)$,
where $R = \oplus_{n\ge 0} R^n$ is the diagram of $D(H)$. Let $W = R^1$. 
By the preceding paragraph applied to $\hopfdos = D(H)$ and either $\hopfuno =H$ or $\hopfuno = H^\#$,
we have morphisms of braided vector spaces $V \hookrightarrow W$ and $V^\# \hookrightarrow W$;
we have $V \oplus V^\# \hookrightarrow W$ by comparing the $D(K)$-comodule structures. Recall that
$\dim \toba(V) \dim \toba(V^\#) \leq \dim \toba(V \oplus V^\#)$ and the equality holds iff   
$V$ and $V^\#$ braided commute \cite[Theorem 2.2]{Grana}. Then 
\begin{align*}
\dim D(H) &= \dim \toba(V) \dim K \dim \toba(V^\#) \dim K^\# 
\\ &\leq \dim \toba(W) \dim D(K) \leq \dim R \dim D(K) = \dim D(H),
\end{align*}
hence $R = \toba(W) = \toba(V \oplus V^\#)$ and $V$ and $V^\#$ braided commute.
\epf

Assume next that $K = \ku \Gamma$ where $\Gamma$ is a finite abelian group; recall that $\widehat{\Gamma}$ is the group
of characters of $\Gamma$.   Then
\begin{align*} 
\ku\widehat{\Gamma} &\cong (\ku\Gamma)^\#,\\
D(\ku\Gamma) &\cong \ku (\Gamma \times \widehat{\Gamma}) \cong \ku \Gamma \otimes (k\Gamma)^\#.
 \end{align*} 
Let  $(g_i)_{i \in \I_\theta}$ and   
$(\chi_i)_{i \in \I_\theta}$ be (dual) generating families in $\Gamma$ and $\widehat{\Gamma}$ respectively.
Let $V\in \yd{\ku\Gamma}$ with a basis $(x_{i})_{i \in \I}$ 
such that the action and coaction of $\Gamma$ on $x_i$ are given by $\chi_i$ and $g_i$ respectively, $i\in \I$.
Assume that $\toba(V)$ has finite dimension and let $H = \toba(V) \# \ku \Gamma$.
Let $(y_i)_{i\in \I}$ be the basis of $V^{\#}$ dual to $(x_{i})_{i \in \I}$.
Then  ${H^\#}^{\cop} \simeq \toba(V^{\#}) \# \ku \widehat{\Gamma}$ where 
the action and coaction of $\widehat{\Gamma}$ on $y_i$ are given by $g_i$ and $\chi_i^{-1}$ respectively,  $i\in \I$.
Also $W = V \oplus V^\#$ can be realized in $\yd{\ku(\Gamma \times \widehat{\Gamma})}$ extending these structures. 
See \cite[Theorem 2.5]{Beattie} for details.
Let $\mathcal{I} (V) = \ker (T(V) \to \toba(V))$ be the ideal of defining relations of the Nichols algebra $\toba(V)$. The following statement is well-known.

\begin{prop}\label{prop:double-relations}  $D(H)$ is isomorphic to the quotient of $T(W) \# \ku(\Gamma \times \widehat{\Gamma})$ by the ideal generated by 
$\mathcal{I} (V) $, $\mathcal{I} (V^\#) $ and the relations
\begin{align}\label{eq:double-relations}
x_{i} y_{j} - \chi_j^{-1}(g_i) y_{j} x_{i} &= \delta_{i,j}  \chi_i^{-1}(g_i) \left(1 - g_i \chi_i \right), & i,j &\in \I. 
\end{align}
\end{prop}

\noindent\emph{Outline of the proof.} 
By the preceding discussion there is a morphism of Hopf algebras $T(W) \# \ku(\Gamma \times \widehat{\Gamma}) \to D(H)$
whose kernel $\mathcal J$ contains  $\mathcal{I} (V) $, $\mathcal{I} (V^\#) $ and the relations \eqref{eq:double-relations}, see the proof of 
\cite[Theorem 2.5]{Beattie}. The induced map $T(W) \# \ku(\Gamma \times \widehat{\Gamma}) / \mathcal J \to D(H)$
is clearly surjective and preserves the coradical filtration. Since the associated graded map is injective, the claim follows. \qed

\section{ Cohomology}\label{sec:preliminaries-cohomology}

\subsection{Invariants}\label{subsec:smash}
Let $H$ be a Hopf algebra and $A$ an $H$-module algebra.
The ring of invariants is the  subalgebra
\begin{align*}
A^H = \{x \in A: h\cdot x = \varepsilon_H(h) x\quad \forall h \in H \}.
\end{align*}

\bigbreak
Let now $K$ be a semisimple, hence finite-dimensional, Hopf algebra. Let $t \in K$ be a normalized integral, that is
$kt = \varepsilon(k)t = tk$ for all $k \in K$ and $\varepsilon(t) = 1$.  
In other words $t$ is a left, hence, right, integral \cite[2.2.4]{Mo}.

Let $A$ be a $K$-module algebra and let $\Rc: A \to A$ be the Reynolds operator
$\Rc(x) = t\cdot x$, $x \in A$. Then

\begin{itemize}[leftmargin=*] \renewcommand{\labelitemi}{$\circ$}
\item The Reynolds operator is a projector,  $\Rc^2 = \Rc$, and $\im \Rc = A^K$.

\medbreak
\item $\Rc$ is a morphism of $K$-modules.

\medbreak
\item $\Rc$ is a morphism of $R^K$-bimodules: $\Rc(xyz) = x\Rc(y)z$, for $x,z \in A^K$ and $y \in A$. 
\end{itemize}

\medbreak
The following result,  a variation of a classical argument by Hilbert, is well-known.
\begin{lemma}\label{lema:hilbert-invariants}
Let $K$ be a semisimple Hopf algebra. 
Let $A = \oplus_{n\in \N_0} A^n$ be a graded $K$-module algebra that 
is connected and (right) Noetherian.
Let $M$ be a finitely generated $A\# K$-module.
Then $A^K$ is finitely generated and $M^K$ is a finitely generated
$A^K$-module. 
\end{lemma}

\pf Let $I = A(A^K)_+ $ be the left ideal of $A$ generated by the augmentation ideal of $A^K$.
Since $A$ is Noetherian, $I$ is finitely generated; we may assume that $I = \langle f_1, \dots, f_M\rangle$,
where $f_i \in A^K$ is homogeneous of degree $d_i$. We claim that $A^K  = \ku\langle f_1, \dots, f_M\rangle$.
For this we shall prove that any $f\in A^K$ homogeneous of degree $d$ belongs to $\ku\langle f_1, \dots, f_M\rangle$.
If $d =0$, this follows by connectedness. If $d > 0$, then we may write $f = \sum_{i} a_i f_i$ with $a_i$ either 0 or else  homogeneous of degree $d - d_i$. Then $f = \Rc(f) = \sum_{i} \Rc(a_i) f_i$ and $\Rc(a_i) \in \ku\langle f_1, \dots, f_M\rangle$
by the recursive hypothesis, so $f  \in \ku\langle f_1, \dots, f_M\rangle$.

For the module statement, note that 
the hypothesis of~\cite[Theorem~4.4.2]{Mo} holds; namely the map denoted there $\hat t$ is the Reynolds operator $\Rc$.
Hence $A$ is a right Noetherian $A^K$-module,
and thus finitely generated over $A^K$.
Thus $M$ is a Noetherian $A^K$-module.
Since $M^K$ is an $A^K$-submodule, it is also Noetherian, therefore finitely generated over $A^K$.
\epf

If $R$ is an $H$-module algebra, then $H$ acts on $\Omega^*(R)$ via the comultiplication and the antipode,
and a fortiori on $\coh(R, \ku)$.
The following proposition is well-known.

\begin{prop} 
\label{prop:SV}
Let $K$ be a semisimple Hopf algebra and let $R$ be a finite-dimensional $K$-module algebra.
Let $M$ be an $R\# K$-module.
Then  
\begin{align*}
\coh(R \# K, \ku) &\simeq \coh(R, \ku)^K, & 
 \coh(R\# K,M) &\simeq  \coh (R,M)^K,
 \end{align*} 
 and the action of 
$\coh(R\# K , \ku)$ on $\coh(R\# K,M)$ is precisely that induced
by the isomorphisms and the action of $\coh(R, \ku)$ on $\coh(R,M)$.
\end{prop}

\pf
This is well-known; 
the first isomorphism is for example~\cite[Theorem 2.17]{stefan-vay}.
In case $K$ is a group algebra, the relevant spectral sequence
is the Lyndon-Hochschild-Serre spectral sequence~\cite[\S\S7.2, 7.3]{evens}
which collapses since $K$ is semisimple.
\epf

\medbreak
Now we pass to algebras in $\ydh$. 
An algebra $A$ in $\ydh$ is \emph{braided commutative} if the multiplication $m_A$ satisfies
$m_A = m_A c_{A, A}$, that is 
\begin{align}\label{eq:braided-comm}
xy &= (x\_{-1} \cdot y) x\_{0},& x,y&\in A.
\end{align}
If $A$ is braided commutative, then $A^H$ is central in $A$.
We elaborate on an idea of \cite{MPSW}; for this we do not need the commutativity of $\Gamma$.

\begin{lemma}\label{lema:braidedcommut-Noetherian}
Let $\Gamma$ be a finite  group and 
let $A$ be a braided commutative algebra either in $\ydkg$ or in $\ydkgd$. Assume that
$A$ is finitely generated (as an algebra).
Then $A$ is Noetherian.
\end{lemma}

\pf Let $N$ be the exponent of $\Gamma$.  We deal first with $\ydkg$.
As an object in $\ydkg$, $A$ is $\Gamma$-graded: $A = \oplus_{g\in \Gamma} A_g$. Thus, if
$A = \ku\langle f_1, \dots, f_M\rangle$, then we may assume that each $f_i$ belongs to $A_{g_i}$ for some 
$g_i\in \Gamma$.  Then $f_i^N \in A_{g_i^N} = A_e$.
Since $A$ is braided commutative,   
$f_i^N f_j = (g_i^N \cdot f_j) f_i^N = f_j f_i^N$ for all $i,j$. 
Then $B = \ku\langle f_1^N, \dots, f_M^N \rangle$ is a central subalgebra of $A$ and  is
Noetherian by Hilbert's Basissatz. Now $A$ is a finitely generated $B$-module, actually  
$A = \sum_{0 \le a_i \le N} B\, f_1^{a_1} \dots, f_M^{a_M}$.
Thus $A$ is a Noetherian $B$-module hence a  Noetherian algebra.
We deal next with $\ydkgd$. Since $H= \ku^{\Gamma}$ has a basis of idempotents $\delta_g$, $g \in \Gamma$, 
again $A$ is $\Gamma$-graded: $A = \oplus_{g\in \Gamma} A_g$ where $A_g = \delta_gA$. 
Thus, if
$A = \ku\langle f_1, \dots, f_M\rangle$, with each $f_i \in A_{g_i}$ for some 
$g_i\in \Gamma$, then $f_i^N \in A_{g_i^N} = A_e = \delta_eA$. But $\delta_e$ is the integral of $\ku^{\Gamma}$, thus again 
$f_i^N \in A^H$ is central. Then we proceed as previously.
\epf

We wonder whether any finitely generated braided commutative algebra is Noetherian.
We need the following result from \cite{MPSW}.

\begin{prop} \cite[Corollary 3.13]{MPSW} \label{prop:braided-commutative}
Let $H$ be a  Hopf algebra and let $R$ be a bialgebra in $\ydh$. Assume that either $H$
or $R$ is finite-dimensional. Then the (opposite of) the 
Hochschild cohomology $\hoch(R, \ku)$  is a  braided commutative graded algebra
 in $\ydh$. \qed
\end{prop}

Actually \cite[Theorem 3.12]{MPSW} gives more: the claim is true if $R$ is a bialgebra in an abelian braided monoidal category 
$\Cc$ where the needed hom-objects exist.

\begin{theorem} \label{th:RtoRsmashH}
Let $\Gamma$ be a finite group and let $R$ be a finite-dimensional Hopf algebra in $\ydkg$.
f $R$ has fgc, then so does $R\# \ku \Gamma $.
\end{theorem}

\pf By Proposition \ref{prop:braided-commutative}, Lemma \ref{lema:braidedcommut-Noetherian} and the hypothesis,
$\coh(R, \ku)$ is Noetherian. Then $\coh(R, \ku)^{\Gamma}$ is finitely generated by Lemma \ref{lema:hilbert-invariants}.
By Proposition \ref{prop:SV}, $\coh(R \# \ku \Gamma, \ku) \simeq \coh(R, \ku)^{\Gamma}$ is finitely generated.
We next prove: If $M$ is a finitely generated module, then
$\coh(R\# \ku \Gamma , M)$ is finitely generated as an
$\coh(R\# \ku \Gamma , \ku)$-module. 
For this, we may induct on the length of the composition series
of $M$, and so it suffices to prove it in case $M$ is simple.
Let $R_+$ denote the augmentation ideal of $R$. Note that $R_+ M$ is
an $R\# \ku \Gamma$-submodule of $M$ and therefore $R_+ M = 0$ (by Nakayama),
that is, $M | _R$ is a trivial $R$-module.
We conclude that $\coh(R, M)$ is finitely generated as an $\coh(R,\ku)$-module.
By Lemma~\ref{lema:hilbert-invariants}, it follows that 
$\coh(R,M)^K$ is finitely generated over $\coh(R,\ku)^K$.
\epf

We are ready for one of our main results.

\begin{theorem} \label{thm:ppalrealiz-Kss}
Let $V$ be a braided vector space of diagonal type such that 
\begin{enumerate}[leftmargin=*,label=\rm{(\alph*)}]
\item\label{item:ppalrealiz-Kss1} the Nichols algebra $\toba(V)$ is finite-dimensional,

\item\label{item:ppalrealiz-Kss2} $V$ is realizable over a finite abelian group,

\item\label{item:ppalrealiz-Kss3} $\coh(\toba(V), \ku)$ is finitely generated.
\end{enumerate}

Let $K$ be a semisimple Hopf algebra and assume that  $V$ is realizable over $K$. 
Then $\toba(V) \# K$  has fgc.
\end{theorem}

\pf The proof is the same as for the previous result. By \ref{item:ppalrealiz-Kss2}, \ref{item:ppalrealiz-Kss3},
Proposition \ref{prop:braided-commutative} and Lemma \ref{lema:braidedcommut-Noetherian}, 
$\coh(\toba(V), \ku)$ is Noetherian. Then $\coh(\toba(V), \ku)^{K}$ is finitely generated by Lemma \ref{lema:hilbert-invariants}.
By Proposition \ref{prop:SV} and \ref{item:ppalrealiz-Kss1}, $\coh(\toba(V) \# K, \ku) \simeq \coh(\toba(V), \ku)^{K}$ is finitely generated.
The proof of the module statement is similar.
\epf

\begin{observation}
If $V$ admits a  principal realization over $K$, then \ref{item:ppalrealiz-Kss2} holds.
Notice that \ref{item:ppalrealiz-Kss1} does not imply \ref{item:ppalrealiz-Kss2}: take $V$ of dimension 2 with 
braiding matrix $\begin{pmatrix}
q_{11} & q_{12}  \\ q_{12}^{-1} &q_{22}
\end{pmatrix}$ where $q_{11} \in \G_M'$, $q_{12} \notin \G_{\infty}$, $q_{22} \in \G_N'$, $N, M > 1$.
However we do not know if $V$ being realizable over $K$ semisimple implies \ref{item:ppalrealiz-Kss2}.
\end{observation}

\subsection{Subalgebras, extensions}\label{subsec:subalgebras}

\begin{theorem}\label{lema:frobenius-reciprocity} Let $R$ be an 
augmented subalgebra of a finite-dimensional 
augmented algebra $\Ac$, over which $\Ac$ is projective as
a right $R$-module under multiplication.
If $\Ac$ has fgc, then so does $R$.
\end{theorem}

\begin{proof}
By the right module version of the 
Eckmann-Shapiro Lemma~\cite[Corollary 2.8.4]{benson91a}, for each $n$, 
and any $R$-module $M$,
there is an isomorphism of vector spaces,
\[
\coh^n(R,M) \simeq  \Ext^n_{\Ac}(\ku,\Hom_R(\Ac, M))
   = \coh^n(\Ac , \Hom_R(\Ac,M)),
\]
where $\Hom_R(\Ac,M)$ is the coinduced right $\Ac$-module.
(The action is given by $(f\cdot a)(b) = f(ab)$ for all $a,b\in\Ac$,
$f\in\Hom_R(A,M)$.
Then $f\cdot a$ is indeed a right $R$-module homomorphism.)
These isomorphisms, one for each $n$,
provide an isomorphism of $\coh(A,\ku)$-modules
$\coh(R, M)\simeq  \coh(\Ac , \Hom_R (\Ac, M))$. 
Now when $M$ is a finite-dimensional $R$-module, 
$\Hom_R(\Ac, M)$ is finite-dimensional as a vector space.
For $M = \ku$, a set of generators of $\coh(\Ac,\Hom_R(\Ac, \ku))$ as a 
module for $\coh(\Ac,\ku)$, together with
the restriction to $R$ of a set of generators of $\coh(\Ac,\ku)$,
generates $\coh(R,\ku)$ as a $k$-algebra.   For an arbitrary 
finite-dimensional module $M$,  $\Hom_R(\Ac, M)$ is then a finite-dimensional module 
over $\coh(\Ac,\ku)$ and, hence, over $\coh(R,\ku)$.
\end{proof}

If $K$ is a Hopf subalgebra of a finite-dimensional Hopf algebra $H$, then $H$ is free as a left or right module over $K$ with respect
to multiplication by the Nichols--Z\"oller Theorem. 
Thus Theorem \ref{lema:frobenius-reciprocity} applies to inclusions of Hopf algebras, in particular to the inclusion of
a finite-dimensional Hopf algebra into its Drinfeld double, see \cite[Theorem 3.4]{NP}.
For further reference we state a useful application of Theorem \ref{lema:frobenius-reciprocity}.

\begin{cor}\label{cor:frobeniusreciprocity-nichols}
Let $H$ be a finite-dimensional Hopf algebra and $V\in \ydh$ such that $dim \toba(V) <\infty$.
If $D(\toba(V) \# H)$ has fgc, then so does $\toba(V)$.  
\end{cor}

These ideas apply in particular to Morita equivalence of  Hopf algebras as  in \S \ref{subsec:morita}.

\begin{cor}\label{cor:morita}
If $H \mor H'$ and $D(H)$ has fgc, then so does $H'$.
\end{cor}

\begin{question}\label{question:morita}
Is the fgc property for Hopf algebras invariant under Morita equivalence in the sense of \S \ref{subsec:morita}?
\end{question}
The reader might want to compare this question to \cite[Conjecture 1.1]{NP}.

\begin{lema}\label{lema:extensions} Let $\ku \to K \rightarrow H \rightarrow L \to \ku$ be an extension of finite-dimensional
Hopf algebras.
If $K$ is semisimple and $L$ has finitely generated cohomology, then so does $H$.
\end{lema}

The proof makes use of a variation of the classical Hochschild-Serre spectral sequence.

\pf Let $M$ be an $L$-module and $N$ an $H$-module. 
By \cite[Chapter 16, Theorem 6.1]{cartan-eilenberg} there exists a convergent spectral sequence
\begin{align*}
H^p(L, H^q(K, N)) \implies H^{p + q}(H, N).
\end{align*} 

Since $K$ is semisimple, we have $H^q(K, N) = 0$ when $q>0$. Hence, the spectral sequence 
degenerates giving an isomorphism $H(H,N)  \cong H(L, N^K)$.  For $N = \ku$ we 
immediately get that $H(H,\ku) \cong H(L,\ku)$ is finitely generated. For an arbitrary 
finitely generated $H$-module $N$, we have that $N^K$ is a finitely generated $L$-module, and, hence, by assumption $H(L, N^K)$ is a finite $H(L,\ku) \cong H(H, \ku)$-module.
\epf

\subsection{Evens Lemma}\label{subsec:permanent-cycles}

Let $R = \oplus_{n\in \N_0} R^n$ be an $\N_0$-graded ring with a decreasing algebra filtration $F^{n}R$, $n\in \N_0$, compatible with the grading. We shall assume that $F^{i}R^n = 0$ for $i$ sufficiently large.
Then the associated graded ring $E_0(R) = \sum_i F^{i}R/F^{i+1}R$ is $\N_0^2$-graded.

Similarly, the graded $E_0(R)$-module associated to an $\N_0$-graded $R$-module $N$ with a decreasing module filtration 
is $\N_0^2$-graded. Again, $F^{i}N^j = 0$ for $i$ sufficiently large.
The following proposition is \cite[Section 2, Proposition 2.1]{evens}.

\begin{prop}\label{Pch2}
Let $R$ be a graded filtered ring  and $N$ a graded filtered $R$-module as above. If $E_0(N)$ is (left) Noetherian over $E_0(R)$, then $N$ is Noetherian over $R$. \qed
\end{prop}

The following result is  a non-commutative version of 
\cite[Lemma 2.5]{MPSW}, adapted in turn from \cite[Lemma 1.6]{FS} and inspired by early work of Evens. 

\medbreak
Let $E_1^{p,q} \Rightarrow E_{\infty}^{p+q}$ be a multiplicative spectral sequence of bigraded $\ku$-algebras concentrated in the half plane $p+q \geq 0$.
Recall that $x \in E_r^{p,q}$ is called a \emph{permanent cycle} if $d_i(x)=0$ for all $i\geq r$. 
More precisely, if $i> r,\ d_i$ is applied to the image of $x$ in $E_i$.

\begin{lemma}\label{Lch2} \cite[Lemma 2.6]{Shroff}
\begin{enumerate}[leftmargin=*,label=\rm{(\alph*)}]
\item
Let $C^{*,*}$ be a bigraded $\ku$-algebra such that for each fixed $q$,  $C^{p,q} = 0$ for $p$ sufficiently large. Assume that there exists a bigraded map of algebras $\phi: C^{*,*}\rightarrow E_1^{*,*}$ such that 

\begin{enumerate}[leftmargin=*,label=\rm{(\arabic*)}]
\medbreak
\item $\phi$ makes $E_1^{*,*}$ into a left Noetherian $C^{*,*}$-module, and

\medbreak
\item the image of $C^{*,*}$ in $E_1^{*,*}$ consists of permanent cycles.
\end{enumerate}
Then $E_\infty^{*}$ is a left Noetherian module over Tot($C^{*,*}$).

\medbreak
\item 
Let $\widetilde{E}_1^{p,q} \Rightarrow \widetilde{E}_\infty^{p+q}$ be a spectral sequence that is a bigraded module over the spectral sequence $E^{*,*}$. Assume that $\widetilde{E}_1^{*,*}$ is a left Noetherian module over $C^{*,*}$ where $C^{*,*}$ acts on $\widetilde{E}_1^{*,*}$ via the map $\phi$. Then $\widetilde{E}_\infty^{*}$ is a finitely generated $E_\infty^{*}$-module.
\qed
\end{enumerate}
\end{lemma}

\subsection{The May spectral sequence}\label{subsec:May}
Let $A$ be a Hopf algebra equipped with an increasing multiplicative filtration $A_0 \subset A_1 \subset A_2 \ldots \subset A$.  We fix a (non-canonical) vector space splitting 
$A \simeq  A_0 \oplus A_+$ so that $A/A_+ \simeq  A_0$. 
Let $(P,d) = ( V_n\ot A, d)$ be a free resolution of the trivial module $k$ satisfying the following properties. 
\begin{assumption}
\label{as:resol} 
\begin{enumerate} [leftmargin=*,label=\rm{(\arabic*)}]
\item $V_n$ is a finite-dimensional vector space, the action of $A$ on $P_n$ is on the last factor $A$. 
\item $P_{\bu}$ is equipped with an increasing filtration $ \ldots F_iP_n \subset F_{i+1}P_n \ldots$. 
\item For any $x \in F_iV_n = F_i(V_n\ot 1) := F_iP_n \cap ( V_n\ot 1)$, we have 
\begin{align*}
d(x) \in  F_iV_{n-1}\ot A_0 +  F_{i-1}V_{n-1}\ot A_+.
\end{align*}
\end{enumerate}
\end{assumption}

\begin{example} We will apply this setting in at least two different situations. 
\begin{enumerate}[leftmargin=*,label=\rm{(\roman*)}]
\item When $P_{\bu}$ is the bar resolution and $F_{\bu}$ is the coradical filtration, see Theorem \ref{thm:fingencoh-drinfeld-double}. 
\item When $P_{\bu}$ is the Anick resolution of $k$ for $\toba(V)$, and the filtration is given by the PBW basis induced by the convex ordering of the roots,
see Theorem \ref{lemma:reduction-to-cocycles}. In this case $A_0 =k$ and we identify $A_ +$ with the augmentation ideal. 
\end{enumerate} 

\end{example} 
 
We set up a version of May spectral sequence analogous to the one in \cite[5.5]{GK}.   We follow the construction in May \cite{May} but without assuming that the module $M$ is filtered. Such a spectral sequence is also constructed in \cite[\S9]{BKN} using a non-canonical filtration on $M$ induced by the filtration on $A$. 
\begin{theorem}
\label{may}
Let $A$ be a filtered finite-dimensional Hopf algebra , $(P_{\bu}, d)$ be a projective resolution of the trivial module $k$, and assume that $A$
and $P_{\bu}$ satisfy Condition~\ref{as:resol}. Let $M$ be an $A$-module. Then there exists a converging cohomological spectral sequence 
\[\widetilde E_1 = H(\gr A, M_{A_0}) \Rightarrow H(A,M) \]
equipped with a natural module structure over the multiplicative spectral sequence
\[E_1 = H(\gr A, k) \Rightarrow H(A,k). \] 
The action of $\gr A$ on $M_{A_0}$ is via the projection $\gr A \to A_0$ and then restricting the action of $A$ on $M$ to the action of the subalgebra $A_0 \subset A$.

\end{theorem} 
 
 \begin{proof} 
 Let $C^{\bu}(A,M) := \Hom_A(P_{\bu}, M)$ be the complex computing $H(A,M)$.  Let 
 \[F^iC^n(A,M):= \{f \in \Hom_A(P_n, A) \, | \, f{\downarrow_{F_{i-1}P_n}}  =0\} \subset C^n(A,M) \] 
 be a decreasing filtration on $C^{\bu}(A,M)$ making it into a filtered complex. 
 
 As $V_n \subset P_n$,  we have an induced filtration (of vector spaces) on $V_{\bu}$: 
 $F_iV_n = F_i( V_n\ot 1) = F_i P_n \cap ( V_n\ot 1)$. Using the isomorphism $\Hom_A( V_n\ot A, M) \simeq  \Hom_k(V_n, M)$, we make the identifications: 
 \begin{align*} 
 \frac{F^iC^n(A,M)}{F^{i+1}C^n(A,M)} = & \frac{\{f: V_n \to M \, | \, f{\downarrow_{F_{i-1}V_n}}=0\}}{\{f: V_n \to M \, | \, f{\downarrow_{F_{i}V_n}}=0\}}\\ 
 & \simeq  \frac{\Hom_k(V_n/F_{i-1}V_n, M)}{\Hom_k(V_n/F_{i}V_n, M)} \\
&  \simeq  \Hom_k\left(\frac{F_iV_n}{F_{i-1}V_n},M\right). 
 \end{align*}  
 Letting $n$ be the total and $i$ be the internal degree, we have 
 \begin{align*}
 \Hom_k\left(\dfrac{F_iV_n}{F_{i-1}V_n},M\right) 
 \simeq  \Hom_{\gr A}\left(\dfrac{F_i P_n}{F_{i-1}P_n},M_{\rm tr}\right) 
 \simeq  H^{n,i}(\gr A, M_{\rm tr}).
 \end{align*}
 Let $\widetilde E^{i, n-i}_0(M) := \dfrac{F^iC^n(A,M)}{F^{i+1}C^n(A,M)}$. Since $A$ is finite-dimensional, the filtration is finite and, hence, 
 this defines the $0$ page of the spectral sequence  of a filtered complex $C^{\bu}(A,M)$ converging to $H(A,M)$. 
 We have identified the terms of the double complex $\widetilde E_0(M)$ with the complex $C^{\bu}(\gr A, M_{A_0})$ computing cohomology 
 $H(\gr A, M_{A_0})$. To identify $\widetilde E_1(M)$ and $H(\gr A, M_{A_0})$ as complexes, it suffices to show that the differentials in $\widetilde E(M)$ and 
 $\widetilde E(M_{A_0})$ are the same, that is, that the differential $d_M$ in the spectral sequence only depends on the $A_0$-module structure on $M$.
 
 Consider the differential $d_M$: 
 \[\xymatrix{ \widetilde E^{i, n-i}_0(M) \ar@{=}[r] & \Hom_k\left(\frac{F_{i}V_{n}}{F_{i-1}V_n},M\right)\\ 
\widetilde E^{i, n-1-i}_0(M)\ar@{=}[r]   \ar[u]_-{d_M} &  \Hom_k\left(\frac{F_iV_{n-1}}{F_{i-1}V_{n-1}},M\right) \ar[u]_-{d_M}.  
 }\]
 Let $f \in  \Hom_k\left(\dfrac{F_iV_{n-1}}{F_{i-1}V_{n-1}},M\right)$, $\bar x \in \dfrac{F_{i}V_{n}}{F_{i-1}V_n}$, and let $x \in F_iV_n$ be a representative of $\bar x$.  By Condition~\ref{as:resol}(3) we can write $d(x) = v_0\ot a_0 +  v^\prime\ot a^\prime$ with $a_0 \in A_0$ and $v^\prime \in F_{i-1}V_{n-1}$. 
 We now compute 
 \begin{align*}
 d_M(f)(\bar x) & = f(\overline{d_M(x)})  \\
 & = f(\bar v_0)a_0 + f(\bar v^\prime)a^\prime \\
 & \stackrel{*}{=} f(\bar v_0)a_0\\
 & = f(\overline{d_{M_{A_0}}(x)}) = d_{M_{A_0}}(f)(\bar x)\\
 \end{align*} 
 The equality (*) holds since $\bar v^\prime = v^\prime \mod F_{i-1}V_{n-1} = 0$. 
 
 The statement about the action of $E_1(k)$ acting on $\widetilde E_1(M)$ follows from the construction of the spectral sequence.
 \end{proof} 

\section{The Anick resolution}\label{sec:Anick}

\subsection{The setup}\label{subsec:Anick}

In this section we discuss the Anick graph and the construction of 
the Anick resolution \cite{Anick,Fa,CU}.
Let $V$ be a finite-dimensional vector space with a basis $(x_i)_{i\in\I}$, and $\Ic$ an ideal of $T(V)$ such that $\epsilon(\Ic)=0$, where 
$\epsilon : T(V)\rightarrow \ku$ is the standard augmentation map,
$\epsilon(x_i) =0$ for all $i\in \I$. 
Thus the algebra $\Ac:=T(V)/\Ic$ has an augmentation map $\epsilon:\Ac\to\Bbbk$.

\subsubsection{The tips}
Let $X$ be the set of words on the letters $(x_i)_{i\in\I}$ (including the empty word 1). Notice that $X$ is a basis of $T(V)$. 
Let $x, y\in X$. We say that $x$ is a subword of $y$ if there exist $w, z\in X$ such that $y = wxz$.
If $w = 1$, respectively $z = 1$, then we say that $x$ is a prefix, respectively a suffix, of $y$.

\medbreak
Let  $\ell:X\to\N_0$ be the length function.
The lex-length order $<$ on $X$ is defined as follows: 
given $v,w\in X$, we say that $v<w$ if either $\ell(v)<\ell(w)$ or else $\ell(v)=\ell(w)$ 
and $v$ is less than $w$ for the lexicographical order (induced by the numeration of the basis). 
This is a total order on $X$ compatible with left and right multiplication.

\medbreak
Here is a way to give a set of generators of the ideal $\Ic$.
Given $f\in\Ic-0$, write $f$ as a linear combination of elements of $X$; let $x_f$ be the largest element of $X$ (with respect to $<$) 
with non-zero coefficient. Then $x_f$ is called the \emph{tip} of $f$. Consider the set of all tips of all elements in $\Ic-0$.
A tip $t$ is \emph{minimal} if each subword of $x$ is not a tip (Anick calls a minimal tip an \emph{obstruction}). Let $\Tc$ be the set of minimal tips of $\Ic$.
For each $t\in\Tc$ we pick $\omega_t\in\Ic$ such that $t$ is the tip of $\omega_t$ (which is not unique in general). Arguing recursively on $<$, it is possible to show that
\begin{align*}
\Ic &= \langle \omega_t: \, t\in\Tc\rangle.
\end{align*}

For each $w\in X$ we also denote by $w$ its image in $\Ac=T(V)/\Ic$. By \cite[Lemmas 1.1 and 1.2]{Anick}, the set
\begin{align}\label{eq:basis-A}
\basis=\{ w\in X: t \text{ is not a subword of }w \, \, \forall t\in\Tc \}
\end{align}
is a basis of $\Ac$.

\subsubsection{The chains} Let $n\in\N_0$. We describe the $n$-chains which are words defined from the minimal tips;
they will provide a basis of the $n$-th term of the Anick resolution of $\Ac$. 
The unique $0$-chain is the empty word $1$. The  $1$-chains are the letters, 
i.e.~the  $x_i$'s. Let $n\ge 1$. An $n$-chain is a word $w$ such that:

\medbreak
\begin{enumerate}[leftmargin=*,label=\rm{(\alph*)}]

\medbreak
\item\label{item:n-chain-1} $w$ admits a factorization $w=uv$ such that $u$ is an $(n-1)$-chain and the suffix $v$ 
does not contain any minimal tip as a subword (i.e., does not contain any tip);

\medbreak
\item\label{item:n-chain-2} for every suffix $y \neq 1$ of $u$ as in \ref{item:n-chain-1}, the word $yv$ contains a minimal tip as a subword;

\medbreak
\item\label{item:n-chain-3} any other prefix $w'$ of $w$ does not satisfy \ref{item:n-chain-1} and \ref{item:n-chain-2} simultaneously. 
\end{enumerate}
Let $\chain(n)$ be the set of $n$-chains. We urge the reader to check that $\chain(2)$ is the set of minimal tips---all requirements are needed.

\medbreak
There exists an alternative way to express the definition of  $n$-chains. 
A word $w=x_{i_1}\cdots x_{i_t}$, $i_j\in\I$, is an $(n+1)$-chain if there exist integers $a_j$, $b_j$, $1\le j\le n$, such that

\medbreak
\begin{enumerate}[leftmargin=*]

\medbreak
\item\label{item:n-chain-alt-1} $1=a_1<a_2\le b_1<a_3\le b_2 <\cdots <a_n\le b_{n-1}<b_n=t$;

\medbreak
\item\label{item:n-chain-alt-2} $x_{i_{a_j}}x_{i_{a_j + 1}}\dots x_{i_{b_j - 1}}x_{i_{b_j}}\in \Tc$ for all $1\le j\le n$;

\medbreak
\item\label{item:n-chain-alt-3} for all $1\le m\le n$, the words $x_{i_1}\cdots x_{i_s}$, $s<b_m$, are not $m$-chains. 
\end{enumerate}

By \cite[Lemma 1.3]{Anick} the integers $a_j,b_j$ satisfying \eqref{item:n-chain-alt-1}--\eqref{item:n-chain-alt-3} are uniquely determined and
\begin{itemize}
\item $x_{i_1}\cdots x_{i_{b_{n-1}}}$ is the unique prefix which is an $(n-1)$-chain, and 

\item $x_{i_{b_{n-1}+1}} \cdots x_{i_t}$ does not contain any element of $\Tc$ as a subword.

\end{itemize}

\begin{example}
We fix $N\ge 2$, $V$ of dimension 1, $x\in V-0$, $\Ic=\langle x^N \rangle$. Thus $\Tc=\{x^N\}$. Now $\chain(0)=\{1\}$, $\chain(1)=\{x\}$, and we claim that 
\begin{align*}
\chain(2k)&=\{x^{Nk}\}, & 
\chain(2k+1)&=\{x^{Nk+1}\}, &
&k\ge 1.
\end{align*}
Moreover $a_{2i-1}=(i-1)N+1$, $a_{2i}=(i-1)N+2$, $b_{2i-1}=iN$, $b_{2i}=iN+1$. We proceed by induction on $\ku$. If $k=2$, then $\chain(1)= \{x^N\}$ since this is the unique minimal tip, and $\{x^{N+1}\}$ is a $2$-chain with $a_1=1$, $a_2=2$, $b_1=N$, $b_2=N+1$; thus, each word $x^{N+j}$, $j>1$, is not a $2$-chain since $x^{N+1}$ is a prefix of $x^{N+j}$.

Assume that $k\ge 2$ and the statement holds for $\ku$. To compute $\chain(2k+2)$, we start with the unique $(2k+1)$-chain $x^{kN+1}$ and the integers $a_j$, $b_j$ already determined: $a_{2k+1}$ should satisfy $b_{2k-1}=kN < a_{2k+1}\le b_{2k} = kN+1$, hence $a_{2k+1}=kN+1$. Hence $\chain(2k+2)=\{x^{(k+1)N}\}$. For $\chain(2k+3)$, we have that $x^{(k+1)N+1}$ is a $(2k+3)$-chain, hence this is the unique $(2k+3)$-chain: indeed, if $w=x^{s} \in \chain(2k+3)$, then $s>(k+1)N$ since $w$ should contain the $(2k+2)$-chain $x^{(k+1)N}$ as a prefix; but if $s>(k+1)N$, then $w$ contains the $(2k+3)$-chain $x^{(k+1)N+1}$ as a prefix so it cannot be a $(2k+3)$-chain.
\end{example}

Let $\Vt(n)$ be the $\Bbbk$-vector space with basis $\chain(n)$. Then
\begin{align*}
\wchain(n):=\{ u\otimes w : u\in\chain(n), \, w\in \basis  \}
\end{align*}
is a basis of $\Vt(n)\ot \Ac$. Given $u\otimes w, v\otimes z\in \wchain(n)$, if $uw=vz$, then $u=v$, $w=z$; indeed, if $\ell(u)\leq \ell(v)$, then the $n$-chain $u$ is a prefix of the $n$-chain $v$ and \eqref{item:n-chain-alt-3} implies that $u=v$. Hence the order on $X$ induces an order on $\wchain(n)$:
$u\otimes w< v\otimes z$ if $uw<vz$.

\subsubsection{The Anick graph}

We next introduce a graph which helps to compute the chains of the Anick resolution \cite{CU}. Let $\Gamma$ be the graph whose set of vertices is given by the union of $\{1\}$, $X$ and the set of all proper suffixes of the minimal tips. For the arrows, there exists one arrow from $1 \to x$ for each $x\in X$, and 
one arrow $u\to v$ if the word $uv$ contains a unique minimal tip such that it is a suffix of $uv$ (possibly the word $uv$).

A basis of the free module of $n$-chains of the Anick resolution is given by paths of length $n$ starting at $1$. Thus:
\begin{itemize}
\item There exists a unique $0$-chain: $1$.
\item The set $X$ gives a basis of the $1$-chains.
\item The set of minimal tips gives a basis of the $2$-chains.
\end{itemize}

Notice that vertices $v$ not connected to $1$ (that is, without a path from $1$ to $v$) do not contribute new elements for the basis of chains, hence we may omit them and the related arrows.

\begin{example}
Let $\zeta\in\G'_{12}$, $q\in\Bbbk^\times$. 
We want to determine the Anick graph of the Nichols algebra $\toba_{\bq}$ of 
\cite[\S 10.7.5]{AA17} (The scalar $q$ corresponds to $q_{12}$ in loc.~cit.) In terms of the PBW generators,
$\toba_{\bq}$ is presented by generators
$x_1$, $x_{1112}$, $x_{112}$, $x_{12}$, $x_2$, and relations
\begin{align*}
x_1^4&=0, & 
x_1x_{1112} &= \zeta^3q\, x_{1112}x_1,  &
x_1x_{112} &= -q \, x_{112}x_1+x_{1112},
\\
x_{1112}^2&=0, & 
x_1x_{12} &=\zeta^9q\,x_{12}x_1+x_{112},  &
x_1x_{2} &=q \, x_2x_1+x_{12}
\\
x_{112}^3&=0, & 
x_{1112}x_{112} &= \zeta^2 q \, x_{112}x_{1112},        &
x_{1112}x_{12} &=\zeta^{10}q^2 \, x_{12}x_{1112}-q(1+\zeta)x_{112}^2,
\\
x_{12}^3&=0, & 
x_{112}x_{12} &=\zeta q \, x_{12}x_{112}, &
x_{1112}x_{2} &=-q^3x_2x_{1112}+\zeta^5q^2 \, x_{12}x_{112},
\\
x_{2}^2&=0, &
x_{12}x_{2} &=-q\, x_2 x_{12},       &
x_{112}x_{2} &= -q^2\, x_2x_{112}-q_{12}(1+\zeta^9) x_{12}^2.
\end{align*}
Thus the set of 2-chains (equivalently, minimal tips or obstructions) is
\begin{align*}
\chain(2) &= \{
x_1^4, x_{1112}^2, x_{112}^3, x_{12}^3,x_{2}^2,
x_1x_{1112}, x_1x_{112}, x_1x_{12}, x_1x_{2},
\\
& \quad 
x_{1112}x_{112}, x_{1112}x_{12}, x_{1112}x_{2}, x_{112}x_{12}, x_{112}x_{2}, x_{12}x_{2} \}
\end{align*}
and the Anick graph is  

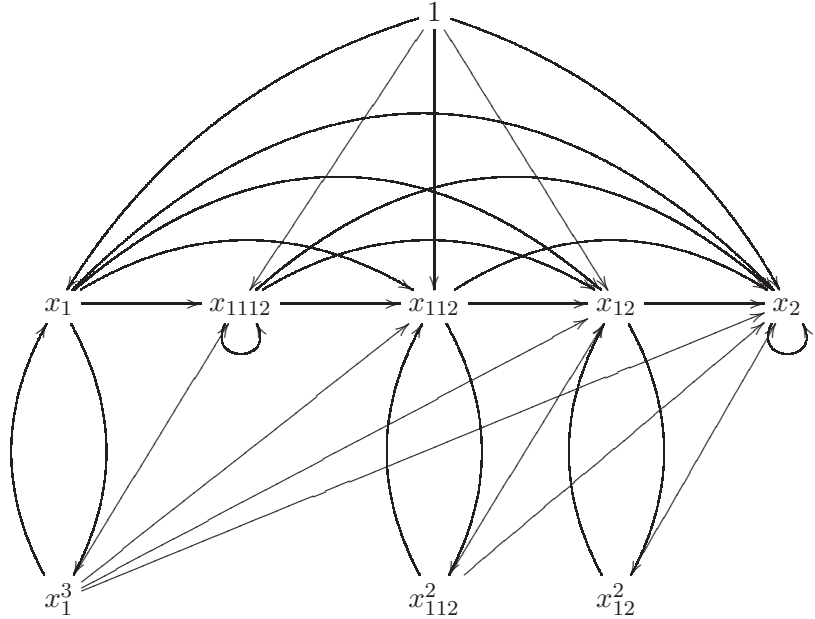
\begin{align*}
\xymatrix@=45pt{ & &  1 \ar@/_1.5pc/[ddll] \ar[ddl] \ar[dd] \ar[ddr] \ar@/^1.5pc/[ddrr] & & 
\\
&&&&
\\
x_1\ar@/^1.5pc/[dd] \ar[r] \ar@/^2pc/[rr] \ar@/^4pc/[rrr] \ar@/^6pc/[rrrr]
& 
x_{1112} \ar[r] \ar@(dl,dr)[] \ar@/^2pc/[rr] \ar@/^4pc/[rrr]
&
x_{112} \ar[r]\ar@/^1.5pc/[dd] \ar@/^2pc/[rr]
&
x_{12} \ar[r]\ar@/^1.5pc/[dd] 
&
x_2\ar@(dl,dr)[]
\\
&&&&
\\
x_1^{3}\ar@/^1.5pc/[uu]\ar[uur] \ar[uurr] \ar[uurrr] \ar[uurrrr]
& &
x_{112}^{2}\ar@/^1.5pc/[uu] \ar[uur] \ar[uurr]
&
x_{12}^{2}\ar@/^1.5pc/[uu] \ar[uur]
& } 
\end{align*}
There exists a vertex $x_1^2$ with a loop on itself which we omit since this vertex is not connected to $1$.
Using the graph we compute 
\begin{align*}
\chain(3) &=
\{
x_1^5, x_1^4x_{1112},x_1^4x_{112},x_1^4x_{12},x_1^4x_2,
x_1x_{1112}^2, x_{1112}^3, x_{1112}^2x_{112}, x_{1112}^2x_{12}, x_{1112}^2x_2,
\\ & \quad 
x_1x_{112}^3, x_{1112}x_{112}^3, x_{112}^4, x_{112}^3x_{12}, x_{112}^3x_2, 
x_1x_{12}^3, x_{1112}x_{12}^3, x_{112}x_{12}^3, x_{12}^4, x_{12}^3x_2, 
\\ & \quad 
x_1x_{2}^2, x_{1112}x_{2}^2, x_{112}x_{2}^2, x_{12}x_{2}^2, x_{2}^3, 
x_1x_{1112}x_{112}, x_1x_{1112}x_{12}, x_1x_{1112}x_2, x_1x_{112}x_{12}, 
\\ & \quad 
x_1x_{112}x_2, x_1x_{12}x_{2},
x_{1112}x_{112}x_{12}, x_{1112}x_{112}x_2, 
x_{1112}x_{12}x_2, x_{112}x_{12}x_{2}\}.
\end{align*}

\end{example}

\subsubsection{The resolution}
We consider the Anick resolution of the $\Ac$-module $\Bbbk$. Recall that $\Vt(n)$ is the $\Bbbk$-vector space with basis $\chain(n)$. In \cite[Theorem 1.4]{Anick} Anick introduced an $\Ac$-free complex
\begin{align}\label{eq:anick-resolution}
\xymatrix@C=20pt{ \ar@{.}[r] & \Vt(n)\ot \Ac \ar[r]^-{d_n} & \Vt(n -1)\ot \Ac \ar[r]^-{d_{n-1}} & \ar@{.}[r] & \Vt(1)\ot \Ac \ar[r]^-{d_1} & \Ac \ar[r]^-{\epsilon} & \ku
\ar[r] & 0}
\end{align}
and $\Bbbk$-linear maps $s_n:\Vt(n)\ot \Ac \to \Vt(n+1)\ot \Ac$, $n\in\N_0$, such that \eqref{eq:anick-resolution} is an $\Ac$-resolution of $\Bbbk$ and $s_{\bu}$ is a contracting homotopy:
\begin{align}\label{eq:contracting-homotopy}
d_{n+1}s_n+s_{n-1}d_n &= \id_{\Vt(n)\ot \Ac} & \text{for all }& n\in\N.
\end{align}

The maps $d_n$, $s_{n-1}$ are defined recursively; see for example \cite[\S 4.1]{NWW} for a left module version. For $n=1$, 
the map $d_1$ is determined by 
\begin{align*}
d_1 (x \otimes 1)&=x & \text{for all }&x\in\chain(1),
\end{align*}
while for $s_0$ we give the values on each $w\in \basis$: 
\begin{align*}
s_0(w)&=x\otimes z,  & w=xz, & \, x\in \chain(1), \, z\in \basis.
\end{align*}

Now assume that $d_1,\dots,d_{n-1}$, $s_0,\dots,s_{n-2}$ were already defined and satisfy:
\begin{align*}
&\eqref{eq:contracting-homotopy} & 
d_{i-1}d_{i} &=0, &
s_{i-2}s_{i-1} &=0, &
\text{for all }&1\le i\le n-1.
\end{align*}
The morphisms of $\Ac$-modules
$d_n: \Vt(n)\ot \Ac \to \Vt(n -1)\ot \Ac$ are determined by
\begin{align*}
d_n (u \otimes 1) & = v \ot t - s_{n-2}d_{n-1}(v \ot t), & u&\in \chain(n), \, u=vt, \, v\in \chain(n-1), \, t\in\basis.
\end{align*}
Now we define $s_{n-1}$. From \eqref{eq:contracting-homotopy}, 
$\Vt(n -1)\ot \Ac=\ker d_{n-1}\oplus \im s_{n-2}$. 
We start by defining $(s_{n-1})_{\vert \im s_{n-2}}\equiv 0$, so 
$s_{n-1}s_{n-2}=0$. Now we define $(s_{n-1})_{\vert \ker d_{n-1}}$ recursively on the order of the leading term of each element of $\ker d_{n-1}$, which we write in terms of the basis $\wchain(n-1)$. We require $d_n (s_{n-1})_{\vert \ker d_{n-1}} = \id_{\ker d_{n-1}}$.
Let
\begin{align*}
K &=\sum_{j\in\I_m} a_j \, u_j\ot b_j\in\ker d_{n-1}, &
a_j\in \Bbbk^{\times}, \, & u_j\in \chain(n-1), b_j\in\basis.
\end{align*}
We assume that $u_1\ot b_1$ is bigger than $u_j\ot b_j$ for all $j>1$. We write $u_j=v_j t_j$, where $v_j\in\chain(n-2)$, $t_j\in \basis$. Hence
\begin{align*}
0& = d_{n-1}(K)=a_1 \, v_1\ot t_1b_1 + \widehat{K}, & 
\widehat{K}&:= \sum_{j\ge 2} a_j d_{n-1}(u_j\ot b_j)
-s_{n-3}d_{n-2} (a_1 v_1\ot t_1b_1)
\end{align*}
Hence $t_1b_1\notin \basis$, otherwise $v_1\ot t_1b_1$ is the biggest element of $\wchain(n-1)$ in the previous expression of $d_{n-1}(K)$ with non-zero coefficient. Now we write $b_1=w_1y_1$, where $w_1$ is the shortest prefix of $b_1$ such that $t_1w_1 \notin\basis$. Hence $u_1w_1=v_1 t_1 w_1\in\chain(n)$. Hence we set
\begin{align}\label{eq:Anick-def-sn}
s_{n-1}(K) := a_1\, u_1w_1 \ot y_1 + s_{n-1} \left( \sum _{j\in\I_m} a_j\, u_j \ot b_j - d_n (a_1\, u_1w_1 \ot y_1)\right).
\end{align}
Below the differentials $d_n$ and the maps $s_n$ will be denoted simply by $d$ and $s$.

\vspace{0.1in} We shall refer to the complex \eqref{eq:anick-resolution} as the {\it Anick resolution} of $\Ac$, and denote it by $(C_*(\Ac), d)$: 
\begin{align} 
\label{eq:anick-notation} 
\xymatrix@C=20pt{ \ar@{.}[r] & C_n(\Ac)  \ar[r]^-{d_n} &C_{n-1}(\Ac) \ar[r]^-{d_{n-1}} & \ar@{.}[r] &C_1(\Ac) \ar[r]^-{d_1} & \Ac \ar[r]^-{\epsilon} & \ku
\ar[r] & 0}
\end{align} 
so that $C_n(\Ac) =  \Vt(n)\ot \Ac$. 

\subsection{Application to Nichols algebras}

We consider now the Anick resolution of a Nichols algebra $\toba_{\bq}$ of diagonal type associated to the presentation given by PBW generators $(x_{\beta})_{\beta\in\varDelta_+^{\bq}}$. We fix a convex order on $\varDelta_+^{\bq}$, which induces a total order on the letters: $x_{\beta_1}>x_{\beta_2}>\dots >x_{\beta_{m}}$.
We set $N_{\beta} = \ord \bq_{\beta\beta}$.

The defining relations of $\toba_{\bq}$ are
\begin{align}\label{eq:powerrootvector}
x_{\beta}^{N_{\beta}}&=0, & &N_\beta \mbox{ finite};
\\ \label{eq:convex-combination}
\left[ x_{\beta_i}, x_{\beta_j} \right]_c&= 
\sum_{n_{i+1}, \dots, n_{j-1} \in\N_0} c_{n_{i+1}, \dots, n_{j-1}}^{(i,j)} \ x_{\beta_{j-1}}^{n_{j-1}} \dots x_{\beta_{i+1}}^{n_{i+1}}, & & i<j.
\end{align}
where $c_{n_{i+1}, \dots, n_{j-1}}^{(i,j)} \in \Bbbk$ can be computed explicitly \cite[Lemma 4.5]{A-jems}.

We denote by $\x_\beta$, $\beta\in\varDelta_+^{\bq}$, the set of letters for the tips and the chains, to distinguish them from the generators $x_\beta$ of the Nichols algebra $\toba_{\bq}$.
The minimal tips are the following: 
\begin{align}\label{eq:Nichols-minimal-tips}
&\x_\alpha \x_{\beta}, \quad \alpha>\beta \in \beta\in\varDelta_+^{\bq}; &
& \x_{\beta}^{N_{\beta}}, \quad \beta\in\varDelta_+^{\bq}.
\end{align}
Hence the Anick graph looks \emph{locally}  as
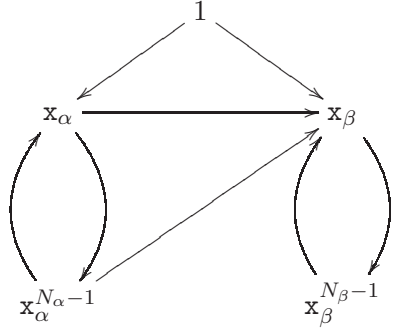
\begin{align*}
\xymatrix@C=30pt{ &  1 \ar[dl]  \ar[dr] & \\
\x_\alpha\ar@/^1.5pc/[dd] \ar[rr] && \x_{\beta}
\ar@/^1.5pc/[dd] 
\\ & & \\
\x_\alpha^{N_\alpha-1}\ar@/^1.5pc/[uu]\ar[uurr]
&& \x_{\beta}^{N_{\beta}-1}\ar@/^1.5pc/[uu] }
\end{align*}
for $\alpha>\beta$, where if $N_\alpha=2$, then the loop between $x_\alpha$ and $x_\alpha^{N_\alpha-1}$ is understood to be collapsed 
to a loop from $x_\alpha$ to itself, and similarly for $x_\beta$.
If $N_\alpha>3$, then there exist vertices $\x_{\alpha}^{t}$ and $\x_{\alpha}^{N_\alpha-t}$ for each $2\le t\le N_\alpha-2$, and arrows between them. These vertices are not connected to $1$ and are omitted.

Now we describe the set of chains. For each $\delta\in \varDelta_+^{\bq}$ we set
\begin{align}\label{eq:power-root-vector-chain}
f_\delta: & \ \N_0\to\N_0, 
& 
f_{\delta}(2k)&= N_\beta k,
&
f_{\delta}(2k+1)&=N_\beta k+1, 
& 
k&\in\N_0.
\end{align}
The set of all $n$-chains, $n\in\N$, is given by
\begin{align}\label{eq:chains-Nichols}
\chain(n) & = \left\{ \x_{\beta_1}^{f_{\beta_1}(n_1)} \x_{\beta_2}^{f_{\beta_2}(n_2)} \dots \x_{\beta_m}^{f_{\beta_m}(n_m)} \, : \, \sum n_j=n  \right\}.
\end{align}

\bigbreak




\subsection{Quantum linear spaces}\label{subsec:qls}
These are the less complicated Nichols algebras of diagonal type.
Let $\bq = (q_{ij})_{i, j \in \I}$ be as above and assume that $q_{ij}q_{ji} = 1$ for all
$i\neq j \in \I$. Let $\I' = \{i \in \I: q_{ii} \in \G'_{\infty}\}$ and for $i\in \I'$, set $N_i = \ord q_{ii}$.
Then $\toba_{\bq}$ is presented by generators $y_i$, $i\in \I$, with relations
\begin{align} \label{eq:qls1}
y_iy_j &= q_{ij} y_j y_i,&  i&<j \in \I, 
\\ \label{eq:qls2}
y_i^{N_i}&=0,&  i &\in I'.
\end{align}

\begin{prop} \label{prop:qci}
The cohomology ring of $\toba_{\bq}$ is generated by $\eta_i$ for $i  \in \I$ and 
$\xi_j$, for $j  \in \I'$ with relations
\begin{align}\label{eq:relations-qci}
\begin{aligned}
\xi_h\xi_j &= q_{hj}^{N_hN_j}\xi_j\xi_h, & h &< j \in \I',
\\
\eta_i\xi_j &= q_{ij}^{N_j}\xi_j\eta_i, & i &\in \I, \ j \in \I',
\\
\eta_i\eta_k &= -q_{ik}\eta_k\eta_i, & i &< k \in \I,
\\ \eta_i^2 &= 0, & i &\in \I, \ N_i>2, \\ \eta_i\xi_i &= \xi_i\eta_i,&  i &\in \I'.
\end{aligned}
\end{align}
If $N_i=2$, then $\eta_i^2=\xi_i$. 
If $M$ is a finitely generated $\toba_{\bq}$-module,
then $\coh(\toba_{\bq}, M)$ is finitely generated as a 
module over $\coh(\toba_{\bq}, \ku)$. 
\end{prop}

\pf
We will construct the Anick resolution $K_{\bu}$ of $\ku$ as a $\toba_{\bq}$-module.
The following notation will be helpful. 
For each $i$, $1\leq i\leq\theta$,
let $\sigma_i, \tau_i: \N_0\rightarrow \N_0$ be the functions defined
by
$$
\sigma_i(a) = \left\{\begin{array}{cl}
1, & \mbox{ if }a\mbox{ is odd}\\
N_i-1, & \mbox{ if }a\mbox{ is even},
\end{array}\right.
$$
and $\tau_i(a) = \displaystyle{\sum_{j=1}^a \sigma_i(j)}$ for $a\geq 1$, $\tau(0)=0$.

We claim that the differential $d$
of the Anick resolution is given as follows: 
$$
d ( \y_1^{\tau_1(a_1)}\cdots \y_{\theta}^{\tau_{\theta}(a_{\theta})} \ot 1) = 
\sum_{i=1}^{\theta}\left(\prod_{\ell < i} (-1)^{a_{\ell}}
q_{ \ell i} ^{-\sigma_i(a_i)\tau_{\ell}(a_{\ell})}\right)
\y_1^{\tau_1(a_1)}\cdots \y_i^{\tau_i(a_i-1)}\cdots y_{\theta}^{\tau_{\theta}(a_{\theta})}
  \ot y_i^{\sigma_i(a_i)} ,
$$
where we set $y_i^{\sigma_i(0)}=0$.
This is the right module analog
of a formula from~\cite[\S4]{MPSW}, for an explicitly constructed minimal
resolution of $\ku$ as a $\toba_{\bq}$-module. 
(There
is a slight difference in comparison of our formulas 
to those in \cite{MPSW} since
we are working with right modules. This also leads to a slight
difference in relations among the generators.)
The resolution from~\cite{MPSW} is in fact the Anick resolution:
The modules in the Anick
resolution will have the same vector space dimension in each degree by construction,
implying that the Anick resolution is also minimal in this case,
and so comparison maps between the Anick resolution and
this one must be isomorphisms in each degree.



Next apply  $\Hom_S( - , \ku)$ to $K_{\bu}$
in order to compute $\Ext^*_S(\ku,\ku)$. Note that 
$d^* $ is the zero map since $y_i^{\sigma_i(a_i)}$ acts as 0 on $\ku$.
Thus  in degree $n$ the cohomology is a vector space of
dimension $\binom{n+\theta -1}{\theta -1}$.
Now let $\xi_i\in\Hom_S(K_2,\ku)$ be the function dual to
$\y_i^{N_i}\ot 1$ 
and $\eta_i\in \Hom_S(K_1,\ku)$ be the function dual to
$\y_i\ot 1$.
Identify these functions with the
corresponding elements in $\coh^2(S,\ku)$ and $\coh^1(S,\ku)$, respectively.
We claim that $\xi_i$, $\eta_i$ generate $\coh(S,\ku)$.
We also  denote by $\xi_i$ and $\eta_i$ the
corresponding chain maps $\xi_i : K_n\rightarrow K_{n-2}$
and $\eta_i : K_n \rightarrow K_{n-1}$ given by 
\begin{multline*}
\xi_i(\y_1^{\tau_1(a_1)}\cdots\y_{\theta}^{\tau_{\theta}(a_{\theta})}\ot 1)
\! = \! \prod_{\ell>i} q_{i\ell }
^{-N_i\tau_{\ell}(a_{\ell})} \y_1^{\tau_1(a_1)}\cdots\y_i^{\tau_i(a_i-2)}\cdots
\y_{\theta}^{\tau_{\theta}(a_{\theta})} \ot 1 \\
\eta_i (\y_1^{\tau_1(a_1)}\cdots \y_{\theta}^{\tau_{\theta}(a_{\theta})}\ot 1)
\! = 
\\
\! \prod_{\ell<i}q_{\ell i } ^{-(\sigma_i(a_i)-1)\tau_{\ell}
(a_{\ell})} \prod_{\ell >i} (-1)^{a_{\ell}} q_{i\ell }^{-\tau_{\ell}
(a_{\ell})} 
\y_1^{\tau_1(a_1)}\cdots \y_i^{\tau_i(a_i-1)}\cdots \y_{\theta}^{\tau_{\theta}(a_{\theta})} \ot y_i^{\sigma_i(a_i)-1}.
\end{multline*}
Note this implies that if $a_i$ is even and $N_i>2$, 
then $\eta_i (\y_1^{\tau_1(a_1)}\cdots \y_{\theta}^{\tau_{\theta}(a_{\theta})}\ot 1)
= 0$, since $y_i^{N_i-2}$ acts as 0 on $\ku$. 
Calculations show that these maps satisfy the following equations:
\begin{equation}\label{relations}
\xi_i\xi_j=q_{ij}^{N_iN_j}\xi_j\xi_i, \ \ \
\eta_i\xi_j=q_{ij}^{N_j}\xi_j\eta_i, \ \ \
\mbox{ and } \ \eta_i\eta_j=-q_{ij}\eta_j\eta_i
\end{equation}
for all $i<j$, and $\eta_i\xi_i=\xi_i\eta_i$ for all $i$, 
with one exception:
If $N_i=2$, then $\eta_i^2=\xi_i$ 
(so that we may leave $\xi_i$ out of our choice of generators, or not,
as is convenient),
while if $N_i>2$, then $\eta_i^2 =0$. 
Due to these equations, 
any element in the subalgebra of $\Ext^*_S(k,k)$ 
generated by the $\xi_i$ and $\eta_i$
may be written as a linear combination of elements of the form
$\xi_1^{b_1}\cdots \xi_{\theta}^{b_{\theta}}\eta_1^{c_1}\cdots
\eta_{\theta}^{c_{\theta}}$ with $b_i\geq 0$ and $c_i\in\{0,1\}$.
Such an element takes $\y^{b_1N_1+c_1}\cdots \y_{\theta}^{b_{\theta}N_{\theta}
+c_{\theta}}\ot 1$ to a nonzero scalar multiple of $1$ and all other
$S$-basis elements of $K_{\sum (2b_i+c_i)}$ to 0.
Recall that the dimension of $\coh^n(S,\ku)$ is $\binom{n+\theta -1}{\theta -1}$;
consequently, the elements 
$\xi_1^{b_1}\cdots \xi_{\theta}^{b_{\theta}}\eta_1^{c_1}\cdots
\eta_{\theta}^{c_{\theta}}$ form a $\ku$-basis for $\coh(S,\ku)$.

For the last statement, we may induct on the length of
a composition series of $M$, and it suffices to prove 
the statement in the case that $M$ is a simple module. 
The generators of $\toba_{\bq}$ are all nilpotent, and
so the only simple module is the trivial module $\ku$, for
which the statement is clear. 
\epf

\subsection{Cohomology of graded algebras with convex PBW basis} \label{subsec:gral-lemma}

Here we consider a graded connected algebra $R = \oplus_{n\in \N_0} R^n$
with a finite PBW-basis $B = B(\{1\},\roots,<,h)$; that is 
$\roots$ is a finite subset of $R$ with $r = \vert \roots \vert$ elements,
$<$ a total order on $\roots$ 
(with a numeration $\roots = \{x_1,x_2,\dots\}$ such that $i<j$ if and only if $x_i<x_j$)
and a function $h: \roots \to \N \cup \{ \infty \}$, $x_i \mapsto N_i$ ($h$ is called the height), such that
\begin{align*}
B &= \big\{x_r^{e_r}x_{r-1}^{e_{r-1}}\dots x_2^{e_2}x_1^{e_1}:\, 
0 \leq e_i<N_i \big\}
\end{align*}
is a $\ku$-basis of $R$. Let $\I = \{1, \dots, r \}$ and $\I' = \{i\in \I: N_i < \infty \}$.
We assume that the elements of $\roots$ are homogeneous: 
$x_j \in R^{d_j}$,  $d_j \in \N$.
We set
\begin{align*}
\deg b &= (e_1,\dots, e_r, \sum_j e_jd_j) \in \N_0^{r + 1}, & 
b &= x_r^{e_r}x_{r-1}^{e_{r-1}}\dots x_2^{e_2}x_1^{e_1} \in B.
\end{align*}
Let $\preceq$ be the lexicographical order, reading from the right, on $\N_0^{r + 1}$. 
We consider the $\N_0^{r + 1}$-filtration on $R$ given by
\begin{align*}
R_f &= \langle b \in B: \deg b \preceq f \rangle,& f &= (f_1, \dots, f_{r+1}) \in \N_0^{r+1}.
\end{align*}
Inspired by \cite{DCK}, we also assume that 
the PBW-basis $B$ is \emph{convex}, i.e.~$(R_f)_{f \in \N_0^{r+1}}$ is an algebra filtration.
It can be shown that the PBW-basis $B$ is convex if and only if 
\begin{enumerate}[leftmargin=*,label=\rm{(\alph*)}]
\item\label{item:PBW-convex-1} for every  $i, j\in \I$ with $i <j$, there exists $q_{ij} \in \ku$ such that
\begin{align}\label{eq:PBW-convex-sisj}
x_ix_j &= q_{ij} x_jx_i + \sum_{f \prec \deg x_{i} + \deg x_{j}} R_f;
\end{align}
\item\label{item:PBW-convex-2} for every  $i\in \I'$,
\begin{align}\label{eq:PBW-convex-prv}
x_i^{N_i} & \in \sum_{f \prec N_i \deg x_{i}} R_f.
\end{align}
\end{enumerate} 
See \cite{AAH-infinite}. We call this the PBW-filtration.
Assume that in \eqref{eq:PBW-convex-sisj}, $q_{ij}\neq 0$ for all $i<j$. Then the associated graded algebra 
$S := \gr R$ is a  quantum linear space, i.e.~it is presented by generators $y_i$ (the  class of $x_i$)
and relations 
\eqref{eq:qls1}, \eqref{eq:qls2}.

The PBW basis gives rise to the Anick resolution $C_*(R)$ with $C_n(R) = \Vt(n) \otimes R$ computing 
$\coh(R,k)$ and $\coh(R,M)$ for an $R$-module $M$ as in \S\ref{subsec:Anick}.  
By construction of the resolution, $\x_i^{\ell_iN_i}$ (or rather $\x_i^{\ell_iN_i} \ot 1$) 
are homogeneous elements  of cohomological degree $2\ell_i$. Moreover, again by construction, 
$\x_i^{\ell_iN_i}$ is a basis element in $\Vt(2\ell_i)$. We denote by $(\x_i^{\ell_iN_i})^*$ the 
corresponding element in the dual basis of  $\Hom_\ku(\Vt(2\ell_i),k) \cong \Hom_R(C_{2\ell_i}(R), \ku)$, 
that is the function which evaluates to $1$ on  $\x_i^{\ell_iN_i}$ and to $0$ on all other basis elements. 
By construction, these are homogeneous elements (or cochains) in the complex $\Hom_R(C_*(R), \ku)$ which computes $\coh(R, \ku)$. 

In the next theorem we prove that if all cochains $(\x_i^{\ell_iN_i})^*$ are cocycles, then $\coh(R, \ku)$ has fgc.

\begin{theorem}\label{lemma:reduction-to-cocycles} Let $R$ be a graded connected 
algebra with a finite convex PBW-basis satisfying all of the assumptions above; let $C_*(R)$ be the Anick resolution of $R$. 
Suppose that there exist positive integers $\ell_i$
for any $i$ with $N_i<\infty$, such that the cochains $\left(\x_i^{\ell_iN_i} \right)^*$ are
cocycles in $\Hom(C_*(R),\ku)$, that is, represent elements in $\coh(R, \ku)$. Then
$\coh(R, \ku)$ is finitely generated
and $\coh(R,M)$ is finitely generated as a module over
$\coh(R,\ku)$ for any finitely generated $R$-module $M$. 
\end{theorem}

\pf Observe that $\gr R$ is a quantum linear space. 
By Proposition~\ref{prop:qci}, 
the cohomology $\coh( \gr R,\ku)$ is finitely generated over its subalgebra
generated by $\xi_{i}^{\ell_i}$ for all $i\in \I'$, 
since it is generated by all
$\xi_{i},\eta_{i}$.
Since the Anick resolution for $R$ is compatible with the PBW-filtration on $R$,
there exists an associated spectral sequence $E$ convergent to $\coh(R,\ku)$ whose 
$E_1$-page is $\coh(\gr R,\ku)$; see Theorem~\ref{may}.
Moreover, the cochains $(\x_i^{\ell_iN_i})^*$ are the images of
$\xi_i^{\ell_i}$ in the spectral sequence (see the proof of Proposition~\ref{prop:qci})
and so by assumption, the $\xi_{i}^{\ell_{i}}$ are permanent cycles.
Thus the hypotheses of Lemma~\ref{Lch2} are
satisfied, and, hence, $\coh(R,\ku)$ is left Noetherian.
(That is, $\gr \coh(R,\ku)$ is Noetherian, from which it follows
that $\coh(R,\ku)$ is Noetherian.)
Finite generation follows from Lemma~\ref{lema:hilbert-invariants} 
taking $K$ to be a trivial algebra there. 


If $M$ is a finitely generated $R$-module, 
then since the Condition~\ref{as:resol} is satisfied, 
Theorem~\ref{may} implies that $\coh(R,M)$ is finitely generated
as an $\coh(R,\ku)$-module. 
\epf

The following corollary is immediate since Nichols algebras of diagonal type have convex $PBW$ bases. 

\begin{cor}\label{cor:Nichols-fgc} Let $\toba_q$ be a finite-dimensional Nichols algebra of diagonal type. 
If $\toba_q$ satisfies Condition~\ref{assumption:intro-combinatorial}, then it has fgc. \qed
\end{cor}


\section{Cohomology of the Drinfeld double}\label{subsec:drinfeld-double}
In this section we prove that if the bosonization of a Nichols algebra of diagonal type has fgc then so does its Drinfeld double.

We briefly recall the general definition of a twisting map
$\tau$ for two algebras $A$ and $B$: 
Let $\tau: B\ot A\rightarrow A\ot B$ be a bijective $\ku$-linear map for which
$\tau(1_B\ot a) = a\ot 1_B$, $\tau(b\ot 1_A)=1_A\ot b$ for all
$a\in A$ and $b\in B$, and the following compositions of maps
from $B\ot B\ot A\ot A$ to $A\ot B$ are equal:
\[
   \tau \circ (m_B\ot m_A) = (m_A\ot m_B)(1\ot \tau\ot 1)
    (\tau\ot \tau)(1\ot \tau \ot 1),
\]
where $m_A$ (respectively, $m_B$) denotes multiplication on $A$
(respectively, on $B$), and 1 denotes an identity map. 
The {\em twisted tensor product algebra} $A\ot_{\tau}B$ is $A\ot B$
as a vector space, and its multiplication is the composition
$(m_A\ot m_B) (1\ot \tau\ot 1)$. 

If $A$ and $B$ are Hopf algebras, we say that $\tau$ is a Hopf twisting if 
$A\ot_{\tau}B$ is a Hopf algebra with coalgebra structure being the usual 
tensor product of coalgebras (no twisting), and $A$, $B$ are Hopf subalgebras. 
The augmentation map is $\epsilon_A \ot \epsilon_B : A \ot_\tau B\rightarrow \ku$.


We assume that the Hopf twisting $\tau$ is compatible with coradical filtrations, 
that is for 
\[ 
    C^A_0\subset C^A_1\subset\cdots \ \ \ \mbox{ and } \ \ \
   C^B_0\subset C^B_1\subset\cdots
\]
the coradical filtrations of $A$ and $B$, 
we have 
\[\tau(C^B_b \ot C^A_a)  \subset \sum\limits_{r+s\leq a+b} C^A_r \ot C^B_s.\] 
Then the associated graded space 
$\gr (A \ot_\tau B)$ is again a Hopf algebra.

We will need a special case of the twisting construction to apply it to Nichols algebras:
Assume that $A$ and $B$ are graded by 
abelian groups $\Gamma$ and $\Gamma'$. 
Let $t:\Gamma\times\Gamma' \rightarrow \ku^{\times}$ be
a bicharacter (that is, it induces a homomorphism
$\Gamma\ot_{\Z}\Gamma'\rightarrow \ku^{\times}$ of abelian groups). 
Define $\tau:B\ot A\rightarrow A\ot B$ by
$\tau (b\ot a) = t(|a|, |b|) a\ot b$ for all homogeneous $a\in A$,
$b\in B$, where $|a|\in\Gamma$, $|b|\in \Gamma'$ denote grading. 
In order to distinguish a twisted tensor product algebra
$A\ot_{\tau}B$ for which the twisting $\tau$ is defined by
a bicharacter $t$ in this way, we will write $A\ot ^t B$ for
this twisted tensor product algebra.

Due to the following result of Bergh and Oppermann~\cite[Theorem~3.7]{Bergh-Oppermann}, 
the cohomology of $A \ot^t B$ can be computed. 

\begin{theorem} \label{lem:BO}
Let $A$ and $B$ be augmented algebras graded by abelian groups $\Gamma$
and $\Gamma '$.
Let $t$ be a bicharacter on $\Gamma\times \Gamma '$. 
   There is a twisting map $\hat{t}$, induced by the bicharacter $t$, for which
   \[ \coh( A\ot^t B , \ku)\cong \coh(A,\ku)\ot^{\hat{t}} \coh(B,\ku) . \]
\end{theorem}

Let $A=R\# \ku \Gamma$ and $B=(R\# \ku \Gamma)^{\#}$ with $R = \toba (V)$.
Let $D = D(A)$ be the Drinfeld double of $A$. 
Since $A$ and $B$ are subalgebras of $D=D(A)$ and, as a
vector space, $D$ is isomorphic to $A\ot B$,
there is an isomorphism of algebras,
\[
  D \cong A\ot_{\tau} B,
\]
where $A\ot_{\tau}B$ is a twisted tensor product algebra whose twisting map
$\tau: B\ot A\rightarrow A\ot B$ is defined to correspond to multiplication in $D$. The augmentation map on $D(A)$ is
\[
\epsilon_D=  \epsilon_A \ot \epsilon_B.
\] 

Recall that $A$ and $B$ are both coradically graded.
With respect to the coradical filtration on $D$, there is indeed an 
isomorphism 
\[
    \gr D \cong A\ot^t B
\]
for some bicharacter $t$ on grading groups of $A$ and $B$.
The bicharacter $t$ is
defined by the braiding $c$ and the group action. 
See, for example,~\cite{Beattie}.
Hence, as a consequence of Theorem~\ref{lem:BO}, $\coh(\gr D,\ku)$ can be computed in terms of $\coh(A,\ku)$ and $\coh(B,\ku)$. 
Moreover, since $\gr D$ is a Hopf algebra, its cohomology is graded commutative so the bicharacter $\hat t$ takes values $\pm1$. 

To show that $D$ has finitely generated cohomology we will establish 
that $\coh(D,\ku)$ has ``enough" cocycles and apply Evens Lemma (\ref{Lch2}). 

Let $P_{\bu}$, $Q_{\bu}$ be bar resolutions of $\ku$
as an $A$-module and as a $B$-module.
Then as in~\cite{Bergh-Oppermann,Shepler-Witherspoon} for left modules, 
we may form the twisted tensor product resolutions
$P_{\bu}\ot_{\tau} Q_{\bu}$   and $P_{\bu}\ot^t Q_{\bu}$ of $\ku$
as a right $A\ot_{\tau}B$-module and a right $A\ot^t B$-module, respectively.
We recall here briefly this construction, and translate to right modules:
As a complex of vector spaces, each of $P_{\bu}\ot_{\tau}Q_{\bu}$
and $P_{\bu}\ot^t Q_{\bu}$ is simply $P_{\bu}\ot Q_{\bu}$,
and it remains to define the $A\ot_{\tau}B$- and $A\ot^t B$-module
structures on each vector space $P_i\ot Q_j$.
We will do this for $A\ot_{\tau} B$, and $A\ot^t B$ is similar.
For each $j$, define $\tau_j:Q_j\ot A\rightarrow A\ot Q_j$ by iterating $\tau$.
The right module structure is defined by the following composition of maps:

\begin{equation}\label{eqn:mod-str}
\begin{xy}*!C\xybox{
\xymatrix{
P_i\!\ot \! Q_j \ot\! A\! \ot\!  B 
\ar[rr]^{1\ot\, \tau_{j}\ot 1}
&& P_i \! \ot \! A\ot \! Q_j\! \ot\!  B 
\ar[rr]^{\hspace{.5cm}\rho_{P_i}\ot \rho_{Q_j}}
&& P_i\ot Q_j  ,
}}
\end{xy}
\end{equation}
where $\rho_{P_i}$ and $\rho_{Q_j}$ denote the module structure maps.

\begin{lemma}\label{lema:extend}
Let $f \in \Hom_A(P_i,\ku)$ be a cocycle. Then $f$ extends to a cocycle representing an element in $H^*(D,\ku)$. A similar statement holds for $\Hom_B(Q_j,\ku)$. 

\end{lemma}

\begin{proof}
The first statement will follow from the construction of the resolution
$P_{\bu}\ot_{\tau} Q_{\bu}$ and the definitions.
The second statement involves switching the order of $P_{\bu}$ and $Q_{\bu}$.
This asymmetry in the proof is due to the asymmetry of choosing
to work with right modules instead of left modules. 

Let $f \in \Hom_A(P_i,\ku)$ be a cocycle. 
We first claim that $f\ot \epsilon_B$, as a function on $P_i\ot_{\tau} Q_0 = P_i\ot_{\tau} B $, 
is an $A\ot_{\tau}B$-module homomorphism. 

\vspace{0.1in}

Consider 
\begin{equation}\label{eqn:f-epsilon}
   (f\ot \epsilon_B) ((x\ot y)\cdot (a\ot b))
\end{equation}
where $x\in P_i$, $a\in A$, and $b,y\in B$.
Expression~(\ref{eqn:f-epsilon}) can be evaluated by first
applying $1\ot \tau\ot 1$ to $x\ot y\ot a\ot b$, then applying
$f\ot\epsilon_A\ot\epsilon_B\ot\epsilon_B$ to the result,
since $f$ is an $A$-module homomorphism and $A$ acts trivially on $\ku$.
We wish to show this is equal to
\begin{equation}\label{eqn:f-epsilon2}
  ((f\ot \epsilon_B) (x\ot y))\cdot (a\ot b) = (f\ot \epsilon_B)(x \ot y) \epsilon_{D(A)}(a \ot b) =  f(x)\epsilon_B(y)
      \epsilon_A(a)\epsilon_B(b) .
\end{equation}
It suffices to show this for all $y$ from a set of generators
of $B$, for all $a$ from a set of generators of $A$,
and for all $x\in P_i$ and $b\in B$. 
If either $a$ or $y$ is an element of the field $\ku$, 
the expression~(\ref{eqn:f-epsilon}) is equal to
\begin{equation*}
   (f\ot \epsilon_B)((x\cdot a) \ot (y\cdot b))
   = f(x\cdot a) \epsilon_B(y\cdot b) = f(x)\epsilon_B(y)\epsilon_A(a)
   \epsilon_B(b) ,
\end{equation*}
as desired. Now assume that $a$ and $y$ are generators in the kernel
of the augmentation maps for $A$ and $B$, respectively,
so each is either a root vector or a difference of group elements.
The case where either is a difference of group elements 
is straightforward since applying $\epsilon$ to either of the
middle two factors will yield 0 after applying $\tau$.
Now assume that $a$ and $y$ are both root vectors. Then by Proposition~\ref{prop:double-relations} we have that in $D(A)$,
\[
     ya = \lambda ay +\kappa (1 -  g \chi)
\]
where $\lambda,\kappa$ are scalars, $g\in \Gamma$, $\chi \in \widehat\Gamma$.
Then 
\begin{align*}
  (f\ot\epsilon_B)((x\ot y ) \cdot (a\ot b)) =\\
  (f\ot\epsilon_A\ot\epsilon_B\ot\epsilon_B)(\lambda x \ot a \ot y \ot b + \kappa x \ot 1 \ot 1 \ot b - \kappa x \ot g \ot \chi \ot b) =  \\
  \lambda f(x)\epsilon_A(a)\epsilon_B(y)\epsilon_B(b) + \kappa (f(x)\epsilon_B(b) - f(x)\epsilon_A(g)\epsilon_B(\chi)\epsilon_B(b)  =\\
  0 
  \end{align*}
  where the first term disappears since $\epsilon_B(y)=0$ and the second two terms cancel out since $\epsilon_A(g) = \epsilon_B(\chi) =1$. 
We also have $f(x)\epsilon_B(y)
      \epsilon_A(a)\epsilon_B(b)  = 0$ since $\epsilon_B(y)=0$.
Therefore the expressions~(\ref{eqn:f-epsilon})
and~(\ref{eqn:f-epsilon2}) are equal, as desired.
It follows that $f\ot\epsilon_B$ is an $A\ot_{\tau}B$-module homomorphism,
that is,  $f\ot\epsilon_B \in \Hom_{A\ot_{\tau}B} (P_i\ot Q_0 , \ku)$.

By hypothesis, $0=d_{i+1}^*(f) = f d_{i+1}$ where $d_{i+1} : P_{i+1}\rightarrow P_i$
is the differential. Letting $d$ denote the differential on $P_{\bu}\ot_{\tau}Q_{\bu}$,
\[
    d^*(f\ot \epsilon_B) = (f\ot\epsilon_B) (d_{i+1}\ot 1
   + (-1)^i \ot d_1) = fd_{i+1}\ot\epsilon_B+ (-1)^i f\ot \epsilon_B d_1 = 
   0 .
\]
Therefore $f\ot \epsilon_B$ is a cocycle representing an
element of $\coh (D,\ku)$. 

Now let $g\in \Hom_B(Q_j,\ku)$ be a cocycle representing an
element of $\coh(B,k)$. 
Note that $A\ot_{\tau}B\cong B\ot_{\tau^{-1}}A$.
Let $Q_{\bu}\ot_{\tau^{-1}} P_{\bu}$ be the twisted
tensor product resolution corresponding to this inverse twisting $\tau^{-1}$.
By the above arguments, $g\ot \epsilon$ is a cocycle representing
an element of $\coh(A\ot_{\tau}B,k)$.
Since $P_{\bu}\ot_{\tau}Q_{\bu}$ is quasi-isomorphic to $Q_{\bu}\ot_{\tau^{-1}}P_{\bu}$
(in fact, a comparison map is given by iterating $\tau$),
there is a cocycle $g'$ defined on the resolution $P_{\bu}\ot_{\tau}Q_{\bu}$
corresponding to $g\ot\epsilon$ on $Q_i\ot_{\tau^{-1}} P_0$.
Note that in general $g'$ will not equal $\epsilon_A\ot g$, due to the twisting.

\end{proof}



\begin{theorem} \label{thm:fingencoh-drinfeld-double}
If the Nichols algebra $R = \toba(V)$ and its dual $R^\#$ have fgc, then 
the Drinfeld double $D = D(\toba(V) \# \ku \Gamma)$ of the bosonization $\toba(V) \# \ku \Gamma$ has
fgc. 
\end{theorem}

\begin{proof}
Let $A = R  \# \ku \Gamma$ and $B = A^\#$. By hypothesis, $A$ and $B$ have finitely generated cohomology,
specifically, the cohomology $\coh(A,\ku)$ is a finite module over
a finitely generated commutative subalgebra, and similarly
for $\coh(B,\ku)$.
Choose generators of these commutative subalgebras and representative
cocycles on $P_{\bu}$ and $Q_{\bu}$; we will use these in
a spectral sequence argument in combination with Theorem~\ref{lem:BO}.

As a consequence of the Theorem~\ref{lem:BO}, $A\ot^t B$ has finitely generated cohomology
since both $A$ and $B$ do. 
Of necessity, since $A\ot^t B$ is also a Hopf algebra, $\coh(\gr (A\ot_{\tau}B),\ku)$
is graded commutative, and so $\hat{t}$ will in the end only take values $\pm 1$.
We will next show that $A\ot_{\tau}B$ also has finitely
generated cohomology. 
This relies on existence of needed cocycles.
Let $f\in\Hom_A (P_i,\ku)$ be a cocycle representing an element
of $\coh(A,\ku)$.
Then by Lemma~\ref{lema:extend},  $f$ extends to a 
cocycle representing an element of $\coh(A\ot_{\tau}B,\ku)$.
A similar statement holds for $\Hom_B(Q_j,\ku)$.

Next note that 
the filtration on $A\ot_{\tau}B$ induces a filtration on the resolution
$P_{\bu}\ot_{\tau}Q_{\bu}$.
Let $E$ be the corresponding spectral sequence. 
Page $E_1$ is $\coh(A\ot^t B, \ku)$, which by Theorem~\ref{lem:BO} is
isomorphic to  $\coh(A,\ku)\ot^{\hat{t}} \coh(B,\ku)$.
The cohomology $\coh(A\ot^{\tau}B,\ku)$ is the homology of the total complex of the
bicomplex
\[
  \Hom_{A\ot_{\tau}B} (P_{\bu}\ot_{\tau} Q_{\bu} , \ku ) .
\]
By the above argument, a cocycle $f\in \Hom_A(P_i,\ku)$ representing
an element of $\coh^i(A,\ku)$ may be extended to a cocycle
$f\ot\epsilon\in\Hom_{A\ot_{\tau}B}(P_i\ot_{\tau}Q_0,\ku)$ representing
an element of $\coh(A\ot^{\tau}B,\ku)$.
This is thus a permanent cocycle in the spectral sequence $E$.
Moreover, it corresponds to $f\ot \epsilon$, this time
representing an element of the $E_1$-page $\coh(A,\ku)\ot^{\hat{t}}\coh(B,\ku)$.
Similarly, a cocycle $g\in\Hom_B(Q_j,\ku)$ representing an element of
$\coh^j(B,\ku)$ may be extended to a cocycle $g'\in\Hom_{A\ot_{\tau}B}
(P_0\ot_{\tau}Q_j,\ku)$ representing an element of $\coh(A\ot_{\tau}B,\ku)$.
Thus we obtain, for each chosen generator of $\coh(\gr A\ot_{\tau}B,\ku)$, a permanent
cocycle in the spectral sequence.
Applying the spectral sequence Lemma~\ref{Lch2}, 
since $\coh(A\ot^t B , \ku)$ is a finite
module over a finitely generated (commutative) subalgebra,
$\coh(A\ot_{\tau}B, \ku)$ is finitely generated.
(A commutative subalgebra can be found by taking high enough
powers of the chosen generators since the defining parameters
and thus also the values of the bicharacter $\hat{t}$
are all roots of unity.)

Now let $M$ be a finitely generated $A\ot_{\tau}B$-module.
Then $\coh(A\ot_{\tau}B,M)$ is a graded module over $\coh(A\ot_{\tau}B,\ku)$.
The coradical filtration on $A\ot_{\tau}B$ induces a filtration on
the bar resolution $K_{\bu} $ of $\ku$ as $A\ot_{\tau}B$-module and thus on 
$\Hom_{A\ot_{\tau}B}(K_{\bu}, M)$. 
Let $\widetilde{E^*}$ be the corresponding spectral sequence.
Arguing as in~\cite[\S I.9.13]{Jantzen}, we get a spectral sequence
\[
  \widetilde{E}^*_M = \coh (\gr A\ot_{\tau}B , \ku) \ot M 
   \Rightarrow \coh (A\ot_{\tau}B, M) .
\]
which is a module over the spectral sequence $\widetilde{E^*}$.
By Lemma~\ref{Lch2}, $\coh(A\ot_{\tau}B,M)$ is finitely generated over $\coh(A\ot_{\tau}B,\ku)$.
\end{proof}


\vspace{0.5in}
\part{Permanent cocycles for Nichols algebras of diagonal type}\label{part:nichols}

In this Part we deal with 

\medbreak
\textbf{Condition \ref{assumption:intro-combinatorial}.}
\emph{Let $U$ be a braided vector space of diagonal type whose Nichols algebra is finite-dimensional.
For every positive root $\gamma \in \varDelta_+^{U}$, there exists
$L_{\gamma} \in \N$ such that
the cochain $\left(\x_{\gamma}^{L_{\gamma}} \right)^*$ is a
cocycle, that is, represents an element in $\coh(\toba(U), \ku)$.
}

\medbreak
We shall prove that Condition \ref{assumption:intro-combinatorial} holds for one representative $U$ of each Weyl-equivalence class 
in the classification of \cite{H-classif}. By Theorem \ref{lemma:reduction-to-cocycles} this shows that $\toba(U)$ has fgc
and as explained in \S \ref{subsec:intro-nichols-fgc}, this implies Theorem \ref{th:fingencoh-nichols-diagonal}.
We argue also by induction on $\dim U$; in other words we often assume that the root $\gamma$ has full support, i.e.~$\supp \gamma = \I$.
Towards this, we choose the representative $U$ in the Weyl-equivalence class in such a way that Condition \ref{assumption:intro-combinatorial}
was already verified for any proper subdiagram.

\medbreak
We discuss the strategy in Section \ref{sec:nichols-strategy}, a summary been given in \S\ \ref{subsec:summary}.
Proofs of the technical  statements in this Section are deferred to Section \ref{sec:computational-lemmas}.
We proceed case by case in Sections \ref{sec:classical} (classical Cartan and super types), 
\ref{sec:exceptional} (exceptional Cartan and super types) and \ref{sec:discrete}
(modular types $\Bgl(4)$ and $\Brown(2)$).
The remaining Nichols algebras of  diagonal type in the classification
are dealt with in \cite{AAPPW}.

\section{The strategy}\label{sec:nichols-strategy}
\subsection{The setup}
Let $\bq$ be the braiding matrix of $U$ and denote by $\toba_{\bq}$ the corresponding Nichols algebra as before. 
For $\delta \in \varDelta_+^{\bq}$ recall that $N_{\delta} = \ord q_{\delta \delta}$. 
Recall that the set of $n$-chains, $n\in\N$, is given by
\begin{align*}\tag{\ref{eq:chains-Nichols}}
\chain(n) & = \left\{ \x_{\delta_1}^{f_{\delta_1}(n_1)} \x_{\delta_2}^{f_{\delta_2}(n_2)} \dots \x_{\delta_m}^{f_{\delta_m}(n_m)} \, : \, \sum n_j=n  \right\},
\end{align*}
where for $\delta\in \varDelta_+^{\bq}$ we introduce $f_\delta: \ \N_0\to\N_0$ by
\begin{align*}
f_{\delta}(2k)&= N_\delta k,
&
f_{\delta}(2k+1)&=N_\delta k+1, 
& 
k&\in\N_0.
\end{align*}

\medbreak The starting point is the following straightforward observation.

\begin{remark} \label{obs:main}
Let $\gamma \in \varDelta_+^{\bq}$ and $L\in \N$. 
\begin{enumerate}[leftmargin=*,label=\rm{(\alph*)}]
\item\label{item:obs-chains1}
$\x_\gamma^L$ is a chain if and only if $L$ is of the form $\ell N_\gamma$ or $\ell N_\gamma + 1$, for some $\ell \in \N$.

\item\label{item:obs-chains2} A cochain  $(\x_\gamma^L)^*$ is a cocycle if and only if for any chain $c \in \chain(n+1)$ such that $d(c \otimes 1)  \in \Vt(n) \otimes \toba_{\bq}$,
when written as a linear combination of basis elements, the term $\x_\gamma^L \otimes 1$ has zero coefficient.
\end{enumerate}
\end{remark}

Because of Theorem \ref{lemma:reduction-to-cocycles} and Remark \ref{obs:main} \ref{item:obs-chains1} we shall assume that $L  = \ell N_{\gamma}$ for $\ell \in \N$.

\medbreak
We reduce the set of chains $c \in \chain(n+1)$ to be considered in  \ref{item:obs-chains2}  using degree and grading constraints.
First, since the relations of $\toba_{\bq}$ are $\mathbb N_0^{\mathbb I}$-homogeneous by definition we have:
\begin{lemma}
\label{lem:grading} 
The differential of the Anick resolution preserves the $\mathbb N^{\mathbb I}$-grading. \qed
\end{lemma}

Let $c =  \x_{\delta_1}^{f_{\delta_1}(n_1)} \x_{\delta_2}^{f_{\delta_2}(n_2)} \dots \x_{\delta_m}^{f_{\delta_m}(n_m)} \in \chain(n+1)$ such that 
\begin{align*}
d(c \ot 1) &= \ldots + \lambda \x_\gamma^{\ell N_{\gamma}} \otimes 1 + \dots,& \lambda &\neq 0,
\end{align*}
 as a linear combination of basis elements. By Lemma~\ref{lem:grading} and since $\x_\gamma^{\ell N_{\gamma}}$ is a $2\ell$-chain, 
 we have the following constraints: 
\begin{align}
\label{eq:grading}
f_{\delta_1}(n_1) \delta_1 + \cdots + f_{\delta_m}(n_m) \delta_m &= \ell N_{\gamma} \gamma,
\\ \label{eq:degree}
n_1  + \cdots + n_m &= 2\ell +1.
\end{align}
Henceforth  we refer to the conditions \eqref{eq:grading} and \eqref{eq:degree} on the chains 
as the {\bf $\mathbb N_0^{\mathbb I}$-grading and homological degree constraints}.   
Writing the roots as linear combinations of simple roots, \eqref{eq:grading} and \eqref{eq:degree}
boil down to a system of equations on $\{n_1, \ldots, n_m, \ell \}$. 

\medbreak
Let $\gamma \in \varDelta_+^{\bq}$ and $\ell \in \N$. 
We summarize now the approaches to verify that  $(\x_{\gamma}^{\ell N_{\gamma}})^*$ 
is a cocycle using Remark \ref{obs:main} \ref{item:obs-chains2}. Thus we need to 
consider the chains $c \in \chain(n+1)$  in  \ref{item:obs-chains2}  up to degree and grading constraints.

\begin{itemize} [leftmargin=*] \renewcommand{\labelitemi}{$\circ$}
\item  We introduce integers $P_\gamma, Q_\gamma$ in \S \ref{subsec:gral-lemmas}. If  $N_\gamma > P_\gamma, Q_\gamma$, 
then  $(\x_\gamma^{N_\gamma})^*$ is a cocycle of degree $2$, by Lemma \ref{lem:largeN}. Here $\ell = 1$.
\end{itemize}

\medbreak We may assume that $\gamma$ is not simple, otherwise it is covered by the previous discussion.
Then there are $\beta,\delta\in \varDelta^{\bq}_{+}$ such that $\beta<\gamma<\delta$, $\gamma = \beta + \delta$.
If $N_{\gamma}$ is small, typically 2 or 3, then the condition $N_\gamma > P_\gamma, Q_\gamma$ does not hold. 

\begin{itemize} [leftmargin=*] \renewcommand{\labelitemi}{$\circ$}

\medbreak
\item  Assume that $N_{\gamma} = 2$ and  condition \eqref{eq:second-technique-2cocycles} holds. Then $(\x_{\gamma}^{N_{\gamma}})^*$
is a 2-cocycle by Lemma \ref{lem:second-technique-2cocycles}.  Again $\ell = 1$.

\medbreak
\item   
Condition \eqref{eq:second-technique-2cocycles} is about the relations between the root vectors 
$x_{\beta}$, $x_{\gamma}$, $x_{\delta}$. If it does not hold, then a finer analysis is needed. We summarize in Proposition
\ref{prop:cocycle-xgamma-N=2} all possible cases that we need to check in this setting when $N_{\gamma} = 2$.

\medbreak
\item   
Similarly we summarize in Proposition
\ref{prop:cocycle-xgamma-N>2} all possible  cases to check when $N_{\gamma} > 2$ and the condition $N_\gamma > P_\gamma, Q_\gamma$ does not hold.
\end{itemize}

Propositions \ref{prop:cocycle-xgamma-N=2} and \ref{prop:cocycle-xgamma-N>2} depend on several Lemmas
whose proof is deferred to Section \ref{sec:computational-lemmas}.

 \medbreak
 The techniques presented in Lemma \ref{lem:largeN}, Lemma \ref{lem:second-technique-2cocycles}, Proposition \ref{prop:cocycle-xgamma-N=2} and Proposition \ref{prop:cocycle-xgamma-N>2}  are applied to those Dynkin diagrams in the list of \cite{H-classif} with a continuous parameter in Sections
 \ref{sec:classical}, \ref{sec:exceptional} and \ref{sec:discrete}. The remaining diagrams in the classification are treated in Part III.


\subsection{Degree 2 cocycles}\label{subsec:gral-lemmas}
We discuss two techniques to get generators of degree $2$ in cohomology from root vectors. 
First we introduce $P_{\gamma}$ and $Q_{\gamma}$ that under suitable conditions imply the existence of the cocycles. 
\begin{definition}
Let $\gamma \in \varDelta^{\bq}_+$. We define
\begin{align*}
P_\gamma  &= \max \{p \in \N: \exists \text{ distinct } \delta_1,\delta_2, \delta_3 \in \varDelta^{\bq}_+ \text{  such that } \delta_1 + \delta_2 + \delta_3 = p \gamma\}, \\
Q_\gamma  &= \max \{q \in \N: \exists \text{ distinct } \delta_1,\delta_2 \in \varDelta^{\bq}_+ \text{  such that } N_{\delta_1}\delta_1 + \delta_2 = q \gamma\}.
\end{align*}
We set $P_\gamma=0$, respectively $Q_\gamma= 0$ if no such relation exists.
\end{definition}

\begin{remark}\label{rem:comptutation-P-Q-straightforward}
For any specific  $\gamma \in \varDelta^{\bq}_+$, the computation of $P_{\gamma}$, $Q_{\gamma}$ 
depends only on the combinatorics of the corresponding root system, see for example Lemma \ref{lema:non-simple-roots-Palfa}. 
We will leave these calculations for an interested reader in the later sections as they are straightforward to do in any specific case. 
\end{remark}

\begin{example} If  $\gamma$ is simple, then $P_\gamma = Q_\gamma = 0$. Also, if $\gamma$ is not simple, then
\begin{align}\label{eq:Palfa-notsimple}
P_{\gamma} \geq 2.
\end{align}
For, since $\gamma$ is not simple, $\gamma = \beta + \delta$ for some  distinct $\beta, \delta \in \varDelta^{\bq}_+$, 
hence $2\gamma = \gamma + \beta + \delta$.
\end{example}

\begin{lema}\label{lem:largeN} 
Let $\gamma \in \varDelta^{\bq}_+$.  If  $N_\gamma > P_\gamma, Q_\gamma$, 
then  $(\x_\gamma^{N_\gamma})^*$ is a cocycle of degree $2$.  
In particular, if $\gamma$ is simple, then $(\x_\gamma^{N_\gamma})^*$ is a cocycle of degree $2$.  
\end{lema} 

\begin{proof}  To show that $(\x_\gamma^{N_\gamma})^*$ is a cocycle, we use Remark \ref{obs:main} \ref{item:obs-chains2}.
That is,  we show 
that there is no chain $c \in C_3$ such that 
$x_\gamma^{N_\gamma} \otimes 1$ is among the terms, with nonzero
coefficient, of $d_3(c \otimes 1) \in C_2 \otimes A$. 

The chains in $C_3$ are of the form $\x_{\beta}^{N_\beta +1}$,  $\x_\beta^{N_\beta} \x_\delta$, $\x_\beta \x_\delta^{N_\delta}$ and $\x_\beta \x_\delta \x_\eta$.  
We consider these cases separately. 

For $c =  \x_{\beta}^{N_\beta +1}$, 
$d_3(c) = \x_{\beta}^{N_\beta} \otimes x_\beta$, which is not of the required form.

Let $c = \x_\beta^{N_\beta} \x_\delta$.  If $\x_\gamma^{N_\gamma} \otimes 1$ is present in $d_3(c \otimes 1)$, then Lemma~\ref{lem:grading} implies that we have a numerical relation 
\[N_\beta  \beta + \delta = N_\gamma \gamma,\]
which contradicts the assumption $N_\gamma>Q_\gamma $.  The case $\x_\beta \x_\delta^{N_\delta}$ is similar. 

Finally, if $c = \x_\beta \x_\delta \x_\eta$, and $\x_\gamma^{N_\gamma} \otimes 1$ is present in $d_3(c \otimes 1)$, then we have a relation 
\[
\beta + \delta +\eta = N_\gamma \gamma 
\]
which again contradicts the assumption $N_\gamma>P_\gamma$. 
\end{proof}

Because of the previous Lemma we need to compute $P_{\gamma}$ and $Q_\gamma$; this is simplified via the following result.
Let $\Wc$ be the Weyl groupoid of the Nichols algebra $\toba_{\bq}$, see \cite{H-Weyl grp} or \cite{AA17}.

\begin{lema}\label{lema:3roots}
Let $\delta_1$, $\delta_2$, $\delta_3 \in \varDelta_+^{\bq}$. 
Then there exist $w \in \Wc$ 
 and $\tau \in \mathbb S_{\theta}$ such that
\begin{align*}
w(\delta_i) &\in \varDelta_+^{\mathfrak p} \cap \left(\Z \gamma_{\tau(1)} + \Z \gamma_{\tau(2)} + \Z \gamma_{\tau(3)} \right),&
i &= 1,2,3, \end{align*}
for a suitable $\mathfrak p$.
\end{lema}
\pf See \cite[Theorem 2.3]{CH-rank 3}.
\epf

\begin{lema}\label{lema:non-simple-roots-Palfa} Assume that $\bq$ is of Cartan type and that $\gamma$ is not simple.

\begin{enumerate}[leftmargin=*,label=\rm{(\alph*)}]
\item\label{item:typeA} In types $A_{\theta}$, $D_{\theta}$ and $E_{\theta}$, 
we have $P_{\gamma} = 2$ and $Q_\gamma = 1$.

\item\label{item:typeB} In types $B_{\theta}$, $C_{\theta}$ and $F_{4}$,
we have $P_{\gamma} \leq 3$ and $Q_\gamma = 2$.

\item\label{item:typeG2} In type $G_{2}$, $P_{\gamma} \leq 4$ and $Q_\gamma \leq  3$.
\end{enumerate}

\end{lema} 

\pf 
\ref{item:typeA} Type $A$: 
Let $\gamma = \gamma_{i j}$ with $i < j$. Suppose that there exists $P \in \N$ such that
$P\gamma_{i j} = \gamma_{k \ell} + \gamma_{m n} + \gamma_{s u}$ with $k \leq m \leq s$ (and the three roots in the right are different). 
Then the coefficient of $\gamma_k$ in the right hand side is at most 3, so $P \leq 3$. 
If $P= 3$, then $k = m = s = i$ and $\ell$, $n$, $u$ are all different. If, say, $u$ is the largest of them,
then the coefficient of $\gamma_u$ in the right hand side is $1$, a contradiction. Thus $P_{\gamma} = 2$ by \eqref{eq:Palfa-notsimple}.
Next, suppose that there exists $P, t \in \N$ such that
$P\gamma_{i j} = t\gamma_{k \ell} + \gamma_{m n}$ (and the two roots in the right are different). Arguing as before, we see that $P \leq 1$, and \eqref{eq:Palfa-notsimple} applies.
Types $D, E$: this follows from Type $A$ and Lemma \ref{lema:3roots}.

\ref{item:typeB}  Type $B$: Let $\gamma_1, \ldots, \gamma_\theta$ be the simple roots with $\gamma_\theta$ the short root. Then the roots come in three flavors: $\gamma_i + \ldots + \gamma_{j}$, $i <j<\theta$, $\gamma_i + \ldots + \gamma_{\theta}$, and $\gamma_i + \ldots + \gamma_{j-1} + 2\gamma_j + \ldots + 2\gamma_\theta$,
$i <j \leq \theta$. Hence, the maximum $P$ is $3$ and  $P_\gamma \leq 3$; similarly, $Q_\gamma \leq 2$. 
Note that $P_{\gamma} =3$ and $Q_\gamma =2$ can occur, e.g. 
\begin{align*}
(\gamma_{\theta-2} + \gamma_{\theta-1}) + (\gamma_{\theta-2} + \gamma_{\theta-1} + \gamma_{\theta}) + (\gamma_{\theta-2} + \gamma_{\theta-1} + 2\gamma_\theta) = 3(\gamma_{\theta-2} + \gamma_{\theta-1} + \gamma_\theta).
\end{align*}

Type $C$:
The coefficient of $\gamma_\theta$ in $\gamma$ is 0 or 1, but in the former, $\gamma$ belongs to a 
sub-diagram of type $A_{\theta - 1}$ that was already settled. Looking at the coefficient of
$\gamma_\theta$ in both sides of $\delta_1 + \delta_2 + \delta_3 = P \gamma$, we conclude that $P \leq 3$.
Similarly, $Q_\gamma \leq 2$.
Note that $P_{\gamma} =3$ and $Q_\gamma =2$ can occur, e.g. 
\begin{align*}
(2\gamma_{\theta-2} + 2\gamma_{\theta-1} + \gamma_\theta) + (\gamma_{\theta-2} + \gamma_{\theta-1} + \gamma_{\theta}) + \gamma_\theta = 3(\gamma_{\theta-2} + \gamma_{\theta-1} + \gamma_\theta).
\end{align*}

Type $F$: this follows from Types $B$ and $C$ and Lemma \ref{lema:3roots}.

\ref{item:typeG2} By inspection.
\epf

By Lemma \ref{lem:largeN} we may assume $\gamma$ is not simple. If $N_{\gamma}=2$,
then the Lemma does not apply. 
The second technique provides an explicit computation of the differential of a suitable chain in this setting.
Let $s = s_n$ be as in \eqref{eq:Anick-def-sn}.

\begin{lemma}\label{lem:second-technique-2cocycles} Let $\gamma \in \varDelta_{+}^{\bq}$ be such that $N_{\gamma}=2$ and the following conditions hold:
\medbreak

\begin{enumerate}[leftmargin=*,label=\rm{(\alph*)}]
\item\label{item:second-technique-2cocycles-1} For all $\beta,\delta\in \varDelta^{\bq}_{+}$, $\beta<\delta$, $\gamma = \beta + \delta$,
\begin{align}\label{eq:second-technique-2cocycles}
x_{\beta}x_{\gamma}&=q_{\beta\gamma} x_{\gamma}x_{\beta}, &
x_{\gamma}x_{\delta}&=q_{\gamma\delta} x_{\delta}x_{\gamma}, &
q_{\beta\beta}&=q_{\delta\delta}.
\end{align}
\medbreak

\item\label{item:second-technique-2cocycles-2} If $\gamma_1,\gamma_2,\gamma_3\in \varDelta_{+}^{\bq}$ are three different roots, $\gamma_i\neq \gamma$, then $\gamma_1+\gamma_2+\gamma_3\neq 2\gamma$.
\medbreak

\item\label{item:second-technique-2cocycles-3} If $\gamma_1,\gamma_2\in \varDelta_{+}^{\bq}$, $\gamma_1 \neq \gamma_2$, then $N_{\gamma_1}\gamma_1+\gamma_2 \neq 2\gamma$.
\end{enumerate}
Then $(\x_{\gamma}^2)^*$ is a cocycle of degree two.
\end{lemma}

\pf
By Remark \ref{obs:main} we have to check that the coefficient of $\x_{\gamma}^2\ot 1$ in $d(c\ot 1)$ is zero for all 3-chains $c$ of degree $2\gamma$. By \ref{item:second-technique-2cocycles-2} and  \ref{item:second-technique-2cocycles-3} we have to deal with $c=\x_{\beta} \x_{\gamma} \x_{\delta}$, where $\beta,\delta\in \varDelta^{\bq}_{+}$, $\beta<\delta$, $\gamma = \beta + \delta$: Here we use the convexity to deduce that $\beta<\gamma<\delta$. 

Fix $\beta,\delta\in \varDelta^{\bq}_{+}$ such that $\beta<\delta$ and $\gamma = \beta + \delta$. By \eqref{eq:convex-combination},
\begin{align*}
x_{\beta}x_{\delta}&=q_{\beta\delta} x_{\delta}x_{\beta}
+ \Bsj x_{\gamma} + \sum_{\beta<\nu_1\le \dots \le \nu_k <\delta: \, \sum \nu_i=\gamma} \Bsj_{\nu_1,\dots,\nu_k} x_{\nu_k} \dots x_{\nu_1}
\end{align*}
for some $\Bsj x_{\gamma},\Bsj_{\nu_1,\dots,\nu_k}\in\Bbbk$.
Using the convexity again we see that if $\nu_1\le \dots \le \nu_k$ are such that $\sum \nu_i=\gamma$, then $\nu_1<\gamma<\nu_k$.

By definition of the differential on 2-chains and \ref{item:second-technique-2cocycles-1},
\begin{align*}
d( & \x_{\beta}\x_{\delta} \ot 1) =
\x_{\beta} \ot x_{\delta}
-q_{\beta\delta} \x_{\delta} \ot x_{\beta}
-\Bsj \x_{\gamma} \ot 1 
-\sum \Bsj_{\nu_1,\dots,\nu_k} \x_{\nu_k} \ot x_{\nu_{k-1}} \dots x_{\nu_1},
\\
d( & \x_{\beta}\x_{\gamma} \ot 1) = \x_{\beta} \ot x_{\gamma}
-q_{\beta\gamma} \x_{\gamma} \ot x_{\beta},
\quad 
d( \x_{\gamma}\x_{\delta} \ot 1) = \x_{\gamma} \ot x_{\delta}
-q_{\gamma\delta} \x_{\delta} \ot x_{\gamma}.
\end{align*}
Using these computations and  \ref{item:second-technique-2cocycles-1},
\begin{align*}
d( & \x_{\beta} \x_{\gamma} \x_{\delta} \ot 1  )= 
\x_{\beta}\x_{\gamma} \ot x_{\delta} - 
sd \big( \x_{\beta}\x_{\gamma} \ot x_{\delta} \big)
= \x_{\beta}\x_{\gamma} \ot x_{\delta} - 
s \big(\x_{\beta} \ot x_{\gamma}x_{\delta} 
-q_{\beta\gamma} \x_{\gamma} \ot x_{\beta}x_{\delta}
\big)
\\ & = 
\x_{\beta}\x_{\gamma} \ot x_{\delta} - 
s \left(q_{\gamma\delta} \x_{\beta} \ot x_{\delta}x_{\gamma} 
-q_{\beta\gamma} \x_{\gamma} \ot \big( q_{\beta\delta} x_{\delta}x_{\beta}
+ \Bsj x_{\gamma} + \sum \Bsj_{\nu_1,\dots,\nu_k} x_{\nu_k} \dots x_{\nu_1} \big)
\right)
\\ & = 
\x_{\beta}\x_{\gamma} \ot x_{\delta}
-q_{\gamma\delta} \x_{\beta}\x_{\delta} \ot x_{\gamma} 
-s \Big(
(q_{\gamma\delta}-q_{\beta\gamma}) \Bsj \x_{\gamma} \ot x_{\gamma}
-q_{\beta\gamma}q_{\beta\delta} \x_{\gamma} \ot  x_{\delta}x_{\beta}
\\ & \quad
+\sum \Bsj_{\nu_1,\dots,\nu_k} (q_{\gamma\delta}\x_{\nu_k} \ot x_{\nu_{k-1}} \dots x_{\nu_1}x_{\gamma}
- q_{\beta\gamma} \x_{\gamma} \ot  x_{\nu_k} \dots x_{\nu_1})
+q_{\beta\gamma}q_{\beta\delta}q_{\gamma\delta} \x_{\delta} \ot x_{\gamma}x_{\beta}\Big)
\\ & = 
\x_{\beta}\x_{\gamma} \ot x_{\delta}
-q_{\gamma\delta} \x_{\beta}\x_{\delta} \ot x_{\gamma} 
+(q_{\beta\gamma}-q_{\gamma\delta}) \Bsj \x_{\gamma}^2 \ot 1
+q_{\beta\gamma}q_{\beta\delta} \x_{\gamma}\x_{\delta} \ot  x_{\beta}
\\ & \quad
+\sum \Bsj_{\nu_1,\dots,\nu_k} q_{\beta\gamma} \x_{\gamma}\x_{\nu_k} \ot  x_{\nu_{k-1}} \dots x_{\nu_1}
\\ & \quad
-s \Big(
\sum \Bsj_{\nu_1,\dots,\nu_k} \x_{\nu_k} \ot 
\big( q_{\gamma\delta} x_{\nu_{k-1}} \dots x_{\nu_1}x_{\gamma} -
q_{\beta\gamma}q_{\gamma\nu_k} x_{\gamma}x_{\nu_{k-1}} \dots x_{\nu_1} \big)
\\ & \quad
- \sum \Bsj_{\nu_1,\dots,\nu_k} q_{\beta\gamma}
s(f_{x_{\gamma},x_{\nu_k}}x_{\nu_{k-1}} \dots x_{\nu_1}) \Big)
\end{align*}
Here $f_{x_{\gamma},x_{\nu_k}}= [x_{\gamma},x_{\nu_k}]_c-x_{\gamma}x_{\nu_k}+q_{\gamma\nu_k} \x_{\nu_k} \ot x_{\gamma}$. We claim that 
\begin{align*}
d( & \x_{\beta} \x_{\gamma} \x_{\delta} \ot 1  )= \x_{\beta}\x_{\gamma} \ot x_{\delta}
-q_{\gamma\delta} \x_{\beta}\x_{\delta} \ot x_{\gamma} 
+(q_{\beta\gamma}-q_{\gamma\delta}) \Bsj \x_{\gamma}^2 \ot 1
+q_{\beta\gamma}q_{\beta\delta} \x_{\gamma}\x_{\delta} \ot  x_{\beta}
\\ & \quad
+\sum \Bsj_{\nu_1,\dots,\nu_k} q_{\beta\gamma} \x_{\gamma}\x_{\nu_k} \ot  x_{\nu_{k-1}} \dots x_{\nu_1}.
\end{align*}

According with the previous computation we should prove that
$s$ annihilates
\begin{align}\label{eq:s-annihilated-terms}
&\x_{\nu_k} \ot 
\big( q_{\gamma\delta} x_{\nu_{k-1}} \dots x_{\nu_1}x_{\gamma} -
q_{\beta\gamma}q_{\gamma\nu_k} x_{\gamma}x_{\nu_{k-1}} \dots x_{\nu_1}\big), && s(f_{x_{\gamma},x_{\nu_k}}x_{\nu_{k-1}} \dots x_{\nu_1}).
\end{align}
For the elements on the right of \eqref{eq:s-annihilated-terms} we use that $s^2=0$. For the elements on the left of \eqref{eq:s-annihilated-terms}, due to the convexity of the PBW basis, 
$x_{\nu_{k-1}} \dots x_{\nu_1}x_{\gamma}$ and $x_{\gamma}x_{\nu_{k-1}} \dots x_{\nu_1}$ are linear combinations of products $x_{\mu_j} \dots x_{\mu_1}$, with $\nu_1 \le \mu_1\le \dots\le \mu_j \le \nu_{k-1}$ and $\mu_j \le \gamma$. Hence $\mu_j\le \nu_k$, so they are linear combinations of
\begin{align*}
\x_{\nu_k} \ot x_{\mu_j} \dots x_{\mu_1} = s( x_{\nu_k} x_{\mu_j} \dots x_{\mu_1} ),
\end{align*}
and we use again that $s^2=0$. Finally, using the claim and that
\begin{align*}
q_{\beta\gamma}-q_{\gamma\delta} =
q_{\beta\beta}q_{\beta\delta}-q_{\beta\delta}q_{\delta\delta}
\overset{\eqref{eq:second-technique-2cocycles}}{=} 0,
\end{align*}
the coefficient of $\x_{\gamma}^2 \ot 1$ is zero.
\epf

\subsection{Higher degree cocycles}
We now assume that we are not in the situations of \S \ref{subsec:gral-lemmas}.
We shall compute all chains $c \in \chain(2\ell+1)$ satisfying the degree and grading constraints
\eqref{eq:grading} and \eqref{eq:degree} and verify that the condition in  Remark \ref{obs:main}  \ref{item:obs-chains2}  is satisfied.
As before $\gamma \in \varDelta_+^{\bq}$ is fixed.

Let $f_\delta: \N_0\to\N_0$ be the function defined in \eqref{eq:power-root-vector-chain} for  
$\delta \in  \varDelta_+$. 

\subsubsection{$N_\gamma=2$} Here the constraints \eqref{eq:grading} and \eqref{eq:degree}
take the form
\begin{align}\label{eq:equation-roots-cocycle-N=2}
\sum_{\delta\in\varDelta_+} f_\delta(n_\delta) \delta &= L \gamma, &
\sum_{\delta\in\varDelta_+} n_\delta &= L +1.
\end{align}

In the following mega statement we collect all possible conditions that we may need to verify on $\gamma$   
to conclude that 
$(\x_\gamma^{L})^*$ is an $L$-cocycle (that is, a cocycle of degree $L$). We explain the scheme of the proof up to the specific computation of differentials
that is postponed to Section \ref{sec:computational-lemmas}.

\begin{prop}\label{prop:cocycle-xgamma-N=2}
Let $L = 2 \ell\in \N$ even. Assume that each
solution $(n_\delta)_{\delta\in\varDelta_+} \in \N_0^{\varDelta_+}$ of the equations \eqref{eq:equation-roots-cocycle-N=2}
is of one of the forms \ref{item:cocycle-xgamma-N=2-1}, \ref{item:cocycle-xgamma-N=2-2},
\ref{item:cocycle-xgamma-N=2-3}, \ref{item:cocycle-xgamma-N=2-4}, \ref{item:cocycle-xgamma-N=2-5}, \ref{item:cocycle-xgamma-N=2-6},
\ref{item:cocycle-xgamma-N=2-7}, \ref{item:cocycle-xgamma-N=2-8},
\ref{item:cocycle-xgamma-N=2-9} or \ref{item:cocycle-xgamma-N=2-10}.
Then   $(\x_\gamma^{L})^*$ is an $L$-cocycle.

\begin{enumerate}[leftmargin=*,label=\rm{(\Alph*)}]
\item\label{item:cocycle-xgamma-N=2-1}
$n_{\gamma}= L-1$, $n_{\alpha}=n_{\beta}=1$ and $n_{\varphi}= 0$ for the other $\varphi \in\varDelta_+$;
where  $\alpha, \beta \in \varDelta_+$ satisfy 
\begin{align}\label{eq:alfa-beta-N=2}
&\alpha < \beta,\quad  \alpha+\beta = \gamma,
\\
\notag &\text{the corresponding PBW generators satisfy \eqref{eq:diff2-hypothesis},}
\\ \label{eq:Csj-Dsj-N=2}
&\text{and $L$ satisfies } (L)_{-\frac{q_{\alpha \alpha}}{q_{\beta \beta}}} = 0.
\end{align}

\item\label{item:cocycle-xgamma-N=2-2} $n_{\gamma}= L-2$, 
$n_{\alpha}=n_{\beta}=n_{\delta}=1$ and $n_{\varphi}= 0$ for the other $\varphi \in\varDelta_+$.
where  $\alpha, \beta,\delta,\eta \in \varDelta_+$ satisfy 
\begin{align}\label{eq:alfa-beta-delta-eta-N=2}
&\alpha  < \eta <\gamma < \beta <\delta,   \quad \gamma + \eta =\alpha+\beta, \quad  \eta+\delta = \gamma,
\\
\notag &\text{the corresponding PBW generators satisfy \eqref{eq:diff-case2-hypothesis},}
\\ \label{eq:Csj-Dsj-N=2-2}
&\text{and $L$ satisfies } \coef{\alpha\beta\gamma}{L}:=\sum_{k=0}^L (-\widetilde{q}_{\alpha\gamma})^{k} (k+1)_{\widetilde{q}_{\beta\gamma}} = 0.
\end{align}

\item\label{item:cocycle-xgamma-N=2-3}  $n_{\gamma}= L-3$, 
$n_{\alpha}=2$, $n_{\beta}=n_{\delta}=1$ and $n_{\varphi}= 0$ for the other $\varphi \in\varDelta_+$ where $\alpha, \beta,\delta,\eta,\tau$  
satisfy
\begin{align}\label{eq:alfa-beta-delta-eta-tau-N=2}
&\begin{aligned}
&\alpha < \eta <\gamma < \tau < \beta <\delta, & \gamma + \tau &=\alpha+\beta, &  \eta+\delta &= \gamma, 
\\
&N_{\alpha}=2 & \alpha+\tau &= \gamma+\eta, & \eta+\beta &= 2\tau,
\end{aligned}
\\
\notag &\text{the corresponding PBW generators satisfy \eqref{eq:diff-case3-hypothesis},}
\\\label{eq:Csj-Dsj-N=2-3}
&\text{and $L$ satisfies }
(L)_{\widetilde{q}_{\gamma\alpha}\widetilde{q}_{\gamma\beta}} + \sum_{j=1}^{L-1} \coef{\alpha\tau\gamma}{j} = 0.
\end{align}

\item\label{item:cocycle-xgamma-N=2-4} $n_{\gamma}= L-2$, 
$n_{\alpha}=n_{\beta}=n_{\delta}=1$ and $n_{\varphi}= 0$ for the other $\varphi \in\varDelta_+$, where
$\alpha, \beta,\delta,\eta,\tau \in \varDelta_+$ satisfy
\begin{align}\label{eq:alfa-beta-delta-eta-tau-N=2-diff4}
&\alpha < \beta <\gamma < \tau < \eta <\delta, \quad \begin{aligned} 
&\alpha + \delta=\gamma+\eta, &  &\beta+\delta= \eta+\tau,
\\ &\gamma+\delta=2\eta+\tau, & &\eta+\beta=\gamma.
\end{aligned}
\\
\notag &\text{the corresponding PBW generators satisfy \eqref{eq:diff-case4-hypothesis},}
\\\label{eq:Csj-Dsj-N=2-4}
&\text{and $L$ satisfies }
\coef{\alpha\beta\gamma}{L} = 0.
\end{align}

\item\label{item:cocycle-xgamma-N=2-5} $n_{\gamma}= L-3$, 
$n_{\alpha}=n_{\beta}=n_{\delta}=n_{\eta}=1$ and $n_{\varphi}= 0$ for the other $\varphi \in\varDelta_+$, where $\alpha, \beta, \delta, \tau, \mu, \eta \in \varDelta_+$ satisfy
\begin{align}\label{eq:alfa-beta-delta-etc-N=2-diff5}
&\begin{aligned}
&\alpha < \beta < \delta < \gamma < \tau < \mu < \eta, & 
\alpha + \mu &=\gamma = \beta + \tau, &  \eta+\delta &= \gamma+\tau+\mu.
\end{aligned}
\\
\notag &\text{the corresponding PBW generators satisfy \eqref{eq:diff-case5-hypothesis},}
\\\label{eq:Csj-Dsj-N=2-5}
&\text{and $L$ satisfies }
\coeff{\alpha\beta\delta\gamma}{L}:=\sum\limits_{k=0}^{L-1} \widetilde{q}_{\delta\gamma}^{\, k} (k+1)_{\widetilde{q}_{\alpha\gamma}}
(k+2)_{\widetilde{q}_{\beta\gamma}}= 0.
\end{align}

\item\label{item:cocycle-xgamma-N=2-6} $n_{\gamma}= L-2$, 
$n_{\alpha}=n_{\beta}=n_{\delta}=1$ and $n_{\varphi}= 0$ for the other $\varphi \in\varDelta_+$, where $\alpha, \beta,\delta,\eta \in \varDelta_+$ satisfy
\begin{align}\label{eq:roots-N=2-case6}
& 
\begin{aligned}
&\alpha < \eta <\gamma < \beta <\delta,   & \gamma + \eta &=\alpha+\delta, &  \eta+\beta &= \gamma,
\end{aligned}
\\
\notag &\text{the corresponding PBW generators satisfy \eqref{eq:diff-case6-hypothesis},}
\\\label{eq:Csj-Dsj-N=2-6}
&\text{and $L$ satisfies }
\coef{\alpha\delta\gamma}{L}= 0.
\end{align}

\item\label{item:cocycle-xgamma-N=2-7} $n_{\gamma}= L-2$, 
$n_{\alpha}=n_{\beta}=n_{\delta}=1$ and $n_{\varphi}= 0$ for the other $\varphi \in\varDelta_+$, where $\alpha, \beta,\delta,\eta \in \varDelta_+$ satisfy
\begin{align}\label{eq:roots-N=2-case7}
& \begin{aligned}
\alpha & < \beta <\gamma < \eta <\delta,   & \gamma + \eta &=\beta+\delta, &  \eta+\alpha &= \gamma,
\end{aligned}
\\
\notag &\text{the corresponding PBW generators satisfy \eqref{eq:diff-case7-hypothesis},}
\\\label{eq:Csj-Dsj-N=2-7}
&\text{and $L$ satisfies }
\coef{-\delta\alpha\gamma}{L}= 0.
\end{align}

\item\label{item:cocycle-xgamma-N=2-8} $n_{\gamma}= L-2$, 
$n_{\alpha}=n_{\beta}=n_{\delta}=1$ and $n_{\varphi}= 0$ for the other $\varphi \in\varDelta_+$, where $\alpha,\beta, \delta,\eta, \tau,\mu,\nu \in \varDelta_+$ satisfy
\begin{align}\label{eq:roots-N=2-case8}
&\begin{aligned}
\alpha &< \tau < \beta <\gamma < \mu < \nu < \eta <\delta,  
&
\alpha + \delta &= \eta + \tau, &
\beta + \delta &= \nu + \gamma, 
\\ & &
\beta + \eta &= \mu + \gamma, &
\alpha + \nu &= \gamma,
\end{aligned}
\\
\notag &\text{the corresponding PBW generators satisfy \eqref{eq:diff-case8-hypothesis},}
\\\label{eq:Csj-Dsj-N=2-8}
&\text{and $L$ satisfies }
\coef{-\delta\alpha\gamma}{L}= 0.
\end{align}

\item\label{item:cocycle-xgamma-N=2-9} $n_{\gamma}= L-3$, 
$n_{\alpha}=n_{\beta}=n_{\delta}=n_{\eta}=1$ and $n_{\varphi}= 0$ for the other $\varphi \in\varDelta_+$, where $\alpha, \beta, \nu, \mu, \delta, \eta \in \varDelta_+$ satisfy
\begin{align}\label{eq:roots-N=2-case9}
& \begin{aligned}
\alpha &< \beta < \nu < \gamma < \mu <\delta < \eta,  
&
\beta+\delta &= \gamma+\nu, &
\nu+\eta &= \mu+\gamma, &
\alpha+\mu &= \gamma,
\end{aligned}
\\
\notag &\text{the corresponding PBW generators satisfy \eqref{eq:diff-case9-hypothesis},}
\\\label{eq:Csj-Dsj-N=2-9}
&\text{and $L$ satisfies }
\coeff{\alpha+\beta,\delta,\alpha,\gamma}{L}= 0.
\end{align}

\item\label{item:cocycle-xgamma-N=2-10} $n_{\gamma}= L-3$, 
$n_{\alpha}=n_{\beta}=n_{\delta}=n_{\eta}=1$ and $n_{\varphi}= 0$ for the other $\varphi \in\varDelta_+$, where $\alpha, \beta, \delta, \eta, \nu, \mu \in \varDelta_+$ satisfy
\begin{align}\label{eq:roots-N=2-case10}
& \begin{aligned}
\alpha &< \beta < \delta < \gamma < \mu <\nu < \eta,  
&
\beta+\eta &= \gamma+\nu, &
\delta+\nu &= \mu+\gamma, &
\alpha+\mu &= \gamma,
\end{aligned}
\\
\notag &\text{the corresponding PBW generators satisfy \eqref{eq:diff-case9-hypothesis},}
\\\label{eq:Csj-Dsj-N=2-10}
&\text{and $L$ satisfies }
\coeff{\beta \, -\nu \, \alpha \gamma}{L}= 0.
\end{align}
\end{enumerate}

\end{prop}

\pf
As $N_\gamma=2$, $\x_\gamma^n\ot 1$ is a $n$-chain for all $n\in\N$; hence $\x_\gamma^L\ot 1$ is so. By assumption all $(L+1)$-chains of degree $L \gamma$ are one of the following forms: $\x_\alpha \x_\gamma^{L-1}\x_\beta$, for a pair $(\alpha, \beta)$ satisfying \eqref{eq:alfa-beta-N=2}; $\x_\alpha \x_\gamma^{L-2}\x_\beta\x_{\delta}$, for a $4$-tuple  $(\alpha, \beta,\delta,\eta)$ satisfying \eqref{eq:alfa-beta-delta-eta-N=2}; $\x_\alpha^2 \x_\gamma^{L-3}\x_\beta\x_{\delta}$, for a $5$-tuple  $(\alpha, \beta,\delta,\eta,\tau)$ satisfying \eqref{eq:alfa-beta-delta-eta-tau-N=2},
$\x_\alpha \x_\beta\x_\gamma^{L-2}\x_{\delta}$, for a $5$-tuple $(\alpha, \beta,\delta,\eta,\tau)$ satisfying \eqref{eq:alfa-beta-delta-eta-tau-N=2-diff4},
$\x_\alpha \x_\beta \x_{\delta} \x_\gamma^{L-2} \x_{\eta}$, for a $6$-tuple  $(\alpha, \beta, \delta, \tau, \varphi, \eta)$ satisfying \eqref{eq:alfa-beta-delta-etc-N=2-diff5},
$\x_\alpha \x_\gamma^{L-2}\x_\beta\x_{\delta}$, for a $4$-tuple  $(\alpha, \beta,\delta,\eta)$ satisfying \eqref{eq:roots-N=2-case6},
$\x_\alpha \x_\beta\x_\gamma^{L-2}\x_{\delta}$, for a $4$-tuple  $(\alpha, \beta,\delta,\eta)$ satisfying \eqref{eq:roots-N=2-case7},
$\x_\alpha \x_\beta\x_\gamma^{L-2}\x_{\delta}$, for a $7$-tuple  $(\alpha, \beta,\delta,\eta,\tau,\mu,\nu)$ satisfying \eqref{eq:roots-N=2-case8},
$\x_\alpha \x_\beta\x_\gamma^{L-3}\x_{\delta}\x_{\eta}$, for a $6$-tuple $(\alpha, \beta,\nu,\mu\delta,\eta)$ satisfying \eqref{eq:roots-N=2-case9}.

\begin{itemize}[leftmargin=*,label=$\circ$]
\item Fix a pair $(\alpha, \beta)$ satisfying \eqref{eq:alfa-beta-N=2}. To simplify the notation, call $\zeta:=-\frac{\Csj}{\Dsj}=-
\frac{q_{\alpha \alpha}}{q_{\beta \beta}}$. We can apply Lemma \ref{lem:diff2} since conditions \eqref{eq:diff2-hypothesis} hold by hypothesis.
Assume first that $L=2a+1$ is odd. Hence the coefficient of $\x_\gamma^L\ot 1$ in $d(\x_\alpha \x_\gamma^{L-1}\x_\beta\ot 1)$ is 
\begin{align*}
\Bsj\Dsj^{2a} &
\left\{ \left(-\zeta-1 \right) (a)_{(-\zeta)^2}
- \left(-\zeta\right)^{2a}\right\}
=
-\Bsj\Dsj^{L-1}
\left\{ \left(1+\zeta\right) (a)_{\zeta^2}
+\zeta^{2a}\right\} = -\Bsj\Dsj^{L-1} (L)_{\zeta}.
\end{align*}
If $L=2a$ is even, then the coefficient of $\x_\gamma^L\ot 1$ in $d(\x_\alpha \x_\gamma^{L-1}\x_\beta\ot 1)$ is
\begin{align*}
&\Bsj\Dsj^{2a-1} \left(-\zeta-1\right) (a)_{\left(-\zeta\right)^{2}}
=
-\Bsj\Dsj^{L-1} \left(1+\zeta\right) (a)_{\zeta^{2}}
=
-\Bsj\Dsj^{L-1} (L)_{\zeta}.
\end{align*}
By \ref{item:cocycle-xgamma-N=2-1} such coefficient is zero in both cases. 

\medbreak

\item Fix a $4$-tuple  $(\alpha, \beta,\delta,\eta)$ satisfying \eqref{eq:alfa-beta-delta-eta-N=2}. We can apply Lemma \ref{lem:diff-case2} since conditions \eqref{eq:diff-case2-hypothesis} hold by hypothesis.
Hence the coefficient of $\x_\gamma^L\ot 1$ in $d(\x_\alpha \x_\gamma^{L-2}\x_\beta\x_{\delta}\ot 1)$ is zero by \ref{item:cocycle-xgamma-N=2-2}.

\medbreak

\item Fix a $5$-tuple  $(\alpha, \beta,\delta,\eta,\tau)$ satisfying \eqref{eq:alfa-beta-delta-eta-tau-N=2}. We can apply Lemma \ref{lem:diff-case3} since conditions \eqref{eq:diff-case3-hypothesis} hold by hypothesis.
Hence the coefficient of $\x_\gamma^L\ot 1$ in $d(\x_\alpha^2 \x_\gamma^{L-3}\x_\beta\x_{\delta}\ot 1)$ is zero by \ref{item:cocycle-xgamma-N=2-3}.

\medbreak

\item Fix a $5$-tuple  $(\alpha, \beta,\delta,\eta,\tau)$ satisfying \eqref{eq:alfa-beta-delta-eta-tau-N=2-diff4}. We can apply Lemma \ref{lem:diff-case4} since conditions \eqref{eq:diff-case4-hypothesis} hold by hypothesis.
Hence the coefficient of $\x_\gamma^L\ot 1$ in $d(\x_\alpha \x_\beta\x_\gamma^{L-2}\x_{\delta}\ot 1)$ is zero by \ref{item:cocycle-xgamma-N=2-4}.

\medbreak

\item Fix a $6$-tuple  $(\alpha, \beta, \delta, \tau, \varphi, \eta)$ satisfying \eqref{eq:alfa-beta-delta-etc-N=2-diff5}. We can apply Lemma \ref{lem:diff-case5} since conditions \eqref{eq:diff-case5-hypothesis} hold by hypothesis.
Hence the coefficient of $\x_\gamma^L\ot 1$ in $d(\x_\alpha \x_\beta \x_{\delta} \x_\gamma^{L-2} \x_{\eta} \ot 1)$ is zero by \ref{item:cocycle-xgamma-N=2-5}.

\medbreak

\item Fix a $4$-tuple  $(\alpha, \beta,\delta,\eta)$ satisfying \eqref{eq:roots-N=2-case6}. We can apply Lemma \ref{lem:diff-case6} since \eqref{eq:diff-case6-hypothesis} holds by hypothesis.
Hence the coefficient of $\x_\gamma^L\ot 1$ in $d(\x_\alpha \x_\gamma^{L-2}\x_\beta\x_{\delta}\ot 1)$ is zero by \ref{item:cocycle-xgamma-N=2-6}.

\medbreak

\item Fix a $4$-tuple  $(\alpha, \beta,\delta,\eta)$ satisfying \eqref{eq:roots-N=2-case7}. We apply Lemma \ref{lem:diff-case7} since \eqref{eq:diff-case7-hypothesis} holds by hypothesis: The coefficient of $\x_\gamma^L\ot 1$ in $d(\x_\alpha \x_\beta\x_\gamma^{L-2}\x_{\delta}\ot 1)$ is zero by \ref{item:cocycle-xgamma-N=2-7}.

\medbreak

\item Fix a $7$-tuple  $(\alpha, \beta,\delta,\eta,\tau,\mu,\nu)$ satisfying \eqref{eq:roots-N=2-case8}. We can apply Lemma \ref{lem:diff-case8} since \eqref{eq:diff-case8-hypothesis} holds by hypothesis: The coefficient of $\x_\gamma^L\ot 1$ in $d(\x_\alpha \x_\beta\x_\gamma^{L-2}\x_{\delta}\ot 1)$ is zero by \ref{item:cocycle-xgamma-N=2-8}.

\medbreak

\item Fix a $6$-tuple  $(\alpha, \beta,\nu,\mu,\delta,\eta)$ satisfying \eqref{eq:roots-N=2-case9}. We apply Lemma \ref{lem:diff-case9} since \eqref{eq:diff-case9-hypothesis} holds by hypothesis: The coefficient of $\x_\gamma^L\ot 1$ in $d(\x_\alpha \x_\beta\x_\gamma^{L-3}\x_{\delta}\x_{\eta}\ot 1)$ is zero by \ref{item:cocycle-xgamma-N=2-9}.

\medbreak

\item Fix a $6$-tuple  $(\alpha, \beta, \delta, \eta, \nu, \mu)$ satisfying \eqref{eq:roots-N=2-case10}. We apply Lemma \ref{lem:diff-case10} since \eqref{eq:diff-case10-hypothesis} holds by hypothesis: The coefficient of $\x_\gamma^L\ot 1$ in $d(\x_\alpha \x_\beta\x_{\delta}\x_\gamma^{L-3}\x_{\eta}\ot 1)$ is zero by \ref{item:cocycle-xgamma-N=2-10}.
\end{itemize}

\medbreak

Thus the coefficient of $\x_\gamma^{L}\ot 1$ in $d(c)$ is zero for all $c\in\chain(L+1)$ and Remark \ref{obs:main} applies.
\epf

\subsubsection{$N_\gamma > 2$} We carry out a similar analysis when the assumption is $N_\gamma > 2$ instead.

\begin{prop}\label{prop:cocycle-xgamma-N>2}
Let $\gamma \in \varDelta_+$ be such that $N_\gamma>2$ and for all pairs $(\alpha, \beta) \in \varDelta_+^2$ such that 
\begin{align}\label{eq:alfa-beta}
\alpha &< \beta& &\text{and}& \alpha+\beta &= (N_\gamma-1)\gamma,
\end{align}
the corresponding PBW generators satisfy \eqref{eq:diff2-hypothesis}. 
Let $f_\delta: \N_0\to\N_0$ be the function defined in \eqref{eq:power-root-vector-chain} for each 
$\delta \in  \varDelta_+$. 
Assume that $\ell\in \N$ satisfies the following two conditions:
\begin{enumerate}[leftmargin=*,label=\rm{(\alph*)}]
\item\label{item:cocycle-xgamma-N>2-1} For each pair $(\alpha, \beta)  \in \varDelta_+^2$ satisfying \eqref{eq:alfa-beta} the scalars $\Csj$ and $\Dsj$ satisfy
\begin{align}\label{eq:Csj-Dsj}
\big(\frac{\Csj}{\Dsj}-1 \big) (\ell)_{(\frac{\Csj}{\Dsj})^{N_\gamma}} &= 0.
\end{align}

\bigbreak
\item\label{item:cocycle-xgamma-N>2-2} The solutions $(n_\delta)_{\delta\in\varDelta_+} \in \N_0^{\varDelta_+}$ of the equations
\begin{align}\label{eq:equation-roots-cocycle-N>2}
\sum_{\delta\in\varDelta_+} f_\delta(n_\delta) \delta &= \ell N_{\gamma}\gamma,  &
\sum_{\delta\in\varDelta_+} n_\delta &= 2\ell +1
\end{align}
are all of the form $n_{\gamma}= 2(\ell-1)+1$, 
$n_{\alpha}=n_{\beta}=1$, for any pair  $(\alpha, \beta)$ satisfying \eqref{eq:alfa-beta}, and $n_{\delta}= 0$ for the remaining $\delta \in\varDelta_+$.
\end{enumerate}
Then  $(\x_\gamma^{\ell N_\gamma })^*$ is a $2\ell$-cocycle.    
\end{prop}
\pf
Fix a pair of positive roots $(\alpha, \beta)$ satisfying \eqref{eq:alfa-beta}. We can apply Lemma \ref{lem:diff2} since conditions \eqref{eq:diff2-hypothesis} hold by hypothesis and conclude that the coefficient of $\x_\gamma^{LN_\gamma}\ot 1$ in $d(\x_\alpha \x_\gamma^{N_{\gamma}(L-1)+1}\x_\beta \ot 1)$ is zero by \ref{item:cocycle-xgamma-N>2-1}. 

By \ref{item:cocycle-xgamma-N>2-2} all $(2L+1)$-chains of degree $L N_{\gamma} \gamma$ are of the form $\x_\alpha \x_\gamma^{N_{\gamma}(L-1)+1}\x_\beta$, for a pair $(\alpha, \beta)$ satisfying \eqref{eq:alfa-beta}. 
Thus the coefficient of $\x_\gamma^{LN_\gamma}\ot 1$ in $d(c)$ is zero for all $c\in\chain(2L+1)$ and Remark \ref{obs:main} applies.
\epf

Next we deal with the scalars $\coef{\alpha\beta\gamma}{L}$ and $\coeff{\alpha\beta\delta\gamma}{L}$. Given $r,s,t \in \Bbbk$, let
\begin{align}\label{eq:coef-general-formula}
\coef{r,s}{L}&:=\sum_{k=0}^{L-1} r^{k} (k+1)_{s}, &
\coeff{r,s,t}{L}&:=\sum_{k=0}^{L-1} r^{k} (k+1)_{s}(k+2)_{t}.
\end{align}
Notice that $\coef{\alpha\beta\gamma}{L}=\coef{\widetilde{q}_{\alpha\gamma}, \widetilde{q}_{\beta\gamma}}{L}$ and $\coeff{\alpha\delta\beta\gamma}{L}= \coeff{\widetilde{q}_{\alpha\gamma}, \widetilde{q}_{\beta\gamma}, \widetilde{q}_{\delta\gamma}}{L}$.

\begin{lemma}\label{lemma:coef-roots-unity}
\begin{enumerate}[leftmargin=*,label=\rm{(\alph*)}]
\item\label{item:coef-roots-unity-i} Assume that $(L)_r=0$. Then
$\coef{r,s}{L} = -s \coef{s,r}{L}$
\item\label{item:coef-roots-unity-ii} Assume that $(L)_r=0=(L)_s$ and $rs\neq 1$. Then $\coef{r,s}{L}=0$.
\item\label{item:coef-roots-unity-iii} Assume that $(L)_r=0=(L)_s$ for $L\ge 3$ and $rs\neq 1$. Then $\coeff{r,s,s}{L}=0$.
\end{enumerate}
\end{lemma}
\pf
For \ref{item:coef-roots-unity-i} we compute
\begin{align*}
\coef{r,s}{L} &=\sum_{k=0}^{L-1} r^{k} \Big(\sum_{j=0}^{k} s^j\Big) 
= \sum_{i=0}^{L-1} s^{i} \Big(\sum_{k=i}^{L-1} r^k\Big)
= \sum_{i=1}^{L-1} s^{i} \Big( (L)_r - (i)_r \Big)
= -\sum_{i=1}^{L-1} s^{i} (i)_r
\\ &= -\sum_{k=0}^{L-2} s^{k+1} (k+1)_r = -s\sum_{k=0}^{L-1} s^{k} (k+1)_r 
= -s \coef{s,r}{L}.
\end{align*}
Now \ref{item:coef-roots-unity-ii} follows using \ref{item:coef-roots-unity-i}. Indeed, we have that
$\coef{r,s}{L} = -s \coef{s,r}{L} = rs \coef{r,s}{L}$;
as $rs\neq 1$ by hypothesis, we have that $\coef{r,s}{L}=0$.

Next we deal with \ref{item:coef-roots-unity-iii}. As $rs\neq 1$ and $(L)_r=(L)_s=0$, we have that $(L)_{rs}=0$; thus $\coef{rs,s}{L}=0$ by \ref{item:coef-roots-unity-ii}. Also, $\binom{L+1}{2}_s=0$. Then we compute
\begin{align*}
\coeff{r,s,s}{L}&= (2)_s \sum_{k=0}^{L-1} r^{k} \binom{k+2}{2}_s =
(2)_s \sum_{k=0}^{L-1} r^{k} \Big( \binom{k+1}{2}_s + s^{k}(k+1)_s \Big)
\\
&= (2)_s \sum_{k=1}^{L-1} r^{k} \binom{k+1}{2}_s
+(2)_s \coef{rs,s}{L}
= (2)_s \sum_{j=0}^{L-2} r^{j+1} \binom{j+2}{2}_s
\\
&= (2)_s r \sum_{j=0}^{L-1} r^{j} \binom{j+2}{2}_s = r \coeff{r,s,s}{L}.
\end{align*}
As $r\ne 1$ we have that $\coeff{r,s,s}{L}=0$.
\epf

For each $\delta\in\varDelta_+$, let $a_{i}^{\delta}\in\N_0$ be the coordinate of $\alpha_i$ in $\delta$: that is, $\delta=\sum_{i\in\I} a_{i}^{\delta} \alpha_i$.

\subsection{Summary of the algorithm}\label{subsec:summary} 

As we have seen in Part~\ref{part:hopf}, to prove finite generation it is remained to establish Condition~\eqref{assumption:intro-combinatorial} 
for Nichols algebras of diagonal type. Now that we introduced all the necessary players we can describe the procedure we will follow to establish this condition. Going through these steps will occupy the remainder of the paper. 

\begin{itemize}[leftmargin=*,label=$\circ$]
\item We fix one type in the classification of Nichols algebras of diagonal type.
We choose a representative of the Weyl-equivalence (as defined in \S~\ref{sub:weyl}) with the care that the proper subdiagrams were already treated.

\medbreak
\item We fix  $\gamma \in \varDelta_{+}^{\bq}$. We assume that $\gamma$ has full support (see definition in \S~\ref{subsec:na-diag}); recall that Lemma \ref{lem:largeN} takes care of 
simple roots.

\medbreak
\item  We compute $N_\gamma$, $P_\gamma$, $Q_\gamma$.
\end{itemize}

Then we apply one of the following criteria:

\begin{enumerate}[leftmargin=*,label=\rm{(\Roman*)}]

\medbreak
\item\label{item:summary-1} If  $N_\gamma {>} P_\gamma, Q_\gamma$, 
then  $(\x_\gamma^{N_\gamma})^*$ is a cocycle of degree $2$.

\medbreak
\item\label{item:summary-2} If $N_{\gamma} = 2$ and  \eqref{eq:second-technique-2cocycles} holds, then $(\x_{\gamma}^{N_{\gamma}})^*$
is a 2-cocycle.

\medbreak
\item\label{item:summary-3} Assume that $N_{\gamma} = 2$ but \eqref{eq:second-technique-2cocycles} does not hold. We define 
\begin{align*}
L_{\gamma} &= \lcm \Big(\{2\} \cup \big\{\ord (- \frac{q_{\alpha\alpha}}{q_{\beta \beta}}) : \  \gamma = \alpha + \beta, \alpha, \beta  \in \varDelta_{+}^{\bq}  \big\}\Big).
\end{align*}
We find all families $(n_{\delta})_{\delta \in \varDelta_{+}^{\bq}}$ of non-negative integers satisfying \eqref{eq:equation-roots-cocycle-N=2} 
with $L = L_{\gamma}$.

\medbreak
\noindent We check that any of these families $(n_{\delta})_{\delta \in \varDelta_{+}^{\bq}}$   has one  of the forms \ref{item:cocycle-xgamma-N=2-1}, \dots, 
or \ref{item:cocycle-xgamma-N=2-10} in Proposition \ref{prop:cocycle-xgamma-N=2}.
Then $(\x_{\gamma}^{L_{\gamma}})^*$
is a cocycle.

\bigbreak
\item\label{item:summary-4} Assume that $N_{\gamma} > 2$ but the inequality in \ref{item:summary-1}  does not hold. We define 
\begin{align*}
L_{\gamma} &= \lcm \Big(\{N_{\gamma} \} \cup \big\{\ord \left( \big(\frac{q_{\alpha\gamma}}{q_{\gamma \beta}}\big)^{N_{\gamma}} \right): 
\  (N_{\gamma} - 1)\gamma = \alpha + \beta, \alpha, \beta  \in \varDelta_{+}^{\bq}  \big\}\Big).
\end{align*}
We find all families $(n_{\delta})_{\delta \in \varDelta_{+}^{\bq}}$ of non-negative integers satisfying \eqref{eq:equation-roots-cocycle-N>2} 
with $L = L_{\gamma}$. 
 We check that any of these families $(n_{\delta})_{\delta \in \varDelta_{+}^{\bq}}$   has one  of the forms \ref{item:cocycle-xgamma-N>2-1}, \dots
in Proposition \ref{prop:cocycle-xgamma-N>2}.
Then $(\x_{\gamma}^{L_{\gamma}})^*$ is a cocycle.
\end{enumerate}

Actually we distinguish two classes of types in the classification of finite-dimensional Nichols algebras of diagonal type.
In the first the braiding matrices have continuous parameters and correspondingly the values of $N_{\gamma}$ might depend on these parameters. 
Then arguments by hand are needed. These are the types  treated in Sections \ref{sec:classical}, \ref{sec:exceptional} and \ref{sec:discrete}.

\medbreak
The second class consists of the remaining types where the braiding matrices are so to say discrete. 
For them we compute $N_\gamma$, $P_\gamma$, $Q_\gamma$ and the suitable families $(n_{\delta})_{\delta \in \varDelta_{+}^{\bq}}$
using a computer program developed by H\'ector Pe\~na Pollastri towards these goals.
We then check whether $(n_{\delta})_{\delta \in \varDelta_{+}^{\bq}}$   has one  of the forms \ref{item:cocycle-xgamma-N=2-1}, \dots, 
or \ref{item:cocycle-xgamma-N=2-10}, respectively \ref{item:cocycle-xgamma-N>2-1}, \dots by hand using the 
defining relations, or at least the convex order, of $\toba_{\bq}$.
The types in this class are those treated in Part III (to appear later).

\medbreak
The implicit numeration of any generalized Dynkin diagram is from the left to the right and from bottom to top; otherwise,
the numeration appears below the vertices.

\section{Classical types}\label{sec:classical}
We shall use notations and conventions introduced in \S~\ref{subsection:conventions}. For a positive integer $\theta$ defining the Dynkin type let $\I = \I_\theta$. 

\subsection{Types \texorpdfstring{$A_{\theta}$ and $\supera{j}{\theta - j}$, $\theta \ge 1$, $j \in \I_{\lfloor\frac{\theta+1}{2} \rfloor}$}{}}\label{sec:An}

Let $q$ be a root of 1 of order $N \geq 2$.
In this subsection, we deal with the Nichols algebra $\toba_{\bq}$ of standard diagonal type $A_{\theta}$, that is associated to
the Dynkin diagram
\begin{align*}
\xymatrix{ \overset{q_{11}}{\underset{\ }{\circ}}\ar  @{-}[r]^{\widetilde{q}_{12}}  &
\overset{q_{22}}{\underset{\ }{\circ}}\ar  @{-}[r]^{\widetilde{q}_{23}} 
&  \overset{q_{33}}{\underset{\ }{\circ}}\ar@{.}[r] 
& \overset{q_{\theta-1\, \theta-1}}{\underset{\ }{\circ}} \ar  @{-}[rr]^{\widetilde{q}_{\theta-1\, \theta}}&  &
\overset{q_{\theta \theta}}{\underset{\ }{\circ}}}
\end{align*}
where the $q_{ii}$'s are either $-1$, $q$ or $q^{-1}$ and locally the edges are of the following forms:
\begin{align*}
&\xymatrix{ &\overset{q}{\underset{\ }{\circ}}\ar  @{-}[r]^{q^{-1}} \ar  @{-}[l]_{q^{-1}}  & }, &
&\xymatrix{ &\overset{q^{-1}}{\underset{\ }{\circ}}\ar  @{-}[r]^{q} \ar  @{-}[l]_{q}  & },&
&\xymatrix{ &\overset{-1}{\underset{\ }{\circ}}\ar  @{-}[r]^{q^{-1}} \ar  @{-}[l]_{q}  & }, &
&\xymatrix{ &\overset{-1}{\underset{\ }{\circ}}\ar  @{-}[r]^{q} \ar  @{-}[l]_{q^{-1}}  & }.
\end{align*}

For more information, see \cite[\S 4.1, \S 5.1]{AA17}.
The aim of this Section is to prove that Condition \ref{assumption:intro-combinatorial} holds for types
$A_{\theta}$ and $\supera{j}{\theta - j}$, $\theta \ge 1$, $j \in \I_{\lfloor\frac{\theta+1}{2} \rfloor}$. That is,

\begin{prop}\label{prop:roots-cocycles-Atheta}
For every  $\gamma \in \varDelta_+^{\bq}$, there exists $L_{\gamma}\in\N$ 
such that $(\x_{\gamma}^{L_\gamma})^*$ is a cocycle.
\end{prop}

We start by setting the notation. Let
\begin{align}\label{eq:alfaij}
\alpha_{i j} &= \sum_{k \in \I_{i,j}} \alpha_k,&  i&\leq j \in \I.
\end{align}
The set of positive roots is 
$\varDelta_+ =\{\alpha_{i\, j}\,|\, i, j \in \I,\, i\leq j\}$; this set is ordered lexicographically on 
the subindex $(i,j)$.
Let $r = \vert \varDelta_+ \vert = \binom{\theta + 1}{2}$; we have a numeration
$\varDelta_+ =\{\beta_{i}\,|\, i\in \I_r \}$ so that $\beta_k < \beta_\ell$ if $k < \ell$.

We set $N_{ij} := \ord q_{\alpha_{i j}}$, that is $N_{ij} = N_{\alpha}$ if $\alpha = \alpha_{i j}$.
For simplicity we set $N_{i} = N_{ii}$ for all $i$. Then 
\begin{align}\label{eq:ord-root-Atheta}
N_{ij} &= \begin{cases} 2 & \vert \{k\in \I_{i,j}: q_{kk} = -1\}\vert \text{ is odd,}
\\ N & \vert \{k\in \I_{i,j}: q_{kk} = -1\}\vert \text{ is even.}
\end{cases}
\end{align}
The root vectors are 
\begin{align*}
x_{\alpha_{ii}} &= x_{\alpha_{i}} = x_{i},& i \in \I, \\
x_{\alpha_{i j}} &= x_{(ij)} = [x_{i}, x_{\alpha_{i+1\, j}}]_c,& i <  j \in \I,
\end{align*}
see \eqref{eq:roots-Atheta}; we order them lexicographically: $ \ x_1 < x_{(12)} < \dots <x_2 < \dots < x_{\theta}$. Thus
\begin{align*}
\{ x_{ \theta}^{n_{\theta  \theta}} x_{(\theta-1  \theta)}^{n_{\theta-1  \theta}} x_{\theta-1}^{n_{\theta-1  \theta-1}} \dots x_{(1  \theta)}^{n_{1  \theta}} \dots x_{1}^{n_{11}} \, | \, 0\le n_{ij}<N_{ij}\}
\end{align*}
is a PBW-basis of $\toba_{\bq}$. 
The defining relations in terms of the PBW-generators are
\begin{align*}
x_{(ij)} x_{(ik)} &= q_{\alpha_{i j}\alpha_{i k}} x_{(ik)}x_{(ij)}, & i &\leq j < k;
\\
x_{(ik)} x_{(jk)} &= q_{\alpha_{i k}\alpha_{j k}} x_{(jk)}x_{(ik)}, & i &< j \leq k;
\\
x_{(ij)} x_{(j+1 \, k)} &= q_{\alpha_{i j}\alpha_{j+1\, k}} x_{(j+1\, k)}x_{(ij)}
+ x_{(i k)}, & i &\leq j < k;
\\
x_{(ij)} x_{(k \ell)} &= q_{\alpha_{i j}\alpha_{k \ell}} x_{(k \ell)}x_{(i j)}, & i &\leq j < k - 1 \leq \ell - 1;\\
x_{(i \ell)} x_{(j k)} &= q_{\alpha_{i \ell}\alpha_{j k}} x_{(j k)}x_{(i \ell)}, & i &< j \leq k < \ell;
\\
x_{(ij)} x_{(k \ell)} &= q_{\alpha_{i j}\alpha_{k \ell}} x_{(k \ell)}x_{(i j)} 
+ (1 - \widetilde{q}_{j, j+1}) q_{\alpha_{i j}\alpha_{k j}} x_{(i \ell)}x_{(k j)} , & 
i &< k \leq j < \ell;
\\
x_{(ij)}^{N_{ij}} &= 0, & i&\leq j.
\end{align*}
The relations are homogeneous with respect to the $\N_0^{\theta}$-grading, an observation
that will be useful later.
As an example we draw the Cojocaru-Ufnarovski graph that encodes the Anick resolution for the Nichols
algebra $\toba_{\bq}$ of type $A_2$  in case $N_1,N_{12},N_2 >2$:
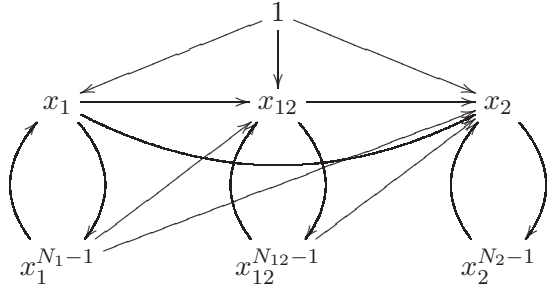
\[
\xymatrix@=20pt{ &&  1 \ar[dll] \ar[d]  \ar[drr] && \\
x_1\ar@/^1.5pc/[dd] \ar[rr] \ar@/_2pc/[rrrr] && x_{12}
\ar[rr]\ar@/^1.5pc/[dd] && x_2\ar@/^1.5pc/[dd]\\
&&&&\\
x_1^{N_1-1}\ar@/^1.5pc/[uu]\ar[uurr] \ar[uurrrr]
&& x_{12}^{N_{12}-1}\ar@/^1.5pc/[uu] \ar[uurr]&& x_2^{N_2-1}\ar@/^1.5pc/[uu] 
} 
\]
Observe that in this case $N_1 = N_{12} = N_2$ and we are in Cartan type. The chains are then 
\begin{align*}
\chain(0) \!  &= \! \{ \x_1, \ \x_{12}, \ \x_2 \} ,\\
\chain(1) \! &= \! \{ \x_1^{N_1}, \ \x_{12}^{N_{12}}, \ \x_2^{N_2}, \x_1\x_{12}, \ \x_1 \x_2, \ \x_{12}\x_2 \} , \\
\chain(2)\! &= \! \{ \x_1^{N_1+1}, \ \x_{12}^{N_{12}+1}, \ \x_2^{N_2+1}, 
\ \x_1^{N_1} \x_{12}, \ \x_1^{N_1}\x_2,  \ \x_1 \x_{12}^{N_{12}}, \ 
\x_1\x_2^{N_2}, \ \x_{12}\x_2^{N_2}, \ \x_{12}^{N_{12}}\x_2,
\ \x_1\x_{12}\x_2\} ,  \\
\chain(3) \! &= \! \{ \x_1^{2N_1},\ \x_{12}^{2N_{12}}, \ \x_2 ^{2N_2},  
\ \x_1^{N_1} \x_{12}^{N_{12}}, \ \x_1^{N_1}\x_{12}\x_2 ,
\ \x_1^{N_1} \x_2^{N_2},  
\ \x_1^{N_1+1}\x_{12}, \ \x_1^{N_1+1}\x_2, \\ 
&     \x_1\x_{12}^{N_{12}+1}, \ \x_1\x_2^{N_2+1}, \
\x_1 \x_{12}\x_2^{N_2}, \ \x_{12}^{N_{12}+1}\x_2, \ \x_{12}^{N_{12}}\x_2^{N_2}, \ \x_{12}\x_2^{N_2+1}  \} , 
\end{align*}
and so on. 
In case $N_1=2$, the loop between $\x_1$ and $\x_1^{N_1-1}$ is understood to be collapsed 
to a loop from $\x_1$ to itself, and similarly for the other root vectors.

\bigbreak
In order to apply Remark \ref{obs:main}, we start by the following Claim.

\begin{claimA}\label{claimA:1}
Let $\alpha$ be a non-simple root with $N_{\alpha} = 2$.  Let $C_*(\toba_{\bq})$ 
be the Anick resolution of the Nichols algebra $\toba_{\bq}$.
Let $c \in \chain(N+1)$ and assume that $d(c\ot 1) \in C_N(\toba_{\bq})$
has a term $\x_\alpha^N \otimes 1$ with nonzero coefficient.
Write $c = \x_{\beta_1}^{a_1} \dots \x_{\beta_r}^{a_r}$, where $a_j \in \N_0$ for all $j$. 
If $N_{\beta_{i}} > 2$, then $a_i =0, 1$ and
\begin{align}\label{eq:sum-ai}
\sum_{i\in \I}  a_i =N + 1.
\end{align}
\end{claimA}

Indeed, we may safely assume that $\alpha = \alpha_{1 \theta} = \alpha_{1} + \dots + \alpha_{\theta}$ for simplicity.
Since $d$ is homogeneous, $N  \alpha_{1 \theta} \overset{\ast}{=} \sum_{i\in \I_r} a_i \beta_i$.
Assume that
\begin{align*}
\mathfrak S := \{i\in \I_r:  N_{\beta_i} > 2, a_i > 1\}
\end{align*}
is non-empty. If $i\in \mathfrak S$, then $a_i \geq N$ by looking at the Anick graph, hence
$a_i = N$ by $\ast$. Also, the supports of the $\beta_i$ with $i\in \mathfrak S$ are disjoint.
Let $\mathfrak K = \{k\in \I: k\notin \supp \beta_i, i\in \mathfrak S\}$; observe that  $\mathfrak K$ is non-empty because $N_{\alpha} = 2$, cf. \eqref{eq:ord-root-Atheta}. 
Now
\begin{align*}
\sum_{i\notin \mathfrak S}  a_i \beta_i \overset{\text{by } \ast}{=} N \alpha - N \sum_{i\in \mathfrak S} \beta_i 
= N \sum_{k \in \mathfrak K} \alpha_k.
\end{align*}
Pick $k \in \mathfrak K$ and compute the coefficient of $\alpha_k$ in the last expression; then
\begin{align*}
\sum_{i\notin \mathfrak S}  a_i \geq \sum_{i\notin \mathfrak S, k\in \supp \beta_i} a_i = N.
\end{align*}
Now the cohomological degree  of $x_{\beta_1}^{a_1} \dots x_{\beta_r}^{a_r}$
sums up $a_i$ for each $i\notin \mathfrak S$ and 2 for each $i\in \mathfrak S$ (one arrow in, one out).
Thus 
\begin{align*}
N + 1 = \sum_{i\notin \mathfrak S}  a_i  + 2 \vert \mathfrak S\vert \geq N + 2.
\end{align*}
This contradiction shows that $\mathfrak S = \emptyset$ and the claim is proved.

\medbreak The following result on root systems of type A should be well-known; we provide a proof for completeness of the argument.

\begin{lemma} 
\label{le:roots} 
Let $\gamma, \gamma_1, \ldots, \gamma_{n+1} \in \varDelta_+$
with $\gamma_1 \leq \gamma_2 \leq \ldots \leq \gamma_{n+1}$. Assume that 
\begin{align}\label{eq:roots-A} 
\gamma_1 + \ldots  + \gamma_{n+1}  = n \gamma.
\end{align}
Then \begin{align*} 
\gamma_2 &= \gamma_3 = \ldots = \gamma_{n} = \gamma,&
\gamma &= \gamma_1 + \gamma_{n+1}. 
\end{align*}   
\end{lemma}

\begin{proof} 
We may assume that 
$\gamma = \alpha_1 + \ldots + \alpha_\theta$ for otherwise we reduce to a smaller $\theta$.  
Since we have $n+1$ roots in the sum which contains each simple 
root with coefficient exactly $n$, we must have that $\gamma_1, \ldots, \gamma_{n}$ have $\alpha_1$ in the support but $\gamma_{n+1}$ does not. 
Similarly, $\gamma_2, \ldots, \gamma_{n+1}$ must have $\alpha_\theta$ in their support but $\gamma_1$ does not. Hence, the supports of 
$\gamma_2, \ldots, \gamma_n$ contain all simple roots.  We conclude that $\gamma_2 = \ldots = \gamma_n = \gamma$,
thus $\gamma_1 + \gamma_{n+1} = \gamma$ by \eqref{eq:roots-A}. 
\end{proof}

\begin{claimA}\label{claimA:2}
Let $\gamma \in \varDelta_+$ and let $c \in \chain(n)$. Assume that 
\begin{enumerate}[leftmargin=*,label=\rm{(\alph*)}]
\item\label{item:claim2-Atheta1} $d(c \otimes 1)  =  \ldots + C \x_\gamma^n \otimes 1 + \ldots $ 
for some   $C \not = 0$,
\item\label{item:claim2-Atheta2} The polynomial degree of $c$ is $n + 1$. 
\end{enumerate} 
Then there exist $\alpha, \beta \in \varDelta_+$ such that 
\begin{align*} 
c &= \x_\alpha \x_\gamma^{n-1} \x_\beta,&
\gamma &= \alpha + \beta.
\end{align*} 
\end{claimA} 

\pf By \ref{item:claim2-Atheta2}, we may write 
$c = \x_{\gamma_1}  \dots \x_{\gamma_{n+1}}  \otimes 1$, where  
$\gamma_1 \leq \gamma_2 \leq \ldots \leq \gamma_{n+1} \in \varDelta_+$.
Since $d$ is $\N_0^{\theta}$-homogeneous, \eqref{eq:roots-A} holds by \ref{item:claim2-Atheta1}. 
Thus Lemma \ref{le:roots} applies.
\epf

\begin{claimA}\label{claimA:3} Let $\alpha,\beta, \gamma \in \varDelta_+$ with $\gamma = \alpha + \beta$. 
Let $\wp = -\dfrac{q_{\alpha\alpha}}{q_{\beta \beta}}$. Assume that $N_{\gamma} = 2$.
Then 
\begin{multline*}
 d(\x_{\alpha} \x_{\beta}^{j-1} \x_{\gamma}\ot 1)=
\x_{\alpha}\x_{\gamma}^{j-1}
\ot x_{\beta} - q_{\gamma, \beta}\x_{\alpha} \x_{\gamma}^{j-2} \x_{\beta}\ot x_{\gamma} 
\\
+ (-1)^{j} q_{\alpha, \beta}q_{\alpha, \gamma}^{j-1} \x_{\gamma}^{j-1} \x_{\beta}\ot x_{\alpha} -q_{\gamma, \beta}^{j-1}(j)_{\wp}  \x_{\gamma}^{j}\ot 1.
\end{multline*}
\end{claimA}

\begin{proof} The defining relations say that
\begin{align*}
x_{\alpha}x_{\gamma} &= q_{\alpha, \beta} x_{\gamma}x_{\alpha} +x_{\gamma},&
x_{\alpha}x_{\beta} &= q_{\alpha, \gamma} x_{\beta}x_{\alpha},&
x_{\beta}x_{\gamma} &= q_{\gamma, \beta}x_{\gamma}x_{\beta}.
\end{align*}
Hence we are in the setting of Lemma~\ref{lem:diff2}. We consider two cases.

First we assume that $j = 2a+1$ is odd. The only thing we need to prove is that
 the coefficient of $\x_{\gamma}^{j}\ot 1$ is of the given form. We compute 
\begin{align*} 
&q_{\gamma, \beta}^{j-1}
\left\{ \left(\frac{q_{\alpha, \gamma}}{q_{\gamma, \beta}}-1 \right) (a)_{(\frac{q_{\alpha, \gamma}}{q_{\gamma, \beta}})^2}
- \left(\frac{q_{\alpha, \gamma}}{q_{\gamma, \beta}}\right)^{2a}\right\} = 
 q_{\gamma, \beta}^{j-1}( (-\wp -1)(a)_{(-\wp)^2} - (-\wp)^{2a})  \\
&= -q_{\gamma, \beta}^{j-1}(1+\wp)(1+\wp^2 + \ldots + \wp^{2a-2} + \wp^{2a}) =  
-q_{\gamma, \beta}^{j-1}(2a+1)_{\wp} = -q_{\gamma, \beta}^{j-1}(j)_{\wp}.    
\end{align*}

Next  we assume that $j = 2a+2$ is even. 
We do a similar calculation for the coefficient of $\x_{\gamma}^{j}\ot 1$.
\begin{align*} 
&q_{\gamma, \beta}^{j-1}\left(\frac{q_{\alpha, \gamma}}{q_{\gamma, \beta}}-1\right)(a+1)_{\left(\frac{q_{\alpha, \gamma}}{q_{\gamma, \beta}}\right)^2} = q_{\gamma, \beta}^{j-1}(-\wp-1)(a+1)_{(-\wp)^2} 
\\&= 
-q_{\gamma, \beta}^{j-1}(1+\wp)(1+\wp^2 + \ldots + \wp^{2a}) = -q_{\gamma, \beta}^{j-1}(2a+2)_{\wp} =  -q_{\gamma, \beta}^{j-1}(j)_{\wp}. 
\end{align*}
\end{proof}

\medbreak
\emph{Proof of Proposition \ref{prop:roots-cocycles-Atheta}.} \ 
First, if $N_{\gamma} > 2$, then $(\x_{\gamma}^{N_{\gamma}})^*$ is a
cocycle in degree~2 by Lemma~\ref{lem:largeN}, cf. Lemma \ref{lema:non-simple-roots-Palfa} \ref{item:typeA}.
Let $\gamma$ be a non-simple root with $N_{\gamma} = 2$. 

By Remark \ref{obs:main}, it suffices to show that there is no $c \in C_{N} (R)$ such that
$d(c)$ contains  $\x_{\gamma}^{N}\ot 1$  with a non-zero coefficient.
Assume that this happens. By Claim \ref{claimA:1}, $c$ satisfies the hypothesis \ref{item:claim2-Atheta2}
of Claim \ref{claimA:2}. Hence $c = \x_{\alpha} \x_{\gamma}^{N-1} \x_{\beta}\ot 1$ for some $\alpha$ and $\beta$.
But $\x_{\gamma}^{N}\ot 1$ does not enter
in $d(\x_{\alpha} \x_{\gamma}^{N-1} \x_{\beta}\ot 1)$ with a non-zero coefficient for any $\alpha, \beta, \gamma \in \Delta^+$ with $\gamma = \alpha + \beta$.
Indeed, $\wp = -\dfrac{q_{\alpha\alpha}}{q_{\beta \beta}} = q^{\pm 1}$, see \eqref{eq:ord-root-Atheta}; taking $j = N$, the coefficient in Lemma \ref{claimA:3} is  $q_{\gamma, \beta}^{j-1}(j)_{\wp} = 0$.
\qed

\vspace{0.1in} 

\subsection{Types \texorpdfstring{$B_{\theta}$ and $\superb{j}{\theta - j}$, $\theta \ge 1$, $j \in \I_{\theta-1}$}{}}\label{sec:Bn}

Let $q$ be a root of 1 of order $N>2$.
In this subsection, we deal with the Nichols algebras $\toba_{\bq}$ of diagonal types $B_{\theta}$ or $\superb{j}{\theta - j}$. 
In the first case, the Dynkin diagram is
\begin{align*}
\xymatrix{ \overset{q^2}{\underset{\ }{\circ}}\ar  @{-}[r]^{q^{-2}}  &
\overset{q^2}{\underset{\ }{\circ}}\ar  @{-}[r]^{q^{-2}} 
&  \overset{q^2}{\underset{\ }{\circ}}\ar@{.}[r] 
& \overset{q^2}{\underset{\ }{\circ}} \ar  @{-}[r]^{q^{-2}}&
\overset{q}{\underset{\ }{\circ}}}
\end{align*}
while in the second we may assume that the corresponding diagram is
\begin{align*}
\xymatrix{ \overset{q^{-2}}{\underset{\ }{\circ}}\ar  @{-}[r]^{q^{2}}  
& \overset{q^{-2}}{\underset{\ }{\circ}}\ar@{.}[r] 
& \overset{q^{-2}}{\underset{\ }{\circ}}\ar  @{-}[r]^{q^2} 
& \overset{-1}{\underset{j}{\circ}}\ar  @{-}[r]^{q^{-2}} 
&  \overset{q^2}{\underset{\ }{\circ}}\ar@{.}[r] 
& \overset{q^2}{\underset{\ }{\circ}} \ar  @{-}[r]^{q^{-2}}&
\overset{q}{\underset{\ }{\circ}}}
\end{align*}
The set of positive roots in both cases is
\begin{align}\label{eq:root-system-Bn}
\varDelta_+ =\{\alpha_{ik}\,|\, i \leq k \in \I \} \cup
\{\alpha_{i\theta} + \alpha_{k\theta}\,|\, i < k \in \I \}. 
\end{align}
We fix the following convex order:
\begin{align}\label{eq:convex-order-Bn}
\begin{aligned}
&\alpha_{1}< \alpha_{12}< \dots < \alpha_{1\theta} < \alpha_{1\theta} + \alpha_{\theta} < \dots < \alpha_{1\theta} + \alpha_{2 \, \theta}
\\
&< \alpha_{2}< \alpha_{23}< \dots <\alpha_{\theta-1} 
<\alpha_{\theta-1\,\theta}<\alpha_{\theta-1 \, \theta}+\alpha_{\theta}<\alpha_{\theta}.
\end{aligned}
\end{align}

We set $N_{ik} := \ord q_{\alpha_{ik}}$, 
$M_{ik} := \ord q_{\alpha_{i\theta} + \alpha_{k\theta}}$, that is $N_{ik} = N_{\alpha}$ if $\alpha = \alpha_{ik}$, $M_{ik} = N_{\alpha}$ if $\alpha =\alpha_{i\theta} + \alpha_{k\theta}$.
Let $M=\ord q^2$, $P=\ord (-q)$.
Then 
\begin{align*}
N_{ik} &= \begin{cases} 
2 & i\le j\le k<\theta,
\\ M & i\le k<j \text{ or }j<i\le k<\theta;
\\ P & i\le j < k=\theta,
\\ N & j<i <k=\theta;
\end{cases}
&
M_{ik} &= \begin{cases} 2 & i\le j<k
\\ M & i<k\le j \text{ or }j<i<k.
\end{cases}
\end{align*}
Here we set $j=0$ if $\bq$ is of Cartan type.
For more information, see \cite[\S 4.2, \S 5.2]{AA17}.

In this Section we prove that Condition \ref{assumption:intro-combinatorial} holds for types
$B_{\theta}$ and $\superb{j}{\theta - j}$, $\theta \ge 1$, $j \in \I_{\theta-1}$. We need the following technical result. 

\begin{lem}\label{lem:Btheta-relations-root-vectors} Let $\gamma = \alpha_{1\theta}$, assume $N_\gamma = 3$, and let 
$(\alpha,\beta)$ be a pair of positive roots such that 
\begin{enumerate} 
\item $\alpha  + \beta = 2\gamma$,
\item $\alpha=\alpha_{1k}$, $\beta=\alpha_{1\theta}+\alpha_{k+1 \, \theta}$ for some $k\in\I_{\theta-1}$, 
\end{enumerate} 
(in particular, $(\alpha, \beta)$ satisfy $\alpha < \beta$). 

Then the relations between the root vectors $x_{\alpha}$, $x_{\beta}$ and $x_{\gamma}$ are of the form
\begin{align*}
x_{\alpha}x_{\gamma} &= q_{\alpha\gamma} x_{\gamma}x_{\alpha}, &
x_{\gamma}x_{\beta} &= q_{\gamma\beta} x_{\beta}x_{\gamma}, &
x_{\alpha}x_{\beta} &= q_{\alpha\beta} x_{\beta}x_{\alpha}+ \Bsj x_{\gamma}^2,
&
&\Bsj\in \Bbbk
\end{align*}
\end{lem}
\pf The statement follows by the repeated application of the convexity and homogeneity of the relations.We prove $x_{\alpha}x_{\gamma} = q_{\alpha\gamma} x_{\gamma}x_{\alpha}$ first. For any $\eta\in\varDelta_+$ such that $\alpha<\eta<\gamma$, we must have $\eta=\alpha_{1r}$ for some $k<r<\theta$ by \eqref{eq:convex-order-Bn}. From here, we see that is impossible to find $\eta_1, \ldots, \eta_t$ such that $\alpha+\gamma\ne \eta_1+\cdots+\eta_t$  and $\alpha<\eta_i<\gamma$. Hence, $x_{\alpha}x_{\gamma} = q_{\alpha\gamma} x_{\gamma}x_{\alpha}$ by \eqref{eq:convex-order-relations}.

For the second relation, if $\eta\in\varDelta_+$ is such that $\gamma<\eta<\beta$, then $\eta=\alpha_{1\theta}+\alpha_{r\theta}$ for some $k<r<\theta$. Hence $\beta+\gamma\ne \eta_1+\cdots+\eta_t$ if $\alpha<\eta_i<\gamma$, so $x_{\gamma}x_{\beta}=q_{\gamma\beta} x_{\beta}x_{\gamma}$.

The last relation follows similarly: $\alpha<\eta_1\le \cdots \le\eta_t<\beta$ are such that $\sum_i \eta_i=\alpha+\beta$ if and only if $k=2$ and $\eta_1=\eta_2=\gamma$.
\epf

\begin{prop}\label{prop:roots-cocycles-Btheta}
For every  $\gamma \in \varDelta_+^{\bq}$, there exists $L_{\gamma}\in\N$ 
such that $(\x_{\gamma}^{L_\gamma})^*$ is a cocycle.
\end{prop}

\pf 
Let $\gamma$ be a positive non-simple root.
Arguing recursively we may assume that $\gamma$ has full support. Hence, either $\gamma=\alpha_{1\theta}$ or else there exists $k\in\I_{2, \theta}$ such that $\gamma=\alpha_{1\theta} +\alpha_{k\theta}$.

\medbreak

First we consider $\gamma=\alpha_{1\theta}$. It is easy to check that
$P_{\gamma} = 3$ and $Q_{\gamma} = 1$. Hence, if $N_{\gamma}>3$, then Lemma \ref{lem:largeN} applies and $(\x_{\gamma}^{N_\gamma})^*$ is a $2$-cocycle. 
Next we consider the case $N_{\gamma}=3$: that is, either $N=3$ if the braiding is of Cartan type, or else $N=6$ if not.

All the pairs $(\alpha,\beta)$ as in \eqref{eq:alfa-beta} are of the form $\alpha=\alpha_{1k}$, $\beta=\alpha_{1\theta}+\alpha_{k+1 \, \theta}$ for some $k\in\I_{\theta-1}$. Fix a pair $(\alpha,\beta)$. By Lemma~\ref{lem:Btheta-relations-root-vectors}, 
$x_{\alpha}$, $x_{\beta}$, $x_{\gamma}$ satisfy \eqref{eq:diff2-hypothesis}. Also,
$\frac{q_{\alpha\gamma}}{q_{\gamma\beta}}=1$ for all of them, so we take $L=1$. Now we look for solutions of \eqref{eq:equation-roots-cocycle-N>2}. That is, $\sum_{\delta\in\varDelta_+} f_\delta(n_\delta) \delta =3\gamma$, $\sum_{\delta\in\varDelta_+} n_\delta = 3$. We check that there is no solution with $n_{\delta}=3$ neither with $n_{\delta}=2$,  $\delta\in\varDelta_+$. Hence we are forced to look for solutions with $n_{\gamma_t}=1$ for three different roots $\gamma_t\in\varDelta_+$ (and $0$ for the other roots). We write $\gamma_t=\sum_{i\in\I} a_{i}^{(t)}\alpha_i$. As $\sum_t a_i^{(t)}=3$ for all $i\in\I$, we have $a_1^{(t)}=1$ for all $t\in\I_3$, and either $a_\theta^{(t)}=1$ for all $t\in\I_3$, or $a_\theta^{(t)}=2$ for some $t\in\I_3$. If $a_\theta^{(t)}=1$ for all $t\in\I_3$, then $\gamma_t=\gamma$ for all $t$, a contradiction. Hence we may assume $a_{\theta}^{(1)}=0$, $a_{\theta}^{(2)}=2$, $a_{\theta}^{(3)}=1$.
Then $\gamma_1=\alpha_{1k}$, $\gamma_2=\alpha_{1\theta}+\alpha_{k+1 \, \theta}$, $\gamma_3=\gamma$, for some $k\in\I_{\theta-1}$.
Thus Proposition \ref{prop:cocycle-xgamma-N>2} applies and $(\x_{\gamma}^{3})^{\ast}$ is a 2-cocycle.

\medbreak
Now we consider $\gamma=\alpha_{1\theta}+\alpha_{k\theta}$, $k\in\I_{2, \theta}$.
Here,
$P_{\gamma} = 2$ and $Q_{\gamma} = 1$. If $N_{\gamma}>2$, then Lemma \ref{lem:largeN} applies and $(\x_{\gamma}^{N_\gamma})^*$ is a $2$-cocycle. 
Next we consider the case $N_{\gamma}=2$: that is, either $N=4$ if the braiding is of Cartan type, or else $k>j$ if the braiding is of type $\superb{j}{\theta-j}$.
Let $\alpha<\beta$ be a pair of positive roots as in \eqref{eq:alfa-beta-N=2}. We have several possibilities:
\begin{itemize}[leftmargin=*,label=$\circ$]
\item $\alpha=\alpha_{1 \, i-1}$, $\beta=\alpha_{i\theta}+\alpha_{k\theta}$, $i<k$. Arguing as in Lemma \ref{lem:Btheta-relations-root-vectors}, the relations between the root vectors are of the form:
\begin{align*}
x_{\alpha}x_{\gamma} &= q_{\alpha\gamma} x_{\gamma}x_{\alpha}, &
x_{\gamma}x_{\beta} &= q_{\gamma\beta} x_{\beta}x_{\gamma}, &
x_{\alpha}x_{\beta} &= q_{\alpha\beta} x_{\beta}x_{\alpha}+ \Bsj x_{\gamma}, & &\Bsj\in\Bbbk.
\end{align*}
\item $\alpha=\alpha_{1\, i-1}$, $\beta=\alpha_{k\theta}+\alpha_{i\theta}$, $k<i\le \theta$. The relations between root vectors are of the form
\begin{align*}
x_{\alpha}x_{\gamma} &= q_{\alpha\gamma} x_{\gamma}x_{\alpha}, &
x_{\gamma}x_{\beta} &= q_{\gamma\beta} x_{\beta}x_{\gamma}, &
x_{\alpha}x_{\beta} &= q_{\alpha\beta} x_{\beta}x_{\alpha}+ \Bsj x_{\gamma}+\sum_{t=k+1}^{i-1} \Bsj_t x_{\alpha_{k\theta}+\alpha_{t\theta}} x_{\alpha_{1\, t-1}},
\end{align*}
for some  $\Bsj, \Bsj_t\in\Bbbk$.
\item $\alpha=\alpha_{1\theta}$, $\beta=\alpha_{k\theta}$. Arguing as in Lemma \ref{lem:Btheta-relations-root-vectors},
\begin{align*}
x_{\alpha}x_{\gamma} &= q_{\alpha\gamma} x_{\gamma}x_{\alpha}, &
x_{\gamma}x_{\beta} &= q_{\gamma\beta} x_{\beta}x_{\gamma}, &
x_{\alpha}x_{\beta} &= q_{\alpha\beta} x_{\beta}x_{\alpha}+ \Bsj x_{\gamma}+\sum_{t=k+1}^{\theta-1} \Bsj_t x_{\alpha_{1\theta}+\alpha_{t\theta}} x_{\alpha_{k\, t-1}},
\end{align*}
for some $\Bsj, \Bsj_t\in\Bbbk$.
\end{itemize}
In all cases the root vectors satisfy \eqref{eq:diff2-hypothesis}, and $-\frac{q_{\alpha\alpha}}{q_{\beta\beta}}=q^{\pm 2}$, so we take $L=M$. 

Next we look for solutions of \eqref{eq:equation-roots-cocycle-N=2}. 
In the Cartan case with $N=4$, we have $\sum_{\delta\in\varDelta_+} n_{\delta}=3$: we easily discard the possibility that $n_{\delta}\ge 2$ for some $\delta\in\varDelta_+$, so exactly three of them are one, and $n_{\varphi}= 0$ for the remaining $\varphi \in\varDelta_+$. Arguing as in the case $\gamma=\alpha_{1\theta}$ we check that all solutions are of form $n_{\gamma}=n_{\alpha}=n_{\beta}=1$ for a pair $(\alpha, \beta)$ satisfying \eqref{eq:alfa-beta-N=2}.
Next we consider the case $\superb{j}{\theta-j}$, $k>j$. 
We have that $\sum_{\delta\in\varDelta_+: \, 1\in\supp \delta} f_{\delta}(n_{\delta})=M$.
Let $\eta\in\varDelta_+$ such that $1\in\supp \eta$ and $n_{\eta}>0$. 

\begin{itemize}[leftmargin=*]
\item If $\eta=\alpha_{1\theta}$, then $n_{\eta}=1$: otherwise, $f_{\eta}(n_{\eta})\ge P =2M>M=\sum_{\delta\in\varDelta_+: \, 1\in\supp \delta} f_{\delta}(n_{\delta})$, a contradiction.

\item If $\eta=\alpha_{1i}$, $i<j$, then $n_{\eta}=1$. Suppose on the contrary that $n_{\eta}>1$. Then
$n_{\eta}=2$, since $f_{\eta}(n_{\eta})\le M=N_{\eta}$.
This implies that $n_{\delta}=0$ for all $\delta\neq \alpha_{1i}$ such that $i\in\supp \delta$.

Let $\delta\in\varDelta_+$ such that $a_{i+1}^{\delta}=2$. Then $n_{\delta}=0$ since $i\in\supp \delta$. Thus
\begin{align*}
M &=\sum_{\delta\in\varDelta_+: \, a_{i+1}^{\delta}=1} f_{\delta}(n_{\delta}), & \sum_{\delta\in\varDelta_+: \, a_{i+1}^{\delta}=1} n_{\delta} &<M.
\end{align*} 
Then $n_{i+1 \, \ell}=2$ for some $\ell>i$, which implies that $n_{\delta}=0$ for all $\delta\neq \alpha_{i+1 \, \ell}$ such that $\ell\in\supp \delta$. Recursively, there exist $i_0=1<i_1=i<i_2< \dots <i_s=j-1$ such that 
\begin{align*}
n_{\delta}&=\begin{cases}
2 & \delta=\alpha_{i_{r-1} i_r}, \, r\in\I_s,
\\
0 & \delta \neq \alpha_{i_{r-1} i_r}, \, \supp\delta \cap \I_{j-1} \neq \emptyset.
\end{cases}
\end{align*}
On the other hand, if $a_j^{\delta}=1$, then either $N_{\delta}=1$, or $\delta=\alpha_{i\theta}$, $i<j$, in which case $n_{\delta}\leq 1$ since $N_{\delta}=P=2M$. Hence $n_{\delta}=f_{\delta}(n_{\delta})$ for all $\delta$ such that $a_j^{\delta}=1$. Using this fact,
\begin{align*}
M &=\sum_{\delta\in\varDelta_+} f_{\delta}(n_{\delta})a_{j}^{\delta}
=
\sum_{\delta\in\varDelta_+: \, a_j^{\delta}=1} n_{\delta}
+2 \sum_{\delta\in\varDelta_+: \, a_j^{\delta}=2} f_{\delta}(n_{\delta})
\\ &\le
\sum_{\delta\in\varDelta_+-\{\eta\}} n_{\delta}
+2 \sum_{\delta\in\varDelta_+: \, a_j^{\delta}=2} f_{\delta}(n_{\delta})
\leq M-1+2 \sum_{\delta\in\varDelta_+: \, a_j^{\delta}=2} f_{\delta}(n_{\delta}).
\end{align*} 
Thus $n_{\mu}\ne 0$ for some $\mu\in\varDelta_+$ such that $a_j^{\mu}=2$; but $j-1\in\supp \mu$, a contradiction.

\item If either $\eta=\alpha_{1i}$ or else $\eta=\alpha_{1\theta}+\alpha_{i\theta}$, $j\le i<\theta$, then $N_{\eta}=2$, so $f_{\eta}(n_{\eta})=n_{\eta}$.

\item If $\eta=\alpha_{1\theta}+\alpha_{i\theta}$, $i<j$, then $n_{\eta}=1$. Otherwise, 
\begin{align*}
M=\sum_{\delta\in\varDelta_+: \, j-1\in\supp \delta} f_{\delta}(n_{\delta})a_{j-1}^{\delta} \ge 2 f_{\eta}(n_{\eta}) \ge 2M,
\end{align*}
and we get a contradiction.
\end{itemize}
Hence $f_{\delta}(n_{\delta})=n_{\delta}$ for all $\delta$ such that $1\in\supp \delta$. Therefore,
\begin{align*}
\sum_{\delta\in\varDelta_+: \, 1\notin\supp \delta} n_{\delta}=M+1-\sum_{\delta\in\varDelta_+: \, 1\in\supp \delta} n_{\delta}=M+1-\sum_{\delta\in\varDelta_+: \, 1\in\supp \delta} f_{\delta}(n_{\delta})=1.
\end{align*}
That is, there exists a unique root $\beta\in\varDelta_+$ such that $1\notin\supp \beta$ and $n_{\beta}\neq 0$; moreover $n_{\beta}=1$ for this root $\beta$. Thus $f_{\delta}(n_{\delta})=n_{\delta}$ for all $\delta$ and we may translate to solve \eqref{eq:equation-roots-cocycle-N=2} as follows: find $\gamma_t\in\varDelta_+$, $t\in\I_{M+1}$, not necessarily different such that $\sum_t \gamma_t=M\gamma$. Write $\gamma_t=\sum_{i\in\I} a_i^{(t)} \alpha_i$. Hence we may assume that $a_1^{(t)}=1$ for $t\in\I_M$, $a_{1}^{(M+1)}=0$. As $\sum_{t\in\I_{M+1}} a_k^{(t)}=2M$ and each $a_{k}^{(t)}\le 2$, at least $M-1$ of them are $2$, and for the other two $a_{k}^{(t)}$'s, either both are $1$, or else one of them is $2$ and the other is $0$. Hence we may assume that $a_k^{(t)}=1$ for $t\in\I_{M-1}$. This forces to have $\gamma_t=\gamma$ for $t\in\I_{M-1}$, $\gamma_M+\gamma_{M+1}=\gamma$.
\medbreak

Hence all the hypotheses of Proposition \ref{prop:cocycle-xgamma-N=2} hold, and $(\x_{\gamma}^{L_{\gamma}})^*$ is a cocycle.
\epf

\subsection{Type \texorpdfstring{$B_{\theta,j}$ standard, $j \in \I_{\theta-1}$}{}}\label{sec:Bn-standard}

Here $\zeta \in\G_3'$.
In this subsection, we deal with Nichols algebras $\toba_{\bq}$ of standard type $B_{\theta,j}$. We assume that the corresponding diagram is
\begin{align*}
\xymatrix{ \overset{-\zeta}{\underset{\ }{\circ}}\ar  @{-}[r]^{-\ztu}  
& \overset{-\zeta}{\underset{\ }{\circ}}\ar@{.}[r] 
& \overset{-\zeta}{\underset{\ }{\circ}}\ar  @{-}[r]^{-\ztu} 
& \overset{-1}{\underset{j}{\circ}}\ar  @{-}[r]^{-\zeta} 
&  \overset{-\ztu}{\underset{\ }{\circ}}\ar@{.}[r] 
& \overset{-\ztu}{\underset{\ }{\circ}} \ar  @{-}[r]^{-\zeta}&
\overset{\zeta}{\underset{\ }{\circ}}}
\end{align*}
The set of positive roots is \eqref{eq:root-system-Bn}, and we fix the same convex order, see \eqref{eq:convex-order-Bn}.
For more information, see \cite[\S 6.1]{AA17}.
We prove that Condition \ref{assumption:intro-combinatorial} holds for type
$B_{\theta,j}$ standard:

\begin{prop}\label{prop:roots-cocycles-Btheta-standard}
For every  $\gamma \in \varDelta_+^{\bq}$, there exists $L_{\gamma}\in\N$ 
such that $(\x_{\gamma}^{L_\gamma})^*$ is a cocycle.
\end{prop}

\pf 
Let $\gamma$ be a positive non-simple root.
Arguing recursively we may assume that $\gamma$ has full support. Hence, either $\gamma=\alpha_{1\theta}$ or else there exists $k\in\I_{2, \theta}$ such that $\gamma=\alpha_{1\theta} +\alpha_{k\theta}$.

\medbreak

First we consider $\gamma=\alpha_{1\theta}$. Here,
$P_{\gamma} = 3$ and $Q_{\gamma} = 2$, $N_{\gamma}=3$. All the pairs $(\alpha,\beta)$ as in \eqref{eq:alfa-beta} are of the form $\alpha=\alpha_{1k}$, $\beta=\alpha_{1\theta}+\alpha_{k+1 \, \theta}$ for some $k\in\I_{\theta-1}$. By Lemma \ref{lem:Btheta-relations-root-vectors},
\begin{align*}
x_{\alpha}x_{\gamma} &= q_{\alpha\gamma} x_{\gamma}x_{\alpha}, &
x_{\gamma}x_{\beta} &= q_{\gamma\beta} x_{\beta}x_{\gamma}, &
x_{\alpha}x_{\beta} &= q_{\alpha\beta} x_{\beta}x_{\alpha}+ \Bsj_k x_{\gamma}^2, & \text{for some }&\Bsj_{k}\in \Bbbk,
\end{align*}
so the root vectors satisfy \eqref{eq:diff2-hypothesis}. Also,
$\frac{q_{\alpha\gamma}}{q_{\gamma\beta}}=1$ for all of them, so we take $L=1$. Now we look for solutions of \eqref{eq:equation-roots-cocycle-N>2}. As in the proof of Proposition \ref{prop:roots-cocycles-Btheta}, there exists a pair $(\alpha,\beta)$ as in \eqref{eq:alfa-beta} such that  $n_{\gamma}=n_{\alpha}=n_{\beta}=1$, and $n_{\delta}=0$ for the remaining roots $\delta\in\varDelta_+^{\bq}$. Thus Proposition \ref{prop:cocycle-xgamma-N>2} applies and $(\x_{\gamma}^{3})^{\ast}$ is a $2$-cocycle.

\medbreak

Now we consider $\gamma=\alpha_{1\theta}+\alpha_{k\theta}$, $k\in\I_{2, \theta}$.
Here we have
$P_{\gamma} = 2$ and $Q_{\gamma} = 1$. Hence, if $k\le j$, then  Lemma \ref{lem:largeN} applies since $N_{\gamma}=6$, so $(\x_{\gamma}^{6})^*$ is a $2$-cocycle. 
Next we consider the case $k>j$, so $N_{\gamma}=2$. 
The pairs $\alpha<\beta$ as in \eqref{eq:alfa-beta-N=2} are the same as in Proposition \ref{prop:roots-cocycles-Btheta}, the root vectors satisfy \eqref{eq:diff2-hypothesis} and $-\frac{q_{\alpha\alpha}}{q_{\beta\beta}}=-\zeta^{\pm 1}$, so we take $L=6$. Also, the same argument as in the case $\superb{j}{\theta-j}$ in Proposition \ref{prop:roots-cocycles-Btheta} shows that 
there exists a pair $(\alpha,\beta)$ as in \eqref{eq:alfa-beta} such that  $n_{\gamma}=4$, $n_{\alpha}=n_{\beta}=1$, and $n_{\delta}=0$ for the remaining roots $\delta\in\varDelta_+^{\bq}$. Hence all the hypotheses of Proposition \ref{prop:cocycle-xgamma-N=2} hold, and $(\x_{\gamma}^{6})^*$ is a cocycle.
\epf

\subsection{Type \texorpdfstring{$C_{\theta}$}{}}\label{sec:Cn}

Let $q$ be a root of 1 of order $N>2$.
In this subsection, we deal with the Nichols algebras $\toba_{\bq}$ of diagonal type $C_{\theta}$: the Dynkin diagram is
\begin{align*}
\xymatrix{ \overset{q}{\underset{\ }{\circ}}\ar  @{-}[r]^{q^{-1}}  &
\overset{q}{\underset{\ }{\circ}}\ar  @{-}[r]^{q^{-1}} 
&  \overset{q}{\underset{\ }{\circ}}\ar@{.}[r] 
& \overset{q}{\underset{\ }{\circ}} \ar  @{-}[r]^{q^{-2}}&
\overset{q^2}{\underset{\ }{\circ}}}.
\end{align*}
The set of positive roots is
\begin{align}\label{eq:roots-Cn}
\varDelta_+ =\{\alpha_{ij}\,|\, i \leq j \in \I \} \cup
\{\alpha_{i\theta} + \alpha_{j\, \theta-1}\,|\, i \le j \in \I_{\theta-1} \}. 
\end{align}
We fix the following convex order:
\begin{align}\label{eq:roots-Cn-convex-order}
\begin{aligned}
&\alpha_{1}< \alpha_{12}< \dots < \alpha_{1\theta-1} < \alpha_{1\theta} + \alpha_{1 \, \theta-1} < \alpha_{1\theta} + \alpha_{\theta-1} < \dots <\alpha_{1\theta} + \alpha_{2\theta-1}
\\
&< \alpha_{1\theta} < \alpha_{2}< \alpha_{23}< \dots <\alpha_{\theta-1} 
<\alpha_{\theta-1\,\theta}<\alpha_{\theta-1 \, \theta}+\alpha_{\theta-1}<\alpha_{\theta}.
\end{aligned}
\end{align}

Let $M=\ord q^2$. It follows from the definition that
\begin{align*}
N_{\gamma} &= \begin{cases} 
M, & \gamma=\alpha_{i\theta}+\alpha_{i\theta-1} \text{ or } \gamma= \alpha_{\theta};
\\ 
N, & \gamma=\alpha_{i\theta}+\alpha_{j\theta-1}, \, i<j, \text{ or }\gamma=\alpha_{ij}, \, (i,j)\ne (\theta,\theta).
\end{cases}
\end{align*}

For more information, see \cite[\S 4.3]{AA17}.
The aim of this Section is to prove that Condition \ref{assumption:intro-combinatorial} holds for type
$C_{\theta}$. More precisely,

\begin{prop}\label{prop:roots-cocycles-Ctheta}
For every $\gamma  \in \varDelta_+^{\bq}$,  $(\x_{\gamma}^{N_\gamma})^*$ is a $2$-cocycle.
\end{prop}

\pf 
Let $\gamma$ be a positive non-simple root.
Arguing recursively we may assume that $\gamma$ has full support. Hence, either $\gamma=\alpha_{1\theta}$ or else there exists $j\in\I_{\theta-1}$ such that $\gamma=\alpha_{1\theta} +\alpha_{j\, \theta-1}$.

\medbreak

First we consider $\gamma=\alpha_{1\theta}$. In this case we have
$P_{\gamma} = 3$ and $Q_{\gamma} = 1$. Hence, if $N_{\gamma}>3$, then Lemma \ref{lem:largeN} applies and $(\x_{\gamma}^{N_\gamma})^*$ is a $2$-cocycle. 

Now assume that $N_{\gamma}=3$: that is, $N=3$. The unique pair $(\alpha,\beta)$ as in \eqref{eq:alfa-beta} is $\alpha=\alpha_{1\theta}+\alpha_{1 \, \theta-1}$, $\beta=\alpha_{\theta}$. 
Arguing as in Lemma \ref{lem:Btheta-relations-root-vectors}, the relations are of the form
\footnote{Indeed, $x_{\alpha}=[x_{1\,\theta-1}, x_{1\theta}]_c$, $x_{\gamma}=x_{1\theta}$, and using $q$-Jacobi identity we have that $\Bsj=q (q-1) \prod_{i\in\I_{\theta-1}} q_{i\theta}$.}
\begin{align*}
x_{\alpha}x_{\gamma} &= q_{\alpha\gamma} x_{\gamma}x_{\alpha}, &
x_{\gamma}x_{\beta} &= q_{\gamma\beta} x_{\beta}x_{\gamma}, &
x_{\alpha}x_{\beta} &= q_{\alpha\beta} x_{\beta}x_{\alpha}+ \Bsj x_{\gamma}^2 & &\text{ for some }\Bsj\in\Bbbk,
\end{align*}
and then the root vectors satisfy \eqref{eq:diff2-hypothesis}. Also,
$\frac{q_{\alpha\gamma}}{q_{\gamma\beta}}=1$ for all of them, so we take $L=1$. Now we look for solutions of \eqref{eq:equation-roots-cocycle-N>2}. That is, $\sum_{\delta\in\varDelta_+} f_\delta(n_\delta) \delta =3\gamma$, $\sum_{\delta\in\varDelta_+} n_\delta = 3$. We check that there is no solution with $n_{\delta}=3$ neither with $n_{\delta}=2$,  $\delta\in\varDelta_+$. Hence we look for solutions with $n_{\gamma_t}=1$, for three different roots $\gamma_t\in\varDelta_+$ (and $n_{\delta}=0$ for the other roots). We write $\gamma_t=\sum_{i\in\I} a_{i}^{(t)}\alpha_i$. As $\sum_t a_i^{(t)}=3$ for all $i\in\I$, we have $a_{\theta}^{(t)}=1$ for all $t\in\I_3$, and either $a_1^{(t)}=1$ for all $t\in\I_3$, or $a_1^{(t)}=2$ for some $t\in\I_3$. The case $a_\theta^{(t)}=1$ for all $t\in\I_3$ gives a contradiction: either $\gamma_t=\gamma$ for all $t$, or else $\sum_t a_{\theta-1}^{(t)}>3$. Hence we may assume $a_{1}^{(1)}=2$, $a_{1}^{(2)}=0$, $a_{1}^{(3)}=1$.
Then $\gamma_1=\alpha_{1\theta}+\alpha_{1 \, \theta-1}$, $\gamma_2=\alpha_{\theta}$, $\gamma_3=\gamma$.
Thus Proposition \ref{prop:cocycle-xgamma-N>2} applies and $(\x_{\gamma}^{3})^{\ast}$ is a $2$-cocycle.

\medbreak
Now we consider $\gamma=\alpha_{1\theta}+\alpha_{j\theta-1}$, $j\in\I_{\theta-1}$.
In this case,
$P_{\gamma} = 2$ and $Q_{\gamma} = 1$. If $N_{\gamma}>2$, then Lemma \ref{lem:largeN} applies and $(\x_{\gamma}^{N_\gamma})^*$ is a $2$-cocycle. 
Hence we need to study the case $N_{\gamma}=2$: that is, $N=4$ and  $\gamma=\alpha_{1\theta}+\alpha_{1\theta-1}$.
Let $\alpha<\beta$ be a pair of positive roots as in \eqref{eq:alfa-beta-N=2}. We have the following possibilities:
\begin{itemize}[leftmargin=*,label=$\circ$]
\item $\alpha=\alpha_{1 \, i-1}$, $\beta=\alpha_{1\theta}+\alpha_{i\theta}$, $i\in\I_{2, \theta-1}$. There exist $\Bsj, \Bsj_t\in\Bbbk$ such that
\begin{align*}
x_{\alpha}x_{\gamma} &= q_{\alpha\gamma} x_{\gamma}x_{\alpha}, &
x_{\gamma}x_{\beta} &= q_{\gamma\beta} x_{\beta}x_{\gamma}, &
x_{\alpha}x_{\beta} &= q_{\alpha\beta} x_{\beta}x_{\alpha}+ \Bsj x_{\gamma}+\sum_{t=i+1}^{\theta-1} \Bsj_t x_{\alpha_{1\theta}+\alpha_{t\, \theta-1}} x_{\alpha_{1\, t-1}}.
\end{align*}
\item $\alpha=\alpha_{1\theta}$, $\beta=\alpha_{1\, \theta-1}$. In this case,
\begin{align*}
x_{\alpha}x_{\gamma} &= q_{\alpha\gamma} x_{\gamma}x_{\alpha}, &
x_{\gamma}x_{\beta} &= q_{\gamma\beta} x_{\beta}x_{\gamma}, &
x_{\alpha}x_{\beta} &= q_{\alpha\beta} x_{\beta}x_{\alpha}+ x_{\gamma}.
\end{align*}
\end{itemize}
In all cases the root vectors satisfy \eqref{eq:diff2-hypothesis}, and $-\frac{q_{\alpha\alpha}}{q_{\beta\beta}}=-1$, so we take $L=2$. 

Next we look for solutions of \eqref{eq:equation-roots-cocycle-N=2}. That is, $\sum_{\delta\in\varDelta_+} f_\delta(n_\delta) \delta =2\gamma$, $\sum_{\delta\in\varDelta_+} n_\delta = 3$.
Let $\mu\in\varDelta_+$ be such that $n_{\mu}\neq 0$, $\theta\in\supp \mu$. Notice that 
\begin{align*}
2= \sum_{\delta\in\varDelta_+: \, \theta\in\supp \delta} f_\delta(n_\delta)a_{\theta}^{\delta} = \sum_{\delta\in\varDelta_+: \, \theta\in\supp \delta} f_\delta(n_\delta),
\end{align*} 
so $n_{\mu} \le f_\mu(n_\mu)\le 2$. Suppose that $n_{\mu}=2$. Then $\mu=\alpha_{i\theta}+\alpha_{i \, \theta-1}$ for some $i\in\I_{\theta-1}$ since $N_{\mu}=f_\mu(2)\le 2$, and there exists $\eta\neq \mu$ such that $n_{\eta}=1$, $n_{\delta}=0$ if $\delta\neq \mu,\eta$. But then
\begin{align*}
4\alpha_{1 \, i-1} &=2\gamma-2\mu = \sum_{\delta\neq \mu} f_\delta(n_\delta) \delta  = f_\eta(1) \eta =\eta,
\end{align*}
a contradiction. Hence $n_{\mu}=1$ for all $\mu\in\varDelta_+$ such that $n_{\mu}\neq 0$, $\theta\in\supp \mu$. Then there exist three different roots $\gamma_i\in\varDelta_+$ such that $n_{\gamma_i}=1$, and we may assume that $\theta\in\supp\gamma_2\cap\supp\gamma_3$, $\theta\notin\supp\gamma_1$. As $\gamma_2\neq \gamma_3$, we may assume $\gamma_2\neq \gamma$, so $a_{1}^{\gamma_2}=1$. This implies that $a_{1}^{\gamma_1}=1$, so $\gamma_1=\alpha_{1i}$ for some $i\in\I_{\theta-1}$, and $a_{1}^{\gamma_3}=2$, so $\gamma_3=\gamma$.

\medbreak

Hence all the hypotheses of Proposition \ref{prop:cocycle-xgamma-N=2} hold, and $(\x_{\gamma}^{N_{\gamma}})^*$ is a $2$-cocycle.
\epf

\subsection{Type \texorpdfstring{$D_\theta$}{}}\label{sec:D-theta}

Let $q\in\G_N'$, $n\ge 2$. 
Let $\toba_{\bq}$ be a Nichols algebra of type $D_{\theta}$. That is, the generalized Dynkin diagram of $\toba_{\bq}$ has the form
\begin{align*}
\xymatrix{ & & & &  \overset{q}{\circ} &\\
\overset{q}{\underset{\ }{\circ}}\ar  @{-}[r]^{q^{-1}}  & \overset{q}{\underset{\
}{\circ}}\ar  @{-}[r]^{q^{-1}} &  \overset{q}{\underset{\ }{\circ}}\ar@{.}[r] &
\overset{q}{\underset{\ }{\circ}} \ar  @{-}[r]^{q^{-1}}  & \overset{q}{\underset{\
}{\circ}} \ar @<0.7ex> @{-}[u]_{q^{-1}}^{\quad} \ar  @{-}[r]^{q^{-1}} &
\overset{q}{\underset{\ }{\circ}}}
\end{align*}
The set of positive roots is
\begin{align}\label{eq:root-system-D}
\begin{aligned}
\varDelta^\bq_+&=\{\alpha_{i\, j}\,|\, i\leq j\in\I, \, (i,j)\neq (\theta-1,\theta) \}
\\ 
& \quad \cup \{\alpha_{i\, \theta-2}+\alpha_{\theta}\,|\, i\in\I_{\theta-2} \} \cup \{\alpha_{i\, \theta}+\alpha_{j\, \theta-2}\,|\, i<j\in\I_{\theta-2} \}.
\end{aligned}\end{align}
We fix the following convex order:
\begin{align*}
&\alpha_{11}< \alpha_{12}< \dots <\alpha_{1 \, \theta-1}< \alpha_{1\, \theta-2}+\alpha_{\theta} < \alpha_{1\theta} < \alpha_{1\theta} + \alpha_{\theta-2} < \dots 
\\
&< \alpha_{1\theta} + \alpha_{2 \, \theta-2}< \alpha_{22}< \alpha_{23}< \dots <\alpha_{\theta-2} 
<\alpha_{\theta-2\,\theta-1}<\alpha_{\theta-2}+\alpha_{\theta}
\\
&<\alpha_{\theta-2\, \theta}<\alpha_{\theta-1}<\alpha_{\theta}.
\end{align*}

For more information, see \cite[\S 4.4]{AA17}.
The aim of this Section is to prove that Condition \ref{assumption:intro-combinatorial} holds for type
$D_{\theta}$. More precisely,
\begin{prop}\label{prop:roots-cocycles-Dtheta}
For every $\gamma  \in \varDelta_+^{\bq}$,  $(\x_{\gamma}^{N_\gamma})^*$ is a $2$-cocycle.
\end{prop}

\pf By Lemma \ref{lema:non-simple-roots-Palfa}, $P_{\alpha} = 2$ and $Q_\alpha = 1$ for all non-simple roots $\alpha$. Hence, if $N>2$, then Lemma \ref{lem:largeN} applies and $(\x_{\gamma}^{N_\gamma})^*$ is a $2$-cocycle for all roots $\alpha$.

\bigbreak

Next we assume $N=2$, that is, $q=-1$. We will apply Proposition \ref{prop:cocycle-xgamma-N=2}. 
Let $\gamma$ be a positive non-simple root.
Arguing recursively we may assume that $\gamma$ has full support. Hence, either $\gamma=\alpha_{1\theta}$ or else there exists $k\in\I_{2, \theta-2}$ such that $\gamma=\alpha_{1\theta}+\alpha_{k\, \theta-2}$. We look for pairs $\alpha<\beta \in\varDelta_+$ such that $\gamma=\alpha+\beta$. Notice that $q_{\alpha\alpha}=q_{\beta\beta}=-1=-\frac{q_{\alpha\alpha}}{q_{\beta\beta}}$ in any case so we may guess that $L=2$.

\medbreak

First we consider $\gamma=\alpha_{1\theta}$. All the pairs $(\alpha,\beta)$ as in \eqref{eq:alfa-beta-N=2} are of the form $\alpha=\alpha_{1j}$, $\beta=\alpha_{j+1 \, \theta}$ for some $j\in\I_{\theta-1}$. Similar to Lemma \ref{lem:Btheta-relations-root-vectors},
\begin{align*}
x_{\alpha}x_{\gamma} &= q_{\alpha\gamma} x_{\gamma}x_{\alpha}, &
x_{\gamma}x_{\beta} &= q_{\gamma\beta} x_{\beta}x_{\gamma}, &
x_{\alpha}x_{\beta} &= q_{\alpha\beta} x_{\beta}x_{\alpha}+ \Bsj x_{\gamma} &&\text{ for some } \Bsj\in\Bbbk,
\end{align*}
so the root vectors satisfy \eqref{eq:diff2-hypothesis}. Next we look for solutions of \eqref{eq:equation-roots-cocycle-N=2}. That is, $2\gamma=\delta_1+\delta_2+\delta_3$, $\delta_i\in\varDelta_+$. We write $\delta_j=\sum_{i\in\I} a_i^{(j)} \alpha_i$. As $a_i^{(j)}$ is 0 or 1 for $i=1,\theta-1,\theta$, we may fix $a_1^{(1)}=a_1^{(2)}=1$, $a_1^{(3)}=0$ and see the possible pairs of roots such that $a_{\theta-1}^{(m)}=a_{\theta-1}^{(n)}=1$, respectively $a_{\theta}^{(p)}=a_{\theta}^{(r)}=1$. Suppose that no one of the $\delta_i$'s has coefficient $1$ for the three simple roots simultaneously. Then we may assume $a_{\theta-1}^{(1)}=a_{\theta-1}^{(3)}=1$,  $a_{\theta}^{(2)}=a_{\theta}^{(3)}=1$, so $a_{\theta-2}^{(j)}>0$ for all $j\in\I_3$, a contradiction. Hence we assume $a_{\theta-1}^{(1)}=a_{\theta}^{(1)}=1$, so $a_{i}^{(1)}\ge 1$ for all $i\in\I$. If either $a_{\theta-1}^{(2)}=1$ or  $a_{\theta}^{(2)}=1$, then $a_{\theta-2}^{(2)}\ge 1$, which implies
$a_{\theta-2}^{(1)}=1$ and so $\delta_1=\gamma$. Otherwise 
$a_{\theta-1}^{(3)}=1=a_{\theta}^{(3)}$, then $a_{\theta-2}^{(3)}\ge 1$, which implies again $\delta_1=\gamma$.

\medbreak
Finally, let $\gamma=\alpha_{1\theta}+\alpha_{k\, \theta-2}$.
Let $\alpha<\beta$ be a pair of positive roots as in \eqref{eq:alfa-beta-N=2}. Then the coefficient of $\alpha_1$ is one for just one of them (and zero for the other): it should be $\alpha$, since $\alpha<\beta$. We have several possibilities:
\begin{itemize}[leftmargin=*,label=$\circ$]
\item $\alpha=\alpha_{1 \, j-1}$, $\beta=\alpha_{j\theta}+\alpha_{k\, \theta-2}$, $j<k$. Then
\begin{align*}
x_{\alpha}x_{\gamma} &= q_{\alpha\gamma} x_{\gamma}x_{\alpha}, &
x_{\gamma}x_{\beta} &= q_{\gamma\beta} x_{\beta}x_{\gamma}, &
x_{\alpha}x_{\beta} &= q_{\alpha\beta} x_{\beta}x_{\alpha}+ \Bsj x_{\gamma} && \text{ for some } \Bsj\in\Bbbk.
\end{align*}
\item $\alpha=\alpha_{1\, j-1}$, $\beta=\alpha_{k\theta}+\alpha_{j\, \theta-2}$, $k<j\le \theta-2$. Then
\begin{align*}
x_{\alpha}x_{\gamma} &= q_{\alpha\gamma} x_{\gamma}x_{\alpha}, &
x_{\gamma}x_{\beta} &= q_{\gamma\beta} x_{\beta}x_{\gamma}, &
x_{\alpha}x_{\beta} &= q_{\alpha\beta} x_{\beta}x_{\alpha}+ \Bsj x_{\gamma}+\sum_{t=k+2}^{j-1} \Bsj_t x_{\alpha_{k\theta}+\alpha_{t\, \theta-2}} x_{\alpha_{1\, t-1}},
\end{align*}
for some $\Bsj, \Bsj_t\in\Bbbk$.
\item $\alpha=\alpha_{1\, \theta-1}$, $\beta=\alpha_{k\theta-2}+\alpha_{\theta}$.
Then
\begin{align*}
x_{\alpha}x_{\gamma} &= q_{\alpha\gamma} x_{\gamma}x_{\alpha}, &
x_{\gamma}x_{\beta} &= q_{\gamma\beta} x_{\beta}x_{\gamma}, &
x_{\alpha}x_{\beta} &= q_{\alpha\beta} x_{\beta}x_{\alpha}+ \Bsj x_{\gamma},
\end{align*}
for some $\Bsj, \Bsj_t\in\Bbbk$.

\item $\alpha=\alpha_{1\theta}$, $\beta=\alpha_{k \, \theta-2}$. Then
\begin{align*}
x_{\alpha}x_{\gamma} &= q_{\alpha\gamma} x_{\gamma}x_{\alpha}, &
x_{\gamma}x_{\beta} &= q_{\gamma\beta} x_{\beta}x_{\gamma}, &
x_{\alpha}x_{\beta} &= q_{\alpha\beta} x_{\beta}x_{\alpha}+ \Bsj x_{\gamma}+\sum_{t=k+1}^{j-1} \Bsj_t x_{\alpha_{1\theta}+\alpha_{t\, \theta-2}} x_{\alpha_{k\, t-1}},
\end{align*}
for some $\Bsj, \Bsj_t\in\Bbbk$.

\item $\alpha=\alpha_{1\theta}+\alpha_{j\, \theta-2}$, $\beta=\alpha_{k \, j-1}$, $k<j\le \theta-2$. Then
\begin{align*}
x_{\alpha}x_{\gamma} &= q_{\alpha\gamma} x_{\gamma}x_{\alpha}, &
x_{\gamma}x_{\beta} &= q_{\gamma\beta} x_{\beta}x_{\gamma}, &
x_{\alpha}x_{\beta} &= q_{\alpha\beta} x_{\beta}x_{\alpha}+ \Bsj x_{\gamma}+\sum_{t=k+1}^{j-1} \Bsj_t x_{\alpha_{1\theta}+\alpha_{t\, \theta-2}} x_{\alpha_{k\, t-1}},
\end{align*}
for some $\Bsj, \Bsj_t\in\Bbbk$.

\end{itemize}

In each case the justification relies on the homogeneity of the relations and is similar to
Lemma \ref{lem:Btheta-relations-root-vectors}; we leave the details to an interested reader.
Therefore the root vectors satisfy \eqref{eq:diff2-hypothesis}. Next we look for solutions of \eqref{eq:equation-roots-cocycle-N=2}. That is, $2\gamma=\delta_1+\delta_2+\delta_3$, $\delta_i\in\varDelta_+$. We write $\delta_j=\sum_{i\in\I} a_i^{(j)} \alpha_i$. 

When $k=\theta-2$, first consider the case $\delta_1=\alpha_{\theta-2}$. Then $a_{j}^{(i)}=1$ for $i=1,\theta-1,\theta$ and $j=2,3$, so $\delta_2$, $\delta_3$ have full support. This implies that $a_{j}^{(i)}=1$ for $2\le i <\theta-2$ and $j=2,3$, and we need that $a_{j}^{(\theta-2)}=1$ for one of them; that is, either $\delta_2=\gamma$ or else $\delta_3=\gamma$. If $\delta_j \neq \alpha_{\theta-2}$ for all $j\in\I_3$, then
\begin{align*}
2 \alpha_{1\theta} &= 2 s_{\theta-2}(\gamma) = s_{\theta-2}(\delta_1) + s_{\theta-2}(\delta_2) + s_{\theta-2}(\delta_3), & s_{\theta-2}(\delta_j) &\in\varDelta_+.
\end{align*}
Applying the previous case, $s_{\theta-2}(\delta_j)=\alpha_{1\theta}$ for some $j\in\I_3$, so $\delta_j=\gamma$.

If $k<\theta-2$, then we argue recursively. Indeed, we first consider the case $\delta_1=\alpha_k$ and argue as in the case $k=\theta-2$ to show that either $\delta_2=\gamma$ or else $\delta_3=\gamma$. If $\delta_j \neq \alpha_{\theta-2}$ for all $j\in\I_3$, then
\begin{align*}
2 \alpha_{1\theta}+ \alpha_{k+1 \, \theta-2} &= 2 s_{k}(\gamma) = s_{k}(\delta_1) + s_{k}(\delta_2) + s_{k}(\delta_3), & s_{k}(\delta_j) &\in\varDelta_+.
\end{align*}
Hence $s_{k}(\delta_j)=\alpha_{1\theta}+ \alpha_{k+1 \, \theta-2}$ for some $j\in\I_3$, which means that $\delta_j=\gamma$ for some $j\in\I_3$.

\medbreak

Hence all the hypotheses of Proposition \ref{prop:cocycle-xgamma-N=2} hold, and $(\x_{\gamma}^2)^*$ is a $2$-cocycle.
\epf

\subsection{Type \texorpdfstring{$\superd{j}{\theta - j}$, $\theta \ge 1$, $j \in \I_{\theta-1}$}{}}\label{sec:Dsuper}

Let $q$ be a root of 1 of order $N>2$.
In this subsection, we deal with the Nichols algebras $\toba_{\bq}$ of type $\superd{j}{\theta - j}$. 
We may assume that the corresponding diagram is
\begin{align*}
\xymatrix{ \overset{q^{-1}}{\underset{\ }{\circ}}\ar  @{-}[r]^{q}  
& \overset{q^{-1}}{\underset{\ }{\circ}}\ar@{.}[r] 
& \overset{q^{-1}}{\underset{\ }{\circ}}\ar  @{-}[r]^{q} 
& \overset{-1}{\underset{j}{\circ}}\ar  @{-}[r]^{q^{-1}} 
&  \overset{q}{\underset{\ }{\circ}}\ar@{.}[r] 
& \overset{q}{\underset{\ }{\circ}} \ar  @{-}[r]^{q^{-2}}&
\overset{q^2}{\underset{\ }{\circ}}}
\end{align*}
The set of positive roots is
\begin{align}\label{eq:roots-Dsuper}
\begin{aligned}
\varDelta_+^{\bq} &=\{\alpha_{ik}\,|\, i \leq k \in \I \} \cup
\{\alpha_{i\theta} + \alpha_{k\, \theta-1}\,|\, i < k \in \I_{\theta-1} \} 
\\
& \quad \cup \{\alpha_{i\theta} + \alpha_{i\theta-1}\,|\, i \in \I_{j+1,\theta-1} \}. 
\end{aligned}
\end{align}
Thus \eqref{eq:roots-Dsuper} is a subset of the set \eqref{eq:roots-Cn} of positive roots of type $C_{\theta}$: We fix the convex order in $\varDelta_+^{\bq}$ obtained from \eqref{eq:roots-Cn-convex-order}.
For more information, see \cite[\S 5.3]{AA17}.
We prove Condition \ref{assumption:intro-combinatorial} for type
$\superd{j}{\theta - j}$:

\begin{prop}\label{prop:roots-cocycles-Dsuper}
For every  $\gamma \in \varDelta_+^{\bq}$, there exists $L_{\gamma}\in\N$ such that $(\x_{\gamma}^{L_\gamma})^*$ is a cocycle.
\end{prop}

\pf 
Let $\gamma$ be a positive non-simple root.
Arguing recursively we may assume that $\gamma$ has full support. Hence, either $\gamma=\alpha_{1\theta}$ or else there exists $k\in\I_{\theta-1}$ such that 
$\gamma=\alpha_{1\theta} +\alpha_{k \, \theta-1}$.

\medbreak

First we consider $\gamma=\alpha_{1\theta}$. Again, one can easily check that
$N_{\gamma}=2$, $P_{\gamma} = 3$ and $Q_{\gamma} = 1$. 
Let $\alpha<\beta$ be a pair of positive roots as in \eqref{eq:alfa-beta-N=2}. Then there exists $i\in\I_{\theta-1}$ such that $\alpha=\alpha_{1i}$, $\beta=\alpha_{i+1\, \theta}$. Now,
\begin{align*}
x_{\alpha}x_{\gamma} &= q_{\alpha\gamma} x_{\gamma}x_{\alpha}, &
x_{\gamma}x_{\beta} &= q_{\gamma\beta} x_{\beta}x_{\gamma}, &
x_{\alpha}x_{\beta} &= q_{\alpha\beta} x_{\beta}x_{\alpha}+ \Bsj x_{\gamma} &&\text{ for some }
\Bsj \in \Bbbk,
\end{align*}
hence the root vectors satisfy \eqref{eq:diff2-hypothesis}, and $-\frac{q_{\alpha\alpha}}{q_{\beta\beta}}\in \{ q^{-1}, q^{-2}\}$, so we take $L=N$. 

There exist $4$-tuples $(\alpha, \beta,\delta,\eta) \in \varDelta_+^4$ as in \eqref{eq:alfa-beta-delta-eta-N=2}:
\begin{align*}
\alpha=\alpha_{1i} & < \eta=\alpha_{1\, \theta-1} <\gamma < \beta=\alpha_{1\theta}+\alpha_{i+1\, \theta-1} <\delta=\alpha_{\theta}.
\end{align*}
The corresponding PBW generators satisfy \eqref{eq:diff-case2-hypothesis}; indeed, there exists $\Bsj\in\Bbbk$ such that
$[x_{\alpha},x_{\beta}]_c = q_{\alpha \beta} x_{\beta}x_{\alpha} +\Bsj x_{\gamma}x_{\eta}$, $x_{\eta}x_{\delta} = q_{\eta\delta}  x_{\delta}x_{\eta} + x_{\gamma}$,
and the other pairs of root vectors $q$-commute. Now $\coef{\alpha\beta\gamma}{N}=0$ by Lemma \ref{lemma:coef-roots-unity} \ref{item:coef-roots-unity-ii} since $\widetilde{q}_{\alpha\gamma} =\widetilde{q}_{\beta\gamma}=q^{-1}$.

Next we look for solutions of \eqref{eq:equation-roots-cocycle-N=2}. 
We claim that $f_{\eta}(n_{\eta})=n_{\eta}$ for all $\eta\in\varDelta_+$. 

\begin{itemize}[leftmargin=*,label=$\circ$]
\item If $N_{\eta}=2$, then this holds by definition of $f_{\eta}$.
\item If $\eta=\alpha_{i\theta}+\alpha_{k \, \theta-1}$, $i<k\le \theta-1$, then $N_{\eta}=N$, so $f_{\eta}(k)\ge N$ if $k\ge 2$. As
\begin{align*}
2 f_{\eta}(n_{\eta}) \le \sum_{\delta\in\varDelta_+^{\bq}} f_{\eta}(n_{\eta}) a_{\theta-1}^{\delta} =N,
\end{align*}
we have that $n_{\eta}\le 1$, so $f_{\eta}(n_{\eta})=n_{\eta}$.
\item Let $\eta=\alpha_{ik}$, with $i\le k<j$. Then $N_{\eta}=N$. Suppose that $n_{\eta}\ge 2$: as $f_{\eta}(s)>N$ if $s>2$, we may have $n_{\eta}=2$: moreover, $n_{\delta}=0$ for all $\delta\neq \alpha_{ik}$ such that $\delta\cap \I_{i,k}\neq\emptyset$ since
\begin{align*}
N &= \sum_{\delta\in\varDelta_+^{\bq}} f_{\delta}(n_{\delta}) a_{t}^{\delta} = N + \sum_{\delta\neq \alpha_{ik}: t\in\supp \delta} f_{\delta}(n_{\delta}) a_{t}^{\delta} & &\text{for all }t\in\I_{i,k}.
\end{align*}
Now if $k+1<j$, then $n_{\delta}=0$ for all $\delta\neq\alpha_{k+1 \, t}$, $t\ge k+1$: as $\sum_{\delta\in\varDelta_+^{\bq}} f_{\delta}(n_{\delta}) a_{k+1}^{\delta}=N$, we have $n_{\alpha_{k+1\, t}}=2$ for some $k+1\le t<j$. Thus we may assume $k=j-1$. 
Let $\delta\in\varDelta_+^{\bq}$ be such that $j\in\supp\delta$. Then $n_{\delta}=0$ if $j-1\in\supp\delta$, and $N_{\delta}=2$ if $j-1\notin\supp \delta$, so
\begin{align*}
N &= \sum_{\delta\in\varDelta_+^{\bq}: j\in\supp\delta} f_{\delta}(n_{\delta}) a_{j}^{\delta} 
= \sum_{\delta: j\in\supp\delta, \,j-1\notin\supp \delta} f_{\alpha_{jt}}(n_{\alpha_{jt}})
= \sum_{\delta: j\in\supp\delta, \,j-1\notin\supp \delta} n_{\alpha_{jt}}.
\end{align*}
This implies that $\sum_{\delta\in\varDelta_+^{\bq}} n_{\delta}\ge N+2$, a contradiction. Then $n_\eta\le 1$, so $f_{\eta}(n_{\eta})=n_{\eta}$.

\item Let $\eta=\alpha_{ik}$, $j<i\le k$. Then $N_{\eta}=N$ and an argument as in the previous case shows that we have that $n_{\eta}\le 1$, so $f_{\eta}(n_{\eta})=n_{\eta}$.

\item Similar situation holds for $\eta=\alpha_{i\theta} +\alpha_{i \, \theta-1}$: $N_{\eta}=M$ but again $n_{\eta}\le 1$, so $f_{\eta}(n_{\eta})=n_{\eta}$.
\end{itemize}

As the claim holds, we may rewrite the problem as follows: find $\gamma_i\in\varDelta_+^{\bq}$, $i\in\I_{N+1}$, such that $\sum \gamma_i=N\gamma$. As $a_{1}^{\delta}=1$ if $1\in\supp\delta$, $a_{\theta}^{\delta}=1$ if $\theta\in\supp\delta$, 
there exist $\theta-1$ roots such that $1,\theta\in\supp\gamma_i$: we may fix that $1,\theta\in\supp\gamma_i$ for $i\ge 3$. As $N=\sum_{i=3}^{N+1} a_{\theta-1}^{\gamma_i}$ and $a_{\theta-1}^{\gamma_i}\ge 1$, there exists at most one $i\ge 3$ such that $a_{\theta-1}^{\gamma_i}=2$:

\begin{itemize}[leftmargin=*]
\item if $a_{\theta-1}^{\gamma_i}=1$ for all $i\ge 3$, then $\gamma_i=\gamma$ for all $i\ge 3$ and $\gamma_1+\gamma_2=\gamma$.
\item if $a_{\theta-1}^{\gamma_3}=2$, then $\gamma_i=\gamma$ for all $i\ge 4$ and $\gamma_3=\alpha_{1\theta}+\alpha_{k \, \theta-1}$ for some $k\in\I_{2, \theta-1}$. Hence $\gamma_1+\gamma_2=\alpha_{1\, k-1} + \alpha_{\theta}$, so $\gamma_1$, $\gamma_2$ are $\alpha_{1\, k-1}$, $\alpha_{\theta}$.
\end{itemize}

Hence all the hypotheses of Proposition \ref{prop:cocycle-xgamma-N=2} hold, and $(\x_{\gamma}^{L_{\gamma}})^*$ is a cocycle.

\medbreak
Now we consider $\gamma=\alpha_{1\theta}+\alpha_{i\, \theta-1}$, $i\in\I_{2, \theta-1}$. In this case, $P_{\gamma} = 2$ and $Q_{\gamma} = 1$. 
Let $i\le j$. Then $N_{\gamma}=N>2$, so Lemma \ref{lem:largeN} applies and $(\x_{\gamma}^{N_\gamma})^*$ is a $2$-cocycle.

Next we assume that $i>j$, so $N_{\gamma}=2$. The pairs $(\alpha,\beta)$ as in \eqref{eq:alfa-beta-N=2} are the following:
\begin{itemize}[leftmargin=*,label=$\circ$]
\item $\alpha=\alpha_{1\theta}$, $\beta=\alpha_{i \, \theta-1}$. As in Lemma \ref{lem:Btheta-relations-root-vectors},
\begin{align*}
x_{\alpha}x_{\gamma} &= q_{\alpha\gamma} x_{\gamma}x_{\alpha}, &
x_{\gamma}x_{\beta} &= q_{\gamma\beta} x_{\beta}x_{\gamma}, &
x_{\alpha}x_{\beta} &= q_{\alpha\beta} x_{\beta}x_{\alpha}+ \Bsj x_{\gamma} &&\text{ for some }\Bsj\in\Bbbk.
\end{align*}

\item $\alpha=\alpha_{1\theta}+\alpha_{k \, \theta}$, $\beta=\alpha_{i \, k-1}$, $k\in\I_{i+1,\theta-1}$. Then
\begin{align*}
x_{\alpha}x_{\gamma} &= q_{\alpha\gamma} x_{\gamma}x_{\alpha}, &
x_{\gamma}x_{\beta} &= q_{\gamma\beta} x_{\beta}x_{\gamma}, &
x_{\alpha}x_{\beta} &= q_{\alpha\beta} x_{\beta}x_{\alpha}+ \Bsj x_{\gamma} + \sum_{t=k+1}^{\theta-1} \Bsj_t 
x_{\alpha_{i \, t-1}}x_{\alpha_{1\theta}+\alpha_{t \, \theta}},
\end{align*}
for some  $\Bsj,\Bsj_{t}\in \Bbbk$.

\item $\alpha=\alpha_{1k-1}$, $\beta=\alpha_{k\, \theta}+\alpha_{i \, \theta-1}$, $k\in\I_{j+1,i-1}$. Then
\begin{align*}
x_{\alpha}x_{\gamma} &= q_{\alpha\gamma} x_{\gamma}x_{\alpha}, &
x_{\gamma}x_{\beta} &= q_{\gamma\beta} x_{\beta}x_{\gamma}, &
x_{\alpha}x_{\beta} &= q_{\alpha\beta} x_{\beta}x_{\alpha}+ \Bsj x_{\gamma} + \sum_{t=k+1}^{i-1} \Bsj_t 
x_{\alpha_{t\theta}+\alpha_{i \, \theta}}x_{\alpha_{1 \, t-1}}
\end{align*}
for some  $\Bsj,\Bsj_{t}\in \Bbbk$.
\item $\alpha=\alpha_{1 \,i-1}$, $\beta=\alpha_{i\theta}+\alpha_{i \, \theta-1}$. In this case,
\begin{align*}
x_{\alpha}x_{\gamma} &= q_{\alpha\gamma} x_{\gamma}x_{\alpha}, &
x_{\gamma}x_{\beta} &= q_{\gamma\beta} x_{\beta}x_{\gamma}, &
x_{\alpha}x_{\beta} &= q_{\alpha\beta} x_{\beta}x_{\alpha}+ \Bsj x_{\gamma} &&\text{ for some }\Bsj\in\Bbbk.
\end{align*}
\end{itemize}
Hence all the pairs of root vectors satisfy \eqref{eq:diff2-hypothesis} and $-\frac{q_{\alpha\alpha}}{q_{\beta\beta}}\in\{q^{-1},q^{-2}\}$, so we take $L=N$. 
Now we look for solutions of \eqref{eq:equation-roots-cocycle-N=2}; i.e.~$\sum_{\delta\in\varDelta_+} f_\delta(n_\delta) \delta =N\gamma$, $\sum_{\delta\in\varDelta_+} n_\delta = N+1$. Let $\eta\in\varDelta_+^{\bq}$ such that $1\in\supp\eta$, $N_{\eta}>2$ and $n_{\eta}\neq 0$. Suppose that $n_{\eta}\ge 2$. Arguing as for the case $\gamma=\alpha_{1\theta}$, there exists $t\in\I_{j-1}$ such that $\eta=\alpha_{1t}$ and $n_{\eta}=2$. This implies that $n_{\delta}=0$ for all $\delta\in\varDelta_+^{\bq}$ such that $\supp \delta \cap \I_t\neq\emptyset$. Recursively, there exist $t_0=0< t_1=t<t_2<\dots <t_s=j-1$ such that $n_{\eta}=2$ if $\eta=\alpha_{i_{r-1}+1 \, i_r}$, $r\in\I_s$, and $n_{\delta}=0$ for all $\delta\in\varDelta_+^{\bq}$ such that $\supp \delta \cap \I_{j-1}\neq\emptyset$. Now, if $j\in\supp\delta$, then either $j-1\supp\delta$ (so $n_{\delta}=0$ by the previous argument) or $a_{j}^{\delta}=1$, $N_{\delta}=2$, so $f_{\delta}(n_{\delta})=n_{\delta}$. Thus,
\begin{align*}
N &= \sum_{\delta \in\varDelta_+^{\bq}: j\in\supp\delta} f_{\delta}(n_{\delta}) a_j^{\delta} = \sum_{\delta \in\varDelta_+^{\bq}: j\in\supp\delta} n_{\delta}.
\end{align*}
But then
\begin{align*}
N+1=\sum_{\delta \in\varDelta_+^{\bq}} n_{\delta} \ge n_{\eta} + \sum_{\delta \in\varDelta_+^{\bq}: j\in\supp\delta} n_{\delta}=N+2,
\end{align*}
a contradiction. Thus we have that $f_{\eta}(n_{\eta})=n_{\eta}$ for all $\eta\in\varDelta_+^{\bq}$ such that $1\in\supp \eta$ since either $N_{\eta}=2$ or else $n_{\eta}\ge 1$. From here, 
\begin{align*}
N=\sum_{\delta \in\varDelta_+^{\bq}: 1\in\supp\delta} f_{\delta}(n_{\delta}) a_j^{\delta} = \sum_{\delta \in\varDelta_+^{\bq}: 1\in\supp\delta} n_{\delta},
\end{align*}
As $\sum_{\delta \in\varDelta_+^{\bq}} n_{\delta}=N+1$, there exists a unique $\eta\in\varDelta_+$ such that $n_{\eta}\neq 0$ and $1\notin \supp\eta$; moreover, $n_{\eta}=1$. Again we may rewrite the problem as follows: find $\gamma_k\in\varDelta_+^{\bq}$, $k\in\I_{N+1}$, such that $\sum_k \gamma_k=N\gamma$. As $a_{1}^{\delta}=1$ if $1\in\supp\delta$, $a_{\theta}^{\delta}=1$ if $\theta\in\supp\delta$, 
there exist $\theta-1$ roots such that $1,\theta\in\supp\gamma_k$: we may fix that $1,\theta\in\supp\gamma_k$ for $k\ge 3$. Also, $a_{i}^{\gamma_k}\le 2$ and $\sum_t a_{i}^{\gamma_k}=2N$, so either $a_{i}^{\gamma_k}=2$ for exactly $N$ of them and 0 for the remaining one $a_{i}^{\gamma_k}$, or else $a_{i}^{\gamma_k}=2$ for exactly $N-1$ of them and $1$ for the remaining two $a_{i}^{\gamma_k}$'s. A detailed study case-by-case shows that $\gamma_i=\gamma$ for $i\ge 3$, and $\gamma_1+\gamma_2=\gamma$. Hence all the hypotheses of Proposition \ref{prop:cocycle-xgamma-N=2} hold, and $(\x_{\gamma}^{N})^*$ is a cocycle.
\epf

\section{Exceptional types}\label{sec:exceptional}

\subsection{Type \texorpdfstring{$E_\theta$}{}}\label{sec:E-theta}

Let $q\in\G_N'$, $n\ge 2$. 
Let $\toba_{\bq}$ be a Nichols algebra of type $E_{\theta}$, $6\le\theta\le 8$. That is, the generalized Dynkin diagram of $\toba_{\bq}$ has the form
\begin{align*}
\xymatrix{ & & \overset{q}{\underset{2}{\circ}} & & \\
\overset{q}{\underset{1}{\circ}}\ar@{-}[r]^{q^{-1}} &  \overset{q}{\underset{3}{\circ}}\ar@{-}[r] &
\overset{q}{\underset{4}{\circ}} \ar@{-}[u]_{q^{-1}}^{\quad} \ar@{-}[r]^{q^{-1}} &
\overset{q}{\underset{5}{\circ}} \ar@{.}[r]^{q^{-1}} & \overset{q}{\underset{\theta}{\circ}} }
\end{align*}
Here $\varDelta^{\bq}=\varDelta$ is a root system of type $E_{\theta}$.
We fix the following convex orders on the sets of positive roots:
\begin{align*}
\mathbf{E_6:} & 1, 2, 13, 3, 1234, 134, 234, 24, 34, 4, 123^24^25, 1234^25, 234^25, 12345, 1345, 2345, 345, 245, 
\\ & 45, 5, 12^23^24^35^26, 123^24^35^26, 123^24^25^26, 123^24^256, 1234^25^26, 23^24^25^26, 1234^256, 234^256,
\\ & 123456, 23456, 2456, 13456, 3456, 456, 56, 6;
\\
\mathbf{E_7:} &\text{ roots of support contained in } \I_6\text{ ordered as for }E_6\text{ followed by }\\
& 1^22^23^34^45^36^27, 12^23^34^45^36^27, 12^23^24^45^36^27, 12^23^24^35^36^27, 123^24^35^36^27, 12^23^24^35^26^27,
\\ & 12^23^24^35^267, 123^24^35^26^27, 123^24^35^267, 123^24^25^26^27, 123^24^25^267, 123^24^2567,
\\ & 1234^25^26^27, 1234^25^267, 1234^2567, 1234567, 134567, 234^25^26^27, 234^25^267, 234^2567, 
\\ & 234567, 24567, 34567, 4567, 567, 67, 7;
\\
\mathbf{E_8:} &\text{ roots of support contained in } \I_7\text{ ordered as for }E_7\text{ followed by }\\
& 1^22^33^44^65^56^47^38, 1^22^33^44^65^56^47^28, 1^22^33^44^65^56^37^28, 1^22^33^44^65^46^37^28, 1^22^33^44^55^46^37^28, 
\\ & 1^22^23^44^55^46^37^28, 1^22^33^34^55^46^37^28, 12^33^34^55^46^37^28, 1^22^23^34^55^46^37^28, 
12^23^34^55^46^37^28, 
\\ & 1^22^23^34^45^46^37^28, 12^23^34^45^46^37^28, 12^23^24^45^46^37^28, 
1^22^23^34^45^36^37^28, 12^23^34^45^36^37^28,
\\ & 12^23^24^45^36^37^28, 12^23^24^35^36^37^28, 
123^24^35^36^37^28, 1^22^23^34^45^36^27^28, 12^23^34^45^36^27^28, 
\\ & 12^23^24^45^36^27^28, 
12^23^24^35^36^27^28, 12^23^24^35^26^27^28, 123^24^35^36^27^28, 123^24^35^26^27^28, 
\\ & 123^24^25^26^27^28, 
1234^25^26^27^28, 234^25^26^27^28, 1^22^33^44^65^56^47^38^2, 
1^22^23^34^45^36^278, 
\\ & 12^23^34^45^36^278, 12^23^24^45^36^278, 12^23^24^35^36^278, 
123^24^35^36^278, 12^23^24^35^26^278, 
\\ & 12^23^24^35^2678, 123^24^35^26^278, 
123^24^35^2678, 123^24^25^26^278, 123^24^25^2678, 123^24^25678, 
\\ & 1234^25^26^278, 1234^25^2678, 1234^25678, 12345678, 
1345678, 234^25^26^278, 234^25^2678, 
\\ & 234^25678, 2345678, 245678, 345678, 45678, 5678, 678, 78, 8.
\end{align*}

For more information, see \cite[\S 4.5]{AA17}.
The aim of this Section is to prove that Condition \ref{assumption:intro-combinatorial} holds for type
$E_{\theta}$. We need first the following result.

\begin{lemma}\label{lem:Etheta-sum-roots}
Let $\beta<\delta\in \varDelta^{\bq}_{+}$ be such that $\gamma = \beta + \delta\in \varDelta^{\bq}_{+}$.

\begin{enumerate}[leftmargin=*,label=\rm{(\alph*)}]
\item\label{item:Etheta-sum-roots-beta-gamma} If $\mu_1\le \dots\le \mu_{k}\in\varDelta_+$ satisfy $\sum_i \mu_i=\beta+\gamma$, then either
$\mu_1\le \beta$ or else $\mu_k\ge \gamma$.

\item\label{item:Etheta-sum-roots-delta-gamma} If $\mu_1\le \dots\le \mu_{k}\in\varDelta_+$ satisfy $\sum_i \mu_i=\delta+\gamma$, then either
$\mu_1\le \gamma$ or else $\mu_k\ge \delta$.
\end{enumerate}
\end{lemma}

\pf
Let $(\cdot,\cdot)$ be the symmetric positive definite form on $\mathbb{R}^{\theta}$ such that $(\nu,\nu)=2$ for all $\nu\in\varDelta$. Then 
$-1\le (\mu,\mu')\le 1$ if $\nu \neq \nu' \in\varDelta$.
As $\beta + \delta\in \varDelta$, $(\beta,\delta)=-1$; thus $(\beta,\gamma)=1$.

Next we prove \ref{item:Etheta-sum-roots-beta-gamma}: the proof of \ref{item:Etheta-sum-roots-delta-gamma} is analogous. 
Let $\mu_1\le \dots\le \mu_{k}$ be such that $\sum_i \mu_i=\beta+\gamma$.
Note that $k\ge 2$, since $\beta+\gamma=2\beta+\delta\notin\varDelta$. Suppose on the contrary that $\beta<\mu_1\le \dots\le \mu_{k}<\gamma$. Then $k\ge 3$, since $(\mu_i,\beta)\le 1$ and
\begin{align*}
3 &=(\beta+\gamma,\beta)= \sum_i (\mu_i,\beta).
\end{align*}
Assume that $k\ge 4$. Then there exist $j\ne \ell$ such that $(\mu_j,\mu_\ell)=-1$ since
\begin{align*}
6 &=(\beta+\gamma,\beta+\gamma)= \sum_i (\mu_i,\mu_i)+\sum_{i\ne j} (\mu_i,\mu_j)\le 8+\sum_{i\ne j} (\mu_i,\mu_j).
\end{align*}
Thus $\mu_j+\mu_\ell\in\varDelta_{+}$, $\mu_j<\mu_j+\mu_\ell<\mu_\ell$ and we can replace the set $\{\mu_i\}_{i\in\I_k}$ by
\begin{align*}
\big(\{\mu_i\}_{i\in\I_k}-\{\mu_j,\mu_{\ell} \} \big) \cup \{\mu_j+\mu_\ell\}.
\end{align*}
Hence, recursively, we may assume that $k=3$. But using the computer we check that $\nu_1+\nu_2+\nu_3\ne \beta+\gamma$ for all the $3$-uples 
$\beta<\nu_1\le\nu_2\le\nu_3<\gamma$ so we get a contradiction.
\epf

\begin{prop}\label{prop:roots-cocycles-Etheta}
For every $\gamma  \in \varDelta_+^{\bq}$,  $(\x_{\gamma}^{N_\gamma})^*$ is a $2$-cocycle.
\end{prop}

\pf By Lemma \ref{lema:non-simple-roots-Palfa}, $P_{\gamma} = 2$ and $Q_\gamma = 1$ for all non-simple roots $\gamma$. Hence, if $N>2$, then Lemma \ref{lem:largeN} applies and $(\x_{\gamma}^{N_\gamma})^*$ is a $2$-cocycle for all roots $\gamma$.

\medbreak

Next we assume $N=2$, that is, $q=-1$. We will apply Lemma \ref{lem:second-technique-2cocycles}. 
Let $\gamma$ be a positive non-simple root. For each pair $\beta<\delta\in \varDelta^{\bq}_{+}$ such that $\gamma = \beta + \delta$, we have that
\begin{align*}
x_{\beta}x_{\gamma}&=q_{\beta\gamma} x_{\gamma}x_{\beta}, &
x_{\gamma}x_{\delta}&=q_{\gamma\delta} x_{\delta}x_{\gamma}
\end{align*}
by \eqref{eq:convex-combination} and Lemma \ref{lem:Etheta-sum-roots}, and $q_{\beta\beta}=q_{\delta\delta}=-1$. Hence \eqref{eq:second-technique-2cocycles} holds.

Let $\gamma_1,\gamma_2,\gamma_3\in \varDelta_{+}^{\bq}$ be three different roots such that $\gamma_1+\gamma_2+\gamma_3=2\gamma$. 
By Lemma \ref{lema:3roots} there exists $w\in\Wc$ such that the support of $w(\gamma_i)$ is of size $\le 3$, and a fortiori $\gamma$ too. 
As $w(\gamma_1)+w(\gamma_2)+w(\gamma_3)=2w(\gamma)$ and these roots are contained in a subdiagram of type $A_3$ or $A_2\times A_1$, we conclude that
$\gamma_i=\gamma$ for some $i\in\I_3$. Using a similar argument we also check that $2\gamma_1+\gamma_2 \neq 2\gamma$ for all $\gamma_1 \neq  \gamma_2\in \varDelta_{+}^{\bq}$. Hence all the hypothesis of Lemma \ref{lem:second-technique-2cocycles} hold, and $(\x_{\gamma}^2)^*$ is a $2$-cocycle.
\epf

\subsection{Type \texorpdfstring{$F_4$}{}}\label{sec:F4}
Let $q$ be a root of 1 of order $N > 2$, $M=\ord q^2$.
In this section, we deal with a Nichols algebra $\toba_{\bq}$ of Cartan type $F_{4}$, that is associated to
the Dynkin diagram
\begin{align*}
\xymatrix{ \overset{\,\,q}{\underset{\ }{\circ}}\ar  @{-}[r]^{q^{-1}}  &
\overset{\,\,q}{\underset{\ }{\circ}}\ar  @{-}[r]^{q^{-2}} &   \overset{q^2}{\underset{\
}{\circ}} \ar  @{-}[r]^{q^{-2}}  &  \overset{q^2}{\underset{\ }{\circ}} }
\end{align*}
For more information, see \cite[\S 4.6]{AA17}.
The set of positive roots with full support is 
\begin{align*}
\{1^22^43^34,\; 1^22^43^24,\; 1^22^33^24, \; 1^22^23^24, \; 1^22^234, \; 12^33^24,\; 12^23^24,\; 1^22^43^34^2, \;
12^234, \; 1234\}.
\end{align*}
The aim of this Section is to prove that Condition \ref{assumption:intro-combinatorial} holds for type
$F_{4}$. More precisely,

\begin{prop}\label{prop:roots-cocycles-F4}
For every $\gamma  \in \varDelta_+^{\bq}$,  $(\x_{\gamma}^{N_\gamma})^*$ is a $2$-cocycle.
\end{prop}

\pf By induction on the rank it is enough to consider $\gamma$ with full support. We have:
\begin{align*}
&\text{ if } \gamma \in
\{1^22^33^24, \; 12^33^24,\; 12^23^24,\; 12^234, \; 1234\}, \text{ then} & N_{\gamma}&=N, \, P_{\gamma}=3, \, Q_{\gamma}=2;
\\
&\text{ if } \gamma \in
\{1^22^43^34,\; 1^22^43^24,\; 1^22^23^24, \; 1^22^234, \; 1^22^43^34^2\}, \text{ then} & N_{\gamma}&=M, \, P_{\gamma}=2, \, Q_{\gamma}=2.
\end{align*}
Hence, if $N>4$, then $N_{\gamma}>P_{\gamma}, Q_{\gamma}$  for all $\gamma$  with full support. 
Thus $(\x_{\gamma}^{N_{\gamma}})^*$ is a $2$-cocycle for all $\gamma$  with full support by Lemma \ref{lem:largeN}.

\bigskip
Next we assume $N=4$. If $\gamma\in\{1^22^33^24, \; 12^33^24,\; 12^23^24,\; 12^234, \; 1234\}$, then $N_{\gamma}=4>P_{\gamma}, Q_{\gamma}$, so $(\x_{\gamma}^{N_{\gamma}})^*$ is a $2$-cocycle by Lemma \ref{lem:largeN}.

Now we consider $\gamma \in
\{1^22^43^34,\; 1^22^43^24,\; 1^22^23^24, \; 1^22^234, \; 1^22^43^34^2\}$. 
Let $\alpha<\beta$ be a pair of positive roots as in \eqref{eq:alfa-beta-N=2}. We have the following possibilities:
\begin{itemize}[leftmargin=*,label=$\circ$]
\item $\gamma=1^22^43^34$, $\alpha=3$, $\beta=1^22^43^24$. There exists $\Bsj\in\Bbbk$ such that $[x_{\alpha},x_{\beta}]_c = \Bsj x_{\gamma}$.
\medbreak

\item $\gamma=1^22^43^34$, $\alpha=23$, $\beta=1^22^33^24$. There exist $\Bsj, \Bsj_1\in\Bbbk$ such that
\begin{align*}
[x_{\alpha},x_{\beta}]_c &= \Bsj x_{\gamma} + \Bsj_1 x_{1^22^43^24} x_{3}.
\end{align*}

\medbreak

\item $\gamma=1^22^43^34$, $\alpha=2^23$, $\beta=1^22^23^24$. There exist $\Bsj, \Bsj_t\in\Bbbk$ such that
\begin{align*}
[x_{\alpha},x_{\beta}]_c &= \Bsj x_{\gamma} + \Bsj_1 x_{1^22^43^24} x_{3}
+ \Bsj_2 x_{1^22^33^24} x_{23}.
\end{align*}
\medbreak

\item $\gamma=1^22^43^34$, $\alpha=123$, $\beta=12^33^24$. There exist $\Bsj, \Bsj_t\in\Bbbk$ such that
\begin{align*}
[x_{\alpha},x_{\beta}]_c &= \Bsj x_{\gamma} + \Bsj_1 x_{1^22^43^24} x_{3}
+ \Bsj_2 x_{1^22^33^24} x_{23}+\Bsj_3 x_{1^22^23^24}x_{2^23}.
\end{align*}

\medbreak

\item $\gamma=1^22^43^34$, $\alpha=12^23$, $\beta=12^23^24$. There exist $\Bsj, \Bsj_t\in\Bbbk$ such that
\begin{align*}
[x_{\alpha},x_{\beta}]_c &= \Bsj x_{\gamma} + \Bsj_1 x_{1^22^43^24} x_{3}
+ \Bsj_2 x_{1^22^33^24} x_{23}+\Bsj_3 x_{1^22^23^24}x_{2^23} +\Bsj_4 x_{12^33^24}x_{123}.
\end{align*}

\medbreak

\item $\gamma=1^22^43^34$, $\alpha=1^22^23$, $\beta=2^23^24$. There exist $\Bsj, \Bsj_t\in\Bbbk$ such that
\begin{align*}
[x_{\alpha},x_{\beta}]_c &= \Bsj x_{\gamma} + \Bsj_1 x_{1^22^43^24} x_{3}
+ \Bsj_2 x_{1^22^33^24} x_{23}+\Bsj_3 x_{1^22^23^24}x_{2^23} 
\\ & \quad +\Bsj_4 x_{12^33^24}x_{123} +\Bsj_5 x_{12^23^24}x_{123}.
\end{align*}

\medbreak

\item $\gamma=1^22^43^24$, $\alpha=2^23$, $\beta=1^22^234$. There exist $\Bsj\in\Bbbk$ such that $[x_{\alpha},x_{\beta}]_c = \Bsj x_{\gamma}$.

\medbreak

\item $\gamma=1^22^43^24$, $\alpha=12^23$, $\beta=12^234$. There exist $\Bsj, \Bsj_1\in\Bbbk$ such that
\begin{align*}
[x_{\alpha},x_{\beta}]_c &= \Bsj x_{\gamma} + \Bsj_1 x_{1^22^234} x_{2^23}.
\end{align*}

\medbreak

\item $\gamma=1^22^43^24$, $\alpha=1^22^23$, $\beta=2^234$. There exist $\Bsj, \Bsj_t\in\Bbbk$ such that
\begin{align*}
[x_{\alpha},x_{\beta}]_c &= \Bsj x_{\gamma} + \Bsj_1 x_{1^22^234} x_{2^23}
+ \Bsj_2 x_{12^234} x_{12^23}.
\end{align*}
\medbreak

\item $\gamma=1^22^43^24$, $\alpha=2$, $\beta=1^22^33^24$. There exists $\Bsj\in\Bbbk$ such that $[x_{\alpha},x_{\beta}]_c= \Bsj x_{\gamma}$.

\medbreak

\item $\gamma=1^22^43^24$, $\alpha=12$, $\beta=12^33^24$. There exist $\Bsj, \Bsj_t\in\Bbbk$ such that
\begin{align*}
[x_{\alpha},x_{\beta}]_c &= \Bsj x_{\gamma} + \Bsj_1 x_{1^22^234} x_{2^23}
+ \Bsj_2 x_{1^22^33^24}x_{2}.
\end{align*}

\medbreak

\item $\gamma=1^22^23^24$, $\alpha=3$, $\beta=1^22^234$. There exists $\Bsj\in\Bbbk$ such that $[x_{\alpha},x_{\beta}]_c = \Bsj x_{\gamma}$.

\medbreak

\item $\gamma=1^22^23^24$, $\alpha=123$, $\beta=1234$. There exist $\Bsj, \Bsj_1\in\Bbbk$ such that
\begin{align*}
[x_{\alpha},x_{\beta}]_c &= \Bsj x_{\gamma} + \Bsj_1 x_{1^22^234} x_{3}.
\end{align*}

\medbreak

\item $\gamma=1^22^23^24$, $\alpha=1^22^23$, $\beta=34$. There exist $\Bsj, \Bsj_t\in\Bbbk$ such that
\begin{align*}
[x_{\alpha},x_{\beta}]_c &= \Bsj x_{\gamma} + \Bsj_1 x_{1^22^234} x_{3}
+ \Bsj_2 x_{1234} x_{123}.
\end{align*}

\medbreak

\item $\gamma=1^22^23^24$, $\alpha=1$, $\beta=12^23^24$. There exist $\Bsj, \Bsj_1\in\Bbbk$ such that
\begin{align*}
[x_{\alpha},x_{\beta}]_c &= \Bsj x_{\gamma} + \Bsj_1 x_{1^22^234} x_{3}.
\end{align*}

\medbreak

\item $\gamma=1^22^234$, $(\alpha,\beta)$ one of the pairs $(12,1234)$, $(1^22^23,4)$, $(1,12^234)$. There exists $\Bsj\in\Bbbk$ such that $[x_{\alpha},x_{\beta}]_c = \Bsj x_{\gamma}$.

\medbreak

\item $\gamma=1^22^43^34^2$, $\alpha=12^23^24$, $\beta=12^234$. There exists $\Bsj\in\Bbbk$ such that $[x_{\alpha},x_{\beta}]_c = \Bsj x_{\gamma}$.
\medbreak

\item $\gamma=1^22^43^34^2$, $\alpha=12^33^24$, $\beta=1234$. There exist $\Bsj, \Bsj_1\in\Bbbk$ such that
\begin{align*}
[x_{\alpha},x_{\beta}]_c &= \Bsj x_{\gamma} + \Bsj_1 x_{12^234} x_{12^23^24}.
\end{align*}

\medbreak

\item $\gamma=1^22^43^34^2$, $\alpha=1^22^234$, $\beta=2^23^24$. There exist $\Bsj, \Bsj_t\in\Bbbk$ such that
\begin{align*}
[x_{\alpha},x_{\beta}]_c &= \Bsj x_{\gamma} + \Bsj_1 x_{12^234} x_{12^23^24} + \Bsj_2 x_{1234} x_{12^33^24}.
\end{align*}
\medbreak

\item $\gamma=1^22^43^34^2$, $\alpha=1^22^23^24$, $\beta=2^234$. There exist $\Bsj, \Bsj_t\in\Bbbk$ such that
\begin{align*}
[x_{\alpha},x_{\beta}]_c &= \Bsj x_{\gamma} + \Bsj_1 x_{12^234} x_{12^23^24} + \Bsj_2 x_{1234} x_{12^33^24} +\Bsj_3 x_{2^23^24} x_{1^22^234}.
\end{align*}

\medbreak

\item $\gamma=1^22^43^34^2$, $\alpha=1^22^33^24$, $\beta=234$. There exist $\Bsj, \Bsj_t\in\Bbbk$ such that
\begin{align*}
[x_{\alpha},x_{\beta}]_c &= \Bsj x_{\gamma} + \Bsj_1 x_{12^234} x_{12^23^24} + \Bsj_2 x_{1234} x_{12^33^24} +\Bsj_3 x_{2^23^24} x_{1^22^234} +\Bsj_4 x_{234}x_{1^22^33^24}.
\end{align*}

\medbreak

\item $\gamma=1^22^43^34^2$, $\alpha=1^22^43^24$, $\beta=34$. There exist $\Bsj, \Bsj_t\in\Bbbk$ such that
\begin{align*}
[x_{\alpha},x_{\beta}]_c &= \Bsj x_{\gamma} + \Bsj_1 x_{12^234} x_{12^23^24} + \Bsj_2 x_{1234} x_{12^33^24} +\Bsj_3 x_{2^23^24} x_{1^22^234} 
\\ & \quad +\Bsj_4 x_{234}x_{1^22^33^24} +\Bsj_5 x_{234}x_{1^22^33^24}.
\end{align*}

\medbreak

\item $\gamma=1^22^43^34^2$, $\alpha=1^22^43^34$, $\beta=4$. There exist $\Bsj, \Bsj_t\in\Bbbk$ such that
\begin{align*}
[x_{\alpha},x_{\beta}]_c &= \Bsj x_{\gamma} + \Bsj_1 x_{12^234} x_{12^23^24} + \Bsj_2 x_{1234} x_{12^33^24} +\Bsj_3 x_{2^23^24} x_{1^22^234} 
\\ & \quad +\Bsj_4 x_{234}x_{1^22^33^24} +\Bsj_5 x_{234}x_{1^22^33^24} +\Bsj_6 x_{34}x_{1^22^43^24}.
\end{align*}

\end{itemize}

In all cases $[x_{\alpha},x_{\gamma}]_c =0 = [x_{\gamma},x_{\beta}]_c$: the proof of all these relations follow as in Lemma \ref{lem:Btheta-relations-root-vectors}. Hence the root vectors satisfy \eqref{eq:diff2-hypothesis}, and  $-\frac{q_{\alpha\alpha}}{q_{\beta\beta}}=-1$; then we take $L=2$. 

Next we look for solutions of \eqref{eq:equation-roots-cocycle-N=2}. That is, $\sum_{\delta\in\varDelta_+^{\bq}} f_\delta(n_\delta) \delta =2\gamma$, $\sum_{\delta\in\varDelta_+^{\bq}} n_\delta = 3$.
Notice that $n_{\delta}\ne 3$ for all $\delta\in\varDelta_+^{\bq}$: otherwise $n_{\eta}=0$ for all $\eta\neq \delta$, so $2\gamma=f_\delta(3) \delta$, a contradiction. Hence, either 
$2\gamma=f_{\gamma_1}(2)\gamma_1+\gamma_2$ or else $2\gamma=\gamma_1+\gamma_2+\gamma_3$ for some $\gamma_i\neq\gamma_j\in\varDelta_+^{\bq}$. By Lemma \ref{lema:3roots} there exists $w\in W$ such that
$\gamma_i'=w(\gamma_i)$ have support in a rank 3 subdiagram, so $\gamma'=w(\gamma)$  has the same support: this subdiagram is either of Cartan type $B_3$ or $C_3$. Looking at the corresponding cases (see the proofs of Propositions \ref{prop:roots-cocycles-Btheta} and \ref{prop:roots-cocycles-Ctheta}) the solutions for $\gamma'$ are $\gamma_3'=\gamma'$, $\gamma_1'+\gamma_2'=\gamma'$. Hence, all solutions for $\gamma$ are of form $n_{\gamma}=n_{\alpha}=n_{\beta}=1$ for a pair $(\alpha, \beta)$ satisfying \eqref{eq:alfa-beta-N=2}, and $n_{\varphi}= 0$ for the remaining $\varphi \in\varDelta_+$. Hence all the hypotheses of Proposition \ref{prop:cocycle-xgamma-N=2} hold, and $(\x_{\gamma}^{L_{\gamma}})^*$ is a $2$-cocycle.

\bigskip

Finally we consider $N=3$. If $\gamma \in\{1^22^43^34,\; 1^22^43^24,\; 1^22^23^24, \; 1^22^234, \; 1^22^43^34^2\}$, then $N_{\gamma}=3>P_{\gamma}, Q_{\gamma}$, so $(\x_{\gamma}^{N_{\gamma}})^*$ is a $2$-cocycle by Lemma \ref{lem:largeN}.

Now we consider $\gamma\in\{1^22^33^24, \; 12^33^24,\; 12^23^24,\; 12^234, \; 1234\}$. Let $\alpha<\beta$ be a pair of positive roots as in \eqref{eq:alfa-beta}. We have the following possibilities:
\begin{itemize}[leftmargin=*,label=$\circ$]
\item $\gamma=1^22^33^24$, $\alpha=1^22^43^24$, $\beta=1^22^23^24$. There exist $\Bsj\in\Bbbk$ such that $[x_{\alpha},x_{\beta}]_c = \Bsj x_{\gamma}^2$.

\medbreak

\item $\gamma=1^22^33^24$, $\alpha=1^22^43^34$, $\beta=1^22^234$. There exist $\Bsj, \Bsj'\in\Bbbk$ such that
\begin{align*}
[x_{\alpha}, x_{\beta}]_c &= \Bsj x_{\gamma}^2 + \Bsj' \; x_{1^22^23^24} x_{1^22^43^24}.
\end{align*}

\medbreak

\item $\gamma=12^33^24$, $\alpha=1^22^43^24$, $\beta=2^23^24$. There exist $\Bsj\in\Bbbk$ such that $[x_{\alpha},x_{\beta}]_c = \Bsj x_{\gamma}^2$.

\medbreak

\item $\gamma=12^33^24$, $\alpha=1^22^43^34$, $\beta=2^234$. There exist $\Bsj, \Bsj'\in\Bbbk$ such that
\begin{align*}
[x_{\alpha}, x_{\beta}]_c &= \Bsj x_{\gamma}^2 + \Bsj' \; x_{2^23^24} x_{1^22^43^24}.
\end{align*}

\medbreak

\item $\gamma=12^23^24$, $\alpha=1^22^23^24$, $\beta=2^23^24$. There exist $\Bsj\in\Bbbk$ such that $[x_{\alpha},x_{\beta}]_c = \Bsj x_{\gamma}^2$.

\medbreak

\item $\gamma=12^23^24$, $\alpha=1^22^43^34$, $\beta=34$. There exist $\Bsj, \Bsj'\in\Bbbk$ such that
\begin{align*}
[x_{\alpha}, x_{\beta}]_c &= \Bsj x_{\gamma}^2 + \Bsj' \; x_{2^23^24} x_{1^22^23^24}.
\end{align*}

\medbreak

\item $\gamma=12^234$, $\alpha=1^22^234$, $\beta=2^234$. There exist $\Bsj\in\Bbbk$ such that $[x_{\alpha},x_{\beta}]_c = \Bsj x_{\gamma}^2$.

\medbreak

\item $\gamma=12^234$, $\alpha=1^22^43^24$, $\beta=4$. There exist $\Bsj, \Bsj'\in\Bbbk$ such that
\begin{align*}
[x_{\alpha}, x_{\beta}]_c &= \Bsj x_{\gamma}^2 + \Bsj' \; x_{2^234} x_{1^22^234}.
\end{align*}

\medbreak

\item $\gamma=1234$, $\alpha=1^22^234$, $\beta=34$. There exist $\Bsj\in\Bbbk$ such that $[x_{\alpha},x_{\beta}]_c = \Bsj x_{\gamma}^2$.

\medbreak

\item $\gamma=1234$, $\alpha=1^22^23^24$, $\beta=4$. There exist $\Bsj, \Bsj'\in\Bbbk$ such that
\begin{align*}
[x_{\alpha}, x_{\beta}]_c &= \Bsj x_{\gamma}^2 + \Bsj' \; x_{34} x_{1^22^234}.
\end{align*}

\medbreak

\end{itemize}
In all cases $[x_{\alpha},x_{\gamma}]_c =0 = [x_{\gamma},x_{\beta}]_c$, and $\frac{q_{\alpha\gamma}}{q_{\gamma\beta}}=1$: the proof of all these relations follow as in Lemma \ref{lem:Btheta-relations-root-vectors}.
Therefore the root vectors satisfy \eqref{eq:diff2-hypothesis}, so we take $L=1$. 

Next we look for solutions of \eqref{eq:equation-roots-cocycle-N=2}. That is, $\sum_{\delta\in\varDelta_+^{\bq}} f_\delta(n_\delta) \delta =3\gamma$, $\sum_{\delta\in\varDelta_+^{\bq}} n_\delta = 3$.
Notice that $n_{\delta}\ne 3$ for all $\delta\in\varDelta_+^{\bq}$: otherwise $n_{\eta}=0$ for all $\eta\neq \delta$, so $3\gamma=f_\delta(3) \delta$, a contradiction. Hence, either 
$3\gamma=f_{\gamma_1}(2)\gamma_1+\gamma_2$ or else $3\gamma=\gamma_1+\gamma_2+\gamma_3$ for some $\gamma_i\neq\gamma_j\in\varDelta_+^{\bq}$. Again we use Lemma \ref{lema:3roots} to reduce to rank 3 subdiagrams and looking at the proofs of Propositions \ref{prop:roots-cocycles-Btheta} and \ref{prop:roots-cocycles-Ctheta} we conclude that all the solutions for $\gamma$ are of form $n_{\gamma}=n_{\alpha}=n_{\beta}=1$ for a pair $(\alpha, \beta)$ satisfying \eqref{eq:alfa-beta}, and $n_{\varphi}= 0$ for the remaining $\varphi \in\varDelta_+$. Hence all the hypotheses of Proposition \ref{prop:cocycle-xgamma-N>2} hold, and $(\x_{\gamma}^{N_{\gamma}})^*$ is a $2$-cocycle.
\epf

%
%
%
%

\subsection{Type \texorpdfstring{$G_2$ Cartan}{}}\label{sec:G2-cartan}

Let $q$ be a root of 1 of order $N > 3$. Set 
\begin{align*}
M&=\begin{cases}
N, & 3 \text{ does not divide }N;
\\
N/3, & 3 \text{ divides }N.
\end{cases}
\end{align*}
In this section, we deal with a Nichols algebra $\toba_{\bq}$ of Cartan type $G_{2}$, with Dynkin diagram
\begin{align*}
\xymatrix{  \overset{\,\,q}{\underset{\ }{\circ}} \ar  @{-}[r]^{q^{-3}} &
\overset{q^3}{\underset{\ }{\circ}}}.
\end{align*}
The set of positive roots is 
\begin{align}\label{eq:positive-roots-G2}
\varDelta_+ &= \{ 1, 1^32, 1^22,
1^32^2, 12, 2 \} .
\end{align}

We take as generators $x_1, x_2$, as well as
\begin{align}\label{eq:root-vectors-G2}
\begin{aligned}
x_{1112} & := (\ad_c x_1)^3 x_2 , &
x_{112} & := (\ad_c x_1)^2 x_2 , \\
x_{11212} & := [x_{112}, x_{12} ] _c , &
x_{12} & := (\ad_c x_1) x_2 .
\end{aligned}
\end{align}
We order these root vectors:
$x_1 < x_{1112} < x_{112} < x_{11212} < x_{12} < x_2$.

\medspace
The aim of this Section is to prove that Condition \ref{assumption:intro-combinatorial} holds for type
$G_{2}$. More precisely,

\begin{prop}\label{prop:roots-cocycles-G2-cartan}
For every $\gamma  \in \varDelta_+^{\bq}$,  $(\x_{\gamma}^{N_\gamma})^*$ is a $2$-cocycle.
\end{prop}

\pf
\noindent $\circ$ For $\gamma=1^32$, the case $N_{1112}>2$ follows by Lemma \ref{lem:largeN} again. Assume now that $N_{1112}=2$ (that is, $N=6$). We will apply Proposition \ref{prop:cocycle-xgamma-N=2}. The unique pair as in  \eqref{eq:alfa-beta-N=2} is $\alpha=\alpha_1$, $\beta=2\alpha_1+\alpha_2$, since the following relations hold:
\begin{align*}
x_{1}x_{112}&=x_{1112}+q^2q_{12} \, x_{112}x_1, &
x_{1}x_{1112}&=q^3q_{12} \, x_{1112}x_1, &
x_{1112}x_{112}&=q^3q_{12} \, x_{112}x_{1112}.
\end{align*}
In this case, $-\frac{q_{\alpha\alpha}}{q_{\beta\beta}}=-1$ so we take $L=2$.
The unique solution of \eqref{eq:equation-roots-cocycle-N=2} is $n_1=n_{1^32}=n_{1^22}=1$, and $n_{\delta}=0$ for the remaining roots. Hence Proposition \ref{prop:cocycle-xgamma-N=2} applies and $(\x_{1112}^2)^*$ is a 2-cocycle.

\medbreak
\noindent $\circ$ For $\gamma=1^22$, 
the case $N_{112}>4$ follows by Lemma \ref{lem:largeN}. Assume now that $N_{112}=4$. We will apply Proposition \ref{prop:cocycle-xgamma-N>2}. The unique pair as in \eqref{eq:alfa-beta} is $\alpha=3\alpha_1+\alpha_2$, $\beta=3\alpha_1+2\alpha_2$, since the following relations hold:
\begin{align*}
x_{1112}x_{11212}&= -(q+1)q_{12}^2 x_{112}^3 
-q_{12}^3 \, x_{11212}x_{1112}, &
x_{1112}x_{112}&=q^3q_{12} \, x_{112}x_{1112}, 
\\ & &
x_{112}x_{11212}&=q^3q_{12} \, x_{11212}x_{112}.
\end{align*}
Here $\frac{q_{\alpha\gamma}}{q_{\gamma\beta}}=1$, so we take $L=1$.
The unique solution of \eqref{eq:equation-roots-cocycle-N>2} is $n_{1^32}=n_{1^22}=n_{1^32^2}=1$, and $n_{\delta}=0$ for the remaining roots. Hence Proposition \ref{prop:cocycle-xgamma-N>2} applies: $(\x_{112}^4)^*$ is a 2-cocycle.

\medbreak
\noindent $\circ$ For $\gamma=1^32^2$, 
the case $N_{11212}>2$ follows by Lemma \ref{lem:largeN}. Assume now that $N_{11212}=2$ (that is, $N=6$). We will apply Proposition \ref{prop:cocycle-xgamma-N=2}. The pairs as in  \eqref{eq:alfa-beta-N=2} are $\alpha=2\alpha_1+\alpha_2$, $\beta=\alpha_1+\alpha_2$, 
and $\alpha=3\alpha_1+\alpha_2$, $\beta=\alpha_2$, 
since the following relations hold:
\begin{align*}
x_{112}x_{12}&=x_{11212}+q^2q_{12} \, x_{12}x_{112}, &
x_{112}x_{11212}&=q_{12}^2 \, x_{11212}x_{112},
\\ & &
x_{11212}x_{12}&=q_{12}^2 \, x_{12}x_{11212}; 
\\
x_{1112}x_{2}&=-(3)_q q_{12}\, x_{11212}-q_{12}^3 \, x_{2}x_{1112} -2q^2q_{12}^2 \, x_{12}x_{112}, &
x_{1112}x_{11212}&=q_{12}^3 \, x_{11212}x_{1112}, 
\\ & &
x_{11212}x_{2}&=q_{12}^3 \, x_{2}x_{11212}.
\end{align*}
In both cases, $-\frac{q_{\alpha\alpha}}{q_{\beta\beta}}=-1$ so we take $L=2$.
The solutions of \eqref{eq:equation-roots-cocycle-N=2} are 

\begin{itemize}
\item $n_{112}=n_{1^32^2}=n_{12}=1$, and $n_{\delta}=0$ for the remaining roots, or

\item $n_{1^32}=n_{1^32^2}=n_{2}=1$, and $n_{\delta}=0$ for the remaining roots.
\end{itemize}
Hence Proposition \ref{prop:cocycle-xgamma-N=2} applies and $(\x_{11212}^2)^*$ is a 2-cocycle.

\medbreak
\noindent $\circ$ For $\gamma=12$, 
the case $N_{12}>4$ follows by Lemma \ref{lem:largeN}. Assume now that $N_{12}=4$. We will apply Proposition \ref{prop:cocycle-xgamma-N>2}. The unique pair as in \eqref{eq:alfa-beta} is $\alpha=3\alpha_1+2\alpha_2$, $\beta=\alpha_2$, since the following relations hold:
\begin{align*}
x_{11212}x_{2}&= - 2(1+q)q_{12}^2 \, x_{12}^3 
+qq_{12}^4 \, x_{2}x_{11212}, &
x_{11212}x_{12}&=q^3q_{12} \, x_{12}x_{11212}, 
\\ & &
x_{12}x_{2}&=q^3q_{12} \, x_{2}x_{12}.
\end{align*}
Here $\frac{q_{\alpha\gamma}}{q_{\gamma\beta}}=1$, so we take $L=1$. The unique solution of \eqref{eq:equation-roots-cocycle-N>2} is $n_{1^32^2}=n_{12}=n_{2}=1$, $n_{\delta}=0$ for the remaining roots. Hence Proposition \ref{prop:cocycle-xgamma-N>2} applies: $(\x_{12}^4)^*$ is a 2-cocycle.
\epf
  

\begin{table}[ht]
\caption{The roots with full support of $G_2$; $\gamma_1 < \gamma = \gamma_1 + \gamma_2 < \gamma_2$}\label{tab:G2-1-alpha}
\begin{center}

\begin{tabular}{|c|c|c|c|c|c c|c|c|}
\hline
$\gamma$& $N_{\gamma}$, Cartan & $N_{\gamma}$, \eqref{eq:dynkin-G2-st} b & $P_{\gamma}$ & $Q_{\gamma}$ & $\gamma_1$ & $\gamma_2$ & $L_{\gamma}$, Cartan & $L_{\gamma}$, \eqref{eq:dynkin-G2-st} b \\ \hline
$1^32$ & $M$ & 8 & 2 & 2 & $1$ & $1^22$ & $M$ & 8 \\ \hline
$1^22$ & $N$ & 4 & 4 & 3 & $1$ & $12$ & $N$ & 8 \\ \hline
$1^32^2$ & $M$ & 2 & 2 & 1 & $1^22$ & $12$ & $M$ & 8 \\ 
&  & & &  & $1^32$ & $2$ & & \\ \hline
$12$ & $N$ & 8 & 4 & 3 & $1$ & $2$ & $N$ & 8 \\ \hline
\end{tabular}
\end{center}
\end{table}

\subsection{Type \texorpdfstring{$G_2$}{} standard}\label{sec:G2-standard}

Let $\zeta \in \G'_8$.
In this section, we deal with a Nichols algebra $\toba_{\bq}$ of standard type $G_{2}$ associated to
any of the Dynkin diagrams
\begin{align}\label{eq:dynkin-G2-st}
&\text{a: }\xymatrix{\overset{\,\, \zeta^2}{\underset{\ }{\circ}} \ar  @{-}[r]^{\zeta}  &
\overset{\ztu}{\underset{\ }{\circ}},} &
&\text{b: }\xymatrix{\overset{\,\, \zeta^2}{\underset{\ }{\circ}} \ar  @{-}[r]^{\zeta^3}  &
\overset{-1}{\underset{\ }{\circ}},}& 
&\text{c: }\xymatrix{\overset{\zeta}{\underset{\ }{\circ}} \ar
@{-}[r]^{\zeta^5}  & \overset{-1}{\underset{\ }{\circ}}.}
\end{align}
The set of positive roots is again \eqref{eq:positive-roots-G2}. Thus we take as generators $x_1, x_2$, as well as
\eqref{eq:root-vectors-G2} with the same order for these root vectors:
$x_1 < x_{1112} < x_{112} < x_{11212} < x_{12} < x_2$.
For more information, see \cite[\S 6.2]{AA17}.
We prove Condition \ref{assumption:intro-combinatorial} for type
$G_2$ standard:

\begin{prop}\label{prop:roots-cocycles-G2}
For every  $\gamma \in \varDelta_+^{\bq}$, there exists $L_{\gamma}\in\N$ 
such that $(\x_{\gamma}^{L_\gamma})^*$ is a cocycle.
\end{prop}

\pf
We just consider the diagram \eqref{eq:dynkin-G2-st} b.

\medbreak
\noindent $\circ$ For $\gamma=1^32$, $(\x_{1112}^8)^*$ is a 2-cocycle by Lemma \ref{lem:largeN}.

\medbreak
\noindent $\circ$ For $\gamma=1^22$, we will use Proposition \ref{prop:cocycle-xgamma-N>2}. The unique pair as in \eqref{eq:alfa-beta} is $\alpha=3\alpha_1+\alpha_2$, $\beta=3\alpha_1+2\alpha_2$, since the following relations hold:
\begin{align*}
x_{1112}x_{11212}&= q_{12}^2 x_{112}^3 
-q_{12}^3 \, x_{11212}x_{1112}, &
x_{1112}x_{112}&=-q_{12} \, x_{112}x_{1112}, 
\\ & &
x_{112}x_{11212}&=\zeta^5q_{12} \, x_{11212}x_{112}.
\end{align*}
In this case, $\left(\frac{q_{\alpha\gamma}}{q_{\gamma\beta}} \right)^{N_{\gamma}}=-1$ so we take $L=2$.
The unique solution of \eqref{eq:equation-roots-cocycle-N>2} is $n_{1^22}=3$, $n_{1^32}=n_{1^32^2}=1$, $n_{\delta}=0$ for the remaining roots. Hence Proposition \ref{prop:cocycle-xgamma-N>2} applies and $(\x_{112}^8)^*$ is a 4-cocycle.

\medbreak
\noindent $\circ$ For $\gamma=1^32^2$, we will use Proposition \ref{prop:cocycle-xgamma-N=2}. The pairs as in  \eqref{eq:alfa-beta-N=2} are $\alpha=2\alpha_1+\alpha_2$, $\beta=\alpha_1+\alpha_2$, 
and $\alpha=3\alpha_1+\alpha_2$, $\beta=\alpha_2$, 
since the following relations hold:
\begin{align*}
x_{112}x_{12}&=x_{11212}+\zeta^3q_{12} \, x_{12}x_{112}, \\
x_{112}x_{11212}&=\zeta^7q_{12} \, x_{11212}x_{112},
\\ 
x_{11212}x_{12}&=q_{12} \, x_{12}x_{11212}; 
\\
x_{1112}x_{2}&=\zeta^3(4)_{\zeta^5} q_{12}\, x_{11212}-q_{12}^3 \, x_{2}x_{1112} -\zeta^2(2)_{\zeta^3}q_{12}^2 \, x_{12}x_{112}, \\
x_{1112}x_{11212}&=\zeta^3q_{12}^3 \, x_{11212}x_{1112}, 
\\ 
x_{11212}x_{2}&=q_{12}^3 \, x_{2}x_{11212}.
\end{align*}
In both cases, $-\frac{q_{\alpha\alpha}}{q_{\beta\beta}}=\zeta^7$ so we take $L=8$.
The solutions of \eqref{eq:equation-roots-cocycle-N=2} are 

\begin{itemize}
\item $n_{1^32^2}=7$, $n_{112}=n_{12}=1$, and $n_{\delta}=0$ for the remaining roots, or

\item $n_{1^32^2}=7$, $n_{1^32}=n_{2}=1$, and $n_{\delta}=0$ for the remaining roots.
\end{itemize}
Hence Proposition \ref{prop:cocycle-xgamma-N=2} applies and $(\x_{11212}^8)^*$ is an 8-cocycle.

\medbreak
\noindent $\circ$ For $\gamma=12$, $(\x_{12}^8)^*$ is a 2-cocycle by Lemma \ref{lem:largeN}.
\epf

\subsection{Type \texorpdfstring{$\superda{\alpha}$}{}}\label{sec:type-D2-1-alpha}
Here $q,r,s\neq 1$, $qrs=1$; $N=\ord q$, $M=\ord r$, $P=\ord s$.
In this section, we deal with a Nichols algebra $\toba_{\bq}$ of super type $\superda{\alpha}$ with Dynkin diagram
\begin{align*}
\xymatrix{ \overset{q}{\underset{1 }{\circ}}\ar  @{-}[r]^{q^{-1}}  &
\overset{-1}{\underset{2 }{\circ}} \ar  @{-}[r]^{r^{-1}}  & \overset{r}{\underset{3}{\circ}}}.
\end{align*}
We fix the following convex order on the set of positive roots:
\begin{align*}
1 < 12 < 123  < 12^23 < 2 < 23 <3.
\end{align*}
For more information, see \cite[\S 5.4]{AA17}.
We prove Condition \ref{assumption:intro-combinatorial} for type
$\superda{\alpha}$:

\begin{prop}\label{prop:roots-cocycles-D2,1;alfa4}
For every  $\gamma \in \varDelta_+^{\bq}$, there exists $L_{\gamma}\in\N$ 
such that $(\x_{\gamma}^{L_\gamma})^*$ is a cocycle.
\end{prop}

\pf By induction on the rank it is enough to consider $\gamma$  with full support. 
We start with $\gamma=123$. The pairs $(\alpha,\beta)$ of positive roots as in \eqref{eq:alfa-beta-N=2} are $(1,23)$ and $(12,3)$. As 
\begin{align*}
x_1 x_{23}&=q_{12}q_{13}x_{23}x_1+x_{123}, &
x_1x_{123}&=qq_{12}q_{13}x_{123}x_1, &
x_{123}x_{23} &=-q_{12}q_{13}x_{23}x_{123};
\\
x_{12} x_{3}&=q_{13}q_{23}x_{3}x_{12}+x_{123}, &
x_{12}x_{123}&=-q_{13}q_{23}x_{123}x_{12}, &
x_{123}x_{3} &=rq_{13}q_{23}x_{3}x_{123}.
\end{align*}
the root vectors satisfy \eqref{eq:diff2-hypothesis}. As $-\frac{q_{\alpha\alpha}}{q_{\beta\beta}}=q$, respectively $=r$, we may take $L=\lcm \{N,M,P\}$. There exists a $4$-tuples $(\alpha,\beta,\delta,\eta) \in \varDelta_+^4$ as in \eqref{eq:alfa-beta-delta-eta-N=2}:
\begin{align*}
\alpha=1 & < \eta=12 <\gamma=123 < \beta=12^23 <\delta=3.
\end{align*}
The corresponding PBW generators satisfy \eqref{eq:diff-case2-hypothesis}; indeed,
\begin{align*}
x_{\alpha}x_{\beta}&= q_{\alpha \beta} x_{\beta}x_{\alpha} + qq_{12}q_{13}(1-s) x_{\gamma}x_{\eta}, & x_{\eta}x_{\delta} = q_{\eta\delta}  x_{\delta}x_{\eta} + x_{\gamma},
\end{align*}
and the other pairs of root vectors $q$-commute. Now $\coef{\alpha\beta\gamma}{N}=0$ by Lemma \ref{lemma:coef-roots-unity} \ref{item:coef-roots-unity-ii} since $\widetilde{q}_{\alpha\gamma}=q$, $\widetilde{q}_{\beta\gamma}=s$. 
Next we look for solutions of \eqref{eq:equation-roots-cocycle-N=2}. 
There exist three solutions:
\begin{itemize}[leftmargin=*,label=$\circ$]
\item $n_{123}=L-1$, $n_{1}=n_{23}=1$, $n_{\delta}=0$ if $\delta\neq 123,1,23$;
\item $n_{123}=L-1$, $n_{12}=n_{3}=1$, $n_{\delta}=0$ if $\delta\neq 123,12,3$;
\item $n_{123}=L-2$, $n_{1}=n_{12^3}=n_3=1$, $n_{\delta}=0$ if $\delta\neq 123,1,12^3,3$;
\end{itemize}

Hence all the hypotheses of Proposition \ref{prop:cocycle-xgamma-N=2} hold, and $(\x_{\gamma}^{L_{\gamma}})^*$ is a cocycle.

\medbreak

Now we consider $\gamma=12^23$. If $N_{\gamma}>2=P_{\gamma}$, then Lemma \ref{lem:largeN} applies and $(\x_{\gamma}^{N_{\gamma}})^*$ is a $2$-cocycle. Now assume that $N_{\gamma}=2$. The following relations between root vectors hold:
\begin{align*}
&x_{123} x_{2}=-q_{12}q_{32}x_{2}x_{123}+x_{12^23},
\\
&
\begin{aligned}
x_{123}x_{12^23}&=q_{12}q_{32}x_{12^23}x_{123}, &
x_{12^23}x_{2} &=q_{12}q_{32}x_{2}x_{12^23};
\end{aligned}
\\
&x_{12} x_{23}=-q_{12}q_{13}q_{23}x_{23}x_{12} -q_{23}x_{12^23}+q_{12}(r-1)x_2x_{123},
\\
&
\begin{aligned}
x_{12}x_{12^23}&=-q_{12}q_{13}q_{23}x_{12^3}x_{12}, &
x_{12^23}x_{23} &=-q_{12}q_{13}q_{23}x_{23}x_{123}.
\end{aligned}
\end{align*}
Hence for each pair $(\alpha,\beta)$ as in \eqref{eq:alfa-beta-N=2}, the PBW generators satisfy \eqref{eq:diff2-hypothesis}. Now $L=2$ satifies the hypothesis of Proposition \ref{prop:cocycle-xgamma-N=2}, so $(\x_{\gamma}^2)^\ast$ is a $2$-cocycle.
\epf 

\subsection{Type \texorpdfstring{$\superf$}{}}\label{sec:F(4)}
Let $q$ be a root of 1 of order $N>3$. Set $M=\ord q^2$, $P=\ord q^3$.

In this subsection, we deal with the Nichols algebras $\toba_{\bq}$ of diagonal type $\superf$. We may assume that the corresponding diagram is
\begin{align*}
\xymatrix@C-4pt{ \overset{\,\, q^2}{\underset{1}{\circ}}\ar  @{-}[r]^{q^{-2}}  & \overset{\,\,
q^2}{\underset{2}{\circ}} \ar  @{-}[r]^{q^{-2}}  & \overset{q}{\underset{3}{\circ}} &
\overset{-1}{\underset{4}{\circ}} \ar  @{-}[l]_{q^{-1}}}
\end{align*}
We fix the following convex order on $\varDelta_+^{\bq}$:
\begin{align*}
&1, 12, 123, 123^2, 12^23^2, 2, 23, 23^2, 3, 12^23^34, 12^23^24, 123^24, 23^24, 12^23^34^2, 1234, 234, 34, 4.
\end{align*}
For more information, see \cite[\S 5.5]{AA17}.
The aim of this Section is to prove:

\begin{prop}\label{prop:roots-cocycles-F(4)}
For every positive root $\gamma$, there exists a positive integer $L_{\gamma}$ 
such that $(\x_{\gamma}^{L_\gamma})^*$ is a cocycle.
\end{prop}

\pf 
Let $\gamma$ be a positive non simple root.
Arguing recursively we may assume that $\gamma$ has full support. That is, $\gamma\in \{12^23^34, 12^23^24, 123^24,  12^23^34^2, 1234\}$.

\medbreak

If $\gamma=12^23^24$, then $N_{\gamma}=2=P_{\gamma}$. First we look for pairs $\alpha<\beta$ as in \eqref{eq:alfa-beta-N=2}. We have the following posibilities:

\begin{itemize}[leftmargin=*,label=$\circ$]
\item $(\alpha,\beta)=(2,123^24)$ or $(\alpha,\beta)=(23,1234)$. There exists $\Bsj\in\Bbbk$ such that $[x_{\alpha},x_{\beta}]_c = \Bsj x_{\gamma}$.

\smallbreak

\item $(\alpha,\beta)=(12,23^24)$. There exist $\Bsj, \Bsj_1\in\Bbbk$ such that
\begin{align*}
[x_{\alpha},x_{\beta}]_c &= \Bsj x_{\gamma} + \Bsj_1 x_{123^24} x_{2}.
\end{align*}

\smallbreak

\item $(\alpha,\beta)=(123,234)$ or $(\alpha,\beta)=(12^23^2,4)$. There exist $\Bsj, \Bsj_1\in\Bbbk$ such that
\begin{align*}
[x_{\alpha},x_{\beta}]_c &= \Bsj x_{\gamma} + \Bsj_1 x_{1234} x_{23}.
\end{align*}
\end{itemize}
In all cases $[x_{\alpha},x_{\gamma}]_c =0 = [x_{\gamma},x_{\beta}]_c$, so the root vectors satisfy \eqref{eq:diff2-hypothesis}, and  $\left(-\frac{q_{\alpha\alpha}}{q_{\beta\beta}}\right)^N=1$: the proof of all these relations follow as in Lemma \ref{lem:Btheta-relations-root-vectors}. Hence we take $L=N$.

\medbreak

Next we check that $(12, 2, 12^23^34, 123^24, 23^24, 12^23^34^2)$ is a $6$-uple  $(\alpha, \beta, \delta, \tau, \varphi, \eta)$ satisfying \eqref{eq:alfa-beta-delta-etc-N=2-diff5}. As $\widetilde{q}_{\alpha\gamma}=q^2=\widetilde{q}_{\beta\gamma}$, $\widetilde{q}_{\delta\gamma}=q^{-1}$, we have that
$\coeff{\alpha\beta\delta\gamma}{L}=\coeff{q^{-1},q^2,q^2}{L}=0$ by Lemma
\ref{lemma:coef-roots-unity} \ref{item:coef-roots-unity-iii}.
Hence \ref{item:cocycle-xgamma-N=2-5} holds in this case.

Now we look for $4$-uples  $(\alpha, \beta,\delta,\eta)$ satisfying \eqref{eq:roots-N=2-case6}. There are three possibilities:
$(2,123,123^24,12^23^34^2)$, $(2,12^23^34,123^24,1234)$ and $(12^23^2, 1234, 234,23)$. In these cases, $\widetilde{q}_{\alpha\gamma}=q^2$,  $\widetilde{q}_{\beta\gamma}=q^{-1}$. Hence
$\coef{\alpha\beta\gamma}{L}=\coef{q^2,q^{-1}}{L}=0$ by Lemma
\ref{lemma:coef-roots-unity} \ref{item:coef-roots-unity-ii}, and \ref{item:cocycle-xgamma-N=2-6} holds.

We look for $4$-uples  $(\alpha, \beta,\delta,\eta)$ satisfying \eqref{eq:roots-N=2-case7}. A possibility is
$(12,23,12^23^34^2,23^24)$. As $\widetilde{q}_{\alpha\gamma}=q^2$ and  $\widetilde{q}_{\delta\gamma}=q^{-1}$; thus
$\coef{-\delta\alpha\gamma}{L}=\coef{q,q^2}{L}=0$ by Lemma
\ref{lemma:coef-roots-unity} \ref{item:coef-roots-unity-ii}, and \ref{item:cocycle-xgamma-N=2-7} holds.

Also, $(12, 12^23^34, 234, 12^23^34^2, 1234, 2, 123^24, 23^24)$ is a $7$-uple  $(\alpha, \beta, \delta, \eta, \tau, \mu, \nu)$ satisfying \eqref{eq:roots-N=2-case8}. As $\widetilde{q}_{\alpha\gamma}=q^2$, $\widetilde{q}_{\delta\gamma}=q^{-1}$, we have that
$\coef{-\delta,\alpha,\gamma}{L}=\coef{q,q^2}{L}=0$ by Lemma
\ref{lemma:coef-roots-unity} \ref{item:coef-roots-unity-i}.
Hence \ref{item:cocycle-xgamma-N=2-8} holds in this case.
Now $(2, 12^23^2, 12^23^34, 123^24, 12^23^34^2, 1234)$ and
$(12, 12^23^2, 12^23^34, 23^24, 12^23^34^2, 234)$ are $6$-uples  $(\alpha, \beta, \nu, \mu, \delta, \eta)$ satisfying \eqref{eq:roots-N=2-case9}. As $\widetilde{q}_{\alpha\gamma}=\widetilde{q}_{\beta\gamma}=q^2$, and $\widetilde{q}_{\delta\gamma}=q^{-1}$ in both cases, we have that
$\coeff{\alpha+\beta,\delta,\alpha,\gamma}{L}
=\coeff{q^4,q^{-1},q^{2}}{L}=0$ by Lemma \ref{lemma:coef-roots-unity} \ref{item:coef-roots-unity-ii}.
Thus \ref{item:cocycle-xgamma-N=2-9} holds in this case.

Notice that $\gamma_{1}=2$, $\gamma_{2}=12$, $\gamma_3=12^23^2$, $\gamma_4=\gamma_5=12^23^34^2$ satisfy
$\sum_{i\in\I_5} \gamma_i=4\gamma$. Hence, if $N_{12^23^34^2}=2$, then $n_{2}=n_{12}=n_{12^23^2}=1$, $n_{\gamma}=n_{12^23^34^2}=2$ 
is a solution of \eqref{eq:equation-roots-cocycle-N=2}. The coefficient of $\x_{\gamma}^6 \ot 1$ is zero in
$d(\x_{2}\x_{12}\x_{12^23^2}\x_{\gamma}^2\x_{12^23^34^2}^2 \ot 1)$ by Lemma \ref{lem:diff-case11}.

\medbreak

Finally we look for solutions of \eqref{eq:equation-roots-cocycle-N=2}, i.e. $\sum_{\delta\in\varDelta_+^{\bq}} f_\delta(n_\delta) \delta =N\gamma$, $\sum_{\delta\in\varDelta_+^{\bq}} n_\delta = N+1$. Set $\eta=12^23^34^2$. Looking at the coefficient of $\alpha_4$:
\begin{align*}
N =\sum_{\delta\in\varDelta_+^{\bq}} f_{\delta}(n_{\delta})a_4^{\delta} = 
2f_{\eta}(n_{\eta})+
\sum_{\delta \neq \eta, \, 4\in\supp \delta} n_{\delta} \ge 2f_{\eta}(n_{\eta}).
\end{align*}
As $N_{\eta}=M$, we have that $n_{\eta}\le 3$.
Suppose that $n_{\eta}=3$: necessarily $N=3M$ and
\begin{align*}
\sum_{\delta \neq \eta, \, 3,4\in\supp \delta} n_{\delta}
=3M-n_4-2f_{\eta}(3)=M-n_4-2.
\end{align*}
Looking at the coefficient of $\alpha_3$:
\begin{align*}
\sum_{\delta : 3\in\supp \delta, 4\notin \supp\delta}& f_{\delta}(n_{\delta})a_3^{\delta}
=6M-\sum_{\delta \neq \eta, \, 3,4\in\supp \delta} n_{\delta}a_3^{\delta}-3f_{\eta}(3)
\\
& \le 6M-\sum_{\delta \neq \eta, \, 3,4\in\supp \delta}   n_{\delta}-3f_{\eta}(3) = 2M + n_4-1 <3M=N.
\end{align*}
By inspection, 
$f_{\delta}(2)a_3^{\delta}=N_{\delta}a_3^{\delta}=N$ for all $\delta$ such that $3\in\supp \delta, 4\notin \supp\delta$, so $n_{\delta}\le 1$ for those $\delta$. This implies that $f_{\delta}(n_{\delta})=n_{\delta}$ for all $\delta\neq \eta$ such that $3\in\supp \delta$. Using this fact, the coefficients of $\alpha_2$ and $\alpha_3$ give the following equalities:
\begin{align*}
6M &= \sum_{\delta: 2\in\supp \delta} f_{\delta}(n_{\delta})a_2^{\delta} = f_{12}(n_{12}) + f_{2}(n_{2}) + 2(M+1) + \sum_{\delta\neq \eta: 2,3\in\supp \delta} n_{\delta}a_2^{\delta}
\\
6M &= \sum_{\delta: 3\in\supp \delta} f_{\delta}(n_{\delta})a_3^{\delta} = 3(M+1) + \sum_{\delta\neq \eta: 3\in\supp \delta} n_{\delta}a_3^{\delta}
\end{align*}
From these two equalities:
\begin{align*}
f_{12}(n_{12}) + f_{2}(n_{2}) 
&= M+1 + \sum_{\delta\neq \eta: 3\in\supp \delta} n_{\delta}a_3^{\delta}- \sum_{\delta\neq \eta: 2,3\in\supp \delta} n_{\delta}a_2^{\delta}
\\
&= M+1 + n_{123^2} + n_{23^2} + n_{3} + n_{12^23^34} + n_{123^24} + n_{23^24} + n_{34} \ge M+1.
\end{align*}
As $n_{12^23^34} + n_{123^24} + n_{23^24} + n_{34} \le M-n_4-2$ and $n_{123^2},n_{23^2},n_{3}\le 1$, we have that
\begin{align*}
f_{12}(n_{12}) + f_{2}(n_{2}) 
\le M+1 + 3 + M-n_4-2 \ge 2M+2-n_4\le 2M-2.
\end{align*}
But this is a contradiction since $N_2=N_{12}=P$. A similar argument holds if we suppose that $n_{\eta}=2$, so $n_{\eta}\le 1$. 

Hence $n_{\delta}=f_{\delta}(n_{\delta})$ for all $\delta$ such that $4\in\supp \delta$, so we may translate the equations to the following problem: Find $\gamma_i\in\varDelta_+^{\bq}$, $i\in\I_{N+1}$, such that $\sum_{i\in\I_{N+1}} \gamma_i=N\gamma$.
As $\sum a_1^{\gamma_i}=N$ and $a_1^{\delta}\le 1$ for all $\delta\in\varDelta_+^{\bq}$, we may assume that $a_1^{\gamma_i}=1$ for $i\in\I_N$, $a_1^{\gamma_{N+1}}=0$. As $\sum a_2^{\gamma_i}=2N$ and $a_2^{\delta}\le 2$ for all $\delta\in\varDelta_+^{\bq}$, there are two possible cases: either $a_2^{\gamma_i}=2$ for $N$ of them, $a_2^{\gamma_i}=0$ for the remaining root, or else $a_2^{\gamma_i}=2$ for $N-1$ of them, $a_2^{\gamma_i}=1$ for the remaining two roots. In any case $N-1$ roots have $a_2^{\gamma_i}=2$, and as $a_2^{\delta}=2$ implies that $a_1^{\delta}=1$, we may assume that $a_2^{\gamma_i}=2$ for all $i\in\I_{N-1}$, so $a_3^{\gamma_i}\ge 2$ for all $i\in\I_{N-1}$. As $\sum a_3^{\gamma_i}=2N$, at most two of $a_3^{\gamma_i}$'s are equal to $3$. Therefore we have three cases:

\smallbreak
\begin{enumerate}[leftmargin=*,label=\rm{(\alph*)}]
\item $a_3^{\gamma_1}=a_3^{\gamma_2}=3$, $a_3^{\gamma_i}=2$ for $i\in\I_{3,N-1}$. Hence $a_3^{\gamma_N}=a_3^{\gamma_{N+1}}=0$, which implies that $a_4^{\gamma_N}=a_4^{\gamma_{N+1}}=0$. As $\sum a_4^{\gamma_i}=2N$, at least one of them is equal to $2$. With all these conditions we find exactly two solutions:
\begin{align*}
\gamma_1&=12^23^34^2, & \gamma_2&=12^23^34, & \gamma_i&=\gamma, \, i\in\I_{3,N-1}, & \gamma_{N}&=12, & \gamma_{N+1}&=2;
\\
\gamma_1&=\gamma_2=12^23^34^2, & \gamma_3&=12^23^2, & \gamma_i&=\gamma, \, i\in\I_{4,N-1}, & \gamma_{N}&=12, & \gamma_{N+1}&=2.
\end{align*}
The last solution requires $2=f_{12^23^34^2}(n)$ for some $n\in\N$: the unique possibility is $N_{12^23^34^2}=2$, $n=2$.

\smallbreak

\item $a_3^{\gamma_1}=3$, $a_3^{\gamma_i}=2$ for $i\in\I_{2,N-1}$.
Hence either $a_3^{\gamma_N}=1$, $a_3^{\gamma_{N+1}}=0$ or else $a_3^{\gamma_N}=0$, $a_3^{\gamma_{N+1}}=1$. In the first case, $\gamma_{N+1}=2$, so $a_2^{N}=1$. The solutions are:
\begin{align*}
\gamma_1&=12^23^34^2, & \gamma_i&=\gamma, \, i\in\I_{2,N-2}, & \gamma_{N-1}&=12^23^2, & \gamma_{N}&=1234, & \gamma_{N+1}&=2;
\\
\gamma_1&=12^23^34^2, & \gamma_i&=\gamma, \, i\in\I_{2,N-1}, & \gamma_{N}&=123, & \gamma_{N+1}&=2;
\\
\gamma_1&=12^23^34, & \gamma_i&=\gamma, \, i\in\I_{2,N-1}, & \gamma_{N}&=1234, & \gamma_{N+1}&=2.
\end{align*}
Now we consider $a_3^{\gamma_N}=0$, $a_3^{\gamma_{N+1}}=1$. Notice that $a_2^{\gamma_{N}}, a_2^{\gamma_{N+1}}\le 1$, so $a_2^{\gamma_{N}}=a_2^{\gamma_{N+1}}=1$. This implies that $\gamma_N=12$. We have three solutions:
\begin{align*}
\gamma_1&=12^23^34^2, & \gamma_i&=\gamma, \, i\in\I_{2,N-2}, & \gamma_{N-1}&=12^23^2, & \gamma_{N}&=12, & \gamma_{N+1}&=234;
\\
\gamma_1&=12^23^34^2, & \gamma_i&=\gamma, \, i\in\I_{2,N-1}, & \gamma_{N}&=12, & \gamma_{N+1}&=23;
\\
\gamma_1&=12^23^34, & \gamma_i&=\gamma, \, i\in\I_{2,N-1}, & \gamma_{N}&=12, & \gamma_{N+1}&=234.
\end{align*}

\smallbreak
\item $a_3^{\gamma_i}=2$ for all $i\in\I_{N-1}$. In this case, exactly $N$ of the $a_4^{\gamma_i}$'s are $1$, and the remaining one is $0$. Hence either
\begin{align*}
\gamma_1&=12^23^2, & \gamma_i&=\gamma, \, i\in\I_{2,N-1}, & \gamma_{N}&=1234, & \gamma_{N+1}&=234,
\end{align*}
or $\gamma_i=\gamma$, for all $i\in\I_{N-1}$, so $\gamma_{N}+\gamma_{N+1}=\gamma$: the possible pairs $(\gamma_{N},\gamma_{N+1})$ are
$(123,234)$, $(1234,23)$, $(2,123^4)$, $(12,23^24)$, $(12^23^2,4)$.

\end{enumerate}
Hence all the hypothesis of Proposition \ref{prop:cocycle-xgamma-N=2} hold, and $(\x_{\gamma}^{L_{\gamma}})^*$ is a cocycle.

\bigbreak

If $\gamma=12^23^34$, then $N_{\gamma}=2=P_{\gamma}$.
First we look for pairs $\alpha<\beta$ as in \eqref{eq:alfa-beta-N=2}. We have the following posibilities:
\begin{align*}
&(12^23^2,34), & &(123^2,234), & &(23^2,1234), & &(3,12^23^24), & &(23,123^24), & &(123,23^24).
\end{align*}
In all cases $[x_{\alpha},x_{\gamma}]_c =0 = [x_{\gamma},x_{\beta}]_c$, so the root vectors satisfy \eqref{eq:diff2-hypothesis}, and  $\left(-\frac{q_{\alpha\alpha}}{q_{\beta\beta}}\right)^N=1$. Hence we take $L=N$.

Next we check that $(123^2, 23^2, 12^23^24, 1234, 234, 12^23^34^2)$ is a $6$-uple  $(\alpha, \beta, \delta, \tau, \varphi, \eta)$ satisfying \eqref{eq:alfa-beta-delta-etc-N=2-diff5}. As $\widetilde{q}_{\alpha\gamma}=q^2=\widetilde{q}_{\beta\gamma}$, $\widetilde{q}_{\delta\gamma}=q^{-1}$, we have that $\coeff{\alpha\beta\delta\gamma}{L}=\coeff{q^{-1},q^2,q^2}{L}=0$ by Lemma
\ref{lemma:coef-roots-unity} \ref{item:coef-roots-unity-iii}.
Hence \ref{item:cocycle-xgamma-N=2-5} holds in this case.

Now we look for $4$-uples  $(\alpha, \beta,\delta,\eta)$ satisfying \eqref{eq:roots-N=2-case6}. There are five possibilities:
\begin{align*}
&(123^2,23,12^23^34^2,123^24), &&(12^23^2,3,12^23^34^2,12^23^24), &&(123^2,12^23^24,23^24,3), 
\\
&(23^2,12^23^24,123^24,3), &&(12^23^2,123^24,23^24,23).
\end{align*}
Here $\widetilde{q}_{\alpha\gamma}=q^2$,  $\widetilde{q}_{\beta\gamma}=q^{-1}$. Hence
$\coef{\alpha\beta\gamma}{L}=\coef{q^2,q^{-1}}{L}=0$ by Lemma
\ref{lemma:coef-roots-unity} \ref{item:coef-roots-unity-ii}, and \ref{item:cocycle-xgamma-N=2-6} holds.

We look for $4$-uples  $(\alpha, \beta,\delta,\eta)$ satisfying \eqref{eq:roots-N=2-case7}: the unique is
$(123,23^2,12^23^34^2,23^24)$. As $\widetilde{q}_{\alpha\gamma}=q^2$ and  $\widetilde{q}_{\delta\gamma}=q^{-1}$; thus
$\coef{-\delta\alpha\gamma}{L}=\coef{q,q^2}{L}=0$ by Lemma
\ref{lemma:coef-roots-unity} \ref{item:coef-roots-unity-ii}, and \ref{item:cocycle-xgamma-N=2-7} holds.

There are two $6$-uples  $(\alpha, \beta, \nu, \mu, \delta, \eta)$ satisfying \eqref{eq:roots-N=2-case9}: 
\begin{align*}
&(23^2, 12^23^2, 12^23^24, 1234, 12^23^34^2, 123^24) && \text{and} &&(123^2, 12^23^2, 12^23^24, 234, 12^23^34^2, 23^24). 
\end{align*}
As $\widetilde{q}_{\alpha\gamma}=\widetilde{q}_{\beta\gamma}=q^2$, and $\widetilde{q}_{\delta\gamma}=q^{-1}$ in both cases, we have that
$\coeff{\alpha+\beta,\delta,\alpha,\gamma}{L}
=\coeff{q^4,q^{-1},q^{2}}{L}=0$ by Lemma \ref{lemma:coef-roots-unity} \ref{item:coef-roots-unity-ii}.
Thus \ref{item:cocycle-xgamma-N=2-9} holds.

Also $\gamma_{1}=12^23^2$, $\gamma_{2}=123^2$, $\gamma_3=23^2$, $\gamma_4=\gamma_5=12^23^34^2$ satisfy
$\sum_{i\in\I_5} \gamma_i=4\gamma$. Hence, if $N_{12^23^34^2}=2$, then $n_{12^23^2}=n_{123^2}=n_{23^2}=1$, $n_{\gamma}=n_{12^23^34^2}=2$ 
is a solution of \eqref{eq:equation-roots-cocycle-N=2}. The coefficient of $\x_{\gamma}^6 \ot 1$ is zero in
$d(\x_{12^23^2}\x_{123^2}\x_{23^2}\x_{\gamma}^2\x_{12^23^34^2}^2 \ot 1)$ by Lemma \ref{lem:diff-case11}.

\medbreak

Finally we look for solutions of \eqref{eq:equation-roots-cocycle-N=2}, i.e. $\sum_{\delta\in\varDelta_+^{\bq}} f_\delta(n_\delta) \delta =N\gamma$, $\sum_{\delta\in\varDelta_+^{\bq}} n_\delta = N+1$. Looking at the coefficient of $\alpha_4$ we get as in the previous case that $n_{\delta}=f_{\delta}(n_{\delta})$ for all $\delta$ such that $4\in\supp \delta$, so we may translate the equations to the following problem: Find $\gamma_i\in\varDelta_+^{\bq}$, $i\in\I_{N+1}$, such that $\sum_{i\in\I_{N+1}} \gamma_i=N\gamma$.
As $\sum a_1^{\gamma_i}=N$ and $a_1^{\delta}\le 1$ for all $\delta\in\varDelta_+^{\bq}$, we may assume that $a_1^{\gamma_i}=1$ for $i\in\I_N$, $a_1^{\gamma_{N+1}}=0$. As $\sum a_3^{\gamma_i}=3N$ and $a_3^{\delta}\le 3$ for all $\delta\in\varDelta_+^{\bq}$, at least $N-2$ of these roots satisfy $a_3^{\gamma_i}=3$ and we have three cases:

\smallbreak
\begin{enumerate}[leftmargin=*,label=\rm{(\alph*)}]
\item $a_3^{\gamma_1}=0$, $a_3^{\gamma_i}=3$ for $i\in\I_{2,N+1}$. Then $a_1^{\gamma_i}=1$, $a_2^{\gamma_i}=2$, $a_4^{\gamma_i}\ge 1$ for $i\in\I_{2,N+1}$ but there is no solution in this case.
\smallbreak

\item $a_3^{\gamma_1}=1$, $a_3^{\gamma_2}=2$, $a_3^{\gamma_i}=3$ for $i\in\I_{3,N+1}$. Then $a_1^{\gamma_i}=1$, $a_2^{\gamma_i}=2$, $a_4^{\gamma_i}\ge 1$ if $i\ge 3$, so
\begin{align*}
a_1^{\gamma_1}+a_1^{\gamma_2}&=1, & a_2^{\gamma_1}+a_2^{\gamma_2}&=2, & a_4^{\gamma_1}+a_4^{\gamma_2}&\le 1.
\end{align*}
If $a_4^{\gamma_1}=a_4^{\gamma_2}=0$, then we obtain the following solutions
\begin{align*}
\gamma_1&=123, & \gamma_2&=23^2, & \gamma_3&=12^23^34^2, & \gamma_i&=\gamma \text{  if }i\in\I_{4,N+1}; \\
\gamma_1&=123^2, & \gamma_2&=23, & \gamma_3&=12^23^34^2, & \gamma_i&=\gamma \text{  if }i\in\I_{4,N+1}; \\
\gamma_1&=12^23^2, & \gamma_2&=3, & \gamma_3&=12^23^34^2, & \gamma_i&=\gamma \text{  if }i\in\I_{4,N+1}.
\end{align*}
Otherwise $\gamma_i=\gamma$ for all $i\ge 3$; that is, $\gamma_1+\gamma_2=\gamma$, and the possible pairs $(\gamma_1,\gamma_2)$ are
\begin{align*}
&(12^23^2,34), & &(123^2,234), & &(23^2,1234), & &(3,12^23^24), & &(23,123^24), & &(123,23^24).
\end{align*}

\smallbreak
\item $a_3^{\gamma_i}=2$ for $i\in\I_{3}$, $a_3^{\gamma_i}=3$ for $i\in\I_{4,N+1}$. Then $a_1^{\gamma_i}=1$, $a_2^{\gamma_i}=2$, $a_4^{\gamma_i}\ge 1$ if $i\ge 4$, so
\begin{align*}
a_1^{\gamma_1}+a_1^{\gamma_2}+a_1^{\gamma_3}&=2, & a_2^{\gamma_1}+a_2^{\gamma_2}+a_2^{\gamma_3}&=4, &a_4^{\gamma_1}+a_4^{\gamma_2}+a_4^{\gamma_3}&\le 2.
\end{align*}
If $a_4^{\gamma_1}=a_4^{\gamma_2}=a_4^{\gamma_3}=0$, then the unique solution is 
\begin{align*}
\gamma_1&=12^23^2, & \gamma_2&=123^2, & \gamma_3&=23^2, & \gamma_4&=\gamma_5=12^23^34^2, & \gamma_i&=\gamma \text{  if }i\in\I_{6,N+1}.
\end{align*}
For this solution we need $N_{12^23^34^2}=2$, which implies that $N=6$.

If $a_4^{\gamma_1}=1$, $a_4^{\gamma_2}=a_4^{\gamma_3}=0$, then the solutions are 
\begin{align*}
\gamma_1&=12^23^24, & \gamma_2&=123^2, & \gamma_3&=23^2, & \gamma_4&=12^23^34^2, & \gamma_i&=\gamma \text{  if }i\in\I_{5,N+1}; \\
\gamma_1&=123^24, & \gamma_2&=12^23^2, & \gamma_3&=23^2, & \gamma_4&=12^23^34^2, & \gamma_i&=\gamma \text{  if }i\in\I_{5,N+1}; \\
\gamma_1&=23^24, & \gamma_2&=123^2, & \gamma_3&=12^23^2, & \gamma_4&=12^23^34^2, & \gamma_i&=\gamma \text{  if }i\in\I_{5,N+1}.
\end{align*}

If $a_4^{\gamma_1}=a_4^{\gamma_2}=1$, $a_4^{\gamma_3}=0$, then the unique solution is 
\begin{align*}
\gamma_1&=12^23^24, & \gamma_2&=23^24, & \gamma_3&=123^2, & \gamma_i&=\gamma \text{  if }i\in\I_{4,N+1}; \\
\gamma_1&=123^24, & \gamma_2&=12^23^24, & \gamma_3&=23^2, & \gamma_i&=\gamma \text{  if }i\in\I_{4,N+1}; \\
\gamma_1&=23^24, & \gamma_2&=123^24, & \gamma_3&=12^23^2, & \gamma_i&=\gamma \text{  if }i\in\I_{4,N+1}.
\end{align*}

\end{enumerate}
Hence all the hypothesis of Proposition \ref{prop:cocycle-xgamma-N=2} hold, and $(\x_{\gamma}^{L_{\gamma}})^*$ is a cocycle.

\bigbreak

If $\gamma=123^24$, then $N_{\gamma}=2=P_{\gamma}$. The pairs $\alpha<\beta$ as in \eqref{eq:alfa-beta-N=2} are 
$(1,23^24)$, $(123,34)$, $(\alpha,\beta)=(123^2,4)$. In all cases $[x_{\alpha},x_{\gamma}]_c =0 = [x_{\gamma},x_{\beta}]_c$, so the root vectors satisfy \eqref{eq:diff2-hypothesis}, and  $\left(-\frac{q_{\alpha\alpha}}{q_{\beta\beta}}\right)^N=1$. Hence we take $L=N$.

\medbreak

Now $(123^2,1234,34,3)$ is a $4$-uples  $(\alpha, \beta,\delta,\eta)$ satisfying \eqref{eq:roots-N=2-case6}. Here, $\widetilde{q}_{\alpha\gamma}=q^2$,  $\widetilde{q}_{\beta\gamma}=q^{-1}$, so
$\coef{\alpha\beta\gamma}{L}=\coef{q^2,q^{-1}}{L}=0$ by Lemma
\ref{lemma:coef-roots-unity} \ref{item:coef-roots-unity-ii}, and \ref{item:cocycle-xgamma-N=2-6} holds.

Next we check that $(1,3,12^23^34^2,23^24)$ and $(1,12^23^34,34,23^24)$ are $4$-uples  $(\alpha,\beta,\delta,\eta)$ satisfying \eqref{eq:roots-N=2-case7}. As $\widetilde{q}_{\alpha\gamma}=q^2$ and  $\widetilde{q}_{\delta\gamma}=q^{-1}$ in both cases,
$\coef{-\delta\alpha\gamma}{L}=\coef{q,q^2}{L}=0$ by Lemma
\ref{lemma:coef-roots-unity} \ref{item:coef-roots-unity-ii}, and \ref{item:cocycle-xgamma-N=2-7} holds.

Also, $(1, 123^2, 12^23^34, 23^24, 12^23^34^2, 34)$ is a $6$-uple satisfying \eqref{eq:roots-N=2-case9}. As $\widetilde{q}_{\alpha\gamma}=\widetilde{q}_{\beta\gamma}=q^2$ and $\widetilde{q}_{\delta\gamma}=q^{-1}$, we have that
$\coeff{\alpha+\beta,\delta,\alpha,\gamma}{L}
=\coeff{q^4,q^{-1},q^{2}}{L}=0$ by Lemma \ref{lemma:coef-roots-unity} \ref{item:coef-roots-unity-ii}.
Thus \ref{item:cocycle-xgamma-N=2-9} holds in this case.

\medbreak

Finally we look for solutions of \eqref{eq:equation-roots-cocycle-N=2}, i.e. $\sum_{\delta\in\varDelta_+^{\bq}} f_\delta(n_\delta) \delta =N\gamma$, $\sum_{\delta\in\varDelta_+^{\bq}} n_\delta = N+1$. 
Looking at the coefficient of $\alpha_4$ and arguing as in the case $\gamma=12^23^24$, we find that $n_{\delta}=f_{\delta}(n_{\delta})$ for all $\delta$, 
so we translate the equations to the following problem: Find $\gamma_i\in\varDelta_+^{\bq}$, $i\in\I_{N+1}$, such that $\sum_{i\in\I_{N+1}} \gamma_i=N\gamma$.
As $\sum a_1^{\gamma_i}=N$ and $a_1^{\delta}\le 1$ for all $\delta\in\varDelta_+^{\bq}$, we may assume that $a_1^{\gamma_i}=1$ for $i\in\I_N$, $a_1^{\gamma_{N+1}}=0$. Using a detailed study as the previous case we check that the solutions are
\begin{align*}
\gamma_1&=12^23^34^2, & \gamma_i&=\gamma, \, i\in\I_{2,N-2}, & \gamma_{N-1}&=123^2, & \gamma_{N}&=1, & \gamma_{N+1}&=34;
\\
\gamma_1&=12^23^34^2, & \gamma_i&=\gamma, \, i\in\I_{2,N-1}, & \gamma_{N}&=1, & \gamma_{N+1}&=3;
\\
\gamma_1&=12^23^34, & \gamma_i&=\gamma, \, i\in\I_{2,N-1}, & \gamma_{N}&=1, & \gamma_{N+1}&=34.
\\
\gamma_1&=123^2, & \gamma_i&=\gamma, \, i\in\I_{2,N-1}, & \gamma_{N}&=1234, & \gamma_{N+1}&=34,
\end{align*}
or $\gamma_i=\gamma$, for all $i\in\I_{N-1}$, so $\gamma_{N}+\gamma_{N+1}=\gamma$: the possible pairs $(\gamma_{N},\gamma_{N+1})$ are
$(123,34)$, $(1234,3)$, $(1,23^24)$, $(123^2,4)$.
Hence all the hypothesis of Proposition \ref{prop:cocycle-xgamma-N=2} hold, and $(\x_{\gamma}^{L_{\gamma}})^*$ is a cocycle.

\bigbreak

If $\gamma=1234$, then $N_{\gamma}=2=P_{\gamma}$. The pairs $\alpha<\beta$ as in \eqref{eq:alfa-beta-N=2} are 
$(1,234)$, $(123,4)$, $(12,34)$. In all cases $[x_{\alpha},x_{\gamma}]_c =0 = [x_{\gamma},x_{\beta}]_c$, so the root vectors satisfy \eqref{eq:diff2-hypothesis}, and  $\left(-\frac{q_{\alpha\alpha}}{q_{\beta\beta}}\right)^N=1$. Hence we take $L=N$.

\medbreak

Next we check that $(1,12^23^24^2,234,4)$ and $(12,123^24,4,34)$ are $4$-uples  $(\alpha,\beta,\delta,\eta)$ satisfying \eqref{eq:roots-N=2-case7}. As $\widetilde{q}_{\alpha\gamma}=q^2$ and  $\widetilde{q}_{\delta\gamma}=q^{-1}$ in both cases,
$\coef{-\delta\alpha\gamma}{L}=\coef{q,q^2}{L}=0$ by Lemma
\ref{lemma:coef-roots-unity} \ref{item:coef-roots-unity-ii}, so \ref{item:cocycle-xgamma-N=2-7} holds.

Notice that $\gamma_{1}=1$, $\gamma_{2}=12$, $\gamma_3=12^23^34^2$, $\gamma_4=4$ satisfy
$\sum_{i\in\I_4} \gamma_i=3\gamma$. The corresponding root vectors $q$-commute so the coefficient of $\x_{\gamma}^6 \ot 1$ in
$d(\x_{1}\x_{12}\x_{12^23^34^2}\x_{\gamma}^{N-3}\x_{4} \ot 1)$ is zero by Remark \ref{rem:differential-q-commute}.

\medbreak

Finally we look for solutions of \eqref{eq:equation-roots-cocycle-N=2}, i.e. $\sum_{\delta\in\varDelta_+^{\bq}} f_\delta(n_\delta) \delta =N\gamma$, $\sum_{\delta\in\varDelta_+^{\bq}} n_\delta = N+1$. 
Looking at the coefficient of $\alpha_4$ and arguing as in the case $\gamma=12^23^24$, we find that $n_{\delta}=f_{\delta}(n_{\delta})$ for all $\delta$, 
so we translate the equations to the following problem: Find $\gamma_i\in\varDelta_+^{\bq}$, $i\in\I_{N+1}$, such that $\sum_{i\in\I_{N+1}} \gamma_i=N\gamma$.
As $\sum a_1^{\gamma_i}=N$ and $a_1^{\delta}\le 1$ for all $\delta\in\varDelta_+^{\bq}$, we may assume that $a_1^{\gamma_i}=1$ for $i\in\I_N$, $a_1^{\gamma_{N+1}}=0$. Using a detailed study as the previous case we check that the solutions are
\begin{align*}
\gamma_1&=12^23^34^2, & \gamma_i&=\gamma, \, i\in\I_{2,N-2}, & \gamma_{N-1}&=12, & \gamma_{N}&=1, & \gamma_{N+1}&=4;
\\
\gamma_1&=12^23^24, & \gamma_i&=\gamma, \, i\in\I_{2,N-1}, & \gamma_{N}&=1, & \gamma_{N+1}&=4.
\\
\gamma_1&=12, & \gamma_i&=\gamma, \, i\in\I_{2,N-1}, & \gamma_{N}&=123^24, & \gamma_{N+1}&=4,
\end{align*}
or $\gamma_i=\gamma$, for all $i\in\I_{N-1}$, so $\gamma_{N}+\gamma_{N+1}=\gamma$: the possible pairs $(\gamma_{N},\gamma_{N+1})$ are
$(123,4)$, $(1,234)$, $(12,34)$.
Hence Proposition \ref{prop:cocycle-xgamma-N=2} applies and $(\x_{\gamma}^{L_{\gamma}})^*$ is a cocycle.

\bigbreak

Finally, if $\gamma=12^23^34^2$, then $N_{\gamma}=M$, $P_{\gamma}=2$, $Q_{\gamma}=1$. If $N\neq 6$, then $N_{\gamma}>P_{\gamma}=2,Q_{\gamma}$, so $(\x_{\gamma}^{N_{\gamma}})^*$ is a $2$-cocycle by Lemma \ref{lem:largeN}.
Next we assume $N=6$; that is, $N_{\gamma}=2$.
Let $\alpha<\beta$ be a pair of positive roots as in \eqref{eq:alfa-beta-N=2}. We have the following posibilities:
\begin{itemize}[leftmargin=*,label=$\circ$]
\item $\alpha=23^24$, $\beta=1234$. There exists $\Bsj\in\Bbbk$ such that $[x_{\alpha},x_{\beta}]_c = \Bsj x_{\gamma}$.
\smallbreak

\item $\alpha=123^24$, $\beta=234$. There exist $\Bsj, \Bsj_1\in\Bbbk$ such that
\begin{align*}
[x_{\alpha},x_{\beta}]_c &= \Bsj x_{\gamma} + \Bsj_1 x_{1234} x_{23^24}.
\end{align*}

\smallbreak

\item $\alpha=12^23^24$, $\beta=34$. There exist $\Bsj, \Bsj_t\in\Bbbk$ such that
\begin{align*}
[x_{\alpha},x_{\beta}]_c &= \Bsj x_{\gamma} + \Bsj_1 x_{1234} x_{23^24} + \Bsj_2 x_{234} x_{123^24}.
\end{align*}

\smallbreak

\item $\alpha=12^23^34$, $\beta=4$. There exist $\Bsj, \Bsj_t\in\Bbbk$ such that
\begin{align*}
[x_{\alpha},x_{\beta}]_c &= \Bsj x_{\gamma} + \Bsj_1 x_{1234} x_{23^24} + \Bsj_2 x_{234} x_{123^24}+ \Bsj_3 x_{34} x_{12^23^24}.
\end{align*}
\end{itemize}

In all cases $[x_{\alpha},x_{\gamma}]_c =0 = [x_{\gamma},x_{\beta}]_c$, so the root vectors satisfy \eqref{eq:diff2-hypothesis}, and  $-\frac{q_{\alpha\alpha}}{q_{\beta\beta}}=-1$: the proof of all these relations follow as in Lemma \ref{lem:Btheta-relations-root-vectors}. Hence we take $L=2$. 

Next we look for solutions of \eqref{eq:equation-roots-cocycle-N=2}. That is, $\sum_{\delta\in\varDelta_+^{\bq}} f_\delta(n_\delta) \delta =2\gamma$, $\sum_{\delta\in\varDelta_+^{\bq}} n_\delta = 3$.
Suppose that $n_{\eta}=3$ for some $\eta\in\varDelta_+^{\bq}$, then $n_{\delta}=0$ for $\delta\neq \eta$ and $2\gamma=(N_{\eta}+1)\eta$, a contradiction. Now suppose that $n_{\eta}=2$, $n_{\tau}=1$ for $\eta\neq \tau$: $2\gamma=N_{\eta}\eta+\tau$. As $a_1^{\eta}, a_1^{\tau}\le 1$ and $2=N_{\eta}a_1^{\eta}+a_1^{\tau}$, we have that $a_1^{\eta}=1$, $N_{\eta}=2$, $a_1^{\tau}=0$. As $4=2a_2^{\eta}+a_2^{\tau}$ and $a_2^{\eta}\le 2$, $a_2^{\tau}\le 1$,  we have that $a_2^{\eta}=2$, $a_2^{\tau}=0$. Thus $\tau\in\{3,34,4\}$, but there is no solution for these cases, a contradiction.

Therefore, $n_{\eta}=n_{\tau}=n_{\mu}=1$ for three different roots $\eta, \tau, \mu$. As $a_1^{\eta}, a_1^{\tau}, a_1^{\mu}\le 1$ and $2=a_1^{\eta}+a_1^{\tau}+a_1^{\mu}$, we may assume $a_1^{\eta}=a_1^{\tau}=1$, $a_1^{\mu}=0$. As $a_2^{\eta}, a_2^{\tau}\le 2$, $a_2^{\mu}\le 1$ and $4=a_2^{\eta}+a_2^{\tau}+a_2^{\mu}$, either $a_2^{\eta}=a_2^{\tau}=2$, $a_2^{\mu}=0$ or else $a_2^{\eta}=2$, $a_2^{\tau}=a_2^{\mu}=1$. In the first case, $a_3^{\eta}, a_3^{\tau}\le 3$, $a_3^{\mu}\le 1$ and $6=a_3^{\eta}+a_3^{\tau}+a_3^{\mu}$, so either $a_3^{\eta}=a_3^{\tau}=3$, $a_3^{\mu}=0$ or else $a_3^{\eta}=3$, $a_3^{\tau}=2$, $a_3^{\mu}=1$; in both cases we are forced to get $\eta=\gamma=\tau+\mu$, and moreover we obtain only two possibilities, either $\tau=12^23^24$, $\mu=34$ or else $\tau=12^23^34$, $\mu=4$. In the second case, $a_4^{\eta}\le 2$, $a_4^{\tau},a_4^{\mu}\le 1$ and $4=a_4^{\eta}+a_4^{\tau}+a_4^{\mu}$, so $a_4^{\eta}=2$, $a_4^{\tau}=a_4^{\mu} =1$, and again we are forced to get $\eta=\gamma=\tau+\mu$, and moreover we obtain only two possibilities, either $\tau=123^24$, $\mu=234$ or else $\tau=1234$, $\mu=23^24$.

Hence all the hypothesis of Proposition \ref{prop:cocycle-xgamma-N>2} hold, and $(\x_{\gamma}^{N_{\gamma}})^*$ is a $2$-cocycle.
\epf

\subsection{Type \texorpdfstring{$\superg$}{}}\label{sec:g(3)}
Let $q$ be a root of 1 of order $N > 3$.
In this section, we deal with a Nichols algebra $\toba_{\bq}$ of super type $\superg$, associated to the Dynkin diagram
\begin{align*}
&\xymatrix{ \overset{-1}{\underset{1}{\circ}}\ar  @{-}[r]^{q^{-1}}  &
\overset{q}{\underset{2}{\circ}} \ar  @{-}[r]^{q^{-3}}  & \overset{\,\, q^3}{\underset{3}{\circ}}.}
\end{align*}
For more information, see \cite[\S 5.6]{AA17}.
The set of positive roots with full support is
\begin{align*}
& \{123,12^23,12^33,12^33^2,12^43^2\}.
\end{align*}
We fix the following convex order of $\varDelta_+^{\bq}$:
\begin{align*}
1<12<123<12^23<12^33<12^33^2<12^43^2<2<2^33<2^23<2^33^2<23<3.
\end{align*}
It comes from Lyndon words once we fix the order of the letters $1<2<3$ and differs from the one in \cite[\S 5.6]{AA17}.
We prove Condition \ref{assumption:intro-combinatorial} for type
$\superg$:

\begin{prop}\label{prop:roots-cocycles-G(3)}
For every  $\gamma \in \varDelta_+^{\bq}$, there exists $L_{\gamma}\in\N$ such that $(\x_{\gamma}^{L_\gamma})^*$ is a cocycle.
\end{prop}

\pf
It is enough to prove the statement for $\gamma$  with full support.

\medbreak

\noindent $\circ$ For $\gamma=123$, we apply Proposition \ref{prop:cocycle-xgamma-N=2}. The pairs as in  \eqref{eq:alfa-beta-N=2} are
$\alpha=\alpha_1$, $\beta=\alpha_2+\alpha_3$, and 
$\alpha=\alpha_1+\alpha_2$, $\beta=\alpha_3$ since the following relations hold:
\begin{align*}
x_{\alpha}x_{\beta}&=x_{\gamma}+q_{\alpha\beta} \, x_{\beta}x_{\alpha}, &
x_{\alpha}x_{\gamma}&=q_{\alpha} \, x_{\gamma}x_{\alpha}, &
x_{\gamma}x_{\beta}&=q_{\beta} \, x_{\beta}x_{\gamma}.
\end{align*}
As $-\frac{q_{\alpha\alpha}}{q_{\beta\beta}}=q$, respectively $q^3$, $L$ should be a multiple of $N$.

Let $L=\ord (-q)$. Now \eqref{eq:alfa-beta-delta-eta-N=2} holds for $\alpha=1$, $\beta=12^23$, $\delta=3$, $\eta=12$, and the root vectors satisfy \eqref{eq:diff-case2-hypothesis}; the scalars $\widetilde{q}_{\alpha\gamma}=q^{-1}$ and $\widetilde{q}_{\beta\gamma}=q^{-2}$ satisfy
$\coef{\alpha\beta\gamma}{L}= 0$.

Also, \eqref{eq:alfa-beta-delta-eta-tau-N=2} holds for $\alpha=1$, $\beta=12^33^2$, $\delta=3$, $\eta=12$, $\tau=12^23$, and the root vectors satisfy \eqref{eq:diff-case3-hypothesis}; the scalars $\widetilde{q}_{\alpha\gamma}=q^{-1}$, $\widetilde{q}_{\beta\gamma}=q^{-1}$, $\widetilde{q}_{\tau\gamma}=q^{-2}$ satisfy \eqref{eq:Csj-Dsj-N=2-3}.

Let $(n_{\delta})_{\delta\in\varDelta_+^{\bq}}$ be a solution of \eqref{eq:equation-roots-cocycle-N=2}. If $n_3=0$, then
\begin{align*}
L(\alpha_1+\alpha_2) = s_3(L\gamma) = \sum_{\delta\in\varDelta_+^{\bq}} f_{\delta}(n_{\delta}) s_3(\delta),
\end{align*}
and $s_3(\delta)\in\varDelta_+^{\bq}$ if $\delta\neq \alpha_3$. As $N_{\delta}=N_{s_3(\delta)}$, we have that $f_{\delta}=f_{s_3(\delta)}$, so we have a system as in \eqref{eq:equation-roots-cocycle-N=2} for $\alpha_1+\alpha_2$ in place of $\gamma$ and we may restrict the support to $\alpha_1$, $\alpha_2$. The new system has a unique solution, which gives place to the solution of the original system: 
\begin{itemize}
\item $n_{1}=n_{23}=1$, $n_{123}=L-1$, $n_{\delta}=0$ for all the other $\delta\in\varDelta_+^{\bq}$.
\end{itemize}
Next we assume $n_3\neq 0$. Suppose that $n_{12^23}>1$. Then $f_{12^23}(n_{12^23})\ge N$, so the coefficient of $\alpha_2$ in $\sum_{\delta\in\varDelta_+^{\bq}} f_{\delta}(n_{\delta}) \delta$ is $\ge 2N$, a contradiction. Hence $n_{12^23}\le 1$, and then $f_{12^23}(n_{12^23})=n_{12^23}\le 1$. The coefficient of $\alpha_1$ in this sum is 
\begin{align*}
n_1+n_{12}+n_{123}+n_{12^23}+n_{12^33}+n_{12^33^2}+n_{12^43^2}=L.
\end{align*}
As the sum of all $n_{\delta}$'s is $L+1$ and $n_3\neq 0$, we have $n_3=1$, $n_2=n_{2^33}=n_{2^23}=n_{2^33^2}=n_{23}=0$. Now we look at the coefficients of $\alpha_2$, $\alpha_3$ in the equality $\sum_{\delta\in\varDelta_+^{\bq}} f_{\delta}(n_{\delta}) \delta =L\gamma$:
\begin{align*}
L&= n_{12}+n_{123}+2n_{12^23}+3n_{12^33}+3n_{12^33^2}+4n_{12^43^2}, \\
L&= 1+n_{123}+n_{12^23}+n_{12^33}+2n_{12^33^2}+2n_{12^43^2}.
\end{align*}
Thus $n_{12}+n_{12^23}+2n_{12^33}+n_{12^33^2}+2n_{12^43^2}=1$, which implies that $n_{12^33}=n_{12^43^2}=0$ and two of the three numbers $n_{12}$, $n_{12^23}$, $n_{12^33^2}$ are zero (the remaining one being 1).
Looking at the three possibilities, we have three solutions:
\begin{itemize}
\item $n_{12}=n_3=1$, $n_{123}=L-1$, $n_{\delta}=0$ for all the other $\delta\in\varDelta_+^{\bq}$; 
\item $n_{1}=n_3=n_{12^23}=1$, $n_{123}=L-2$, $n_{\delta}=0$ for all the other $\delta\in\varDelta_+^{\bq}$; 
\item $n_{1}=2$, $n_{12^33^2}=n_3=1$, $n_{123}=L-3$, $n_{\delta}=0$ for all the other $\delta\in\varDelta_+^{\bq}$.
\end{itemize}

Hence Proposition \ref{prop:cocycle-xgamma-N=2} applies and $(\x_{123}^L)^*$ is a cocycle.

\medbreak
\noindent $\circ$ For $\gamma=12^23$, 
the case $N_{\gamma}>3$ follows by Lemma \ref{lem:largeN}. Assume now that $N_{\gamma}=3$. We will work as in Proposition \ref{prop:cocycle-xgamma-N>2}. The pairs as in \eqref{eq:alfa-beta} are $(1,12^43^2)$, $(12,12^33^2)$, $(123,12^33)$,
since for each one of these pairs the following relations hold:
\begin{align*}
x_{\alpha}x_{\beta}&= \Bsj_{\alpha\beta} x_{\gamma}^2+q_{\alpha\beta} \, x_{\beta}x_{\alpha}, &
x_{\alpha}x_{\gamma}&=q_{\alpha\gamma} \, x_{\gamma}x_{\alpha}, &
x_{\gamma}x_{\beta}&= q_{\gamma\beta} \, x_{\beta}x_{112},
& & \Bsj_{\alpha\beta}\in\Bbbk.
\end{align*}
As $\frac{q_{\alpha\gamma}}{q_{\gamma\beta}}=1$ for the three cases, we take $L=1$.
We look for solutions of \eqref{eq:equation-roots-cocycle-N>2}: 
\begin{align*}
\sum_{\delta\in\varDelta_+^{\bq}} n_{\delta} &=3, & 
\sum_{\delta\in\varDelta_+^{\bq}} f_{\delta}(n_{\delta})\delta &=3\gamma.
\end{align*}

If $n_{\gamma}\ge 2$, then $f_{\gamma}(n_{\gamma})\ge 3$, a contradiction. Then $n_{\gamma}\le 1$, so $f_{\gamma}(n_{\gamma})=n_{\gamma}$. Looking at the coefficient of $\alpha_1$ we get the equation:
\begin{align}\label{eq:G3-system-alpha1}
3 &= n_1+n_{12}+n_{123}+n_{12^23}+n_{12^33}+n_{12^33^2}+n_{12^43^2}.
\end{align}
Hence $n_{\delta}=0$ for $\delta=2,2^33,2^23,2^33^2,23,3$.
Looking at the coefficients of $\alpha_2$ and $\alpha_3$,
\begin{align}\label{eq:G3-system-alpha2}
6 &= n_{12}+n_{123}+2n_{12^23}+3n_{12^33}+3n_{12^33^2}+4n_{12^43^2},
\\
\label{eq:G3-system-alpha3}
3 &= n_{123}+n_{12^23}+n_{12^33}+2n_{12^33^2}+2n_{12^43^2}.
\end{align}
From \eqref{eq:G3-system-alpha1} and \eqref{eq:G3-system-alpha3}, $n_1+n_{12} = n_{12^33^2}+n_{12^43^2}$. From \eqref{eq:G3-system-alpha3}, $n_{12^33^2}+n_{12^43^2}\le 1$. If $n_1=n_{12}=0$, then $n_{12^33^2}=n_{12^43^2}=0$: the solution is $n_{123}=n_{12^23}=n_{12^33}=1$. Next we assume $n_1+n_{12}=1 = n_{12^33^2}+n_{12^43^2}$. If $n_{12^23}=1$, then the solutions give pairs as in \eqref{eq:alfa-beta}. Otherwise we have a unique solution: $n_{12}=n_{123}=n_{12^43^2}=1$, $n_{\delta}=0$ otherwise. 
Hence we have to compute $d(\x_{12}\x_{123}\x_{12^43^2} \ot 1)$. Notice that
\begin{align*}
x_{12}x_{12^43^2} &= -q^3q_{12}^3q_{13}^2q_{23}^2 \, x_{12^43^2}x_{12}+ \Bsj_1 x_{12^33} x_{12^23}, 
\\
x_{123}x_{12^43^2}&= -q^3q_{12}^3q_{13}q_{32}^2 \, x_{12^43^2} x_{123} + \Bsj_2 x_{12^33^2} x_{12^23}, & 
\\
x_{12}x_{12^33^2} &= -q^2q_{12}^2q_{13}^2q_{23}^2 \, x_{12^33^2}x_{12}+ \Bsj_3 x_{12^23}^2 + \Bsj_4 x_{12^33} x_{123},
\\
x_{123}x_{12^33} &= -q^2q_{12}^2q_{32}^2 \, x_{12^33}x_{123}+ \Bsj_5 x_{12^23}^2,
\end{align*}
for some $\Bsj_j\in\Bbbk$.
Using these relations, we get
\begin{align*}
d( & \x_{12}\x_{123}\x_{12^43^2} \ot 1) = \x_{12} \x_{123}\ot x_{12^43^2} - s (\x_{12} \ot x_{123} x_{12^43^2} +q_{13}q_{23} \x_{123}\ot x_{12} x_{12^43^2})
\\
& = \x_{12} \x_{123}\ot x_{12^43^2} - s \big(-q^3q_{12}^3q_{13}q_{32}^2 \,\x_{12} \ot  x_{12^43^2} x_{123} +
\Bsj_2 \x_{12} \ot x_{12^33^2} x_{12^23}
\\
& \quad -q^3q_{12}^3q_{13}^3q_{23}^3 \, \x_{123}\ot  x_{12^43^2}x_{12}
+q_{13}q_{23}\Bsj_1 \x_{123}\ot  x_{12^33} x_{12^23} \big)
\\
& = \x_{12} \x_{123}\ot x_{12^43^2}-\Bsj_2 \x_{12} \x_{12^33^2} \ot  x_{12^23} - s \big(-q^3q_{12}^3q_{13}q_{32}^2 \,\x_{12} \ot  x_{12^43^2} x_{123} 
\\
& \quad 
- \Bsj_2 q^2q_{12}^2q_{13}^2q_{23}^2 \, \x_{12^33^2} \ot x_{12} x_{12^23}+\Bsj_2\Bsj_3 \x_{12^23} \ot x_{12^23}^2 + \Bsj_2 \Bsj_4 \x_{12^33} \ot x_{123} x_{12^23}
\\
& \quad -q^3q_{12}^3q_{13}^3q_{23}^3 \, \x_{123}\ot  x_{12^43^2}x_{12}
+q_{13}q_{23}\Bsj_1 \x_{123}\ot  x_{12^33} x_{12^23} \big)
\\
& = \x_{12} \x_{123}\ot x_{12^43^2}-\Bsj_2 \x_{12} \x_{12^33^2} \ot  x_{12^23} 
+q^3q_{12}^3q_{13}q_{32}^2 \,\x_{12} \x_{12^43^2} \ot   x_{123}
\\
&\quad
- s \big(  q_{12}^7q_{23}q_{13}^3 \, \x_{12^43^2} \otimes x_{123}x_{12} - \Bsj_1 q^3q_{12}^3q_{13}q_{32}^2 \x_{12^33} \otimes  x_{12^23}x_{123}
\\
& \quad 
+ \Bsj_2 q^3q_{12}^3q_{13}^3q_{23}^3 \, \x_{12^33^2} \ot  x_{12^23} x_{12} +\Bsj_2\Bsj_3 \x_{12^23} \ot x_{12^23}^2 -qq_{12}q_{32} \Bsj_2 \Bsj_4 \x_{12^33} \ot  x_{12^23}x_{123}
\\
& \quad -q^3q_{12}^3q_{13}^3q_{23}^3 \, \x_{123}\ot  x_{12^43^2}x_{12}
+q_{13}q_{23}\Bsj_1 \x_{123}\ot  x_{12^33} x_{12^23} \big)
\\
& = \x_{12} \x_{123}\ot x_{12^43^2}-\Bsj_2 \x_{12} \x_{12^33^2} \ot  x_{12^23} 
+q^3q_{12}^3q_{13}q_{32}^2 \,\x_{12} \x_{12^43^2} \ot   x_{123}
\\
&\quad 
-q_{13}q_{23}\Bsj_1 \x_{123} \x_{12^33}\ot   x_{12^23}
+q^3q_{12}^3q_{13}^3q_{23}^3 \, \x_{123} \x_{12^43^2} \ot  x_{12}
\\ & \quad 
+ (q_{13}q_{23}\Bsj_1\Bsj_5 -\Bsj_2\Bsj_3) \x_{12^23}^3 \ot 1.
\end{align*}
We compute the scalars $\Bsj_j$ using the form of the Lyndon words and the $q$-Jacobi identity:
\begin{align*}
\Bsj_1&= q_{12}^2q_{13}q_{23} q(1-q), \
\Bsj_2= q_{12}^2q_{13} q(1-q), \
\Bsj_3= q_{12}q_{13}q_{23}^3 q(1+q), \
\Bsj_5= q_{12}q_{23} q(1+q).
\end{align*}
Hence $q_{13}q_{23}\Bsj_1\Bsj_5 -\Bsj_2\Bsj_3=0$. Thus the coefficient of $\x_{12^23}^3 \ot 1$ in $d(c)$ is zero for all 2-chains $c$, so $(\x_{12^23}^3)^{\ast}$ is a $2$-cocycle.

\medbreak

\noindent $\circ$ For $\gamma=12^33$, we will apply Proposition \ref{prop:cocycle-xgamma-N=2}. The pairs $(\alpha,\beta)$ as in \eqref{eq:alfa-beta-N=2} are
$(1,2^33)$, $(12,2^23)$, $(12^23,2)$, since the following relations hold:
\begin{align*}
x_{\alpha}x_{\beta}&=x_{\gamma}+q_{\alpha\beta} \, x_{\beta}x_{\alpha}, &
x_{\alpha}x_{\gamma}&=q_{\alpha} \, x_{\gamma}x_{\alpha}, &
x_{\gamma}x_{\beta}&=q_{\beta} \, x_{\beta}x_{\gamma}.
\end{align*}
As $-\frac{q_{\alpha\alpha}}{q_{\beta\beta}}=q$, respectively $q^2$, $q^3$, 
$L$ should be a multiple of $N$.

Let $L=\ord (-q)$. Now \eqref{eq:alfa-beta-delta-eta-tau-N=2-diff4} holds for $\alpha=12$, $\beta=12^23$, $\delta=2^33$, and the root vectors satisfy \eqref{eq:diff-case2-hypothesis}; the scalars $\widetilde{q}_{\alpha\gamma}=q^{-1}$, $\widetilde{q}_{\beta\gamma}=q^{-2}$ satisfy \eqref{eq:Csj-Dsj-N=2-4}.

Also, \eqref{eq:alfa-beta-delta-eta-tau-N=2} holds for $\alpha=12$, $\beta=12^43^2$, $\delta=2^33$, $\eta=123$, $\tau=12^23$, and the root vectors satisfy \eqref{eq:diff-case3-hypothesis}; the scalars $\widetilde{q}_{\alpha\gamma}=q^{-1}$, $\widetilde{q}_{\beta\gamma}=q^{-1}$, $\widetilde{q}_{\tau\gamma}=q^{-2}$ satisfy \eqref{eq:Csj-Dsj-N=2-3}.

Let $(n_{\delta})_{\delta\in\varDelta_+^{\bq}}$ be a solution of \eqref{eq:equation-roots-cocycle-N=2}.
If $n_2=0$, then
\begin{align*}
N(\alpha_1+\alpha_2+\alpha_3) = s_2(N\gamma) = \sum_{\delta\in\varDelta_+^{\bq}} f_{\delta}(n_{\delta}) s_3(\delta),
\end{align*}
and $s_2(\delta)\in\varDelta_+^{\bq}$ if $\delta\neq \alpha_2$. As $N_{\delta}=N_{s_2(\delta)}$, we have that $f_{\delta}=f_{s_2(\delta)}$, so we have a system as in \eqref{eq:equation-roots-cocycle-N=2} for $\alpha_1+\alpha_2+\alpha_3$ in place of $\gamma$. The new system has four solutions, which gives place to the following solutions of the original system: 
\begin{itemize}
\item $n_{12}=n_{2^23}=1$, $n_{12^33}=N-1$, $n_{\delta}=0$ for all the other $\delta\in\varDelta_+^{\bq}$;
\item $n_{1}=n_{2^33}=1$, $n_{12^33}=N-1$, $n_{\delta}=0$ for all the other $\delta\in\varDelta_+^{\bq}$; 
\item $n_{12}=n_{2^33}=n_{12^23}=1$, $n_{12^33}=N-2$, $n_{\delta}=0$ for all the other $\delta\in\varDelta_+^{\bq}$; 
\item $n_{12}=2$, $n_{12^43^2}=n_{2^33}=1$, $n_{12^33}=N-3$, $n_{\delta}=0$ for all the other $\delta\in\varDelta_+^{\bq}$.
\end{itemize}
Next we assume $n_2\neq 0$. Suppose that $n_{12^23}>1$. Then $f_{12^23}(n_{12^23})\ge N$, so the coefficient of $\alpha_1$ in $\sum_{\delta\in\varDelta_+^{\bq}} f_{\delta}(n_{\delta}) \delta$ is $\ge N$; this forces to $n_{\delta}=0$ for any $\delta\neq \alpha_2$ since $f_{\delta}(n_{\delta})\delta$ must have $\alpha_1$, $\alpha_3$ with coefficient zero, and this gives a contradiction. Hence $n_{12^23}=f_{12^23}(n_{12^23})\le 1$. The coefficcient of $\alpha_1$ in this sum is 
\begin{align*}
n_1+n_{12}+n_{123}+n_{12^23}+n_{12^33}+n_{12^33^2}+n_{12^43^2}=N.
\end{align*}
As the sum of all $n_{\delta}$'s is $N+1$ and $n_2\neq 0$, we have $n_2=1$, $n_3=n_{2^33}=n_{2^23}=n_{2^33^2}=n_{23}=0$. Now we look at the coefficients of $\alpha_2$, $\alpha_3$ in the equality $\sum_{\delta\in\varDelta_+^{\bq}} f_{\delta}(n_{\delta}) \delta =N\gamma$:
\begin{align*}
3N&= 1+ n_{12}+n_{123}+2n_{12^23}+3n_{12^33}+3n_{12^33^2} +4n_{12^43^2}, \\
N&= n_{123}+n_{12^23}+n_{12^33}+2n_{12^33^2}+2n_{12^43^2}.
\end{align*}
Thus $n_{1}+2n_{123}+n_{12^33}+2n_{12^33^2}+4n_{12^43^2}=1$, which implies that $n_{123}=n_{12^33^2}=0$ and two of the three numbers $n_{1}$, $n_{12^33}$, $n_{12^43^2}$ are zero (the remaining one being 1). Reducing the three previous equations, 
we get $n_{12}+n_{12^33}=N-1$, $n_{12}+3n_{12^33}\ge 3N-3$, so 
we have a unique solution:
\begin{itemize}
\item $n_{12^23}=n_2=1$, $n_{12^33}=N-1$, $n_{\delta}=0$ for all the other $\delta\in\varDelta_+^{\bq}$.
\end{itemize}

Hence Proposition \ref{prop:cocycle-xgamma-N=2} applies and $(\x_{12^33}^L)^*$ is a cocycle.

\medbreak

\noindent $\circ$ For $\gamma=12^33^2$, we will apply Proposition \ref{prop:cocycle-xgamma-N=2}. The pairs $(\alpha,\beta)$ as in \eqref{eq:alfa-beta-N=2} are
$(1,2^33^2)$, $(123,2^23)$, $(12^23,23)$, $(12^33,3)$,
since the following relations hold:
\begin{align*}
x_{\alpha}x_{\beta}&=x_{\gamma}+q_{\alpha\beta} \, x_{\beta}x_{\alpha}, &
x_{\alpha}x_{\gamma}&=q_{\alpha} \, x_{\gamma}x_{\alpha}, &
x_{\gamma}x_{\beta}&=q_{\beta} \, x_{\beta}x_{\gamma}.
\end{align*}
As $-\frac{q_{\alpha\alpha}}{q_{\beta\beta}}\in \{q,q^2,q^3\}$, 
$L$ should be a multiple of $N$.

Let $L=\ord (-q)$. Now \eqref{eq:alfa-beta-delta-eta-tau-N=2-diff4} holds for $\alpha=123$, $\beta=12^23$, $\delta=2^33^2$, and the root vectors satisfy \eqref{eq:diff-case2-hypothesis}; the scalars $\widetilde{q}_{\alpha\gamma}=q^{-1}$, $\widetilde{q}_{\beta\gamma}=q^{-2}$ satisfy \eqref{eq:Csj-Dsj-N=2-4}.

Also, \eqref{eq:alfa-beta-delta-eta-tau-N=2} holds for $\alpha=123$, $\beta=12^43^2$, $\delta=2^33^2$, $\eta=12$, $\tau=12^23$, and the root vectors satisfy \eqref{eq:diff-case3-hypothesis}; the scalars $\widetilde{q}_{\alpha\gamma}=q^{-1}$, $\widetilde{q}_{\beta\gamma}=q^{-1}$, $\widetilde{q}_{\tau\gamma}=q^{-2}$ satisfy \eqref{eq:Csj-Dsj-N=2-3}.

Let $(n_{\delta})_{\delta\in\varDelta_+^{\bq}}$ be a solution of \eqref{eq:equation-roots-cocycle-N=2}.
If $n_3=0$, then
\begin{align*}
N(\alpha_1+3\alpha_2+\alpha_3) = s_3(N\gamma) = \sum_{\delta\in\varDelta_+^{\bq}} f_{\delta}(n_{\delta}) s_3(\delta),
\end{align*}
and $s_3(\delta)\in\varDelta_+^{\bq}$ if $\delta\neq \alpha_3$. As $N_{\delta}=N_{s_3(\delta)}$, we have that $f_{\delta}=f_{s_3(\delta)}$, so we have a system as in \eqref{eq:equation-roots-cocycle-N=2} for $\alpha_1+3\alpha_2+\alpha_3$ in place of $\gamma$. The new system has five solutions, which gives place to the following solutions of the original system: 
\begin{itemize}
\item $n_{12^23}=n_{23}=1$, $n_{12^33^2}=N-1$, $n_{\delta}=0$ for all the other $\delta\in\varDelta_+^{\bq}$;
\item $n_{123}=n_{2^23}=1$, $n_{12^33^2}=N-1$, $n_{\delta}=0$ for all the other $\delta\in\varDelta_+^{\bq}$;
\item $n_{1}=n_{2^33^2}=1$, $n_{12^33^2}=N-1$, $n_{\delta}=0$ for all the other $\delta\in\varDelta_+^{\bq}$; 
\item $n_{123}=n_{2^33^2}=n_{12^23}=1$, $n_{12^33^2}=N-2$, $n_{\delta}=0$ for all the other $\delta\in\varDelta_+^{\bq}$; 
\item $n_{123}=2$, $n_{12^43^2}=n_{2^33^2}=1$, $n_{12^33^2}=N-3$, $n_{\delta}=0$ for all the other $\delta\in\varDelta_+^{\bq}$.
\end{itemize}
Next we assume $n_3\neq 0$. An analogous analysis as for the root $12^33$ shows that the unique solution is:
\begin{itemize}
\item $n_{12^33}=n_3=1$, $n_{12^33^2}=N-1$, $n_{\delta}=0$ for all the other $\delta\in\varDelta_+^{\bq}$.
\end{itemize}

Hence Proposition \ref{prop:cocycle-xgamma-N=2} applies and $(\x_{12^33^2}^L)^*$ is a cocycle.

\medbreak

\noindent $\circ$ For $\gamma=12^43^2$, we apply Proposition \ref{prop:cocycle-xgamma-N=2} again. The pairs $(\alpha,\beta)$ as in \eqref{eq:alfa-beta-N=2} are
$(12,2^33^2)$, $(12^33,23)$, $(12^23,2^23)$, $(123,2^33)$, $(12^33^2,2)$,
since the following relations hold:
\begin{align*}
x_{\alpha}x_{\beta}&=x_{\gamma}+q_{\alpha\beta} \, x_{\beta}x_{\alpha}, &
x_{\alpha}x_{\gamma}&=q_{\alpha} \, x_{\gamma}x_{\alpha}, &
x_{\gamma}x_{\beta}&=q_{\beta} \, x_{\beta}x_{\gamma}.
\end{align*}
As $-\frac{q_{\alpha\alpha}}{q_{\beta\beta}}\in \{q,q^2,q^3\}$, 
$L$ should be a multiple of $N$.

Let $L=\ord (-q)$. Now \eqref{eq:alfa-beta-delta-eta-tau-N=2-diff4} holds for $\alpha=12^33$, $\beta=12^23$, $\delta=2^33^2$, and the root vectors satisfy \eqref{eq:diff-case2-hypothesis}; the scalars $\widetilde{q}_{\alpha\gamma}=q^{-1}$, $\widetilde{q}_{\beta\gamma}=q^{-2}$ satisfy \eqref{eq:Csj-Dsj-N=2-4}.

Also, \eqref{eq:alfa-beta-delta-eta-tau-N=2} holds for $\alpha=12^33$, $\beta=12^33^2$, $\delta=2^33^2$, $\eta=1$, $\tau=12^23$, and the root vectors satisfy \eqref{eq:diff-case3-hypothesis}; the scalars $\widetilde{q}_{\alpha\gamma}=q^{-1}$, $\widetilde{q}_{\beta\gamma}=q^{-1}$, $\widetilde{q}_{\tau\gamma}=q^{-2}$ satisfy \eqref{eq:Csj-Dsj-N=2-3}.

Let $(n_{\delta})_{\delta\in\varDelta_+^{\bq}}$ be a solution of \eqref{eq:equation-roots-cocycle-N=2}.
If $n_2=0$, then
\begin{align*}
N(\alpha_1+3\alpha_2+2\alpha_3) = s_2(N\gamma) = \sum_{\delta\in\varDelta_+^{\bq}} f_{\delta}(n_{\delta}) s_2(\delta),
\end{align*}
and $s_2(\delta)\in\varDelta_+^{\bq}$ if $\delta\neq \alpha_2$. As $N_{\delta}=N_{s_2(\delta)}$, we have that $f_{\delta}=f_{s_2(\delta)}$, so we have a system as in \eqref{eq:equation-roots-cocycle-N=2} for $\alpha_1+3\alpha_2+2\alpha_3$ in place of $\gamma$. The new system has six solutions, which gives place to the following solutions of the original system: 
\begin{itemize}
\item $n_{12^23}=n_{2^23}=1$, $n_{12^43^2}=N-1$, $n_{\delta}=0$ for all the other $\delta\in\varDelta_+^{\bq}$;
\item $n_{12^33^2}=n_{2^23}=1$, $n_{12^43^2}=N-1$, $n_{\delta}=0$ for all the other $\delta\in\varDelta_+^{\bq}$;
\item $n_{12}=n_{2^33^2}=1$, $n_{12^43^2}=N-1$, $n_{\delta}=0$ for all the other $\delta\in\varDelta_+^{\bq}$;
\item $n_{123}=n_{2^33}=1$, $n_{12^43^2}=N-1$, $n_{\delta}=0$ for all the other $\delta\in\varDelta_+^{\bq}$; 
\item $n_{12^33}=n_{2^33^2}=n_{12^23}=1$, $n_{12^43^2}=N-2$, $n_{\delta}=0$ for all the other $\delta\in\varDelta_+^{\bq}$; 
\item $n_{12^33}=2$, $n_{12^33^2}=n_{2^33^2}=1$, $n_{12^43^2}=N-3$, $n_{\delta}=0$ for all the other $\delta\in\varDelta_+^{\bq}$.
\end{itemize}
Next we assume $n_2\neq 0$. An analogous analysis as for the root $12^33$ shows that the unique solution is:
\begin{itemize}
\item $n_{12^33^2}=n_2=1$, $n_{12^43^2}=N-1$, $n_{\delta}=0$ for all the other $\delta\in\varDelta_+^{\bq}$.
\end{itemize}

Hence Proposition \ref{prop:cocycle-xgamma-N=2} applies and $(\x_{12^43^2}^L)^*$ is a cocycle.
\epf

\section{Parametric modular types}\label{sec:discrete}

\subsection{Modular type \texorpdfstring{$\Bgl(4)$}{}}\label{subsec:type-bgl(4,alpha)}
Here $\theta = 4$, $q\neq \pm 1$. 
In this subsection, we deal with a Nichols algebra $\toba_{\bq}$ of diagonal type $\Bgl(4)$. We may assume that the corresponding diagram is
\begin{align}\label{eq:dynkin-bgl(4,alpha)}
&\xymatrix@C-4pt{\overset{q}{\underset{\ }{\circ}}\ar  @{-}[r]^{q ^{-1}}  & \overset{-1}{\underset{\ }{\circ}}
\ar  @{-}[r]^{-1}  & \overset{-1}{\underset{\ }{\circ}}
\ar  @{-}[r]^{-q}  & \overset{-q^{-1}}{\underset{\ }{\circ}}}.
\end{align}
We fix the following convex order on $\varDelta_+^{\bq}$:
\begin{align*}
&1, 12, 2, 12^23, 123, 23, 3, 
12^23^24, 123^24, 23^24, 12^234, 1234, 234, 34, 4.
\end{align*}
For more information, see \cite[\S 7.1]{AA17}. Let $M=\ord (-q)$: We may assume that $N\le M$. 
Note that 
\begin{align*}
N_{\delta} &= \begin{cases}
N &\text{if } \delta\in\{1,123^24,23^24\}, \\
M &\text{if } \delta\in\{4,12^234,12^23\}, \\
2 &\text{ otherwise}.
\end{cases}
\end{align*}

We prove Condition \ref{assumption:intro-combinatorial} for type $\Bgl(4)$:

\begin{prop}\label{prop:roots-cocycles-wk4}
For every  $\gamma \in \varDelta_+^{\bq}$, there exists $L_{\gamma}\in\N$ such that $(\x_{\gamma}^{L_\gamma})^*$ is a cocycle.
\end{prop}

\pf
We may assume that $\gamma$ has full support i.e.~$\gamma\in \{12^23^24, 1234, 123^24,  12^234\}$.

\medbreak
First we consider $\gamma = 12^23^24$. Here $N_{\gamma}=2$. 
The pairs $\alpha<\beta$ such that $\alpha+\beta=\gamma$ are:
\begin{align*}
&(3,12^234), & &(23,1234), & &(123,234), & 
&(12^23,34), & &(2,123^24), & &(12,23^24).
\end{align*}
For all pairs, $x_{\alpha}x_{\gamma} = q_{\alpha\gamma} x_{\gamma}x_{\alpha}$, $x_{\gamma}x_{\beta} = q_{\gamma\beta} x_{\beta}x_{\gamma}$. Also, there exist $\Bsj_i\in\Bbbk$ such that
\begin{align*}
[x_{3}, x_{12^234}]_c &= \Bsj_1 \, x_{\gamma}, &
[x_{123}, x_{234}]_c &= \Bsj_2 \, x_{\gamma} + \Bsj_3 \, x_{1234}x_{23},
\\
[x_{23}, x_{1234}]_c &= \Bsj_4 \, x_{\gamma}, &
[x_{12^23}, x_{34}]_c &= \Bsj_5 \, x_{\gamma} + \Bsj_6 \, x_3x_{12^234} + \Bsj_7 \, x_{23}x_{1234} + \Bsj_8 \, x_{234}x_{123},
\\
[x_{2}, x_{123^24}]_c &= \Bsj_9 \, x_{\gamma}, &
[x_{12}, x_{23^24}]_c &= \Bsj_{10} \, x_{\gamma} + \Bsj_{11} \, x_{123^24} x_2.
\end{align*}
Thus the root vectors satisfy \eqref{eq:diff2-hypothesis}, and  $-\frac{q_{\alpha\alpha}}{q_{\beta\beta}}\in \{-1, \pm q^{\pm1} \}$; hence we take $L=M$.

\medbreak

Next we check that $(23, 123^24, 12^234, 3)$, $(123, 23^24, 12^234, 3)$ and $(12^23, 123^24, 234, 123)$ are $4$-tuples $(\alpha, \beta, \delta, \eta)$ satisfying \eqref{eq:alfa-beta-delta-eta-N=2}. As $\widetilde{q}_{\alpha\gamma}=-1$, $\widetilde{q}_{\beta\gamma}=q$, we have that
$\coef{\alpha\beta\gamma}{L}=\coef{-1,q}{L}=0$ by Lemma
\ref{lemma:coef-roots-unity} \ref{item:coef-roots-unity-ii}.
Hence \ref{item:cocycle-xgamma-N=2-2} holds.

\medbreak

Finally we compute the solutions of \eqref{eq:equation-roots-cocycle-N=2}. That is, 
\begin{align}
\label{eq:equation-roots-cocycle-N=2-wk4-1}
\sum_{\delta: 1\in \supp\delta} f_{\delta}(n_{\delta}) &=M, \quad \sum_{\delta: 4\in \supp\delta} f_{\delta}(n_{\delta}) = M, 
\\ \label{eq:equation-roots-cocycle-N=2-wk4-2}
\sum_{\delta: 2\in \supp\delta} f_{\delta}(n_{\delta})a_2^{\delta}&=2M, \quad
\sum_{\delta: 3\in \supp\delta} f_{\delta}(n_{\delta}) a_3^{\delta} = 2M. 
\\ \label{eq:equation-roots-cocycle-N=2-wk4-3}
\sum_{\delta\in\varDelta_+^{\bq}} n_\delta &= M+1.
\end{align}
Let $(n_{\delta})$ be a solution of \eqref{eq:equation-roots-cocycle-N=2}. We claim that $n_{\delta}\le 1$ if $N_{\delta}\neq 2$. To prove it, first we note that $n_4, n_{12^23},n_{12^234} \le 2$ by \eqref{eq:equation-roots-cocycle-N=2-wk4-1}. If $n_4=2$, then $n_{\delta}=0$ if $4\in\supp\delta$, $\delta\neq 4$, so 
\begin{align*}
2M &= \sum_{\delta\in\varDelta_+^{\bq}} f_\delta(n_\delta) a_3^{\delta} = f_{12^23}(n_{12^23})+n_{123}+n_{23}+n_3 
\le f_{12^23}(2)+ \sum_{\delta\ne 4,12^23} n_{\delta}
\\
& \le M + (M-1)=2M-1,
\end{align*}
a contradiction. Hence $n_4\le 1$. Next we suppose that $n_{12^23}=2$. Then $n_{\delta}=0$ if either $1\in\supp\delta$ or else $2\in\supp\delta$, $\delta\neq 12^23$. From \eqref{eq:equation-roots-cocycle-N=2-wk4-3}, 
$n_3+n_{34}+n_4=M+1-n_{12^23}=M-1$, but we check directly that there are no solutions of \eqref{eq:equation-roots-cocycle-N=2-wk4-1} and \eqref{eq:equation-roots-cocycle-N=2-wk4-2} with these conditions. Hence $n_{12^23} \le 1$. Analogously, $n_{12^234} \le 1$.

Next we check that $n_1, n_{23^24},n_{123^24} \le 1$. The proof is analogous to the cases $n_4$, $n_{12^23}$, $n_{12^234}$ if $N=M$. Thus we assume that $N<M$: that is, $M=2N$, $N$ odd. By \eqref{eq:equation-roots-cocycle-N=2-wk4-1}, $n_1, n_{23^24},n_{123^24} \le 4$. We deal first with $n_{123^24}$.
\medbreak

\begin{itemize}[leftmargin=*]
\item Suppose that $n_{123^24}=4$. Then $n_{\delta}=0$ if $\supp \delta \cap \{1,3,4\}\neq \emptyset$ by \eqref{eq:equation-roots-cocycle-N=2-wk4-1} and \eqref{eq:equation-roots-cocycle-N=2-wk4-2}, and also $n_2=2N$. But from \eqref{eq:equation-roots-cocycle-N=2-wk4-3}, $n_2=2N-3$, a contradiction.
\medbreak

\item Suppose that $n_{123^24}=3$. Then $\sum_{\delta\ne 123^24} n_\delta = 2N-2$. By \eqref{eq:equation-roots-cocycle-N=2-wk4-2},
\begin{align*}
n_{12^23^24} &= f_{12^23^24}(n_{12^23^24})a_4^{12^23^24}\le N-1.
\end{align*}
By the first equality of \eqref{eq:equation-roots-cocycle-N=2-wk4-2} and the previous computations
\begin{align*}
3N-1 &= \sum_{\delta\ne 123^24} f_{\delta}(n_{\delta})a_2^{\delta}
= 2 f_{12^23^24}(n_{12^23^24})+ 2 f_{12^23}(n_{12^23}) +2 f_{12^234}(n_{12^234})
\\ 
&+ \sum_{\delta: a_2^{\delta}=1} n_{\delta} \le 2 n_{12^23^24} + 2+2+(2N-2-n_{12^23^24}) \le 3N-1.
\end{align*}
The equality holds if and only if $n_{12^23^24}=N-1$, $n_{12^23}= n_{12^234}=1$, but in this case the second equation of \eqref{eq:equation-roots-cocycle-N=2-wk4-2} does not hold.
\medbreak

\item Suppose that $n_{123^24}=2$. Then $\sum_{\delta\ne 123^24} n_\delta = 2N-1$, and by \eqref{eq:equation-roots-cocycle-N=2-wk4-2}, $n_{12^23^24}\le N$.
A similar computation as for the previous case shows that the equality $3N-1 = \sum_{\delta\ne 123^24} f_{\delta}(n_{\delta})a_2^{\delta}$ holds if and only if $n_{12^23^24}=N$, $n_{12^23}= n_{12^234}=1$, but again the second equation of \eqref{eq:equation-roots-cocycle-N=2-wk4-2} does not hold.
\medbreak
\end{itemize}
The same argument applies for $n_{23^24}$. Finally we check that $n_1\le 1$. If $n_1=4$, then $n_\delta=0$ for all $\delta\neq 4$ such that $4\in\supp \delta$, so
\begin{align*}
2M &= \sum_{\delta: 3\in \supp\delta} f_{\delta}(n_{\delta}) a_3^{\delta} = n_3 +n_{23} + n_{123} + n_{12^23} \le 
\sum_{\delta\ne 4} n_\delta \overset{\eqref{eq:equation-roots-cocycle-N=2-wk4-3}}{=} M-3,
\end{align*}
a contradiction. Now suppose that $2\le n_1 \le3$: by \eqref{eq:equation-roots-cocycle-N=2-wk4-1}, $n_{12^23^24}\le N$ and
\begin{align*}
4N &= \sum_{\delta\ne 123^24} f_{\delta}(n_{\delta})a_3^{\delta}
\le 2 n_{12^23^24} + 2+2+(2N-2-n_{12^23^24}) \le 3N+2,
\end{align*}
a contradiction.  Hence $n_{\delta}\le 1$ if $N_{\delta}\neq 2$, so $f_{\delta}(n_{\delta})=n_{\delta}$ for all $\delta\in\varDelta_{+}^{\bq}$. Then we look for $\gamma_i\in\varDelta_+^{\bq}$, $i\in\I_{M+1}$, such that $\sum_{i\in\I_{M+1}}\gamma_i=M\gamma$. As $a_3^{\delta}\le 2$ for all $\delta\in\varDelta_{+}^{\bq}$, there exist two possibilities up to permutations of these roots: 

\medbreak
\begin{itemize}[leftmargin=*]\renewcommand{\labelitemi}{$\circ$}
\item $a_3^{\gamma_i}=2$ for $\I_{M}$, $a_3^{\gamma_{M+1}}=0$. As $a_2^{\delta}\le 2$ for all $\delta\in\varDelta_{+}^{\bq}$, at least $M-1$ roots satisfy that $a_2^{\gamma_i}=2$, and we know that $a_2^{\gamma_{M+1}}\le 1$, so $\gamma_i=\gamma$ for $i\in\I_{M-1}$ up to permutation, and $\gamma_{M}+\gamma_{M+1}=\gamma$.

\medbreak
\item $a_3^{\gamma_i}=2$ for $\I_{M-1}$, $a_3^{\gamma_{M}}= a_3^{\gamma_{M+1}}=1$. Again at least $M-1$ roots satisfy that $a_2^{\gamma_i}=2$. If these roots are $\gamma_i$, $i\in\I_{M-1}$, then $\gamma_i=\gamma$ for $i\in\I_{M-1}$ up to permutation, and $\gamma_{M}+\gamma_{M+1}=\gamma$. Otherwise we may assume that $\gamma_i=\gamma$ for $i\in\I_{M-2}$, $a_2^{\gamma_{M-1}}=a_2^{\gamma_{M}}=1$, $a_2^{\gamma_{M+1}}=2$. We have three possibilities for $(\gamma_{M-1}, \gamma_{M}, \gamma_{M+1})$:
\begin{align*}
&(123^24, 23, 12^234), & &(23^24, 123, 12^234), & 
&(123^24, 234, 12^23).
\end{align*}
\end{itemize}

Hence all the hypotheses of Proposition \ref{prop:cocycle-xgamma-N=2} hold, and $(\x_{\gamma}^M)^*$ is a cocycle.

\medbreak
Next we consider $\gamma = 1234$.
Here $N_{\gamma}=2$. 
The pairs $\alpha<\beta$ such that $\alpha+\beta=\gamma$ are $(1,234)$, $(12,34)$, $(123,4)$.
For all pairs, $x_{\alpha}x_{\gamma} = q_{\alpha\gamma} x_{\gamma}x_{\alpha}$, $x_{\gamma}x_{\beta} = q_{\gamma\beta} x_{\beta}x_{\gamma}$ and there exist $\Bsj\in\Bbbk$ such that
$[x_{\alpha}, x_{\beta}]_c = \Bsj \, x_{\gamma}$. Thus the root vectors satisfy \eqref{eq:diff2-hypothesis} and  $-\frac{q_{\alpha\alpha}}{q_{\beta\beta}}\in \{-1, \pm q \}$; hence we take $L=M$. 

Next we check that $(12, 123^24, 4, 34)$, $(1, 123^24, 34, 234)$ and $(1, 12^23^24, 4, 234)$ are $4$-tuples $(\alpha, \beta, \delta, \eta)$ satisfying \eqref{eq:roots-N=2-case7}. As $(\widetilde{q}_{\delta\gamma},\widetilde{q}_{\beta\gamma})$ are respectively $(-q^{-1},-1)$, $(-1,q)$, $(-q^{-1},q)$, we have that
$\coef{-\delta\alpha\gamma}{L}=0$ by Lemma
\ref{lemma:coef-roots-unity} \ref{item:coef-roots-unity-ii}.
Hence \ref{item:cocycle-xgamma-N=2-7} holds.

Finally we check that $(1, 123^24, 12^234, 4, 234, 34)$ is a $6$-tuple $(\alpha, \beta, \delta, \eta, \mu, \nu)$ satisfying \eqref{eq:roots-N=2-case10}. As $\widetilde{q}_{\alpha\gamma}=q=\widetilde{q}_{\beta\gamma}$, $\widetilde{q}_{\nu\gamma}=-1$, we have that
$\coeff{\beta-\nu\alpha\gamma}{L}=0$ by Lemma
\ref{lemma:coef-roots-unity} \ref{item:coef-roots-unity-iii}.
Hence \ref{item:cocycle-xgamma-N=2-10} holds.

\medbreak

Next we look for solutions of \eqref{eq:equation-roots-cocycle-N=2}. That is, \eqref{eq:equation-roots-cocycle-N=2-wk4-3} and
\begin{align}
\label{eq:equation-roots-cocycle-N=2-wk4-4}
\sum_{\delta: 1\in \supp\delta} f_{\delta}(n_{\delta}) &=M, \quad \sum_{\delta: 4\in \supp\delta} f_{\delta}(n_{\delta}) = M, 
\\ \label{eq:equation-roots-cocycle-N=2-wk4-5}
\sum_{\delta: 2\in \supp\delta} f_{\delta}(n_{\delta})a_2^{\delta}&=M, \quad
\sum_{\delta: 3\in \supp\delta} f_{\delta}(n_{\delta}) a_3^{\delta} = M.
\end{align}
Let $(n_{\delta})$ be a solution of \eqref{eq:equation-roots-cocycle-N=2}. We claim that $n_{\delta}\le 1$ if $N_{\delta}\neq 2$. By \eqref{eq:equation-roots-cocycle-N=2-wk4-5},
\begin{align*}
M & \geq 2f_{12^23}(n_{12^23}), 2f_{12^234}(n_{12^234}).
\end{align*}
As $N_{12^23}=N_{12^234}=M$ we have that $n_{12^23}, n_{12^234} \le 1$. Now suppose that $n_4\ge 2$. By \eqref{eq:equation-roots-cocycle-N=2-wk4-4} we have that $n_4=2$. Then $n_{\delta}=0$ for all $\delta\neq 4$ such that $4\in\supp \delta$. By \eqref{eq:equation-roots-cocycle-N=2-wk4-5},
\begin{align*}
M = \sum_{\delta: 3\in \supp\delta} f_{\delta}(n_{\delta}) a_3^{\delta} = n_{123}+n_{23}+n_3 \le \sum_{\delta\ne 4} n_\delta =M-1,
\end{align*}
a contradiction. 

We also have that $n_{23^24}, n_{123^24} \le 1$ if $N=M$, and $n_{23^24}, n_{123^24} \le 2$ if $M=2N$. Suppose that $M=2N$ and  $n_{1^a23^24}=2$, $a\in\{0,1\}$. We have that $n_{\delta}=0$ for all $\delta$ such that $3\in\supp\delta$, $\delta\ne 1^a23^24$. By \eqref{eq:equation-roots-cocycle-N=2-wk4-4}, $f_4(n_4)=M$: that is, $n_4=2$, a contradiction.

Finally suppose that $n_1\ge 2$. By \eqref{eq:equation-roots-cocycle-N=2-wk4-4},
\begin{align*}
M= \sum_{\delta: 4\in \supp\delta} f_{\delta}(n_{\delta}) = \sum_{\delta: 4\in \supp\delta} n_{\delta} \le \sum_{\delta \ne 1} n_{\delta} \le M-1,
\end{align*}
a contradiction. Hence $n_{\delta}\le 1$ if $N_{\delta}\neq 2$, so $f_{\delta}(n_{\delta})=n_{\delta}$ for all $\delta\in\varDelta_{+}^{\bq}$. Then we look for $\gamma_i\in\varDelta_+^{\bq}$, $i\in\I_{M+1}$, such that $\sum_{i\in\I_{M+1}}\gamma_i=M\gamma$. As $a_1^{\delta}, a_4^{\delta}\le 1$ for all $\delta\in\varDelta_{+}^{\bq}$, there exist exactly $M$ roots such that $a_1^{\gamma_i}=1$, respectively $a_4^{\gamma_i}=1$. Thus there exist $M-1$ roots $\gamma_i$ such that $a_1^{\delta}=a_4^{\delta}=1$, which implies that $\supp \gamma_i= \{1,2,3,4\}$ for these roots. We may assume that $a_1^{\gamma_i}=a_4^{\gamma_i}=1$ for all $i\in\I_{M-1}$ and we have three possibilities up to permutations of the roots: 
either $a_1^{\gamma_M}=a_4^{\gamma_M}=1$, $a_1^{\gamma_{M+1}}=a_4^{\gamma_{M+1}}=0$, or else $a_1^{\gamma_M}=a_4^{\gamma_{M+1}}=1$, $a_1^{\gamma_{M+1}}=a_4^{\gamma_{M}}=0$. For the first case, $a_2^{\gamma_i}, a_3^{\gamma_i}\ge 1$ for all $i\in\I_{M}$ and we have a contradiction. Hence $a_1^{\gamma_M}=a_4^{\gamma_{M+1}}=1$, $a_1^{\gamma_{M+1}}=a_4^{\gamma_{M}}=0$. For $i\in\I_{M-1}$ we write $\gamma_i=12^{a_i}3^{b_i}4$, $a_i,b_i\in\I_2$. We also write 
$\gamma_M=12^{a_M}3^{b_M}$, $\gamma_{M+1}=2^{a_{M+1}}3^{b_{M+1}}4$.
At most one of the $a_i$, respectively $b_i$, is 2 for $i\in\I_{M-1}$. We analyze each case.

\medbreak
\begin{itemize}[leftmargin=*]\renewcommand{\labelitemi}{$\circ$}
\item $a_i=1=b_i=1$ for all $i\in\I_{M-1}$. Hence $\gamma_i=\gamma$ for all $i\in\I_{M-1}$ and $\gamma_M+\gamma_{M+1}=\gamma$.

\medbreak
\item $a_i=1$ for all $i\in\I_{M-1}$, $b_i=1$ for all $i\in\I_{M-2}$, $b_{M-1}=2$. Here $b_M=b_{M+1}=0$,  so 
\begin{align*}
\gamma_i&=\gamma\text{ for all }i\in\I_{M-2}, & \gamma_{M-1}&=123^24, & \gamma_M&=12, & \gamma_{M+1}&=4.
\end{align*}

\medbreak
\item $a_i=1$ for all $i\in\I_{M-2}$, $b_i=1$ for all $i\in\I_{M-1}$, $a_{M-1}=2$. Here $a_M=a_{M+1}=0$, so 
\begin{align*}
\gamma_i&=\gamma\text{ for all }i\in\I_{M-2}, & \gamma_{M-1}&=12^234, & \gamma_M&=1, & \gamma_{M+1}&=34.
\end{align*}

\medbreak
\item $a_i=b_i=1$ for all $i\in\I_{M-2}$, $a_{M-1}=b_{M-1}=2$. Here $a_M=b_M=a_{M+1}=b_{M+1}=0$, so 
\begin{align*}
\gamma_i&=\gamma\text{ for all }i\in\I_{M-2}, & \gamma_{M-1}&=12^23^24, & \gamma_M&=1, & \gamma_{M+1}&=4.
\end{align*}

\medbreak
\item $a_i=1$ for all $i\in\I_{M-2}$, $b_i=1$ for all $i\in\I_{M-1}-\{M-2\}$, $a_{M-1}=b_{M-2}=2$. Here
\begin{align*}
\gamma_i&=\gamma\text{ for all }i\in\I_{M-3}, & \gamma_{M-2}&=123^24, & \gamma_{M-1}&=12^234, &
\gamma_M&=1, & \gamma_{M+1}&=4.
\end{align*}
\end{itemize}

Hence all the hypotheses of Proposition \ref{prop:cocycle-xgamma-N=2} hold, and $(\x_{\gamma}^M)^*$ is a cocycle.

\medbreak
Finally we consider $\gamma = 123^24, 12^234$. We have $P_{\gamma}=2$, $Q_{\gamma}=1$, so $P_{\gamma},Q_{\gamma}<N_{\gamma}$. By Lemma \ref{lem:largeN} $(\x_{\gamma}^{N_\gamma})^*$ is a $2$-cocycle.
\epf

\medbreak
\subsection{Modular type  \texorpdfstring{$\Brown(2)$}{}}\label{subsec:type-br(2,a)}
Here $\theta = 2$, $\zeta \in \G_3$, $q\notin \G_3$.
In this subsection, we deal with a Nichols algebra $\toba_{\bq}$ of modular type $\Brown(2)$, that is associated to
any of the Dynkin diagrams
\begin{align}\label{eq:dynkin-br(2,a)}
&\xymatrix{a_1 \ar  @{|->}[r]  &  \overset{\zeta}{\underset{\ }{\circ}} \ar  @{-}[r]^{q^{-1}}  &
\overset{q}{\underset{\ }{\circ}}},
& &\xymatrix{a_2 \ar  @{|->}[r]  &  \overset{\zeta}{\underset{\ }{\circ}} \ar  @{-}[r]^{\zeta^2q}  &
\overset{\zeta q^{-1}}{\underset{\ }{\circ}}}.
\end{align}
For more information, see \cite[\S 7.2]{AA17}.
Since (\ref{eq:dynkin-br(2,a)} b)
has the same shape as (\ref{eq:dynkin-br(2,a)} a) but with $\zeta q^{-1}$
instead of $q$, we just discuss  the latter. Essentially this is very similar to standard $B_2$.
The corresponding set of positive roots with full support is
\begin{align*}
& \{2\alpha_1+\alpha_2,  \alpha_1+\alpha_2\}.
\end{align*}

\begin{table}[ht]
\caption{The roots with full support of $\Brown(2)$; $\gamma_1 < \gamma_2$, $(N_{\gamma}-1)\gamma = \gamma_1 + \gamma_2$} \label{tab:br2}
\begin{center}

\begin{tabular}{|c|c|c|c|c c|c|}
\hline
$\gamma$& $N_{\gamma}$ & $P_{\gamma}$ & $Q_{\gamma}$ & $\gamma_1$ & $\gamma_2$ & $L_{\gamma}$ \\ \hline
$1^22$ & $M$ & 2 & 1 & $1$ & $12$ & $2$ \\ \hline
$12$ & $3$ & 3 & 2 & $1^22$ & $2$ & $\ord q^3$ \\ \hline
\end{tabular}
\end{center}
\end{table}

Let $M=\ord (\zeta q^{-1})$. We order the root vectors:
$x_1 < x_{112} < x_{12} < x_2$.

\medbreak
We prove Condition \ref{assumption:intro-combinatorial} for type $\Brown(2)$.

\begin{prop}\label{prop:roots-cocycles-br2}
For every $\gamma  \in \varDelta_+^{\bq}$,  $(\x_{\gamma}^{N_\gamma})^*$ is a $2$-cocycle.
\end{prop}

\pf
As before we just consider non-simple roots, i.e.~with full support.

\medbreak
\noindent $\circ$ For $\gamma=1^22$, 
the case $N_{112}>2$ follows by Lemma \ref{lem:largeN}. Assume now that $N_{112}=2$. We will apply Proposition \ref{prop:cocycle-xgamma-N=2}. The unique pair as in  \eqref{eq:alfa-beta-N=2} is $\alpha=\alpha_1$, $\beta=\alpha_1+\alpha_2$, since the following relations hold:
\begin{align*}
x_{1}x_{12}&=x_{112}+\zeta q_{12} \, x_{12}x_{1}, &
x_{1}x_{112}&=\zeta^2q_{12} \, x_{112}x_{1}, &
x_{112}x_{12}&=\zeta^2 q_{12} \, x_{12}x_{112}.
\end{align*}
As $-\frac{q_{\alpha\alpha}}{q_{\beta\beta}}=-1$, we take $L=2$.
The unique solution of \eqref{eq:equation-roots-cocycle-N=2} is
$n_{112}=n_{1}=n_{12}=1$, and $n_{2}=0$.
Hence Proposition \ref{prop:cocycle-xgamma-N=2} applies and $(\x_{112}^2)^*$ is a 2-cocycle.

\medbreak
\noindent $\circ$ For $\gamma=12$, we will apply Proposition \ref{prop:cocycle-xgamma-N>2}. The unique pair as in \eqref{eq:alfa-beta} is $\alpha=2\alpha_1+\alpha_2$, $\beta=\alpha_2$, since the following relations hold:
\begin{align*}
x_{112}x_{2}&= (q-\zeta)q_{12} \, x_{12}^2 
+qq_{12}^2 \, x_{2}x_{112}, &
x_{112}x_{12}&=\zeta^2 q_{12} \, x_{12}x_{112}, 
\\ & &
x_{12}x_{2}&=qq_{12} \, x_{2}x_{12}.
\end{align*}
In this case, $\frac{q_{\alpha\gamma}}{q_{\gamma\beta}}=\zeta^2 q^{-1}$: as $q\neq \zeta^2$, we take $L=\ord q^3$. The unique solution of \eqref{eq:equation-roots-cocycle-N>2} is $n_{12}=L-1$, $n_{12}=n_{2}=1$, and $n_{2}=0$, so Proposition \ref{prop:cocycle-xgamma-N>2} applies and $(\x_{12}^L)^*$ is a $2L$-cocycle.
\epf


\section{Proofs of the Computational Lemmas}\label{sec:computational-lemmas}

\subsection{}

Given $\gamma\in\varDelta_+$, let $g_{\gamma}:\N_0\to\N_0$ be the function 
\begin{align}\label{eq:g-gamma-definition}
g_{\gamma}(n) := f_{\gamma}(n)-f_{\gamma}(n-1)= \begin{cases} 1, & n \text{ odd}, \\ N_{\gamma}-1 & n \text{ even}.\end{cases}
\end{align}

\begin{remark}\label{rem:differential-q-commute}
Let $\beta_1<\beta_2<\beta_3$ be positive roots such that the corresponding root vectors $q$-commute:
\begin{align*}
x_{\beta_i} x_{\beta_j} &= q_{\beta_i\beta_j} \, x_{\beta_j}x_{\beta_i},  &
& \text{for all }i<j.
\end{align*}
For each $n\in\N$,
\begin{align}\label{eq:diff-commuting roots-1}
d(\x_{\beta_1} \x_{\beta_2}^{f_{\beta_2}(n)} \ot 1 ) &= \x_{\beta_1} \x_{\beta_2}^{f_{\beta_2}(n-1)} \ot x_{\beta_2}^{g_{\beta_2}(n)} +(-1)^n q_{\beta_1\beta_2}^{f_{\beta_2}(n)} \x_{\beta_2}^{f_{\beta_2}(n)} \ot x_{\beta_1},
\\ \label{eq:diff-commuting roots-2}
d(\x_{\beta_1}^{f_{\beta_1}(n)} \x_{\beta_2}\ot 1 ) &=   \x_{\beta_1}^{f_{\beta_1}(n)} \ot x_{\beta_2} - q_{\beta_1\beta_2}^{g_{\beta_1}(n)}
\x_{\beta_1}^{f_{\beta_1}(n-1)} \x_{\beta_2} \ot x_{\beta_1}^{g_{\beta_1}(n)}
\\ \label{eq:diff-commuting roots-3}
d( \x_{\beta_1}\x_{\beta_2}^{f_{\beta_2}(n)}\x_{\beta_3} \ot 1 ) & =\x_{\beta_1}\x_{\beta_2}^{f_{\beta_2}(n)}\otimes x_{\beta_3} -q_{\beta_2\beta_3}^{g_{\beta_2}(n)} \x_{\beta_1}\x_{\beta_2}^{f_{\beta_2}(n-1)}\x_{\beta_3} \ot x_{\beta_2}^{g_{\beta_2}(n)} 
\\ \nonumber & \quad 
- (-1)^n q_{\beta_1 \beta_3} q_{\beta_1\beta_2}^{f_{\beta_2}(n)} \x_{\beta_2}^{f_{\beta_2}(n)}\x_{\beta_3} \otimes x_{\beta_1},
\\ \label{eq:diff-commuting roots-4}
d(\x_{\beta_1}^{f_{\beta_1}(n)} \x_{\beta_2} \x_{\beta_3}\ot 1)  &= \x_{\beta_1}^{f_{\beta_1}(n)} \x_{\beta_2} \ot x_{\beta_3} -q_{\beta_2\beta_3}\x_{\beta_1}^{f_{\beta_1}(n)} \x_{\beta_3}\ot x_{\beta_2}
\\ \nonumber & \quad 
+q_{\beta_1\beta_2}^{g_{\beta_1}(n)}q_{\beta_1\beta_3}^{g_{\beta_1}(n)} \x_{\beta_1}^{f_{\beta_1}(n-1)} \x_{\beta_2} \x_{\beta_3}\ot x_{\beta_1}^{g_{\beta_1}(n)}.
\end{align}

In fact, the root vectors $q$-commute by hypothesis so we can compute the differentials as in the proof of Proposition \ref{prop:qci}.
\end{remark}

The next results will allow us to identify some  cocycles
of degree higher  than~2.

\begin{lema}\label{lem:diff2}
Let $n\in\N_0$. Let $\alpha < \delta_1, \dots ,\delta_n <\gamma < \eta_1,\dots ,\eta_n < \beta$ be positive roots.
Assume that  the relations among the corresponding root
vectors take the form
\begin{align}\label{eq:diff2-hypothesis}
\begin{aligned}
x_{\alpha}x_{\beta} &= q_{\alpha \beta} x_{\beta}x_{\alpha} +\Bsj x_{\gamma}^{N_{\gamma}-1} + \sum_{j=1}^{n} \Bsj_j x_{\eta_j}x_{\delta_j}, \\
x_{\alpha}x_{\gamma} &= q_{\alpha\gamma}  x_{\gamma}x_{\alpha}, &
x_{\delta_j}x_{\gamma} &= q_{\delta_j\gamma}  x_{\gamma}x_{\delta_j}, \\
x_{\gamma}x_{\beta} &= q_{\gamma\beta} x_{\beta}x_{\gamma}, &
x_{\gamma}x_{\eta_j} &= q_{\gamma\eta_j} x_{\eta_j}x_{\gamma},
\end{aligned}
\end{align}
for some scalars $\Bsj,\Bsj_1,\dots,\Bsj_n$.
Then, for all $a\geq 1$, $d(\x_{\alpha} \x_{\gamma}^{aN_{\gamma}} \x_{\beta}\ot 1)  =$
\begin{align*}
&\x_{\alpha} \x_{\gamma}^{aN_{\gamma}}
\ot x_{\beta} - \Dsj^{N_{\gamma}-1} \x_{\alpha} \x_{\gamma}^{N_{\gamma}(a-1)+1} \x_{\beta}\ot x_{\gamma}^{N_{\gamma}-1} 
\\
& - \Asj\Csj^{aN_{\gamma}} \x_{\gamma}^{aN_{\gamma}} \x_{\beta}\ot x_{\alpha} - \sum_{j=1}^{n} \Bsj_j q_{\alpha \gamma} ^{aN_{\gamma}} \x_{\gamma}^{aN_{\gamma}}\x_{\eta_j} \ot x_{\delta_j}
\\ &+ \Bsj\Dsj^{aN_{\gamma}}
\left\{ \left(\frac{\Csj}{\Dsj}-1 \right) (a)_{(\frac{\Csj}{\Dsj})^{N_{\gamma}}}
- \left(\frac{\Csj}{\Dsj}\right)^{aN_{\gamma}}\right\}  \x_{\gamma}^{ aN_{\gamma}+1}\ot x_{\gamma}^{N_{\gamma}-2};
\end{align*}
and for all $a\geq 0$, $d(\x_{\alpha} \x_{\gamma}^{aN_{\gamma}+1} \x_{\beta}\ot 1)  =$
\begin{align*}
& \x_{\alpha}\x_{\gamma}^{aN_{\gamma}+1}
\ot x_{\beta} - \Dsj\x_{\alpha} \x_{\gamma}^{aN_{\gamma}} \x_{\beta}\ot x_{\gamma} 
+ \Asj\Csj^{aN_{\gamma}+1} \x_{\gamma}^{aN_{\gamma}+1} \x_{\beta} \ot x_{\alpha}  
\\
& 
+ \sum_{j=1}^{n} \Bsj_j q_{\alpha \gamma} ^{aN_{\gamma}+1} \x_{\gamma}^{aN_{\gamma}+1}\x_{\eta_j} \ot x_{\delta_j}
+ \Bsj\Dsj^{aN_{\gamma}+1} \left(\frac{\Csj}{\Dsj} -1\right) (a+1)_{\left(\frac{\Csj}{\Dsj}\right)^{N_{\gamma}}}
\x_{\gamma}^{N_{\gamma}(a+1)} \ot 1 .
\end{align*}
\end{lema}

\bigbreak
Notice that the first equality in \eqref{eq:diff2-hypothesis} forces
\begin{align}\label{eq:diff2-hypothesis-roots}
(N_{\gamma} - 1)\gamma&=\alpha+\beta=\delta_j+\eta_j & &\text{for all }j\in\I_n.
\end{align}

\begin{proof}
We need the following computation:
\begin{align*}
d(\x_{\alpha}\x_{\beta} \ot 1) &= \x_{\alpha}\ot x_{\beta} - q_{\alpha \beta} \x_{\beta} \ot x_{\alpha}
- \Bsj \x_{\gamma} \ot x_{\gamma}^{N_{\gamma}-2}
- \sum_{j=1}^{n} \Bsj_j \x_{\eta_j} \ot x_{\delta_j}
\end{align*}

The proof of the lemma is by induction on $a$.
First we compute: 
\begin{align*}
d&( \x_{\alpha}  \x_{\gamma}\x_{\beta}\ot 1) = 
\x_{\alpha}\x_{\gamma}\ot x_{\beta} - s(d(\x_{\alpha}\x_{\gamma}\ot 1) x_{\beta})\\
&= \x_{\alpha}\x_{\gamma}\ot x_{\beta} - s(
\x_{\alpha}\ot x_{\gamma}x_{\beta} - q_{\alpha \gamma}\x_{\gamma}\ot x_{\alpha}
x_{\beta})
\\
&= \x_{\alpha} \x_{\gamma}\ot x_{\beta} - s\Big(q_{\gamma\beta} \x_{\alpha}\ot  x_{\beta}x_{\gamma}
- q_{\alpha \gamma}  \x_{\gamma}\ot \big( q_{\alpha \beta} x_{\beta}x_{\alpha} +\Bsj x_{\gamma}^{N_{\gamma}-1} + \sum_{j=1}^{n} \Bsj_j x_{\eta_j}x_{\delta_j} \big) \Big)
\\
&= \x_{\alpha} \x_{\gamma}\ot x_{\beta} -q_{\gamma\beta} \x_{\alpha} \x_{\beta}\ot  x_{\gamma} + s\Big( q_{\alpha \gamma}q_{\alpha \beta}  \big(\x_{\gamma}\ot  x_{\beta} - q_{\gamma\beta}
\x_{\beta} \ot x_{\gamma} \big)x_{\alpha}
\\ & \,
+ \Bsj (q_{\alpha \gamma} - q_{\gamma\beta}) \x_{\gamma}\ot  x_{\gamma}^{N_{\gamma}-1}
+ \sum_{j=1}^{n} \Bsj_j  \big( q_{\alpha \gamma}  \x_{\gamma}\ot x_{\eta_j}
- q_{\gamma\beta} q_{\delta_j \gamma} \x_{\eta_j} \ot x_{\gamma} \big) x_{\delta_j}
\Big)
\\
&= \x_{\alpha} \x_{\gamma}\ot x_{\beta} -q_{\gamma\beta} \x_{\alpha} \x_{\beta}\ot  x_{\gamma} + 
\sum_{j=1}^{n} \Bsj_j q_{\alpha \gamma}  \x_{\gamma}\x_{\eta_j} \ot x_{\delta_j}
+ q_{\alpha \gamma}q_{\alpha \beta}\x_{\gamma} \x_{\beta}\ot x_{\alpha}
\\ & \,
+ \Bsj (q_{\alpha \gamma} - q_{\gamma\beta}) \x_{\gamma}^{N_{\gamma}}\ot 1
+ s\Big( \sum_{j=1}^{n} \Bsj_j  \big( q_{\alpha \gamma} q_{\gamma\eta_j}
- q_{\gamma\beta} q_{\delta_j \gamma} \big)  s(1\ot x_{\eta_j} x_{\gamma} x_{\delta_j}) \Big),
\end{align*}
which agrees with the second formula for $a=0$ since $s\circ s=0$.
Next we compute
\begin{align*}
& d(\x_{\alpha}  \x_{\gamma}^{N_{\gamma}}\x_{\beta}\ot 1) = 
\x_{\alpha}\x_{\gamma}^{N_{\gamma}}\ot x_{\beta} - s(d(\x_{\alpha}\x_{\gamma}^{N_{\gamma}}\ot 1) x_{\beta})
\\
&= 
\x_{\alpha}\x_{\gamma}^{N_{\gamma}}\ot x_{\beta} - s \Big(
q_{\gamma \beta}^{N_{\gamma}-1}\x_{\alpha}\x_{\gamma}
\ot x_{\beta}x_{\gamma}^{N_{\gamma}-1} 
+ q_{\alpha \beta} q_{\alpha \gamma}^{N_{\gamma}} \x_{\gamma}^{N_{\gamma}}\ot   x_{\beta}x_{\alpha} 
\\ & \quad
+ q_{\alpha \gamma}^{N_{\gamma}}\Bsj \x_{\gamma}^{N_{\gamma}}\ot  x_{\gamma}^{N_{\gamma}-1}
+ \sum_{j=1}^{n} \Bsj_j q_{\alpha \gamma}^{N_{\gamma}} \x_{\gamma}^{N_{\gamma}}\ot x_{\eta_j}x_{\delta_j}
\Big)
\\
&= 
\x_{\alpha}\x_{\gamma}^{N_{\gamma}}\ot x_{\beta} -q_{\gamma \beta}^{N_{\gamma}-1}\x_{\alpha}\x_{\gamma} \x_{\beta}
\ot x_{\gamma}^{N_{\gamma}-1} 
\\ & \quad 
- \Bsj \left(q_{\alpha \gamma}^{N_{\gamma}} - q_{\gamma \beta}^{N_{\gamma}-1}q_{\alpha \gamma} +q_{\gamma \beta}^{N_{\gamma}} \right) \x_{\gamma}^{N_{\gamma}+1}\ot x_{\gamma}^{N_{\gamma}-2}
- \sum_{j=1}^{n} \Bsj_j q_{\alpha \gamma}^{N_{\gamma}} \x_{\gamma}^{N_{\gamma}} \x_{\eta_j}\ot x_{\delta_j}
\\ & \quad - q_{\alpha \beta} q_{\alpha \gamma}^{N_{\gamma}} \x_{\gamma}^{N_{\gamma}}\x_{\beta}\ot x_{\alpha},
\end{align*}
and by~(\ref{eq:diff2-hypothesis-roots}), this
agrees with the first formula in the lemma when $a=1$. 
Now assume the second formula given in the lemma holds when $a$ is replaced by $a-1$. Then
\begin{align*}
d(\x_{\alpha} & \x_{\gamma}^{aN_{\gamma}} \x_{\beta}\ot 1)  
= \x_{\alpha} \x_{\gamma}^{aN_{\gamma}}\ot x_{\beta} - s ( d(\x_{\alpha} \x_{\gamma}^{aN_{\gamma}}\ot 1)
x_{\beta}) \\
& = \x_{\alpha} \x_{\gamma}^{aN_{\gamma}} \ot x_{\beta} - s \big(\x_{\alpha} \x_{\gamma}^{N_{\gamma}(a-1)+1}
\ot x_{\gamma}^{N_{\gamma}-1} x_{\beta} + \Csj ^{aN_{\gamma}} \x_{\gamma}^{aN_{\gamma}}\ot x_{\alpha}x_{\beta} \big)
\\
& = \x_{\alpha} \x_{\gamma}^{aN_{\gamma}}\ot x_{\beta} - s \Big( \Dsj^{N_{\gamma}-1} \x_{\alpha} 
\x_{\gamma}^{N_{\gamma}(a-1)+1} \ot x_{\beta} x_{\gamma}^{N_{\gamma}-1} 
+ \Asj\Csj^{aN_{\gamma}} \x_{\gamma}^{aN_{\gamma}}
\ot x_{\beta}x_{\alpha}
\\ & \quad   + \Bsj \Csj ^{aN_{\gamma}} \x_{\gamma}^{aN_{\gamma}}\ot x_{\gamma}^{N_{\gamma}-1} 
+ \sum_{j=1}^{n} \Bsj_j \Csj ^{aN_{\gamma}} \x_{\gamma}^{aN_{\gamma}}\ot x_{\eta_j} x_{\delta_j} \Big).
\end{align*}

Use the induction hypothesis to rewrite
the term $\Dsj ^{N_{\gamma}-1} \x_{\alpha} \x_{\gamma}^{N_{\gamma}(a-1)+1}\ot x_{\beta}x_{\gamma}^{N_{\gamma}-1}$
to obtain
\begin{align*}
d&(\x_{\alpha} \x_{\gamma}^{aN_{\gamma}} \x_{\beta}\ot 1)  
= \x_{\alpha} \x_{\gamma}^{aN_{\gamma}}\ot x_{\beta} - s \Big( \Dsj^{N_{\gamma}-1} d(\x_{\alpha} \x_{\gamma}^{(a-1)N_{\gamma}+1} \x_{\beta}\ot x_{\gamma}^{N_{\gamma}-1})   
\\ & \quad - \Asj\Csj^{(a-1)N_{\gamma}+1}\Dsj^{N_{\gamma}-1} \x_{\gamma}^{(a-1)N_{\gamma}+1} \x_{\beta} \ot x_{\alpha}x_{\gamma}^{N_{\gamma}-1} 
\\ & \quad - \Bsj\Dsj^{aN_{\gamma}} \left(\frac{\Csj}{\Dsj} -1\right) (a)_{\left(\frac{\Csj}{\Dsj}\right)^{N_{\gamma}}}
\x_{\gamma}^{aN_{\gamma}} \ot x_{\gamma}^{N_{\gamma}-1}
\\ & \quad -\sum_{j=1}^{n} \Bsj_j \Dsj^{N_{\gamma}-1} q_{\alpha \gamma} ^{(a-1)N_{\gamma}+1} \x_{\gamma}^{(a-1)N_{\gamma}+1}\x_{\eta_j} \ot x_{\delta_j}x_{\gamma}^{N_{\gamma}-1} 
+ \Asj\Csj^{aN_{\gamma}} \x_{\gamma}^{aN_{\gamma}}
\ot x_{\beta}x_{\alpha}
\\ & \quad   + \Bsj \Csj ^{aN_{\gamma}} \x_{\gamma}^{aN_{\gamma}}\ot x_{\gamma}^{N_{\gamma}-1} 
+ \sum_{j=1}^{n} \Bsj_j \Csj ^{aN_{\gamma}} \x_{\gamma}^{aN_{\gamma}}\ot x_{\eta_j} x_{\delta_j} \Big)
\\
& = \x_{\alpha} \x_{\gamma}^{aN_{\gamma}}\ot x_{\beta} - 
\Dsj^{N_{\gamma}-1} \x_{\alpha} \x_{\gamma}^{(a-1)N_{\gamma}+1} \x_{\beta}\ot x_{\gamma}^{N_{\gamma}-1}  
\\ & \quad +
s \Big( \Asj\Csj^{aN_{\gamma}}\Dsj^{N_{\gamma}-1} \x_{\gamma}^{(a-1)N_{\gamma}+1} \x_{\beta} \ot x_{\gamma}^{N_{\gamma}-1} x_{\alpha}
\\ & \quad  + \Bsj\Dsj^{aN_{\gamma}} \left( \left(\frac{\Csj}{\Dsj} -1\right) (a)_{\left(\frac{\Csj}{\Dsj}\right)^{N_{\gamma}}}
- \frac{\Csj ^{aN_{\gamma}}}{\Dsj^{aN_{\gamma}}} \right)
\x_{\gamma}^{aN_{\gamma}} \ot x_{\gamma}^{N_{\gamma}-1} - \Asj\Csj^{aN_{\gamma}} \x_{\gamma}^{aN_{\gamma}}
\ot x_{\beta}x_{\alpha}
\\ & \quad  +\sum_{j=1}^{n} \Bsj_j \Dsj^{N_{\gamma}-1} q_{\alpha \gamma} ^{(a-1)N_{\gamma}+1} q_{\delta_j\gamma}^{N_{\gamma}-1}
\x_{\gamma}^{(a-1)N_{\gamma}+1}\x_{\eta_j} \ot x_{\gamma}^{N_{\gamma}-1} x_{\delta_j} - \sum_{j=1}^{n} \Bsj_j \Csj ^{aN_{\gamma}} \x_{\gamma}^{aN_{\gamma}}\ot x_{\eta_j} x_{\delta_j} \Big).
\end{align*}
Now use the formula $d(\x_{\gamma}^{aN_{\gamma}+1}\ot 1) = x_{\gamma}^{aN_{\gamma}}\ot x_{\gamma}$,
\eqref{eq:diff-commuting roots-2}
and~(\ref{eq:diff2-hypothesis-roots})  to rewrite
the above expression as
\begin{align*}
&  \x_{\alpha} \x_{\gamma}^{aN_{\gamma}} \ot x_{\beta} - \Dsj ^{N_{\gamma}-1} \x_{\alpha} \x_{\gamma}^{N_{\gamma}(a-1)+1} \x_{\beta}\ot x_{\gamma}^{N_{\gamma}-1} 
- \Asj\Csj^{aN_{\gamma}} \x_{\gamma}^{aN_{\gamma}} \x_{\beta}\ot x_{\alpha}
\\
& \quad  + \Bsj \Dsj ^{aN_{\gamma}} \{ (\frac{\Csj}{\Dsj} -1)
(a)_{(\frac{\Csj}{\Dsj})^{N_{\gamma}}} - (\frac{\Csj}{\Dsj})^{aN_{\gamma}} \} \x_{\gamma}^{aN_{\gamma}+1}\ot x_{\gamma}^{N_{\gamma}-2} 
- \sum_{j=1}^{n} \Bsj_j \Csj ^{aN_{\gamma}} \x_{\gamma}^{aN_{\gamma}} \x_{\eta_j}\ot  x_{\delta_j}.
\end{align*}
This agrees with the first claimed formula in the lemma. 

Now we use the first formula to obtain the second formula: 
\begin{align*}
d(\x_{\alpha} \x_{\gamma}^{aN_{\gamma}+1} & \x_{\beta}\ot 1) 
= \x_{\alpha} \x_{\gamma}^{aN_{\gamma}+1} \ot x_{\beta} - s ( d(\x_{\alpha} \x_{\gamma}^{aN_{\gamma}+1}
\ot 1) x_{\beta}) \\
& = \x_{\alpha} \x_{\gamma}^{aN_{\gamma}+1} \ot x_{\beta} - s (\x_{\alpha} \x_{\gamma}^{aN_{\gamma}}
\ot x_{\gamma} x_{\beta} - \Csj ^{aN_{\gamma}+1} \x_{\gamma}^{aN_{\gamma}+1} \ot x_{\alpha}x_{\beta} ) \\
& = \x_{\alpha} \x_{\gamma}^{aN_{\gamma}+1} \ot x_{\beta} - s \Big( \Dsj  \x_{\alpha}
\x_{\gamma}^{aN_{\gamma}}\ot x_{\beta}x_{\gamma} - \Asj\Csj^{aN_{\gamma}+1} \x_{\gamma}^{aN_{\gamma}+1} \ot x_{\beta}
x_{\alpha}
\\ & \quad  - \Bsj \Csj ^{aN_{\gamma}+1} \x_{\gamma}^{aN_{\gamma}+1} \ot x_{\gamma}^{N_{\gamma}-1} 
- \sum_{j=1}^{n} \Bsj_j \Csj ^{aN_{\gamma}+1} \x_{\gamma}^{aN_{\gamma}+1} \ot x_{\eta_j} x_{\delta_j} \Big).
\end{align*}

By our induction hypothesis, we may  use the first formula in the statement 
of the lemma to rewrite the term $\Dsj  \x_{\alpha} \x_{\gamma}^{aN_{\gamma}}\ot
x_{\beta} x_{\gamma}$, obtaining
\begin{align*}
& \x_{\alpha} \x_{\gamma}^{aN_{\gamma}+1} \ot x_{\beta} - s \Big( \Dsj  d(\x_{\alpha}
\x_{\gamma}^{aN_{\gamma}}  \x_{\beta}\ot x_{\gamma})  + \Asj\Csj^{aN_{\gamma}} \Dsj  \x_{\gamma}^{aN_{\gamma}} \x_{\beta}
\ot x_{\alpha} x_{\gamma} \\
& \quad - \Bsj\Dsj^{aN_{\gamma}+1} \{ (\frac{\Csj}{\Dsj}-1) (a)_{(\frac{\Csj}{\Dsj})^{N_{\gamma}}}
- (\frac{\Csj}{\Dsj})^{aN_{\gamma}}\} \x_{\gamma}^{aN_{\gamma}+1}\ot x_{\gamma}^{N_{\gamma}-1} 
\\
& \quad  + \sum_{j=1}^{n} \Bsj_j \Csj ^{aN_{\gamma}} \Dsj \x_{\gamma}^{aN_{\gamma}} \x_{\eta_j}\ot  x_{\delta_j}x_{\gamma}
-\Asj\Csj^{aN_{\gamma}+1} \x_{\gamma}^{aN_{\gamma}+1}\ot x_{\beta}x_{\alpha}
\\ & \quad  - \Bsj \Csj ^{aN_{\gamma}+1}
\x_{\gamma}^{aN_{\gamma}+1} \ot x_{\gamma}^{N_{\gamma}-1} - \sum_{j=1}^{n} \Bsj_j \Csj ^{aN_{\gamma}+1} \x_{\gamma}^{aN_{\gamma}+1} \ot x_{\eta_j} x_{\delta_j} \Big)
\\
& = \x_{\alpha} \x_{\gamma}^{aN_{\gamma}+1}\ot x_{\beta} 
- \Dsj  \x_{\alpha} \x_{\gamma}^{aN_{\gamma}} \x_{\beta}
\ot x_{\gamma} 
\\ & \quad - s \Big( \Asj\Csj^{aN_{\gamma}+1} (\Dsj  \x_{\gamma}^{aN_{\gamma}}\x_{\beta} \ot x_{\gamma}
x_{\alpha} - \x_{\gamma}^{aN_{\gamma}+1} \ot x_{\beta}x_{\alpha}) \\
& \quad - \Bsj\Dsj^{aN_{\gamma}+1} \{ (\frac{\Csj}{\Dsj} -1) (a)_{(\frac{\Csj}{\Dsj})^{N_{\gamma}}}
- (\frac{\Csj}{\Dsj})^{aN_{\gamma}} + (\frac{\Csj}{\Dsj})^{aN_{\gamma}+1}\} 
\x_{\gamma}^{aN_{\gamma}+1}\ot x_{\gamma}^{N_{\gamma}-1} 
\\
& \quad - \sum_{j=1}^{n} \Bsj_j \Csj ^{aN_{\gamma}} \big(\Csj \x_{\gamma}^{aN_{\gamma}+1} \ot x_{\eta_j}
- \Dsj q_{\delta_j \gamma} \x_{\gamma}^{aN_{\gamma}} \x_{\eta_j}\ot x_{\gamma} \big)x_{\delta_j}
\Big).
\end{align*}
Now we use the formula $d(\x_{\gamma}^{N_{\gamma}(a+1)}\ot 1) = \x_{\gamma}^{aN_{\gamma}+1}\ot
x_{\gamma}^{N_{\gamma}-1}$,  \eqref{eq:diff-commuting roots-2} and~(\ref{eq:diff2-hypothesis-roots}) 
to rewrite the above expression as
\begin{align*}
& \x_{\alpha} \x_{\gamma}^{aN_{\gamma}+1}\ot x_{\beta} - \Dsj \x_{\alpha} \x_{\gamma}^{aN_{\gamma}}
\x_{\beta}\ot x_{\gamma}  + \Asj\Csj^{aN_{\gamma}+1} \x_{\gamma}^{aN_{\gamma}+1}\x_{\beta}\ot x_{\alpha} 
\\ & \quad
+ \Bsj\Dsj^{aN_{\gamma}+1} \{ (\frac{\Csj}{\Dsj}-1) (a+1)_{(\frac{\Csj}{\Dsj})^{N_{\gamma}} } \}
\x_{\gamma}^{N_{\gamma}(a+1)}\ot 1 
+\sum_{j=1}^{n} \Bsj_j \Csj ^{aN_{\gamma}+1}  \x_{\gamma}^{aN_{\gamma}+1} \x_{\eta_j} \ot  x_{\delta_j}
,
\end{align*}
which agrees with the second claimed formula in the lemma. 
\end{proof}

\begin{lema}\label{lem:diff-case2}
Let $\alpha < \eta <\gamma < \beta <\delta $ be positive roots
such that $N_{\gamma}=2$ and the relations among the corresponding root vectors take the form
\begin{align}\label{eq:diff-case2-hypothesis}
\begin{aligned}
x_{\alpha}x_{\beta} &= q_{\alpha \beta} x_{\beta}x_{\alpha} +\Bsj_1 x_{\gamma}x_{\eta}, &
x_{\eta}x_{\delta} &= q_{\eta\delta}  x_{\delta}x_{\eta} + \Bsj_2 x_{\gamma}, 
\end{aligned}
\end{align}
for some scalars $\Bsj_1, \Bsj_2$ and the other pairs of root vectors $q$-commute.
Then, for all $n\geq 0$, 
\begin{align}\label{eq:diff-case2-formula}
\begin{aligned}
d(\x_{\alpha} \x_{\gamma}^{n} & \x_{\beta} \x_{\delta}\ot 1)  = \x_{\alpha} \x_{\gamma}^{n} \x_{\beta} \ot x_{\delta}
-q_{\beta\delta}\x_{\alpha} \x_{\gamma}^{n} \x_{\delta}\ot x_{\beta}
+q_{\gamma\beta}q_{\gamma\delta} \x_{\alpha} \x_{\gamma}^{n-1} \x_{\beta} \x_{\delta}\ot x_{\gamma}
\\ &
+(-q_{\alpha\gamma})^nq_{\alpha\beta}q_{\alpha\delta} \x_{\gamma}^{n} \x_{\beta} \x_{\delta}\ot x_{\alpha}
-q_{\eta\delta} \Bsj_1 (-1)^{n-1} q_{\alpha\gamma}^{n} (n+1)_{\widetilde{q}_{\beta\gamma}}  \x_{\gamma}^{n+1} \x_{\delta}\ot x_{\eta}
\\ & + \Bsj_1 \Bsj_2 \coef{\alpha\beta\gamma}{n} (-q_{\gamma\alpha})^{-n}  \x_{\gamma}^{n+2} \ot 1.
\end{aligned}
\end{align}
where $\coef{\alpha\beta\gamma}{n}:=\sum\limits_{k=0}^{n} (-\widetilde{q}_{\alpha\gamma})^{k} (k+1)_{\widetilde{q}_{\beta\gamma}}$, $n\in\N$.
\end{lema} 

\bigbreak
Notice that the equalities in \eqref{eq:diff-case2-hypothesis} force
\begin{align}\label{eq:diff-case2-hypothesis-roots}
\gamma + \eta &=\alpha+\beta, &  \eta+\delta &= \gamma.
\end{align}
Hence the following equality also holds: $2\gamma=\alpha+\beta+\delta$.

\pf
First we claim that
\begin{align}
\label{eq:diff-case2-auxiliar}
\begin{aligned}
d(\x_{\alpha} \x_{\gamma}^{n} \x_{\beta}\ot 1)  &=
\x_{\alpha} \x_{\gamma}^{n}\otimes x_{\beta} 
-q_{\gamma\beta} \x_{\alpha} \x_{\gamma}^{n-1}\x_{\beta} \ot x_{\gamma} 
-q_{\alpha\beta}(-q_{\alpha\gamma})^{n} \x_{\gamma}^n\x_{\beta}\ot x_{\alpha}
\\ &
- \Bsj_1 (-q_{\alpha\gamma})^{n} (n+1)_{\widetilde{q}_{\beta\gamma}} \x_{\gamma}^{n+1} \ot x_{\eta}
\end{aligned}
\end{align}
The proof is by induction on $n$. When $n=0$,
\begin{align*}
d( & \x_{\alpha} \x_{\beta} \ot 1 ) =
\x_{\alpha}\otimes x_{\beta} - s_0 d_1(\x_{\alpha}\otimes x_{\beta})
=
\x_{\alpha}\otimes x_{\beta} - q_{\alpha \beta} \x_{\beta} \otimes x_{\alpha} -\Bsj_1 \x_{\gamma} \otimes x_{\eta}.
\end{align*}
Now assume that \eqref{eq:diff-case2-auxiliar} holds for $n$. Let $c_n=(-q_{\alpha\gamma})^{n} (n+1)_{\widetilde{q}_{\beta\gamma}}$. We compute:
\begin{align*}
d( & \x_{\alpha} \x_{\gamma}^{n+1} \x_{\beta}\ot 1)  =
\x_{\alpha} \x_{\gamma}^{n+1}\otimes x_{\beta}
-sd(\x_{\alpha} \x_{\gamma}^{n+1} \otimes x_{\beta})
\\
&= \x_{\alpha} \x_{\gamma}^{n+1}\otimes x_{\beta} - s \big( q_{\gamma\beta} \x_{\alpha} \x_{\gamma}^{n}
\ot x_{\beta}x_{\gamma} +(-q_{\alpha \gamma})^{n+1} \x_{\gamma}^{n+1}\ot (q_{\alpha \beta} x_{\beta}x_{\alpha} +\Bsj_1 x_{\gamma}x_{\eta})\big)
\\
&= \x_{\alpha} \x_{\gamma}^{n+1}\otimes x_{\beta}-q_{\gamma\beta} \x_{\alpha} \x_{\gamma}^{n}\x_{\beta}\ot x_{\gamma} - s \big( (-q_{\alpha \gamma})^{n+1}q_{\alpha \beta} \x_{\gamma}^{n+1}\ot  x_{\beta}x_{\alpha}
\\ & \quad
+(-q_{\alpha \gamma})^{n+1}\Bsj_1 \x_{\gamma}^{n+1}\ot  x_{\gamma}x_{\eta}
+ q_{\gamma\beta}
q_{\alpha\beta}(-q_{\alpha}\gamma)^{n} \x_{\gamma}^n\x_{\beta}\ot x_{\alpha}x_{\gamma}
+ q_{\gamma\beta} \Bsj_1 c_{n} \x_{\gamma}^{n+1} \ot x_{\eta}x_{\gamma}
\big)
\\
&= \x_{\alpha} \x_{\gamma}^{n+1}\otimes x_{\beta}-q_{\gamma\beta} \x_{\alpha} \x_{\gamma}^{n}\x_{\beta}\ot x_{\gamma} -\Bsj_1 \big(q_{\gamma\beta}q_{\eta\gamma} c_{n}+(-q_{\alpha \gamma})^{n+1}\big) \x_{\gamma}^{n+2} \ot x_{\eta}
\\ & \quad
- (-q_{\alpha \gamma})^{n+1}q_{\alpha \beta} s d\big(  \x_{\gamma}^{n+1} \x_{\beta} \ot x_{\alpha}
\big).
\end{align*}
Now the inductive step follows using Remark \ref{rem:differential-q-commute} and
\begin{align*}
q_{\eta\gamma} &= q_{\alpha\gamma}q_{\beta\gamma}q_{\gamma\gamma}^{-1} = -q_{\alpha\gamma}q_{\beta\gamma}.
\end{align*}

Next we prove by induction on $n$ that there exist $e_n\in\Bbbk$ such that:
\begin{align*}
d(\x_{\alpha} & \x_{\gamma}^{n} \x_{\beta} \x_{\delta}\ot 1)  = \x_{\alpha} \x_{\gamma}^{n} \x_{\beta} \ot x_{\delta}
-q_{\beta\delta}\x_{\alpha} \x_{\gamma}^{n} \x_{\delta}\ot x_{\beta}
+q_{\gamma\beta}q_{\gamma\delta} \x_{\alpha} \x_{\gamma}^{n-1} \x_{\beta} \x_{\delta}\ot x_{\gamma}
\\ &
+(-q_{\alpha\gamma})^nq_{\alpha\beta}q_{\alpha\delta} \x_{\gamma}^{n} \x_{\beta} \x_{\delta}\ot x_{\alpha}
-q_{\eta\delta} \Bsj_1 c_{n-1}  \x_{\gamma}^{n+1} \x_{\delta}\ot x_{\eta}
+ \Bsj_1 \Bsj_2 e_n \x_{\gamma}^{n+2} \ot 1.
\end{align*}
The proof is again by induction. When $n=0$,
\begin{align*}
d( & \x_{\alpha} \x_{\beta} \x_{\delta}\ot 1) =
\x_{\alpha}\x_{\beta} \ot x_{\delta} - s_1 d_2 \left( \x_{\alpha}\x_{\beta} \ot x_{\delta}\right)
\\
&= \x_{\alpha}\x_{\beta} \ot x_{\delta} - s_1 \left( 
q_{\beta\delta} \x_{\alpha}\otimes x_{\delta}x_{\beta} 
- q_{\alpha \beta}q_{\alpha\delta} \x_{\beta} \otimes x_{\delta}x_{\alpha} -\Bsj_1 \x_{\gamma} \otimes (q_{\eta\delta}  x_{\delta}x_{\eta} + \Bsj_2 x_{\gamma})
\right)
\\
&= \x_{\alpha}\x_{\beta} \ot x_{\delta} -q_{\beta\delta} \x_{\alpha}\x_{\delta}\otimes x_{\beta} - s_1 \big( 
- q_{\alpha \beta}q_{\alpha\delta} \x_{\beta} \otimes x_{\delta}x_{\alpha} 
-q_{\eta\delta}\Bsj_1 \x_{\gamma} \otimes   x_{\delta}x_{\eta}
\\
& \quad -\Bsj_1 \Bsj_2 \x_{\gamma} \otimes  x_{\gamma}
+ q_{\beta\delta} q_{\alpha \delta} \x_{\delta} \otimes (q_{\alpha \beta} x_{\beta}x_{\alpha} +\Bsj_1 x_{\gamma}x_{\eta}) \big)
\\
&= \x_{\alpha}\x_{\beta} \ot x_{\delta} -q_{\beta\delta} \x_{\alpha}\x_{\delta}\otimes x_{\beta}
+\Bsj_1 \Bsj_2 \x_{\gamma}^2 \otimes 1
+q_{\eta\delta}\Bsj_1 \x_{\gamma} \x_{\delta} \otimes x_{\eta}
\\
& \quad
- s_1 \big(- q_{\alpha \beta}q_{\alpha\delta} \x_{\beta} \otimes x_{\delta}x_{\alpha} 
+ q_{\beta\delta} q_{\alpha \delta} \x_{\delta} \otimes (q_{\alpha \beta} x_{\beta}x_{\alpha} +\Bsj_1 x_{\gamma}x_{\eta})
- q_{\gamma\delta} q_{\eta\delta}\Bsj_1 \x_{\delta} \otimes x_{\gamma}x_{\eta} \big)
\\
&= \x_{\alpha}\x_{\beta} \ot x_{\delta} -q_{\beta\delta} \x_{\alpha}\x_{\delta}\otimes x_{\beta}
+\Bsj_1 \Bsj_2 \x_{\gamma}^2 \otimes 1
+q_{\eta\delta}\Bsj_1 \x_{\gamma} \x_{\delta} \otimes x_{\eta}
+q_{\alpha \beta}q_{\alpha\delta} \x_{\beta}x_{\delta} \otimes x_{\alpha}
\\
& \quad
- \Bsj_1 (q_{\beta\delta} q_{\alpha \delta}- q_{\gamma\delta} q_{\eta\delta}) s \big(  \x_{\delta} \otimes x_{\gamma}x_{\eta}\big).
\end{align*}
Now assume that the formula holds for $n$. Using \eqref{eq:diff-case2-auxiliar}:
\begin{align*}
d(\x_{\alpha} & \x_{\gamma}^{n+1} \x_{\beta} \x_{\delta}\ot 1)  =
\x_{\alpha} \x_{\gamma}^{n+1} \x_{\beta}\otimes x_{\delta}
-sd(\x_{\alpha} \x_{\gamma}^{n+1} \x_{\beta}\otimes x_{\delta})
\\ &= \x_{\alpha} \x_{\gamma}^{n+1} \x_{\beta}\otimes x_{\delta} -
s \big( q_{\beta\delta} \x_{\alpha} \x_{\gamma}^{n+1}\otimes x_{\delta}x_{\beta}
-q_{\gamma\beta} q_{\gamma\delta} \x_{\alpha} \x_{\gamma}^{n}\x_{\beta} \ot x_{\delta}x_{\gamma}
\\ & \quad -q_{\alpha\beta}q_{\alpha\delta}(-q_{\alpha}\gamma)^{n+1} \x_{\gamma}^{n+1}\x_{\beta}\ot x_{\delta}x_{\alpha} + \Bsj_1 c_{n+1} \x_{\gamma}^{n+2} \ot (q_{\eta\delta}x_{\delta}x_{\eta} + \Bsj_2 x_{\gamma}) \big)
\\ &= \x_{\alpha} \x_{\gamma}^{n+1} \x_{\beta}\otimes x_{\delta} -q_{\beta\delta} \x_{\alpha} \x_{\gamma}^{n+1} \x_{\delta}\otimes x_{\beta} -
s \big( q_{\beta\delta}q_{\gamma\delta}q_{\gamma\beta} \x_{\alpha}\x_{\gamma}^{n}\x_{\delta} \ot x_{\beta}x_{\gamma}
\\ & \quad    
+ q_{\beta\delta} q_{\alpha \delta} (-q_{\alpha\gamma})^{n+1} \x_{\gamma}^{n+1}\x_{\delta} \otimes (q_{\alpha \beta} x_{\beta}x_{\alpha} +\Bsj_1 x_{\gamma}x_{\eta})
-q_{\gamma\beta} q_{\gamma\delta} \x_{\alpha} \x_{\gamma}^{n}\x_{\beta} \ot x_{\delta}x_{\gamma}
\\ & \quad -q_{\alpha\beta}q_{\alpha\delta}(-q_{\alpha}\gamma)^{n+1} \x_{\gamma}^{n+1}\x_{\beta}\ot x_{\delta}x_{\alpha} - \Bsj_1 c_{n+1} \x_{\gamma}^{n+2} \ot (q_{\eta\delta}x_{\delta}x_{\eta} + \Bsj_2 x_{\gamma}) \big)
\end{align*}
Next we use the inductive hypothesis, the relation $x_{\gamma}^2=0$, \eqref{eq:diff-commuting roots-3} and \eqref{eq:diff-commuting roots-4}:
\begin{align*}
d(\x_{\alpha} & \x_{\gamma}^{n+1} \x_{\beta} \x_{\delta}\ot 1)  =
\x_{\alpha} \x_{\gamma}^{n+1} \x_{\beta}\otimes x_{\delta} -q_{\beta\delta} \x_{\alpha} \x_{\gamma}^{n+1} \x_{\delta}\otimes x_{\beta} 
+q_{\gamma\beta} q_{\gamma\delta} \x_{\alpha} \x_{\gamma}^{n}\x_{\beta} \x_{\delta} \ot x_{\gamma}
\\ & \quad    
-s \big(
-q_{\gamma\beta} q_{\gamma\delta}(-q_{\alpha\gamma})^{n+1} q_{\alpha\beta}q_{\alpha\delta} \x_{\gamma}^{n} \x_{\beta} \x_{\delta}\ot x_{\gamma}x_{\alpha}
-q_{\gamma\beta} q_{\gamma\delta}q_{\eta\gamma} q_{\eta\delta} \Bsj_1 c_{n+1}  \x_{\gamma}^{n+1} \x_{\delta}\ot x_{\gamma}x_{\eta}
\\ & \quad    
+\Bsj_1 \Bsj_2 e_nq_{\gamma\beta} q_{\gamma\delta} \x_{\gamma}^{n+2} \ot x_{\gamma}
+ q_{\beta\delta} q_{\alpha \delta} (-q_{\alpha\gamma})^{n+1} \x_{\gamma}^{n+1}\x_{\delta} \otimes (q_{\alpha \beta} x_{\beta}x_{\alpha} +\Bsj_1 x_{\gamma}x_{\eta})
\\ & \quad -q_{\alpha\beta}q_{\alpha\delta}(-q_{\alpha\gamma})^{n+1} \x_{\gamma}^{n+1}\x_{\beta}\ot x_{\delta}x_{\alpha} - \Bsj_1 c_{n+1} \x_{\gamma}^{n+2} \ot (q_{\eta\delta}x_{\delta}x_{\eta} + \Bsj_2 x_{\gamma}) \big)
\\ & =
\x_{\alpha} \x_{\gamma}^{n+1} \x_{\beta}\otimes x_{\delta} -q_{\beta\delta} \x_{\alpha} \x_{\gamma}^{n+1} \x_{\delta}\otimes x_{\beta} 
+q_{\gamma\beta} q_{\gamma\delta} \x_{\alpha} \x_{\gamma}^{n}\x_{\beta} \x_{\delta} \ot x_{\gamma}
\\ & \quad    
+q_{\alpha\beta}q_{\alpha\delta}(-q_{\alpha\gamma})^{n+1} \x_{\gamma}^{n+1}\x_{\beta}\x_{\delta}\ot x_{\alpha}
-\Bsj_1 \Bsj_2 (e_nq_{\gamma\beta} q_{\gamma\delta} - c_{n+1}) \x_{\gamma}^{n+3} \ot  1
\\ & \quad    
+q_{\eta\delta} \Bsj_1 c_{n+1} \x_{\gamma}^{n+2}\x_{\delta} \ot x_{\eta},
\end{align*}
and the inductive step follows. To finish the proof we have to compute $e_n$. Note that
\begin{align*}
e_0&=1, &
e_{n+1} &= -e_nq_{\gamma\beta} q_{\gamma\delta} + c_{n+1}, & 
\mbox{for all } & n\ge 0.
\end{align*}
By \eqref{eq:diff-case2-hypothesis-roots} and using that $q_{\gamma\gamma}=-1$,
\begin{align*}
q_{\gamma\delta} &= q_{\gamma\gamma}^2q_{\gamma\alpha}^{-1}q_{\gamma\beta}^{-1} =q_{\gamma\alpha}^{-1}q_{\gamma\beta}^{-1}.
\end{align*}
Hence $e_n=\coef{\alpha\beta\gamma}{n} (-q_{\gamma\alpha})^{-n}$ for all $n\ge 0$.
\epf

\begin{lema}\label{lem:diff-case3}
Let $\alpha < \eta <\gamma < \tau < \beta <\delta $ be positive roots
such that $N_{\gamma}=N_{\alpha}=2$ and the relations among the corresponding root vectors take the form
\begin{align}\label{eq:diff-case3-hypothesis}
\begin{aligned}
x_{\alpha}x_{\beta} &= q_{\alpha \beta} x_{\beta}x_{\alpha} +\Bsj_1 x_{\tau}x_{\gamma}, &
x_{\eta}x_{\delta} &= q_{\eta\delta}  x_{\delta}x_{\eta} + \Bsj_2 x_{\gamma},
\\
x_{\alpha}x_{\tau} &= q_{\alpha \tau} x_{\tau}x_{\alpha} +\Bsj_3 x_{\gamma}x_{\eta}, &
x_{\eta}x_{\beta} &= q_{\eta\beta}  x_{\beta}x_{\eta} + \Bsj_4 x_{\tau}^2,
\end{aligned}
\end{align}
for some scalars $\Bsj_j\in\Bbbk$ and the other pairs of root vectors $q$-commute.
Then, for all $n\geq 0$, 
\begin{align}\label{eq:diff-case3-formula}
&\begin{aligned}
d(\x_{\alpha}^2 \x_{\gamma}^{n} & \x_{\beta} \x_{\delta}\ot 1)  = 
\x_{\alpha}^2 \x_{\gamma}^{n} \x_{\beta} \ot x_{\delta}
-q_{\beta\delta}\x_{\alpha}^2 \x_{\gamma}^{n} \x_{\delta}\ot x_{\beta}
+q_{\gamma\beta}q_{\gamma\delta} \x_{\alpha}^2 \x_{\gamma}^{n-1} \x_{\beta} \x_{\delta}\ot x_{\gamma}
\\ & +(-q_{\alpha\gamma})^nq_{\alpha\beta}q_{\alpha\delta} \x_{\alpha}\x_{\gamma}^{n} \x_{\beta} \x_{\delta}\ot x_{\alpha}
+ q_{\gamma\beta}^{n} q_{\eta\gamma}^{n+1}q_{\eta\delta} \Bsj_1\Bsj_3 \x_{\gamma}^{n+2} \x_{\delta} \ot x_{\eta}
\\ & + \frac{q_{\alpha\gamma}^{n+1}}{q_{\gamma\alpha}^{n+2}q_{\gamma\beta}} \Big( (n+1)_{\widetilde{q}_{\gamma\alpha}\widetilde{q}_{\gamma\beta}} + \sum_{j=1}^n \coef{\alpha\tau\gamma}{j}\Big) \Bsj_1\Bsj_2\Bsj_3 \x_{\gamma}^{n+3} \ot 1.
\end{aligned}
\end{align}
\end{lema} 

\bigbreak
Notice that the equalities in \eqref{eq:diff-case3-hypothesis} force
\begin{align}\label{eq:diff-case3-hypothesis-roots}
\gamma + \tau &=\alpha+\beta, &  \eta+\delta &= \gamma, &
\alpha+\tau &= \gamma+\eta, & \eta+\beta &= 2\tau.
\end{align}
Thus the following equality also holds: $3\gamma=2\alpha+\beta+\delta$.

\pf
Let $n\in\N_0$. A computation similar to \eqref{eq:diff-case2-auxiliar} proves that
\begin{align}\label{eq:diff-case3-auxiliar-1}
\begin{aligned}
d(\x_{\alpha} \x_{\gamma}^{n} \x_{\tau}\ot 1) & =
\x_{\alpha} \x_{\gamma}^{n} \otimes x_{\tau}
- q_{\gamma\tau} \x_{\alpha} \x_{\gamma}^{n-1} \x_{\tau} \ot x_{\gamma}
- q_{\alpha\tau} (-q_{\alpha\gamma})^n \x_{\gamma}^{n} \x_{\tau} \ot x_{\alpha} 
\\ & \quad -\Bsj_3 (-q_{\alpha\gamma})^{n} (n+1)_{\widetilde{q}_{\tau\gamma}} \x_{\gamma}^{n+1} \ot x_{\eta},
\end{aligned}
\\\label{eq:diff-case3-auxiliar-2}
\begin{aligned}
d(\x_{\alpha} \x_{\gamma}^{n} \x_{\beta}\ot 1) & =
\x_{\alpha} \x_{\gamma}^{n} \otimes x_{\beta}
- q_{\gamma\beta} \x_{\alpha} \x_{\gamma}^{n-1} \x_{\beta} \ot x_{\gamma}
- q_{\alpha\beta} (-q_{\alpha\gamma})^n \x_{\gamma}^{n} \x_{\beta} \ot x_{\alpha} 
\\ & \quad -(-q_{\alpha\gamma})^n \Bsj_1 \x_{\gamma}^{n} \x_{\tau} \ot x_{\gamma}.
\end{aligned}
\end{align}

Now we compute more differentials:
\begin{align}\label{eq:diff-case3-auxiliar}
&\begin{aligned}
d(\x_{\alpha}^2 \x_{\gamma}^{n} \x_{\beta}\ot 1)  &=
\x_{\alpha}^2 \x_{\gamma}^{n}\otimes x_{\beta} -q_{\gamma\beta} \x_{\alpha}^2 \x_{\gamma}^{n-1}\x_{\beta} \ot x_{\gamma} -q_{\alpha\beta}(-q_{\alpha\gamma})^{n} \x_{\alpha} \x_{\gamma}^n\x_{\beta}\ot x_{\alpha}
\\ & + q_{\gamma\beta}^n q_{\eta\gamma}^{n+1} \Bsj_1 \Bsj_3 \x_{\gamma}^{n+2} \ot x_{\eta}
- (-q_{\alpha\gamma})^{n+1} \Bsj_1 \x_{\alpha} \x_{\gamma}^{n} \x_{\tau} \ot x_{\gamma},
\end{aligned}
\\ \label{eq:diff-case3-auxiliarbis}
&\begin{aligned}
d(\x_{\alpha} \x_{\gamma}^{n} & \x_{\beta} \x_{\delta}\ot 1)  = \x_{\alpha} \x_{\gamma}^{n} \x_{\beta} \ot x_{\delta}
-q_{\beta\delta}\x_{\alpha} \x_{\gamma}^{n} \x_{\delta}\ot x_{\beta}
+q_{\gamma\beta}q_{\gamma\delta} \x_{\alpha} \x_{\gamma}^{n-1} \x_{\beta} \x_{\delta}\ot x_{\gamma}
\\ & +(-q_{\alpha\gamma})^nq_{\alpha\beta}q_{\alpha\delta} \x_{\gamma}^{n} \x_{\beta} \x_{\delta}\ot x_{\alpha}
+ (-q_{\alpha\gamma})^{n} q_{\gamma\delta} \Bsj_1 \x_{\gamma}^{n}\x_{\tau} \x_{\delta} \ot x_{\gamma},
\end{aligned}
\end{align}

First we prove \eqref{eq:diff-case3-auxiliar} by induction on $n$. For $n=0$,
\begin{align*}
d(\x_{\alpha}^2 &\x_{\beta}\ot 1)  = \x_{\alpha}^2 \ot x_{\beta}- 
s \big(q_{\alpha \beta} \x_{\alpha} \ot x_{\beta}x_{\alpha}
+\Bsj_1 \x_{\alpha} \ot x_{\tau}x_{\gamma}
\big)
\\
& = \x_{\alpha}^2 \ot x_{\beta} - \Bsj_1 \x_{\alpha} \x_{\tau} \ot x_{\gamma} - q_{\alpha \beta} \x_{\alpha} \x_{\beta} \ot x_{\alpha}
+ s \big( q_{\eta\gamma} \Bsj_1\Bsj_3 \x_{\gamma} \ot x_{\gamma} x_{\eta}
\\ & \quad
+(q_{\alpha \tau}q_{\alpha\gamma} 
+ q_{\alpha \beta})\Bsj_1 \x_{\tau} \ot  x_{\gamma}x_{\alpha}\big)
\\
& = \x_{\alpha}^2 \ot x_{\beta} - \Bsj_1 \x_{\alpha} \x_{\tau} \ot x_{\gamma} - q_{\alpha \beta} \x_{\alpha} \x_{\beta} \ot x_{\alpha}
+ q_{\eta\gamma} \Bsj_1\Bsj_3 \x_{\gamma}^2 \ot x_{\eta}.
\end{align*}
Assume that \eqref{eq:diff-case3-auxiliar} holds for $n$. By inductive hypothesis, $x_{\gamma}^2=0$, \eqref{eq:diff-case3-auxiliar-1}
and \eqref{eq:diff-case3-auxiliar-2}:
\begin{align*}
d(\x_{\alpha}^2 & \x_{\gamma}^{n+1}\x_{\beta}\ot 1)  = \x_{\alpha}^2 \x_{\gamma}^{n+1} \ot x_{\beta} - s \big( q_{\gamma\beta} \x_{\alpha}^2 \x_{\gamma}^{n} \ot x_{\beta} x_{\gamma} + 
(-q_{\alpha\gamma})^{n+1}\x_{\alpha} \x_{\gamma}^{n+1} \ot x_{\alpha} x_{\beta} \big)
\\
&= \x_{\alpha}^2 \x_{\gamma}^{n+1} \ot x_{\beta} -
q_{\gamma\beta} \x_{\alpha}^2 \x_{\gamma}^{n} \x_{\beta} \ot  x_{\gamma} - s \big( q_{\gamma\beta}q_{\alpha\beta} (-q_{\alpha\gamma})^{n} \x_{\alpha} \x_{\gamma}^n\x_{\beta}\ot x_{\alpha}x_{\gamma}
\\ & \quad - q_{\gamma\beta}^{n+1} q_{\eta\gamma}^{n+1}\Bsj_1 \Bsj_3 \x_{\gamma}^{n+2} \ot x_{\eta}x_{\gamma}
+ (-q_{\alpha\gamma})^{n+1} q_{\alpha \beta} \x_{\alpha} \x_{\gamma}^{n+1} \ot x_{\beta}x_{\alpha}
\\ & \quad 
+ (-q_{\alpha\gamma})^{n+1} \Bsj_1 \x_{\alpha} \x_{\gamma}^{n+1} \ot  x_{\tau}x_{\gamma} \big)
\\
&= \x_{\alpha}^2 \x_{\gamma}^{n+1} \ot x_{\beta} -
q_{\gamma\beta} \x_{\alpha}^2 \x_{\gamma}^{n} \x_{\beta} \ot  x_{\gamma} - (-q_{\alpha\gamma})^{n+1} q_{\alpha \beta} \x_{\alpha} \x_{\gamma}^{n+1}\x_{\beta} \ot x_{\alpha} 
\\ & \quad - s \big(
(-q_{\alpha\gamma})^{n+1} \Bsj_1 \x_{\alpha} \x_{\gamma}^{n+1} \ot  x_{\tau}x_{\gamma} -q_{\gamma\beta}^{n+1} q_{\eta\gamma}^{n+1} \Bsj_1 \Bsj_3 \x_{\gamma}^{n+2} \ot x_{\eta}x_{\gamma}
\\ & 
+(-q_{\alpha\gamma})^{2n+2} q_{\alpha \beta} \Bsj_1 \x_{\gamma}^{n+1} \x_{\tau} \ot x_{\gamma}x_{\alpha}
\big)
\\
&= \x_{\alpha}^2 \x_{\gamma}^{n+1} \ot x_{\beta} -
q_{\gamma\beta} \x_{\alpha}^2 \x_{\gamma}^{n} \x_{\beta} \ot  x_{\gamma} - (-q_{\alpha\gamma})^{n+1} q_{\alpha \beta} \x_{\alpha} \x_{\gamma}^{n+1}\x_{\beta} \ot x_{\alpha} 
\\ & \quad
- (-q_{\alpha\gamma})^{n+1} \Bsj_1 \x_{\alpha} \x_{\gamma}^{n+1} \x_{\tau} \ot  x_{\gamma} + q_{\gamma\beta}^{n+1} q_{\eta\gamma}^{n+2} \Bsj_1 \Bsj_3 s \big( \x_{\gamma}^{n+2} \ot x_{\gamma} x_{\eta}) \big),
\end{align*}
so \eqref{eq:diff-case3-auxiliar} follows since 
$s \big( \x_{\gamma}^{n+2} \ot x_{\gamma} x_{\eta}) = \x_{\gamma}^{n+3} \ot x_{\eta}$.
\medbreak

Now we prove \eqref{eq:diff-case3-auxiliarbis} by induction on $n$. For $n=0$,
\begin{align*}
d(\x_{\alpha} &\x_{\beta}\x_{\delta} \ot 1)  = \x_{\alpha}\x_{\beta}\ot x_{\delta} - s \big( q_{\beta\delta}
\x_{\alpha} \ot x_{\delta} x_{\beta} 
- q_{\alpha \beta}q_{\alpha\delta} \x_{\beta} \ot x_{\delta}x_{\alpha}
- q_{\gamma\delta} \Bsj_1 \x_{\tau} \ot x_{\delta}x_{\gamma} \big)
\\ & = \x_{\alpha}\x_{\beta}\ot x_{\delta} -q_{\beta\delta}
\x_{\alpha} \x_{\delta} \ot  x_{\beta} - s \big( 
q_{\beta\delta}q_{\alpha\delta} \x_{\delta} \ot (q_{\alpha \beta} x_{\beta}x_{\alpha} +\Bsj_1 x_{\tau}x_{\gamma})
\\ & \quad
- q_{\alpha \beta}q_{\alpha\delta} \x_{\beta} \ot x_{\delta}x_{\alpha}
- q_{\gamma\delta} \Bsj_1 \x_{\tau} \ot x_{\delta}x_{\gamma} \big)
\\ & = \x_{\alpha}\x_{\beta}\ot x_{\delta} -q_{\beta\delta}
\x_{\alpha} \x_{\delta} \ot  x_{\beta} 
+ q_{\gamma\delta} \Bsj_1 \x_{\tau} \x_{\delta} \ot x_{\gamma}
+ q_{\alpha \beta}q_{\alpha\delta} \x_{\beta} \x_{\delta} \ot x_{\alpha}.
\end{align*}
Now assume that \eqref{eq:diff-case3-auxiliarbis} holds for $n$. Using \eqref{eq:diff-case3-auxiliar-2}, Remark \ref{rem:differential-q-commute} three times, inductive hypothesis, $x_{\gamma}^2=0=x_{\alpha}^2$,
\begin{align*}
d(\x_{\alpha} & \x_{\gamma}^{n+1} \x_{\beta}\x_{\delta} \ot 1)  = \x_{\alpha}\x_{\gamma}^{n+1}\x_{\beta}\ot x_{\delta} - 
s\big(q_{\beta\delta} \x_{\alpha} \x_{\gamma}^{n+1} \otimes x_{\delta}x_{\beta}
- q_{\gamma\beta} q_{\gamma\delta} \x_{\alpha} \x_{\gamma}^{n} \x_{\beta} \ot x_{\delta}x_{\gamma}
\\ & \quad 
- q_{\alpha\beta} (-q_{\alpha\gamma})^{n+1} q_{\alpha\delta} \x_{\gamma}^{n+1} \x_{\beta} \ot x_{\delta}x_{\alpha} 
-(-q_{\alpha\gamma})^{n+1} q_{\gamma\delta} \Bsj_1 \x_{\gamma}^{n+1} \x_{\tau} \ot x_{\delta}x_{\gamma}
\big)
\\ & = \x_{\alpha}\x_{\gamma}^{n+1}\x_{\beta}\ot x_{\delta} -
q_{\beta\delta} \x_{\alpha} \x_{\gamma}^{n+1} \x_{\delta} \otimes x_{\beta}
- s\big(- q_{\gamma\beta} q_{\gamma\delta} \x_{\alpha} \x_{\gamma}^{n} \x_{\beta} \ot x_{\delta}x_{\gamma}
\\ & \quad
+ q_{\beta\delta} q_{\gamma\delta} q_{\gamma\beta} \x_{\alpha} \x_{\gamma}^{n} \x_{\delta} \ot x_{\beta} x_{\gamma}
+ q_{\alpha\delta} (-q_{\alpha\gamma})^{n+1} q_{\beta\delta}
\x_{\gamma}^{n+1} \x_{\delta} \ot (q_{\alpha \beta} x_{\beta}x_{\alpha} +\Bsj_1 x_{\tau}x_{\gamma})
\\ & \quad 
- q_{\alpha\beta} (-q_{\alpha\gamma})^{n+1} q_{\alpha\delta} \x_{\gamma}^{n+1} \x_{\beta} \ot x_{\delta}x_{\alpha} 
-(-q_{\alpha\gamma})^{n+1} q_{\gamma\delta} \Bsj_1 \x_{\gamma}^{n+1} \x_{\tau} \ot x_{\delta}x_{\gamma} \big)
\\ & = \x_{\alpha}\x_{\gamma}^{n+1}\x_{\beta}\ot x_{\delta} -
q_{\beta\delta} \x_{\alpha} \x_{\gamma}^{n+1} \x_{\delta} \otimes x_{\beta}
+q_{\gamma\beta} q_{\gamma\delta} \x_{\alpha} \x_{\gamma}^{n} \x_{\beta} \x_{\delta} \ot x_{\gamma}
\\ &\quad 
- s\big(
-q_{\gamma\beta} q_{\gamma\delta} (-q_{\alpha\gamma})^{n+1}q_{\alpha\beta}q_{\alpha\delta} \x_{\gamma}^{n} \x_{\beta} \x_{\delta}\ot x_{\gamma}x_{\alpha}
\\ & \quad
+ q_{\alpha \beta} q_{\alpha\delta} (-q_{\alpha\gamma})^{n+1} q_{\beta\delta}
\x_{\gamma}^{n+1} \x_{\delta} \ot  x_{\beta}x_{\alpha}
+ q_{\alpha\delta} (-q_{\alpha\gamma})^{n+1} q_{\beta\delta} \Bsj_1
\x_{\gamma}^{n+1} \x_{\delta} \ot  x_{\tau}x_{\gamma}
\\ & \quad 
- q_{\alpha\beta} (-q_{\alpha\gamma})^{n+1} q_{\alpha\delta} \x_{\gamma}^{n+1} \x_{\beta} \ot x_{\delta}x_{\alpha} 
-(-q_{\alpha\gamma})^{n+1} q_{\gamma\delta} \Bsj_1 \x_{\gamma}^{n+1} \x_{\tau} \ot x_{\delta}x_{\gamma} \big)
\\ & = \x_{\alpha}\x_{\gamma}^{n+1}\x_{\beta}\ot x_{\delta} -
q_{\beta\delta} \x_{\alpha} \x_{\gamma}^{n+1} \x_{\delta} \otimes x_{\beta}
+q_{\gamma\beta} q_{\gamma\delta} \x_{\alpha} \x_{\gamma}^{n} \x_{\beta} \x_{\delta} \ot x_{\gamma}
\\ &\quad 
+(-q_{\alpha\gamma})^{n+1} q_{\gamma\delta} \Bsj_1 \x_{\gamma}^{n+1} \x_{\tau} \x_{\delta} \ot x_{\gamma}
- s\big(
-q_{\gamma\beta} q_{\gamma\delta} (-q_{\alpha\gamma})^{n+1}q_{\alpha\beta}q_{\alpha\delta} \x_{\gamma}^{n} \x_{\beta} \x_{\delta}\ot x_{\gamma}x_{\alpha}
\\ & \quad
+ q_{\alpha \beta} q_{\alpha\delta} (-q_{\alpha\gamma})^{n+1} q_{\beta\delta}
\x_{\gamma}^{n+1} \x_{\delta} \ot  x_{\beta}x_{\alpha}
- q_{\alpha\beta} (-q_{\alpha\gamma})^{n+1} q_{\alpha\delta} \x_{\gamma}^{n+1} \x_{\beta} \ot x_{\delta}x_{\alpha} \big)
\\ & = \x_{\alpha}\x_{\gamma}^{n+1}\x_{\beta}\ot x_{\delta} -
q_{\beta\delta} \x_{\alpha} \x_{\gamma}^{n+1} \x_{\delta} \otimes x_{\beta}
+q_{\gamma\beta} q_{\gamma\delta} \x_{\alpha} \x_{\gamma}^{n} \x_{\beta} \x_{\delta} \ot x_{\gamma}
\\ &\quad 
+(-q_{\alpha\gamma})^{n+1} q_{\gamma\delta} \Bsj_1 \x_{\gamma}^{n+1} \x_{\tau} \x_{\delta} \ot x_{\gamma}
+q_{\alpha\beta} (-q_{\alpha\gamma})^{n+1} q_{\alpha\delta} \x_{\gamma}^{n+1} \x_{\beta}\x_{\delta} \ot x_{\alpha}.
\end{align*}

\medbreak

Finally we prove \eqref{eq:diff-case3-formula} by induction on $n$. Notice that root vectors corresponding to $\alpha<\eta<\gamma<\tau<\delta$ satisfy \eqref{eq:diff-case2-hypothesis}, so $d(\x_{\alpha}\x_{\gamma}^n\x_{\tau}\x_{\delta} \ot 1)$ is given by  \eqref{eq:diff-case2-formula}. We claim that
\begin{align*}
d(\x_{\alpha}^2 \x_{\gamma}^{n} & \x_{\beta} \x_{\delta}\ot 1)  = 
\x_{\alpha}^2 \x_{\gamma}^{n} \x_{\beta} \ot x_{\delta}
-q_{\beta\delta}\x_{\alpha}^2 \x_{\gamma}^{n} \x_{\delta}\ot x_{\beta}
+q_{\gamma\beta}q_{\gamma\delta} \x_{\alpha}^2 \x_{\gamma}^{n-1} \x_{\beta} \x_{\delta}\ot x_{\gamma}
\\ & +(-q_{\alpha\gamma})^nq_{\alpha\beta}q_{\alpha\delta} \x_{\alpha}\x_{\gamma}^{n} \x_{\beta} \x_{\delta}\ot x_{\alpha}
+ q_{\gamma\beta}^{n} q_{\eta\gamma}^{n+1}q_{\eta\delta} \Bsj_1\Bsj_3 \x_{\gamma}^{n+2} \x_{\delta} \ot x_{\eta} + c_n \Bsj_1\Bsj_2\Bsj_3 \x_{\gamma}^{n+3} \ot 1
\end{align*}
for some scalar $c_n$. For $n=0$,
\begin{align*}
d(\x_{\alpha}^2 &\x_{\beta}\x_{\delta} \ot 1)  = \x_{\alpha}^2 \x_{\beta}\ot x_{\delta} - s \big( q_{\beta\delta}
\x_{\alpha}^2 \ot x_{\delta} x_{\beta} 
- q_{\gamma\delta} \Bsj_1 \x_{\alpha} \x_{\tau} \ot x_{\delta}x_{\gamma} 
- q_{\alpha \beta}q_{\alpha\delta} \x_{\alpha} \x_{\beta} \ot x_{\delta}x_{\alpha}
\\ & \quad + q_{\eta\gamma} \Bsj_1\Bsj_3 \x_{\gamma}^2 \ot (q_{\eta\delta}  x_{\delta}x_{\eta} + \Bsj_2 x_{\gamma}) \big)
\\ &
= \x_{\alpha}^2 \x_{\beta}\ot x_{\delta}- q_{\beta\delta}
\x_{\alpha}^2 \x_{\delta} \ot  x_{\beta} - s \big( 
- q_{\gamma\delta} \Bsj_1 \x_{\alpha} \x_{\tau} \ot x_{\delta}x_{\gamma} 
- q_{\alpha \beta}q_{\alpha\delta} \x_{\alpha} \x_{\beta} \ot x_{\delta}x_{\alpha}
\\ & \quad + q_{\eta\gamma}q_{\eta\delta} \Bsj_1\Bsj_3 \x_{\gamma}^2 \ot   x_{\delta}x_{\eta} 
+ q_{\eta\gamma} \Bsj_1\Bsj_2\Bsj_3 \x_{\gamma}^2 \ot x_{\gamma}
+ q_{\beta\delta} q_{\alpha\delta} \x_{\alpha}\x_{\delta} \ot (q_{\alpha \beta} x_{\beta}x_{\alpha} +\Bsj_1 x_{\tau}x_{\gamma}) \big)
\\ &
= \x_{\alpha}^2 \x_{\beta}\ot x_{\delta}- q_{\beta\delta}
\x_{\alpha}^2 \x_{\delta} \ot  x_{\beta} 
+ q_{\gamma\delta} \Bsj_1 \x_{\alpha} \x_{\tau} \x_{\delta} \ot x_{\gamma} - s \big( - q_{\alpha \beta}q_{\alpha\delta} \x_{\alpha} \x_{\beta} \ot x_{\delta}x_{\alpha} 
\\
& \quad 
+q_{\alpha\tau} q_{\alpha\delta}q_{\gamma\delta} \Bsj_1 \x_{\tau}\x_{\delta} \ot x_{\alpha}x_{\gamma}
+q_{\eta\delta} q_{\gamma\delta} \Bsj_1 \Bsj_3 \x_{\gamma}\x_{\delta} \ot x_{\eta}x_{\gamma}
+ q_{\eta\gamma}q_{\eta\delta} \Bsj_1\Bsj_3 \x_{\gamma}^2 \ot   x_{\delta}x_{\eta}
\\
& \quad  
+ (q_{\eta\gamma}+q_{\gamma\delta}) \Bsj_1\Bsj_2\Bsj_3 \x_{\gamma}^2 \ot x_{\gamma}
+ q_{\beta\delta} q_{\alpha\delta} q_{\alpha \beta} \x_{\alpha}\x_{\delta} \ot  x_{\beta}x_{\alpha} \big)
\\ &
= \x_{\alpha}^2 \x_{\beta}\ot x_{\delta}- q_{\beta\delta}
\x_{\alpha}^2 \x_{\delta} \ot  x_{\beta} 
+ q_{\gamma\delta} \Bsj_1 \x_{\alpha} \x_{\tau} \x_{\delta} \ot x_{\gamma} 
+ q_{\alpha \beta}q_{\alpha\delta} \x_{\alpha} \x_{\beta} \x_{\delta} \ot x_{\alpha}
\\ & \quad
+ (q_{\eta\gamma}+q_{\gamma\delta}) \Bsj_1\Bsj_2\Bsj_3 \x_{\gamma}^3 \ot 1 
- s \big((q_{\alpha \beta}+q_{\alpha\tau} q_{\alpha\gamma}) q_{\alpha\delta}q_{\gamma\delta} \Bsj_1 \x_{\tau}\x_{\delta} \ot x_{\gamma}x_{\alpha}
\\
& \quad 
+q_{\eta\delta} q_{\gamma\delta} \Bsj_1 \Bsj_3 \x_{\gamma}\x_{\delta} \ot x_{\eta}x_{\gamma}
+ q_{\eta\gamma}q_{\eta\delta} \Bsj_1\Bsj_3 \x_{\gamma}^2 \ot   x_{\delta}x_{\eta} \big)
\\ &
= \x_{\alpha}^2 \x_{\beta}\ot x_{\delta}- q_{\beta\delta}
\x_{\alpha}^2 \x_{\delta} \ot  x_{\beta} 
+ q_{\gamma\delta} \Bsj_1 \x_{\alpha} \x_{\tau} \x_{\delta} \ot x_{\gamma} 
+ q_{\alpha \beta}q_{\alpha\delta} \x_{\alpha} \x_{\beta} \x_{\delta} \ot x_{\alpha}
\\ & \quad
+ q_{\gamma\delta}(1-\widetilde{q}_{\gamma\delta}^{\, -1}) \Bsj_1\Bsj_2\Bsj_3 \x_{\gamma}^3 \ot 1 
+q_{\eta\gamma}q_{\eta\delta} \Bsj_1\Bsj_3 \x_{\gamma}^2 \x_{\delta} \ot   x_{\eta}.
\end{align*}
Now assume that \eqref{eq:diff-case3-formula} holds for $n$. Using \eqref{eq:diff-case3-auxiliar}, Remark \ref{rem:differential-q-commute}, inductive hypothesis, $x_{\gamma}^2=0=x_{\alpha}^2$, \eqref{eq:diff-case2-formula}, \eqref{eq:diff-case3-auxiliarbis},
\begin{align*}
d(\x_{\alpha}^2 & \x_{\gamma}^{n+1} \x_{\beta}\x_{\delta} \ot 1)  = \x_{\alpha}^2\x_{\gamma}^{n+1}\x_{\beta}\ot x_{\delta} - 
s\big(
q_{\beta\delta} \x_{\alpha}^2 \x_{\gamma}^{n+1}\otimes x_{\delta}x_{\beta} 
-q_{\gamma\beta} q_{\gamma\delta} \x_{\alpha}^2 \x_{\gamma}^{n} \x_{\beta} \ot x_{\delta}x_{\gamma} 
\\ & \quad 
-q_{\alpha\beta}(-q_{\alpha\gamma})^{n+1} q_{\alpha\delta} \x_{\alpha} \x_{\gamma}^{n+1}\x_{\beta}\ot x_{\delta}x_{\alpha}
- (-q_{\alpha\gamma})^{n+2} q_{\gamma\delta} \Bsj_1 \x_{\alpha} \x_{\gamma}^{n+1} \x_{\tau} \ot x_{\delta}x_{\gamma}
\\ & \quad 
+ q_{\gamma\beta}^{n+1} q_{\eta\gamma}^{n+2} \Bsj_1 \Bsj_2 \Bsj_3 \x_{\gamma}^{n+3} \ot  x_{\gamma}
+ q_{\gamma\beta}^{n+1} q_{\eta\gamma}^{n+2}q_{\eta\delta} \Bsj_1 \Bsj_3 \x_{\gamma}^{n+3} \ot x_{\delta}x_{\eta} \big)
\\
& = \x_{\alpha}^2\x_{\gamma}^{n+1}\x_{\beta}\ot x_{\delta} -
q_{\beta\delta} \x_{\alpha}^2 \x_{\gamma}^{n+1} \x_{\delta} \otimes x_{\beta} 
-s\big(
q_{\gamma\beta}q_{\gamma\delta}q_{\beta\delta}\x_{\alpha}^2 \x_{\gamma}^{n} \x_{\delta} \ot x_{\beta}x_{\gamma}
\\ & \quad 
+q_{\alpha\delta} (-q_{\alpha\gamma})^{n+1}q_{\beta\delta} \x_{\alpha} \x_{\gamma}^{n+1} \x_{\delta} \ot (q_{\alpha \beta} x_{\beta}x_{\alpha} +\Bsj_1 x_{\tau}x_{\gamma})  
-q_{\gamma\beta} q_{\gamma\delta} \x_{\alpha}^2 \x_{\gamma}^{n} \x_{\beta} \ot x_{\delta}x_{\gamma} 
\\ & \quad 
-q_{\alpha\beta}(-q_{\alpha\gamma})^{n+1} q_{\alpha\delta} \x_{\alpha} \x_{\gamma}^{n+1}\x_{\beta}\ot x_{\delta}x_{\alpha}
- (-q_{\alpha\gamma})^{n+2} q_{\gamma\delta} \Bsj_1 \x_{\alpha} \x_{\gamma}^{n+1} \x_{\tau} \ot x_{\delta}x_{\gamma}
\\ & \quad 
+ q_{\gamma\beta}^{n+1} q_{\eta\gamma}^{n+2} \Bsj_1 \Bsj_2 \Bsj_3 \x_{\gamma}^{n+3} \ot  x_{\gamma}
+ q_{\gamma\beta}^{n+1} q_{\eta\gamma}^{n+2}q_{\eta\delta} \Bsj_1 \Bsj_3 \x_{\gamma}^{n+3} \ot x_{\delta}x_{\eta} \big)
\\
& = \x_{\alpha}^2\x_{\gamma}^{n+1}\x_{\beta}\ot x_{\delta} -
q_{\beta\delta} \x_{\alpha}^2 \x_{\gamma}^{n+1} \x_{\delta} \otimes x_{\beta} 
+q_{\gamma\beta} q_{\gamma\delta} \x_{\alpha}^2 \x_{\gamma}^{n} \x_{\beta} \x_{\delta} \ot x_{\gamma}
\\ &\quad 
-s\big(
-q_{\gamma\beta} q_{\gamma\delta}  
(-q_{\alpha\gamma})^{n+1}q_{\alpha\beta}q_{\alpha\delta} \x_{\alpha}\x_{\gamma}^{n} \x_{\beta} \x_{\delta}\ot x_{\gamma}x_{\alpha}
\\ & \quad
+ q_{\gamma\beta} q_{\gamma\delta} c_n \Bsj_1\Bsj_2\Bsj_3 \x_{\gamma}^{n+3} \ot x_{\gamma} 
+ q_{\gamma\beta} q_{\gamma\delta}q_{\eta\gamma} q_{\gamma\beta}^{n} q_{\eta\gamma}^{n+1}q_{\eta\delta} \Bsj_1\Bsj_3 \x_{\gamma}^{n+2} \x_{\delta} \ot x_{\gamma}x_{\eta}
\\ & \quad
+ q_{\alpha \beta}q_{\alpha\delta} (-q_{\alpha\gamma})^{n+1}q_{\beta\delta} \x_{\alpha} \x_{\gamma}^{n+1} \x_{\delta} \ot  x_{\beta}x_{\alpha} 
+q_{\alpha\delta} (-q_{\alpha\gamma})^{n+1}q_{\beta\delta}\Bsj_1 \x_{\alpha} \x_{\gamma}^{n+1} \x_{\delta} \ot  x_{\tau}x_{\gamma}
\\ & \quad 
-q_{\alpha\beta}(-q_{\alpha\gamma})^{n+1} q_{\alpha\delta} \x_{\alpha} \x_{\gamma}^{n+1}\x_{\beta}\ot x_{\delta}x_{\alpha}
- (-q_{\alpha\gamma})^{n+2} q_{\gamma\delta} \Bsj_1 \x_{\alpha} \x_{\gamma}^{n+1} \x_{\tau} \ot x_{\delta}x_{\gamma}
\\ & \quad 
+ q_{\gamma\beta}^{n+1} q_{\eta\gamma}^{n+2} \Bsj_1 \Bsj_2 \Bsj_3 \x_{\gamma}^{n+3} \ot  x_{\gamma}
+ q_{\gamma\beta}^{n+1} q_{\eta\gamma}^{n+2}q_{\eta\delta} \Bsj_1 \Bsj_3 \x_{\gamma}^{n+3} \ot x_{\delta}x_{\eta} \big)
\\
& = \x_{\alpha}^2\x_{\gamma}^{n+1}\x_{\beta}\ot x_{\delta} -
q_{\beta\delta} \x_{\alpha}^2 \x_{\gamma}^{n+1} \x_{\delta} \otimes x_{\beta} 
+q_{\gamma\beta} q_{\gamma\delta} \x_{\alpha}^2 \x_{\gamma}^{n} \x_{\beta} \x_{\delta} \ot x_{\gamma}
\\ &\quad 
+(-q_{\alpha\gamma})^{n+2} q_{\gamma\delta} \Bsj_1 \x_{\alpha} \x_{\gamma}^{n+1} \x_{\tau} \x_{\delta} \ot x_{\gamma}
-s\big(
-q_{\gamma\beta} q_{\gamma\delta}  
(-q_{\alpha\gamma})^{n+1}q_{\alpha\beta}q_{\alpha\delta} \x_{\alpha}\x_{\gamma}^{n} \x_{\beta} \x_{\delta}\ot x_{\gamma}x_{\alpha}
\\ & \quad
+ q_{\gamma\beta} q_{\gamma\delta} c_n \Bsj_1\Bsj_2\Bsj_3 \x_{\gamma}^{n+3} \ot x_{\gamma} 
+ q_{\gamma\beta} q_{\gamma\delta}q_{\eta\gamma} q_{\gamma\beta}^{n} q_{\eta\gamma}^{n+1}q_{\eta\delta} \Bsj_1\Bsj_3 \x_{\gamma}^{n+2} \x_{\delta} \ot x_{\gamma}x_{\eta}
\\ & \quad
+ q_{\alpha \beta}q_{\alpha\delta} (-q_{\alpha\gamma})^{n+1}q_{\beta\delta} \x_{\alpha} \x_{\gamma}^{n+1} \x_{\delta} \ot  x_{\beta}x_{\alpha} 
-q_{\alpha\beta}(-q_{\alpha\gamma})^{n+1} q_{\alpha\delta} \x_{\alpha} \x_{\gamma}^{n+1}\x_{\beta}\ot x_{\delta}x_{\alpha}
\\ & \quad 
+ q_{\gamma\beta}^{n+1} q_{\eta\gamma}^{n+2} \Bsj_1 \Bsj_2 \Bsj_3 \x_{\gamma}^{n+3} \ot  x_{\gamma}
+ q_{\gamma\beta}^{n+1} q_{\eta\gamma}^{n+2}q_{\eta\delta} \Bsj_1 \Bsj_3 \x_{\gamma}^{n+3} \ot x_{\delta}x_{\eta} 
\\ & \quad 
- q_{\gamma\delta} q_{\eta\gamma}q_{\eta\delta} q_{\alpha\gamma}^{2n+3} (n+2)_{\widetilde{q}_{\tau\gamma}} \Bsj_1 \Bsj_3  \x_{\gamma}^{n+2} \x_{\delta}\ot x_{\gamma}x_{\eta}
-q_{\gamma\delta} q_{\alpha\gamma}^{2n+4} q_{\alpha\tau}q_{\alpha\delta} \Bsj_1 \x_{\gamma}^{n+1} \x_{\tau} \x_{\delta}\ot x_{\gamma}x_{\alpha}
\\ & \quad 
+ (-q_{\alpha\gamma})^{n+2} q_{\gamma\delta} \Bsj_1\Bsj_2 \Bsj_3 (-q_{\gamma\alpha})^{-n-1} \coef{\alpha\tau\gamma}{n+1} \x_{\gamma}^{n+3} \ot x_{\gamma} \big)
\\
& = \x_{\alpha}^2\x_{\gamma}^{n+1}\x_{\beta}\ot x_{\delta} -
q_{\beta\delta} \x_{\alpha}^2 \x_{\gamma}^{n+1} \x_{\delta} \otimes x_{\beta} 
+q_{\gamma\beta} q_{\gamma\delta} \x_{\alpha}^2 \x_{\gamma}^{n} \x_{\beta} \x_{\delta} \ot x_{\gamma}
\\ &\quad 
+(-q_{\alpha\gamma})^{n+2} q_{\gamma\delta} \Bsj_1 \x_{\alpha} \x_{\gamma}^{n+1} \x_{\tau} \x_{\delta} \ot x_{\gamma}
+q_{\alpha\beta}(-q_{\alpha\gamma})^{n+1} q_{\alpha\delta} \x_{\alpha} \x_{\gamma}^{n+1}\x_{\beta}\x_{\delta}\ot x_{\alpha}
\\ & \quad 
-s\big(q_{\gamma\beta}^{n+1} q_{\eta\gamma}^{n+2}q_{\eta\delta} \Bsj_1 \Bsj_3 \x_{\gamma}^{n+3} \ot x_{\delta}x_{\eta}
\\ & \quad 
+ (q_{\gamma\beta}  q_{\gamma\beta}^{n} q_{\eta\gamma}^{n+1}q_{\eta\delta} 
- q_{\eta\delta} q_{\alpha\gamma}^{2n+3} (n+2)_{\widetilde{q}_{\tau\gamma}}) q_{\gamma\delta}q_{\eta\gamma} \Bsj_1 \Bsj_3  \x_{\gamma}^{n+2} \x_{\delta}\ot x_{\gamma}x_{\eta} 
\\ & \quad 
+ \Big(q_{\gamma\beta}^{n+1} q_{\eta\gamma}^{n+2} 
+q_{\gamma\beta} q_{\gamma\delta} c_n 
- q_{\alpha\gamma}^{n+2} q_{\gamma\delta} q_{\gamma\alpha}^{-n-1} \coef{\alpha\tau\gamma}{n+1} \Big) \Bsj_1\Bsj_2 \Bsj_3
\x_{\gamma}^{n+3} \ot x_{\gamma} \big).
\end{align*}
Hence the claim follows using Remark \ref{rem:differential-q-commute} and that $d(\x_{\gamma}^{n+4}\ot 1) = \x_{\gamma}^{n+3}\ot x_{\gamma}$. The scalars $c_n$ are defined recursively by the equation:
\begin{align*}
c_{n+1} &= q_{\gamma\beta}^{n+1} q_{\eta\gamma}^{n+2} 
+q_{\gamma\beta} q_{\gamma\delta} c_n 
- q_{\alpha\gamma}^{n+2} q_{\gamma\delta} q_{\gamma\alpha}^{-n-1} \coef{\alpha\tau\gamma}{n+1}.
\end{align*}
Thus \eqref{eq:diff-case3-formula} follows using \eqref{eq:diff-case3-hypothesis-roots} to express all the roots in terms of $\alpha$, $\beta$, $\gamma$.
\epf

\begin{lema}\label{lem:diff-case4}
Let $\alpha < \beta <\gamma < \tau < \eta <\delta $ be positive roots
such that $N_{\gamma}=2$ and the relations among the corresponding root vectors take the form
\begin{align}\label{eq:diff-case4-hypothesis}
\begin{aligned}
x_{\alpha}x_{\delta} &= q_{\alpha \delta} x_{\delta}x_{\alpha} +\Bsj_1 x_{\eta}x_{\gamma}, &
x_{\beta}x_{\delta} &= q_{\beta\delta}  x_{\delta}x_{\beta} + \Bsj_2 x_{\eta} x_{\tau},
\\
x_{\gamma}x_{\delta} &= q_{\gamma \delta} x_{\delta}x_{\gamma} +\Bsj_3 x_{\eta}^2 x_{\tau}, &
x_{\alpha}x_{\tau} &= q_{\alpha\tau} x_{\tau}x_{\alpha} + \Bsj_4 x_{\gamma} x_{\beta},
\\
x_{\beta}x_{\eta} &= q_{\beta\eta} x_{\eta}x_{\beta} +\Bsj_5 x_{\gamma},
\end{aligned}
\end{align}
for some scalars $\Bsj_j\in\Bbbk$ and the other pairs of root vectors $q$-commute.
Then, for all $n\geq 0$, 
\begin{align}\label{eq:diff-case4-formula}
&\begin{aligned}
d( & \x_{\alpha} \x_{\beta} \x_{\gamma}^{n} \x_{\delta}\ot 1)  = 
\x_{\alpha} \x_{\beta}\x_{\gamma}^{n} \ot x_{\delta} 
-q_{\gamma\delta}\x_{\alpha} \x_{\beta}\x_{\gamma}^{n-1} \x_{\delta} \ot x_{\gamma} 
\\ &
- (-q_{\beta\gamma})^{n} q_{\beta\delta} \x_{\alpha} \x_{\gamma}^{n}\x_{\delta} \ot  x_{\beta} +q_{\alpha \beta} (-q_{\alpha\gamma})^{n} q_{\alpha \delta} \x_{\beta}\x_{\gamma}^{n} \x_{\delta} \ot x_{\alpha}
\\ &
-\big(q_{\gamma\eta}^{n-1}(n)_{-\frac{q_{\beta\gamma}}{q_{\gamma\eta}}} \Bsj_3\Bsj_5 +(-q_{\beta\gamma})^{n}\Bsj_2 \big)
\x_{\alpha} \x_{\gamma}^{n} \x_{\eta} \ot x_{\tau}
\\ & 
+q_{\alpha\beta}(-q_{\alpha\gamma})^{n}\Bsj_1 \Bsj_1 \x_{\beta} \x_{\gamma}^{n} \x_{\eta} \ot x_{\gamma}
-\Bsj_3 \x_{\alpha} \x_{\beta}\x_{\gamma}^{n-1} \x_{\eta} \ot x_{\eta}x_{\tau}
\\ &
- \frac{q_{\alpha\beta}}{q_{\alpha\gamma}^n q_{\gamma\beta}^n} \left( \sum_{k=0}^{n} (-\widetilde{q}_{\alpha\gamma})^k (k+1)_{\widetilde{q}_{\beta\gamma}} \right) \Bsj_1\Bsj_5 \x_{\gamma}^{n+2} \ot 1.
\end{aligned}
\end{align}
\end{lema}

\bigbreak
Notice that the equalities in \eqref{eq:diff-case4-hypothesis} force
\begin{align}\label{eq:diff-case4-hypothesis-roots}
\alpha + \delta &=\gamma+\eta, &  \beta+\delta &= \eta+\tau, &
\gamma+\delta &= 2\eta+\tau, & \eta+\beta &= \gamma.
\end{align}
Thus the following equality also holds: $2\gamma=\alpha+\beta+\delta$.

\pf
A recursive computation on $n\in\N$ shows that
\begin{align}\label{eq:diff-case4-gamma-delta}
d(\x_{\gamma}^{n}\x_{\delta}\ot 1)&= \x_{\gamma}^{n} \ot x_{\delta} -q_{\gamma\delta} \x_{\gamma}^{n-1} \x_{\delta}\ot x_{\gamma}
-\Bsj_3 \x_{\gamma}^{n-1} \x_{\eta}\ot x_{\eta}x_{\tau}.
\end{align}

The root vectors corresponding to $\beta<\gamma<\eta$ satisfy \eqref{eq:diff2-hypothesis}, so by Lemma \ref{lem:diff2},
\begin{align}\label{eq:diff4-betagammadelta}
\begin{aligned}
d(\x_{\beta}& \x_{\gamma}^{n} \x_{\eta}\ot 1)  =
\x_{\beta} \x_{\gamma}^{n}\ot x_{\eta} 
- q_{\gamma\eta}\x_{\beta} \x_{\gamma}^{n-1} \x_{\eta}\ot x_{\gamma} 
- q_{\beta\eta}(-q_{\beta\gamma})^{n} \x_{\gamma}^{n}\x_{\eta}\ot x_{\beta}
\\ &+ \Bsj_5 q_{\gamma\eta}^{n} (n+1)_{-\frac{q_{\beta\gamma}}{q_{\gamma\eta}}} \x_{\gamma}^{n+1} \ot 1.
\end{aligned}
\end{align}

We need more auxiliary results:
\begin{align}\label{eq:diff-case4-auxiliar-1}
&\begin{aligned}
d(\x_{\alpha} & \x_{\beta} \x_{\gamma}^{n} \x_{\eta}\ot 1)  =
\x_{\alpha} \x_{\beta}\x_{\gamma}^{n} \ot x_{\eta} 
-q_{\gamma\eta}\x_{\alpha} \x_{\beta}\x_{\gamma}^{n-1} \x_{\eta} \ot x_{\gamma} 
\\ &
- (-q_{\beta\gamma})^{n} q_{\beta\eta} \x_{\alpha} \x_{\gamma}^{n}\x_{\eta} \ot  x_{\beta} +q_{\alpha \beta} (-q_{\alpha\gamma})^{n} q_{\alpha \eta} \x_{\beta}\x_{\gamma}^{n} \x_{\eta} \ot x_{\alpha}
\\ &
-q_{\gamma\eta}^{n}(n+1)_{-\frac{q_{\beta\gamma}}{q_{\gamma\eta}}} \Bsj_5 \x_{\alpha} \x_{\gamma}^{n+1} \ot 1,
\end{aligned}
\\ \label{eq:diff-case4-auxiliar-2}
&\begin{aligned}
d(\x_{\alpha} & \x_{\gamma}^{n} \x_{\delta}\ot 1) =
\x_{\alpha}\x_{\gamma}^{n}\ot x_{\delta} 
-q_{\gamma \delta} \x_{\alpha}\x_{\gamma}^{n-1}\x_{\delta}\ot x_{\gamma}
-(-q_{\alpha\gamma})^{n} q_{\alpha \delta} \x_{\gamma}^{n}\x_{\delta} \ot x_{\alpha}
\\ & - \Bsj_3 \x_{\alpha}\x_{\gamma}^{n-1}\x_{\eta}\ot x_{\eta} x_{\tau}
-(-q_{\alpha\gamma})^{n} \Bsj_1 \x_{\gamma}^{n}\x_{\eta}\ot x_{\gamma},
\end{aligned}
\\ \label{eq:diff-case4-auxiliar-3}
&\begin{aligned}
d(\x_{\beta} & \x_{\gamma}^{n} \x_{\delta}\ot 1) =
\x_{\beta}\x_{\gamma}^{n}\ot x_{\delta} 
-q_{\gamma \delta} \x_{\beta}\x_{\gamma}^{n-1}\x_{\delta}\ot x_{\gamma}
-(-q_{\beta\gamma})^{n} q_{\beta\delta} \x_{\gamma}^{n}\x_{\delta} \ot x_{\beta}
\\ & - \Bsj_3 \x_{\beta}\x_{\gamma}^{n-1}\x_{\eta}\ot x_{\eta} x_{\tau}
- \big(
(-q_{\beta\gamma})^{n}\Bsj_2-\Bsj_3\Bsj_5 q_{\gamma\eta}^{n} (n)_{-\frac{q_{\beta\gamma}}{q_{\gamma\eta}}} \big) \x_{\gamma}^{n}\x_{\eta} \ot x_{\tau}.
\end{aligned}
\end{align}

We start with the proof of \eqref{eq:diff-case4-auxiliar-1} by induction on $n$. For $n=0$,
\begin{align*}
d(\x_{\alpha} & \x_{\beta} \x_{\eta}\ot 1) = \x_{\alpha}\x_{\beta} \ot x_{\eta} - s\Big( 
q_{\beta\eta} \x_{\alpha} \ot x_{\eta}x_{\beta}
+ \Bsj_5 \x_{\alpha} \ot x_{\gamma}
- q_{\alpha\beta}q_{\alpha\eta} \x_{\beta} \ot x_{\eta}x_{\alpha} \Big)
\\
& = \x_{\alpha}\x_{\beta} \ot x_{\eta}
-q_{\beta\eta} \x_{\alpha}\x_{\eta} \ot x_{\beta}
+q_{\alpha\beta}q_{\alpha\eta} \x_{\beta}\x_{\eta} \ot x_{\alpha}
-\Bsj_5 \x_{\alpha}\x_{\gamma} \ot 1.
\end{align*}
Now assume that \eqref{eq:diff-case4-auxiliar-1} holds for $n$. By Remark \ref{rem:differential-q-commute}, inductive hypothesis 
and \eqref{eq:diff4-betagammadelta},
\begin{align*}
d(\x_{\alpha} & \x_{\beta} \x_{\gamma}^{n+1} \x_{\eta}\ot 1)  =
\x_{\alpha} \x_{\beta} \x_{\gamma}^{n+1} \ot x_{\eta}
-s\Big( 
q_{\gamma \eta} \x_{\alpha} \x_{\beta} \x_{\gamma}^{n} \ot  x_{\eta}x_{\gamma} 
\\ & \quad 
+ (-q_{\beta\gamma})^{n+1}\x_{\alpha} \x_{\gamma}^{n+1} \ot (q_{\beta\eta}  x_{\eta}x_{\beta} + \Bsj_5 x_{\gamma})
- q_{\alpha\beta}(-q_{\alpha\gamma})^{n+1}q_{\alpha\eta} \x_{\beta} \x_{\gamma}^{n+1} \ot x_{\eta}x_{\alpha}
\Big)
\\
& = \x_{\alpha} \x_{\beta} \x_{\gamma}^{n+1} \ot x_{\eta}
-q_{\gamma \eta} \x_{\alpha} \x_{\beta} \x_{\gamma}^{n}\x_{\eta} \ot  x_{\gamma}
-s\Big( 
(-q_{\beta\gamma})^{n} q_{\beta\eta}q_{\gamma\eta}q_{\beta\gamma} \x_{\alpha} \x_{\gamma}^{n}\x_{\eta} \ot x_{\gamma}x_{\beta}
\\ & \quad
-q_{\alpha \beta} (-q_{\alpha\gamma})^{n} q_{\alpha \eta}q_{\gamma\eta}q_{\alpha\gamma} \x_{\beta}\x_{\gamma}^{n} \x_{\eta} \ot x_{\gamma}x_{\alpha}
- q_{\alpha\beta}(-q_{\alpha\gamma})^{n+1}q_{\alpha\eta} \x_{\beta} \x_{\gamma}^{n+1} \ot x_{\eta}x_{\alpha}
\\ & \quad 
+ (-q_{\beta\gamma})^{n+1}q_{\beta\eta} \x_{\alpha} \x_{\gamma}^{n+1} \ot x_{\eta}x_{\beta}
+(q_{\gamma\eta}^{n+1}(n+1)_{-\frac{q_{\beta\gamma}}{q_{\gamma\eta}}}+ (-q_{\beta\gamma})^{n+1})\Bsj_5 \x_{\alpha} \x_{\gamma}^{n+1} \ot x_{\gamma} \Big)
\\
& = \x_{\alpha} \x_{\beta} \x_{\gamma}^{n+1} \ot x_{\eta}
-q_{\gamma \eta} \x_{\alpha} \x_{\beta} \x_{\gamma}^{n}\x_{\eta} \ot  x_{\gamma}
-q_{\gamma\eta}^{n+1}(n+2)_{-\frac{q_{\beta\gamma}}{q_{\gamma\eta}}}\Bsj_5 \x_{\alpha} \x_{\gamma}^{n+2} \ot 1
\\ & \quad
- (-q_{\beta\gamma})^{n+1}q_{\beta\eta} \x_{\alpha} \x_{\gamma}^{n+1}\x_{\eta} \ot x_{\beta}
+ q_{\alpha\beta}(-q_{\alpha\gamma})^{n+1}q_{\alpha\eta} \x_{\beta} \x_{\gamma}^{n+1}\x_{\eta} \ot x_{\alpha}.
\end{align*}

\bigbreak

Next we prove \eqref{eq:diff-case4-auxiliar-2} by induction on $n$. For $n=0$,
\begin{align*}
d(\x_{\alpha} & \x_{\delta}\ot 1) = \x_{\alpha}\ot x_{\delta} - q_{\alpha\delta} \x_{\delta} \ot x_{\alpha} - \Bsj_1 \x_{\eta} \ot x_{\gamma}.
\end{align*}

Now assume that \eqref{eq:diff-case4-auxiliar-2} holds for $n$. Using Remark \ref{rem:differential-q-commute} three times, inductive hypothesis and \eqref{eq:diff-case4-gamma-delta}: 
\begin{align*}
d(\x_{\alpha} & \x_{\gamma}^{n+1} \x_{\delta}\ot 1)  =
\x_{\alpha} \x_{\gamma}^{n+1} \ot x_{\delta}
-s\Big( \x_{\alpha} \x_{\gamma}^{n} \ot (q_{\gamma\delta} x_{\delta}x_{\gamma} +\Bsj_3 x_{\eta}^2 x_{\tau})
\\ & \quad 
+ (-q_{\alpha\gamma})^{n+1}\x_{\gamma}^{n+1} \ot (q_{\alpha \delta} x_{\delta}x_{\alpha} +\Bsj_1 x_{\eta}x_{\gamma})
\Big)
\\
& = \x_{\alpha} \x_{\gamma}^{n+1} \ot x_{\delta}
-\Bsj_3 \x_{\alpha}\x_{\gamma}^{n}\x_{\eta} \ot x_{\eta} x_{\tau}
-q_{\gamma\delta}\x_{\alpha} \x_{\gamma}^{n}\x_{\delta} \ot x_{\gamma}
-s\Big( (-q_{\alpha\gamma})^{n+1}\Bsj_1 \x_{\gamma}^{n+1} \ot x_{\eta}x_{\gamma}
\\ & \quad
+ q_{\alpha \delta}(-q_{\alpha\gamma})^{n+1} \x_{\gamma}^{n+1} \ot x_{\delta}x_{\alpha}
-(-q_{\alpha\gamma})^{n+1} q_{\alpha \delta}q_{\gamma\delta} \x_{\gamma}^{n}\x_{\delta} \ot x_{\gamma}x_{\alpha} 
\\ & \quad
+q_{\alpha\eta}^2(-q_{\alpha\eta})^nq_{\alpha\tau} \Bsj_3 \x_{\gamma}^{n}\x_{\eta} \ot x_{\eta}x_{\tau}x_{\alpha}
+q_{\alpha\eta}^2(-q_{\alpha\gamma})^n\Bsj_3\Bsj_4 \x_{\gamma}^{n}\x_{\eta} \ot x_{\eta}x_{\gamma}x_{\beta}
\Big)
\\
& = \x_{\alpha} \x_{\gamma}^{n+1} \ot x_{\delta}
-\Bsj_3 \x_{\alpha}\x_{\gamma}^{n}\x_{\eta} \ot x_{\eta} x_{\tau}
-q_{\gamma\delta}\x_{\alpha} \x_{\gamma}^{n}\x_{\delta} \ot x_{\gamma}
-(-q_{\alpha\gamma})^{n+1}\Bsj_1 \x_{\gamma}^{n+1}\x_{\eta} \ot x_{\gamma}
\\ & \quad
-q_{\alpha\delta}(-q_{\alpha\gamma})^{n+1} \x_{\gamma}^{n+1} \x_{\delta} \ot x_{\alpha}.
\end{align*}

\bigbreak

Now we prove \eqref{eq:diff-case4-auxiliar-3} by induction on $n$. For $n=0$,
\begin{align*}
d(\x_{\beta} & \x_{\delta}\ot 1) = \x_{\beta}\ot x_{\delta} - q_{\beta\delta} \x_{\delta} \ot x_{\beta} - \Bsj_2 \x_{\eta} \ot x_{\tau}.
\end{align*}


We assume that \eqref{eq:diff-case4-auxiliar-3} holds for $n$. Using 
Remark \ref{rem:differential-q-commute} three times, \eqref{eq:diff4-betagammadelta}, inductive hypothesis and \eqref{eq:diff-case4-gamma-delta}: 
\begin{align*}
d(\x_{\beta} & \x_{\gamma}^{n+1} \x_{\delta}\ot 1)  =
\x_{\beta} \x_{\gamma}^{n+1} \ot x_{\delta}
-s\Big( 
q_{\gamma\delta} \x_{\beta}\x_{\gamma}^{n} \ot x_{\delta}x_{\gamma}
+\Bsj_3 \x_{\beta}\x_{\gamma}^{n} \ot x_{\eta}^2x_{\tau}
\\ & \quad 
+ q_{\beta\delta}(-q_{\beta\gamma})^{n+1} \x_{\gamma}^{n+1} \ot x_{\delta}x_{\beta}
+ (-q_{\beta\gamma})^{n+1}\Bsj_2 \x_{\gamma}^{n+1} \ot x_{\eta}x_{\tau}
\Big)
\\
&=
\x_{\beta} \x_{\gamma}^{n+1} \ot x_{\delta}
-q_{\gamma\delta} \x_{\beta}\x_{\gamma}^{n}\x_{\delta} \ot x_{\gamma}
-\Bsj_3 \x_{\beta}\x_{\gamma}^{n}\x_{\eta} \ot x_{\eta}x_{\tau}
-s\Big(d_n q_{\gamma\delta} \x_{\gamma}^{n}\x_{\eta}\ot x_{\tau}x_{\gamma}
\\ & \quad
-q_{\gamma\delta}(-q_{\beta\gamma})^{n+1} q_{\beta\delta} \x_{\gamma}^{n}\x_{\delta} \ot x_{\gamma}x_{\beta}
+ q_{\beta\delta}(-q_{\beta\gamma})^{n+1} \x_{\gamma}^{n+1} \ot x_{\delta}x_{\beta}
\\ & \quad
+q_{\beta\eta}^2q_{\beta\tau}(-q_{\beta\gamma})^{n}\Bsj_3 \x_{\gamma}^{n}\x_{\eta}\ot x_{\eta}x_{\tau}x_{\beta}
+q_{\beta\eta}(-q_{\beta\gamma})^{n}q_{\gamma\tau} \Bsj_3\Bsj_5 \x_{\gamma}^{n}\x_{\eta}\ot x_{\tau}x_{\gamma}
\\ & \quad 
+ \big( (-q_{\beta\gamma})^{n+1}\Bsj_2-\Bsj_3\Bsj_5 q_{\gamma\eta}^{n} (n+1)_{-\frac{q_{\beta\gamma}}{q_{\gamma\eta}}} \big) \x_{\gamma}^{n+1} \ot x_{\eta}x_{\tau}
\Big)
\\
&=
\x_{\beta} \x_{\gamma}^{n+1} \ot x_{\delta}
-q_{\gamma\delta} \x_{\beta}\x_{\gamma}^{n}\x_{\delta} \ot x_{\gamma}
-\Bsj_3 \x_{\beta}\x_{\gamma}^{n}\x_{\eta} \ot x_{\eta}x_{\tau}
- q_{\beta\delta}(-q_{\beta\gamma})^{n+1} \x_{\gamma}^{n+1}\x_{\delta} \ot x_{\beta}
\\ & \quad
- \big( (-q_{\beta\gamma})^{n+1}\Bsj_2-\Bsj_3\Bsj_5 q_{\gamma\eta}^{n+1} (n+1)_{-\frac{q_{\beta\gamma}}{q_{\gamma\eta}}} \big) \x_{\gamma}^{n+1}\x_{\eta} \ot x_{\tau}.
\end{align*}

\bigbreak

Finally we prove \eqref{eq:diff-case4-formula}. To do so,
we prove that there exist $c_n,d_n,e_n\in\Bbbk$ such that
\begin{align*}
&\begin{aligned}
d( & \x_{\alpha} \x_{\beta} \x_{\gamma}^{n} \x_{\delta}\ot 1)  = 
\x_{\alpha} \x_{\beta}\x_{\gamma}^{n} \ot x_{\delta} 
-q_{\gamma\delta}\x_{\alpha} \x_{\beta}\x_{\gamma}^{n-1} \x_{\delta} \ot x_{\gamma} 
- (-q_{\beta\gamma})^{n} q_{\beta\delta} \x_{\alpha} \x_{\gamma}^{n}\x_{\delta} \ot  x_{\beta} 
\\ &
+q_{\alpha \beta} (-q_{\alpha\gamma})^{n} q_{\alpha \delta} \x_{\beta}\x_{\gamma}^{n} \x_{\delta} \ot x_{\alpha}-c_n \, \x_{\alpha} \x_{\gamma}^{n} \x_{\eta} \ot x_{\tau}
-\Bsj_3 \x_{\alpha} \x_{\beta}\x_{\gamma}^{n-1} \x_{\eta} \ot x_{\eta}x_{\tau}
\\ &
+ d_n \, \x_{\gamma}^{n} \x_{\eta} \ot x_{\gamma}
- e_n \Bsj_1\Bsj_5 \x_{\gamma}^{n+2} \ot 1.
\end{aligned}
\end{align*}

For $n=0$,
\begin{align*}
d( &\x_{\alpha} \x_{\beta} \x_{\delta} \ot 1)  = 
\x_{\alpha} \x_{\beta} \ot x_{\delta}-
s \big(\x_{\alpha} \ot (q_{\beta\delta}  x_{\delta}x_{\beta} + \Bsj_2 x_{\eta} x_{\tau}) 
- q_{\alpha \beta} \x_{\beta} \ot x_{\alpha} x_{\delta}\big)
\\
& = \x_{\alpha} \x_{\beta} \ot x_{\delta}
-q_{\beta\delta} \x_{\alpha}\x_{\delta} \ot x_{\beta}
-\Bsj_2 \x_{\alpha}\x_{\eta} \ot  x_{\tau}
+s \Big(
q_{\alpha \beta}q_{\alpha \delta} \x_{\beta} \ot x_{\delta}x_{\alpha}
+ q_{\alpha \beta}\Bsj_1 \x_{\beta} \ot x_{\eta}x_{\gamma}
\\ & \quad
-(q_{\beta\delta} \Bsj_1+q_{\alpha\eta} \Bsj_2\Bsj_4) \x_{\eta} \ot x_{\gamma} x_{\beta}
-q_{\alpha\tau}q_{\alpha\eta} \Bsj_2 \x_{\eta} \ot  x_{\tau}x_{\alpha} 
-q_{\alpha\beta}q_{\alpha\delta}q_{\beta\delta} \x_{\delta}\ot x_{\beta}x_{\alpha} \Big)
\\
& = \x_{\alpha} \x_{\beta} \ot x_{\delta}
-q_{\beta\delta} \x_{\alpha}\x_{\delta} \ot x_{\beta}
-\Bsj_2 \x_{\alpha}\x_{\eta} \ot  x_{\tau}
+q_{\alpha \beta}q_{\alpha \delta} \x_{\beta}\x_{\delta} \ot x_{\alpha}
+q_{\alpha \beta}\Bsj_1 \x_{\beta}\x_{\eta} \ot x_{\gamma}
\\ & \quad + q_{\alpha \beta}\Bsj_1\Bsj_5 \x_{\gamma}^2 \ot 1
-s \Big((q_{\beta\delta} \Bsj_1+q_{\alpha\eta} \Bsj_2\Bsj_4-q_{\beta\eta}q_{\beta\gamma}q_{\alpha \beta}\Bsj_1) s(x_{\eta}x_{\gamma} x_{\beta}) \Big).
\end{align*}

We assume that \eqref{eq:diff-case4-formula} holds for $n$. Using Remark \ref{rem:differential-q-commute} twice, \eqref{eq:diff-case4-auxiliar-1}, inductive hypothesis, \eqref{eq:diff-case4-auxiliar-2}, \eqref{eq:diff4-betagammadelta}, \eqref{eq:diff-case4-auxiliar-3}
\begin{align*}
d( & \x_{\alpha} \x_{\beta} \x_{\gamma}^{n+1} \x_{\delta}\ot 1)  = 
\x_{\alpha} \x_{\beta} \x_{\gamma}^{n+1} \ot x_{\delta}
-s\Big( 
q_{\gamma \delta} \x_{\alpha} \x_{\beta} \x_{\gamma}^{n} \ot  x_{\delta}x_{\gamma}
+\Bsj_3 \x_{\alpha} \x_{\beta} \x_{\gamma}^{n} \ot x_{\eta}^2 x_{\tau}
\\ & \quad 
+ (-q_{\beta\gamma})^{n+1}\x_{\alpha} \x_{\gamma}^{n+1} \ot x_{\beta}x_{\delta}
- q_{\alpha\beta}(-q_{\alpha\gamma})^{n+1} \x_{\beta} \x_{\gamma}^{n+1} \ot x_{\alpha}x_{\delta} \Big)
\\ & = 
\x_{\alpha} \x_{\beta} \x_{\gamma}^{n+1} \ot x_{\delta}
-q_{\gamma\delta} \x_{\alpha}\x_{\beta}\x_{\gamma}^{n}\x_{\delta} \ot x_{\gamma}
-\Bsj_3 \x_{\alpha}\x_{\beta}\x_{\gamma}^{n}\x_{\eta} \ot x_{\eta}x_{\tau}
-s\Big(q_{\gamma\delta}e_n \Bsj_1\Bsj_5 \x_{\gamma}^{n+2} \ot x_{\gamma}
\\ & \quad
+q_{\gamma\delta}(-q_{\beta\gamma})^{n}q_{\beta\gamma}q_{\beta\delta} \x_{\alpha}\x_{\gamma}^{n}\x_{\delta} \ot x_{\gamma}x_{\beta} 
+q_{\gamma\delta}q_{\alpha \beta}(-q_{\alpha\gamma})^{n+1} q_{\alpha\delta} \x_{\beta}\x_{\gamma}^{n}\x_{\delta} \ot x_{\gamma}x_{\alpha}
\\ & \quad 
+ (-q_{\beta\gamma})^{n+1}q_{\beta\delta} \x_{\alpha} \x_{\gamma}^{n+1} \ot x_{\delta}x_{\beta}
+q_{\gamma\delta}c_n \Bsj_2 \x_{\alpha} \x_{\gamma}^{n} \x_{\eta} \ot x_{\tau}x_{\gamma}
\\ & \quad
+(-q_{\beta\gamma})^{n}q_{\beta\eta}^2q_{\beta\tau}\Bsj_3 \x_{\alpha} \x_{\gamma}^{n}\x_{\eta} \ot x_{\eta}x_{\tau}x_{\beta}
+(-q_{\beta\gamma})^{n}q_{\beta\eta}q_{\gamma\tau}\Bsj_3\Bsj_5 \x_{\alpha}\x_{\gamma}^{n}\x_{\eta} \ot x_{\tau}x_{\gamma}
\\ & \quad
-q_{\alpha\beta}(-q_{\alpha\gamma})^{n}q_{\alpha\eta}^2q_{\alpha\tau}\Bsj_3 \x_{\beta}\x_{\gamma}^{n}\x_{\eta} \ot x_{\eta}x_{\tau}x_{\alpha}
-q_{\alpha \beta}(-q_{\alpha\gamma})^{n}q_{\alpha\eta}^2\Bsj_3\Bsj_4 \x_{\beta}\x_{\gamma}^{n}\x_{\eta} \ot x_{\eta}x_{\gamma}x_{\beta}
\\ & \quad
+\big(q_{\gamma\eta}^{n}(n+1)_{-\frac{q_{\beta\gamma}}{q_{\gamma\eta}}} \Bsj_3\Bsj_5 +(-q_{\beta\gamma})^{n+1}\Bsj_2 \big) \x_{\alpha} \x_{\gamma}^{n+1} \ot x_{\eta}x_{\tau}
\\ & \quad 
- q_{\alpha\beta}(-q_{\alpha\gamma})^{n+1}\Bsj_1 \x_{\beta}\x_{\gamma}^{n+1} \ot x_{\eta}x_{\gamma} 
- q_{\alpha\beta}(-q_{\alpha\gamma})^{n+1}q_{\alpha \delta} \x_{\beta}\x_{\gamma}^{n+1} \ot x_{\delta}x_{\alpha}\Big)
\\ & = 
\x_{\alpha} \x_{\beta} \x_{\gamma}^{n+1} \ot x_{\delta}
-q_{\gamma\delta} \x_{\alpha}\x_{\beta}\x_{\gamma}^{n}\x_{\delta} \ot x_{\gamma}
-(-q_{\beta\gamma})^{n+1}q_{\beta\delta} \x_{\alpha} \x_{\gamma}^{n+1}\x_{\delta} \ot x_{\beta}
\\ & \quad
-\Bsj_3 \x_{\alpha}\x_{\beta}\x_{\gamma}^{n}\x_{\eta} \ot x_{\eta}x_{\tau}
-\big(q_{\gamma\eta}^{n}(n+1)_{-\frac{q_{\beta\gamma}}{q_{\gamma\eta}}} \Bsj_3\Bsj_5 +(-q_{\beta\gamma})^{n+1}\Bsj_2 \big) \x_{\alpha} \x_{\gamma}^{n+1}\x_{\eta} \ot x_{\tau}
\\ & \quad
-s\Big(
- q_{\alpha\beta}(-q_{\alpha\gamma})^{n+1}\Bsj_1 \x_{\beta}\x_{\gamma}^{n+1} \ot x_{\eta}x_{\gamma} 
- q_{\alpha\beta}(-q_{\alpha\gamma})^{n+1}q_{\alpha \delta} \x_{\beta}\x_{\gamma}^{n+1} \ot x_{\delta}x_{\alpha}
\\ & \quad
+\big(q_{\gamma\eta}^{n}(n+1)_{-\frac{q_{\beta\gamma}}{q_{\gamma\eta}}} \Bsj_3\Bsj_5 +(-q_{\beta\gamma})^{n+1}\Bsj_2 \big) 
q_{\alpha\eta}(-q_{\alpha\gamma})^{n+1} \x_{\gamma}^{n+1}\x_{\eta} 
\ot x_{\alpha}x_{\tau}
\\ & \quad
+q_{\gamma\delta}e_n \Bsj_1\Bsj_5 \x_{\gamma}^{n+2} \ot x_{\gamma}
+q_{\alpha\beta}q_{\alpha\gamma}^{n+1} q_{\alpha \delta} q_{\beta\gamma}^{n+1}q_{\beta\delta} \x_{\gamma}^{n+1}\x_{\delta} \ot x_{\beta}x_{\alpha}
\\ & \quad
+q_{\gamma\delta}q_{\alpha \beta}(-q_{\alpha\gamma})^{n+1} q_{\alpha\delta} \x_{\beta}\x_{\gamma}^{n}\x_{\delta} \ot x_{\gamma}x_{\alpha}
+(-q_{\beta\gamma})^{n+1}q_{\beta\delta}(-q_{\alpha\gamma})^{n+1} \Bsj_1 \x_{\gamma}^{n+1}\x_{\eta}\ot x_{\gamma}x_{\beta}
\\ & \quad
-q_{\alpha\beta}(-q_{\alpha\gamma})^{n}q_{\alpha\eta}^2q_{\alpha\tau}\Bsj_3 \x_{\beta}\x_{\gamma}^{n}\x_{\eta} \ot x_{\eta}x_{\tau}x_{\alpha}
-q_{\alpha \beta}(-q_{\alpha\gamma})^{n}q_{\alpha\eta}^2\Bsj_3\Bsj_4 \x_{\beta}\x_{\gamma}^{n}\x_{\eta} \ot x_{\eta}x_{\gamma}x_{\beta}  
\Big)
\\ & = 
\x_{\alpha} \x_{\beta} \x_{\gamma}^{n+1} \ot x_{\delta}
-q_{\gamma\delta} \x_{\alpha}\x_{\beta}\x_{\gamma}^{n}\x_{\delta} \ot x_{\gamma}
-(-q_{\beta\gamma})^{n+1}q_{\beta\delta} \x_{\alpha} \x_{\gamma}^{n+1}\x_{\delta} \ot x_{\beta}
\\ & \quad
-\Bsj_3 \x_{\alpha}\x_{\beta}\x_{\gamma}^{n}\x_{\eta} \ot x_{\eta}x_{\tau}
-\big(q_{\gamma\eta}^{n}(n+1)_{-\frac{q_{\beta\gamma}}{q_{\gamma\eta}}} \Bsj_3\Bsj_5 +(-q_{\beta\gamma})^{n+1}\Bsj_2 \big) \x_{\alpha} \x_{\gamma}^{n+1}\x_{\eta} \ot x_{\tau}
\\ & \quad
+q_{\alpha\beta}(-q_{\alpha\gamma})^{n+1}\Bsj_1 \x_{\beta}\x_{\gamma}^{n+1}\x_{\eta} \ot x_{\gamma}
+q_{\alpha\beta}(-q_{\alpha\gamma})^{n+1}q_{\alpha \delta} \x_{\beta}\x_{\gamma}^{n+1}\x_{\delta} \ot x_{\alpha}
\\ & \quad
-s\Big(
\big(q_{\alpha\beta}(-q_{\alpha\gamma})^{n+1}q_{\gamma\eta}^{n+1} (n+2)_{-\frac{q_{\beta\gamma}}{q_{\gamma\eta}}}+q_{\gamma\delta}e_n \big) \Bsj_1\Bsj_5 \x_{\gamma}^{n+2} \ot x_{\gamma}
\\ & \quad
+\big(q_{\gamma\eta}^{n}(n+1)_{-\frac{q_{\beta\gamma}}{q_{\gamma\eta}}} \Bsj_3\Bsj_5 +(-q_{\beta\gamma})^{n+1}\Bsj_2 \big) 
q_{\alpha\eta}(-q_{\alpha\gamma})^{n+1}\Bsj_4 s(\x_{\gamma}^{n+1} \ot x_{\eta}x_{\gamma} x_{\beta})
\\ & \quad
-q_{\alpha \beta}(-q_{\alpha\gamma})^{n}q_{\alpha\eta}^2\Bsj_3\Bsj_4 s(\x_{\beta}\x_{\gamma}^{n}\ot x_{\eta}^2x_{\gamma}x_{\beta})  
\Big).
\end{align*}
Hence the inductive step follows since $s^2=0$ and $s\big(\x_{\gamma}^{n+2} \ot x_{\gamma}\big) =\x_{\gamma}^{n+3} \ot 1$; for the last step we use the equalities
\begin{align*}
s(\x_{\gamma}^{n+1} \ot x_{\eta}x_{\gamma} x_{\beta}) &= \x_{\gamma}^{n+1}\x_{\eta} \ot x_{\gamma} x_{\beta}, &
s(\x_{\beta}\x_{\gamma}^{n}\ot x_{\eta}^2x_{\gamma}x_{\beta}) &= \x_{\beta}\x_{\gamma}^{n}\x_{\eta}\ot x_{\eta}x_{\gamma}x_{\beta},
\end{align*}
which follow since $x_{\gamma}^2=0$ and $x_{\gamma}x_{\eta} = q_{\gamma\eta}x_{\eta}x_{\gamma}$.  The formula for $c_n$, $d_n$ is explicit, while for the $e_n$'s we have the recursive expression: $e_0=q_{\alpha\beta}$,
\begin{align*}
e_{n+1} &= q_{\alpha\beta}q_{\alpha\gamma}^{n+1} q_{\gamma\beta}^{-n-1} (n+2)_{\widetilde{q}_{\beta\gamma}} + q_{\gamma\alpha}^{-1}q_{\gamma\beta}^{-1}e_n, & &n\ge 0,
\end{align*}
where we use \eqref{eq:diff-case4-hypothesis-roots} to express $\delta$ and $\eta$ in terms of $\alpha$, $\beta$ and $\gamma$.
\epf

\begin{lema}\label{lem:diff-case5}
Let $\alpha < \beta < \delta < \gamma < \tau < \varphi < \eta $ be positive roots such that $N_{\gamma}=2$ and the relations among the corresponding root vectors take the form
\begin{align}\label{eq:diff-case5-hypothesis}
\begin{aligned}
x_{\alpha}x_{\varphi} &= q_{\alpha\varphi} x_{\varphi}x_{\alpha} +\Bsj_1 x_{\gamma}, \\
x_{\beta}x_{\tau} &= q_{\beta\tau} x_{\tau}x_{\beta} + \Bsj_2 x_{\gamma}, &
x_{\delta}x_{\eta} &= q_{\delta\eta} x_{\eta}x_{\delta} + \Bsj_3 x_{\varphi} x_{\tau} x_{\gamma},
\end{aligned}
\end{align}
for some scalars $\Bsj_i$, and the other pairs of root vectors $q$-commute.
Then, for all $n\geq 0$, 
\begin{align}\label{eq:diff-case5-formula}
\begin{aligned}
d( & \x_{\alpha} \x_{\beta} \x_{\delta}\x_{\gamma}^{n}\x_{\eta}\ot 1)  = 
\x_{\alpha}\x_{\beta}\x_{\delta}\x_{\gamma}^{n} \ot x_{\eta}
- q_{\gamma\eta}\x_{\alpha}\x_{\beta}\x_{\delta} \x_{\gamma}^{n-1}\x_{\eta} \ot x_{\gamma}
\\ &
-(-q_{\delta\gamma})^n q_{\delta\eta}\x_{\alpha}\x_{\beta} \x_{\gamma}^{n}\x_{\eta} \ot x_{\delta}
+q_{\beta\delta}(-q_{\beta\gamma})^n q_{\beta\eta} \x_{\alpha}\x_{\delta}\x_{\gamma}^{n}\x_{\eta} \ot x_{\beta}
\\ &
-q_{\alpha\beta}q_{\alpha\delta}(-q_{\alpha\gamma})^n q_{\alpha\eta} \x_{\beta}\x_{\delta}\x_{\gamma}^{n}\x_{\eta} \ot x_{\alpha}
- (-q_{\delta\gamma})^{n} \Bsj_3 \x_{\alpha}\x_{\beta} \x_{\gamma}^{n}\x_{\varphi} \ot x_{\tau} x_{\gamma}
\\ &
- q_{\alpha\beta}q_{\gamma\varphi} (n+1)_{\widetilde{q}_{\alpha\gamma}}(-q_{\delta\gamma})^{n} \Bsj_1 \Bsj_3 \x_{\beta}\x_{\gamma}^{n+1}\x_{\tau} \ot  x_{\gamma}
-\frac{q_{\alpha\beta}q_{\gamma\eta}}{(-q_{\delta\gamma})^{n}}\coeff{\alpha\beta\delta\gamma}{n} \x_{\gamma}^{n+3} \ot 1,
\end{aligned}
\end{align}
where $\coeff{\alpha\beta\delta\gamma}{n}:=\sum\limits_{k=0}^{n} \widetilde{q}_{\delta\gamma}^{\, k} (k+1)_{\widetilde{q}_{\alpha\gamma}}
(k+2)_{\widetilde{q}_{\beta\gamma}}$, $n\in\N$.
\end{lema} 

\bigbreak
Notice that the equalities in \eqref{eq:diff-case5-hypothesis} force
\begin{align}\label{eq:diff-case5-hypothesis-roots}
\alpha + \varphi &=\gamma = \beta + \tau, &  \eta+\delta &= \gamma+\tau+\varphi.
\end{align}
Hence the following equality also holds: $3\gamma=\alpha+\beta+\delta+\eta$.

\pf
We may apply Lemma \ref{lem:diff2} to the 3-tuples $\alpha<\gamma<\varphi$ and $\beta<\gamma<\tau$ to obtain
\begin{align}\label{eq:diff-case5-auxiliar0}
\begin{aligned}
d(\x_{\alpha} \x_{\gamma}^{n} \x_{\varphi}\ot 1)  &=
\x_{\alpha}\x_{\gamma}^{n}\ot x_{\varphi} 
- q_{\gamma\varphi} \x_{\alpha}\x_{\gamma}^{n-1}\x_{\varphi}
\ot x_{\gamma} 
- (-q_{\alpha\gamma})^{n}q_{\alpha\varphi} \x_{\gamma}^{n} \x_{\varphi}\ot x_{\alpha} 
\\ & \quad + \Bsj_1 q_{\gamma\varphi} (n+1)_{\widetilde{q}_{\alpha\gamma}} \x_{\gamma}^{n+1}\ot 1,
\end{aligned}
\\ \label{eq:diff-case5-auxiliar00}
\begin{aligned}
d(\x_{\beta} \x_{\gamma}^{n} \x_{\tau}\ot 1)  &=
\x_{\beta}\x_{\gamma}^{n}\ot x_{\tau} 
- q_{\gamma\tau} \x_{\beta}\x_{\gamma}^{n-1}\x_{\tau}
\ot x_{\gamma} 
- (-q_{\beta\gamma})^{n}q_{\beta\tau} \x_{\gamma}^{n} \x_{\tau}\ot x_{\beta} 
\\ & \quad + \Bsj_2 q_{\gamma\tau} (n+1)_{\widetilde{q}_{\beta\gamma}} \x_{\gamma}^{n+1}\ot 1,
\end{aligned}
\end{align}
for all $n\geq 0$. The next step is to prove by induction on $n$ the following equalities:
\begin{align}\label{eq:diff-case5-auxiliar1}
&\begin{aligned}
d(\x_{\alpha}  & \x_{\beta} \x_{\gamma}^{n} \x_{\varphi}\ot 1)  =
\x_{\alpha}\x_{\beta} \x_{\gamma}^{n} \ot x_{\varphi}
-q_{\gamma\varphi} \x_{\alpha}\x_{\beta} \x_{\gamma}^{n-1} \x_{\varphi} \ot x_{\gamma} 
\\ &
-(-q_{\beta\gamma})^n q_{\beta\varphi} \x_{\alpha}\x_{\gamma}^{n} \x_{\varphi} \ot x_{\beta}
+q_{\alpha\beta}(-q_{\alpha\gamma})^n q_{\alpha\varphi} \x_{\beta}\x_{\gamma}^{n} \x_{\varphi} \ot x_{\alpha}
\\ &
+ q_{\alpha\beta}q_{\gamma\varphi}^n (n+1)_{\widetilde{q}_{\alpha\gamma}}
\Bsj_1 \x_{\beta}\x_{\gamma}^{n+1} \ot 1,
\end{aligned}
\\
\label{eq:diff-case5-auxiliar2}
&\begin{aligned}
d(\x_{\delta} & \x_{\gamma}^{n} \x_{\eta}\ot 1)  =
\x_{\delta}\x_{\gamma}^{n} \ot x_{\eta}
-q_{\gamma\eta}\x_{\delta}\x_{\gamma}^{n-1}\x_{\eta} \ot x_{\gamma}
-(-q_{\delta\gamma})^n q_{\delta\eta} \x_{\gamma}^{n}\x_{\eta} \ot x_{\delta} 
\\
& -(-q_{\delta\gamma})^{n} \Bsj_3 \x_{\gamma}^{n}\x_{\varphi} \ot x_{\tau} x_{\gamma},
\end{aligned}
\\
\label{eq:diff-case5-auxiliar3}
&\begin{aligned}
d(\x_{\alpha}  & \x_{\delta} \x_{\gamma}^{n} \x_{\eta}\ot 1)  =
\x_{\alpha}\x_{\delta}\x_{\gamma}^{n} \ot x_{\eta}
-q_{\gamma\eta}\x_{\alpha}\x_{\delta}\x_{\gamma}^{n-1}\x_{\eta} \ot x_{\gamma}
\\ & 
-(-q_{\delta\gamma})^n q_{\delta\eta} \x_{\alpha} \x_{\gamma}^{n} \x_{\eta} \ot x_{\delta} 
+q_{\alpha\delta}(-q_{\alpha\gamma})^n q_{\alpha\eta} \x_{\delta} \x_{\gamma}^{n} \x_{\eta} \ot x_{\alpha}
\\ &
-(-q_{\delta\gamma})^{n} \Bsj_3 \x_{\alpha} \x_{\gamma}^{n} \x_{\varphi} \ot x_{\tau} x_{\gamma}
-q_{\gamma\varphi} (n+1)_{\widetilde{q}_{\alpha\gamma}} (-q_{\delta\gamma})^{n} \Bsj _1\Bsj_3 \x_{\gamma}^{n+1}\x_{\tau}\ot x_{\gamma},
\end{aligned}
\\
\label{eq:diff-case5-auxiliar4}
&\begin{aligned}
d(\x_{\beta}  & \x_{\delta} \x_{\gamma}^{n} \x_{\eta}\ot 1)  = \x_{\beta}\x_{\delta}\x_{\gamma}^{n} \ot x_{\eta}
-q_{\gamma\eta}\x_{\beta}\x_{\delta}\x_{\gamma}^{n-1}\x_{\eta} \ot x_{\gamma}
\\ & 
-(-q_{\delta\gamma})^n q_{\delta\eta} \x_{\beta} \x_{\gamma}^{n} \x_{\eta} \ot x_{\delta} 
+q_{\beta\delta}(-q_{\beta\gamma})^n q_{\beta\eta} \x_{\delta} \x_{\gamma}^{n} \x_{\eta} \ot x_{\beta}
\\ &
-(-q_{\delta\gamma})^{n} \Bsj_3 \x_{\beta} \x_{\gamma}^{n} \x_{\varphi} \ot x_{\tau} x_{\gamma}.
\end{aligned}
\end{align}
We start with \eqref{eq:diff-case5-auxiliar1}. For $n=0$ we have
\begin{align*}
d(\x_{\alpha}  & \x_{\beta}\x_{\varphi}\ot 1)  =
\x_{\alpha}\x_{\beta}\ot x_{\varphi} - s\big(q_{\beta\varphi} \x_{\alpha} \ot x_{\varphi}x_{\beta} - q_{\alpha\beta}\x_{\beta}\ot (q_{\alpha\varphi} x_{\varphi}x_{\alpha} +\Bsj_1 x_{\gamma}) \big)
\\ &=
\x_{\alpha}\x_{\beta}\ot x_{\varphi}
-q_{\beta\varphi} \x_{\alpha}\x_{\varphi} \ot x_{\beta}
- s\big(q_{\beta\varphi}q_{\alpha\varphi}q_{\alpha\beta} \x_{\varphi} \ot x_{\beta} x_{\alpha}
+q_{\beta\varphi}\Bsj_1 \x_{\gamma} \ot x_{\beta}
\\ &
\quad - q_{\alpha\beta}q_{\alpha\varphi} \x_{\beta}\ot  x_{\varphi}x_{\alpha} 
-q_{\alpha\beta}\Bsj_1 \x_{\beta}\ot x_{\gamma}
\big)
\\ &=
\x_{\alpha}\x_{\beta}\ot x_{\varphi}
-q_{\beta\varphi} \x_{\alpha}\x_{\varphi} \ot x_{\beta}
+q_{\alpha\beta}q_{\alpha\varphi} \x_{\beta}\x_{\varphi} \ot  x_{\alpha}
+q_{\alpha\beta}\Bsj_1 \x_{\beta}\x_{\gamma}\ot 1.
\end{align*}
We assume that \eqref{eq:diff-case5-auxiliar1} holds for $n$. Using Remark \ref{rem:differential-q-commute} three times, inductive hypothesis,
\eqref{eq:diff-case5-auxiliar0}, \eqref{eq:diff-case5-hypothesis} and $x_{\gamma}^2=0$, we compute
\begin{align*}
d( & \x_{\alpha}\x_{\beta}\x_{\gamma}^{n+1} \x_{\varphi}\ot 1)  =
\x_{\alpha}\x_{\beta}\x_{\gamma}^{n+1} \ot x_{\varphi}
-s\big(
q_{\gamma\varphi} \x_{\alpha}\x_{\beta}\x_{\gamma}^{n} \ot x_{\varphi}x_{\gamma} 
\\ & \quad
+ (-q_{\beta\gamma})^{n+1}q_{\beta\varphi} \x_{\alpha} \x_{\gamma}^{n+1} \ot x_{\varphi}x_{\beta}
- q_{\alpha\beta} (-q_{\alpha\gamma})^{n+1} \x_{\beta} \x_{\gamma}^{n+1} \ot (q_{\alpha\varphi} x_{\varphi}x_{\alpha} + \Bsj_1 x_{\gamma}) \big)
\\ &
=
\x_{\alpha}\x_{\beta}\x_{\gamma}^{n+1} \ot x_{\varphi}
-q_{\gamma\varphi} \x_{\alpha}\x_{\beta}\x_{\gamma}^{n}\x_{\varphi} \ot x_{\gamma}
-(-q_{\beta\gamma})^{n+1}q_{\beta\varphi} \x_{\alpha} \x_{\gamma}^{n+1}\x_{\varphi} \ot x_{\beta}
\\ & \quad
-s\big(
q_{\alpha\beta}(-q_{\alpha\gamma})^{n+1} q_{\alpha\varphi} q_{\gamma\varphi} \x_{\beta}\x_{\gamma}^{n} \x_{\varphi} \ot x_{\gamma}x_{\alpha}
+q_{\beta\varphi} q_{\alpha\beta} q_{\alpha\gamma}^{n+1} q_{\beta\gamma}^{n+1}q_{\alpha\varphi} \x_{\gamma}^{n+1} \x_{\varphi}\ot x_{\beta}x_{\alpha}
\\ & \quad
- q_{\alpha\beta} (-q_{\alpha\gamma})^{n+1} q_{\alpha\varphi} \x_{\beta} \x_{\gamma}^{n+1} \ot x_{\varphi}x_{\alpha}
-(-q_{\beta\gamma})^{n+1}q_{\beta\varphi} \Bsj_1 q_{\gamma\varphi} (n+2)_{\widetilde{q}_{\alpha\gamma}} \x_{\gamma}^{n+2}\ot x_{\beta}
\\ & \quad
-(q_{\gamma\varphi} (n+1)_{\widetilde{q}_{\alpha\gamma}} +(-q_{\alpha\gamma})^{n+1})q_{\alpha\beta} \Bsj_1 \x_{\beta} \x_{\gamma}^{n+1} \ot x_{\gamma} \big)
\\ &
=
\x_{\alpha}\x_{\beta}\x_{\gamma}^{n+1} \ot x_{\varphi}
-q_{\gamma\varphi} \x_{\alpha}\x_{\beta}\x_{\gamma}^{n}\x_{\varphi} \ot x_{\gamma}
-(-q_{\beta\gamma})^{n+1}q_{\beta\varphi} \x_{\alpha} \x_{\gamma}^{n+1}\x_{\varphi} \ot x_{\beta}
\\ & \quad
+q_{\alpha\beta}q_{\gamma\varphi} (n+2)_{\widetilde{q}_{\alpha\gamma}}\Bsj_1 \x_{\beta} \x_{\gamma}^{n+2} \ot 1
+ q_{\alpha\beta} (-q_{\alpha\gamma})^{n+1} q_{\alpha\varphi} \x_{\beta} \x_{\gamma}^{n+1}\x_{\varphi} \ot x_{\alpha}.
\end{align*}

\medbreak
Now we prove \eqref{eq:diff-case5-auxiliar2}. For $n=0$,
\begin{align*}
d (& \x_{\delta}\x_{\eta} \ot 1) = \x_{\delta} \ot x_{\eta} -
q_{\delta\eta} \x_{\eta} \ot x_{\delta}
-\Bsj_3 \x_{\varphi} \ot x_{\tau} x_{\gamma}.
\end{align*}
We assume that \eqref{eq:diff-case5-auxiliar2} holds for $n$. By Remark \ref{rem:differential-q-commute} and inductive hypothesis,
\begin{align*}
d( & \x_{\delta}\x_{\gamma}^{n+1}\x_{\eta} \ot 1)  =
\x_{\delta}\x_{\gamma}^{n+1}\ot x_{\eta} - s \big(q_{\gamma\eta}\x_{\delta}\x_{\gamma}^{n} \ot x_{\eta}x_{\gamma}
+ (-q_{\delta\gamma})^{n+1} \x_{\gamma}^{n+1} \ot x_{\delta} x_{\eta}\big)
\\ &
= \x_{\delta}\x_{\gamma}^{n+1}\ot x_{\eta}
-q_{\gamma\eta}\x_{\delta}\x_{\gamma}^{n}\x_{\eta} \ot x_{\gamma}
-(-q_{\delta\gamma})^{n+1}\Bsj_3 \x_{\gamma}^{n+1}\x_{\varphi} \ot  x_{\tau}x_{\gamma} 
\\ & \quad
-(-q_{\delta\gamma})^{n+1}q_{\delta\eta} \x_{\gamma}^{n+1}\x_{\eta} \ot x_{\delta}.
\end{align*}

\medbreak
Next we prove \eqref{eq:diff-case5-auxiliar3}. For $n=0$,
\begin{align*}
d (\x_{\alpha} & \x_{\delta}\x_{\eta} \ot 1) = \x_{\alpha}\x_{\delta} \ot x_{\eta} -s \big( 
q_{\delta\eta} \x_{\alpha}\ot x_{\eta}x_{\delta}
+\Bsj_3 \x_{\alpha}\ot x_{\varphi} x_{\tau} x_{\gamma}
- q_{\alpha\delta}q_{\alpha\eta} \x_{\delta}\ot x_{\eta}x_{\alpha} \big)
\\ &
= \x_{\alpha}\x_{\delta} \ot x_{\eta} 
-\Bsj_3 \x_{\alpha} \x_{\varphi} \ot x_{\tau} x_{\gamma}
-q_{\delta\eta} \x_{\alpha}\x_{\eta} \ot x_{\delta}
-s \big( q_{\delta\eta}q_{\alpha\eta}q_{\alpha\delta} 
\x_{\eta} \ot x_{\delta}x_{\alpha}
\\ & \quad
+q_{\alpha\varphi}q_{\alpha\tau}q_{\alpha\gamma}\Bsj_3 \x_{\varphi}\ot x_{\tau}x_{\gamma}x_{\alpha} 
+ \Bsj _1\Bsj_3 \x_{\gamma}\ot x_{\tau} x_{\gamma}
- q_{\alpha\delta}q_{\alpha\eta} \x_{\delta}\ot x_{\eta}x_{\alpha} \big)
\\ &
= \x_{\alpha}\x_{\delta} \ot x_{\eta} 
-\Bsj_3 \x_{\alpha} \x_{\varphi} \ot x_{\tau} x_{\gamma}
-q_{\delta\eta} \x_{\alpha}\x_{\eta} \ot x_{\delta}
+ q_{\alpha\delta}q_{\alpha\eta} \x_{\delta}\x_{\eta}\ot x_{\alpha}
-\Bsj _1\Bsj_3 \x_{\gamma}\x_{\tau}\ot x_{\gamma}.
\end{align*}
We assume that \eqref{eq:diff-case5-auxiliar3} holds for $n$. Using Remark \ref{rem:differential-q-commute} three times, inductive hypothesis, \eqref{eq:diff-case5-auxiliar0}, \eqref{eq:diff-case5-auxiliar2}, we have
\begin{align*}
d( & \x_{\alpha}\x_{\delta}\x_{\gamma}^{n+1}\x_{\eta} \ot 1)  =
\x_{\alpha}\x_{\delta}\x_{\gamma}^{n+1}\ot x_{\eta}
-s\big( 
q_{\gamma\eta}\x_{\alpha}\x_{\delta}\x_{\gamma}^{n} \ot x_{\eta}x_{\gamma}
+(-q_{\delta\gamma})^{n+1}q_{\delta\eta} \x_{\alpha} \x_{\gamma}^{n+1} \ot x_{\eta}x_{\delta}
\\ & \quad
+(-q_{\delta\gamma})^{n+1} \Bsj_3 \x_{\alpha}\x_{\gamma}^{n+1} \ot  x_{\varphi}x_{\tau}x_{\gamma}
+q_{\alpha\delta}(-q_{\alpha\gamma})^{n+1}q_{\alpha\eta} \x_{\delta} \x_{\gamma}^{n+1} \ot x_{\eta}x_{\alpha} \big)
\\ &
= \x_{\alpha}\x_{\delta}\x_{\gamma}^{n+1}\ot x_{\eta}
-q_{\gamma\eta}\x_{\alpha}\x_{\delta}\x_{\gamma}^{n}\x_{\eta} \ot x_{\gamma}
-(-q_{\delta\gamma})^{n+1} \Bsj_3 \x_{\alpha}\x_{\gamma}^{n+1} \x_{\varphi} \ot  x_{\tau}x_{\gamma}
\\ & \quad
-s\big( 
-(-q_{\delta\gamma})^{n+1} q_{\delta\eta}q_{\gamma\eta} \x_{\alpha} \x_{\gamma}^{n} \x_{\eta} \ot x_{\gamma}x_{\delta}
+q_{\alpha\delta}(-q_{\alpha\gamma})^{n+1} q_{\alpha\eta} q_{\gamma\eta} \x_{\delta}\x_{\gamma}^{n}\x_{\eta} \ot x_{\gamma}x_{\alpha}
\\ & \quad 
+(-q_{\delta\gamma})^{n+1}q_{\delta\eta} \x_{\alpha} \x_{\gamma}^{n+1} \ot x_{\eta}x_{\delta}
+q_{\alpha\delta}(-q_{\alpha\gamma})^{n+1}q_{\alpha\eta} \x_{\delta} \x_{\gamma}^{n+1} \ot x_{\eta}x_{\alpha} 
\\ & \quad
+ q_{\alpha\gamma}^{n+2}q_{\alpha\varphi}q_{\alpha\tau} q_{\delta\gamma}^{n+1} \Bsj_3 \x_{\gamma}^{n+1} \x_{\varphi}\ot x_{\tau}x_{\gamma}x_{\alpha}
+q_{\gamma\varphi}(n+2)_{\widetilde{q}_{\alpha\gamma}} (-q_{\delta\gamma})^{n+1} \Bsj_1 \Bsj_3 \x_{\gamma}^{n+2}\ot x_{\tau}x_{\gamma} \big)
\\ &
= \x_{\alpha}\x_{\delta}\x_{\gamma}^{n+1}\ot x_{\eta}
-q_{\gamma\eta}\x_{\alpha}\x_{\delta}\x_{\gamma}^{n}\x_{\eta} \ot x_{\gamma}
-(-q_{\delta\gamma})^{n+1} \Bsj_3 \x_{\alpha}\x_{\gamma}^{n+1} \x_{\varphi} \ot  x_{\tau}x_{\gamma}
\\ & \quad 
-(-q_{\delta\gamma})^{n+1}q_{\delta\eta} \x_{\alpha} \x_{\gamma}^{n+1}\x_{\eta} \ot x_{\delta}
+q_{\alpha\delta}(-q_{\alpha\gamma})^{n+1}q_{\alpha\eta} \x_{\delta} \x_{\gamma}^{n+1}\x_{\eta} \ot x_{\alpha}
\\ & \quad
- q_{\gamma\varphi} (n+2)_{\widetilde{q}_{\alpha\gamma}} (-q_{\delta\gamma})^{n+1} \Bsj_1 \Bsj_3 s\big( \x_{\gamma}^{n+2}\ot x_{\tau}x_{\gamma} \big).
\end{align*}
As $x_{\gamma}^2=0$ and $x_{\gamma}x_{\tau}=q_{\gamma\tau}x_{\tau}x_{\gamma}$, we have that $s\big(\x_{\gamma}^{n+2}\ot x_{\tau}x_{\gamma} \big) =
\x_{\gamma}^{n+2}\x_{\tau}\ot x_{\gamma}$; hence the inductive step follows.

\medbreak
Now we prove \eqref{eq:diff-case5-auxiliar4}. For $n=0$,
\begin{align*}
d (\x_{\beta} & \x_{\delta}\x_{\eta} \ot 1) = \x_{\beta}\x_{\delta} \ot x_{\eta} -s \big( 
q_{\delta\eta} \x_{\beta}\ot x_{\eta}x_{\delta}
+\Bsj_3 \x_{\beta}\ot x_{\varphi} x_{\tau} x_{\gamma}
- q_{\beta\delta}q_{\beta\eta} \x_{\delta}\ot x_{\eta}x_{\beta} \big)
\\ &
=  \x_{\beta}\x_{\delta} \ot x_{\eta} 
-q_{\delta\eta} \x_{\beta}\x_{\eta}\ot x_{\delta} 
-\Bsj_3 \x_{\beta}\x_{\varphi}\ot  x_{\tau} x_{\gamma}
-s \big( q_{\delta\eta}q_{\beta\eta}q_{\beta\delta} \x_{\eta}\ot x_{\delta}x_{\beta}
\\ & \quad
+q_{\beta\varphi}\Bsj_3 \x_{\varphi}\ot (q_{\beta\tau}x_{\tau}x_{\beta}+\Bsj_2 x_{\gamma}) x_{\gamma}
- q_{\beta\delta}q_{\beta\eta} \x_{\delta}\ot x_{\eta}x_{\beta} \big)
\\ &
=  \x_{\beta}\x_{\delta} \ot x_{\eta} 
-q_{\delta\eta} \x_{\beta}\x_{\eta}\ot x_{\delta} 
-\Bsj_3 \x_{\beta}\x_{\varphi}\ot  x_{\tau} x_{\gamma}
+ q_{\beta\delta}q_{\beta\eta} \x_{\delta}\x_{\eta}\ot x_{\beta}.
\end{align*}
We assume that \eqref{eq:diff-case5-auxiliar4} holds for $n$. Using Remark \ref{rem:differential-q-commute} three times, inductive hypothesis and \eqref{eq:diff-case5-auxiliar2}, we have
\begin{align*}
d( & \x_{\beta}\x_{\delta}\x_{\gamma}^{n+1}\x_{\eta} \ot 1)  =
\x_{\beta}\x_{\delta}\x_{\gamma}^{n+1}\ot x_{\eta}
-s\big( 
q_{\gamma\eta}\x_{\beta}\x_{\delta}\x_{\gamma}^{n} \ot x_{\eta}x_{\gamma}
+(-q_{\delta\gamma})^{n+1}q_{\delta\eta} \x_{\beta} \x_{\gamma}^{n+1} \ot x_{\eta}x_{\delta}
\\ & \quad
+(-q_{\delta\gamma})^{n+1} \Bsj_3 \x_{\beta}\x_{\gamma}^{n+1} \ot  x_{\varphi}x_{\tau}x_{\gamma}
+q_{\beta\delta}(-q_{\beta\gamma})^{n+1}q_{\beta\eta} \x_{\delta} \x_{\gamma}^{n+1} \ot x_{\eta}x_{\beta} \big)
\\ &
= \x_{\beta}\x_{\delta}\x_{\gamma}^{n+1}\ot x_{\eta}
-q_{\gamma\eta}\x_{\beta}\x_{\delta}\x_{\gamma}^{n} \x_{\eta} \ot x_{\gamma}
-(-q_{\delta\gamma})^{n+1} \Bsj_3 \x_{\beta}\x_{\gamma}^{n+1} \x_{\varphi} \ot  x_{\tau}x_{\gamma}
\\ & \quad
-(-q_{\delta\gamma})^{n+1}q_{\delta\eta} \x_{\beta} \x_{\gamma}^{n+1}\x_{\eta} \ot x_{\delta} -s\big(
q_{\beta\delta}(-q_{\beta\gamma})^{n+1} q_{\beta\eta}q_{\gamma\eta} \x_{\delta} \x_{\gamma}^{n} \x_{\eta} \ot x_{\gamma}x_{\beta}
\\ & \quad
+q_{\beta\delta}(-q_{\beta\gamma})^{n+1}q_{\beta\eta} \x_{\delta} \x_{\gamma}^{n+1} \ot x_{\eta}x_{\beta} 
+q_{\delta\gamma}^{n+1} q_{\beta\gamma}^{n+1}q_{\beta\varphi} \Bsj_3 \x_{\gamma}^{n+1} \x_{\varphi} \ot 
(q_{\beta\tau} x_{\tau}x_{\beta} + \Bsj_2 x_{\gamma})x_{\gamma}
\\ & \quad
+q_{\delta\gamma}^{n+1}q_{\delta\eta}q_{\beta\gamma}^{n+1}q_{\beta\eta} q_{\beta\delta} \x_{\gamma}^{n+1}\x_{\eta} 
\ot x_{\delta}x_{\beta})
\big)
\\ &
= \x_{\beta}\x_{\delta}\x_{\gamma}^{n+1}\ot x_{\eta}
-q_{\gamma\eta}\x_{\beta}\x_{\delta}\x_{\gamma}^{n} \x_{\eta} \ot x_{\gamma}
-(-q_{\delta\gamma})^{n+1} \Bsj_3 \x_{\beta}\x_{\gamma}^{n+1} \x_{\varphi} \ot  x_{\tau}x_{\gamma}
\\ & \quad
-(-q_{\delta\gamma})^{n+1}q_{\delta\eta} \x_{\beta} \x_{\gamma}^{n+1}\x_{\eta} \ot x_{\delta} 
+q_{\beta\delta}(-q_{\beta\gamma})^{n+1}q_{\beta\eta} \x_{\delta} \x_{\gamma}^{n+1}\x_{\eta} \ot x_{\beta}.
\end{align*}

\bigbreak

Finally, we prove \eqref{eq:diff-case5-formula} by induction on $n$. For $n=0$,
\begin{align*}
d(\x_{\alpha} & \x_{\beta} \x_{\delta}\x_{\eta}\ot 1)  = 
\x_{\alpha}\x_{\beta} \x_{\delta}\ot x_{\eta}
-s \big(q_{\delta\eta} \x_{\alpha}\x_{\beta} \ot  x_{\eta}x_{\delta} 
+ \Bsj_3 \x_{\alpha}\x_{\beta} \ot  x_{\varphi} x_{\tau} x_{\gamma}
\\ &
\quad -q_{\beta\delta} q_{\beta\eta}\x_{\alpha}\x_{\delta} \ot x_{\eta}x_{\beta} +q_{\alpha\beta} q_{\alpha\delta} q_{\alpha\eta}\x_{\beta} \x_{\delta} \ot x_{\eta}x_{\alpha} \big)
\\ &
= \x_{\alpha}\x_{\beta} \x_{\delta} \ot x_{\eta}
-\Bsj_3 \x_{\alpha}\x_{\beta}\x_{\varphi} \ot x_{\tau}x_{\gamma}
-q_{\delta\eta} \x_{\alpha}\x_{\beta}\x_{\eta} \ot x_{\delta}
\\ & \quad
-s \big( q_{\beta\eta}q_{\beta\delta}q_{\delta\eta} \x_{\alpha}\x_{\eta} \ot x_{\delta}x_{\beta}
-q_{\alpha\beta}q_{\alpha\eta}q_{\alpha\delta}q_{\delta\eta} \x_{\beta}\x_{\eta} \ot x_{\delta}x_{\alpha}
\\ & \quad
+q_{\beta\varphi}q_{\beta\tau}q_{\beta\gamma} \Bsj_3 \x_{\alpha}\x_{\varphi} \ot x_{\tau}x_{\gamma}x_{\beta}
-q_{\alpha\beta} q_{\alpha\varphi} q_{\alpha\tau} q_{\alpha\gamma}
\Bsj_3 \x_{\beta}\x_{\varphi} \ot  x_{\tau}x_{\gamma}x_{\alpha}
\\ & \quad 
-q_{\alpha\beta}\Bsj_1\Bsj_3 \x_{\beta}\x_{\gamma}\ot x_{\tau} x_{\gamma}
-q_{\beta\delta} q_{\beta\eta}\x_{\alpha}\x_{\delta} \ot x_{\eta}x_{\beta} +q_{\alpha\beta} q_{\alpha\delta} q_{\alpha\eta}\x_{\beta} \x_{\delta} \ot x_{\eta}x_{\alpha} \big)
\\ &
= \x_{\alpha}\x_{\beta} \x_{\delta} \ot x_{\eta}
-\Bsj_3 \x_{\alpha}\x_{\beta}\x_{\varphi} \ot x_{\tau}x_{\gamma}
-q_{\delta\eta} \x_{\alpha}\x_{\beta}\x_{\eta} \ot x_{\delta}
+q_{\beta\delta} q_{\beta\eta}\x_{\alpha}\x_{\delta}\x_{\eta} \ot x_{\beta}
\\ & \quad
-q_{\alpha\beta} q_{\alpha\delta} q_{\alpha\eta}\x_{\beta} \x_{\delta}\x_{\eta} \ot x_{\alpha}
+q_{\alpha\beta}\Bsj_1\Bsj_3 \x_{\beta}\x_{\gamma}\x_{\tau} \ot x_{\gamma} 
-q_{\alpha\beta}q_{\gamma\tau}[2]_{\widetilde{q}_{\beta\gamma}} \Bsj_1\Bsj_2\Bsj_3  \x_{\gamma}^3\ot 1.
\end{align*}
We assume that \eqref{eq:diff-case5-formula} holds for $n$. Using Remark \ref{rem:differential-q-commute} twice, inductive hypothesis, \eqref{eq:diff-case5-auxiliar1}, \eqref{eq:diff-case5-auxiliar3},
\eqref{eq:diff-case5-auxiliar4}, \eqref{eq:diff-case5-auxiliar00}, we have
\begin{align*}
d& (  \x_{\alpha}\x_{\beta} \x_{\delta}\x_{\gamma}^{n+1}\x_{\eta}\ot 1)  = \x_{\alpha}\x_{\beta}\x_{\delta}\x_{\gamma}^{n+1} \ot x_{\eta}
-s\big(
q_{\gamma\eta}\x_{\alpha}\x_{\beta}\x_{\delta}\x_{\gamma}^{n} \ot x_{\eta}x_{\gamma}
\\ & \quad 
+ (-q_{\delta\gamma})^{n+1} q_{\delta\eta} \x_{\alpha}\x_{\beta} \x_{\gamma}^{n+1} \ot x_{\eta}x_{\delta}
+ (-q_{\delta\gamma})^{n+1} \Bsj_3 \x_{\alpha}\x_{\beta} \x_{\gamma}^{n+1} \ot x_{\varphi} x_{\tau} x_{\gamma}
\\ & \quad 
- q_{\beta\delta}(-q_{\beta\gamma})^{n+1}q_{\beta\eta} \x_{\alpha}\x_{\delta} \x_{\gamma}^{n+1} \ot x_{\eta}x_{\beta}
+ q_{\alpha\beta}q_{\alpha\delta}(-q_{\alpha\gamma})^{n+1} q_{\alpha\eta} \x_{\beta}\x_{\delta} \x_{\gamma}^{n+1} \ot x_{\eta}x_{\alpha} \big)
\\ &
= \x_{\alpha}\x_{\beta}\x_{\delta}\x_{\gamma}^{n+1} \ot x_{\eta}
-q_{\gamma\eta}\x_{\alpha}\x_{\beta}\x_{\delta}\x_{\gamma}^{n}
\x_{\eta} \ot x_{\gamma}
- (-q_{\delta\gamma})^{n+1} \Bsj_3 \x_{\alpha}\x_{\beta} \x_{\gamma}^{n+1}\x_{\varphi} \ot x_{\tau} x_{\gamma}
\\ & \quad 
-s\big(q_{\gamma\eta}(-q_{\delta\gamma})^n q_{\delta\gamma} q_{\delta\eta}\x_{\alpha}\x_{\beta} \x_{\gamma}^{n}\x_{\eta} \ot x_{\gamma}x_{\delta}
+q_{\beta\gamma}^{n+2} q_{\beta\varphi}q_{\delta\gamma}^{n+1} \Bsj_3 \x_{\alpha} \x_{\gamma}^{n+1} \x_{\varphi} \ot q_{\beta\tau}
x_{\tau}x_{\gamma}x_{\beta}
\\ & \quad 
-q_{\gamma\eta}q_{\beta\delta}(-q_{\beta\gamma})^n q_{\beta\gamma} q_{\beta\eta} \x_{\alpha}\x_{\delta}\x_{\gamma}^{n}\x_{\eta} \ot x_{\gamma}x_{\beta}
-q_{\gamma\eta}q_{\alpha\beta}q_{\alpha\delta} (-q_{\alpha\gamma})^{n+1} q_{\alpha\eta} \x_{\beta}\x_{\delta} \x_{\gamma}^{n} \x_{\eta} \ot x_{\gamma}x_{\alpha}
\\ & \quad 
+ (-q_{\delta\gamma})^{n+1} q_{\delta\eta} \x_{\alpha}\x_{\beta} \x_{\gamma}^{n+1} \ot x_{\eta}x_{\delta}
-q_{\alpha\beta}q_{\alpha\gamma}^{n+2} q_{\alpha\varphi}q_{\alpha\tau} q_{\delta\gamma}^{n+1} \Bsj_3 \x_{\beta}\x_{\gamma}^{n+1} \x_{\varphi} \ot x_{\tau} x_{\gamma}x_{\alpha}
\\ & \quad
+q_{\gamma\eta}q_{\alpha\beta}q_{\delta\gamma}^{-n-1}\coeff{\alpha\beta\delta\gamma}{n} \Bsj_1\Bsj_2\Bsj_3  \x_{\gamma}^{n+3}\ot x_{\gamma}
- q_{\alpha\beta}q_{\gamma\varphi} (n+2)_{\widetilde{q}_{\alpha\gamma}}(-q_{\delta\gamma})^{n+1} \Bsj_3 \Bsj_1 \x_{\beta}\x_{\gamma}^{n+2} \ot x_{\tau} x_{\gamma}
\\ & \quad 
- q_{\beta\delta}(-q_{\beta\gamma})^{n+1}q_{\beta\eta} \x_{\alpha}\x_{\delta} \x_{\gamma}^{n+1} \ot x_{\eta}x_{\beta}
+ q_{\alpha\beta}q_{\alpha\delta}(-q_{\alpha\gamma})^{n+1} q_{\alpha\eta} \x_{\beta}\x_{\delta} \x_{\gamma}^{n+1} \ot x_{\eta}x_{\alpha} \big)
\\ &
= \x_{\alpha}\x_{\beta}\x_{\delta}\x_{\gamma}^{n+1} \ot x_{\eta}
-q_{\gamma\eta}\x_{\alpha}\x_{\beta}\x_{\delta}\x_{\gamma}^{n}
\x_{\eta} \ot x_{\gamma}
- (-q_{\delta\gamma})^{n+1} \Bsj_3 \x_{\alpha}\x_{\beta} \x_{\gamma}^{n+1}\x_{\varphi} \ot x_{\tau} x_{\gamma}
\\ & \quad 
- (-q_{\delta\gamma})^{n+1} q_{\delta\eta} \x_{\alpha}\x_{\beta} \x_{\gamma}^{n+1}\x_{\eta} \ot x_{\delta}
+ q_{\beta\delta}(-q_{\beta\gamma})^{n+1}q_{\beta\eta} \x_{\alpha}\x_{\delta} \x_{\gamma}^{n+1}\x_{\eta} \ot x_{\beta}
\\ & \quad 
-s\big(-q_{\gamma\eta}q_{\alpha\beta}q_{\alpha\delta} (-q_{\alpha\gamma})^{n+1} q_{\alpha\eta} \x_{\beta}\x_{\delta} \x_{\gamma}^{n} \x_{\eta} \ot x_{\gamma}x_{\alpha}
\\ & \quad 
-q_{\alpha\beta}q_{\alpha\gamma}^{n+2} q_{\alpha\varphi}q_{\alpha\tau} q_{\delta\gamma}^{n+1} \Bsj_3 \x_{\beta}\x_{\gamma}^{n+1} \x_{\varphi} \ot x_{\tau} x_{\gamma}x_{\alpha}
\\ & \quad
-q_{\gamma\eta}q_{\alpha\beta}q_{\delta\gamma}^{-n-1}\coeff{\alpha\beta\delta\gamma}{n} \Bsj_1\Bsj_2\Bsj_3  \x_{\gamma}^{n+3}\ot x_{\gamma}
- q_{\alpha\beta}q_{\gamma\varphi} (n+2)_{\widetilde{q}_{\alpha\gamma}}(-q_{\delta\gamma})^{n+1} \Bsj_3 \Bsj_1 \x_{\beta}\x_{\gamma}^{n+2} \ot x_{\tau} x_{\gamma}
\\ & \quad 
+ q_{\alpha\beta}q_{\alpha\delta}(-q_{\alpha\gamma})^{n+1} q_{\alpha\eta} \x_{\beta}\x_{\delta} \x_{\gamma}^{n+1} \ot x_{\eta}x_{\alpha} 
-q_{\delta\gamma}^{n+1} q_{\delta\eta} q_{\alpha\beta} q_{\alpha\gamma}^{n+1} q_{\alpha\eta}q_{\alpha\delta} \x_{\beta} \x_{\gamma}^{n+1}\x_{\eta} \ot x_{\delta}x_{\alpha} 
\\ & \quad 
+q_{\beta\delta}q_{\beta\gamma}^{n+1}q_{\beta\eta}
q_{\alpha\beta}q_{\alpha\delta}q_{\alpha\gamma}^{n+1} q_{\alpha\eta} \x_{\delta} \x_{\gamma}^{n+1} \x_{\eta} \ot x_{\beta}x_{\alpha}
\\ & \quad
-q_{\gamma\varphi} q_{\delta\gamma}^{n+1}q_{\beta\delta} q_{\beta\gamma}^{n+1} q_{\beta\eta} (n+2)_{\widetilde{q}_{\alpha\gamma}}  \Bsj _1\Bsj_3 \x_{\gamma}^{n+2}\x_{\tau}\ot x_{\gamma}x_{\beta}
\big)
\\ &
= \x_{\alpha}\x_{\beta}\x_{\delta}\x_{\gamma}^{n+1} \ot x_{\eta}
-q_{\gamma\eta}\x_{\alpha}\x_{\beta}\x_{\delta}\x_{\gamma}^{n}
\x_{\eta} \ot x_{\gamma}
- (-q_{\delta\gamma})^{n+1} \Bsj_3 \x_{\alpha}\x_{\beta} \x_{\gamma}^{n+1}\x_{\varphi} \ot x_{\tau} x_{\gamma}
\\ & \quad 
- (-q_{\delta\gamma})^{n+1} q_{\delta\eta} \x_{\alpha}\x_{\beta} \x_{\gamma}^{n+1}\x_{\eta} \ot x_{\delta}
+ q_{\beta\delta}(-q_{\beta\gamma})^{n+1}q_{\beta\eta} \x_{\alpha}\x_{\delta} \x_{\gamma}^{n+1}\x_{\eta} \ot x_{\beta}
\\ & \quad 
- q_{\alpha\beta}q_{\alpha\delta}(-q_{\alpha\gamma})^{n+1} q_{\alpha\eta} \x_{\beta}\x_{\delta} \x_{\gamma}^{n+1}\x_{\eta} \ot x_{\alpha} 
-s\big( q_{\alpha\beta}q_{\gamma\eta}q_{\delta\gamma}^{-n-1}\coeff{\alpha\beta\delta\gamma}{n} \Bsj_1\Bsj_2\Bsj_3  \x_{\gamma}^{n+3}\ot x_{\gamma}
\\ & \quad
- q_{\alpha\beta}q_{\gamma\varphi} (n+2)_{\widetilde{q}_{\alpha\gamma}}(-q_{\delta\gamma})^{n+1} \Bsj_3 \Bsj_1 \x_{\beta}\x_{\gamma}^{n+2} \ot x_{\tau} x_{\gamma}
\\ & \quad
-q_{\gamma\varphi} q_{\delta\gamma}^{n+1}q_{\beta\delta} q_{\beta\gamma}^{n+1} q_{\beta\eta} (n+2)_{\widetilde{q}_{\alpha\gamma}}  \Bsj _1\Bsj_3 \x_{\gamma}^{n+2}\x_{\tau}\ot x_{\gamma}x_{\beta} \big)
\\ &
= \x_{\alpha}\x_{\beta}\x_{\delta}\x_{\gamma}^{n+1} \ot x_{\eta}
-q_{\gamma\eta}\x_{\alpha}\x_{\beta}\x_{\delta}\x_{\gamma}^{n}
\x_{\eta} \ot x_{\gamma}
- (-q_{\delta\gamma})^{n+1} q_{\delta\eta} \x_{\alpha}\x_{\beta} \x_{\gamma}^{n+1}\x_{\eta} \ot x_{\delta}
\\ & \quad 
+ q_{\beta\delta}(-q_{\beta\gamma})^{n+1}q_{\beta\eta} \x_{\alpha}\x_{\delta} \x_{\gamma}^{n+1}\x_{\eta} \ot x_{\beta}
- q_{\alpha\beta}q_{\alpha\delta}(-q_{\alpha\gamma})^{n+1} q_{\alpha\eta} \x_{\beta}\x_{\delta} \x_{\gamma}^{n+1}\x_{\eta} \ot x_{\alpha}
\\ & \quad 
- (-q_{\delta\gamma})^{n+1} \Bsj_3 \x_{\alpha}\x_{\beta} \x_{\gamma}^{n+1}\x_{\varphi} \ot x_{\tau} x_{\gamma}
- q_{\alpha\beta}q_{\gamma\varphi} (n+2)_{\widetilde{q}_{\alpha\gamma}}(-q_{\delta\gamma})^{n+1} \Bsj_1 \Bsj_3 \x_{\beta}\x_{\gamma}^{n+2}\x_{\tau} \ot  x_{\gamma}
\\ & \quad
-q_{\alpha\beta}\big( q_{\gamma\eta}(-q_{\delta\gamma})^{-n}\coeff{\alpha\beta\delta\gamma}{n}-q_{\gamma\varphi} (n+2)_{\widetilde{q}_{\alpha\gamma}}(-q_{\delta\gamma})^{n+1} q_{\gamma\tau} (n+3)_{\widetilde{q}_{\beta\gamma}}\big) \Bsj_1 \Bsj_2 \Bsj_3 s(\x_{\gamma}^{n+3}\ot x_{\gamma} ).
\end{align*}
Thus the proof of the inductive step follows since
$q_{\gamma\delta}q_{\gamma\delta} = q_{\gamma\gamma}q_{\gamma\tau} q_{\gamma\varphi}= -q_{\gamma\tau}q_{\gamma\varphi}$
and $s(\x_{\gamma}^{n+3}\ot x_{\gamma} )=\x_{\gamma}^{n+4}\ot 1$.
\epf

\begin{lema}\label{lem:diff-case6}
Let $\alpha < \eta <\gamma < \beta <\delta $ be positive roots
such that $N_{\gamma}=2$ and the relations among the corresponding root vectors take the form
\begin{align}\label{eq:diff-case6-hypothesis}
\begin{aligned}
x_{\alpha}x_{\delta} &= q_{\alpha\delta} x_{\delta}x_{\alpha} +\Bsj_1 x_{\gamma}x_{\eta}+ \Bsj_2 x_{\beta}x_{\eta}^2, &
x_{\eta}x_{\beta} &= q_{\eta\beta}  x_{\beta}x_{\eta} + \Bsj_3 x_{\gamma},
\end{aligned}
\end{align}
for some scalars $\Bsj_1, \Bsj_2, \Bsj_3$ and the other pairs of root vectors $q$-commute.
Then, for all $n\geq 0$, 
\begin{align}\label{eq:diff-case6-formula}
\begin{aligned}
d(\x_{\alpha} \x_{\gamma}^{n} & \x_{\beta} \x_{\delta}\ot 1)  = \x_{\alpha} \x_{\gamma}^{n} \x_{\beta} \ot x_{\delta}
-q_{\beta\delta}\x_{\alpha} \x_{\gamma}^{n} \x_{\delta}\ot x_{\beta}
+q_{\gamma\beta}q_{\gamma\delta} \x_{\alpha} \x_{\gamma}^{n-1} \x_{\beta} \x_{\delta}\ot x_{\gamma}
\\ &
+(-q_{\alpha\gamma})^nq_{\alpha\beta}q_{\alpha\delta} \x_{\gamma}^{n} \x_{\beta} \x_{\delta}\ot x_{\alpha}
+\Bsj_1 (-q_{\alpha\gamma})^{n} (n+1)_{\widetilde{q}_{\delta\gamma}}  \x_{\gamma}^{n+1} \x_{\beta}\ot x_{\eta}
\\ & 
+q_{\alpha\beta}(n+2)_{\widetilde{q}_{\beta\delta}}\delta_{N_{\beta},2} \Bsj_2 \x_{\gamma}^{n}\x_{\beta}^2 \ot x_{\eta}^2
-q_{\beta\delta}(-q_{\gamma\alpha})^{-n} \Bsj_1 \Bsj_3 \coef{\alpha\delta\gamma}{n} \x_{\gamma}^{n+2} \ot 1.
\end{aligned}
\end{align}
\end{lema} 

\bigbreak
Notice that the equalities in \eqref{eq:diff-case6-hypothesis} force
\begin{align}\label{eq:diff-case6-hypothesis-roots}
\gamma + \eta &=\alpha+\delta, &  \eta+\beta &= \gamma.
\end{align}
Hence the following equality also holds: $2\gamma=\alpha+\beta+\delta$.

\pf
Lemma \ref{lem:diff2} applied to $\eta<\gamma<\beta$ says that the following formula holds for $n\geq 0$:
\begin{align} \label{eq:diff-case6-auxiliar0}
\begin{aligned}
d(\x_{\eta} \x_{\gamma}^{n} \x_{\beta}\ot 1)  &=
\x_{\eta}\x_{\gamma}^{n}\ot x_{\beta} 
- q_{\gamma\beta} \x_{\eta}\x_{\gamma}^{n-1}\x_{\beta}
\ot x_{\gamma} 
- (-q_{\eta\gamma})^{n}q_{\eta\beta} \x_{\gamma}^{n} \x_{\beta}\ot x_{\eta} 
\\ & \quad + \Bsj_3 q_{\gamma\beta} (n+1)_{\widetilde{q}_{\eta\gamma}} \x_{\gamma}^{n+1}\ot 1,
\end{aligned}
\end{align}
Next we prove that the following formula holds for all $n\ge 0$:
\begin{align}
\label{eq:diff-case6-auxiliar1}
&\begin{aligned}
d(\x_{\alpha} & \x_{\gamma}^{n} \x_{\delta}\ot 1) =
\x_{\alpha}\x_{\gamma}^{n} \ot x_{\delta}
- q_{\gamma\delta} \x_{\alpha}\x_{\gamma}^{n-1}\x_{\delta} \ot x_{\gamma}
-(-q_{\alpha\gamma})^n\Bsj_2 \x_{\gamma}^{n}\x_{\beta} \ot x_{\eta}^2
\\ &
-(-q_{\alpha\gamma})^n(n+1)_{\widetilde{q}_{\delta\gamma}}\Bsj_1  \x_{\gamma}^{n+1} \ot x_{\eta}
-(-q_{\alpha\gamma})^nq_{\alpha\delta} \x_{\gamma}^{n}\x_{\delta} \ot x_{\alpha}
\end{aligned}
\end{align}
For $n=0$ we have that
\begin{align*}
d( & \x_{\alpha}\x_{\delta}\ot 1) = \x_{\alpha} \ot x_{\delta}- s \big( q_{\alpha\delta} x_{\delta}x_{\alpha} +\Bsj_1 x_{\gamma}x_{\eta}+ \Bsj_2 x_{\beta}x_{\eta}^2 \big)
\\ & = \x_{\alpha} \ot x_{\delta} -\Bsj_1 \x_{\gamma} \ot x_{\eta}
-\Bsj_2 \x_{\beta} \ot x_{\eta}^2
- q_{\alpha\delta} \x_{\delta} \ot x_{\alpha}
\end{align*}
%
Now assume that \eqref{eq:diff-case6-auxiliar1} holds for $n$. Using Remark \ref{rem:differential-q-commute}, inductive hypothesis, 
\begin{align*}
d( & \x_{\alpha}\x_{\gamma}^{n+1} \x_{\delta}\ot 1) = \x_{\alpha}\x_{\gamma}^{n+1} \ot x_{\delta}
-s \big( q_{\gamma\delta} \x_{\alpha}\x_{\gamma}^{n} \ot x_{\delta}x_{\gamma} 
+(-q_{\alpha\gamma})^{n+1} \x_{\gamma}^{n+1} \ot x_{\alpha}x_{\delta}
\big)
\\ & = \x_{\alpha}\x_{\gamma}^{n+1} \ot x_{\delta}
-q_{\gamma\delta} \x_{\alpha}\x_{\gamma}^{n}\x_{\delta} \ot x_{\gamma}
-s \big(
(-q_{\alpha\gamma})^{n+1}q_{\alpha\delta}  \x_{\gamma}^{n+1} \ot x_{\delta}x_{\alpha}
\\ & \quad
+(-q_{\alpha\gamma})^{n+1}(1+\widetilde{q}_{\gamma\delta}(n+1)_{\widetilde{q}_{\delta\gamma}}
)\Bsj_1  \x_{\gamma}^{n+1} \ot x_{\gamma}x_{\eta}
+(-q_{\alpha\gamma})^{n+1}\Bsj_2 \x_{\gamma}^{n+1} \ot x_{\beta}x_{\eta}^2
\\ & \quad
+q_{\eta\gamma}^2q_{\gamma\delta}(-q_{\alpha\gamma})^n\Bsj_2 \x_{\gamma}^{n}\x_{\beta} \ot x_{\gamma}x_{\eta}^2
-(-q_{\alpha\gamma})^{n+1}q_{\alpha\delta}q_{\gamma\delta} \x_{\gamma}^{n}\x_{\delta} \ot x_{\gamma}x_{\alpha}
\big)
\\ & = \x_{\alpha}\x_{\gamma}^{n+1} \ot x_{\delta}
-q_{\gamma\delta} \x_{\alpha}\x_{\gamma}^{n}\x_{\delta} \ot x_{\gamma}
-(-q_{\alpha\gamma})^{n+1}(n+2)_{\widetilde{q}_{\delta\gamma}}\Bsj_1  \x_{\gamma}^{n+2} \ot x_{\eta}
\\ & \quad
-(-q_{\alpha\gamma})^{n+1}q_{\alpha\delta}  \x_{\gamma}^{n+1}\x_{\delta} \ot x_{\alpha}
-(-q_{\alpha\gamma})^{n+1}\Bsj_2 \x_{\gamma}^{n+1}\x_{\beta} \ot x_{\eta}^2.
\end{align*}

Now we prove \eqref{eq:diff-case6-formula} by induction on $n$. When $n=0$,
\begin{align*}
d( & \x_{\alpha}\x_{\beta}\x_{\delta} \ot 1) = \x_{\alpha}\x_{\beta} \ot x_{\delta}
-s \big(q_{\beta\delta} \x_{\alpha} \ot x_{\delta}x_{\beta} -q_{\alpha\beta} \x_{\beta} \ot x_{\alpha}x_{\delta}\big)
\\ & = \x_{\alpha}\x_{\beta} \ot x_{\delta}
-q_{\beta\delta} \x_{\alpha}\x_{\delta} \ot x_{\beta}
-q_{\beta\delta}\Bsj_1\Bsj_3 \x_{\gamma}^2 \ot 1
-q_{\eta\beta}q_{\beta\delta}\Bsj_1 \x_{\gamma}\x_{\beta} \ot x_{\eta}
\\ & \quad
+q_{\alpha\beta}(2)_{\widetilde{q}_{\beta\delta}}\delta_{N_{\beta},2} \Bsj_2 \x_{\beta}^2 \ot x_{\eta}^2
+q_{\alpha\beta}q_{\alpha\delta} \x_{\beta}\x_{\delta} \ot x_{\alpha},
\end{align*}
since $\x_{\beta} \ot x_{\gamma}x_{\eta}=s(x_{\beta}x_{\gamma}x_{\eta})$ and $s(\x_{\beta} \ot x_{\beta}x_{\eta}^2)=0$: either $\x_{\beta} \ot x_{\beta}x_{\eta}^2=s(x_{\beta}^2x_{\eta}^2)$ if $N_{\beta}\ne 2$ or $q_{\beta\delta}q_{\eta\beta}^2=q_{\beta\delta}q_{\gamma\beta}^2 =-q_{\alpha\beta}\widetilde{q}_{\beta\delta}$ if $N_{\beta}=2$.
Now assume that \eqref{eq:diff-case7-formula} holds for $n$. Using Remark \ref{rem:differential-q-commute}, \eqref{eq:diff-case6-auxiliar1} and inductive hypothesis,
\begin{align*}
d(& \x_{\alpha}\x_{\gamma}^{n+1}\x_{\beta}\x_{\delta}\ot 1)  =
\x_{\alpha}\x_{\gamma}^{n+1}\x_{\beta} \ot x_{\delta}
-s \big(
q_{\beta\delta} \x_{\alpha}\x_{\gamma}^{n+1} \ot x_{\delta}x_{\beta}
-q_{\gamma\beta}q_{\gamma\delta} \x_{\alpha}\x_{\gamma}^{n}\x_{\beta} \ot x_{\delta}x_{\gamma}
\\ & \quad
-(-q_{\alpha\gamma})^{n+1}q_{\alpha\beta} \x_{\gamma}^{n+1}\x_{\beta} \ot (q_{\alpha\delta} x_{\delta}x_{\alpha} +\Bsj_1 x_{\gamma}x_{\eta}+ \Bsj_2 x_{\beta}x_{\eta}^2)
\big)
\\ & =
\x_{\alpha}\x_{\gamma}^{n+1}\x_{\beta} \ot x_{\delta}
-q_{\beta\delta} \x_{\alpha}\x_{\gamma}^{n+1}\x_{\delta} \ot x_{\beta}
+q_{\gamma\beta}q_{\gamma\delta} \x_{\alpha}\x_{\gamma}^{n}\x_{\beta}\x_{\delta} \ot x_{\gamma}
\\ & \quad
-q_{\beta\delta}(-q_{\gamma\alpha})^{-n+1}\coef{\alpha\delta\gamma}{n+1}\Bsj_1\Bsj_3 \x_{\gamma}^{n+3} \ot 1
-(-q_{\alpha\gamma})^{n+1}q_{\eta\beta}q_{\beta\delta}(n+2)_{\widetilde{q}_{\delta\gamma}}\Bsj_1  \x_{\gamma}^{n+2}\x_{\beta} \ot x_{\eta}
\\ & \quad
+q_{\alpha\beta}(n+2)_{\widetilde{q}_{\beta\delta}}\delta_{N_{\beta},2} \Bsj_2 \x_{\gamma}^{n}\x_{\beta}^2 \ot x_{\eta}^2
+(-q_{\alpha\gamma})^{n+1}q_{\alpha\beta}q_{\alpha\delta} \x_{\gamma}^{n+1}\x_{\beta}\x_{\delta} \ot x_{\alpha},
\end{align*}
which completes the inductive step.
\epf

\begin{lema}\label{lem:diff-case7}
Let $\alpha < \beta <\gamma < \eta <\delta $ be positive roots
such that $N_{\gamma}=2$ and the relations among the corresponding root vectors take the form
\begin{align}\label{eq:diff-case7-hypothesis}
\begin{aligned}
x_{\beta}x_{\delta} &= q_{\beta\delta} x_{\delta}x_{\beta} +\Bsj_1 x_{\eta}x_{\gamma}, &
x_{\alpha}x_{\eta} &= q_{\alpha\eta}  x_{\eta}x_{\alpha} + \Bsj_2 x_{\gamma}, 
\end{aligned}
\end{align}
for some scalars $\Bsj_1, \Bsj_2$ and the other pairs of root vectors $q$-commute.
Then, for all $n\geq 0$, 
\begin{align}\label{eq:diff-case7-formula}
\begin{aligned}
d(\x_{\alpha}& \x_{\beta} \x_{\gamma}^{n} \x_{\delta}\ot 1)  = \x_{\alpha}\x_{\beta} \x_{\gamma}^{n} \ot x_{\delta}
-q_{\gamma\delta} \x_{\alpha}\x_{\beta}\x_{\gamma}^{n-1} \x_{\delta} \ot x_{\gamma}
\\ &
-(-q_{\beta\gamma})^n q_{\beta\delta} \x_{\alpha}\x_{\gamma}^{n} \x_{\delta}\ot x_{\beta}
+(-q_{\alpha\gamma})^nq_{\alpha\beta}q_{\alpha\delta} \x_{\beta} \x_{\gamma}^{n} \x_{\delta}\ot x_{\alpha}
\\ & 
- (-q_{\beta\gamma})^{n} \Bsj_1 \x_{\alpha} \x_{\gamma}^{n} \x_{\eta}\ot x_{\gamma}
-q_{\gamma\delta}^{n} \coef{-\delta,\alpha,\gamma}{n} \Bsj_1 \Bsj_2 \x_{\gamma}^{n+2} \ot 1.
\end{aligned}
\end{align}
\end{lema} 

\bigbreak
Notice that the equalities in \eqref{eq:diff-case7-hypothesis} force
\begin{align}\label{eq:diff-case7-hypothesis-roots}
\gamma + \eta &=\beta+\delta, &  \eta+\alpha &= \gamma.
\end{align}
Hence the following equality also holds: $2\gamma=\alpha+\beta+\delta$.

\pf 
The following formula holds for all $n\ge 0$:
\begin{align}
\label{eq:diff-case7-auxiliar}
&\begin{aligned}
d(\x_{\beta} & \x_{\gamma}^{n} \x_{\delta}\ot 1) =
\x_{\beta}\x_{\gamma}^{n}\ot x_{\delta} 
-q_{\gamma \delta} \x_{\beta}\x_{\gamma}^{n-1}\x_{\delta}\ot x_{\gamma}
-(-q_{\beta\gamma})^{n} q_{\beta \delta} \x_{\gamma}^{n}\x_{\delta} \ot x_{\beta}
\\ & -(-q_{\beta\gamma})^{n} \Bsj_1 \x_{\gamma}^{n}\x_{\eta}\ot x_{\gamma}.
\end{aligned}
\end{align}
The proof is analogous to \eqref{eq:diff-case2-auxiliar}, see also the proof of \eqref{eq:diff-case4-auxiliar-2}. 
Next, we apply Lemma \ref{lem:diff2} for $\alpha < \gamma < \eta$ (no other intermediate roots) to get 
\begin{align}\label{eq:diff-case7-auxiliar2}
\begin{aligned}
d(\x_{\alpha} & \x_{\gamma}^{n} \x_{\eta}\ot 1)  =
\x_{\alpha}\x_{\gamma}^{n} \ot x_{\eta} 
- q_{\gamma\eta} \x_{\alpha} \x_{\gamma}^{n-1}\x_{\eta}\ot x_{\gamma} 
- q_{\alpha\eta}(-q_{\alpha\gamma})^{n} \x_{\gamma}^{n} \x_{\eta} \ot x_{\alpha}  
\\
& 
- (-q_{\alpha\gamma})^{n} (n+1)_{\widetilde{q}_{\alpha\gamma}} \Bsj_2
\x_{\gamma}^{n+1} \ot 1.
\end{aligned}
\end{align}

Now we prove \eqref{eq:diff-case7-formula} by induction on $n$. When $n=0$,
\begin{align*}
d( & \x_{\alpha} \x_{\beta} \x_{\delta}\ot 1) =
\x_{\alpha}\x_{\beta} \ot x_{\delta} - s \left( 
q_{\beta\delta} \x_{\alpha}\ot x_{\delta}x_{\beta}
+\Bsj_1 \x_{\alpha}\ot x_{\eta}x_{\gamma}
-q_{\alpha\beta}q_{\alpha\delta} \x_{\beta} \ot x_{\delta}x_{\alpha}
\right)
\\ &
= \x_{\alpha}\x_{\beta} \ot x_{\delta} 
-\Bsj_1 \x_{\alpha}\x_{\eta} \ot x_{\gamma}
-q_{\beta\delta} \x_{\alpha} \x_{\delta} \ot x_{\beta}
- s \big( 
q_{\alpha\beta}q_{\alpha\delta}q_{\beta\delta} \x_{\delta} \ot x_{\beta}x_{\alpha}
\\ & \quad
+ q_{\alpha\gamma}q_{\alpha\eta}\Bsj_1 \x_{\eta}\ot x_{\gamma}x_{\alpha} +
\Bsj_1 \Bsj_2 \x_{\gamma}\ot x_{\gamma}
-q_{\alpha\beta}q_{\alpha\delta} \x_{\beta} \ot x_{\delta}x_{\alpha}
\big)
\\ &
= \x_{\alpha}\x_{\beta} \ot x_{\delta} 
-\Bsj_1 \x_{\alpha}\x_{\eta} \ot x_{\gamma}
-q_{\beta\delta} \x_{\alpha} \x_{\delta} \ot x_{\beta}
+q_{\alpha\beta}q_{\alpha\delta} \x_{\beta}\x_{\delta} \ot x_{\alpha}
- \Bsj_1 \Bsj_2 \x_{\gamma}^2 \ot 1.
\end{align*}

Now assume that \eqref{eq:diff-case7-formula} holds for $n$. Using Remark \ref{rem:differential-q-commute}, inductive hypothesis, the relation $x_{\gamma}^2=0$, \eqref{eq:diff-case7-auxiliar2}, \eqref{eq:diff-case7-auxiliar}, 
\begin{align*}
d(\x_{\alpha} & \x_{\beta} \x_{\gamma}^{n+1} \x_{\delta}\ot 1)  =
\x_{\alpha} \x_{\beta}\x_{\gamma}^{n+1} \otimes x_{\delta}
-s \big(
q_{\gamma\delta} \x_{\alpha} \x_{\beta}\x_{\gamma}^{n} \ot x_{\delta} x_{\gamma}
+ (-q_{\beta\gamma})^{n+1}\Bsj_1 \x_{\alpha}\x_{\gamma}^{n+1} \ot  x_{\eta}x_{\gamma}
\\ & \quad 
+ (-q_{\beta\gamma})^{n+1}q_{\beta\delta} \x_{\alpha}\x_{\gamma}^{n+1} \ot  x_{\delta}x_{\beta}
-q_{\alpha\beta}(-q_{\alpha\gamma})^{n+1} q_{\alpha\delta} \x_{\beta}\x_{\gamma}^{n+1} \ot x_{\delta}x_{\alpha} \big)
\\ &
=
\x_{\alpha} \x_{\beta}\x_{\gamma}^{n+1} \otimes x_{\delta}
-q_{\gamma\delta} \x_{\alpha} \x_{\beta}\x_{\gamma}^{n}\x_{\delta} \ot  x_{\gamma}
- (-q_{\beta\gamma})^{n+1}\Bsj_1 \x_{\alpha}\x_{\gamma}^{n+1}\x_{\eta} \ot  x_{\gamma}
\\ & \quad 
-(-q_{\beta\gamma})^{n+1}q_{\beta\delta} \x_{\alpha}\x_{\gamma}^{n+1} \x_{\delta} \ot  x_{\beta} 
-s \big(
-q_{\alpha\beta}(-q_{\alpha\gamma})^{n+1} q_{\alpha\delta} \x_{\beta}\x_{\gamma}^{n+1} \ot x_{\delta}x_{\alpha} 
\\ & \quad 
-q_{\alpha\gamma} (-q_{\alpha\gamma})^n q_{\alpha\beta} q_{\alpha\delta} q_{\gamma\delta} \x_{\beta} \x_{\gamma}^{n} \x_{\delta} \ot x_{\gamma}x_{\alpha}
\\ & \quad 
+ (q_{\gamma\delta}^{n+1} \coef{-\delta,\alpha,\gamma}{n} + q_{\beta\gamma}^{n+1}  q_{\alpha\gamma}^{n+1} (n+2)_{\widetilde{q}_{\alpha\gamma}}) \Bsj_1\Bsj_2 \x_{\gamma}^{n+2} \ot x_{\gamma} 
\\ & \quad
+q_{\beta\gamma}^{n+1}
q_{\alpha\eta}q_{\alpha\gamma}^{n+2}\Bsj_1 \x_{\gamma}^{n+1} \x_{\eta} \ot x_{\gamma}x_{\alpha} 
-q_{\beta\gamma}^{n+1}q_{\beta\delta} 
q_{\alpha\beta}q_{\alpha\gamma}^{n+1}q_{\alpha\delta} \x_{\gamma}^{n+1} \x_{\delta} \ot x_{\beta}x_{\alpha}
\big)
\\ &
=
\x_{\alpha} \x_{\beta}\x_{\gamma}^{n+1} \otimes x_{\delta}
-q_{\gamma\delta} \x_{\alpha} \x_{\beta}\x_{\gamma}^{n}\x_{\delta} \ot  x_{\gamma}
- (-q_{\beta\gamma})^{n+1}\Bsj_1 \x_{\alpha}\x_{\gamma}^{n+1}\x_{\eta} \ot  x_{\gamma}
\\ & \quad 
-(-q_{\beta\gamma})^{n+1}q_{\beta\delta} \x_{\alpha}\x_{\gamma}^{n+1} \x_{\delta} \ot  x_{\beta} 
+q_{\alpha\beta}(-q_{\alpha\gamma})^{n+1} q_{\alpha\delta} \x_{\beta}\x_{\gamma}^{n+1}\x_{\delta} \ot x_{\alpha} 
\\ & \quad 
-(q_{\gamma\delta}^{n+1} \coef{-\delta,\alpha,\gamma}{n} + q_{\gamma\gamma}^{2(n+1)}  q_{\delta\gamma}^{-n-1} (n+2)_{\widetilde{q}_{\alpha\gamma}}) \Bsj_1\Bsj_2 \x_{\gamma}^{n+3} \ot 1,
\end{align*}
and the inductive step follows.
\epf

\begin{lema}\label{lem:diff-case8}
Let $\alpha < \tau < \beta <\gamma < \mu < \nu 
< \eta <\delta $ be positive roots
such that $N_{\gamma}=2$ and the relations among the corresponding root vectors take the form
\begin{align}\label{eq:diff-case8-hypothesis}
\begin{aligned}
x_{\alpha}x_{\delta} &= q_{\alpha\delta} x_{\delta}x_{\alpha} +\Bsj_1 x_{\eta}x_{\tau}, &
x_{\beta}x_{\delta} &= q_{\beta\delta}  x_{\delta}x_{\beta} + \Bsj_2 x_{\nu} x_{\gamma},
\\
x_{\beta}x_{\eta} &= q_{\beta\eta} x_{\eta}x_{\beta} +\Bsj_3 x_{\mu}x_{\gamma}, &
x_{\alpha}x_{\nu} &= q_{\alpha\nu}  x_{\nu}x_{\alpha} + \Bsj_4 x_{\gamma},
\end{aligned}
\end{align}
for some scalars $\Bsj_i$ and the other pairs of root vectors $q$-commute.
Then, for all $n\geq 0$, 
\begin{align}\label{eq:diff-case8-formula}
\begin{aligned}
d(\x_{\alpha}& \x_{\beta} \x_{\gamma}^{n} \x_{\delta}\ot 1)  = \x_{\alpha}\x_{\beta} \x_{\gamma}^{n} \ot x_{\delta}
-q_{\gamma\delta} \x_{\alpha}\x_{\beta}\x_{\gamma}^{n-1} \x_{\delta} \ot x_{\gamma}
\\ &
-(-q_{\beta\gamma})^n q_{\beta\delta} \x_{\alpha}\x_{\gamma}^{n} \x_{\delta}\ot x_{\beta}
+(-q_{\alpha\gamma})^nq_{\alpha\beta}q_{\alpha\delta} \x_{\beta} \x_{\gamma}^{n} \x_{\delta}\ot x_{\alpha}
\\ & 
+q_{\alpha\beta} (-q_{\alpha\gamma})^{n} \Bsj_1 \x_{\beta} \x_{\gamma}^{n} \x_{\eta} \otimes x_{\tau}
-(-q_{\beta\gamma})^{n}\Bsj_2 \x_{\alpha}\x_{\gamma}^{n}\x_{\nu} \otimes x_{\gamma}
\\ & 
- q_{\gamma\delta}^{n} \coef{-\delta,\alpha,\gamma}{n} \Bsj_2 \Bsj_4 \x_{\gamma}^{n+2} \ot 1.
\end{aligned}
\end{align}
\end{lema} 

\bigbreak
Notice that the equalities in \eqref{eq:diff-case8-hypothesis} force
\begin{align}\label{eq:diff-case8-hypothesis-roots}
\begin{aligned}
\alpha + \delta &= \eta + \tau, &
\beta + \delta &= \nu + \gamma, &
\beta + \eta &= \mu + \gamma, &
\alpha + \nu &= \gamma.
\end{aligned}
\end{align}
Hence the following equality also holds: $2\gamma=\alpha+\beta+\delta$.

\pf
We need some auxiliary computations. First we apply Lemma \ref{lem:diff2} to $\alpha < \gamma < \nu$: 
\begin{align}\label{eq:diff-case8-auxiliar1}
\begin{aligned}
d(\x_{\alpha} & \x_{\gamma}^{n} \x_{\nu}\ot 1)  =
\x_{\alpha}\x_{\gamma}^{n} \ot x_{\nu} 
- q_{\gamma\nu} \x_{\alpha} \x_{\gamma}^{n-1}\x_{\nu}\ot x_{\gamma} 
- q_{\alpha\nu}(-q_{\alpha\gamma})^{n} \x_{\gamma}^{n} \x_{\nu} \ot x_{\alpha}  
\\
& 
- (-q_{\alpha\gamma})^{n} (n+1)_{\widetilde{q}_{\alpha\gamma}} \Bsj_4
\x_{\gamma}^{n+1} \ot 1.
\end{aligned}
\end{align}

Next we claim that the following formulas hold for all $n\ge 0$:
\begin{align}
\label{eq:diff-case8-auxiliar2}
&\begin{aligned}
d(\x_{\alpha} & \x_{\gamma}^{n} \x_{\delta}\ot 1) =
\x_{\alpha}\x_{\gamma}^{n}\ot x_{\delta} 
-q_{\gamma \delta} \x_{\alpha}\x_{\gamma}^{n-1}\x_{\delta}\ot x_{\gamma}
-(-q_{\alpha\gamma})^{n} q_{\alpha \delta} \x_{\gamma}^{n}\x_{\delta} \ot x_{\alpha}
\\ & -(-q_{\alpha\gamma})^{n} \Bsj_1 \x_{\gamma}^{n}\x_{\eta}\ot x_{\tau},
\end{aligned}
\\
\label{eq:diff-case8-auxiliar3}
&\begin{aligned}
d(\x_{\beta} & \x_{\gamma}^{n} \x_{\delta}\ot 1) =
\x_{\beta}\x_{\gamma}^{n}\ot x_{\delta} 
-q_{\gamma \delta} \x_{\beta}\x_{\gamma}^{n-1}\x_{\delta}\ot x_{\gamma}
-(-q_{\beta\gamma})^{n} q_{\beta \delta} \x_{\gamma}^{n}\x_{\delta} \ot x_{\beta}
\\ & -(-q_{\beta\gamma})^{n} \Bsj_2 \x_{\gamma}^{n}\x_{\nu}\ot x_{\gamma},
\end{aligned}
\\
\label{eq:diff-case8-auxiliar4}
&\begin{aligned}
d(\x_{\beta} & \x_{\gamma}^{n} \x_{\eta}\ot 1) =
\x_{\beta}\x_{\gamma}^{n}\ot x_{\eta} 
-q_{\gamma \eta} \x_{\beta}\x_{\gamma}^{n-1}\x_{\eta}\ot x_{\gamma}
-(-q_{\beta\gamma})^{n} q_{\beta \eta} \x_{\gamma}^{n}\x_{\eta} \ot x_{\beta}
\\ & -(-q_{\beta\gamma})^{n} \Bsj_3 \x_{\gamma}^{n}\x_{\mu}\ot x_{\gamma}.
\end{aligned}
\end{align}
The proof of each equality is analogous to \eqref{eq:diff-case2-auxiliar}. 

\medskip

Now we prove \eqref{eq:diff-case8-formula} by induction on $n$. When $n=0$,
\begin{align*}
d( & \x_{\alpha} \x_{\beta} \x_{\delta}\ot 1) = 
\x_{\alpha} \x_{\beta} \ot x_{\delta} - s \big( q_{\beta\delta}  \x_{\alpha}  \ot x_{\delta}x_{\beta} + \Bsj_2 \x_{\alpha}  \ot x_{\nu} x_{\gamma} -q_{\alpha\beta} \x_{\beta} \ot x_{\alpha}x_{\delta} \big)
\\ & 
=\x_{\alpha} \x_{\beta} \ot x_{\delta}
-\Bsj_2 \x_{\alpha}\x_{\nu} \ot  x_{\gamma} 
-q_{\beta\delta}  \x_{\alpha}\x_{\delta} \ot x_{\beta} 
-s \big(-q_{\alpha\beta}q_{\alpha\delta} \x_{\beta} \ot  x_{\delta}x_{\alpha}
-q_{\alpha\beta}\Bsj_1 \x_{\beta} \ot  x_{\eta}x_{\tau}
\\ & \quad 
+q_{\alpha\nu} \Bsj_2 \x_{\nu}\ot x_{\alpha}x_{\gamma} +\Bsj_2 \Bsj_4 \x_{\gamma} \ot x_{\gamma}
+q_{\alpha\beta} q_{\alpha\delta}q_{\beta\delta} \x_{\delta}\ot x_{\beta}x_{\alpha} 
+q_{\tau\beta}q_{\beta\delta} \Bsj_1\x_{\eta}\ot x_{\beta}x_{\tau} \big)
\\ & 
=\x_{\alpha} \x_{\beta} \ot x_{\delta}
-\Bsj_2 \x_{\alpha}\x_{\nu} \ot  x_{\gamma} 
-q_{\beta\delta}  \x_{\alpha}\x_{\delta} \ot x_{\beta}
+q_{\alpha\beta}\Bsj_1 \x_{\beta}\x_{\eta} \ot  x_{\tau}
+q_{\alpha\beta}q_{\alpha\delta} \x_{\beta}\x_{\delta} \ot x_{\alpha}
\\ & \quad
-s \big(\Bsj_2 \Bsj_4 \x_{\gamma} \ot x_{\gamma}
-q_{\alpha\beta}\Bsj_1 \Bsj_3 \x_{\mu}\ot x_{\gamma}x_{\tau}
\big),
\end{align*}
which is \eqref{eq:diff-case8-formula} for $n=0$ since
\begin{align*}
s \big(\x_{\gamma} \ot x_{\gamma}\big)&=\x_{\gamma}^2 \ot 1,
&
s \big(\x_{\mu}\ot x_{\gamma}x_{\tau}\big) &= s \circ s \big(x_{\mu}x_{\gamma}x_{\tau}\big)=0.
\end{align*}
Now assume that \eqref{eq:diff-case7-formula} holds for $n$. Using Remark \ref{rem:differential-q-commute}, inductive hypothesis, the relation $x_{\gamma}^2=0$, \eqref{eq:diff-case8-auxiliar1}, \eqref{eq:diff-case8-auxiliar2}, \eqref{eq:diff-case8-auxiliar3}, \eqref{eq:diff-case8-auxiliar4},
\begin{align*}
d(\x_{\alpha} & \x_{\beta} \x_{\gamma}^{n+1} \x_{\delta}\ot 1)  =
\x_{\alpha} \x_{\beta}\x_{\gamma}^{n+1} \otimes x_{\delta}
- s \big(
q_{\gamma\delta} \x_{\alpha} \x_{\beta}\x_{\gamma}^{n} \otimes x_{\delta}x_{\gamma}
+(-q_{\beta\gamma})^{n+1}\Bsj_2 \x_{\alpha} \x_{\gamma}^{n+1} \otimes x_{\nu} x_{\gamma}
\\ & \quad
+q_{\beta\delta} (-q_{\beta\gamma})^{n+1} \x_{\alpha} \x_{\gamma}^{n+1} \otimes  x_{\delta}x_{\beta}
-q_{\alpha\beta} (-q_{\alpha\gamma})^{n+1} q_{\alpha\delta} \x_{\beta}\x_{\gamma}^{n+1} \otimes x_{\delta}x_{\alpha}
\\ & \quad
-q_{\alpha\beta} (-q_{\alpha\gamma})^{n+1} \Bsj_1 \x_{\beta}\x_{\gamma}^{n+1} \otimes x_{\eta}x_{\tau}
\big)
\\ &=
\x_{\alpha} \x_{\beta}\x_{\gamma}^{n+1} \otimes x_{\delta}
-q_{\gamma\delta} \x_{\alpha} \x_{\beta}\x_{\gamma}^{n}\x_{\delta} \otimes x_{\gamma}
-(-q_{\beta\gamma})^{n+1}\Bsj_2 \x_{\alpha} \x_{\gamma}^{n+1} \x_{\nu} \otimes x_{\gamma}
\\ & \quad
- s \big(-q_{\gamma\delta}(-q_{\beta\gamma})^{n+1} q_{\beta\delta} \x_{\alpha}\x_{\gamma}^{n} \x_{\delta}\ot x_{\gamma}x_{\beta}
+q_{\gamma\delta}(-q_{\alpha\gamma})^{n+1}q_{\alpha\beta}q_{\alpha\delta} \x_{\beta} \x_{\gamma}^{n} \x_{\delta}\ot x_{\gamma}x_{\alpha}
\\ & \quad
-q_{\gamma\delta}q_{\alpha\beta} (-q_{\alpha\gamma})^{n} \Bsj_1 \x_{\beta} \x_{\gamma}^{n} \x_{\eta} \otimes x_{\tau}x_{\gamma}
+q_{\gamma\delta}(-q_{\beta\gamma})^{n}\Bsj_2 \x_{\alpha}\x_{\gamma}^{n}\x_{\nu} \otimes x_{\gamma}^2
\\ & \quad
+ (q_{\gamma\delta}^{n+1}\coef{-\delta,\alpha,\gamma}{n} + q_{\alpha\gamma}^{n+1} q_{\beta\gamma}^{n+1} (n+2)_{\widetilde{q}_{\alpha\gamma}}) \Bsj_2 \Bsj_4 \x_{\gamma}^{n+2} \ot x_{\gamma}
\\ & \quad
+q_{\beta\delta} (-q_{\beta\gamma})^{n+1} \x_{\alpha} \x_{\gamma}^{n+1} \otimes  x_{\delta}x_{\beta}
-q_{\alpha\beta} (-q_{\alpha\gamma})^{n+1} q_{\alpha\delta} \x_{\beta}\x_{\gamma}^{n+1} \otimes x_{\delta}x_{\alpha}
\\ & \quad
-q_{\alpha\beta} (-q_{\alpha\gamma})^{n+1} \Bsj_1 \x_{\beta}\x_{\gamma}^{n+1} \otimes x_{\eta}x_{\tau}
+ q_{\alpha\nu}q_{\alpha\gamma}^{n+2}q_{\beta\gamma}^{n+1}\Bsj_2 \x_{\gamma}^{n+1} \x_{\nu} \ot x_{\gamma}x_{\alpha} \big)
\\ &=
\x_{\alpha} \x_{\beta}\x_{\gamma}^{n+1} \otimes x_{\delta}
-q_{\gamma\delta} \x_{\alpha} \x_{\beta}\x_{\gamma}^{n}\x_{\delta} \otimes x_{\gamma}
-(-q_{\beta\gamma})^{n+1}\Bsj_2 \x_{\alpha} \x_{\gamma}^{n+1} \x_{\nu} \otimes x_{\gamma}
\\ & \quad
-q_{\beta\delta} (-q_{\beta\gamma})^{n+1} \x_{\alpha} \x_{\gamma}^{n+1}\x_{\delta} \otimes  x_{\beta}
+q_{\alpha\beta} (-q_{\alpha\gamma})^{n+1} \Bsj_1 \x_{\beta}\x_{\gamma}^{n+1}\x_{\eta} \otimes x_{\tau}
\\ & \quad
+q_{\alpha\beta} (-q_{\alpha\gamma})^{n+1} q_{\alpha\delta} \x_{\beta}\x_{\gamma}^{n+1}\x_{\delta} \otimes x_{\alpha}
- s \big( -q_{\alpha\beta} q_{\alpha\gamma}^{n+1}q_{\beta\gamma}^{n+1} \Bsj_1 \Bsj_3 \x_{\gamma}^{n+1}\x_{\mu}\ot x_{\gamma}x_{\tau}
\\ & \quad
+ q_{\gamma\delta}^{n+1}(\coef{-\delta,\alpha,\gamma}{n} + q_{\alpha\gamma}^{n+1} q_{\beta\gamma}^{n+1}q_{\gamma\delta}^{-n-1} (n+2)_{\widetilde{q}_{\alpha\gamma}}) \Bsj_2 \Bsj_4 \x_{\gamma}^{n+2} \ot x_{\gamma} \big).
\end{align*}
As $\alpha+\beta=2\gamma-\delta$, we have 
$q_{\alpha\gamma}^{n+1} q_{\beta\gamma}^{n+1}q_{\gamma\delta}^{-n-1} = \widetilde{q}_{\gamma\delta}^{\, -n-1}$. Also, $s(\x_{\gamma}^{n+2} \ot x_{\gamma})=\x_{\gamma}^{n+3} \ot 1$, so
\begin{align*}
d(\x_{\alpha} & \x_{\beta} \x_{\gamma}^{n+1} \x_{\delta}\ot 1)  =
\x_{\alpha} \x_{\beta}\x_{\gamma}^{n+1} \otimes x_{\delta}
-q_{\gamma\delta} \x_{\alpha} \x_{\beta}\x_{\gamma}^{n}\x_{\delta} \otimes x_{\gamma}
-(-q_{\beta\gamma})^{n+1}\Bsj_2 \x_{\alpha} \x_{\gamma}^{n+1} \x_{\nu} \otimes x_{\gamma}
\\ & \quad
-q_{\beta\delta} (-q_{\beta\gamma})^{n+1} \x_{\alpha} \x_{\gamma}^{n+1}\x_{\delta} \otimes  x_{\beta}
+q_{\alpha\beta} (-q_{\alpha\gamma})^{n+1} \Bsj_1 \x_{\beta}\x_{\gamma}^{n+1}\x_{\eta} \otimes x_{\tau}
\\ & \quad
+q_{\alpha\beta} (-q_{\alpha\gamma})^{n+1} q_{\alpha\delta} \x_{\beta}\x_{\gamma}^{n+1}\x_{\delta} \otimes x_{\alpha}
- q_{\gamma\delta}^{n+1}\coef{-\delta,\alpha,\gamma}{n+1}\Bsj_2 \Bsj_4 \x_{\gamma}^{n+3} \ot 1
\\ & \quad
+ q_{\alpha\beta} q_{\alpha\gamma}^{n+1}q_{\beta\gamma}^{n+1} \Bsj_1 \Bsj_3 s \big( \x_{\gamma}^{n+1}\x_{\mu}\ot x_{\gamma}x_{\tau}\big).
\end{align*}
Next we claim that $\x_{\gamma}^{n+1}\x_{\mu}\ot x_{\gamma}x_{\tau} = s (\x_{\gamma}^{n+1}\ot x_{\mu}x_{\gamma}x_{\tau})$. Indeed,
\begin{align*}
d(\x_{\gamma}^{n+1}\ot x_{\mu}x_{\gamma}x_{\tau}) &= \x_{\gamma}^{n}\ot x_{\gamma}x_{\mu}x_{\gamma}x_{\tau}
= q_{\gamma\mu} \x_{\gamma}^{n}\ot x_{\mu}x_{\gamma}^2 x_{\tau} =0,
\end{align*}
so $\x_{\gamma}^{n+1}\ot x_{\mu}x_{\gamma}x_{\tau}\in\ker d_n$, and we compute
\begin{align*}
s\big(\x_{\gamma}^{n+1} & \ot x_{\mu}x_{\gamma}x_{\tau}\big) =
\x_{\gamma}^{n+1}\x_{\mu}\ot x_{\gamma}x_{\tau} + s\big(\x_{\gamma}^{n+1}\ot x_{\mu}x_{\gamma}x_{\tau} -d(\x_{\gamma}^{n+1}\x_{\mu}\ot x_{\gamma}x_{\tau}) \big)
\\ &
= \x_{\gamma}^{n+1}\x_{\mu}\ot x_{\gamma}x_{\tau} + s\big(\x_{\gamma}^{n+1}\ot x_{\mu}x_{\gamma}x_{\tau} -(\x_{\gamma}^{n+1}\ot x_{\mu}x_{\gamma}x_{\tau}
-q_{\gamma\mu} \x_{\gamma}^{n} \x_{\mu} \ot x_{\gamma}^2 x_{\tau} ) \big)
\\ &
= \x_{\gamma}^{n+1}\x_{\mu}\ot x_{\gamma}x_{\tau}.
\end{align*}
From this claim, $s \big( \x_{\gamma}^{n+1}\x_{\mu}\ot x_{\gamma}x_{\tau}\big)=0$, and the inductive step follows.
\epf

\begin{lema}\label{lem:diff-case9}
Let $\alpha < \beta < \nu < \gamma < \mu <\delta < \eta $ be positive roots
such that $N_{\gamma}=2$ and the relations among the corresponding root vectors take the form
\begin{align}\label{eq:diff-case9-hypothesis}
\begin{aligned}
x_{\beta}x_{\delta} &= q_{\beta\delta} x_{\delta}x_{\beta} +\Bsj_1 x_{\gamma}x_{\nu}, &
x_{\nu}x_{\eta} &= q_{\nu\eta} x_{\eta}x_{\nu} +\Bsj_2 x_{\mu}x_{\gamma}, 
\\
x_{\alpha}x_{\mu} &= q_{\alpha\mu}  x_{\mu}x_{\alpha} + \Bsj_3 x_{\gamma},
\end{aligned}
\end{align}
for some scalars $\Bsj_i$, $x_{\gamma}$ $q$-commutes with the other root vectors and the following pairs of root vectors also $q$-commute: $(x_{\alpha},x_{\beta})$, $(x_{\alpha},x_{\nu})$, $(x_{\alpha},x_{\delta})$, $(x_{\alpha},x_{\eta})$, $(x_{\beta},x_{\nu})$, $(x_{\beta},x_{\eta})$, $(x_{\nu},x_{\delta})$, $(x_{\nu},x_{\eta})$, $(x_{\mu},x_{\delta})$, $(x_{\mu},x_{\eta})$.
Then, for all $n\geq 0$, 
\begin{align}\label{eq:diff-case9-formula}
\begin{aligned}
d(&\x_{\alpha} \x_{\beta} \x_{\gamma}^{n} \x_{\delta}\x_{\eta}\ot 1)  = \x_{\alpha}\x_{\beta} \x_{\gamma}^{n} \x_{\delta}\ot x_{\eta}
+(-q_{\beta\gamma})^{n} (n+1)_{\widetilde{q}_{\delta\gamma}} \Bsj_1\Bsj_2 \x_{\alpha}\x_{\gamma}^{n+1}\x_{\mu} \ot x_{\gamma}
\\ &
-q_{\delta\eta} \x_{\alpha}\x_{\beta} \x_{\gamma}^{n} \x_{\eta}
\ot x_{\delta}
+q_{\nu\eta}(-q_{\beta\gamma})^{n} (n+1)_{\widetilde{q}_{\delta\gamma}}\Bsj_1 \x_{\alpha}\x_{\gamma}^{n+1}\x_{\eta} \ot x_{\nu}
\\ & 
+q_{\gamma\delta}q_{\gamma\eta} \x_{\alpha}\x_{\beta} \x_{\gamma}^{n-1} \x_{\delta}\x_{\eta} \ot x_{\gamma}
+(-q_{\beta\gamma})^n q_{\beta\delta} q_{\beta\eta}
\x_{\alpha}\x_{\gamma}^{n} \x_{\delta}\x_{\eta} \ot x_{\beta}
\\ &
-q_{\alpha\beta}(-q_{\alpha\gamma})^n q_{\alpha\delta}q_{\alpha\eta}
\x_{\beta} \x_{\gamma}^{n} \x_{\delta}\x_{\eta} \ot x_{\alpha}
-q_{\alpha\gamma}q_{\gamma\delta}^{n}q_{\gamma\eta}^{n} \coeff{\alpha+\beta,\delta,\alpha,\gamma}{n} \Bsj_1\Bsj_2\Bsj_3
\x_{\gamma}^{n+3} \ot 1.
\end{aligned}
\end{align}
\end{lema} 

\bigbreak
Notice that the equalities in \eqref{eq:diff-case9-hypothesis} force
\begin{align}\label{eq:diff-case9-hypothesis-roots}
\begin{aligned}
\beta+\delta &= \gamma+\nu, &
\nu+\eta &= \mu+\gamma, &
\alpha+\mu &= \gamma.
\end{aligned}
\end{align}
Hence the following equality also holds: $3\gamma=\alpha+\beta+\delta+\eta$.

\pf
We need some auxiliary computations. By \eqref{eq:diff-case2-auxiliar}
\begin{align}\label{eq:diff-case9-auxiliar1}
\begin{aligned}
d(\x_{\beta} \x_{\gamma}^{n} \x_{\delta}\ot 1)  &=
\x_{\beta} \x_{\gamma}^{n}\otimes x_{\delta} 
-q_{\gamma\delta} \x_{\beta} \x_{\gamma}^{n-1}\x_{\delta} \ot x_{\gamma} 
-q_{\beta\delta}(-q_{\beta\gamma})^{n} \x_{\gamma}^n\x_{\delta}\ot x_{\beta}
\\ &
- (-q_{\beta\gamma})^{n} (n+1)_{\widetilde{q}_{\delta\gamma}}\Bsj_1  \x_{\gamma}^{n+1} \ot x_{\nu}
\end{aligned}
\end{align}
Next we apply Lemma \ref{lem:diff2} to $\alpha < \gamma < \mu$: 
\begin{align}\label{eq:diff-case9-auxiliar2}
\begin{aligned}
d(\x_{\alpha} & \x_{\gamma}^{n} \x_{\mu}\ot 1)  =
\x_{\alpha}\x_{\gamma}^{n} \ot x_{\mu} 
- q_{\gamma\mu} \x_{\alpha} \x_{\gamma}^{n-1}\x_{\mu}\ot x_{\gamma} 
- q_{\alpha\mu}(-q_{\alpha\gamma})^{n} \x_{\gamma}^{n} \x_{\mu} \ot x_{\alpha}  
\\
& 
- (-q_{\alpha\gamma})^{n} (n+1)_{\widetilde{q}_{\alpha\gamma}} \Bsj_3
\x_{\gamma}^{n+1} \ot 1.
\end{aligned}
\end{align}

Now we prove the following equality by induction on $n$:
\begin{align}\label{eq:diff-case9-auxiliar3}
&\begin{aligned}
d(\x_{\alpha}& \x_{\beta} \x_{\gamma}^{n} \x_{\delta}\ot 1)  = \x_{\alpha}\x_{\beta} \x_{\gamma}^{n} \ot x_{\delta}
-q_{\gamma\delta} \x_{\alpha}\x_{\beta}\x_{\gamma}^{n-1} \x_{\delta} \ot x_{\gamma}
\\ &
-(-q_{\beta\gamma})^n q_{\beta\delta} \x_{\alpha}\x_{\gamma}^{n} \x_{\delta}\ot x_{\beta}
+q_{\alpha\beta}(-q_{\alpha\gamma})^n q_{\alpha\delta}
\x_{\beta}\x_{\gamma}^{n} \x_{\delta} \ot x_{\alpha}
\\ &
-(-q_{\beta\gamma})^{n} (n+1)_{\widetilde{q}_{\delta\gamma}}\Bsj_1 \x_{\alpha}\x_{\gamma}^{n+1} \ot x_{\nu}.
\end{aligned}
\\ \label{eq:diff-case9-auxiliar4}
&\begin{aligned}
d(& \x_{\beta} \x_{\gamma}^{n} \x_{\delta}\x_{\eta}\ot 1)  = \x_{\beta} \x_{\gamma}^{n} \x_{\delta}\ot x_{\eta}
-q_{\delta\eta} \x_{\beta}\x_{\gamma}^{n}\x_{\eta} \ot x_{\delta}
\\ &
+q_{\gamma\delta}q_{\gamma\eta} \x_{\beta}\x_{\gamma}^{n-1} \x_{\delta}\x_{\eta} \ot x_{\gamma}
+ (-q_{\beta\gamma})^{n} (n+1)_{\widetilde{q}_{\delta\gamma}} \Bsj_1\Bsj_2 \x_{\gamma}^{n+1}\x_{\mu} \ot x_{\gamma}
\\ &
+(-q_{\beta\gamma})^n q_{\beta\delta}q_{\beta\eta}
\x_{\gamma}^{n} \x_{\delta}\x_{\eta} \ot x_{\beta} 
+ (-q_{\beta\gamma})^{n}q_{\nu\eta} (n+1)_{\widetilde{q}_{\delta\gamma}} \Bsj_1 \x_{\gamma}^{n+1}\x_{\eta} \ot x_{\nu}.
\end{aligned}
\end{align}

Indeed, for $n=0$ we have:
\begin{align*}
d(\x_{\alpha}& \x_{\beta}\x_{\delta}\ot 1)  
= \x_{\alpha} \x_{\beta}\ot x_{\delta} -s\big( 
q_{\beta\delta} \x_{\alpha} \ot  x_{\delta}x_{\beta}
+\Bsj_1 \x_{\alpha} \ot x_{\gamma}x_{\nu}
-q_{\alpha\beta}q_{\alpha\delta} \x_{\beta}\ot x_{\delta}x_{\alpha} \big)
\\ &= 
\x_{\alpha} \x_{\beta}\ot x_{\delta} 
-q_{\beta\delta} \x_{\alpha}\x_{\delta} \ot  x_{\beta}
-\Bsj_1 \x_{\alpha}\x_{\gamma} \ot x_{\nu} -s\big(
-q_{\alpha\beta}q_{\alpha\delta} \x_{\beta}\ot x_{\delta}x_{\alpha}
\\ & \quad
+q_{\alpha\gamma}q_{\alpha\nu}\Bsj_1 \x_{\gamma}\ot x_{\nu}x_{\alpha}
+q_{\alpha\beta}q_{\beta\delta}q_{\alpha\delta} \x_{\delta}\ot x_{\beta}x_{\alpha}\big)
\\ &= 
\x_{\alpha} \x_{\beta}\ot x_{\delta} 
-q_{\beta\delta} \x_{\alpha}\x_{\delta} \ot  x_{\beta}
-\Bsj_1 \x_{\alpha}\x_{\gamma} \ot x_{\nu} 
+q_{\alpha\beta}q_{\alpha\delta} \x_{\beta}\x_{\delta}\ot x_{\alpha}.
\end{align*}
And for the other equality,
\begin{align*}
d(& \x_{\beta}\x_{\delta}\x_{\eta}\ot 1)  
= \x_{\beta}\x_{\delta}\ot x_{\eta}
-s \big(
q_{\delta\eta}\x_{\beta} \ot x_{\eta}x_{\delta}
- q_{\beta\delta}q_{\beta\eta} \x_{\delta} \ot x_{\eta}x_{\beta} 
-\Bsj_1 \x_{\gamma} \ot x_{\nu}x_{\eta} \big)
\\ & = 
\x_{\beta}\x_{\delta}\ot x_{\eta}
-q_{\delta\eta}\x_{\beta}\x_{\eta} \ot x_{\delta}
-s \big(
- q_{\beta\delta}q_{\beta\eta} \x_{\delta} \ot x_{\eta}x_{\beta} 
-q_{\nu\eta}\Bsj_1 \x_{\gamma} \ot x_{\eta}x_{\nu}
\\ & \quad
-\Bsj_1\Bsj_2 \x_{\gamma} \ot x_{\mu}x_{\gamma} 
+q_{\beta\delta}q_{\beta\eta}q_{\delta\eta} \x_{\eta}\ot  x_{\delta}x_{\beta}
+q_{\beta\eta}q_{\delta\eta}\Bsj_1 \x_{\eta}\ot  x_{\gamma}x_{\nu}
\big)
\\ & = 
\x_{\beta}\x_{\delta}\ot x_{\eta}
-q_{\delta\eta}\x_{\beta}\x_{\eta} \ot x_{\delta}
+\Bsj_1\Bsj_2 \x_{\gamma}\x_{\mu} \ot x_{\gamma}
+q_{\nu\eta}\Bsj_1 \x_{\gamma}\x_{\eta} \ot x_{\nu}
+q_{\beta\delta}q_{\beta\eta} \x_{\delta}\x_{\eta} \ot x_{\beta}.
\end{align*}

Now assume that \eqref{eq:diff-case9-auxiliar3} holds for $n$. Using Remark \ref{rem:differential-q-commute} repeatedly, inductive hypothesis, the relation $x_{\gamma}^2=0$ and \eqref{eq:diff-case9-auxiliar1},
\begin{align*}
d(\x_{\alpha} & \x_{\beta} \x_{\gamma}^{n+1} \x_{\delta} \ot 1)  = 
\x_{\alpha}\x_{\beta} \x_{\gamma}^{n+1} \ot x_{\delta}
-s \big( q_{\gamma\delta}\x_{\alpha}\x_{\beta} \x_{\gamma}^{n} \ot x_{\delta}x_{\gamma}
+(-q_{\beta\gamma})^{n+1}\Bsj_1 \x_{\alpha}\x_{\gamma}^{n+1} \ot x_{\gamma}x_{\nu}
\\ & \quad 
+(-q_{\beta\gamma})^{n+1}q_{\beta\delta} \x_{\alpha}\x_{\gamma}^{n+1} \ot x_{\delta}x_{\beta}
-q_{\alpha\beta}(-q_{\alpha\gamma})^{n+1}q_{\alpha\delta} \x_{\beta} \x_{\gamma}^{n+1} \ot x_{\delta}x_{\alpha} \big)
\\ &  = 
\x_{\alpha}\x_{\beta} \x_{\gamma}^{n+1} \ot x_{\delta}
-q_{\gamma\delta}\x_{\alpha}\x_{\beta} \x_{\gamma}^{n}\x_{\delta} \ot x_{\gamma}
-s \big( 
(-q_{\beta\gamma})^{n+1}q_{\beta\delta} \x_{\alpha}\x_{\gamma}^{n+1} \ot x_{\delta}x_{\beta}
\\ & \quad 
+ (-q_{\beta\gamma})^{n+1} (1 +q_{\gamma\delta}q_{\nu\gamma} (-q_{\beta\gamma})^{-1} (n+1)_{\widetilde{q}_{\delta\gamma}})\Bsj_1 \x_{\alpha}\x_{\gamma}^{n+1} \ot x_{\gamma}x_{\nu}
\\ & \quad
-q_{\alpha\beta}(-q_{\alpha\gamma})^{n+1}q_{\alpha\delta} \x_{\beta} \x_{\gamma}^{n+1} \ot x_{\delta}x_{\alpha} 
-(-q_{\beta\gamma})^{n+1} q_{\beta\delta}q_{\gamma\delta} \x_{\alpha}\x_{\gamma}^{n} \x_{\delta}\ot x_{\gamma}x_{\beta}
\\ & \quad
+q_{\alpha\beta}(-q_{\alpha\gamma})^{n+1} q_{\alpha\delta} q_{\gamma\delta}
\x_{\beta}\x_{\gamma}^{n} \x_{\delta} \ot x_{\gamma}x_{\alpha}\big)
\\ &  = 
\x_{\alpha}\x_{\beta} \x_{\gamma}^{n+1} \ot x_{\delta}
-q_{\gamma\delta}\x_{\alpha}\x_{\beta} \x_{\gamma}^{n}\x_{\delta} \ot x_{\gamma}
-(-q_{\beta\gamma})^{n+1}q_{\beta\delta} \x_{\alpha}\x_{\gamma}^{n+1} \x_{\delta} \ot x_{\beta}
\\ & \quad
-(-q_{\beta\gamma})^{n+1} (n+2)_{\widetilde{q}_{\delta\gamma}} \Bsj_1 \x_{\alpha}\x_{\gamma}^{n+2} \ot x_{\nu}
-s \big( 
-q_{\alpha\beta}(-q_{\alpha\gamma})^{n+1}q_{\alpha\delta} \x_{\beta} \x_{\gamma}^{n+1} \ot x_{\delta}x_{\alpha} 
\\ & \quad
+q_{\alpha\beta}(-q_{\alpha\gamma})^{n+1} q_{\alpha\delta} q_{\gamma\delta}
\x_{\beta}\x_{\gamma}^{n} \x_{\delta} \ot x_{\gamma}x_{\alpha}
+q_{\alpha\beta}q_{\alpha\gamma}^{n+1} q_{\alpha\delta}q_{\beta\gamma}^{n+1}q_{\beta\delta}  \x_{\gamma}^{n+1} \x_{\delta} \ot x_{\beta}x_{\alpha}
\\ & \quad
+q_{\alpha\gamma}^{n+2}q_{\beta\gamma}^{n+1} (n+2)_{\widetilde{q}_{\delta\gamma}} \Bsj_1 \x_{\gamma}^{n+2} \ot x_{\alpha}x_{\nu} \big)
\\ &  = 
\x_{\alpha}\x_{\beta} \x_{\gamma}^{n+1} \ot x_{\delta}
-q_{\gamma\delta}\x_{\alpha}\x_{\beta} \x_{\gamma}^{n}\x_{\delta} \ot x_{\gamma}
-(-q_{\beta\gamma})^{n+1}q_{\beta\delta} \x_{\alpha}\x_{\gamma}^{n+1} \x_{\delta} \ot x_{\beta}
\\ & \quad
-(-q_{\beta\gamma})^{n+1} (n+2)_{\widetilde{q}_{\delta\gamma}} \Bsj_1 \x_{\alpha}\x_{\gamma}^{n+2} \ot x_{\nu}
+q_{\alpha\beta}(-q_{\alpha\gamma})^{n+1}q_{\alpha\delta} \x_{\beta} \x_{\gamma}^{n+1}\x_{\delta} \ot x_{\alpha}.
\end{align*}

Next we assume that \eqref{eq:diff-case9-auxiliar4} holds for $n$. Using \eqref{eq:diff-case9-auxiliar1}, inductive hypothesis, Remark \ref{rem:differential-q-commute}, the relation $x_{\gamma}^2=0$,
\begin{align*}
d(& \x_{\beta} \x_{\gamma}^{n+1} \x_{\delta}\x_{\eta}\ot 1)  = 
\x_{\beta} \x_{\gamma}^{n+1} \x_{\delta}\ot x_{\eta}
-s \big(
q_{\delta\eta}\x_{\beta} \x_{\gamma}^{n+1}\otimes x_{\eta}x_{\delta}
-q_{\gamma\delta}q_{\gamma\eta} \x_{\beta} \x_{\gamma}^{n}\x_{\delta} \ot x_{\eta}x_{\gamma}
\\ & \quad
-(-q_{\beta\gamma})^{n+1}q_{\beta\delta}q_{\beta\eta} \x_{\gamma}^{n+1} \x_{\delta}\ot x_{\eta}x_{\beta}
- (-q_{\beta\gamma})^{n+1} (n+2)_{\widetilde{q}_{\delta\gamma}} \Bsj_1  \x_{\gamma}^{n+2} \ot x_{\nu}x_{\eta} \big)
\\ & = 
\x_{\beta} \x_{\gamma}^{n+1} \x_{\delta}\ot x_{\eta}
-q_{\delta\eta}\x_{\beta} \x_{\gamma}^{n+1}\x_{\eta} \otimes x_{\delta}
+q_{\gamma\delta}q_{\gamma\eta} \x_{\beta} \x_{\gamma}^{n}\x_{\delta}\x_{\eta} \ot x_{\gamma}
-s \big(
\\ & \quad
- (-q_{\beta\gamma})^{n+1} (n+2)_{\widetilde{q}_{\delta\gamma}} \Bsj_1\Bsj_2 \x_{\gamma}^{n+2} \ot x_{\mu}x_{\gamma}
- (-q_{\beta\gamma})^{n+1}q_{\nu\eta} (n+2)_{\widetilde{q}_{\delta\gamma}} \Bsj_1 \x_{\gamma}^{n+2} \ot x_{\eta}x_{\nu}
\\ & \quad
-(-q_{\beta\gamma})^{n+1}q_{\beta\delta}q_{\beta\eta} \x_{\gamma}^{n+1} \x_{\delta}\ot x_{\eta}x_{\beta}
\\ & \quad
+q_{\nu\eta}q_{\gamma\eta} (-q_{\beta\gamma})^{n+1} (1 +\widetilde{q}_{\gamma\delta} (n+1)_{\widetilde{q}_{\delta\gamma}}
) \Bsj_1 \x_{\gamma}^{n+1}\x_{\eta} \ot x_{\gamma}x_{\nu}
\\ & \quad
+q_{\delta\eta}(-q_{\beta\gamma})^{n+1}q_{\beta\delta}q_{\beta\eta} \x_{\gamma}^{n+1}\x_{\eta} \ot  x_{\delta}x_{\beta}
-q_{\gamma\delta}q_{\gamma\eta}(-q_{\beta\gamma})^{n+1} q_{\beta\delta}q_{\beta\eta}
\x_{\gamma}^{n} \x_{\delta}\x_{\eta} \ot x_{\gamma}x_{\beta} 
\big)
\\ & = 
\x_{\beta} \x_{\gamma}^{n+1} \x_{\delta}\ot x_{\eta}
-q_{\delta\eta}\x_{\beta} \x_{\gamma}^{n+1}\x_{\eta} \otimes x_{\delta}
+q_{\gamma\delta}q_{\gamma\eta} \x_{\beta} \x_{\gamma}^{n}\x_{\delta}\x_{\eta} \ot x_{\gamma}
\\ & \quad
+ (-q_{\beta\gamma})^{n+1} (n+2)_{\widetilde{q}_{\delta\gamma}} \Bsj_1\Bsj_2 \x_{\gamma}^{n+2}\x_{\mu} \ot x_{\gamma}
+ (-q_{\beta\gamma})^{n+1}q_{\nu\eta} (n+2)_{\widetilde{q}_{\delta\gamma}} \Bsj_1 \x_{\gamma}^{n+2}\x_{\eta} \ot x_{\nu}
\\ & \quad
-s \big(
-(-q_{\beta\gamma})^{n+1}q_{\beta\delta}q_{\beta\eta} \x_{\gamma}^{n+1} \x_{\delta}\ot x_{\eta}x_{\beta}
+q_{\delta\eta}(-q_{\beta\gamma})^{n+1}q_{\beta\delta}q_{\beta\eta} \x_{\gamma}^{n+1}\x_{\eta} \ot  x_{\delta}x_{\beta}
\\ & \quad
-q_{\gamma\delta}q_{\gamma\eta}(-q_{\beta\gamma})^{n+1} q_{\beta\delta}q_{\beta\eta}
\x_{\gamma}^{n} \x_{\delta}\x_{\eta} \ot x_{\gamma}x_{\beta} 
\big)
\\ & = 
\x_{\beta} \x_{\gamma}^{n+1} \x_{\delta}\ot x_{\eta}
-q_{\delta\eta}\x_{\beta} \x_{\gamma}^{n+1}\x_{\eta} \otimes x_{\delta}
+q_{\gamma\delta}q_{\gamma\eta} \x_{\beta} \x_{\gamma}^{n}\x_{\delta}\x_{\eta} \ot x_{\gamma}
\\ & \quad
+ (-q_{\beta\gamma})^{n+1} (n+2)_{\widetilde{q}_{\delta\gamma}} \Bsj_1\Bsj_2 \x_{\gamma}^{n+2}\x_{\mu} \ot x_{\gamma}
+ (-q_{\beta\gamma})^{n+1}q_{\nu\eta} (n+2)_{\widetilde{q}_{\delta\gamma}} \Bsj_1 \x_{\gamma}^{n+2}\x_{\eta} \ot x_{\nu}
\\ & \quad
+(-q_{\beta\gamma})^{n+1}q_{\beta\delta}q_{\beta\eta} \x_{\gamma}^{n+1} \x_{\delta}\x_{\eta} \ot x_{\beta}.
\end{align*}

\medskip

Finally we prove \eqref{eq:diff-case9-formula} by induction on $n$. When $n=0$,
\begin{align*}
d( & \x_{\alpha} \x_{\beta} \x_{\delta} \x_{\eta} \ot 1) = 
\x_{\alpha} \x_{\beta} \x_{\delta} \ot x_{\eta}
-s\big( q_{\delta\eta}
\x_{\alpha} \x_{\beta}\ot x_{\eta}x_{\delta}
-q_{\beta\eta}q_{\beta\delta} \x_{\alpha}\x_{\delta} \ot x_{\eta}x_{\beta}
\\ & \quad
-q_{\nu\eta}\Bsj_1 \x_{\alpha}\x_{\gamma} \ot  x_{\eta}x_{\nu}
-\Bsj_1\Bsj_2 \x_{\alpha}\x_{\gamma} \ot  x_{\mu}x_{\gamma}
+q_{\alpha\beta}q_{\alpha\delta}q_{\alpha\eta} \x_{\beta}\x_{\delta}\ot x_{\eta}x_{\alpha} \big)
\\ & = 
\x_{\alpha} \x_{\beta} \x_{\delta} \ot x_{\eta}
-q_{\delta\eta}
\x_{\alpha} \x_{\beta}\x_{\eta} \ot x_{\delta}
+q_{\beta\eta}q_{\beta\delta} \x_{\alpha}\x_{\delta}\x_{\eta} \ot x_{\beta}
-s\big(-q_{\nu\eta}\Bsj_1 \x_{\alpha}\x_{\gamma} \ot  x_{\eta}x_{\nu}
\\ & \quad
-\Bsj_1\Bsj_2 \x_{\alpha}\x_{\gamma} \ot  x_{\mu}x_{\gamma}
+q_{\beta\eta}q_{\delta\eta}\Bsj_1 \x_{\alpha}\x_{\eta} \ot x_{\gamma}x_{\nu}
+q_{\alpha\beta}q_{\alpha\delta}q_{\alpha\eta} \x_{\beta}\x_{\delta}\ot x_{\eta}x_{\alpha} 
\\ & \quad
-q_{\alpha\beta}q_{\alpha\eta}q_{\alpha\delta}q_{\delta\eta} \x_{\beta}\x_{\eta} \ot x_{\delta}x_{\alpha}
+q_{\alpha\beta}q_{\alpha\delta}q_{\alpha\eta} q_{\beta\eta}q_{\beta\delta} \x_{\delta}\x_{\eta} \ot x_{\beta}x_{\alpha}
\big)
\\ & = 
\x_{\alpha} \x_{\beta} \x_{\delta} \ot x_{\eta}
-q_{\delta\eta}
\x_{\alpha} \x_{\beta}\x_{\eta} \ot x_{\delta}
+q_{\beta\eta}q_{\beta\delta} \x_{\alpha}\x_{\delta}\x_{\eta} \ot x_{\beta}
+q_{\nu\eta}\Bsj_1 \x_{\alpha}\x_{\gamma}\x_{\eta} \ot x_{\nu}
\\ & \quad
+\Bsj_1\Bsj_2 \x_{\alpha}\x_{\gamma}\x_{\mu} \ot  x_{\gamma}
-s\big(q_{\alpha\beta}q_{\alpha\delta}q_{\alpha\eta} \x_{\beta}\x_{\delta}\ot x_{\eta}x_{\alpha}
+q_{\alpha\gamma} (2)_{\widetilde{q}_{\alpha\gamma}} \Bsj_1\Bsj_2\Bsj_3
\x_{\gamma}^2 \ot x_{\gamma}
\\ & \quad
-q_{\alpha\beta}q_{\alpha\eta}q_{\alpha\delta}q_{\delta\eta} \x_{\beta}\x_{\eta} \ot x_{\delta}x_{\alpha}
+q_{\alpha\beta}q_{\alpha\delta}q_{\alpha\eta} q_{\beta\eta}q_{\beta\delta} \x_{\delta}\x_{\eta} \ot x_{\beta}x_{\alpha}
\\ & \quad
+q_{\alpha\gamma}q_{\alpha\eta}q_{\alpha\nu}q_{\nu\eta}\Bsj_1 \x_{\gamma}\x_{\eta} \ot x_{\nu}x_{\alpha}
+q_{\alpha\mu}q_{\alpha\gamma}^2\Bsj_1\Bsj_2 \x_{\gamma}\x_{\mu} \ot x_{\gamma}x_{\alpha} 
\big)
\\ & = 
\x_{\alpha} \x_{\beta} \x_{\delta} \ot x_{\eta}
-q_{\delta\eta}
\x_{\alpha} \x_{\beta}\x_{\eta} \ot x_{\delta}
+q_{\beta\eta}q_{\beta\delta} \x_{\alpha}\x_{\delta}\x_{\eta} \ot x_{\beta}
+q_{\nu\eta}\Bsj_1 \x_{\alpha}\x_{\gamma}\x_{\eta} \ot x_{\nu}
\\ & \quad
+\Bsj_1\Bsj_2 \x_{\alpha}\x_{\gamma}\x_{\mu} \ot  x_{\gamma}
-q_{\alpha\beta}q_{\alpha\delta}q_{\alpha\eta} \x_{\beta}\x_{\delta}\x_{\eta} \ot x_{\alpha}
-q_{\alpha\gamma} (2)_{\widetilde{q}_{\alpha\gamma}} \Bsj_1\Bsj_2\Bsj_3
\x_{\gamma}^3 \ot 1,
\end{align*}
which is \eqref{eq:diff-case9-formula} for $n=0$.
Now assume that \eqref{eq:diff-case9-formula} holds for $n$. Using \eqref{eq:diff-case9-auxiliar3}, Remark \ref{rem:differential-q-commute}, inductive hypothesis, the relation $x_{\gamma}^2=0$, \eqref{eq:diff-case9-auxiliar2}, \eqref{eq:diff-case9-auxiliar4},
\begin{align*}
d(&\x_{\alpha}\x_{\beta}\x_{\gamma}^{n+1}\x_{\delta}\x_{\eta}\ot 1)  
= \x_{\alpha}\x_{\beta}\x_{\gamma}^{n+1}\x_{\delta}\ot x_{\eta}
-s \big( 
q_{\delta\eta} \x_{\alpha}\x_{\beta} \x_{\gamma}^{n+1} \ot x_{\eta}x_{\delta}
\\ & \quad
-(-q_{\beta\gamma})^{n+1} q_{\beta\delta}q_{\beta\eta} \x_{\alpha}\x_{\gamma}^{n+1} \x_{\delta}\ot x_{\eta}x_{\beta}
+q_{\alpha\beta}(-q_{\alpha\gamma})^{n+1} q_{\alpha\delta}q_{\alpha\eta}
\x_{\beta}\x_{\gamma}^{n+1} \x_{\delta} \ot x_{\eta}x_{\alpha}
\\ & \quad
-q_{\gamma\delta}q_{\gamma\eta} \x_{\alpha}\x_{\beta}\x_{\gamma}^{n} \x_{\delta} \ot x_{\eta}x_{\gamma}
-(-q_{\beta\gamma})^{n+1} (n+2)_{\widetilde{q}_{\delta\gamma}}\Bsj_1 \x_{\alpha}\x_{\gamma}^{n+2} \ot (q_{\nu\eta} x_{\eta}x_{\nu} +\Bsj_2 x_{\mu}x_{\gamma}) \big)
\\ & 
= \x_{\alpha}\x_{\beta}\x_{\gamma}^{n+1}\x_{\delta}\ot x_{\eta}
-q_{\delta\eta} \x_{\alpha}\x_{\beta} \x_{\gamma}^{n+1}\x_{\eta} \ot x_{\delta}
-s \big(-(-q_{\beta\gamma})^{n+1} q_{\beta\delta}q_{\beta\eta} \x_{\alpha}\x_{\gamma}^{n+1} \x_{\delta}\ot x_{\eta}x_{\beta}
\\ & \quad
+q_{\alpha\beta}(-q_{\alpha\gamma})^{n+1} q_{\alpha\delta}q_{\alpha\eta}
\x_{\beta}\x_{\gamma}^{n+1} \x_{\delta} \ot x_{\eta}x_{\alpha}
-q_{\gamma\delta}q_{\gamma\eta} \x_{\alpha}\x_{\beta}\x_{\gamma}^{n} \x_{\delta} \ot x_{\eta}x_{\gamma}
\\ & \quad
-q_{\nu\eta}(-q_{\beta\gamma})^{n+1} (n+2)_{\widetilde{q}_{\delta\gamma}}\Bsj_1 \x_{\alpha}\x_{\gamma}^{n+2} \ot x_{\eta}x_{\nu}
+q_{\gamma\eta}q_{\delta\eta}q_{\gamma\delta} \x_{\alpha}\x_{\beta}\x_{\gamma}^{n}\x_{\eta} \ot x_{\delta}x_{\gamma}
\\ & \quad
-(-q_{\beta\gamma})^{n+1} (n+2)_{\widetilde{q}_{\delta\gamma}} \Bsj_1\Bsj_2 \x_{\alpha}\x_{\gamma}^{n+2} \ot x_{\mu}x_{\gamma}
-q_{\alpha\beta}(-q_{\alpha\gamma})^{n+1}q_{\alpha\delta}
q_{\alpha\eta}q_{\delta\eta} \x_{\beta}\x_{\gamma}^{n+1}\x_{\eta} \ot x_{\delta}x_{\alpha}
\\ & \quad
+(-q_{\beta\gamma})^{n+1}q_{\beta\eta}q_{\delta\eta} \x_{\alpha}\x_{\gamma}^{n+1} \x_{\eta} \ot (q_{\beta\delta} x_{\delta}x_{\beta} +\Bsj_1 x_{\gamma}x_{\nu})
\big)
\\ & 
= \x_{\alpha}\x_{\beta}\x_{\gamma}^{n+1}\x_{\delta}\ot x_{\eta}
-q_{\delta\eta} \x_{\alpha}\x_{\beta} \x_{\gamma}^{n+1}\x_{\eta} \ot x_{\delta}
+q_{\gamma\delta}q_{\gamma\eta} \x_{\alpha}\x_{\beta}\x_{\gamma}^{n} \x_{\delta}\x_{\eta} \ot x_{\gamma}
\\ & \quad
-s \big(
q_{\alpha\beta}(-q_{\alpha\gamma})^{n+1} q_{\alpha\delta}q_{\alpha\eta}
\x_{\beta}\x_{\gamma}^{n+1} \x_{\delta} \ot x_{\eta}x_{\alpha}
-(-q_{\beta\gamma})^{n+1} (n+2)_{\widetilde{q}_{\delta\gamma}} \Bsj_1\Bsj_2 \x_{\alpha}\x_{\gamma}^{n+2} \ot x_{\mu}x_{\gamma}
\\ & \quad
-(-q_{\beta\gamma})^{n+1} q_{\beta\delta}q_{\beta\eta} \x_{\alpha}\x_{\gamma}^{n+1} \x_{\delta}\ot x_{\eta}x_{\beta}
-q_{\nu\eta}(-q_{\beta\gamma})^{n+1} (n+2)_{\widetilde{q}_{\delta\gamma}}\Bsj_1 \x_{\alpha}\x_{\gamma}^{n+2} \ot x_{\eta}x_{\nu}
\\ & \quad
+(-q_{\beta\gamma})^{n}(q_{\nu\eta} (n+1)_{\widetilde{q}_{\delta\gamma}}
q_{\nu\gamma} q_{\gamma\delta}q_{\gamma\eta} -q_{\beta\gamma}q_{\beta\eta}q_{\delta\eta}) \Bsj_1 \x_{\alpha}\x_{\gamma}^{n+1} \x_{\eta} \ot x_{\gamma}x_{\nu}
\\ & \quad
-q_{\alpha\beta}(-q_{\alpha\gamma})^{n+1}q_{\alpha\delta}
q_{\alpha\eta}q_{\delta\eta} \x_{\beta}\x_{\gamma}^{n+1}\x_{\eta} \ot x_{\delta}x_{\alpha}
+(-q_{\beta\gamma})^{n+1}q_{\beta\delta}q_{\beta\eta}q_{\delta\eta} \x_{\alpha}\x_{\gamma}^{n+1} \x_{\eta} \ot  x_{\delta}x_{\beta}
\\ & \quad
- (-q_{\beta\gamma})^{n+1} q_{\beta\delta} q_{\beta\eta} q_{\gamma\delta}q_{\gamma\eta}
\x_{\alpha}\x_{\gamma}^{n} \x_{\delta}\x_{\eta} \ot x_{\gamma}x_{\beta}
-q_{\alpha\gamma}q_{\gamma\delta}^{n+1}q_{\gamma\eta}^{n+1} \coeff{\alpha+\beta,\delta,\alpha,\gamma}{n} \Bsj_1\Bsj_2\Bsj_3 \x_{\gamma}^{n+3} \ot x_{\gamma}
\\ & \quad
+q_{\alpha\beta}(-q_{\alpha\gamma})^{n+1} q_{\alpha\delta} q_{\alpha\eta} q_{\gamma\delta}q_{\gamma\eta} \x_{\beta} \x_{\gamma}^{n} \x_{\delta}\x_{\eta} \ot x_{\gamma}x_{\alpha}
\big)
\\ & 
= \x_{\alpha}\x_{\beta}\x_{\gamma}^{n+1}\x_{\delta}\ot x_{\eta}
-q_{\delta\eta} \x_{\alpha}\x_{\beta} \x_{\gamma}^{n+1}\x_{\eta} \ot x_{\delta}
+q_{\gamma\delta}q_{\gamma\eta} \x_{\alpha}\x_{\beta}\x_{\gamma}^{n} \x_{\delta}\x_{\eta} \ot x_{\gamma}
\\ & \quad
+(-q_{\beta\gamma})^{n+1} (n+2)_{\widetilde{q}_{\delta\gamma}} \Bsj_1\Bsj_2 \x_{\alpha}\x_{\gamma}^{n+2}\x_{\mu} \ot x_{\gamma}
+q_{\nu\eta}(-q_{\beta\gamma})^{n+1} (n+2)_{\widetilde{q}_{\delta\gamma}}\Bsj_1 \x_{\alpha}\x_{\gamma}^{n+2}\x_{\eta} \ot x_{\nu}
\\ & \quad
+(-q_{\beta\gamma})^{n+1} q_{\beta\delta}q_{\beta\eta} \x_{\alpha}\x_{\gamma}^{n+1} \x_{\delta}\x_{\eta}\ot x_{\beta}
-s \big(
-q_{\alpha\beta}(-q_{\alpha\gamma})^{n+1}q_{\alpha\delta}
q_{\alpha\eta}q_{\delta\eta} \x_{\beta}\x_{\gamma}^{n+1}\x_{\eta} \ot x_{\delta}x_{\alpha}
\\ & \quad
+q_{\alpha\beta}(-q_{\alpha\gamma})^{n+1} q_{\alpha\delta}q_{\alpha\eta}
\x_{\beta}\x_{\gamma}^{n+1} \x_{\delta} \ot x_{\eta}x_{\alpha}
+q_{\alpha\beta}q_{\alpha\gamma}^{n+1}q_{\alpha\delta}q_{\alpha\eta}
q_{\beta\gamma}^{n+1} q_{\beta\delta}q_{\beta\eta}
\x_{\gamma}^{n+1} \x_{\delta}\x_{\eta} \ot x_{\beta}x_{\alpha}
\\ & \quad
+q_{\alpha\beta}(-q_{\alpha\gamma})^{n+1} q_{\alpha\delta} q_{\alpha\eta} q_{\gamma\delta}q_{\gamma\eta} \x_{\beta} \x_{\gamma}^{n} \x_{\delta}\x_{\eta} \ot x_{\gamma}x_{\alpha}
\\ & \quad
-q_{\alpha\gamma}q_{\gamma\delta}^{n+1}q_{\gamma\eta}^{n+1}
(\widetilde{q}_{\alpha\gamma}^{n+1}\widetilde{q}_{\beta\gamma}^{n+1} (n+2)_{\widetilde{q}_{\delta\gamma}}
(n+3)_{\widetilde{q}_{\alpha\gamma}}
+\coeff{\alpha+\beta,\delta,\alpha,\gamma}{n}) \Bsj_1\Bsj_2\Bsj_3 \x_{\gamma}^{n+3} \ot x_{\gamma}
\\ & \quad
+q_{\alpha\mu}q_{\alpha\gamma}^{n+3}q_{\beta\gamma}^{n+1} (n+2)_{\widetilde{q}_{\delta\gamma}} \Bsj_1\Bsj_2 \x_{\gamma}^{n+2} \x_{\mu} \ot x_{\gamma}x_{\alpha}
\\ & \quad
+q_{\alpha\nu}q_{\alpha\gamma}^{n+2}q_{\alpha\eta}q_{\nu\eta}
q_{\beta\gamma}^{n+1} (n+2)_{\widetilde{q}_{\delta\gamma}}\Bsj_1 \x_{\gamma}^{n+2}\x_{\eta} \ot x_{\nu}x_{\alpha}
\big)
\\ & 
= \x_{\alpha}\x_{\beta}\x_{\gamma}^{n+1}\x_{\delta}\ot x_{\eta}
-q_{\delta\eta} \x_{\alpha}\x_{\beta} \x_{\gamma}^{n+1}\x_{\eta} \ot x_{\delta}
+q_{\gamma\delta}q_{\gamma\eta} \x_{\alpha}\x_{\beta}\x_{\gamma}^{n} \x_{\delta}\x_{\eta} \ot x_{\gamma}
\\ & \quad
+(-q_{\beta\gamma})^{n+1} (n+2)_{\widetilde{q}_{\delta\gamma}} \Bsj_1\Bsj_2 \x_{\alpha}\x_{\gamma}^{n+2}\x_{\mu} \ot x_{\gamma}
+q_{\nu\eta}(-q_{\beta\gamma})^{n+1} (n+2)_{\widetilde{q}_{\delta\gamma}}\Bsj_1 \x_{\alpha}\x_{\gamma}^{n+2}\x_{\eta} \ot x_{\nu}
\\ & \quad
+(-q_{\beta\gamma})^{n+1} q_{\beta\delta}q_{\beta\eta} \x_{\alpha}\x_{\gamma}^{n+1} \x_{\delta}\x_{\eta}\ot x_{\beta}
-q_{\alpha\beta}(-q_{\alpha\gamma})^{n+1} q_{\alpha\delta}q_{\alpha\eta}
\x_{\beta}\x_{\gamma}^{n+1}\x_{\delta}\x_{\eta} \ot x_{\alpha}
\\ & \quad
-q_{\alpha\gamma}q_{\gamma\delta}^{n+1}q_{\gamma\eta}^{n+1}
\coeff{\alpha+\beta,\delta,\alpha,\gamma}{n+1} \Bsj_1\Bsj_2\Bsj_3 
s \big(\x_{\gamma}^{n+3} \ot x_{\gamma}\big),
\end{align*}
and the inductive step follows.
\epf

\begin{lema}\label{lem:diff-case10}
Let $\alpha < \beta < \delta  < \gamma < \mu < \nu < \eta $ be positive roots
such that $N_{\gamma}=2$ and the relations among the corresponding root vectors take the form
\begin{align}\label{eq:diff-case10-hypothesis}
\begin{aligned}
x_{\beta}x_{\eta} &= q_{\beta\eta} x_{\eta}x_{\beta} +\Bsj_1 x_{\nu}x_{\gamma},
\\
x_{\delta}x_{\nu} &= q_{\delta\nu} x_{\nu}x_{\delta} +\Bsj_2 x_{\mu}x_{\gamma}, 
&
x_{\alpha}x_{\mu} &= q_{\alpha\mu}  x_{\mu}x_{\alpha} + \Bsj_3 x_{\gamma}, &
\end{aligned}
\end{align}
for some scalars $\Bsj_i$ and all other pairs of root vectors $q$-commute except possibly $(x_{\beta},x_{\mu})$.
Then, for all $n\geq 0$, 
\begin{align}\label{eq:diff-case10-formula}
\begin{aligned}
d(&\x_{\alpha} \x_{\beta}\x_{\delta}\x_{\gamma}^{n} \x_{\eta}\ot 1)  = 
\x_{\alpha}\x_{\beta}\x_{\delta}\x_{\gamma}^{n} \ot x_{\eta}
-q_{\gamma\eta} \x_{\alpha} \x_{\beta}\x_{\delta} \x_{\gamma}^{n-1} \x_{\eta} \ot x_{\gamma}
\\ & \quad
-(-q_{\delta\gamma})^{n}q_{\delta\eta} \x_{\alpha}\x_{\beta}\x_{\gamma}^{n}\x_{\eta} \ot x_{\delta}
+q_{\beta\delta}(-q_{\beta\gamma})^{n}q_{\beta\eta} \x_{\alpha}\x_{\delta} \x_{\gamma}^{n}\x_{\eta} \ot x_{\beta}
\\ & \quad
+q_{\beta\delta} (-q_{\beta\gamma})^n \Bsj_1 \x_{\alpha}\x_{\delta} \x_{\gamma}^{n} \x_{\nu} \ot x_{\gamma}
-q_{\alpha\beta}q_{\alpha\delta}(-q_{\alpha\gamma})^{n} q_{\alpha\eta} \x_{\beta}\x_{\delta} \x_{\gamma}^{n}\x_{\eta} \ot x_{\alpha}
\\ & \quad
+q_{\beta\delta} q_{\gamma\eta}^{n} \coeff{\beta \, -\nu \, \alpha\gamma}{n} \Bsj_1\Bsj_2 \Bsj_3 \x_{\gamma}^{n+3} \ot 1.
\end{aligned}
\end{align}
\end{lema} 

\bigbreak
Notice that the equalities in \eqref{eq:diff-case10-hypothesis} force
\begin{align}\label{eq:diff-case10-hypothesis-roots}
\begin{aligned}
\beta+\eta &= \gamma+\nu, &
\delta+\nu &= \gamma+\mu, &
\alpha+\mu &= \gamma.
\end{aligned}
\end{align}
Hence the following equality also holds: $3\gamma=\alpha+\beta+\delta+\eta$.

\pf
We need some auxiliary computations. By \eqref{eq:diff-case8-auxiliar3},
\begin{align}\label{eq:diff-case10-auxiliar1}
&\begin{aligned}
d(\x_{\beta} & \x_{\gamma}^{n} \x_{\eta}\ot 1) =
\x_{\beta}\x_{\gamma}^{n}\ot x_{\eta} 
-q_{\gamma \eta} \x_{\beta}\x_{\gamma}^{n-1}\x_{\eta}\ot x_{\gamma}
-(-q_{\beta\gamma})^{n} q_{\beta \eta} \x_{\gamma}^{n}\x_{\eta} \ot x_{\beta}
\\ & -(-q_{\beta\gamma})^{n} \Bsj_1 \x_{\gamma}^{n}\x_{\nu}\ot x_{\gamma},
\end{aligned}
\\ \label{eq:diff-case10-auxiliar2}
&\begin{aligned}
d(\x_{\delta} & \x_{\gamma}^{n} \x_{\nu}\ot 1) =
\x_{\delta}\x_{\gamma}^{n}\ot x_{\nu} 
-q_{\gamma \nu} \x_{\delta}\x_{\gamma}^{n-1}\x_{\nu}\ot x_{\gamma}
-(-q_{\delta\gamma})^{n} q_{\delta \nu} \x_{\gamma}^{n}\x_{\nu} \ot x_{\delta}
\\ & -(-q_{\delta\gamma})^{n} \Bsj_2 \x_{\gamma}^{n}\x_{\mu}\ot x_{\gamma}.
\end{aligned}
\end{align}

Now we prove the following equality by induction on $n$:
\begin{align}\label{eq:diff-case10-auxiliar3}
&\begin{aligned}
d(& \x_{\alpha}\x_{\beta} \x_{\gamma}^{n} \x_{\eta}\ot 1)  =
\x_{\alpha}\x_{\beta}\x_{\gamma}^{n} \ot x_{\eta} 
-q_{\gamma\eta} \x_{\alpha}\x_{\beta}\x_{\gamma}^{n-1} \x_{\eta} \ot x_{\gamma}
-(-q_{\beta\gamma})^{n} \Bsj_1 \x_{\alpha}\x_{\gamma}^{n}\x_{\nu} \ot x_{\gamma}
\\ & \quad
-(-q_{\beta\gamma})^{n}q_{\beta\eta} \x_{\alpha} \x_{\gamma}^{n} \x_{\eta} \ot x_{\beta}
+q_{\alpha\beta}(-q_{\alpha\gamma})^{n}q_{\alpha\eta} \x_{\beta}\x_{\gamma}^{n}\x_{\eta} \ot x_{\alpha}.
\end{aligned}
\end{align}
Indeed for $n=0$ we have:
\begin{align*}
d(& \x_{\alpha}\x_{\beta}\x_{\eta}\ot 1)  =
\x_{\alpha}\x_{\beta} \ot x_{\eta} -s\big( 
q_{\beta\eta} \x_{\alpha} \ot x_{\eta}x_{\beta}
+\Bsj_1 \x_{\alpha} \ot x_{\nu}x_{\gamma}
-q_{\alpha\beta}q_{\alpha\beta} \x_{\beta} \ot x_{\eta}x_{\alpha}
\big)
\\ & =
\x_{\alpha}\x_{\beta} \ot x_{\eta} 
-\Bsj_1 \x_{\alpha}\x_{\nu} \ot x_{\gamma}
-q_{\beta\eta} \x_{\alpha}\x_{\eta} \ot x_{\beta}
+q_{\alpha\beta}q_{\alpha\eta} \x_{\beta}\x_{\eta} \ot x_{\alpha}.
\end{align*}
Now assume that \eqref{eq:diff-case10-auxiliar3} holds for $n$. By Remark \ref{rem:differential-q-commute}, inductive hypothesis and \eqref{eq:diff-case10-auxiliar1},
\begin{align*}
d(& \x_{\alpha}\x_{\beta} \x_{\gamma}^{n+1} \x_{\eta}\ot 1)  =
\x_{\alpha}\x_{\beta} \x_{\gamma}^{n+1} \ot x_{\eta}
-s \big( 
q_{\gamma\eta} \x_{\alpha}\x_{\beta}\x_{\gamma}^{n} \ot x_{\eta}x_{\gamma}
\\ & \quad
+(-q_{\beta\gamma})^{n+1} \x_{\alpha} \x_{\gamma}^{n+1} \ot (q_{\beta\eta} x_{\eta}x_{\beta} +\Bsj_1 x_{\nu}x_{\gamma})
-q_{\alpha\beta}(-q_{\alpha\gamma})^{n+1}q_{\alpha\eta} \x_{\beta} \x_{\gamma}^{n+1} \ot x_{\eta}x_{\alpha}
\big)
\\ & =
\x_{\alpha}\x_{\beta} \x_{\gamma}^{n+1} \ot x_{\eta}
-q_{\gamma\eta} \x_{\alpha}\x_{\beta}\x_{\gamma}^{n}\x_{\eta} \ot x_{\gamma}
-(-q_{\beta\gamma})^{n+1}\Bsj_1 \x_{\alpha}\x_{\gamma}^{n+1} \x_{\nu} \ot x_{\gamma}
\\ & \quad
-(-q_{\beta\gamma})^{n+1}q_{\beta\eta} \x_{\alpha} \x_{\gamma}^{n+1}\x_{\eta} \ot x_{\beta}
+q_{\alpha\beta}(-q_{\alpha\gamma})^{n+1}q_{\alpha\eta} \x_{\beta} \x_{\gamma}^{n+1}\x_{\eta} \ot x_{\alpha}.
\end{align*}

Next we apply Lemma \ref{lem:diff-case7}
to $\alpha < \delta <\gamma < \mu <\nu$ to get: \begin{align}\label{eq:diff-case10-auxiliar4}
&\begin{aligned}
d(&\x_{\alpha}\x_{\delta} \x_{\gamma}^{n} \x_{\nu}\ot 1)  = \x_{\alpha}\x_{\delta} \x_{\gamma}^{n} \ot x_{\nu}
-q_{\gamma\nu} \x_{\alpha}\x_{\delta}\x_{\gamma}^{n-1} \x_{\nu} \ot x_{\gamma}
\\ &
-(-q_{\delta\gamma})^n q_{\delta\nu} \x_{\alpha}\x_{\gamma}^{n} \x_{\nu}\ot x_{\delta}
+(-q_{\alpha\gamma})^nq_{\alpha\delta}q_{\alpha\nu} \x_{\delta} \x_{\gamma}^{n} \x_{\nu}\ot x_{\alpha}
\\ & 
- (-q_{\delta\gamma})^{n} \Bsj_2 \x_{\alpha} \x_{\gamma}^{n} \x_{\mu}\ot x_{\gamma}
-q_{\gamma\nu}^{n} \coef{-\nu,\alpha,\gamma}{n} \Bsj_2 \Bsj_3 \x_{\gamma}^{n+2} \ot 1.
\end{aligned}
\end{align}

Now we prove by induction on $n$ that
\begin{align}\label{eq:diff-case10-auxiliar5}
&\begin{aligned}
d(&\x_{\beta}\x_{\delta} \x_{\gamma}^{n} \x_{\eta}\ot 1)  = \x_{\beta}\x_{\delta} \x_{\gamma}^{n} \ot x_{\eta}
-q_{\gamma\eta} \x_{\beta}\x_{\delta}\x_{\gamma}^{n-1} \x_{\eta} \ot x_{\gamma}
-(-q_{\delta\gamma})^n q_{\delta\eta} \x_{\beta}\x_{\gamma}^{n} \x_{\eta}\ot x_{\delta}
\\ &
+(-q_{\beta\gamma})^nq_{\beta\delta}q_{\beta\eta} \x_{\delta} \x_{\gamma}^{n} \x_{\eta}\ot x_{\beta}
+q_{\beta\delta}(-q_{\beta\gamma})^{n} \Bsj_1 \x_{\delta}\x_{\gamma}^{n}\x_{\nu} \ot x_{\gamma}.
\end{aligned}
\end{align}
Indeed, for $n=0$ we have
\begin{align*}
d(&\x_{\beta}\x_{\delta}\x_{\eta} \ot 1)  = 
\x_{\beta}\x_{\delta} \ot x_{\eta} -s \big( 
q_{\delta\eta} \x_{\beta} \ot x_{\eta}x_{\delta} 
-q_{\beta\delta}q_{\beta\eta} \x_{\delta} \ot  x_{\eta}x_{\beta}
-q_{\beta\delta}\Bsj_1 \x_{\delta} \ot x_{\nu}x_{\gamma}
\big)
\\ & = 
\x_{\beta}\x_{\delta} \ot x_{\eta}
-q_{\delta\eta} \x_{\beta}\x_{\eta} \ot x_{\delta} 
+q_{\beta\delta}q_{\beta\eta} \x_{\delta}\x_{\eta} \ot x_{\beta}
+q_{\beta\delta}\Bsj_1 \x_{\delta}\x_{\nu} \ot x_{\gamma}.
\end{align*}
Now we assume that \eqref{eq:diff-case10-auxiliar5} holds for $n$. Using Remark \ref{rem:differential-q-commute}, inductive hypothesis, \eqref{eq:diff-case10-auxiliar1}, \eqref{eq:diff-case10-auxiliar2},
\begin{align*}
d(&\x_{\beta}\x_{\delta}\x_{\gamma}^{n+1}\x_{\eta} \ot 1)  = 
\x_{\beta}\x_{\delta}\x_{\gamma}^{n+1} \ot x_{\eta}
-s \big( 
q_{\gamma\eta} \x_{\beta}\x_{\delta}\x_{\gamma}^{n} \ot x_{\eta}x_{\gamma}
+(-q_{\delta\gamma})^{n+1}q_{\delta\eta} \x_{\beta} \x_{\gamma}^{n+1} \ot x_{\eta}x_{\delta}
\\ & \quad
-q_{\beta\delta}(-q_{\beta\gamma})^{n+1}q_{\beta\eta} \x_{\delta} \x_{\gamma}^{n+1} \ot x_{\eta}x_{\beta}
-q_{\beta\delta}(-q_{\beta\gamma})^{n+1}\Bsj_1 \x_{\delta} \x_{\gamma}^{n+1} \ot x_{\nu}x_{\gamma}
\big)
\\ & = 
\x_{\beta}\x_{\delta}\x_{\gamma}^{n+1} \ot x_{\eta}
-q_{\gamma\eta} \x_{\beta}\x_{\delta}\x_{\gamma}^{n}\x_{\eta} \ot x_{\gamma}
-(-q_{\delta\gamma})^{n+1}q_{\delta\eta} \x_{\beta} \x_{\gamma}^{n+1}\x_{\eta} \ot x_{\delta}
\\ & \quad
+q_{\beta\delta}(-q_{\beta\gamma})^{n+1}\Bsj_1 \x_{\delta} \x_{\gamma}^{n+1}\x_{\nu} \ot x_{\gamma}
+q_{\beta\delta}(-q_{\beta\gamma})^{n+1}q_{\beta\eta} \x_{\delta}\x_{\gamma}^{n+1}\x_{\eta} \ot x_{\beta}
\\ & \quad
- q_{\beta\gamma}^{n+1}q_{\delta\gamma}^{n+1}q_{\delta\eta} (1-\widetilde{q}_{\beta\delta}) \Bsj_1  s \big( 
\x_{\gamma}^{n+1}\x_{\nu}\ot x_{\gamma}x_{\delta}\big),
\end{align*}
and the inductive step follows since 
$\x_{\gamma}^{n+1}\x_{\nu}\ot x_{\gamma}x_{\delta} = s(\x_{\gamma}^{n+1}\ot x_{\nu}x_{\gamma}x_{\delta})$.

\medskip

Finally we prove \eqref{eq:diff-case10-formula} by induction on $n$. When $n=0$,
\begin{align*}
d( & \x_{\alpha} \x_{\beta} \x_{\delta} \x_{\eta} \ot 1) = 
\x_{\alpha}\x_{\beta}\x_{\delta} \ot x_{\eta}
-s \big(q_{\delta\eta} \x_{\alpha}\x_{\beta} \ot x_{\eta}x_{\delta}
-q_{\beta\delta}q_{\beta\eta} \x_{\alpha}\x_{\delta} \ot  x_{\eta}x_{\beta}
\\ & \quad
-q_{\beta\delta}\Bsj_1 \x_{\alpha}\x_{\delta} \ot x_{\nu}x_{\gamma}
+q_{\alpha\beta}q_{\alpha\delta}q_{\alpha\eta} \x_{\beta}\x_{\delta} \ot x_{\eta}x_{\alpha}
\big)
\\ &  = 
\x_{\alpha}\x_{\beta}\x_{\delta} \ot x_{\eta}
-q_{\delta\eta} \x_{\alpha}\x_{\beta}\x_{\eta} \ot x_{\delta}
+q_{\beta\delta}q_{\beta\eta} \x_{\alpha}\x_{\delta}\x_{\eta} \ot x_{\beta}
+q_{\beta\delta}\Bsj_1 \x_{\alpha}\x_{\delta}\x_{\nu} \ot x_{\gamma}
\\ & \quad
-q_{\alpha\beta}q_{\alpha\delta}q_{\alpha\eta} \x_{\beta}\x_{\delta}\x_{\eta} \ot x_{\alpha}
+q_{\beta\delta}\Bsj_1\Bsj_2 \Bsj_3 \x_{\gamma}^{3} \ot 1.
\end{align*}
Now assume that \eqref{eq:diff-case10-formula} holds for $n$. Using Remark \ref{rem:differential-q-commute}, inductive hypothesis, \eqref{eq:diff-case10-auxiliar3}, \eqref{eq:diff-case10-auxiliar4} and \eqref{eq:diff-case10-auxiliar5}
\begin{align*}
d(&\x_{\alpha}\x_{\beta}\x_{\delta}\x_{\gamma}^{n+1} \x_{\eta}\ot 1)  =
\x_{\alpha}\x_{\beta}\x_{\delta}\x_{\gamma}^{n+1} \ot x_{\eta}
-s \big( 
q_{\gamma\eta} \x_{\alpha}\x_{\beta}\x_{\delta}\x_{\gamma}^{n} \ot x_{\eta}x_{\gamma}
\\ & \quad 
+(-q_{\delta\gamma})^{n+1}q_{\delta\eta} \x_{\alpha}\x_{\beta} \x_{\gamma}^{n+1} \ot x_{\eta}x_{\delta} 
-q_{\beta\delta}(-q_{\beta\gamma})^{n+1}q_{\beta\eta} \x_{\alpha} \x_{\delta}\x_{\gamma}^{n+1} \ot  x_{\eta}x_{\beta} 
\\ & \quad 
-q_{\beta\delta}(-q_{\beta\gamma})^{n+1}\Bsj_1 \x_{\alpha} \x_{\delta}\x_{\gamma}^{n+1} \ot x_{\nu}x_{\gamma}
+q_{\alpha\beta}q_{\alpha\delta}(-q_{\alpha\gamma})^{n+1} q_{\alpha\eta} \x_{\beta} \x_{\delta} \x_{\gamma}^{n+1} \ot x_{\eta}x_{\alpha} 
\big)
\\ & =
\x_{\alpha}\x_{\beta}\x_{\delta}\x_{\gamma}^{n+1} \ot x_{\eta}
-q_{\gamma\eta} \x_{\alpha}\x_{\beta}\x_{\delta} \x_{\gamma}^{n}\x_{\eta} \ot x_{\gamma}
-(-q_{\delta\gamma})^{n+1}q_{\delta\eta} \x_{\alpha}\x_{\beta} \x_{\gamma}^{n+1}\x_{\eta} \ot x_{\delta} 
\\ & \quad 
+q_{\beta\delta}(-q_{\beta\gamma})^{n+1}\Bsj_1 \x_{\alpha} \x_{\delta}\x_{\gamma}^{n+1}\x_{\nu} \ot x_{\gamma}
+q_{\beta\delta}(-q_{\beta\gamma})^{n+1}q_{\beta\eta} \x_{\alpha} \x_{\delta}\x_{\gamma}^{n+1}\x_{\eta} \ot  x_{\beta} 
\\ & \quad 
-s \big( 
q_{\alpha\beta}q_{\alpha\delta}(-q_{\alpha\gamma})^{n+1} q_{\alpha\eta} \x_{\beta} \x_{\delta} \x_{\gamma}^{n+1} \ot x_{\eta}x_{\alpha}
-q_{\alpha\beta}q_{\alpha\delta}q_{\alpha\gamma}^{n+1} q_{\alpha\eta}q_{\delta\gamma}^{n+1}q_{\delta\eta} \x_{\beta}\x_{\gamma}^{n+1}\x_{\eta} \ot x_{\delta}x_{\alpha}
\\ & \quad 
-q_{\alpha\beta}q_{\alpha\delta}(-q_{\alpha\gamma})^{n+1} q_{\alpha\eta}q_{\gamma\eta} \x_{\beta}\x_{\delta} \x_{\gamma}^{n}\x_{\eta} \ot x_{\gamma}x_{\alpha}
+q_{\beta\delta}q_{\beta\gamma}^{n+1} 
q_{\alpha\gamma}^{n+2} q_{\alpha\delta}q_{\alpha\nu} \Bsj_1 \x_{\delta} \x_{\gamma}^{n+1} \x_{\nu}\ot x_{\gamma}x_{\alpha}
\\ & \quad
-q_{\beta\delta}(q_{\gamma\eta}^{n+1} \coeff{\beta \, -\nu \, \alpha\gamma}{n} +q_{\gamma\nu}^{n+1} \coef{-\nu,\alpha,\gamma}{n+1}(-q_{\beta\gamma})^{n+1}) \Bsj_1 \Bsj_2 \Bsj_3 \x_{\gamma}^{n+3} \ot x_{\gamma}
\\ & \quad
+q_{\beta\delta}q_{\beta\gamma}^{n+1}q_{\beta\eta} 
q_{\alpha\beta}q_{\alpha\delta}q_{\alpha\gamma}^{n+1} q_{\alpha\eta} \x_{\delta}\x_{\gamma}^{n+1}\x_{\eta} \ot x_{\beta}x_{\alpha}
\big)
\\ & =
\x_{\alpha}\x_{\beta}\x_{\delta}\x_{\gamma}^{n+1} \ot x_{\eta}
-q_{\gamma\eta} \x_{\alpha}\x_{\beta}\x_{\delta} \x_{\gamma}^{n}\x_{\eta} \ot x_{\gamma}
-(-q_{\delta\gamma})^{n+1}q_{\delta\eta} \x_{\alpha}\x_{\beta} \x_{\gamma}^{n+1}\x_{\eta} \ot x_{\delta} 
\\ & \quad 
+q_{\beta\delta}(-q_{\beta\gamma})^{n+1}\Bsj_1 \x_{\alpha} \x_{\delta}\x_{\gamma}^{n+1}\x_{\nu} \ot x_{\gamma}
+q_{\beta\delta}(-q_{\beta\gamma})^{n+1}q_{\beta\eta} \x_{\alpha} \x_{\delta}\x_{\gamma}^{n+1}\x_{\eta} \ot  x_{\beta} 
\\ & \quad 
-q_{\alpha\beta}q_{\alpha\delta}(-q_{\alpha\gamma})^{n+1} q_{\alpha\eta} \x_{\beta} \x_{\delta} \x_{\gamma}^{n+1} \x_{\eta} \ot x_{\alpha}
+q_{\beta\delta} q_{\gamma\eta}^{n+1} \coeff{\beta \, -\nu \, \alpha\gamma}{n+1} \Bsj_1 \Bsj_2 \Bsj_3 \x_{\gamma}^{n+4} \ot 1,
\end{align*}
and the inductive step follows.
\epf

\begin{lema}\label{lem:diff-case11}
Let $\alpha < \beta < \tau < \delta < \mu < \nu <\gamma < \kappa < \iota < \eta$ be positive roots
such that $N_{\gamma}=N_{\eta}=N_{\kappa}=2$ and the relations among the corresponding root vectors take the form
\begin{align}\label{eq:diff-case11-hypothesis}
\begin{aligned}
x_{\alpha}x_{\delta} &= q_{\alpha\delta} x_{\delta}x_{\alpha} 
+\Bsj_1 x_{\tau}^2x_{\beta}, &
x_{\delta}x_{\eta} &= q_{\delta\eta} x_{\eta}x_{\delta} + \Bsj_2 x_{\kappa}x_{\gamma}x_{\mu} + \Bsj_3 x_{\gamma} x_{\nu},
\\
x_{\tau}x_{\eta} &= q_{\tau\eta} x_{\eta}x_{\tau} +\Bsj_4 x_{\kappa}x_{\gamma}, &
x_{\alpha}x_{\mu} &= q_{\alpha\mu} x_{\mu}x_{\alpha} + \Bsj_5 x_{\tau}x_{\beta},
\\
x_{\mu}x_{\eta} &= q_{\mu\eta} x_{\eta}x_{\mu} +\Bsj_6 x_{\iota}x_{\gamma}, &
x_{\nu}x_{\eta} &= q_{\nu\eta} x_{\eta}x_{\nu} +\Bsj_7 x_{\iota}x_{\kappa}x_{\gamma},
\\
x_{\beta}x_{\kappa} &= q_{\beta\kappa} x_{\kappa}x_{\beta} +\Bsj_8 x_{\gamma}, &
x_{\alpha}x_{\iota} &= q_{\alpha\iota} x_{\iota}x_{\alpha} +\Bsj_9 x_{\gamma},
\end{aligned}
\end{align}
for some scalars $\Bsj_j\in\Bbbk$ and the other pairs of root vectors $q$-commute, except possibly $(x_{\tau},x_{\mu})$, $(x_{\mu},x_{\kappa})$, $(x_{\delta},x_{\kappa})$.
Then
\begin{align}\label{eq:diff-case11-formula}
&\begin{aligned}
d(& \x_{\alpha}\x_{\beta}\x_{\delta}\x_{\gamma}^{2} \x_{\eta}^2 \ot 1)  = \x_{\alpha} \x_{\beta} \x_{\delta} \x_{\gamma}^{2} \x_{\eta} \ot x_{\eta} 
+q_{\gamma\eta}^2 \x_{\alpha}\x_{\beta}\x_{\delta}\x_{\gamma}\x_{\eta}^2 \ot x_{\gamma}
\\ &
+q_{\delta\gamma}^2(3)_{\widetilde{q}_{\gamma\eta}} \Bsj_3\Bsj_7 \x_{\alpha}\x_{\beta}\x_{\gamma}^3\x_{\iota} \ot x_{\kappa}x_{\gamma}
+q_{\delta\gamma}^2q_{\nu\eta} (3)_{\widetilde{q}_{\gamma\eta}} \Bsj_3 \x_{\alpha}\x_{\beta}\x_{\gamma}^3\x_{\eta} \ot x_{\nu}
\\ &
+q_{\delta\gamma}^2q_{\mu\eta}q_{\gamma\eta} \Bsj_2 \x_{\alpha}\x_{\beta}\x_{\gamma}^2\x_{\kappa}\x_{\eta} \ot x_{\gamma}x_{\mu}
+q_{\delta\gamma}^2q_{\delta\eta}^2 \x_{\alpha}\x_{\beta}\x_{\gamma}^2\x_{\eta}^2 \ot x_{\delta}
\\ &
+q_{\delta\gamma}^2 q_{\gamma\kappa}^2 \coef{\beta\eta\gamma}{3} \Bsj_2\Bsj_6\Bsj_8 \x_{\alpha}\x_{\gamma}^4\x_{\iota} \ot x_{\gamma}
+q_{\delta\gamma}^2 q_{\gamma\kappa}^2q_{\mu\eta} \coef{\beta\eta\gamma}{3} \Bsj_2\Bsj_8 \x_{\alpha}\x_{\gamma}^4\x_{\eta} \ot x_{\mu}
\\ &
-q_{\beta\delta}q_{\beta\gamma}^2q_{\beta\eta}^2 \x_{\alpha} \x_{\delta}\x_{\gamma}^2\x_{\eta}^2 \ot x_{\beta}
+q_{\alpha\beta}q_{\beta\gamma}^2q_{\beta\eta}^2q_{\tau\gamma}^2q_{\gamma\eta} \Bsj_1\Bsj_4 \x_{\beta}\x_{\tau}\x_{\gamma}^2\x_{\kappa}\x_{\eta} \ot x_{\gamma}x_{\beta}
\\ &
-q_{\alpha\beta}q_{\beta\gamma}^2q_{\beta\eta}^2q_{\tau\gamma}^2q_{\tau\eta}^2 \Bsj_1 \x_{\beta}\x_{\tau}\x_{\gamma}^2\x_{\eta}^2 \ot x_{\tau}x_{\beta}
+q_{\alpha\beta}q_{\alpha\delta}q_{\alpha\gamma}^2q_{\alpha\eta}^2 \x_{\beta}\x_{\delta}\x_{\gamma}^2\x_{\eta}^2 \ot x_{\alpha}
\\ &
-q_{\alpha\beta}q_{\beta\tau}q_{\beta\gamma}^2q_{\beta\eta}^2q_{\gamma\beta}^{-2}q_{\tau\gamma}^2\coef{\beta\eta\gamma}{3} \Bsj_1\Bsj_4\Bsj_8 \x_{\tau}\x_{\gamma}^4\x_{\eta} \ot x_{\beta} 
-q_{\delta\gamma}^2 q_{\gamma\kappa}^2 \coeff{\alpha\beta\eta\gamma}{4} \Bsj_2\Bsj_6\Bsj_8\Bsj_9 \x_{\gamma}^6 \ot 1.
\end{aligned}
\end{align}
\end{lema} 

Notice that the equalities in \eqref{eq:diff-case3-hypothesis} forces
\begin{align}\label{eq:diff-case11-hypothesis-roots}
\begin{aligned}
\alpha + \delta &=\beta+ 2\tau, &
\delta+\eta &= \kappa+\gamma+\mu, &
\delta+\eta &= \gamma+\nu,
\\
\tau+\eta &= \kappa+\gamma, &
\alpha+\mu &= \tau+\beta, &
\mu+\eta &= \iota+\gamma,
\\
\nu+\eta &= \iota+\kappa+\gamma, &
\beta+\kappa &= \gamma, &
\alpha+\iota &= \gamma.
\end{aligned}
\end{align}
Thus the following equality also holds: 
$4\gamma=\alpha+\beta+\delta+2\eta$.

\pf
First we note that Lemma \ref{lem:diff2} applies for $\alpha<\gamma<\iota$, and Lemma \ref{lem:diff-case7} applies for
$\beta<\tau<\gamma<\kappa<\eta$. Hence the following formulas hold for all $n\ge0$:
\begin{align}\label{eq:diff-case11-agggi}
&\begin{aligned}
d(\x_{\alpha} & \x_{\gamma}^{n}\x_{\iota}\ot 1)  =
\x_{\alpha}\x_{\gamma}^{n} \ot x_{\iota} 
- q_{\gamma\iota} \x_{\alpha}\x_{\gamma}^{n-1}\x_{\iota}\ot x_{\gamma} 
- q_{\alpha\iota}(-q_{\alpha\gamma})^{n} \x_{\gamma}^{n}\x_{\iota} \ot x_{\alpha}  
\\
& 
- (-q_{\alpha\gamma})^{n} (n+1)_{\widetilde{q}_{\alpha\gamma}} \Bsj_9
\x_{\gamma}^{n+1} \ot 1,
\end{aligned}
\\\label{eq:diff-case11-tggge}
&\begin{aligned}
d(\x_{\tau} & \x_{\gamma}^{n} \x_{\eta}\ot 1) =
\x_{\tau}\x_{\gamma}^{n}\ot x_{\eta} 
-q_{\gamma\eta} \x_{\tau}\x_{\gamma}^{n-1}\x_{\eta}\ot x_{\gamma}
-(-q_{\tau\gamma})^{n} q_{\tau \eta} \x_{\gamma}^{n}\x_{\eta} \ot x_{\tau}
\\ & -(-q_{\tau\gamma})^{n} \Bsj_4 \x_{\gamma}^{n}\x_{\kappa}\ot x_{\gamma},
\end{aligned}
\\ 
\label{eq:diff-case11-bgggk}
&\begin{aligned}
d(\x_{\beta} & \x_{\gamma}^{n} \x_{\kappa}\ot 1)  =
\x_{\beta}\x_{\gamma}^{n} \ot x_{\kappa} 
- q_{\gamma\kappa} \x_{\beta} \x_{\gamma}^{n-1}\x_{\kappa}\ot x_{\gamma} 
- q_{\beta\kappa}(-q_{\beta\gamma})^{n} \x_{\gamma}^{n} \x_{\kappa} \ot x_{\beta}  
\\
& 
- (-q_{\beta\gamma})^{n} (n+1)_{\widetilde{q}_{\beta\gamma}} \Bsj_8
\x_{\gamma}^{n+1} \ot 1,
\end{aligned}
\\ \label{eq:diff-case11-btggge}
&\begin{aligned}
d(\x_{\beta}& \x_{\tau} \x_{\gamma}^{n} \x_{\eta}\ot 1)  = \x_{\beta}\x_{\tau} \x_{\gamma}^{n} \ot x_{\eta}
-q_{\gamma\eta} \x_{\beta}\x_{\tau}\x_{\gamma}^{n-1} \x_{\eta} \ot x_{\gamma}
\\ &
-(-q_{\tau\gamma})^n q_{\tau\eta} \x_{\beta}\x_{\gamma}^{n} \x_{\eta}\ot x_{\tau}
+(-q_{\beta\gamma})^nq_{\beta\tau}q_{\beta\eta} \x_{\tau} \x_{\gamma}^{n} \x_{\eta}\ot x_{\beta}
\\ & 
- (-q_{\tau\gamma})^{n} \Bsj_4 \x_{\beta} \x_{\gamma}^{n} \x_{\kappa}\ot x_{\gamma}
-q_{\gamma\eta}^{n} \coef{-\eta,\beta,\gamma}{n} \Bsj_4 \Bsj_8 \x_{\gamma}^{n+2} \ot 1.
\end{aligned}
\end{align}

We also need some auxiliar computations. These are straightforward and we omit the details:
\begin{align}
\label{eq:diff-case11-abd}&
\begin{aligned}
d(\x_{\alpha} & \x_{\beta} \x_{\delta}
\ot 1)  = \x_{\alpha} \x_{\beta}\ot x_{\delta} -q_{\beta\delta} \x_{\alpha} \x_{\delta}\ot x_{\beta}
+q_{\alpha\beta}\Bsj_1 \x_{\beta}\x_{\tau} \ot x_{\tau}x_{\beta}
\\ &
+q_{\alpha\beta}q_{\alpha\delta} \x_{\beta}\x_{\delta} \ot x_{\alpha},
\end{aligned}
\\
\label{eq:diff-case11-abk}&
\begin{aligned}
d(\x_{\alpha} & \x_{\beta} \x_{\kappa}
\ot 1)  = \x_{\alpha}\x_{\beta} \ot x_{\kappa}
-\Bsj_8 \x_{\alpha}\x_{\gamma} \ot 1
-q_{\beta\kappa} \x_{\alpha}\x_{\kappa} \ot x_{\beta}
\\ &
+q_{\alpha\beta}q_{\alpha\kappa} \x_{\beta}\x_{\kappa} \ot x_{\alpha},
\end{aligned}
\\
\label{eq:diff-case11-adg}&
\begin{aligned}
d(\x_{\alpha} & \x_{\delta}\x_{\gamma} \ot 1)  
= \x_{\alpha}\x_{\delta} \ot x_{\gamma}
-q_{\delta\gamma} \x_{\alpha}\x_{\gamma} \ot x_{\delta}
+q_{\alpha\delta}q_{\alpha\gamma} \x_{\delta}\x_{\gamma} \ot x_{\alpha}
\\ &
+ q_{\tau\gamma} q_{\beta\gamma}\Bsj_1 \x_{\tau}\x_{\gamma} \ot x_{\tau}x_{\beta},
\end{aligned}
\\
\label{eq:diff-case11-ade}&
\begin{aligned}
d(\x_{\alpha} & \x_{\delta}\x_{\eta} \ot 1)  
= \x_{\alpha}\x_{\delta}\ot x_{\eta}
-q_{\delta\eta} \x_{\alpha}\x_{\eta} \ot x_{\delta}
-\Bsj_2 \x_{\alpha}\x_{\kappa} \ot x_{\gamma}x_{\mu}
-\Bsj_3 \x_{\alpha}\x_{\gamma} \ot  x_{\nu}
\\ & \quad 
+q_{\alpha\delta}q_{\alpha\eta} \x_{\delta}\x_{\eta} \ot x_{\alpha}
+ q_{\beta\eta}\Bsj_1\Bsj_4 \x_{\tau}\x_{\kappa} \ot x_{\gamma}x_{\beta}
+ q_{\beta\eta}q_{\tau\eta}\Bsj_1 \x_{\tau}\x_{\eta} \ot  x_{\tau}x_{\beta},
\end{aligned}
\\
\label{eq:diff-case11-btk}&
\begin{aligned}
d(\x_{\beta} & \x_{\tau}\x_{\kappa} \ot 1)  
= \x_{\beta}\x_{\tau}\ot x_{\kappa}
-q_{\tau\kappa} \x_{\beta}\x_{\kappa} \ot x_{\tau}
+q_{\beta\tau}\Bsj_8 \x_{\tau}\x_{\gamma} \ot 1
\\ & \quad
+q_{\beta\tau}q_{\beta\kappa} \x_{\tau}\x_{\kappa} \ot x_{\beta},
\end{aligned}
\\
\label{eq:diff-case11-bde}&
\begin{aligned}
d(\x_{\beta} & \x_{\delta} \x_{\eta}
\ot 1)  = \x_{\beta}\x_{\delta} \ot x_{\eta}
-q_{\delta\eta} \x_{\beta}\x_{\eta} \ot x_{\delta}
-\Bsj_2 \x_{\beta}\x_{\kappa} \ot x_{\gamma}x_{\mu}
-\Bsj_3 \x_{\beta}\x_{\gamma} \ot x_{\nu}
\\ & \quad
+q_{\beta\delta}q_{\beta\eta} \x_{\delta}\x_{\eta} \ot x_{\beta} 
-\Bsj_2\Bsj_8 \x_{\gamma}^2 \ot x_{\mu},
\end{aligned}
\\
\label{eq:diff-case11-bkk}&
\begin{aligned}
d(\x_{\beta} & \x_{\kappa}^2\ot 1) = \x_{\beta}\x_{\kappa} \ot x_{\kappa}
+\Bsj_8 \x_{\gamma}\x_{\kappa} \ot 1 
+q_{\beta\kappa}^2 \x_{\kappa}^2 \ot x_{\beta},
\end{aligned}
\\
\label{eq:diff-case11-bke}&
\begin{aligned}
d(\x_{\beta} & \x_{\kappa}\x_{\eta}\ot 1)  = 
\x_{\beta}\x_{\kappa} \ot x_{\eta}
-q_{\kappa\eta} \x_{\beta}\x_{\eta} \ot x_{\kappa}
+\Bsj_8 \x_{\gamma}\x_{\eta} \ot 1
+q_{\beta\kappa}q_{\beta\eta} \x_{\kappa}\x_{\eta} \ot x_{\beta},
\end{aligned}
\\
\label{eq:diff-case11-tke}&
\begin{aligned}
d(\x_{\tau} & \x_{\kappa} \x_{\eta} \ot 1)  
= \x_{\tau}\x_{\kappa} \ot x_{\eta} 
-q_{\kappa\eta} \x_{\tau}\x_{\eta} \ot x_{\kappa} 
+q_{\tau\kappa}q_{\tau\eta} \x_{\kappa}\x_{\eta} \ot x_{\tau}
\\ & \quad
+q_{\kappa\tau}(2)_{\eta\kappa} \Bsj_4 \x_{\kappa}^2 \ot x_{\gamma},
\end{aligned}
\\
\label{eq:diff-case11-tee}&
\begin{aligned}
d(\x_{\tau} & \x_{\eta}^2\ot 1)  = \x_{\tau}\x_{\eta} \ot x_{\eta}
+q_{\gamma\eta}\Bsj_4 \x_{\kappa}\x_{\eta} \ot x_{\gamma}
+q_{\tau\eta}^2 \x_{\eta}^2 \ot x_{\tau},
\end{aligned}
\\
\label{eq:diff-case11-dge}&
\begin{aligned}
d(\x_{\delta} & \x_{\gamma} \x_{\eta}
\ot 1)  = \x_{\delta}\x_{\gamma} \ot x_{\eta} 
-q_{\gamma\eta} \x_{\delta}\x_{\eta} \ot x_{\gamma}
+q_{\delta\gamma}q_{\delta\eta} \x_{\gamma}\x_{\eta} \ot x_{\delta}
\\ & \quad
+q_{\delta\gamma}\Bsj_2 \x_{\gamma}\x_{\kappa} \ot x_{\gamma}x_{\mu}
+q_{\delta\gamma}(2)_{\widetilde{q}_{\gamma\eta}}\Bsj_3 \x_{\gamma}^2 \ot x_{\nu}.
\end{aligned}
\\
\label{eq:diff-case11-dee}&
\begin{aligned}
d(\x_{\delta} & \x_{\eta}^2\ot 1)  = \x_{\delta}\x_{\eta} \ot x_{\eta}
+\Bsj_3\Bsj_7 \x_{\gamma}\x_{\iota} \ot x_{\kappa}x_{\gamma}
+q_{\nu\eta} \Bsj_3 \x_{\gamma}\x_{\eta} \ot x_{\nu}
\\ & \quad
+q_{\gamma\eta}q_{\mu\eta} \Bsj_2 \x_{\kappa}\x_{\eta} \ot x_{\gamma}x_{\mu}
+q_{\delta\eta}^2 \x_{\eta}^2 \ot x_{\delta}.
\end{aligned}
\end{align}

Next we compute differentials of some $4$-chains, using the previous computations on $3$-chains and Remark \ref{rem:differential-q-commute}:
\begin{align}
\label{eq:diff-case11-abdg}&
\begin{aligned}
d( & \x_{\alpha}\x_{\beta} \x_{\delta} \x_{\gamma}
\ot 1)  = \x_{\alpha}\x_{\beta}\x_{\delta} \ot x_{\gamma}
-q_{\delta\gamma} \x_{\alpha} \x_{\beta}\x_{\gamma} \ot x_{\delta}
+q_{\beta\delta} q_{\beta\gamma} \x_{\alpha} \x_{\delta}\x_{\gamma} \ot x_{\beta} 
\\ & \quad
- q_{\alpha\beta}q_{\beta\gamma} q_{\tau\gamma} \Bsj_1 \x_{\beta}\x_{\tau}\x_{\gamma} \ot x_{\tau}x_{\beta}
-q_{\alpha\beta}q_{\alpha\delta}q_{\alpha\gamma} \x_{\beta}\x_{\delta}\x_{\gamma} \ot x_{\alpha},
\end{aligned}
\\
\label{eq:diff-case11-adgg}&
\begin{aligned}
d( & \x_{\alpha}\x_{\delta} \x_{\gamma}^{2}
\ot 1)  = \x_{\alpha}\x_{\delta} \x_{\gamma} \ot x_{\gamma}
+q_{\delta\gamma}^2 \x_{\alpha}\x_{\gamma}^2 \ot x_{\delta}
-q_{\alpha\delta}q_{\alpha\gamma}^2 \x_{\delta}\x_{\gamma}^2 \ot x_{\alpha}
\\ & 
-q_{\tau\gamma}^2 q_{\beta\gamma}^2 \Bsj_1 \x_{\tau}\x_{\gamma}^2 \ot x_{\tau}x_{\beta},
\end{aligned}
\\
\label{eq:diff-case11-abde}&
\begin{aligned}
d( & \x_{\alpha}\x_{\beta} \x_{\delta} \x_{\eta}
\ot 1)  = \x_{\alpha}\x_{\beta}\x_{\delta} \ot x_{\eta}
-q_{\delta\eta} \x_{\alpha}\x_{\beta}\x_{\eta} \ot x_{\delta}
-\Bsj_3 \x_{\alpha}\x_{\beta}\x_{\gamma} \ot x_{\nu}
\\ & \quad
-\Bsj_2 \x_{\alpha}\x_{\beta}\x_{\kappa} \ot x_{\gamma}x_{\mu}
-\Bsj_2\Bsj_8 \x_{\alpha}\x_{\gamma}^2 \ot x_{\mu}
+q_{\beta\delta}q_{\beta\eta} \x_{\alpha}\x_{\delta}\x_{\eta} \ot x_{\beta}
\\ & \quad
-q_{\alpha\beta}q_{\beta\eta} \Bsj_1\Bsj_4  \x_{\beta}\x_{\tau}\x_{\kappa} \ot x_{\gamma}x_{\beta}
-q_{\alpha\beta}q_{\beta\eta}q_{\tau\eta} \Bsj_1 \x_{\beta}\x_{\tau}\x_{\eta} \ot x_{\tau}x_{\beta}
\\ & \quad
-q_{\alpha\beta}q_{\alpha\delta}q_{\alpha\eta} \x_{\beta}\x_{\delta}\x_{\eta} \ot x_{\alpha}
+q_{\beta\tau}q_{\alpha\beta}q_{\beta\eta} \Bsj_1\Bsj_4 \Bsj_8 \x_{\tau}\x_{\gamma}^2 \ot x_{\beta},
\end{aligned}
\\
\label{eq:diff-case11-abke}&
\begin{aligned}
d( & \x_{\alpha}\x_{\beta}\x_{\kappa}\x_{\eta} \ot 1)  = \x_{\alpha}\x_{\beta}\x_{\kappa} \ot x_{\eta} 
- q_{\kappa\eta} \x_{\alpha}\x_{\beta}\x_{\eta} \ot x_{\kappa} 
+q_{\beta\kappa}q_{\beta\eta} \x_{\alpha}\x_{\kappa}\x_{\eta} \ot x_{\beta}
\\ & \quad
+\Bsj_8 \x_{\alpha}\x_{\gamma}\x_{\eta} \ot 1
-q_{\alpha\beta}q_{\alpha\kappa}q_{\alpha\eta} \x_{\beta}\x_{\kappa}\x_{\eta} \ot x_{\alpha},
\end{aligned}
\\
\label{eq:diff-case11-abgk}&
\begin{aligned}
d( & \x_{\alpha}\x_{\beta}\x_{\gamma}\x_{\kappa} \ot 1)  = \x_{\alpha}\x_{\beta}\x_{\gamma} \ot x_{\kappa} 
-q_{\gamma\kappa} \x_{\alpha}\x_{\beta}\x_{\kappa} \ot x_{\gamma}
+q_{\beta\gamma}q_{\beta\kappa} \x_{\alpha}\x_{\gamma}\x_{\kappa} \ot  x_{\beta}
\\ & \quad
+q_{\gamma\beta}^{-1} (2)_{\widetilde{q}_{\beta\gamma}} \Bsj_8 \x_{\alpha}\x_{\gamma}^2 \ot 1 -q_{\alpha\beta}q_{\alpha\gamma}q_{\alpha\kappa} \x_{\beta}\x_{\gamma}\x_{\kappa} \ot x_{\alpha},
\end{aligned}
\\
\label{eq:diff-case11-adge}&
\begin{aligned}
d( & \x_{\alpha}\x_{\delta}\x_{\gamma}\x_{\eta} \ot 1)  = 
\x_{\alpha}\x_{\delta}\x_{\gamma} \ot x_{\eta}
-q_{\gamma\eta} \x_{\alpha}\x_{\delta}\x_{\eta} \ot x_{\gamma}
+q_{\delta\gamma}q_{\delta\eta} \x_{\alpha}\x_{\gamma} \x_{\eta} \ot  x_{\delta}
\\ & \quad
+q_{\delta\gamma}\Bsj_2 \x_{\alpha}\x_{\gamma}\x_{\kappa} \ot  x_{\gamma}x_{\mu}
-q_{\alpha\delta}q_{\alpha\gamma}q_{\alpha\eta} \x_{\delta}\x_{\gamma}\x_{\eta} \ot x_{\alpha}
\\ & \quad
- q_{\tau\gamma}q_{\beta\gamma}q_{\beta\eta}\Bsj_1\Bsj_4  \x_{\tau}\x_{\gamma}\x_{\kappa} \ot x_{\gamma}x_{\beta}
+q_{\delta\gamma}\Bsj_3 \x_{\alpha}\x_{\gamma}^2 \ot x_{\nu}
\\ & \quad 
-q_{\beta\gamma}q_{\beta\eta} q_{\tau\gamma}q_{\tau\eta} \Bsj_1 \x_{\tau}\x_{\gamma}\x_{\eta} \ot x_{\tau}x_{\beta}
-q_{\alpha\gamma}^2q_{\alpha\kappa}q_{\delta\gamma}\Bsj_2 \Bsj_5 \x_{\gamma}^2\x_{\kappa} \ot x_{\tau}x_{\beta},
\end{aligned}
\\
\label{eq:diff-case11-adee}&
\begin{aligned}
d( & \x_{\alpha}\x_{\delta}\x_{\eta}^2 \ot 1)  = \x_{\alpha}\x_{\delta}\x_{\eta} \ot x_{\eta}
+q_{\nu\eta} \Bsj_3 \x_{\alpha}\x_{\gamma}\x_{\eta} \ot x_{\nu}
+q_{\delta\eta}^2 \x_{\alpha}\x_{\eta}^2 \ot x_{\delta}
\\ & \quad 
+q_{\mu\eta}q_{\gamma\eta} \Bsj_2 \x_{\alpha}\x_{\kappa}\x_{\eta} \ot x_{\gamma}x_{\mu}
-q_{\beta\eta}^2q_{\gamma\eta}\Bsj_1\Bsj_4 \x_{\tau}\x_{\kappa}\x_{\eta} \ot x_{\gamma}x_{\beta}
\\ & \quad 
-q_{\beta\eta}^2q_{\tau\eta}^2\Bsj_1 \x_{\tau}\x_{\eta}^2 \ot x_{\tau}x_{\beta}
-q_{\alpha\delta}q_{\alpha\eta}^2 \x_{\delta}\x_{\eta}^2 \ot x_{\alpha},
\end{aligned}
\\
\label{eq:diff-case11-btgk}&
\begin{aligned}
d( & \x_{\beta}\x_{\tau}\x_{\gamma}\x_{\kappa} \ot 1)  = \x_{\beta}\x_{\tau}\x_{\gamma} \ot x_{\kappa}
-q_{\gamma\kappa} \x_{\beta}\x_{\tau}\x_{\kappa} \ot x_{\gamma}
+q_{\tau\gamma}q_{\tau\kappa} \x_{\beta}\x_{\gamma}\x_{\kappa} \ot x_{\tau}
\\ & \quad
-q_{\beta\tau}q_{\gamma\beta}^{-1} (2)_{\widetilde{q}_{\beta\gamma}} \Bsj_8 \x_{\tau}\x_{\gamma}^2 \ot 1
-q_{\beta\tau}q_{\beta\gamma}q_{\beta\kappa} \x_{\tau}\x_{\gamma} \x_{\kappa} \ot x_{\beta},
\end{aligned}
\\
\label{eq:diff-case11-btke}&
\begin{aligned}
d( & \x_{\beta}\x_{\tau}\x_{\kappa}\x_{\eta} \ot 1)  = \x_{\beta}\x_{\tau}\x_{\kappa} \ot x_{\eta}
-q_{\kappa\eta} \x_{\beta}\x_{\tau}\x_{\eta} \ot x_{\kappa}
+q_{\tau\kappa}q_{\tau\eta} \x_{\beta}\x_{\kappa}\x_{\eta} \ot x_{\tau}
\\ & \quad
-q_{\beta\tau}\Bsj_8 \x_{\tau}\x_{\gamma}\x_{\eta} \ot 1
-q_{\gamma\kappa}q_{\eta\kappa}^{-1}(2)_{\widetilde{q}_{\kappa\eta}} \Bsj_4 \x_{\beta}\x_{\kappa}^2 \ot x_{\gamma}
\\ & \quad
-q_{\beta\tau}q_{\beta\kappa}q_{\beta\eta} \x_{\tau}\x_{\kappa}\x_{\eta} \ot x_{\beta}
-q_{\kappa\eta} \Bsj_4\Bsj_8 \x_{\gamma}^2\x_{\kappa} \ot 1,
\end{aligned}
\\
\label{eq:diff-case11-btee}&
\begin{aligned}
d( & \x_{\beta}\x_{\tau}\x_{\eta}^2 \ot 1) = \x_{\beta}\x_{\tau}\x_{\eta} \ot x_{\eta}
+q_{\tau\eta}^2 \x_{\beta}\x_{\eta}^2 \ot x_{\tau}
-q_{\beta\tau}q_{\beta\eta}^2 \x_{\tau}\x_{\eta}^2 \ot x_{\beta}
\\ & \quad
+q_{\gamma\eta}\Bsj_4 \x_{\beta}\x_{\kappa}\x_{\eta} \ot x_{\gamma}
+\Bsj_4\Bsj_8 \x_{\gamma}^2\x_{\eta} \ot 1,
\end{aligned}
\\
\label{eq:diff-case11-bdge}&
\begin{aligned}
d( & \x_{\beta}\x_{\delta}\x_{\gamma}\x_{\eta} \ot 1)  = \x_{\beta}\x_{\delta}\x_{\gamma} \ot x_{\eta}
-q_{\gamma\eta} \x_{\beta}\x_{\delta}\x_{\eta} \ot x_{\gamma}
+q_{\delta\gamma}q_{\delta\eta} \x_{\beta}\x_{\gamma}\x_{\eta} \ot x_{\delta}
\\ & \quad
+q_{\delta\gamma}\Bsj_2 \x_{\beta}\x_{\gamma}\x_{\kappa} \ot x_{\gamma}x_{\mu}
+q_{\delta\gamma}(2)_{\widetilde{q}_{\gamma\eta}}\Bsj_3  \x_{\beta}\x_{\gamma}^2 \ot x_{\nu}
\\ & \quad
+q_{\beta\gamma}q_{\delta\gamma} (2)_{\widetilde{q}_{\gamma\eta}} \Bsj_2\Bsj_8 \x_{\gamma}^3 \ot x_{\mu}
-q_{\beta\delta}q_{\beta\gamma}q_{\beta\eta} \x_{\delta}\x_{\gamma}\x_{\eta} \ot x_{\beta},
\end{aligned}
\\
\label{eq:diff-case11-bdee}&
\begin{aligned}
d( & \x_{\beta}\x_{\delta}\x_{\eta}^2 \ot 1) = \x_{\beta}\x_{\delta}\x_{\eta} \ot x_{\eta}
+\Bsj_3\Bsj_7 \x_{\beta}\x_{\gamma}\x_{\iota} \ot x_{\kappa}x_{\gamma}
+q_{\nu\eta} \Bsj_3 \x_{\beta}\x_{\gamma}\x_{\eta} \ot x_{\nu}
\\ & \quad
+q_{\mu\eta}q_{\gamma\eta} \Bsj_2 \x_{\beta}\x_{\kappa}\x_{\eta} \ot x_{\gamma}x_{\mu}
+q_{\delta\eta}^2 \x_{\beta}\x_{\eta}^2 \ot x_{\delta}
-q_{\beta\delta}q_{\beta\eta}^2 \x_{\delta}\x_{\eta}^2 \ot x_{\beta}
\\ & \quad
+q_{\mu\eta} \Bsj_2\Bsj_8 \x_{\gamma}^2\x_{\eta} \ot x_{\mu}
+\Bsj_2\Bsj_6\Bsj_8 \x_{\gamma}^2\x_{\iota} \ot x_{\gamma},
\end{aligned}
\\
\label{eq:diff-case11-bgkk}&
\begin{aligned}
d( & \x_{\beta}\x_{\gamma}\x_{\kappa}^2 \ot 1)  = \x_{\beta}\x_{\gamma}\x_{\kappa} \ot x_{\kappa}
+q_{\gamma\kappa}^2 \x_{\beta}\x_{\kappa}^2 \ot x_{\gamma}
-q_{\beta\gamma}\Bsj_8(2)_{\widetilde{q}_{\gamma\kappa}} \x_{\gamma}^2\x_{\kappa} \ot 1
\\ & \quad 
-q_{\beta\gamma}q_{\beta\kappa}^2 \x_{\gamma}\x_{\kappa}^2 \ot x_{\beta},
\end{aligned}
\\
\label{eq:diff-case11-bgke}&
\begin{aligned}
d( & \x_{\beta}\x_{\gamma}\x_{\kappa}\x_{\eta} \ot 1)  = \x_{\beta}\x_{\gamma}\x_{\kappa} \ot x_{\eta}
-q_{\kappa\eta} \x_{\beta}\x_{\gamma}\x_{\eta} \ot x_{\kappa}
+q_{\gamma\kappa}q_{\gamma\eta} \x_{\beta}\x_{\kappa}\x_{\eta} \ot x_{\gamma}
\\ & \quad
-q_{\beta\gamma}(2)_{\widetilde{q}_{\gamma\kappa}}\Bsj_8 \x_{\gamma}^2\x_{\eta} \ot 1
-q_{\beta\gamma}q_{\beta\kappa}q_{\beta\eta} \x_{\gamma}\x_{\kappa}\x_{\eta} \ot x_{\beta},
\end{aligned}
\\
\label{eq:diff-case11-tgke}&
\begin{aligned}
d( & \x_{\tau}\x_{\gamma}\x_{\kappa}\x_{\eta} \ot 1) = \x_{\tau}\x_{\gamma}\x_{\kappa} \ot x_{\eta}
-q_{\kappa\eta} \x_{\tau}\x_{\gamma}\x_{\eta} \ot x_{\kappa}
+q_{\gamma\kappa}q_{\gamma\eta} \x_{\tau}\x_{\kappa}\x_{\eta} \ot x_{\gamma}
\\ & \quad
-q_{\tau\gamma}q_{\tau\kappa}(2)_{\widetilde{q}_{\kappa\eta}} \Bsj_4 \x_{\gamma}\x_{\kappa}^2 \ot x_{\gamma}
-q_{\tau\gamma}q_{\tau\kappa}q_{\tau\eta} \x_{\gamma}\x_{\kappa}\x_{\eta} \ot x_{\tau},
\end{aligned}
\\
\label{eq:diff-case11-tgee}&
\begin{aligned}
d( & \x_{\tau}\x_{\gamma}\x_{\eta}^2 \ot 1) = \x_{\tau}\x_{\gamma}\x_{\eta} \ot x_{\eta} 
+q_{\gamma\eta}^2 \x_{\tau}\x_{\eta}^2 \ot x_{\gamma}
-q_{\tau\gamma}q_{\gamma\eta} \Bsj_4 \x_{\gamma}\x_{\kappa}\x_{\eta} \ot x_{\gamma} 
\\ & \quad
-q_{\tau\gamma}q_{\tau\eta}^2 \x_{\gamma}\x_{\eta}^2 \ot x_{\tau},
\end{aligned}
\\
\label{eq:diff-case11-dgge}&
\begin{aligned}
d( & \x_{\delta}\x_{\gamma}^2\x_{\eta} \ot 1)  = \x_{\delta}\x_{\gamma}^2 \ot x_{\eta}
-q_{\gamma\eta}\x_{\delta}\x_{\gamma}\x_{\eta} \ot x_{\gamma}
-q_{\delta\gamma}^2\Bsj_2 \x_{\gamma}^2\x_{\kappa} \ot x_{\gamma}x_{\mu}
\\ & \quad
-q_{\delta\gamma}^2(3)_{\widetilde{q}_{\gamma\eta}} \Bsj_3 \x_{\gamma}^3 \ot x_{\nu}
-q_{\delta\gamma}^2q_{\delta\eta} \x_{\gamma}^2\x_{\eta} \ot x_{\delta},
\end{aligned}
\\
\label{eq:diff-case11-dgee}&
\begin{aligned}
d( & \x_{\delta}\x_{\gamma}\x_{\eta}^2 \ot 1)  =\x_{\delta}\x_{\gamma}\x_{\eta} \ot x_{\eta}
+q_{\gamma\eta}^2 \x_{\delta}\x_{\eta}^2 \ot x_{\gamma}
-q_{\delta\gamma}(2)_{\widetilde{q}_{\gamma\eta}} \Bsj_3\Bsj_7 \x_{\gamma}^2\x_{\iota} \ot x_{\kappa}x_{\gamma}
\\ &
-q_{\delta\gamma}q_{\nu\eta}(2)_{\widetilde{q}_{\gamma\eta}}\Bsj_3 \x_{\gamma}^2\x_{\eta} \ot x_{\nu}
-q_{\delta\gamma}q_{\mu\eta}q_{\gamma\eta} \Bsj_2 \x_{\gamma}\x_{\kappa}\x_{\eta} \ot x_{\gamma}x_{\mu}
-q_{\delta\gamma}q_{\delta\eta}^2 \x_{\gamma}\x_{\eta}^2 \ot x_{\delta}.
\end{aligned}
\end{align}

\medspace

Next we compute differentials of some $5$-chains:
\begin{align}
\label{eq:diff-case11-abdgg}&
\begin{aligned}
d(& \x_{\alpha}\x_{\beta} \x_{\delta} \x_{\gamma}^{2}
\ot 1)  = \x_{\alpha}\x_{\beta} \x_{\delta} \x_{\gamma} \ot x_{\gamma} 
+q_{\delta\gamma}^2 \x_{\alpha}\x_{\beta}\x_{\gamma}^2 \ot x_{\delta}
-q_{\beta\delta} q_{\beta\gamma}^2 \x_{\alpha} \x_{\delta}\x_{\gamma}^2 \ot x_{\beta}
\\ & \quad
+ q_{\alpha\beta}q_{\beta\gamma}^2 q_{\tau\gamma}^2 \Bsj_1 \x_{\beta}\x_{\tau}\x_{\gamma}^2 \ot x_{\tau}x_{\beta} 
+q_{\alpha\beta}q_{\alpha\delta}q_{\alpha\gamma}^2 \x_{\beta}\x_{\delta}\x_{\gamma}^2 \ot x_{\alpha},
\end{aligned}
\\
\label{eq:diff-case11-abdge}&
\begin{aligned}
d(& \x_{\alpha} \x_{\beta} \x_{\delta} \x_{\gamma}\x_{\eta} 
\ot 1)  = \x_{\alpha}\x_{\beta}\x_{\delta}\x_{\gamma}\ot x_{\eta}
+q_{\delta\gamma}(2)_{\widetilde{q}_{\gamma\eta}}\Bsj_3 \x_{\alpha}\x_{\beta}\x_{\gamma}^2 \ot x_{\nu}
\\ & \quad
-q_{\gamma\eta} \x_{\alpha}\x_{\beta}\x_{\delta}\x_{\eta} \ot x_{\gamma}
+q_{\delta\gamma}q_{\delta\eta} \x_{\alpha}\x_{\beta}\x_{\gamma}\x_{\eta} \ot x_{\delta}
+q_{\delta\gamma}\Bsj_2 \x_{\alpha}\x_{\beta}\x_{\gamma} \x_{\kappa} \ot  x_{\gamma}x_{\mu}
\\ & \quad
-q_{\beta\delta}q_{\beta\gamma}q_{\beta\eta} \x_{\alpha} \x_{\delta}\x_{\gamma}\x_{\eta} \ot x_{\beta}
+q_{\alpha\beta}q_{\beta\gamma}q_{\beta\eta} q_{\tau\gamma} q_{\tau\eta} \Bsj_1 \x_{\beta}\x_{\tau}\x_{\gamma}\x_{\eta} \ot x_{\tau}x_{\beta}
\\ & \quad
-q_{\delta\gamma}q_{\gamma\beta}^{-1} \coef{\beta\eta\gamma}{2}
\Bsj_2\Bsj_8 \x_{\alpha}\x_{\gamma}^3 \ot x_{\mu} 
+ q_{\alpha\beta}q_{\beta\gamma}q_{\beta\eta} q_{\tau\gamma} \Bsj_1\Bsj_4 \x_{\beta}\x_{\tau}\x_{\gamma}\x_{\kappa} \ot x_{\gamma}x_{\beta}
\\ & \quad
+q_{\alpha\beta}q_{\alpha\delta}q_{\alpha\gamma}q_{\alpha\eta} \x_{\beta}\x_{\delta}\x_{\gamma}\x_{\eta} \ot x_{\alpha}
+q_{\alpha\beta}q_{\beta\tau}q_{\beta\eta} q_{\eta\gamma}^{-1} q_{\gamma\beta}^{-1} \coef{\beta\eta\gamma}{2} \Bsj_1 \Bsj_4 \Bsj_8 \x_{\tau}\x_{\gamma}^3 \ot x_{\beta},
\end{aligned}
\\
\label{eq:diff-case11-abdee}&
\begin{aligned}
d(& \x_{\alpha}\x_{\beta}\x_{\delta}\x_{\eta}^{2}\ot 1)  = \x_{\alpha}\x_{\beta}\x_{\delta}\x_{\eta} \ot x_{\eta}
+\Bsj_3\Bsj_7 \x_{\alpha}\x_{\beta}\x_{\gamma}\x_{\iota} \ot x_{\kappa}x_{\gamma}
+q_{\delta\eta}^2 \x_{\alpha}\x_{\beta}\x_{\eta}^2 \ot x_{\delta}
\\ & \quad
+q_{\nu\eta} \Bsj_3 \x_{\alpha}\x_{\beta}\x_{\gamma}\x_{\eta} \ot x_{\nu}
+q_{\mu\gamma}q_{\mu\eta}q_{\gamma\eta} \Bsj_2 \x_{\alpha}\x_{\beta}\x_{\kappa}\x_{\eta} \ot x_{\gamma}x_{\mu}
\\ & \quad
-q_{\beta\delta}q_{\beta\eta}^2 \x_{\alpha}\x_{\delta}\x_{\eta}^2 \ot x_{\beta}
+\Bsj_2\Bsj_6\Bsj_8 \x_{\alpha}\x_{\gamma}^2\x_{\iota} \ot x_{\gamma}
-(4)_{\widetilde{q}_{\alpha\gamma}}\Bsj_2\Bsj_6\Bsj_8\Bsj_9 \x_{\gamma}^4 \ot 1
\\ & \quad
+q_{\mu\eta} \Bsj_2\Bsj_8 \x_{\alpha}\x_{\gamma}^2\x_{\eta} \ot x_{\mu}
+q_{\alpha\beta}q_{\beta\eta}^2q_{\tau\eta}^2 \Bsj_1 \x_{\beta}\x_{\tau}\x_{\eta}^2 \ot x_{\tau}x_{\beta}
\\ & \quad
+q_{\alpha\beta}q_{\beta\gamma}q_{\beta\eta}^2q_{\gamma\eta} \Bsj_1\Bsj_4  \x_{\beta}\x_{\tau}\x_{\kappa}\x_{\eta} \ot x_{\gamma}x_{\beta}
+q_{\alpha\beta}q_{\alpha\delta}q_{\alpha\eta}^2 \x_{\beta}\x_{\delta}\x_{\eta}^2 \ot x_{\alpha}
\\ & \quad
-q_{\beta\tau}q_{\alpha\beta}q_{\beta\eta}^2 \Bsj_1\Bsj_4 \Bsj_8 \x_{\tau}\x_{\gamma}^2\x_{\eta} \ot x_{\beta},
\end{aligned}
\\
\label{eq:diff-case11-abggk}&
\begin{aligned}
d(& \x_{\alpha}\x_{\beta}\x_{\gamma}^2 \x_{\kappa}
\ot 1)  = \x_{\alpha}\x_{\beta}\x_{\gamma}^2 \ot x_{\kappa}
-q_{\gamma\kappa} \x_{\alpha}\x_{\beta}\x_{\gamma}\x_{\kappa} \ot x_{\gamma}
-q_{\beta\gamma}^2q_{\beta\kappa} \x_{\alpha}\x_{\gamma}^2 \x_{\kappa} \ot  x_{\beta}
\\ & \quad
-q_{\gamma\beta}^{-2}(3)_{\widetilde{q}_{\beta\gamma}} \Bsj_8 \x_{\alpha}\x_{\gamma}^3 \ot 1
+q_{\alpha\beta}q_{\alpha\gamma}^2q_{\alpha\kappa} \x_{\beta}\x_{\gamma}^2\x_{\kappa} \ot x_{\alpha},
\end{aligned}
\\
\label{eq:diff-case11-abgke}&
\begin{aligned}
d(& \x_{\alpha}\x_{\beta}\x_{\gamma}\x_{\kappa}\x_{\eta}\ot 1) = \x_{\alpha}\x_{\beta}\x_{\gamma}\x_{\kappa} \ot x_{\eta}
-q_{\kappa\eta} \x_{\alpha}\x_{\beta}\x_{\gamma}\x_{\eta} \ot x_{\kappa}
\\ & \quad
-q_{\gamma\beta}^{-1} (2)_{\widetilde{q}_{\beta\gamma}} \Bsj_8 \x_{\alpha}\x_{\gamma}^2\x_{\eta} \ot 1
+q_{\gamma\kappa}q_{\gamma\eta} \x_{\alpha}\x_{\beta}\x_{\kappa}\x_{\eta} \ot x_{\gamma}
\\ & \quad
-q_{\beta\gamma}q_{\beta\kappa}q_{\beta\eta} \x_{\alpha}\x_{\gamma}\x_{\kappa}\x_{\eta} \ot x_{\beta}
+q_{\alpha\beta}q_{\alpha\gamma}q_{\alpha\kappa}q_{\alpha\eta} \x_{\beta}\x_{\gamma}\x_{\kappa}\x_{\eta} \ot x_{\alpha},
\end{aligned}
\\
\label{eq:diff-case11-adgge}&
\begin{aligned}
d(& \x_{\alpha}\x_{\delta}\x_{\gamma}^2\x_{\eta}\ot 1)  
= \x_{\alpha}\x_{\delta}\x_{\gamma}^2 \ot x_{\eta}
-q_{\gamma\eta} \x_{\alpha}\x_{\delta} \x_{\gamma}\x_{\eta} \ot x_{\gamma}
-q_{\delta\gamma}^2(2)_{\widetilde{q}_{\gamma\eta}} \Bsj_3 \x_{\alpha}\x_{\gamma}^3 \ot x_{\nu}
\\ & \quad
-q_{\delta\gamma}^2\Bsj_2 \x_{\alpha}\x_{\gamma}^2\x_{\kappa} \ot  x_{\gamma}x_{\mu}
-q_{\delta\gamma}^2q_{\delta\eta} \x_{\alpha}\x_{\gamma}^2 \x_{\eta} \ot  x_{\delta}
+q_{\alpha\delta}q_{\alpha\gamma}^2q_{\alpha\eta} \x_{\delta}\x_{\gamma}^2\x_{\eta} \ot x_{\alpha}
\\ & \quad
+q_{\beta\eta}q_{\tau\gamma}^2 q_{\beta\gamma}^2q_{\tau\eta} \Bsj_1 \x_{\tau}\x_{\gamma}^2\x_{\eta} \ot x_{\tau}x_{\beta}
+q_{\beta\eta}q_{\tau\gamma}^2 q_{\beta\gamma}^2 \Bsj_1\Bsj_4 \x_{\tau}\x_{\gamma}^2\x_{\kappa} \ot x_{\gamma}x_{\beta},
\end{aligned}
\\
\label{eq:diff-case11-adgee}&
\begin{aligned}
d(& \x_{\alpha}\x_{\delta}\x_{\gamma}\x_{\eta}^2\ot 1)  
= \x_{\alpha}\x_{\delta}\x_{\gamma}\x_{\eta} \ot x_{\eta}
-q_{\delta\gamma}q_{\nu\eta}(2)_{\widetilde{q}_{\gamma\eta}} \Bsj_3 \x_{\alpha}\x_{\gamma}^2\x_{\eta} \ot x_{\nu}
\\ & \quad
-q_{\delta\gamma}(2)_{\widetilde{q}_{\gamma\eta}} \Bsj_3\Bsj_7 \x_{\alpha}\x_{\gamma}^2\x_{\iota} \ot x_{\kappa}x_{\gamma}
+q_{\delta\gamma}\coef{\alpha\eta\gamma}{2} \Bsj_3\Bsj_7\Bsj_9 \x_{\gamma}^4 \ot x_{\kappa}
\\ & \quad
-q_{\delta\gamma}q_{\mu\eta}q_{\gamma\eta} \Bsj_2 \x_{\alpha}\x_{\gamma}\x_{\kappa}\x_{\eta} \ot x_{\gamma}x_{\mu}
+q_{\gamma\eta}^2 \x_{\alpha}\x_{\delta}\x_{\eta}^2 \ot x_{\gamma}
-q_{\delta\gamma}q_{\delta\eta}^2 \x_{\alpha}\x_{\gamma}\x_{\eta}^2 \ot x_{\delta}
\\ & \quad 
+q_{\beta\gamma}q_{\beta\eta}^2 q_{\tau\gamma}q_{\tau\eta}^2 \Bsj_1 \x_{\tau}\x_{\gamma}\x_{\eta}^2 \ot x_{\tau}x_{\beta}
+q_{\tau\gamma}q_{\beta\gamma}q_{\beta\eta}^2q_{\gamma\eta} \Bsj_1\Bsj_4 \x_{\tau}\x_{\gamma}\x_{\kappa}\x_{\eta} \ot x_{\gamma}x_{\beta}
\\ & \quad
+q_{\alpha\gamma}^2q_{\alpha\kappa}q_{\beta\eta}q_{\delta\gamma} \Bsj_2\Bsj_4\Bsj_5 \x_{\gamma}^2\x_{\kappa}^2 \ot x_{\gamma}x_{\beta}
+q_{\alpha\delta}q_{\alpha\gamma}q_{\alpha\eta}^2 \x_{\delta}\x_{\gamma}\x_{\eta}^2 \ot x_{\alpha}
\\ & \quad
+q_{\alpha\gamma}^2q_{\alpha\kappa}q_{\beta\eta}q_{\delta\gamma}q_{\tau\eta} \Bsj_2\Bsj_5 \x_{\gamma}^2\x_{\kappa}\x_{\eta} \ot x_{\tau}x_{\beta},
\end{aligned}
\\
\label{eq:diff-case11-btggk}&
\begin{aligned}
d(& \x_{\beta}\x_{\tau}\x_{\gamma}^2\x_{\kappa}\ot 1) = \x_{\beta}\x_{\tau}\x_{\gamma}^2 \ot x_{\kappa}
-q_{\gamma\kappa} \x_{\beta}\x_{\tau}\x_{\gamma}\x_{\kappa} \ot x_{\gamma}
-q_{\tau\gamma}^2q_{\tau\kappa} \x_{\beta}\x_{\gamma}^2 \x_{\kappa} \ot x_{\tau}
\\ & \quad 
+q_{\beta\tau}q_{\gamma\beta}^{-2} (3)_{\widetilde{q}_{\beta\gamma}}\Bsj_8 \x_{\tau} \x_{\gamma}^3 \ot 1
+q_{\beta\tau}q_{\beta\gamma}^2q_{\beta\kappa} \x_{\tau} \x_{\gamma}^2\x_{\kappa} \ot x_{\beta},
\end{aligned}
\\
\label{eq:diff-case11-btgke}&
\begin{aligned}
d(& \x_{\beta}\x_{\tau}\x_{\gamma}\x_{\kappa}\x_{\eta} \ot 1) = \x_{\beta}\x_{\tau}\x_{\gamma}\x_{\kappa} \ot x_{\eta}
-q_{\kappa\eta} \x_{\beta}\x_{\tau}\x_{\gamma}\x_{\eta} \ot x_{\kappa}
\\ & \quad
-q_{\tau\gamma}q_{\tau\kappa}(2)_{\widetilde{q}_{\kappa\eta}} \Bsj_4 \x_{\beta}\x_{\gamma}\x_{\kappa}^2 \ot x_{\gamma}
+q_{\gamma\kappa}q_{\gamma\eta} \x_{\beta}\x_{\tau}\x_{\kappa}\x_{\eta} \ot x_{\gamma}
\\ & \quad
-q_{\tau\gamma}q_{\tau\kappa}q_{\tau\eta} \x_{\beta}\x_{\gamma}\x_{\kappa}\x_{\eta} \ot x_{\tau}
+q_{\beta\tau}q_{\gamma\beta}^{-1} (2)_{\widetilde{q}_{\beta\gamma}} \Bsj_8 \x_{\tau}\x_{\gamma}^2\x_{\eta} \ot 1
\\ & \quad
+q_{\beta\tau}q_{\beta\gamma}q_{\beta\kappa}q_{\beta\eta} \x_{\tau}\x_{\gamma}\x_{\kappa}\x_{\eta} \ot x_{\beta}
-q_{\kappa\eta}q_{\eta\gamma}^{-1}(2)_{\widetilde{q}_{\gamma\eta}} \Bsj_4\Bsj_8 \x_{\gamma}^3\x_{\kappa} \ot 1,
\end{aligned}
\\
\label{eq:diff-case11-btgee}&
\begin{aligned}
d(& \x_{\beta}\x_{\tau}\x_{\gamma}\x_{\eta}^2 \ot 1) = \x_{\beta}\x_{\tau}\x_{\gamma}\x_{\eta} \ot x_{\eta}
+q_{\gamma\eta}^2 \x_{\beta}\x_{\tau}\x_{\eta}^2 \ot x_{\gamma}
-q_{\tau\gamma}q_{\gamma\eta} \Bsj_4 \x_{\beta}\x_{\gamma}\x_{\kappa}\x_{\eta} \ot x_{\gamma}
\\ & \quad
-q_{\tau\gamma}q_{\tau\eta}^2 \x_{\beta}\x_{\gamma}\x_{\eta}^2 \ot x_{\tau}
+q_{\beta\tau}q_{\beta\gamma}q_{\beta\eta}^2 \x_{\tau}\x_{\gamma}\x_{\eta}^2 \ot x_{\beta}
+q_{\gamma\eta}^{-1}\coef{-\beta-\eta\gamma}{2} \Bsj_4\Bsj_8 \x_{\gamma}^3\x_{\eta} \ot 1,
\end{aligned}
\\
\label{eq:diff-case11-bdgge}&
\begin{aligned}
d(& \x_{\beta}\x_{\delta}\x_{\gamma}^2\x_{\eta}\ot 1) = \x_{\beta}\x_{\delta}\x_{\gamma}^2 \ot x_{\eta}
-q_{\gamma\eta} \x_{\beta}\x_{\delta}\x_{\gamma}\x_{\eta} \ot x_{\gamma}
-q_{\delta\gamma}^2q_{\delta\eta} \x_{\beta}\x_{\gamma}^2\x_{\eta} \ot x_{\delta}
\\ & \quad
-q_{\delta\gamma}^2\Bsj_2 \x_{\beta}\x_{\gamma}^2\x_{\kappa} \ot x_{\gamma}x_{\mu}
-q_{\delta\gamma}^2(3)_{\widetilde{q}_{\gamma\eta}} \Bsj_3  \x_{\beta}\x_{\gamma}^3 \ot x_{\nu}
\\ & \quad
+q_{\beta\delta}q_{\beta\gamma}^2q_{\beta\eta} \x_{\delta}\x_{\gamma}^2\x_{\eta} \ot x_{\beta}
-q_{\beta\gamma}q_{\gamma\beta}^{-1}q_{\delta\gamma}^2
\coef{\beta\eta}{3} \Bsj_2\Bsj_8 \x_{\gamma}^4 \ot x_{\mu},
\end{aligned}
\\
\label{eq:diff-case11-bdgee}&
\begin{aligned}
d(& \x_{\beta}\x_{\delta}\x_{\gamma}\x_{\eta}^2\ot 1) = \x_{\beta}\x_{\delta}\x_{\gamma}\x_{\eta} \ot x_{\eta}
-q_{\delta\gamma}(2)_{\widetilde{q}_{\gamma\eta}} \Bsj_3\Bsj_7 \x_{\beta}\x_{\gamma}^2\x_{\iota} \ot x_{\kappa}x_{\gamma}
\\ & \quad
-q_{\delta\gamma}(2)_{\widetilde{q}_{\gamma\eta}}q_{\nu\eta} \Bsj_3  \x_{\beta}\x_{\gamma}^2\x_{\eta} \ot x_{\nu}
-q_{\delta\gamma}q_{\mu\eta}q_{\gamma\eta} \Bsj_2 \x_{\beta}\x_{\gamma}\x_{\kappa}\x_{\eta} \ot x_{\gamma}x_{\mu}
\\ & \quad
+q_{\gamma\eta}^2 \x_{\beta}\x_{\delta}\x_{\eta}^2 \ot x_{\gamma}
-q_{\delta\gamma}q_{\delta\eta}^2 \x_{\beta}\x_{\gamma}\x_{\eta}^2 \ot x_{\delta}
+q_{\beta\delta}q_{\beta\gamma}q_{\beta\eta}^2 \x_{\delta}\x_{\gamma}\x_{\eta}^2 \ot x_{\beta}
\\ & \quad
-q_{\beta\gamma}q_{\delta\gamma} (2)_{\widetilde{q}_{\gamma\eta}} \Bsj_2\Bsj_6\Bsj_8 \x_{\gamma}^3\x_{\iota} \ot x_{\gamma}
-q_{\beta\gamma}q_{\delta\gamma}q_{\mu\eta} (2)_{\widetilde{q}_{\gamma\eta}} \Bsj_2\Bsj_8 \x_{\gamma}^3\x_{\eta} \ot x_{\mu},
\end{aligned}
\\
\label{eq:diff-case11-bggkk}&
\begin{aligned}
d(& \x_{\beta}\x_{\gamma}^2\x_{\kappa}^2\ot 1) = \x_{\beta}\x_{\gamma}^2\x_{\kappa} \ot x_{\kappa}
+q_{\gamma\kappa}^2 \x_{\beta} \x_{\gamma}\x_{\kappa}^2 \ot x_{\gamma}
\\ & \quad
+q_{\beta\gamma}^2 (3)_{\widetilde{q}_{\beta\gamma}} \Bsj_8 \x_{\gamma}^3\x_{\kappa} \ot 1
+q_{\beta\kappa}^2q_{\beta\gamma}^2 \x_{\gamma}^2\x_{\kappa}^2 \ot x_{\beta},
\end{aligned}
\\
\label{eq:diff-case11-bggke}&
\begin{aligned}
d(& \x_{\beta}\x_{\gamma}^2\x_{\kappa}\x_{\eta}\ot 1) = \x_{\beta}\x_{\gamma}^2\x_{\kappa} \ot x_{\eta}
-q_{\kappa\eta} \x_{\beta}\x_{\gamma}^2\x_{\eta} \ot x_{\kappa}
+q_{\gamma\kappa}q_{\gamma\eta} \x_{\beta}\x_{\gamma}\x_{\kappa}\x_{\eta} \ot x_{\gamma}
\\ & \quad 
+q_{\beta\kappa}q_{\beta\gamma}^2q_{\beta\eta} \x_{\gamma}^2\x_{\kappa}\x_{\eta} \ot x_{\beta}
+q_{\beta\gamma}^2 (3)_{\widetilde{q}_{\beta\gamma}} \Bsj_8 \x_{\gamma}^3\x_{\eta} \ot 1,
\end{aligned}
\\
\label{eq:diff-case11-tggke}&
\begin{aligned}
d(& \x_{\tau}\x_{\gamma}^2\x_{\kappa}\x_{\eta}\ot 1) = \x_{\tau}\x_{\gamma}^2\x_{\kappa} \ot x_{\eta}
-q_{\kappa\eta} \x_{\tau}\x_{\gamma}^2\x_{\eta} \ot x_{\kappa}
+q_{\gamma\kappa}q_{\gamma\eta} \x_{\tau}\x_{\gamma}\x_{\kappa}\x_{\eta} \ot x_{\gamma}
\\ & \quad
+q_{\tau\gamma}^2q_{\tau\kappa}(2)_{\widetilde{q}_{\kappa\eta}} \Bsj_4 \x_{\gamma}^2\x_{\kappa}^2 \ot x_{\gamma}
+q_{\tau\gamma}^2q_{\tau\kappa}q_{\tau\eta} \x_{\gamma}^2\x_{\kappa}\x_{\eta} \ot x_{\tau},
\end{aligned}
\\
\label{eq:diff-case11-tggee}&
\begin{aligned}
d(& \x_{\tau}\x_{\gamma}^2\x_{\eta}^2\ot 1) = \x_{\tau}\x_{\gamma}^2\x_{\eta} \ot x_{\eta}
+q_{\gamma\eta}^2 \x_{\tau}\x_{\gamma}\x_{\eta}^2 \ot x_{\gamma}
\\ & \quad
+q_{\tau\gamma}^2q_{\gamma\eta} \Bsj_4 \x_{\gamma}^2\x_{\kappa}\x_{\eta} \ot x_{\gamma}
+q_{\tau\gamma}^2q_{\tau\eta}^2 \x_{\gamma}^2\x_{\eta}^2 \ot x_{\tau},
\end{aligned}
\\
\label{eq:diff-case11-dggee}&
\begin{aligned}
d(& \x_{\delta}\x_{\gamma}^2\x_{\eta}^2\ot 1) = \x_{\delta}\x_{\gamma}^2\x_{\eta} \ot x_{\eta}
+q_{\gamma\eta}^2 \x_{\delta}\x_{\gamma}\x_{\eta}^2 \ot x_{\gamma}
\\ & \quad
+q_{\delta\gamma}^2(3)_{\widetilde{q}_{\gamma\eta}} \Bsj_3\Bsj_7 \x_{\gamma}^3\x_{\iota} \ot x_{\kappa}x_{\gamma}
+q_{\delta\gamma}^2q_{\nu\eta} (3)_{\widetilde{q}_{\gamma\eta}} \Bsj_3 \x_{\gamma}^3\x_{\eta} \ot x_{\nu}
\\ & \quad
+q_{\delta\gamma}^2q_{\mu\eta}q_{\gamma\eta} \Bsj_2 \x_{\gamma}^2\x_{\kappa}\x_{\eta} \ot x_{\gamma}x_{\mu}
+q_{\delta\gamma}^2q_{\delta\eta}^2 \x_{\gamma}^2\x_{\eta}^2 \ot x_{\delta}.
\end{aligned}
\end{align}

First we check \eqref{eq:diff-case11-abdgg}. Using \eqref{eq:diff-case11-abdg}, Remark \ref{rem:differential-q-commute}, \eqref{eq:diff-case11-adgg},
\begin{align*}
d( & \x_{\alpha}\x_{\beta} \x_{\delta} \x_{\gamma}^{2}\ot 1) =
\x_{\alpha}\x_{\beta} \x_{\delta} \x_{\gamma} \ot x_{\gamma} 
-s\big( 
-q_{\delta\gamma}^2 \x_{\alpha} \x_{\beta}\x_{\gamma} \ot x_{\gamma}x_{\delta}
+q_{\beta\delta} q_{\beta\gamma}^2 \x_{\alpha} \x_{\delta}\x_{\gamma} \ot x_{\gamma}x_{\beta}
\\ & \quad
- q_{\alpha\beta}q_{\beta\gamma}^2 q_{\tau\gamma}^2 \Bsj_1 \x_{\beta}\x_{\tau}\x_{\gamma} \ot x_{\gamma}x_{\tau}x_{\beta} 
-q_{\alpha\beta}q_{\alpha\delta}q_{\alpha\gamma}^2 \x_{\beta}\x_{\delta}\x_{\gamma} \ot x_{\gamma}x_{\alpha} \big)
\\ & =
\x_{\alpha}\x_{\beta} \x_{\delta} \x_{\gamma} \ot x_{\gamma} 
+q_{\delta\gamma}^2 \x_{\alpha}\x_{\beta}\x_{\gamma}^2 \ot x_{\delta}
-q_{\beta\delta} q_{\beta\gamma}^2 \x_{\alpha} \x_{\delta}\x_{\gamma}^2 \ot x_{\beta}
\\ & \quad
-s\big( 
- q_{\alpha\beta}q_{\beta\gamma}^2 q_{\tau\gamma}^2 \Bsj_1 \x_{\beta}\x_{\tau}\x_{\gamma} \ot x_{\gamma}x_{\tau}x_{\beta} 
-q_{\alpha\beta}q_{\alpha\delta}q_{\alpha\gamma}^2 \x_{\beta}\x_{\delta}\x_{\gamma} \ot x_{\gamma}x_{\alpha} 
\\ & \quad
-q_{\alpha\beta}q_{\alpha\gamma}^2q_{\alpha\delta} q_{\delta\gamma}^2 \x_{\beta}\x_{\gamma}^2 \ot x_{\delta}x_{\alpha} 
-q_{\alpha\beta}q_{\alpha\gamma}^2q_{\delta\gamma}^2 \Bsj_1 \x_{\beta}\x_{\gamma}^2 \ot x_{\tau}^2x_{\beta}
\\ & \quad
+q_{\beta\delta} q_{\beta\gamma}^4 q_{\tau\gamma}^2 \Bsj_1 \x_{\tau}\x_{\gamma}^2 \ot x_{\tau}x_{\beta}^2
+q_{\alpha\beta}q_{\alpha\delta}q_{\alpha\gamma}^2 q_{\beta\delta} q_{\beta\gamma}^2 \x_{\delta}\x_{\gamma}^2 \ot x_{\beta}x_{\alpha}
\big)
\\ & =
\x_{\alpha}\x_{\beta} \x_{\delta} \x_{\gamma} \ot x_{\gamma} 
+q_{\delta\gamma}^2 \x_{\alpha}\x_{\beta}\x_{\gamma}^2 \ot x_{\delta}
-q_{\beta\delta} q_{\beta\gamma}^2 \x_{\alpha} \x_{\delta}\x_{\gamma}^2 \ot x_{\beta}
\\ & \quad
+ q_{\alpha\beta}q_{\beta\gamma}^2 q_{\tau\gamma}^2 \Bsj_1 \x_{\beta}\x_{\tau}\x_{\gamma}^2 \ot x_{\tau}x_{\beta} 
+q_{\alpha\beta}q_{\alpha\delta}q_{\alpha\gamma}^2 \x_{\beta}\x_{\delta}\x_{\gamma}^2 \ot x_{\alpha}.
\end{align*}

For \eqref{eq:diff-case11-abdge}, we use \eqref{eq:diff-case11-abdg}, \eqref{eq:diff-case11-abde}, \eqref{eq:diff-case11-adge}, \eqref{eq:diff-case11-btgk}, \eqref{eq:diff-case11-btggge},
Remark \ref{rem:differential-q-commute}, \eqref{eq:diff-case11-bdge},
$\x_{\alpha}\x_{\gamma}^2 \ot x_{\nu}x_{\beta}= s(\x_{\alpha}\x_{\gamma} \ot x_{\gamma}x_{\nu}x_{\beta})$, 
$\x_{\gamma}^3 \ot x_{\mu}x_{\alpha}= s(\x_{\gamma}^2 \ot x_{\gamma}x_{\mu}x_{\alpha})$,
\begin{align*}
d(& \x_{\alpha}\x_{\beta}\x_{\delta}\x_{\gamma}\x_{\eta} \ot 1)  = \x_{\alpha}\x_{\beta}\x_{\delta}\x_{\gamma}\ot x_{\eta}
- s\big( 
q_{\gamma\eta} \x_{\alpha}\x_{\beta}\x_{\delta} \ot x_{\eta}x_{\gamma}
-q_{\delta\gamma}q_{\delta\eta} \x_{\alpha}\x_{\beta}\x_{\gamma} \ot x_{\eta}x_{\delta}
\\ & \quad
+q_{\beta\delta}q_{\beta\gamma}q_{\beta\eta} \x_{\alpha} \x_{\delta}\x_{\gamma} \ot x_{\eta}x_{\beta}
-q_{\delta\gamma}\Bsj_2 \x_{\alpha}\x_{\beta}\x_{\gamma} \ot  x_{\kappa}x_{\gamma}x_{\mu}
-q_{\delta\gamma}\Bsj_3 \x_{\alpha}\x_{\beta}\x_{\gamma} \ot  x_{\gamma} x_{\nu}
\\ & \quad
- q_{\alpha\beta}q_{\beta\gamma}q_{\beta\eta} q_{\tau\gamma}q_{\tau\eta} \Bsj_1 \x_{\beta}\x_{\tau}\x_{\gamma} \ot x_{\eta}x_{\tau}x_{\beta}
- q_{\alpha\beta}q_{\beta\gamma}q_{\beta\eta} q_{\tau\gamma} \Bsj_1\Bsj_4 \x_{\beta}\x_{\tau}\x_{\gamma} \ot x_{\kappa}x_{\gamma}x_{\beta}
\\ & \quad
-q_{\alpha\beta}q_{\alpha\delta}q_{\alpha\gamma}q_{\alpha\eta} \x_{\beta}\x_{\delta}\x_{\gamma} \ot x_{\eta}x_{\alpha}
\big)
\\ & = \x_{\alpha}\x_{\beta}\x_{\delta}\x_{\gamma}\ot x_{\eta}
-q_{\gamma\eta} \x_{\alpha}\x_{\beta}\x_{\delta}\x_{\eta} \ot x_{\gamma}
+q_{\delta\gamma}q_{\delta\eta} \x_{\alpha}\x_{\beta}\x_{\gamma}\x_{\eta} \ot x_{\delta}
+q_{\delta\gamma}\Bsj_2 \x_{\alpha}\x_{\beta}\x_{\gamma} \x_{\kappa} \ot  x_{\gamma}x_{\mu}
\\ & \quad
+q_{\delta\gamma}(2)_{\widetilde{q}_{\gamma\eta}}\Bsj_3 \x_{\alpha}\x_{\beta}\x_{\gamma}^2 \ot x_{\nu}
-q_{\beta\delta}q_{\beta\gamma}q_{\beta\eta} \x_{\alpha} \x_{\delta}\x_{\gamma}\x_{\eta} \ot x_{\beta}
-q_{\delta\gamma}q_{\gamma\beta}^{-1} \coef{\beta\eta\gamma}{2}
\Bsj_2\Bsj_8 \x_{\alpha}\x_{\gamma}^3 \ot x_{\mu} 
\\ & \quad
+q_{\alpha\beta}q_{\beta\gamma}q_{\beta\eta} q_{\tau\gamma} q_{\tau\eta} \Bsj_1 \x_{\beta}\x_{\tau}\x_{\gamma}\x_{\eta} \ot x_{\tau}x_{\beta}
+ q_{\alpha\beta}q_{\beta\gamma}q_{\beta\eta} q_{\tau\gamma} \Bsj_1\Bsj_4 \x_{\beta}\x_{\tau}\x_{\gamma}\x_{\kappa} \ot x_{\gamma}x_{\beta}
\\ & \quad
+q_{\alpha\beta}q_{\alpha\delta}q_{\alpha\gamma}q_{\alpha\eta} \x_{\beta}\x_{\delta}\x_{\gamma}\x_{\eta} \ot x_{\alpha}
+q_{\alpha\beta}q_{\beta\tau}q_{\beta\eta} q_{\eta\gamma}^{-1} q_{\gamma\beta}^{-1} \coef{\beta\eta\gamma}{2} \Bsj_1 \Bsj_4 \Bsj_8 \x_{\tau}\x_{\gamma}^3 \ot x_{\beta}.
\end{align*}

For \eqref{eq:diff-case11-abdee}, we use \eqref{eq:diff-case11-abde}, \eqref{eq:diff-case11-abke}, \eqref{eq:diff-case11-adee}, \eqref{eq:diff-case11-agggi}, \eqref{eq:diff-case11-btke}, \eqref{eq:diff-case11-btee}, Remark \ref{rem:differential-q-commute}, \eqref{eq:diff-case11-bdee} and \eqref{eq:diff-case11-tggge}:
\begin{align*}
d(& \x_{\alpha}\x_{\beta}\x_{\delta}\x_{\eta}^{2}\ot 1) = \x_{\alpha}\x_{\beta}\x_{\delta}\x_{\eta} \ot x_{\eta} -
s\big( 
-q_{\nu\eta} \Bsj_3 \x_{\alpha}\x_{\beta}\x_{\gamma} \ot x_{\eta}x_{\nu}
-\Bsj_3\Bsj_7 \x_{\alpha}\x_{\beta}\x_{\gamma} \ot x_{\iota}x_{\kappa}x_{\gamma}
\\ & \quad
-q_{\delta\eta}^2 \x_{\alpha}\x_{\beta}\x_{\eta} \ot x_{\eta}x_{\delta}
-q_{\delta\eta} \Bsj_2 \x_{\alpha}\x_{\beta}\x_{\eta} \ot x_{\kappa}x_{\gamma}x_{\mu}
-q_{\delta\eta} \Bsj_3 \x_{\alpha}\x_{\beta}\x_{\eta} \ot x_{\gamma}x_{\nu}
\\ & \quad
-q_{\gamma\eta}q_{\mu\eta}\Bsj_2 \x_{\alpha}\x_{\beta}\x_{\kappa} \ot x_{\eta}x_{\mu}x_{\gamma}
-\Bsj_2\Bsj_8 \x_{\alpha}\x_{\gamma}^2 \ot (q_{\mu\eta} x_{\eta}x_{\mu} +\Bsj_6 x_{\iota}x_{\gamma})
+q_{\beta\delta}q_{\beta\eta}^2 \x_{\alpha}\x_{\delta}\x_{\eta} \ot x_{\eta}x_{\beta}
\\ & \quad
-q_{\alpha\beta}q_{\beta\eta}^2q_{\gamma\eta} \Bsj_1\Bsj_4  \x_{\beta}\x_{\tau}\x_{\kappa} \ot x_{\eta}x_{\beta}x_{\gamma}
-q_{\alpha\beta}q_{\beta\eta}^2q_{\tau\eta} \Bsj_1 \x_{\beta}\x_{\tau}\x_{\eta} \ot (q_{\tau\eta} x_{\eta}x_{\tau} +\Bsj_4 x_{\kappa}x_{\gamma})x_{\beta}
\\ & \quad
-q_{\alpha\beta}q_{\alpha\delta}q_{\alpha\eta}^2 \x_{\beta}\x_{\delta}\x_{\eta} \ot x_{\eta}x_{\alpha}
+q_{\beta\tau}q_{\alpha\beta}q_{\beta\eta}^2 \Bsj_1\Bsj_4 \Bsj_8 \x_{\tau}\x_{\gamma}^2 \ot x_{\eta}x_{\beta}
\big)
\\ & = 
\x_{\alpha}\x_{\beta}\x_{\delta}\x_{\eta} \ot x_{\eta}
+\Bsj_3\Bsj_7 \x_{\alpha}\x_{\beta}\x_{\gamma}\x_{\iota} \ot x_{\kappa}x_{\gamma}
+q_{\nu\eta} \Bsj_3 \x_{\alpha}\x_{\beta}\x_{\gamma}\x_{\eta} \ot x_{\nu}
\\ & \quad
+q_{\mu\gamma}q_{\mu\eta}q_{\gamma\eta} \Bsj_2 \x_{\alpha}\x_{\beta}\x_{\kappa}\x_{\eta} \ot x_{\gamma}x_{\mu}
+q_{\delta\eta}^2 \x_{\alpha}\x_{\beta}\x_{\eta}^2 \ot x_{\delta}
-q_{\beta\delta}q_{\beta\eta}^2 \x_{\alpha}\x_{\delta}\x_{\eta}^2 \ot x_{\beta}
\\ & \quad
+\Bsj_2\Bsj_6\Bsj_8 \x_{\alpha}\x_{\gamma}^2\x_{\iota} \ot x_{\gamma}
-(4)_{\widetilde{q}_{\alpha\gamma}}\Bsj_2\Bsj_6\Bsj_8\Bsj_9 \x_{\gamma}^4 \ot 1
+q_{\mu\eta} \Bsj_2\Bsj_8 \x_{\alpha}\x_{\gamma}^2\x_{\eta} \ot x_{\mu}
\\ & \quad
+q_{\alpha\beta}q_{\beta\eta}^2q_{\tau\eta}^2 \Bsj_1 \x_{\beta}\x_{\tau}\x_{\eta}^2 \ot x_{\tau}x_{\beta}
+q_{\alpha\beta}q_{\beta\gamma}q_{\beta\eta}^2q_{\gamma\eta} \Bsj_1\Bsj_4  \x_{\beta}\x_{\tau}\x_{\kappa}\x_{\eta} \ot x_{\gamma}x_{\beta}
\\ & \quad
+q_{\alpha\beta}q_{\alpha\delta}q_{\alpha\eta}^2 \x_{\beta}\x_{\delta}\x_{\eta}^2 \ot x_{\alpha}
-q_{\beta\tau}q_{\alpha\beta}q_{\beta\eta}^2 \Bsj_1\Bsj_4 \Bsj_8 \x_{\tau}\x_{\gamma}^2\x_{\eta} \ot x_{\beta}.
\end{align*}

Next we check \eqref{eq:diff-case11-abggk} using Remark \ref{rem:differential-q-commute}, \eqref{eq:diff-case11-abgk} and $\x_{\gamma}^3 \ot x_{\alpha}=s\big(\x_{\gamma}^2 \ot x_{\gamma}x_{\alpha}\big)$:
\begin{align*}
d(& \x_{\alpha}\x_{\beta}\x_{\gamma}^2\x_{\kappa} \ot 1) =
\x_{\alpha}\x_{\beta}\x_{\gamma}^2 \ot x_{\kappa}
-s \big( 
q_{\gamma\kappa} \x_{\alpha}\x_{\beta}\x_{\gamma} \ot x_{\kappa}x_{\gamma}
+q_{\beta\gamma}^2q_{\beta\kappa} \x_{\alpha}\x_{\gamma}^2 \ot  x_{\kappa}x_{\beta}
\\ & \quad
+q_{\beta\gamma}^2\Bsj_8 \x_{\alpha}\x_{\gamma}^2 \ot x_{\gamma}
-q_{\alpha\beta}q_{\alpha\gamma}^2q_{\alpha\kappa} \x_{\beta}\x_{\gamma}^2 \ot x_{\kappa}x_{\alpha}
\big)
\\ & =
\x_{\alpha}\x_{\beta}\x_{\gamma}^2 \ot x_{\kappa}
-q_{\gamma\kappa} \x_{\alpha}\x_{\beta}\x_{\gamma}\x_{\kappa} \ot x_{\gamma}
-q_{\beta\gamma}^2q_{\beta\kappa} \x_{\alpha}\x_{\gamma}^2 \x_{\kappa} \ot  x_{\beta}
-q_{\gamma\beta}^{-2}(3)_{\widetilde{q}_{\beta\gamma}} \Bsj_8 \x_{\alpha}\x_{\gamma}^3 \ot 1
\\ & \quad
+q_{\alpha\beta}q_{\alpha\gamma}^2q_{\alpha\kappa} \x_{\beta}\x_{\gamma}^2\x_{\kappa} \ot x_{\alpha}.
\end{align*}

For \eqref{eq:diff-case11-abgke} we use \eqref{eq:diff-case11-abgk}, Remark \ref{rem:differential-q-commute}, \eqref{eq:diff-case11-abke} and \eqref{eq:diff-case11-bgke}:
\begin{align*}
d(& \x_{\alpha}\x_{\beta}\x_{\gamma}\x_{\kappa}\x_{\eta}\ot 1) = \x_{\alpha}\x_{\beta}\x_{\gamma}\x_{\kappa} \ot x_{\eta}
-s \big( 
q_{\eta\kappa} \x_{\alpha}\x_{\beta}\x_{\gamma} \ot x_{\eta}x_{\kappa}
-q_{\gamma\kappa}q_{\gamma\eta} \x_{\alpha}\x_{\beta}\x_{\kappa} \ot x_{\eta}x_{\gamma}
\\ &
+q_{\beta\gamma}q_{\beta\kappa}q_{\beta\eta} \x_{\alpha}\x_{\gamma}\x_{\kappa} \ot x_{\eta}x_{\beta}
+q_{\gamma\beta}^{-1} (2)_{\widetilde{q}_{\beta\gamma}} \Bsj_8 \x_{\alpha}\x_{\gamma}^2 \ot x_{\eta} 
-q_{\alpha\beta}q_{\alpha\gamma}q_{\alpha\kappa} \x_{\beta}\x_{\gamma}\x_{\kappa} \ot x_{\alpha}x_{\eta}
\big)
\\ & = \x_{\alpha}\x_{\beta}\x_{\gamma}\x_{\kappa} \ot x_{\eta}
-q_{\kappa\eta} \x_{\alpha}\x_{\beta}\x_{\gamma}\x_{\eta} \ot x_{\kappa}
-q_{\gamma\beta}^{-1} (2)_{\widetilde{q}_{\beta\gamma}} \Bsj_8 \x_{\alpha}\x_{\gamma}^2\x_{\eta} \ot 1
\\ & \quad
+q_{\gamma\kappa}q_{\gamma\eta} \x_{\alpha}\x_{\beta}\x_{\kappa}\x_{\eta} \ot x_{\gamma}
-q_{\beta\gamma}q_{\beta\kappa}q_{\beta\eta} \x_{\alpha}\x_{\gamma}\x_{\kappa}\x_{\eta} \ot x_{\beta}
+q_{\alpha\beta}q_{\alpha\gamma}q_{\alpha\kappa}q_{\alpha\eta} \x_{\beta}\x_{\gamma}\x_{\kappa}\x_{\eta} \ot x_{\alpha}.
\end{align*}

To check \eqref{eq:diff-case11-adgge} we use \eqref{eq:diff-case11-adgg}, \eqref{eq:diff-case11-adge}, Remark \ref{rem:differential-q-commute}, \eqref{eq:diff-case11-tggge} and \eqref{eq:diff-case11-dgge}:
\begin{align*}
d( & \x_{\alpha}\x_{\delta}\x_{\gamma}^2\x_{\eta}\ot 1)
= \x_{\alpha}\x_{\delta}\x_{\gamma}^2 \ot x_{\eta}
-s \big( 
q_{\gamma\eta} \x_{\alpha}\x_{\delta} \x_{\gamma} \ot x_{\eta}x_{\gamma}
+q_{\delta\gamma}^2q_{\delta\eta} \x_{\alpha}\x_{\gamma}^2 \ot  x_{\eta}x_{\delta}
\\ & \quad
+q_{\delta\gamma}^2\Bsj_2 \x_{\alpha}\x_{\gamma}^2 \ot  x_{\kappa}x_{\gamma}x_{\mu}
+q_{\delta\gamma}^2\Bsj_3 \x_{\alpha}\x_{\gamma}^2 \ot x_{\gamma} x_{\nu}
-q_{\beta\eta}q_{\tau\gamma}^2 q_{\beta\gamma}^2q_{\tau\eta} \Bsj_1 \x_{\tau}\x_{\gamma}^2 \ot x_{\eta}x_{\tau}x_{\beta}
\\ & \quad
-q_{\beta\eta}q_{\tau\gamma}^2 q_{\beta\gamma}^2 \Bsj_1\Bsj_4 \x_{\tau}\x_{\gamma}^2 \ot x_{\kappa}x_{\gamma}x_{\beta}
-q_{\alpha\delta}q_{\alpha\gamma}^2q_{\alpha\eta} \x_{\delta}\x_{\gamma}^2 \ot x_{\eta}x_{\alpha}
\big)
\\ & = 
\x_{\alpha}\x_{\delta}\x_{\gamma}^2 \ot x_{\eta}
-q_{\gamma\eta} \x_{\alpha}\x_{\delta} \x_{\gamma}\x_{\eta} \ot x_{\gamma}
-q_{\delta\gamma}^2(2)_{\widetilde{q}_{\gamma\eta}} \Bsj_3 \x_{\alpha}\x_{\gamma}^3 \ot x_{\nu}
-q_{\delta\gamma}^2\Bsj_2 \x_{\alpha}\x_{\gamma}^2\x_{\kappa} \ot  x_{\gamma}x_{\mu}
\\ & \quad
-q_{\delta\gamma}^2q_{\delta\eta} \x_{\alpha}\x_{\gamma}^2 \x_{\eta} \ot  x_{\delta}
+q_{\alpha\delta}q_{\alpha\gamma}^2q_{\alpha\eta} \x_{\delta}\x_{\gamma}^2\x_{\eta} \ot x_{\alpha}
+q_{\beta\eta}q_{\tau\gamma}^2 q_{\beta\gamma}^2q_{\tau\eta} \Bsj_1 \x_{\tau}\x_{\gamma}^2\x_{\eta} \ot x_{\tau}x_{\beta}
\\ & \quad
+q_{\beta\eta}q_{\tau\gamma}^2 q_{\beta\gamma}^2 \Bsj_1\Bsj_4 \x_{\tau}\x_{\gamma}^2\x_{\kappa} \ot x_{\gamma}x_{\beta}
+q_{\alpha\delta}q_{\alpha\gamma}^2q_{\alpha\eta} 
q_{\delta\gamma}^2(3)_{\widetilde{q}_{\gamma\eta}} \Bsj_3 s \big(\x_{\gamma}^3 \ot x_{\nu}x_{\alpha}\big),
\end{align*}
and we use that 
$\x_{\gamma}^3 \ot x_{\nu}x_{\alpha}
=s\big(\x_{\gamma}^2 \ot x_{\gamma}x_{\nu}x_{\alpha}\big)$. To prove \eqref{eq:diff-case11-adgee} we use \eqref{eq:diff-case11-adge}, \eqref{eq:diff-case11-adee}, \eqref{eq:diff-case11-agggi}, Remark \ref{rem:differential-q-commute}, \eqref{eq:diff-case11-tgke}, \eqref{eq:diff-case11-tgee} and \eqref{eq:diff-case11-dgee}:
\begin{align*}
d(& \x_{\alpha}\x_{\delta}\x_{\gamma}\x_{\eta}^2\ot 1)= \x_{\alpha}\x_{\delta}\x_{\gamma}\x_{\eta} \ot x_{\eta}
-s \big( 
-q_{\gamma\eta}^2 \x_{\alpha}\x_{\delta}\x_{\eta} \ot x_{\eta}x_{\gamma}
+q_{\delta\gamma}q_{\delta\eta} \x_{\alpha}\x_{\gamma} \x_{\eta} \ot x_{\delta}x_{\eta}
\\ & \quad
+q_{\delta\gamma}\Bsj_2 \x_{\alpha}\x_{\gamma}\x_{\kappa} \ot  x_{\gamma}x_{\mu}x_{\eta}
-q_{\alpha\delta}q_{\alpha\gamma}q_{\alpha\eta} \x_{\delta}\x_{\gamma}\x_{\eta} \ot x_{\alpha}x_{\eta}
\\ & \quad
- q_{\tau\gamma}q_{\beta\gamma}q_{\beta\eta}\Bsj_1\Bsj_4  \x_{\tau}\x_{\gamma}\x_{\kappa} \ot x_{\gamma}x_{\beta}x_{\eta}
+q_{\delta\gamma}\Bsj_3 \x_{\alpha}\x_{\gamma}^2 \ot x_{\nu}x_{\eta}
\\ & \quad 
-q_{\beta\gamma}q_{\beta\eta} q_{\tau\gamma}q_{\tau\eta} \Bsj_1 \x_{\tau}\x_{\gamma}\x_{\eta} \ot x_{\tau}x_{\beta}x_{\eta}
-q_{\alpha\gamma}^2q_{\alpha\kappa}q_{\delta\gamma}\Bsj_2 \Bsj_5 \x_{\gamma}^2\x_{\kappa} \ot x_{\tau}x_{\beta}x_{\eta}
\big)
\\ & = \x_{\alpha}\x_{\delta}\x_{\gamma}\x_{\eta} \ot x_{\eta}
+q_{\gamma\eta}^2 \x_{\alpha}\x_{\delta}\x_{\eta}^2 \ot x_{\gamma}
-q_{\delta\gamma}q_{\nu\eta}(2)_{\widetilde{q}_{\gamma\eta}} \Bsj_3 \x_{\alpha}\x_{\gamma}^2\x_{\eta} \ot x_{\nu}
\\ & \quad
-q_{\delta\gamma}(2)_{\widetilde{q}_{\gamma\eta}} \Bsj_3\Bsj_7 \x_{\alpha}\x_{\gamma}^2\x_{\iota} \ot x_{\kappa}x_{\gamma}
+q_{\delta\gamma}\coef{\alpha\eta\gamma}{2} \Bsj_3\Bsj_7\Bsj_9 \x_{\gamma}^4 \ot x_{\kappa}
-q_{\delta\gamma}q_{\delta\eta}^2 \x_{\alpha}\x_{\gamma}\x_{\eta}^2 \ot x_{\delta}
\\ & \quad
-q_{\delta\gamma}q_{\mu\eta}q_{\gamma\eta} \Bsj_2 \x_{\alpha}\x_{\gamma}\x_{\kappa}\x_{\eta} \ot x_{\gamma}x_{\mu}
+q_{\tau\gamma}q_{\beta\gamma}q_{\beta\eta}^2q_{\gamma\eta} \Bsj_1\Bsj_4 \x_{\tau}\x_{\gamma}\x_{\kappa}\x_{\eta} \ot x_{\gamma}x_{\beta}
\\ & \quad 
+q_{\beta\gamma}q_{\beta\eta}^2 q_{\tau\gamma}q_{\tau\eta}^2 \Bsj_1 \x_{\tau}\x_{\gamma}\x_{\eta}^2 \ot x_{\tau}x_{\beta}
+q_{\alpha\gamma}^2q_{\alpha\kappa}q_{\beta\eta}q_{\delta\gamma} \Bsj_2\Bsj_4\Bsj_5 \x_{\gamma}^2\x_{\kappa}^2 \ot x_{\gamma}x_{\beta}
\\ & \quad
+q_{\alpha\delta}q_{\alpha\gamma}q_{\alpha\eta}^2 \x_{\delta}\x_{\gamma}\x_{\eta}^2 \ot x_{\alpha}
+q_{\alpha\gamma}^2q_{\alpha\kappa}q_{\beta\eta}q_{\delta\gamma}q_{\tau\eta} \Bsj_2\Bsj_5 \x_{\gamma}^2\x_{\kappa}\x_{\eta} \ot x_{\tau}x_{\beta}.
\end{align*}

We compute \eqref{eq:diff-case11-btgke} using \eqref{eq:diff-case11-btgk}, \eqref{eq:diff-case11-btggge}, \eqref{eq:diff-case11-btke}, 
\eqref{eq:diff-case11-bgkk}, \eqref{eq:diff-case11-bgke}, \eqref{eq:diff-case11-tggge}, \eqref{eq:diff-case11-tgke} and Remark \ref{rem:differential-q-commute}:
\begin{align*}
d(& \x_{\beta}\x_{\tau}\x_{\gamma}\x_{\kappa}\x_{\eta} \ot 1) = \x_{\beta}\x_{\tau}\x_{\gamma}\x_{\kappa} \ot x_{\eta}
-s\big( 
q_{\kappa\eta} \x_{\beta}\x_{\tau}\x_{\gamma} \ot x_{\eta}x_{\kappa}
-q_{\gamma\kappa}q_{\gamma\eta} \x_{\beta}\x_{\tau}\x_{\kappa} \ot x_{\eta}x_{\gamma}
\\ & \quad
+q_{\tau\gamma}q_{\tau\kappa}q_{\tau\eta} \x_{\beta}\x_{\gamma}\x_{\kappa} \ot x_{\eta}x_{\tau}
+q_{\tau\gamma}q_{\tau\kappa} \Bsj_4 \x_{\beta}\x_{\gamma}\x_{\kappa} \ot x_{\kappa}x_{\gamma}
\\ & \quad
-q_{\beta\tau}q_{\gamma\beta}^{-1} (2)_{\widetilde{q}_{\beta\gamma}} \Bsj_8 \x_{\tau}\x_{\gamma}^2 \ot x_{\eta}
-q_{\beta\tau}q_{\beta\gamma}q_{\beta\kappa}q_{\beta\eta} \x_{\tau}\x_{\gamma} \x_{\kappa} \ot x_{\eta}x_{\beta}
\big)
\\ & = \x_{\beta}\x_{\tau}\x_{\gamma}\x_{\kappa} \ot x_{\eta}
-q_{\kappa\eta} \x_{\beta}\x_{\tau}\x_{\gamma}\x_{\eta} \ot x_{\kappa}
-q_{\tau\gamma}q_{\tau\kappa}(2)_{\widetilde{q}_{\kappa\eta}} \Bsj_4 \x_{\beta}\x_{\gamma}\x_{\kappa}^2 \ot x_{\gamma}
\\ & \quad
+q_{\gamma\kappa}q_{\gamma\eta} \x_{\beta}\x_{\tau}\x_{\kappa}\x_{\eta} \ot x_{\gamma}
-q_{\tau\gamma}q_{\tau\kappa}q_{\tau\eta} \x_{\beta}\x_{\gamma}\x_{\kappa}\x_{\eta} \ot x_{\tau}
+q_{\beta\tau}q_{\gamma\beta}^{-1} (2)_{\widetilde{q}_{\beta\gamma}} \Bsj_8 \x_{\tau}\x_{\gamma}^2\x_{\eta} \ot 1
\\ & \quad
+q_{\beta\tau}q_{\beta\gamma}q_{\beta\kappa}q_{\beta\eta} \x_{\tau}\x_{\gamma}\x_{\kappa}\x_{\eta} \ot x_{\beta}
-q_{\kappa\eta}q_{\eta\gamma}^{-1}(2)_{\widetilde{q}_{\gamma\eta}} \Bsj_4\Bsj_8 \x_{\gamma}^3\x_{\kappa} \ot 1.
\end{align*}

Next we check \eqref{eq:diff-case11-btggk} using Remark \ref{rem:differential-q-commute}, \eqref{eq:diff-case11-btgk} and \eqref{eq:diff-case11-bgggk}:
\begin{align*}
d(& \x_{\beta}\x_{\tau}\x_{\gamma}^2\x_{\kappa}\ot 1) =
\x_{\beta}\x_{\tau}\x_{\gamma}^2 \ot x_{\kappa}
-s \big( 
q_{\gamma\kappa} \x_{\beta}\x_{\tau}\x_{\gamma} \ot x_{\kappa}x_{\gamma}
+q_{\tau\gamma}^2q_{\tau\kappa} \x_{\beta}\x_{\gamma}^2 \ot x_{\kappa}x_{\tau}
\\ & \quad 
-q_{\beta\tau}q_{\beta\gamma}^2q_{\beta\kappa} \x_{\tau} \x_{\gamma}^2 \ot x_{\kappa}x_{\beta}
-q_{\beta\tau}q_{\beta\gamma}^2\Bsj_8 \x_{\tau} \x_{\gamma}^2 \ot x_{\gamma}
\big)
\\ & =
\x_{\beta}\x_{\tau}\x_{\gamma}^2 \ot x_{\kappa}
-q_{\gamma\kappa} \x_{\beta}\x_{\tau}\x_{\gamma}\x_{\kappa} \ot x_{\gamma}
-q_{\tau\gamma}^2q_{\tau\kappa} \x_{\beta}\x_{\gamma}^2 \x_{\kappa} \ot x_{\tau}
+q_{\beta\tau}q_{\gamma\beta}^{-2} (3)_{\widetilde{q}_{\beta\gamma}}\Bsj_8 \x_{\tau} \x_{\gamma}^3 \ot 1
\\ & \quad 
+q_{\beta\tau}q_{\beta\gamma}^2q_{\beta\kappa} \x_{\tau} \x_{\gamma}^2\x_{\kappa} \ot x_{\beta}.
\end{align*}

For \eqref{eq:diff-case11-btgee} we use \eqref{eq:diff-case11-btggge}, \eqref{eq:diff-case11-btee}, \eqref{eq:diff-case11-bgke}, Remark \ref{rem:differential-q-commute} and \eqref{eq:diff-case11-tgee}:
\begin{align*}
d(& \x_{\beta}\x_{\tau}\x_{\gamma}\x_{\eta}^2 \ot 1) = \x_{\beta}\x_{\tau}\x_{\gamma}\x_{\eta} \ot x_{\eta}
-s \big( 
-q_{\gamma\eta}^2 \x_{\beta}\x_{\tau}\x_{\eta} \ot x_{\eta}x_{\gamma}
+q_{\tau\gamma}q_{\tau\eta}^2 \x_{\beta}\x_{\gamma}\x_{\eta} \ot x_{\eta}x_{\tau}
\\ & \quad
+q_{\tau\gamma}q_{\tau\eta} \Bsj_4 \x_{\beta}\x_{\gamma}\x_{\eta} \ot x_{\kappa}x_{\gamma}
+q_{\tau\gamma}q_{\gamma\eta} \Bsj_4 \x_{\beta}\x_{\gamma}\x_{\kappa} \ot x_{\eta}x_{\gamma}
\\ & \quad
-q_{\beta\tau}q_{\beta\gamma}q_{\beta\eta}^2 \x_{\tau}\x_{\gamma} \x_{\eta} \ot x_{\eta}x_{\beta}
-q_{\eta\gamma}^{-1}(2)_{\widetilde{q}_{\gamma\eta}} \Bsj_4\Bsj_8 \x_{\gamma}^3 \ot x_{\eta}
\big)
\\ & = \x_{\beta}\x_{\tau}\x_{\gamma}\x_{\eta} \ot x_{\eta}
+q_{\gamma\eta}^2 \x_{\beta}\x_{\tau}\x_{\eta}^2 \ot x_{\gamma}
-q_{\tau\gamma}q_{\gamma\eta} \Bsj_4 \x_{\beta}\x_{\gamma}\x_{\kappa}\x_{\eta} \ot x_{\gamma}
-q_{\tau\gamma}q_{\tau\eta}^2 \x_{\beta}\x_{\gamma}\x_{\eta}^2 \ot x_{\tau}
\\ & \quad
+q_{\beta\tau}q_{\beta\gamma}q_{\beta\eta}^2 \x_{\tau}\x_{\gamma}\x_{\eta}^2 \ot x_{\beta}
+q_{\gamma\eta}^{-1}\coef{-\beta-\eta\gamma}{2} \Bsj_4\Bsj_8 \x_{\gamma}^3\x_{\eta} \ot 1.
\end{align*}

The proof of \eqref{eq:diff-case11-bdgge} is similar, using Remark \ref{rem:differential-q-commute}, \eqref{eq:diff-case11-bdge}, \eqref{eq:diff-case11-bgggk} and \eqref{eq:diff-case11-dgge}:
\begin{align*}
d(& \x_{\beta}\x_{\delta}\x_{\gamma}^2\x_{\eta}\ot 1) = 
\x_{\beta}\x_{\delta}\x_{\gamma}^2 \ot x_{\eta}
-s \big( 
q_{\gamma\eta} \x_{\beta}\x_{\delta}\x_{\gamma} \ot x_{\eta}x_{\gamma}
+q_{\delta\gamma}^2q_{\delta\eta} \x_{\beta}\x_{\gamma}^2 \ot x_{\eta}x_{\delta}
\\ & \quad
+q_{\delta\gamma}^2\Bsj_2 \x_{\beta}\x_{\gamma}^2 \ot x_{\kappa}x_{\gamma}x_{\mu}
+q_{\delta\gamma}^2\Bsj_3  \x_{\beta}\x_{\gamma}^2 \ot x_{\gamma}x_{\nu}
-q_{\beta\delta}q_{\beta\gamma}^2q_{\beta\eta} \x_{\delta}\x_{\gamma}^2 \ot x_{\eta}x_{\beta}
\big)
\\ & = 
\x_{\beta}\x_{\delta}\x_{\gamma}^2 \ot x_{\eta}
-q_{\gamma\eta} \x_{\beta}\x_{\delta}\x_{\gamma}\x_{\eta} \ot x_{\gamma}
-q_{\delta\gamma}^2q_{\delta\eta} \x_{\beta}\x_{\gamma}^2\x_{\eta} \ot x_{\delta}
-q_{\delta\gamma}^2\Bsj_2 \x_{\beta}\x_{\gamma}^2\x_{\kappa} \ot x_{\gamma}x_{\mu}
\\ & \quad
-q_{\delta\gamma}^2(3)_{\widetilde{q}_{\gamma\eta}} \Bsj_3  \x_{\beta}\x_{\gamma}^3 \ot x_{\nu}
+q_{\beta\delta}q_{\beta\gamma}^2q_{\beta\eta} \x_{\delta}\x_{\gamma}^2\x_{\eta} \ot x_{\beta}
-q_{\beta\gamma}q_{\gamma\beta}^{-1}q_{\delta\gamma}^2
\coef{\beta\eta}{3} \Bsj_2\Bsj_8 \x_{\gamma}^4 \ot x_{\mu}.
\end{align*}

For \eqref{eq:diff-case11-bdgee} we use \eqref{eq:diff-case11-bdge}, \eqref{eq:diff-case11-bdee}, Remark \ref{rem:differential-q-commute}, \eqref{eq:diff-case11-bgke} and \eqref{eq:diff-case11-dgee}:
\begin{align*}
d(& \x_{\beta}\x_{\delta}\x_{\gamma}\x_{\eta}^2\ot 1) = \x_{\beta}\x_{\delta}\x_{\gamma}\x_{\eta} \ot x_{\eta}
-s \big( 
-q_{\gamma\eta}^2 \x_{\beta}\x_{\delta}\x_{\eta} \ot x_{\eta}x_{\gamma}
+q_{\delta\gamma}q_{\delta\eta}^2 \x_{\beta}\x_{\gamma}\x_{\eta} \ot x_{\eta}x_{\delta}
\\ & \quad
+q_{\delta\gamma}q_{\delta\eta} \Bsj_2 \x_{\beta}\x_{\gamma}\x_{\eta} \ot x_{\kappa}x_{\gamma}x_{\mu}
+q_{\delta\gamma}q_{\delta\eta} \Bsj_3 \x_{\beta}\x_{\gamma}\x_{\eta} \ot x_{\gamma} x_{\nu}
\\ & \quad
+q_{\delta\gamma}q_{\mu\eta}q_{\gamma\eta} \Bsj_2 \x_{\beta}\x_{\gamma}\x_{\kappa} \ot x_{\eta}x_{\gamma}x_{\mu}
+q_{\delta\gamma}(2)_{\widetilde{q}_{\gamma\eta}}q_{\nu\eta} \Bsj_3  \x_{\beta}\x_{\gamma}^2 \ot x_{\eta}x_{\nu}
\\ & \quad
+q_{\delta\gamma}(2)_{\widetilde{q}_{\gamma\eta}} \Bsj_3\Bsj_7 \x_{\beta}\x_{\gamma}^2 \ot x_{\iota}x_{\kappa}x_{\gamma}
+q_{\beta\gamma}q_{\delta\gamma}q_{\mu\eta} (2)_{\widetilde{q}_{\gamma\eta}} \Bsj_2\Bsj_8 \x_{\gamma}^3 \ot x_{\eta}x_{\mu}
\\ & \quad
+q_{\beta\gamma}q_{\delta\gamma} (2)_{\widetilde{q}_{\gamma\eta}} \Bsj_2\Bsj_6\Bsj_8 \x_{\gamma}^3 \ot x_{\iota}x_{\gamma}
-q_{\beta\delta}q_{\beta\gamma}q_{\beta\eta}^2 \x_{\delta}\x_{\gamma}\x_{\eta} \ot x_{\eta}x_{\beta}
\big)
\\ & = \x_{\beta}\x_{\delta}\x_{\gamma}\x_{\eta} \ot x_{\eta}
+q_{\gamma\eta}^2 \x_{\beta}\x_{\delta}\x_{\eta}^2 \ot x_{\gamma}
-q_{\delta\gamma}(2)_{\widetilde{q}_{\gamma\eta}} \Bsj_3\Bsj_7 \x_{\beta}\x_{\gamma}^2\x_{\iota} \ot x_{\kappa}x_{\gamma}
\\ & \quad
-q_{\delta\gamma}(2)_{\widetilde{q}_{\gamma\eta}}q_{\nu\eta} \Bsj_3  \x_{\beta}\x_{\gamma}^2\x_{\eta} \ot x_{\nu}
-q_{\delta\gamma}q_{\mu\eta}q_{\gamma\eta} \Bsj_2 \x_{\beta}\x_{\gamma}\x_{\kappa}\x_{\eta} \ot x_{\gamma}x_{\mu}
\\ & \quad
-q_{\delta\gamma}q_{\delta\eta}^2 \x_{\beta}\x_{\gamma}\x_{\eta}^2 \ot x_{\delta}
+q_{\beta\delta}q_{\beta\gamma}q_{\beta\eta}^2 \x_{\delta}\x_{\gamma}\x_{\eta}^2 \ot x_{\beta}
\\ & \quad
-q_{\beta\gamma}q_{\delta\gamma} (2)_{\widetilde{q}_{\gamma\eta}} \Bsj_2\Bsj_6\Bsj_8 \x_{\gamma}^3\x_{\iota} \ot x_{\gamma}
-q_{\beta\gamma}q_{\delta\gamma}q_{\mu\eta} (2)_{\widetilde{q}_{\gamma\eta}} \Bsj_2\Bsj_8 \x_{\gamma}^3\x_{\eta} \ot x_{\mu}.
\end{align*}

We check \eqref{eq:diff-case11-bggkk} using \eqref{eq:diff-case11-bgggk}, \eqref{eq:diff-case11-bgkk} and Remark \ref{rem:differential-q-commute}:
\begin{align*}
d(& \x_{\beta}\x_{\gamma}^2\x_{\kappa}^2\ot 1) = \x_{\beta}\x_{\gamma}^2\x_{\kappa} \ot x_{\kappa}
-s \big( 
-q_{\gamma\kappa}^2 \x_{\beta} \x_{\gamma}\x_{\kappa}\ot x_{\kappa}x_{\gamma}
-q_{\beta\kappa}q_{\beta\gamma}^2 \Bsj_8 \x_{\gamma}^2\x_{\kappa} \ot x_{\gamma} 
\\ & \quad
-q_{\beta\kappa}^2q_{\beta\gamma}^2 \x_{\gamma}^2 \x_{\kappa} \ot x_{\kappa}x_{\beta}
-q_{\beta\gamma}^2 (3)_{\widetilde{q}_{\beta\gamma}} \Bsj_8 \x_{\gamma}^3 \ot x_{\kappa}
\big)
\\ & = \x_{\beta}\x_{\gamma}^2\x_{\kappa} \ot x_{\kappa}
+q_{\gamma\kappa}^2 \x_{\beta} \x_{\gamma}\x_{\kappa}^2 \ot x_{\gamma}
+q_{\beta\gamma}^2 (3)_{\widetilde{q}_{\beta\gamma}} \Bsj_8 \x_{\gamma}^3\x_{\kappa} \ot 1
+q_{\beta\kappa}^2q_{\beta\gamma}^2 \x_{\gamma}^2\x_{\kappa}^2 \ot x_{\beta}.
\end{align*}

For \eqref{eq:diff-case11-bggke} we use \eqref{eq:diff-case11-bgggk}, Remark \ref{rem:differential-q-commute} and \eqref{eq:diff-case11-bgke}:
\begin{align*}
d(& \x_{\beta}\x_{\gamma}^2\x_{\kappa}\x_{\eta}\ot 1) = \x_{\beta}\x_{\gamma}^2\x_{\kappa} \ot x_{\eta}
-s \big( 
q_{\kappa\eta} \x_{\beta}\x_{\gamma}^2 \ot x_{\eta}x_{\kappa}
- q_{\gamma\kappa}q_{\gamma\eta} \x_{\beta}\x_{\gamma}\x_{\kappa}\ot x_{\eta}x_{\gamma}
\\ & \quad 
- q_{\beta\kappa}q_{\beta\gamma}^2q_{\beta\eta} \x_{\gamma}^2\x_{\kappa} \ot x_{\eta}x_{\beta}
-q_{\beta\gamma}^2 (3)_{\widetilde{q}_{\beta\gamma}} \Bsj_8 \x_{\gamma}^3 \ot x_{\eta}
\big)
\\ & = \x_{\beta}\x_{\gamma}^2\x_{\kappa} \ot x_{\eta}
-q_{\kappa\eta} \x_{\beta}\x_{\gamma}^2\x_{\eta} \ot x_{\kappa}
+q_{\gamma\kappa}q_{\gamma\eta} \x_{\beta}\x_{\gamma}\x_{\kappa}\x_{\eta} \ot x_{\gamma}
\\ & \quad 
+q_{\beta\kappa}q_{\beta\gamma}^2q_{\beta\eta} \x_{\gamma}^2\x_{\kappa}\x_{\eta} \ot x_{\beta}
+q_{\beta\gamma}^2 (3)_{\widetilde{q}_{\beta\gamma}} \Bsj_8 \x_{\gamma}^3\x_{\eta} \ot 1.
\end{align*}

Now we compute \eqref{eq:diff-case11-tggke} using Remark \ref{rem:differential-q-commute}, \eqref{eq:diff-case11-tggge} and \eqref{eq:diff-case11-tgke},
\begin{align*}
d(& \x_{\tau}\x_{\gamma}^2\x_{\kappa}\x_{\eta}\ot 1) = \x_{\tau}\x_{\gamma}^2\x_{\kappa} \ot x_{\eta}
-s \big(  
q_{\kappa\eta} \x_{\tau}\x_{\gamma}^2 \ot x_{\eta}x_{\kappa}
-q_{\gamma\eta}^2 \x_{\tau}\x_{\gamma}\x_{\kappa} \ot x_{\eta}x_{\gamma}
\\ & \quad
-q_{\tau\gamma}^2q_{\tau\kappa}q_{\tau\eta} \x_{\gamma}^2\x_{\kappa} \ot x_{\eta}x_{\tau}
-q_{\tau\gamma}^2q_{\tau\kappa} \Bsj_4 \x_{\gamma}^2\x_{\kappa} \ot x_{\kappa}x_{\gamma}
\big)
\\ & = \x_{\tau}\x_{\gamma}^2\x_{\kappa} \ot x_{\eta}
-q_{\kappa\eta} \x_{\tau}\x_{\gamma}^2\x_{\eta} \ot x_{\kappa}
+q_{\gamma\kappa}q_{\gamma\eta} \x_{\tau}\x_{\gamma}\x_{\kappa}\x_{\eta} \ot x_{\gamma}
+q_{\tau\gamma}^2q_{\tau\kappa}(2)_{\widetilde{q}_{\kappa\eta}} \Bsj_4 \x_{\gamma}^2\x_{\kappa}^2 \ot x_{\gamma}
\\ & \quad
+q_{\tau\gamma}^2q_{\tau\kappa}q_{\tau\eta} \x_{\gamma}^2\x_{\kappa}\x_{\eta} \ot x_{\tau}.
\end{align*}

For \eqref{eq:diff-case11-tggee} we use \eqref{eq:diff-case11-tggge}, \eqref{eq:diff-case11-tgee} and Remark \ref{rem:differential-q-commute}:
\begin{align*}
d(& \x_{\tau}\x_{\gamma}^2\x_{\eta}^2\ot 1) = \x_{\tau}\x_{\gamma}^2\x_{\eta} \ot x_{\eta}
- s\big( 
-q_{\gamma\eta}^2 \x_{\tau}\x_{\gamma}\x_{\eta}\ot x_{\eta}x_{\gamma}
-q_{\tau\gamma}^2 q_{\tau\eta} \Bsj_4 \x_{\gamma}^2\x_{\eta} \ot x_{\kappa}x_{\gamma}
\\ & \quad
-q_{\tau\gamma}^2q_{\tau\eta}^2 \x_{\gamma}^2\x_{\eta} \ot x_{\eta}x_{\tau}
-q_{\tau\gamma}^2q_{\gamma\eta} \Bsj_4 \x_{\gamma}^2\x_{\kappa}\ot x_{\eta}x_{\gamma}
\big)
\\ & = \x_{\tau}\x_{\gamma}^2\x_{\eta} \ot x_{\eta}
+q_{\gamma\eta}^2 \x_{\tau}\x_{\gamma}\x_{\eta}^2 \ot x_{\gamma}
+q_{\tau\gamma}^2q_{\gamma\eta} \Bsj_4 \x_{\gamma}^2\x_{\kappa}\x_{\eta} \ot x_{\gamma}
+q_{\tau\gamma}^2q_{\tau\eta}^2 \x_{\gamma}^2\x_{\eta}^2 \ot x_{\tau}.
\end{align*}

Finally we compute \eqref{eq:diff-case11-dggee} using \eqref{eq:diff-case11-dgge}, \eqref{eq:diff-case11-dgee} and Remark \ref{rem:differential-q-commute}:
\begin{align*}
d(& \x_{\delta}\x_{\gamma}^2\x_{\eta}^2\ot 1) = \x_{\delta}\x_{\gamma}^2\x_{\eta} \ot x_{\eta}
-s \big( 
-q_{\gamma\eta}^2 \x_{\delta}\x_{\gamma}\x_{\eta} \ot x_{\eta}x_{\gamma}
-q_{\delta\gamma}^2q_{\mu\eta}q_{\gamma\eta} \Bsj_2 \x_{\gamma}^2\x_{\kappa} \ot x_{\eta}x_{\gamma}x_{\mu}
\\ & \quad
-q_{\delta\gamma}^2(3)_{\widetilde{q}_{\gamma\eta}} \Bsj_3\Bsj_7 \x_{\gamma}^3 \ot x_{\iota}x_{\kappa}x_{\gamma}
-q_{\delta\gamma}^2q_{\nu\eta} (3)_{\widetilde{q}_{\gamma\eta}} \Bsj_3 \x_{\gamma}^3 \ot x_{\eta}x_{\nu}
-q_{\delta\gamma}^2q_{\delta\eta} \Bsj_3 \x_{\gamma}^2\x_{\eta} \ot x_{\gamma} x_{\nu}
\\ & \quad
-q_{\delta\gamma}^2q_{\delta\eta} \Bsj_2 \x_{\gamma}^2\x_{\eta} \ot x_{\kappa}x_{\gamma}x_{\mu}
-q_{\delta\gamma}^2q_{\delta\eta}^2 \x_{\gamma}^2\x_{\eta} \ot x_{\eta}x_{\delta}
\big)
\\ & = \x_{\delta}\x_{\gamma}^2\x_{\eta} \ot x_{\eta}
+q_{\gamma\eta}^2 \x_{\delta}\x_{\gamma}\x_{\eta}^2 \ot x_{\gamma}
+q_{\delta\gamma}^2(3)_{\widetilde{q}_{\gamma\eta}} \Bsj_3\Bsj_7 \x_{\gamma}^3\x_{\iota} \ot x_{\kappa}x_{\gamma}
\\ & \quad
+q_{\delta\gamma}^2q_{\nu\eta} (3)_{\widetilde{q}_{\gamma\eta}} \Bsj_3 \x_{\gamma}^3\x_{\eta} \ot x_{\nu}
+q_{\delta\gamma}^2q_{\mu\eta}q_{\gamma\eta} \Bsj_2 \x_{\gamma}^2\x_{\kappa}\x_{\eta} \ot x_{\gamma}x_{\mu}
+q_{\delta\gamma}^2q_{\delta\eta}^2 \x_{\gamma}^2\x_{\eta}^2 \ot x_{\delta}.
\end{align*}

\medspace

Next we compute differentials of some $6$-chains:
\begin{align}
\label{eq:diff-case11-abdgge}&
\begin{aligned}
d( & \x_{\alpha}\x_{\beta}\x_{\delta}\x_{\gamma}^{2}\x_{\eta}\ot 1)  = 
\x_{\alpha}\x_{\beta} \x_{\delta} \x_{\gamma}^{2} \ot x_{\eta} 
-q_{\gamma\eta} \x_{\alpha}\x_{\beta}\x_{\delta}\x_{\gamma} \x_{\eta} \ot x_{\gamma}
-q_{\delta\gamma}^2(3)_{\widetilde{q}_{\gamma\eta}}\Bsj_3 \x_{\alpha}\x_{\beta} \x_{\gamma}^3 \ot x_{\nu}
\\ & \quad
-q_{\delta\gamma}^2\Bsj_2 \x_{\alpha}\x_{\beta}\x_{\gamma}^2 \x_{\kappa} \ot  x_{\gamma}x_{\mu}
-q_{\delta\gamma}^2q_{\delta\eta} \x_{\alpha}\x_{\beta}\x_{\gamma}^2\x_{\eta} \ot x_{\delta}
+q_{\beta\delta}q_{\beta\gamma}^2q_{\beta\eta} \x_{\alpha} \x_{\delta}\x_{\gamma}^2\x_{\eta} \ot x_{\beta}
\\ & \quad
-q_{\delta\gamma}^2 q_{\gamma\beta}^{-2} \coef{\beta\eta\gamma}{3}
\Bsj_2\Bsj_8 \x_{\alpha}\x_{\gamma}^4 \ot x_{\mu}
- q_{\alpha\beta}q_{\beta\gamma}^2q_{\beta\eta} q_{\tau\gamma}^2 \Bsj_1\Bsj_4 \x_{\beta}\x_{\tau}\x_{\gamma}^2 \x_{\kappa} \ot x_{\gamma}x_{\beta}
\\ & \quad
- q_{\alpha\beta}q_{\beta\gamma}^2q_{\beta\eta} q_{\tau\gamma}^2 q_{\tau\eta} \Bsj_1 \x_{\beta}\x_{\tau}\x_{\gamma}^2\x_{\eta} \ot  x_{\tau}x_{\beta}
-q_{\alpha\beta}q_{\alpha\delta}q_{\alpha\gamma}^2q_{\alpha\eta} \x_{\beta}\x_{\delta}\x_{\gamma}^2\x_{\eta} \ot x_{\alpha}
\\ & \quad
+ q_{\alpha\beta}q_{\beta\gamma}^2q_{\beta\eta}q_{\beta\tau} q_{\gamma\beta}^{-2} (3)_{\widetilde{q}_{\beta\gamma}} q_{\tau\gamma}^2 \Bsj_1\Bsj_4\Bsj_8 \x_{\tau} \x_{\gamma}^4 \ot x_{\beta},
\end{aligned}
\\
\label{eq:diff-case11-abdgee}&
\begin{aligned}
d( & \x_{\alpha}\x_{\beta}\x_{\delta}\x_{\gamma}\x_{\eta}^2 \ot 1)  = \x_{\alpha}\x_{\beta} \x_{\delta} \x_{\gamma}\x_{\eta} \ot x_{\eta}
+q_{\gamma\eta}^2 \x_{\alpha}\x_{\beta}\x_{\delta}\x_{\eta}^2 \ot x_{\gamma}
-q_{\delta\gamma}(2)_{\widetilde{q}_{\gamma\eta}}q_{\nu\eta} \Bsj_3 \x_{\alpha}\x_{\beta}\x_{\gamma}^2\x_{\eta} \ot x_{\nu}
\\ & \quad
-q_{\delta\gamma}(2)_{\widetilde{q}_{\gamma\eta}} \Bsj_3\Bsj_7  \x_{\alpha}\x_{\beta}\x_{\gamma}^2\x_{\iota} \ot x_{\kappa}x_{\gamma}
-q_{\delta\gamma}q_{\gamma\eta}q_{\mu\eta}\Bsj_2 \x_{\alpha}\x_{\beta}\x_{\gamma}\x_{\kappa}\x_{\eta} \ot x_{\gamma}x_{\mu}
\\ & \quad
-q_{\delta\gamma}q_{\delta\eta}^2 \x_{\alpha}\x_{\beta}\x_{\gamma}\x_{\eta}^2 \ot x_{\delta}
+q_{\alpha\beta}q_{\beta\gamma}q_{\beta\eta}^2 q_{\tau\gamma}q_{\gamma\eta} \Bsj_1\Bsj_4 \x_{\beta}\x_{\tau}\x_{\gamma}\x_{\kappa}\x_{\eta} \ot x_{\gamma}x_{\beta}
\\ & \quad
+q_{\delta\gamma}q_{\gamma\beta}^{-1} \coef{\beta\eta\gamma}{2} \Bsj_2\Bsj_6\Bsj_8 \x_{\alpha}\x_{\gamma}^3\x_{\iota} \ot x_{\gamma}
+q_{\delta\gamma}q_{\gamma\kappa} \coeff{\alpha\beta\eta\gamma}{3} \Bsj_2\Bsj_6\Bsj_8\Bsj_9 \x_{\gamma}^5 \ot 1
\\ & \quad
+q_{\delta\gamma}q_{\gamma\beta}^{-1}q_{\mu\eta}  \coef{\beta\eta\gamma}{2} \Bsj_2\Bsj_8 \x_{\alpha}\x_{\gamma}^3\x_{\eta} \ot x_{\mu}
+q_{\beta\delta}q_{\beta\gamma}q_{\beta\eta}^2 \x_{\alpha} \x_{\delta}\x_{\gamma}\x_{\eta}^2 \ot x_{\beta}
\\ & \quad
-q_{\alpha\beta}q_{\beta\gamma}q_{\beta\eta}^2q_{\tau\gamma}q_{\tau\eta}^2 \Bsj_1 \x_{\beta}\x_{\tau}\x_{\gamma}\x_{\eta}^2 \ot x_{\tau} x_{\beta}
-q_{\alpha\beta}q_{\alpha\delta}q_{\alpha\gamma}q_{\alpha\eta}^2 \x_{\beta}\x_{\delta}\x_{\gamma}\x_{\eta}^2 \ot x_{\alpha}
\\ & \quad
-q_{\alpha\beta}q_{\beta\tau}q_{\beta\eta}^2 q_{\eta\gamma}^{-1} q_{\gamma\beta}^{-1} \coef{\beta\eta\gamma}{2} \Bsj_1 \Bsj_4 \Bsj_8 \x_{\tau}\x_{\gamma}^3\x_{\eta} \ot x_{\beta},
\end{aligned}
\\
\label{eq:diff-case11-abggke}&
\begin{aligned}
d( & \x_{\alpha}\x_{\beta}\x_{\gamma}^{2}\x_{\kappa}\x_{\eta}\ot 1)  = \x_{\alpha}\x_{\beta}\x_{\gamma}^{2}\x_{\kappa} \ot x_{\eta}
-q_{\kappa\eta} \x_{\alpha}\x_{\beta}\x_{\gamma}^2\x_{\eta} \ot x_{\kappa}
+q_{\gamma\kappa}q_{\gamma\eta} \x_{\alpha}\x_{\beta}\x_{\gamma}\x_{\kappa}\x_{\eta} \ot x_{\gamma}
\\ & \quad
+q_{\gamma\beta}^{-2}(3)_{\widetilde{q}_{\beta\gamma}} \Bsj_8 \x_{\alpha}\x_{\gamma}^3\x_{\eta} \ot 1
+q_{\beta\gamma}^2q_{\beta\kappa}q_{\beta\eta} \x_{\alpha}\x_{\gamma}^2\x_{\kappa}\x_{\eta} \ot  x_{\beta}
-q_{\alpha\beta}q_{\alpha\gamma}^2q_{\alpha\kappa}q_{\alpha\eta} \x_{\beta}\x_{\gamma}^2\x_{\kappa}\x_{\eta} \ot x_{\alpha},
\end{aligned}
\\
\label{eq:diff-case11-adggee}&
\begin{aligned}
d( & \x_{\alpha}\x_{\delta}\x_{\gamma}^{2}\x_{\eta}^2 \ot 1) = \x_{\alpha}\x_{\delta}\x_{\gamma}^{2}\x_{\eta} \ot x_{\eta}
+q_{\gamma\eta}^2 \x_{\alpha}\x_{\delta}\x_{\gamma}\x_{\eta}^2 \ot x_{\gamma}
+q_{\delta\gamma}^2(2)_{\widetilde{q}_{\gamma\eta}} \Bsj_3\Bsj_7 \x_{\alpha}\x_{\gamma}^3\x_{\iota} \ot x_{\kappa}x_{\gamma}
\\ & \quad
-q_{\delta\gamma}^2\coef{\alpha\eta\gamma}{3} \Bsj_3\Bsj_7\Bsj_9 \x_{\gamma}^5 \ot x_{\kappa}
+q_{\delta\gamma}^2q_{\nu\eta}(2)_{\widetilde{q}_{\gamma\eta}} \Bsj_3 \x_{\alpha}\x_{\gamma}^3\x_{\eta} \ot x_{\nu}
+q_{\delta\gamma}^2q_{\delta\eta}^2 \x_{\alpha}\x_{\gamma}^2\x_{\eta}^2 \ot x_{\delta}
\\ & \quad
+q_{\delta\gamma}^2q_{\mu\eta}q_{\gamma\eta} \Bsj_2 \x_{\alpha}\x_{\gamma}^2\x_{\kappa}\x_{\eta} \ot x_{\gamma}x_{\mu}
-q_{\beta\eta}^2q_{\tau\gamma}^2q_{\beta\gamma}^2q_{\gamma\eta} \Bsj_1\Bsj_4 \x_{\tau}\x_{\gamma}^2\x_{\kappa}\x_{\eta} \ot x_{\gamma}x_{\beta}
\\ & \quad
-q_{\beta\eta}^2q_{\tau\gamma}^2 q_{\beta\gamma}^2q_{\tau\eta}^2 \Bsj_1 \x_{\tau}\x_{\gamma}^2\x_{\eta}^2 \ot x_{\tau}x_{\beta}
-q_{\alpha\delta}q_{\alpha\gamma}^2q_{\alpha\eta}^2 \x_{\delta}\x_{\gamma}^2\x_{\eta}^2 \ot x_{\alpha},
\end{aligned}
\\
\label{eq:diff-case11-btggke}&
\begin{aligned}
d( & \x_{\beta}\x_{\tau}\x_{\gamma}^{2}\x_{\kappa}\x_{\eta}\ot 1) = \x_{\beta}\x_{\tau}\x_{\gamma}^{2}\x_{\kappa} \ot x_{\eta}
-q_{\kappa\eta} \x_{\beta}\x_{\tau}\x_{\gamma}^2\x_{\eta} \ot x_{\kappa}
+q_{\gamma\kappa}q_{\gamma\eta} \x_{\beta}\x_{\tau}\x_{\gamma}\x_{\kappa}\x_{\eta} \ot x_{\gamma}
\\ & \quad
+q_{\tau\gamma}^2q_{\tau\kappa}(2)_{\widetilde{q}_{\kappa\eta}} \Bsj_4 \x_{\beta}\x_{\gamma}^2\x_{\kappa}^2 \ot x_{\gamma}
+q_{\tau\gamma}^2q_{\tau\kappa}q_{\tau\eta} \x_{\beta}\x_{\gamma}^2\x_{\kappa}\x_{\eta} \ot x_{\tau}
-q_{\beta\tau}q_{\gamma\beta}^{-2} (3)_{\widetilde{q}_{\beta\gamma}}\Bsj_8 \x_{\tau}\x_{\gamma}^3\x_{\eta} \ot 1
\\ & \quad 
-q_{\beta\tau}q_{\beta\gamma}^2q_{\beta\kappa}q_{\beta\eta} \x_{\tau} \x_{\gamma}^2\x_{\kappa}\x_{\eta} \ot x_{\beta}
-q_{\gamma\eta}^2q_{\kappa\eta} \coef{-\eta,\beta,\gamma}{2} \Bsj_4 \Bsj_8 \x_{\gamma}^{4}\x_{\kappa} \ot 1,
\end{aligned}
\\
\label{eq:diff-case11-btggee}&
\begin{aligned}
d( & \x_{\beta}\x_{\tau}\x_{\gamma}^{2}\x_{\eta}^2\ot 1) =  \x_{\beta}\x_{\tau}\x_{\gamma}^{2}\x_{\eta} \ot x_{\eta}
+q_{\gamma\eta}^2 \x_{\beta}\x_{\tau}\x_{\gamma}\x_{\eta}^2 \ot x_{\gamma}
+q_{\tau\gamma}^2q_{\gamma\eta} \Bsj_4 \x_{\beta}\x_{\gamma}^2\x_{\kappa}\x_{\eta} \ot x_{\gamma}
\\ & \quad
+q_{\tau\gamma}^2q_{\tau\eta}^2 \x_{\beta}\x_{\gamma}^2\x_{\eta}^2 \ot x_{\tau}
-q_{\beta\tau}q_{\beta\gamma}^2q_{\beta\eta}^2 \x_{\tau}\x_{\gamma}^2\x_{\eta}^2 \ot x_{\beta}
+q_{\gamma\eta}^2 \coef{-\eta,\beta,\gamma}{2} \Bsj_4 \Bsj_8 \x_{\gamma}^4 \ot x_{\eta}.
\end{aligned}
\\
\label{eq:diff-case11-bdggee}&
\begin{aligned}
d( & \x_{\beta}\x_{\delta}\x_{\gamma}^{2}\x_{\eta}^2\ot 1) = \x_{\beta}\x_{\delta}\x_{\gamma}^{2}\x_{\eta} \ot x_{\eta}
+q_{\gamma\eta}^2 \x_{\beta}\x_{\delta}\x_{\gamma}\x_{\eta}^2 \ot x_{\gamma}
+q_{\delta\gamma}^2(3)_{\widetilde{q}_{\gamma\eta}} \Bsj_3\Bsj_7 \x_{\beta}\x_{\gamma}^3\x_{\iota} \ot x_{\kappa}x_{\gamma}
\\ & \quad
+q_{\delta\gamma}^2q_{\nu\eta}(3)_{\widetilde{q}_{\gamma\eta}} \Bsj_3 \x_{\beta}\x_{\gamma}^3\x_{\eta} \ot x_{\nu}
+q_{\delta\gamma}^2q_{\mu\eta}q_{\gamma\eta} \Bsj_2 \x_{\beta}\x_{\gamma}^2\x_{\kappa}\x_{\eta} \ot x_{\gamma}x_{\mu}
\\ & \quad
+q_{\delta\gamma}^2q_{\delta\eta}^2 \x_{\beta}\x_{\gamma}^2\x_{\eta}^2 \ot x_{\delta}
-q_{\beta\delta}q_{\beta\gamma}^2q_{\beta\eta}^2 \x_{\delta}\x_{\gamma}^2\x_{\eta}^2 \ot x_{\beta}
\\ & \quad
-q_{\beta\gamma}q_{\gamma\kappa}q_{\delta\gamma}^2 \coef{\beta\eta}{3} \Bsj_2\Bsj_6\Bsj_8 \x_{\gamma}^4\x_{\iota} \ot x_{\gamma}
-q_{\beta\gamma}q_{\gamma\kappa}q_{\delta\gamma}^2 \coef{\beta\eta}{3}q_{\mu\eta} \Bsj_2\Bsj_8 \x_{\gamma}^4\x_{\eta} \ot x_{\mu}.
\end{aligned}
\end{align}

First we deal with \eqref{eq:diff-case11-abdgge}: using \eqref{eq:diff-case11-abdgg}, Remark \ref{rem:differential-q-commute}, \eqref{eq:diff-case11-abdge}, \eqref{eq:diff-case11-abggk}, \eqref{eq:diff-case11-adgge}, \eqref{eq:diff-case11-btggk}, \eqref{eq:diff-case11-btggge} and \eqref{eq:diff-case11-bdgge}:
\begin{align*}
d(& \x_{\alpha} \x_{\beta} \x_{\delta} \x_{\gamma}^{2} \x_{\eta} 
\ot 1)  = \x_{\alpha}\x_{\beta} \x_{\delta} \x_{\gamma}^{2} \ot x_{\eta} - s \big( 
q_{\gamma\eta} \x_{\alpha}\x_{\beta} \x_{\delta} \x_{\gamma} \ot x_{\eta}x_{\gamma}
+q_{\delta\gamma}^2 \x_{\alpha}\x_{\beta}\x_{\gamma}^2 \ot q_{\delta\eta} x_{\eta}x_{\delta}
\\ & \quad
+q_{\delta\gamma}^2 \x_{\alpha}\x_{\beta}\x_{\gamma}^2 \ot \Bsj_2 x_{\kappa}x_{\gamma}x_{\mu}
+q_{\delta\gamma}^2 \x_{\alpha}\x_{\beta}\x_{\gamma}^2 \ot \Bsj_3 x_{\gamma} x_{\nu}
-q_{\beta\delta}q_{\beta\gamma}^2q_{\beta\eta} \x_{\alpha} \x_{\delta}\x_{\gamma}^2 \ot x_{\eta}x_{\beta}
\\ & \quad
+ q_{\alpha\beta}q_{\beta\gamma}^2q_{\beta\eta} q_{\tau\gamma}^2 \Bsj_1 \x_{\beta}\x_{\tau}\x_{\gamma}^2 \ot (q_{\tau\eta} x_{\eta}x_{\tau} +\Bsj_4 x_{\kappa}x_{\gamma})x_{\beta}
+q_{\alpha\beta}q_{\alpha\delta}q_{\alpha\gamma}^2q_{\alpha\eta} \x_{\beta}\x_{\delta}\x_{\gamma}^2 \ot x_{\eta}x_{\alpha} \big)
\\ & = 
\x_{\alpha}\x_{\beta} \x_{\delta} \x_{\gamma}^{2} \ot x_{\eta} 
-q_{\gamma\eta} \x_{\alpha}\x_{\beta}\x_{\delta}\x_{\gamma} \x_{\eta} \ot x_{\gamma}
-q_{\delta\gamma}^2(3)_{\widetilde{q}_{\gamma\eta}}\Bsj_3 \x_{\alpha}\x_{\beta} \x_{\gamma}^3 \ot x_{\nu}
\\ & \quad
-q_{\delta\gamma}^2\Bsj_2 \x_{\alpha}\x_{\beta}\x_{\gamma}^2 \x_{\kappa} \ot  x_{\gamma}x_{\mu}
-q_{\delta\gamma}^2q_{\delta\eta} \x_{\alpha}\x_{\beta}\x_{\gamma}^2\x_{\eta} \ot x_{\delta}
+q_{\beta\delta}q_{\beta\gamma}^2q_{\beta\eta} \x_{\alpha} \x_{\delta}\x_{\gamma}^2\x_{\eta} \ot x_{\beta}
\\ & \quad
-q_{\delta\gamma}^2 q_{\gamma\beta}^{-2} \coef{\beta\eta\gamma}{3}
\Bsj_2\Bsj_8 \x_{\alpha}\x_{\gamma}^4 \ot x_{\mu}
- q_{\alpha\beta}q_{\beta\gamma}^2q_{\beta\eta} q_{\tau\gamma}^2 \Bsj_1\Bsj_4 \x_{\beta}\x_{\tau}\x_{\gamma}^2 \x_{\kappa} \ot x_{\gamma}x_{\beta}
\\ & \quad
- q_{\alpha\beta}q_{\beta\gamma}^2q_{\beta\eta} q_{\tau\gamma}^2 q_{\tau\eta} \Bsj_1 \x_{\beta}\x_{\tau}\x_{\gamma}^2\x_{\eta} \ot  x_{\tau}x_{\beta}
-q_{\alpha\beta}q_{\alpha\delta}q_{\alpha\gamma}^2q_{\alpha\eta} \x_{\beta}\x_{\delta}\x_{\gamma}^2\x_{\eta} \ot x_{\alpha}
\\ & \quad
+ q_{\alpha\beta}q_{\beta\gamma}^2q_{\beta\eta}q_{\beta\tau} q_{\gamma\beta}^{-2} (3)_{\widetilde{q}_{\beta\gamma}} q_{\tau\gamma}^2 \Bsj_1\Bsj_4\Bsj_8 \x_{\tau} \x_{\gamma}^4 \ot x_{\beta}.
\end{align*}

Next we compute \eqref{eq:diff-case11-abdgee}: using \eqref{eq:diff-case11-abdge}, \eqref{eq:diff-case11-abdee}, Remark \ref{rem:differential-q-commute}, \eqref{eq:diff-case11-abgke}, \eqref{eq:diff-case11-adgee}, \eqref{eq:diff-case11-agggi}, \eqref{eq:diff-case11-btgke}, \eqref{eq:diff-case11-btgee}, \eqref{eq:diff-case11-bdgee} and \eqref{eq:diff-case11-tggge}:
\begin{align*}
d(& \x_{\alpha}\x_{\beta}\x_{\delta}\x_{\gamma}\x_{\eta}^{2}\ot 1)  = 
\x_{\alpha}\x_{\beta} \x_{\delta} \x_{\gamma}\x_{\eta} \ot x_{\eta}
- s \big(
q_{\delta\gamma}(2)_{\widetilde{q}_{\gamma\eta}}\Bsj_3 \x_{\alpha}\x_{\beta}\x_{\gamma}^2 \ot (q_{\nu\eta} x_{\eta}x_{\nu} + \Bsj_7 x_{\iota}x_{\kappa}x_{\gamma})
\\ & \quad
-q_{\gamma\eta}^2 \x_{\alpha}\x_{\beta}\x_{\delta}\x_{\eta} \ot x_{\eta}x_{\gamma}
+q_{\delta\gamma}q_{\delta\eta} \x_{\alpha}\x_{\beta}\x_{\gamma}\x_{\eta} \ot (q_{\delta\eta} x_{\eta}x_{\delta} + \Bsj_2 x_{\kappa}x_{\gamma}x_{\mu} + \Bsj_3 x_{\gamma} x_{\nu})
\\ & \quad
+q_{\delta\gamma}q_{\gamma\eta}q_{\mu\eta}\Bsj_2 \x_{\alpha}\x_{\beta}\x_{\gamma} \x_{\kappa} \ot x_{\eta}x_{\gamma}x_{\mu}
-q_{\beta\delta}q_{\beta\gamma}q_{\beta\eta}^2 \x_{\alpha} \x_{\delta}\x_{\gamma}\x_{\eta} \ot x_{\eta}x_{\beta}
\\ & \quad
+q_{\alpha\beta}q_{\beta\gamma}q_{\beta\eta}^2q_{\tau\gamma}q_{\tau\eta} \Bsj_1 \x_{\beta}\x_{\tau}\x_{\gamma}\x_{\eta} \ot (q_{\tau\eta} x_{\eta}x_{\tau} +\Bsj_4 x_{\kappa}x_{\gamma})x_{\beta}
\\ & \quad
-q_{\delta\gamma}q_{\gamma\beta}^{-1} \coef{\beta\eta\gamma}{2} \Bsj_2\Bsj_8 \x_{\alpha}\x_{\gamma}^3 \ot x_{\mu}x_{\eta}
+ q_{\alpha\beta}q_{\beta\gamma}q_{\beta\eta}^2 q_{\tau\gamma}q_{\gamma\eta} \Bsj_1\Bsj_4 \x_{\beta}\x_{\tau}\x_{\gamma}\x_{\kappa} \ot x_{\eta}x_{\gamma}x_{\beta}
\\ & \quad
+q_{\alpha\beta}q_{\alpha\delta}q_{\alpha\gamma}q_{\alpha\eta}^2 \x_{\beta}\x_{\delta}\x_{\gamma}\x_{\eta} \ot x_{\eta}x_{\alpha}
+q_{\alpha\beta}q_{\beta\tau}q_{\beta\eta}^2 q_{\eta\gamma}^{-1} q_{\gamma\beta}^{-1} \coef{\beta\eta\gamma}{2} \Bsj_1 \Bsj_4 \Bsj_8 \x_{\tau}\x_{\gamma}^3 \ot x_{\eta}x_{\beta}
\big)
\\ & = \x_{\alpha}\x_{\beta} \x_{\delta} \x_{\gamma}\x_{\eta} \ot x_{\eta}
+q_{\gamma\eta}^2 \x_{\alpha}\x_{\beta}\x_{\delta}\x_{\eta}^2 \ot x_{\gamma}
-q_{\delta\gamma}(2)_{\widetilde{q}_{\gamma\eta}} \Bsj_3\Bsj_7  \x_{\alpha}\x_{\beta}\x_{\gamma}^2\x_{\iota} \ot x_{\kappa}x_{\gamma}
\\ & \quad
-q_{\delta\gamma}(2)_{\widetilde{q}_{\gamma\eta}}q_{\nu\eta} \Bsj_3 \x_{\alpha}\x_{\beta}\x_{\gamma}^2\x_{\eta} \ot x_{\nu}
-q_{\delta\gamma}q_{\gamma\eta}q_{\mu\eta}\Bsj_2 \x_{\alpha}\x_{\beta}\x_{\gamma}\x_{\kappa}\x_{\eta} \ot x_{\gamma}x_{\mu}
\\ & \quad
-q_{\delta\gamma}q_{\delta\eta}^2 \x_{\alpha}\x_{\beta}\x_{\gamma}\x_{\eta}^2 \ot x_{\delta}
+q_{\alpha\beta}q_{\beta\gamma}q_{\beta\eta}^2 q_{\tau\gamma}q_{\gamma\eta} \Bsj_1\Bsj_4 \x_{\beta}\x_{\tau}\x_{\gamma}\x_{\kappa}\x_{\eta} \ot x_{\gamma}x_{\beta}
\\ & \quad
+q_{\delta\gamma}q_{\gamma\beta}^{-1} \coef{\beta\eta\gamma}{2} \Bsj_2\Bsj_6\Bsj_8 \x_{\alpha}\x_{\gamma}^3\x_{\iota} \ot x_{\gamma}
+q_{\delta\gamma}q_{\gamma\kappa} \coeff{\alpha\beta\eta\gamma}{3} \Bsj_2\Bsj_6\Bsj_8\Bsj_9 \x_{\gamma}^5 \ot 1
+q_{\delta\gamma}q_{\gamma\beta}^{-1}q_{\mu\eta}  \coef{\beta\eta\gamma}{2} \Bsj_2\Bsj_8 \x_{\alpha}\x_{\gamma}^3\x_{\eta} \ot x_{\mu}
\\ & \quad
+q_{\beta\delta}q_{\beta\gamma}q_{\beta\eta}^2 \x_{\alpha} \x_{\delta}\x_{\gamma}\x_{\eta}^2 \ot x_{\beta}
-q_{\alpha\beta}q_{\beta\gamma}q_{\beta\eta}^2q_{\tau\gamma}q_{\tau\eta}^2 \Bsj_1 \x_{\beta}\x_{\tau}\x_{\gamma}\x_{\eta}^2 \ot x_{\tau} x_{\beta}
\\ & \quad
-q_{\alpha\beta}q_{\alpha\delta}q_{\alpha\gamma}q_{\alpha\eta}^2 \x_{\beta}\x_{\delta}\x_{\gamma}\x_{\eta}^2 \ot x_{\alpha}
-q_{\alpha\beta}q_{\beta\tau}q_{\beta\eta}^2 q_{\eta\gamma}^{-1} q_{\gamma\beta}^{-1} \coef{\beta\eta\gamma}{2} \Bsj_1 \Bsj_4 \Bsj_8 \x_{\tau}\x_{\gamma}^3\x_{\eta} \ot x_{\beta}.
\end{align*}

For \eqref{eq:diff-case11-abggke} we use \eqref{eq:diff-case11-abggk}, Remark \ref{rem:differential-q-commute}, \eqref{eq:diff-case11-abgke} and \eqref{eq:diff-case11-bggke}:
\begin{align*}
d( & \x_{\alpha}\x_{\beta}\x_{\gamma}^{2}\x_{\kappa}\x_{\eta}\ot 1)  = \x_{\alpha}\x_{\beta}\x_{\gamma}^{2}\x_{\kappa} \ot x_{\eta}
-s \big( 
q_{\kappa\eta} \x_{\alpha}\x_{\beta}\x_{\gamma}^2 \ot x_{\eta}x_{\kappa}
-q_{\gamma\kappa}q_{\gamma\eta} \x_{\alpha}\x_{\beta}\x_{\gamma}\x_{\kappa} \ot x_{\eta}x_{\gamma}
\\ & \quad
-q_{\beta\gamma}^2q_{\beta\kappa}q_{\beta\eta} \x_{\alpha}\x_{\gamma}^2\x_{\kappa} \ot  x_{\eta}x_{\beta}
-q_{\gamma\beta}^{-2}(3)_{\widetilde{q}_{\beta\gamma}} \Bsj_8 \x_{\alpha}\x_{\gamma}^3 \ot x_{\eta}
+q_{\alpha\beta}q_{\alpha\gamma}^2q_{\alpha\kappa}q_{\alpha\eta} \x_{\beta}\x_{\gamma}^2\x_{\kappa} \ot x_{\eta}x_{\alpha}
\big)
\\ & = \x_{\alpha}\x_{\beta}\x_{\gamma}^{2}\x_{\kappa} \ot x_{\eta}
-q_{\kappa\eta} \x_{\alpha}\x_{\beta}\x_{\gamma}^2\x_{\eta} \ot x_{\kappa}
+q_{\gamma\kappa}q_{\gamma\eta} \x_{\alpha}\x_{\beta}\x_{\gamma}\x_{\kappa}\x_{\eta} \ot x_{\gamma}
\\ & \quad
+q_{\gamma\beta}^{-2}(3)_{\widetilde{q}_{\beta\gamma}} \Bsj_8 \x_{\alpha}\x_{\gamma}^3\x_{\eta} \ot 1
+q_{\beta\gamma}^2q_{\beta\kappa}q_{\beta\eta} \x_{\alpha}\x_{\gamma}^2\x_{\kappa}\x_{\eta} \ot  x_{\beta}
-q_{\alpha\beta}q_{\alpha\gamma}^2q_{\alpha\kappa}q_{\alpha\eta} \x_{\beta}\x_{\gamma}^2\x_{\kappa}\x_{\eta} \ot x_{\alpha}.
\end{align*}

For \eqref{eq:diff-case11-adggee} we use \eqref{eq:diff-case11-adgge}, \eqref{eq:diff-case11-adgee}, \eqref{eq:diff-case11-agggi}, \eqref{eq:diff-case11-tggke}, \eqref{eq:diff-case11-tggee} and \eqref{eq:diff-case11-dggee}:
\begin{align*}
d( & \x_{\alpha}\x_{\delta}\x_{\gamma}^{2}\x_{\eta}^2 \ot 1) = \x_{\alpha}\x_{\delta}\x_{\gamma}^{2}\x_{\eta} \ot x_{\eta}
-s \big( 
-q_{\gamma\eta}^2 \x_{\alpha}\x_{\delta} \x_{\gamma}\x_{\eta} \ot x_{\eta}x_{\gamma}
-q_{\delta\gamma}^2q_{\nu\eta}(2)_{\widetilde{q}_{\gamma\eta}} \Bsj_3 \x_{\alpha}\x_{\gamma}^3 \ot x_{\eta}x_{\nu}
\\ & \quad
-q_{\delta\gamma}^2(2)_{\widetilde{q}_{\gamma\eta}} \Bsj_3\Bsj_7 \x_{\alpha}\x_{\gamma}^3 \ot x_{\iota}x_{\kappa}x_{\gamma}
-q_{\delta\gamma}^2q_{\mu\eta}q_{\gamma\eta} \Bsj_2 \x_{\alpha}\x_{\gamma}^2\x_{\kappa} \ot x_{\eta}x_{\gamma}x_{\mu}
-q_{\delta\gamma}^2q_{\delta\eta}^2 \x_{\alpha}\x_{\gamma}^2\x_{\eta} \ot x_{\eta}x_{\delta}
\\ & \quad
-q_{\delta\gamma}^2q_{\delta\eta} \Bsj_2 \x_{\alpha}\x_{\gamma}^2\x_{\eta} \ot x_{\kappa}x_{\gamma}x_{\mu}
-q_{\delta\gamma}^2q_{\delta\eta} \Bsj_3 \x_{\alpha}\x_{\gamma}^2\x_{\eta} \ot x_{\gamma} x_{\nu}
+q_{\alpha\delta}q_{\alpha\gamma}^2q_{\alpha\eta}^2 \x_{\delta}\x_{\gamma}^2\x_{\eta} \ot x_{\eta}x_{\alpha}
\\ & \quad
+q_{\beta\eta}^2q_{\tau\gamma}^2 q_{\beta\gamma}^2q_{\tau\eta}^2 \Bsj_1 \x_{\tau}\x_{\gamma}^2\x_{\eta} \ot x_{\eta}x_{\tau}x_{\beta}
+q_{\beta\eta}^2q_{\tau\gamma}^2 q_{\beta\gamma}^2q_{\tau\eta} \Bsj_1\Bsj_4 \x_{\tau}\x_{\gamma}^2\x_{\eta} \ot x_{\kappa}x_{\gamma}x_{\beta}
\\ & \quad
+q_{\beta\eta}^2q_{\tau\gamma}^2q_{\beta\gamma}^2q_{\gamma\eta} \Bsj_1\Bsj_4 \x_{\tau}\x_{\gamma}^2\x_{\kappa} \ot x_{\eta}x_{\gamma}x_{\beta}
\big)
\\ & = \x_{\alpha}\x_{\delta}\x_{\gamma}^{2}\x_{\eta} \ot x_{\eta}
+q_{\gamma\eta}^2 \x_{\alpha}\x_{\delta}\x_{\gamma}\x_{\eta}^2 \ot x_{\gamma}
+q_{\delta\gamma}^2(2)_{\widetilde{q}_{\gamma\eta}} \Bsj_3\Bsj_7 \x_{\alpha}\x_{\gamma}^3\x_{\iota} \ot x_{\kappa}x_{\gamma}
\\ & \quad
-q_{\delta\gamma}^2\coef{\alpha\eta\gamma}{3} \Bsj_3\Bsj_7\Bsj_9 \x_{\gamma}^5 \ot x_{\kappa}
+q_{\delta\gamma}^2q_{\nu\eta}(2)_{\widetilde{q}_{\gamma\eta}} \Bsj_3 \x_{\alpha}\x_{\gamma}^3\x_{\eta} \ot x_{\nu}
+q_{\delta\gamma}^2q_{\mu\eta}q_{\gamma\eta} \Bsj_2 \x_{\alpha}\x_{\gamma}^2\x_{\kappa}\x_{\eta} \ot x_{\gamma}x_{\mu}
\\ & \quad
+q_{\delta\gamma}^2q_{\delta\eta}^2 \x_{\alpha}\x_{\gamma}^2\x_{\eta}^2 \ot x_{\delta}
-q_{\beta\eta}^2q_{\tau\gamma}^2q_{\beta\gamma}^2q_{\gamma\eta} \Bsj_1\Bsj_4 \x_{\tau}\x_{\gamma}^2\x_{\kappa}\x_{\eta} \ot x_{\gamma}x_{\beta}
\\ & \quad
-q_{\beta\eta}^2q_{\tau\gamma}^2 q_{\beta\gamma}^2q_{\tau\eta}^2 \Bsj_1 \x_{\tau}\x_{\gamma}^2\x_{\eta}^2 \ot x_{\tau}x_{\beta}
-q_{\alpha\delta}q_{\alpha\gamma}^2q_{\alpha\eta}^2 \x_{\delta}\x_{\gamma}^2\x_{\eta}^2 \ot x_{\alpha}.
\end{align*}

Next we compute \eqref{eq:diff-case11-btggke} using \eqref{eq:diff-case11-btggk}, \eqref{eq:diff-case11-btggge}, \eqref{eq:diff-case11-btgke}, \eqref{eq:diff-case11-bggkk}, \eqref{eq:diff-case11-bggke}, \eqref{eq:diff-case11-tggge} and \eqref{eq:diff-case11-tggke}:
\begin{align*}
d( & \x_{\beta}\x_{\tau}\x_{\gamma}^{2}\x_{\kappa}\x_{\eta}\ot 1) = \x_{\beta}\x_{\tau}\x_{\gamma}^{2}\x_{\kappa} \ot x_{\eta}
-s \big( 
q_{\kappa\eta} \x_{\beta}\x_{\tau}\x_{\gamma}^2 \ot x_{\eta}x_{\kappa}
-q_{\gamma\kappa}q_{\gamma\eta} \x_{\beta}\x_{\tau}\x_{\gamma}\x_{\kappa} \ot x_{\eta}x_{\gamma}
\\ & \quad 
-q_{\tau\gamma}^2q_{\tau\kappa} \Bsj_4 \x_{\beta}\x_{\gamma}^2 \x_{\kappa} \ot x_{\kappa}x_{\gamma}
-q_{\tau\gamma}^2q_{\tau\kappa}q_{\tau\eta} \x_{\beta}\x_{\gamma}^2 \x_{\kappa} \ot x_{\eta}x_{\tau}
\\ & \quad 
+q_{\beta\tau}q_{\gamma\beta}^{-2} (3)_{\widetilde{q}_{\beta\gamma}}\Bsj_8 \x_{\tau} \x_{\gamma}^3 \ot x_{\eta}
+q_{\beta\tau}q_{\beta\gamma}^2q_{\beta\kappa}q_{\beta\eta} \x_{\tau} \x_{\gamma}^2\x_{\kappa} \ot x_{\eta}x_{\beta}
\big)
\\ & = \x_{\beta}\x_{\tau}\x_{\gamma}^{2}\x_{\kappa} \ot x_{\eta}
-q_{\kappa\eta} \x_{\beta}\x_{\tau}\x_{\gamma}^2\x_{\eta} \ot x_{\kappa}
+q_{\gamma\kappa}q_{\gamma\eta} \x_{\beta}\x_{\tau}\x_{\gamma}\x_{\kappa}\x_{\eta} \ot x_{\gamma}
\\ & \quad
+q_{\tau\gamma}^2q_{\tau\kappa}(2)_{\widetilde{q}_{\kappa\eta}} \Bsj_4 \x_{\beta}\x_{\gamma}^2\x_{\kappa}^2 \ot x_{\gamma}
+q_{\tau\gamma}^2q_{\tau\kappa}q_{\tau\eta} \x_{\beta}\x_{\gamma}^2\x_{\kappa}\x_{\eta} \ot x_{\tau}
-q_{\beta\tau}q_{\gamma\beta}^{-2} (3)_{\widetilde{q}_{\beta\gamma}}\Bsj_8 \x_{\tau}\x_{\gamma}^3\x_{\eta} \ot 1
\\ & \quad 
-q_{\beta\tau}q_{\beta\gamma}^2q_{\beta\kappa}q_{\beta\eta} \x_{\tau} \x_{\gamma}^2\x_{\kappa}\x_{\eta} \ot x_{\beta}
-q_{\gamma\eta}^2q_{\kappa\eta} \coef{-\eta,\beta,\gamma}{2} \Bsj_4 \Bsj_8 \x_{\gamma}^{4}\x_{\kappa} \ot 1.
\end{align*}

For \eqref{eq:diff-case11-btggee} we use \eqref{eq:diff-case11-btggge}, \eqref{eq:diff-case11-btgee}, \eqref{eq:diff-case11-bggke}, Remark \ref{rem:differential-q-commute} and \eqref{eq:diff-case11-tggee}:
\begin{align*}
d( & \x_{\beta}\x_{\tau}\x_{\gamma}^{2}\x_{\eta}^2\ot 1) = \x_{\beta}\x_{\tau}\x_{\gamma}^{2}\x_{\eta} \ot x_{\eta}
-s \big( 
-q_{\gamma\eta}^2 \x_{\beta}\x_{\tau}\x_{\gamma}\x_{\eta} \ot x_{\eta}x_{\gamma}
-q_{\tau\gamma}^2q_{\tau\eta}^2 \x_{\beta}\x_{\gamma}^2\x_{\eta}\ot x_{\eta}x_{\tau}
\\ &
-q_{\tau\gamma}^2q_{\gamma\eta} \Bsj_4 \x_{\beta} \x_{\gamma}^2\x_{\kappa}\ot x_{\eta}x_{\gamma}
-q_{\tau\gamma}^2q_{\tau\eta} \Bsj_4 \x_{\beta}\x_{\gamma}^2\x_{\eta}\ot x_{\kappa}x_{\gamma}
\\ & 
+q_{\beta\tau}q_{\beta\gamma}^2q_{\beta\eta}^2 \x_{\tau}\x_{\gamma}^2\x_{\eta}\ot x_{\eta}x_{\beta}
-q_{\gamma\eta}^2 \coef{-\eta,\beta,\gamma}{2} \Bsj_4 \Bsj_8 \x_{\gamma}^4 \ot x_{\eta}
\big)
\\ & = \x_{\beta}\x_{\tau}\x_{\gamma}^{2}\x_{\eta} \ot x_{\eta}
+q_{\gamma\eta}^2 \x_{\beta}\x_{\tau}\x_{\gamma}\x_{\eta}^2 \ot x_{\gamma}
+q_{\tau\gamma}^2q_{\gamma\eta} \Bsj_4 \x_{\beta}\x_{\gamma}^2\x_{\kappa}\x_{\eta} \ot x_{\gamma}
+q_{\tau\gamma}^2q_{\tau\eta}^2 \x_{\beta}\x_{\gamma}^2\x_{\eta}^2 \ot x_{\tau}
\\ & \quad
-q_{\beta\tau}q_{\beta\gamma}^2q_{\beta\eta}^2 \x_{\tau}\x_{\gamma}^2\x_{\eta}^2 \ot x_{\beta}
+q_{\gamma\eta}^2 \coef{-\eta,\beta,\gamma}{2} \Bsj_4 \Bsj_8 \x_{\gamma}^4 \ot x_{\eta}.
\end{align*}

The proof of \eqref{eq:diff-case11-bdggee} is similar, using in this case \eqref{eq:diff-case11-bdgge}, \eqref{eq:diff-case11-bdgee}, Remark \ref{rem:differential-q-commute}, \eqref{eq:diff-case11-bggke} and \eqref{eq:diff-case11-dggee}:
\begin{align*}
d( & \x_{\beta}\x_{\delta}\x_{\gamma}^{2}\x_{\eta}^2\ot 1) = \x_{\beta}\x_{\delta}\x_{\gamma}^{2}\x_{\eta} \ot x_{\eta}
-s \big( 
-q_{\gamma\eta}^2 \x_{\beta}\x_{\delta}\x_{\gamma}\x_{\eta} \ot x_{\eta}x_{\gamma}
-q_{\delta\gamma}^2q_{\delta\eta} \Bsj_3 \x_{\beta}\x_{\gamma}^2\x_{\eta} \ot x_{\gamma} x_{\nu}
\\ & \quad
-q_{\delta\gamma}^2q_{\delta\eta} \Bsj_2 \x_{\beta}\x_{\gamma}^2\x_{\eta} \ot x_{\kappa}x_{\gamma}x_{\mu}
-q_{\delta\gamma}^2q_{\delta\eta}^2 \x_{\beta}\x_{\gamma}^2\x_{\eta} \ot x_{\eta}x_{\delta}
-q_{\delta\gamma}^2q_{\mu\eta}q_{\gamma\eta} \Bsj_2 \x_{\beta}\x_{\gamma}^2\x_{\kappa} \ot x_{\eta}x_{\gamma}x_{\mu}
\\ & \quad
-q_{\delta\gamma}^2(3)_{\widetilde{q}_{\gamma\eta}} \Bsj_3\Bsj_7 \x_{\beta}\x_{\gamma}^3 \ot x_{\iota}x_{\kappa}x_{\gamma}
-q_{\delta\gamma}^2q_{\nu\eta}(3)_{\widetilde{q}_{\gamma\eta}} \Bsj_3 \x_{\beta}\x_{\gamma}^3 \ot x_{\eta}x_{\nu}
+q_{\beta\delta}q_{\beta\gamma}^2q_{\beta\eta}^2 \x_{\delta}\x_{\gamma}^2\x_{\eta} \ot x_{\eta}x_{\beta}
\\ & \quad
-q_{\beta\gamma}q_{\gamma\beta}^{-1}q_{\delta\gamma}^2 \coef{\beta\eta}{3} \Bsj_2\Bsj_6\Bsj_8 \x_{\gamma}^4 \ot x_{\iota}x_{\gamma}
-q_{\beta\gamma}q_{\gamma\beta}^{-1}q_{\delta\gamma}^2 \coef{\beta\eta}{3}q_{\mu\eta} \Bsj_2\Bsj_8 \x_{\gamma}^4 \ot x_{\eta}x_{\mu}
\big)
\\ & = \x_{\beta}\x_{\delta}\x_{\gamma}^{2}\x_{\eta} \ot x_{\eta}
+q_{\gamma\eta}^2 \x_{\beta}\x_{\delta}\x_{\gamma}\x_{\eta}^2 \ot x_{\gamma}
+q_{\delta\gamma}^2(3)_{\widetilde{q}_{\gamma\eta}} \Bsj_3\Bsj_7 \x_{\beta}\x_{\gamma}^3\x_{\iota} \ot x_{\kappa}x_{\gamma}
\\ & \quad
+q_{\delta\gamma}^2q_{\nu\eta}(3)_{\widetilde{q}_{\gamma\eta}} \Bsj_3 \x_{\beta}\x_{\gamma}^3\x_{\eta} \ot x_{\nu}
+q_{\delta\gamma}^2q_{\mu\eta}q_{\gamma\eta} \Bsj_2 \x_{\beta}\x_{\gamma}^2\x_{\kappa}\x_{\eta} \ot x_{\gamma}x_{\mu}
\\ & \quad
+q_{\delta\gamma}^2q_{\delta\eta}^2 \x_{\beta}\x_{\gamma}^2\x_{\eta}^2 \ot x_{\delta}
-q_{\beta\delta}q_{\beta\gamma}^2q_{\beta\eta}^2 \x_{\delta}\x_{\gamma}^2\x_{\eta}^2 \ot x_{\beta}
\\ & \quad
-q_{\beta\gamma}q_{\gamma\kappa}q_{\delta\gamma}^2 \coef{\beta\eta}{3} \Bsj_2\Bsj_6\Bsj_8 \x_{\gamma}^4\x_{\iota} \ot x_{\gamma}
-q_{\beta\gamma}q_{\gamma\kappa}q_{\delta\gamma}^2 \coef{\beta\eta}{3}q_{\mu\eta} \Bsj_2\Bsj_8 \x_{\gamma}^4\x_{\eta} \ot x_{\mu}.
\end{align*}

\bigbreak
Finally we compute \eqref{eq:diff-case11-formula}. Using \eqref{eq:diff-case11-abdgge}, \eqref{eq:diff-case11-abdgee}, Remark \ref{rem:differential-q-commute},
\eqref{eq:diff-case11-abggke}, \eqref{eq:diff-case11-adggee}, \eqref{eq:diff-case11-btggke}, \eqref{eq:diff-case11-btggee}, \eqref{eq:diff-case11-bdggee} and \eqref{eq:diff-case11-tggge}:
\begin{align*}
d(\x_{\alpha} & \x_{\beta} \x_{\delta} \x_{\gamma}^{2} \x_{\eta}^2 \ot 1)  = \x_{\alpha} \x_{\beta} \x_{\delta} \x_{\gamma}^{2} \x_{\eta} \ot x_{\eta} - s \big( 
-q_{\gamma\eta}^2 \x_{\alpha}\x_{\beta}\x_{\delta}\x_{\gamma} \x_{\eta} \ot x_{\eta}x_{\gamma}
\\ & \quad
-q_{\delta\gamma}^2(3)_{\widetilde{q}_{\gamma\eta}}\Bsj_3 \x_{\alpha}\x_{\beta} \x_{\gamma}^3 \ot x_{\nu}x_{\eta}
-q_{\delta\gamma}^2\Bsj_2 \x_{\alpha}\x_{\beta}\x_{\gamma}^2 \x_{\kappa} \ot  x_{\gamma}x_{\mu}x_{\eta}
-q_{\delta\gamma}^2q_{\delta\eta} \x_{\alpha}\x_{\beta}\x_{\gamma}^2\x_{\eta} \ot x_{\delta}x_{\eta}
\\ & \quad
+q_{\beta\delta}q_{\beta\gamma}^2q_{\beta\eta} \x_{\alpha} \x_{\delta}\x_{\gamma}^2\x_{\eta} \ot x_{\beta}x_{\eta}
- q_{\alpha\beta}q_{\beta\gamma}^2q_{\beta\eta} q_{\tau\gamma}^2 \Bsj_1\Bsj_4 \x_{\beta}\x_{\tau}\x_{\gamma}^2 \x_{\kappa} \ot x_{\gamma}x_{\beta}x_{\eta}
\\ & \quad
-q_{\delta\gamma}^2 q_{\gamma\beta}^{-2} \coef{\beta\eta\gamma}{3}
\Bsj_2\Bsj_8 \x_{\alpha}\x_{\gamma}^4 \ot x_{\mu}x_{\eta}
- q_{\alpha\beta}q_{\beta\gamma}^2q_{\beta\eta} q_{\tau\gamma}^2 q_{\tau\eta} \Bsj_1 \x_{\beta}\x_{\tau}\x_{\gamma}^2\x_{\eta} \ot  x_{\tau}x_{\beta}x_{\eta}
\\ & \quad
-q_{\alpha\beta}q_{\alpha\delta}q_{\alpha\gamma}^2q_{\alpha\eta} \x_{\beta}\x_{\delta}\x_{\gamma}^2\x_{\eta} \ot x_{\alpha}x_{\eta}
+ q_{\alpha\beta}q_{\beta\gamma}^2q_{\beta\eta}q_{\beta\tau} q_{\gamma\beta}^{-2} (3)_{\widetilde{q}_{\beta\gamma}} q_{\tau\gamma}^2 \Bsj_1\Bsj_4\Bsj_8 \x_{\tau} \x_{\gamma}^4 \ot x_{\beta}x_{\eta} 
\big)
\\ &
= \x_{\alpha} \x_{\beta} \x_{\delta} \x_{\gamma}^{2} \x_{\eta} \ot x_{\eta} 
+q_{\gamma\eta}^2 \x_{\alpha}\x_{\beta}\x_{\delta}\x_{\gamma}\x_{\eta}^2 \ot x_{\gamma}
+q_{\delta\gamma}^2(3)_{\widetilde{q}_{\gamma\eta}} \Bsj_3\Bsj_7 \x_{\alpha}\x_{\beta}\x_{\gamma}^3\x_{\iota} \ot x_{\kappa}x_{\gamma}
\\ & \quad
+q_{\delta\gamma}^2q_{\nu\eta} (3)_{\widetilde{q}_{\gamma\eta}} \Bsj_3 \x_{\alpha}\x_{\beta}\x_{\gamma}^3\x_{\eta} \ot x_{\nu}
+q_{\delta\gamma}^2q_{\mu\eta}q_{\gamma\eta} \Bsj_2 \x_{\alpha}\x_{\beta}\x_{\gamma}^2\x_{\kappa}\x_{\eta} \ot x_{\gamma}x_{\mu}
\\ & \quad
+q_{\delta\gamma}^2q_{\delta\eta}^2 \x_{\alpha}\x_{\beta}\x_{\gamma}^2\x_{\eta}^2 \ot x_{\delta}
-q_{\beta\delta}q_{\beta\gamma}^2q_{\beta\eta}^2 \x_{\alpha} \x_{\delta}\x_{\gamma}^2\x_{\eta}^2 \ot x_{\beta}
\\ & \quad
+q_{\delta\gamma}^2 q_{\gamma\kappa}^2 \coef{\beta\eta\gamma}{3} \Bsj_2\Bsj_6\Bsj_8 \x_{\alpha}\x_{\gamma}^4\x_{\iota} \ot x_{\gamma}
+q_{\delta\gamma}^2 q_{\gamma\kappa}^2q_{\mu\eta} \coef{\beta\eta\gamma}{3} \Bsj_2\Bsj_8 \x_{\alpha}\x_{\gamma}^4\x_{\eta} \ot x_{\mu}
\\ & \quad
+q_{\alpha\beta}q_{\beta\gamma}^2q_{\beta\eta}^2q_{\tau\gamma}^2q_{\gamma\eta} \Bsj_1\Bsj_4 \x_{\beta}\x_{\tau}\x_{\gamma}^2\x_{\kappa}\x_{\eta} \ot x_{\gamma}x_{\beta}
-q_{\alpha\beta}q_{\beta\gamma}^2q_{\beta\eta}^2q_{\tau\gamma}^2q_{\tau\eta}^2 \Bsj_1 \x_{\beta}\x_{\tau}\x_{\gamma}^2\x_{\eta}^2 \ot x_{\tau}x_{\beta}
\\ & \quad
+q_{\alpha\beta}q_{\alpha\delta}q_{\alpha\gamma}^2q_{\alpha\eta}^2 \x_{\beta}\x_{\delta}\x_{\gamma}^2\x_{\eta}^2 \ot x_{\alpha}
-q_{\alpha\beta}q_{\beta\tau}q_{\beta\gamma}^2q_{\beta\eta}^2q_{\gamma\beta}^{-2}q_{\tau\gamma}^2\coef{\beta\eta\gamma}{3} \Bsj_1\Bsj_4\Bsj_8 \x_{\tau}\x_{\gamma}^4\x_{\eta} \ot x_{\beta} 
\\ & \quad
-q_{\delta\gamma}^2 q_{\gamma\kappa}^2 \coeff{\alpha\beta\eta\gamma}{4} \Bsj_2\Bsj_6\Bsj_8\Bsj_9 \x_{\gamma}^6 \ot 1.
\end{align*}
This completes the proof.
\epf

\bigbreak

\end{document}